\documentclass[12pt]{ociamthesis}
\usepackage[colorlinks,linkcolor=darkblue,citecolor=darkblue,urlcolor=darkblue,breaklinks=true,pagebackref]{hyperref}
\usepackage[centertags,sumlimits,intlimits,namelimits,reqno]{amsmath}
\usepackage{latexsym,amsfonts,amssymb,exscale,enumerate,amsthm,breakcites}
\usepackage{tikz}
\usepackage{float,color,multirow,array,multicol}
\usepackage[all,cmtip,2cell]{xy}
\usepackage{fancyhdr}

\usepackage{stmaryrd,mathrsfs}
\usepackage{anyfontsize}

\definecolor{darkblue}{rgb}{0,0,0.7}

\newtheorem{theorem}{Theorem}[chapter]
\newtheorem{lemma}[theorem]{Lemma}
\newtheorem{proposition}[theorem]{Proposition}
\newtheorem{corollary}[theorem]{Corollary}

\theoremstyle{definition}
\newtheorem{definition}[theorem]{Definition}
\newtheorem{example}[theorem]{Example}
\newtheorem{examples}[theorem]{Examples}
\newtheorem{remark}[theorem]{Remark}

\newtheorem{aside}[theorem]{Aside}

\newcommand\maps{{\colon}}
\newcommand{\define}[1]{{\bf \boldmath #1}}

  \newcommand{\R}{{\mathbb{R}}}
  \newcommand{\hooklongrightarrow}{\lhook\joinrel\longrightarrow}
  \newcommand{\linsub}{\operatorname{LinSub}}
  \newcommand{\lgraph}{\operatorname{Graph}}
  \newcommand{\res}{\operatorname{Res}}
  \newcommand{\FinSet}{\mathrm{FinSet}}
  \newcommand{\Set}{\mathrm{Set}}
  \newcommand{\opp}{\mathrm{op}}

  \newcommand{\FinVect}{\mathrm{FinVect}}
  \newcommand{\Vect}{\mathrm{Vect}}
  \newcommand{\LinRel}{\mathrm{LinRel}}
  \DeclareMathOperator\corel{{Corel}}
  \DeclareMathOperator\cospan{{Cospan}}

\SetSymbolFont{stmry}{bold}{U}{stmry}{m}{n}

\newcommand\z{{\mathbb Z}}
\renewcommand\k{k}   
\newcommand\nn{{\mathbb N}}
\newcommand\bb{{\mathscr B}}
\newcommand{\idn}{\mathrm{id}}
\newcommand{\tw}{\mathrm{tw}}
\newcommand{\tm}{\tau}
\newcommand\pr{{\R[s,s^{-1}]}}
\newcommand\pk{{\k[s,s^{-1}]}}

\newcommand{\Defeq}{\stackrel{\mathrm{def}}{=}}
\newcommand{\vectfun}{\theta}
\newcommand{\cospanfun}{\Theta}
\newcommand{\cospanfunrest}{\overline{\Theta}}
\newcommand{\ha}{\mathbb{HA}}
\newcommand{\ih}{\mathbb{IH}}
\newcommand{\ihcsp}{\ih^{\mathrm{Csp}}}
\newcommand{\ihcor}{\ih^{\mathrm{Cor}}}
\DeclareMathOperator{\mat}{\mathrm{Mat}}
\DeclareMathOperator{\vect}{\mathrm{Vect}}
\DeclareMathOperator{\ltids}{\mathrm{LTI}}

\DeclareMathOperator{\linrel}{\mathrm{LinRel}}
\DeclareMathOperator{\fmod}{\mathrm{FMod}}

\DeclareMathOperator{\bit}{bit}

\newcommand\addgen{\lower8pt\hbox{$\includegraphics[height=0.7cm]{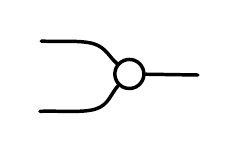}$}}
\newcommand\zerogen{\lower5pt\hbox{$\includegraphics[height=0.5cm]{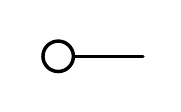}$}}
\newcommand\copygen{\lower8pt\hbox{$\includegraphics[height=0.7cm]{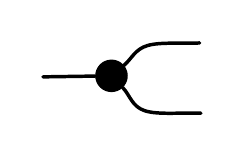}$}}
\newcommand\discardgen{\lower5pt\hbox{$\includegraphics[height=0.5cm]{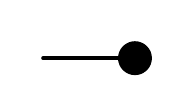}$}}
\newcommand\delaygen{\lower6pt\hbox{$\includegraphics[height=0.6cm]{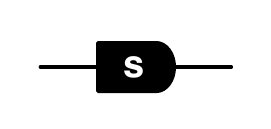}$}}
\newcommand\minonegen{\lower6pt\hbox{$\includegraphics[height=0.6cm]{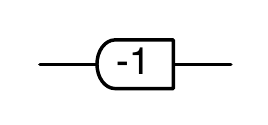}$}}
\newcommand\delayopgen{\lower6pt\hbox{$\includegraphics[height=0.6cm]{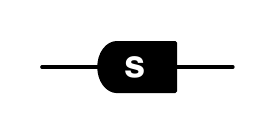}$}}
\newcommand\scalargen{\lower6pt\hbox{$\includegraphics[height=0.6cm]{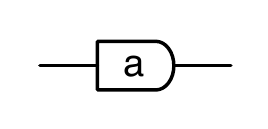}$}}
\newcommand\addopgen{\lower8pt\hbox{$\includegraphics[height=0.7cm]{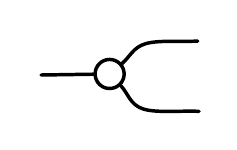}$}}
\newcommand\zeroopgen{\lower5pt\hbox{$\includegraphics[height=0.5cm]{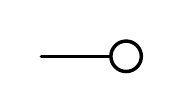}$}}
\newcommand\copyopgen{\lower8pt\hbox{$\includegraphics[height=0.7cm]{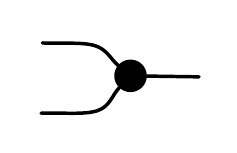}$}}
\newcommand\discardopgen{\lower5pt\hbox{$\includegraphics[height=0.5cm]{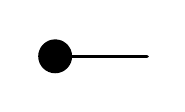}$}}
\newcommand\scalaropgen{\lower6pt\hbox{$\includegraphics[height=0.6cm]{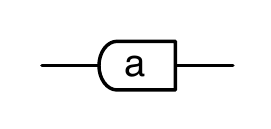}$}}
\newcommand\delaygenl{\lower6pt\hbox{$\includegraphics[height=0.6cm]{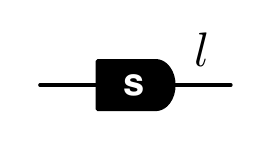}$}}
\newcommand\delayopgenl{\lower6pt\hbox{$\includegraphics[height=0.6cm]{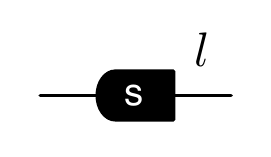}$}}
\newcommand\delaygenk{\lower6pt\hbox{$\includegraphics[height=0.6cm]{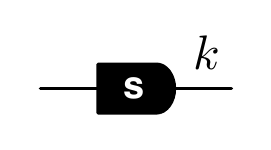}$}}
\newcommand\delayopgenk{\lower6pt\hbox{$\includegraphics[height=0.6cm]{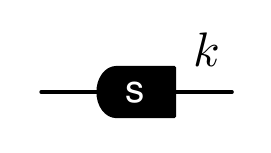}$}}
\newcommand\twist{\lower6pt\hbox{$\includegraphics[height=0.6cm]{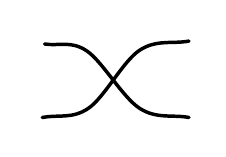}$}}
\newcommand\id{\lower3pt\hbox{$\includegraphics[height=0.3cm]{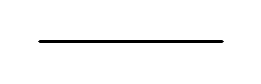}$}}
\newcommand\syntax{\mathbb{S}}
 
  \makeatletter
     \def\moverlay{\mathpalette\mov@rlay}
     \def\mov@rlay#1#2{\leavevmode\vtop{%
        \baselineskip\z@skip \lineskiplimit-\maxdimen
        \ialign{\hfil$#1##$\hfil\cr#2\crcr}}}
   \makeatother
\newcommand\twarr[2]{%
\mathrel{\mathop{\moverlay{\scriptstyle\xrightarrow{\,#1\,}\cr{\lower.2em\hbox{$\scriptstyle{}_{#2}$}}}}}}
\newcommand{\dtrans}[2]{\hbox{$\;\twarr{#1}{#2}\;$}}
\newcommand{\labelSep}{\,}

\newcommand\C{{\mathbb C}}
\newcommand\F{{\mathbb F}}
\newcommand{\vectf}[1]{{\F^{#1} \oplus {(\F^{#1})}^\ast}}

\newcommand{\mc}{\mathcal}
\renewcommand{\a}{\alpha}
\newcommand{\s}{\sigma}
\newcommand{\ot}{\otimes}

\newcommand{\im}{\mathrm{Im}\,}

\newcommand{\Circ}{\mathrm{Circ}}
\newcommand{\LagrRel}{\mathrm{LagrRel}}

\usetikzlibrary{backgrounds,circuits,circuits.ee.IEC,shapes,fit,matrix,automata,decorations.markings}
\UseTwocells

  \pgfdeclarelayer{edgelayer}
  \pgfdeclarelayer{nodelayer}
  \pgfdeclarelayer{background}
  \pgfsetlayers{background,edgelayer,nodelayer,main}

  \tikzset{font=\footnotesize}
  \tikzstyle{none}=[inner sep=0pt]

\tikzstyle{circ}=[circle,fill=black,draw,inner sep=3pt]
\tikzstyle{circ2}=[circle,fill=black,draw,inner sep=1pt]
\tikzstyle{dot}=[circle,fill=black,draw,inner sep=1.5pt]
\tikzstyle{dotbig}=[circle,fill=black,draw,inner sep=2.2pt]
\tikzstyle{sdot}=[circle,fill=white,draw,inner sep=1pt]
\tikzstyle{amp}=[rounded rectangle, rounded rectangle west arc=0pt, draw, inner
sep = 2pt,minimum width=4ex,fill=white]

\tikzstyle{sq}=[rectangle,draw,inner sep=1pt,minimum width=15pt,minimum height=15pt]
  
\tikzstyle{mapsto}=[color=gray,very thick,shorten>=10pt,shorten<=5pt,->,>=stealth]

\tikzstyle{connection}=[circuit symbol open,
    circuit symbol size=width .5 height .5,
    shape=circle ee,
  inner sep=0.25\pgflinewidth]
  \tikzset{->-/.style={decoration={
    markings,
    mark=at position .54 with {\arrow{latex}}},postaction={decorate}}}

\newcommand{\mult}[1]
{
\begin{aligned}
    \resizebox{#1}{!}{
\begin{tikzpicture}
	\begin{pgfonlayer}{nodelayer}
		\node [style=none] (0) at (1, -0) {};
		\node [style=circ] (1) at (0.125, -0) {};
		\node [style=none] (2) at (-1, 0.5) {};
		\node [style=none] (3) at (-1, -0.5) {};
	\end{pgfonlayer}
	\begin{pgfonlayer}{edgelayer}
		\draw[line width=2pt] (0.center) to (1.center);
		\draw[line width=2pt] [in=0, out=120, looseness=1.20] (1.center) to (2.center);
		\draw[line width=2pt] [in=0, out=-120, looseness=1.20] (1.center) to (3.center);
	\end{pgfonlayer}
      \end{tikzpicture}}
\end{aligned}
}

\newcommand{\unit}[1]
{
  \begin{aligned}
    \resizebox{#1}{!}{
\begin{tikzpicture}
	\begin{pgfonlayer}{nodelayer}
		\node [style=none] (0) at (1, -0) {};
		\node [style=none] (1) at (-1, -0) {};
		\node [style=circ] (2) at (0, -0) {};
	\end{pgfonlayer}
	\begin{pgfonlayer}{edgelayer}
		\draw[line width=2pt] (0.center) to (2);
	\end{pgfonlayer}
      \end{tikzpicture}}
  \end{aligned}
}

\newcommand{\comult}[1]
{
\begin{aligned}
    \resizebox{#1}{!}{
\begin{tikzpicture}
	\begin{pgfonlayer}{nodelayer}
		\node [style=none] (0) at (-1, -0) {};
		\node [style=circ] (1) at (-0.125, -0) {};
		\node [style=none] (2) at (1, 0.5) {};
		\node [style=none] (3) at (1, -0.5) {};
	\end{pgfonlayer}
	\begin{pgfonlayer}{edgelayer}
		\draw[line width=2pt] (0.center) to (1.center);
		\draw[line width=2pt] [in=180, out=60, looseness=1.20] (1.center) to (2.center);
		\draw[line width=2pt] [in=180, out=-60, looseness=1.20] (1.center) to (3.center);
	\end{pgfonlayer}
      \end{tikzpicture}}
\end{aligned}
}

\newcommand{\counit}[1]
{
  \begin{aligned}
    \resizebox{#1}{!}{
\begin{tikzpicture}
	\begin{pgfonlayer}{nodelayer}
		\node [style=none] (0) at (-1, -0) {};
		\node [style=none] (1) at (1, -0) {};
		\node [style=circ] (2) at (0, -0) {};
	\end{pgfonlayer}
	\begin{pgfonlayer}{edgelayer}
		\draw[line width=2pt] (0.center) to (2);
	\end{pgfonlayer}
      \end{tikzpicture}}
  \end{aligned}
}

\newcommand{\idone}[1]
{
  \begin{aligned}
    \resizebox{#1}{!}{
     \begin{tikzpicture}
	\begin{pgfonlayer}{nodelayer}
		\node [style=none] (0) at (-1, -0) {};
		\node [style=none] (1) at (1, -0) {};
		\node [style=none] (2) at (0, 0.5) {};
		\node [style=none] (3) at (0, -0.5) {};
	\end{pgfonlayer}
	\begin{pgfonlayer}{edgelayer}
		\draw[line width=2pt] (1.center) to (0.center);
	\end{pgfonlayer}
\end{tikzpicture} 
    }
  \end{aligned}
}

\newcommand{\swap}[1]
{
  \begin{aligned}
    \resizebox{#1}{!}{
\begin{tikzpicture}
	\begin{pgfonlayer}{nodelayer}
		\node [style=none] (2) at (-0.5, -0.5) {};
		\node [style=none] (3) at (-2, 0.5) {};
		\node [style=none] (4) at (-0.5, 0.5) {};
		\node [style=none] (5) at (-2, -0.5) {};
	\end{pgfonlayer}
	\begin{pgfonlayer}{edgelayer}
		\draw[line width=2pt] [in=180, out=0, looseness=1.00] (3.center) to (2.center);
		\draw[line width=2pt] [in=0, out=180, looseness=1.00] (4.center) to (5.center);
	\end{pgfonlayer}
\end{tikzpicture}
    }
  \end{aligned}
}
\newcommand{\assocl}[1]
{
  \begin{aligned}
    \resizebox{#1}{!}{
\begin{tikzpicture}
	\begin{pgfonlayer}{nodelayer}
		\node [style=circ] (0) at (0.125, -0) {};
		\node [style=none] (1) at (-1, 0.5) {};
		\node [style=none] (2) at (-1, -0.5) {};
		\node [style=none] (3) at (0, -1) {};
		\node [style=none] (4) at (2.25, -0.5) {};
		\node [style=none] (5) at (0.25, -0) {};
		\node [style=circ] (6) at (1.25, -0.5) {};
		\node [style=none] (7) at (-1, -1) {};
	\end{pgfonlayer}
	\begin{pgfonlayer}{edgelayer}
		\draw[line width=2pt] [in=0, out=120, looseness=1.20] (0.center) to (1.center);
		\draw[line width=2pt] [in=0, out=-120, looseness=1.20] (0.center) to (2.center);
		\draw[line width=2pt] (4.center) to (6);
		\draw[line width=2pt] [in=0, out=120, looseness=1.20] (6) to (5.center);
		\draw[line width=2pt] [in=0, out=-120, looseness=1.20] (6) to (3.center);
		\draw[line width=2pt] (3.center) to (7.center);
	\end{pgfonlayer}
      \end{tikzpicture}}
  \end{aligned}
}

\newcommand{\assocr}[1]
{
  \begin{aligned}
    \resizebox{#1}{!}{
\begin{tikzpicture}
	\begin{pgfonlayer}{nodelayer}
		\node [style=circ] (0) at (0.125, -0.5) {};
		\node [style=none] (1) at (-1, -1) {};
		\node [style=none] (2) at (-1, 0) {};
		\node [style=none] (3) at (0, 0.5) {};
		\node [style=none] (4) at (2.25, 0) {};
		\node [style=none] (5) at (0.25, -0.5) {};
		\node [style=circ] (6) at (1.25, 0) {};
		\node [style=none] (7) at (-1, 0.5) {};
	\end{pgfonlayer}
	\begin{pgfonlayer}{edgelayer}
		\draw[line width=2pt] [in=0, out=-120, looseness=1.20] (0.center) to (1.center);
		\draw[line width=2pt] [in=0, out=120, looseness=1.20] (0.center) to (2.center);
		\draw[line width=2pt] (4.center) to (6);
		\draw[line width=2pt] [in=0, out=-120, looseness=1.20] (6) to (5.center);
		\draw[line width=2pt] [in=0, out=120, looseness=1.20] (6) to (3.center);
		\draw[line width=2pt] (3.center) to (7.center);
	\end{pgfonlayer}
      \end{tikzpicture}}
  \end{aligned}
}

\newcommand{\coassocl}[1]
{
  \begin{aligned}
    \resizebox{#1}{!}{
\begin{tikzpicture}
	\begin{pgfonlayer}{nodelayer}
		\node [style=circ] (0) at (1.125, -0.5) {};
		\node [style=none] (1) at (2.25, -1) {};
		\node [style=none] (2) at (2.25, 0) {};
		\node [style=none] (3) at (1.25, 0.5) {};
		\node [style=none] (4) at (-1, 0) {};
		\node [style=none] (5) at (1, -0.5) {};
		\node [style=circ] (6) at (0, 0) {};
		\node [style=none] (7) at (2.25, 0.5) {};
	\end{pgfonlayer}
	\begin{pgfonlayer}{edgelayer}
		\draw[line width=2pt] [in=180, out=-60, looseness=1.20] (0.center) to (1.center);
		\draw[line width=2pt] [in=180, out=60, looseness=1.20] (0.center) to (2.center);
		\draw[line width=2pt] (4.center) to (6);
		\draw[line width=2pt] [in=180, out=-60, looseness=1.20] (6) to (5.center);
		\draw[line width=2pt] [in=180, out=60, looseness=1.20] (6) to (3.center);
		\draw[line width=2pt] (3.center) to (7.center);
	\end{pgfonlayer}
      \end{tikzpicture}}
  \end{aligned}
}

\newcommand{\coassocr}[1]
{
  \begin{aligned}
    \resizebox{#1}{!}{
\begin{tikzpicture}
	\begin{pgfonlayer}{nodelayer}
		\node [style=circ] (0) at (1.125, 0) {};
		\node [style=none] (1) at (2.25, 0.5) {};
		\node [style=none] (2) at (2.25, -0.5) {};
		\node [style=none] (3) at (1.25, -1) {};
		\node [style=none] (4) at (-1, -0.5) {};
		\node [style=none] (5) at (1, 0) {};
		\node [style=circ] (6) at (0, -0.5) {};
		\node [style=none] (7) at (2.25, -1) {};
	\end{pgfonlayer}
	\begin{pgfonlayer}{edgelayer}
		\draw[line width=2pt] [in=180, out=60, looseness=1.20] (0.center) to (1.center);
		\draw[line width=2pt] [in=180, out=-60, looseness=1.20] (0.center) to (2.center);
		\draw[line width=2pt] (4.center) to (6);
		\draw[line width=2pt] [in=180, out=60, looseness=1.20] (6) to (5.center);
		\draw[line width=2pt] [in=180, out=-60, looseness=1.20] (6) to (3.center);
		\draw[line width=2pt] (3.center) to (7.center);
	\end{pgfonlayer}
      \end{tikzpicture}}
  \end{aligned}
}

\newcommand{\unitl}[1]
{
  \begin{aligned}
    \resizebox{#1}{!}{
\begin{tikzpicture}
	\begin{pgfonlayer}{nodelayer}
		\node [style=none] (0) at (1, -0) {};
		\node [style=circ] (1) at (0.125, -0) {};
		\node [style=circ] (2) at (-1, 0.5) {};
		\node [style=none] (3) at (-1, -0.5) {};
		\node [style=none] (4) at (-2, -0.5) {};
	\end{pgfonlayer}
	\begin{pgfonlayer}{edgelayer}
		\draw[line width=2pt] (0.center) to (1.center);
		\draw[line width=2pt] [in=0, out=120, looseness=1.20] (1.center) to (2.center);
		\draw[line width=2pt] [in=0, out=-120, looseness=1.20] (1.center) to (3.center);
		\draw[line width=2pt] (4.center) to (3.center);
	\end{pgfonlayer}
\end{tikzpicture}
    }
  \end{aligned}
}

\newcommand{\counitl}[1]
{
  \begin{aligned}
    \resizebox{#1}{!}{
\begin{tikzpicture}
	\begin{pgfonlayer}{nodelayer}
		\node [style=none] (0) at (-2, -0) {};
		\node [style=circ] (1) at (-1.125, -0) {};
		\node [style=circ] (2) at (0, 0.5) {};
		\node [style=none] (3) at (0, -0.5) {};
		\node [style=none] (4) at (1, -0.5) {};
	\end{pgfonlayer}
	\begin{pgfonlayer}{edgelayer}
		\draw[line width=2pt] (0.center) to (1.center);
		\draw[line width=2pt] [in=180, out=60, looseness=1.20] (1.center) to (2.center);
		\draw[line width=2pt] [in=180, out=-60, looseness=1.20] (1.center) to (3.center);
		\draw[line width=2pt] (4.center) to (3.center);
	\end{pgfonlayer}
\end{tikzpicture}
    }
  \end{aligned}
}

\newcommand{\commute}[1]
{
  \begin{aligned}
    \resizebox{#1}{!}{
\begin{tikzpicture}
	\begin{pgfonlayer}{nodelayer}
		\node [style=none] (0) at (1.25, -0) {};
		\node [style=circ] (1) at (0.375, -0) {};
		\node [style=none] (2) at (-0.5, -0.5) {};
		\node [style=none] (3) at (-2, 0.5) {};
		\node [style=none] (4) at (-0.5, 0.5) {};
		\node [style=none] (5) at (-2, -0.5) {};
	\end{pgfonlayer}
	\begin{pgfonlayer}{edgelayer}
		\draw[line width=2pt] (0.center) to (1.center);
		\draw[line width=2pt] [in=0, out=-120, looseness=1.20] (1.center) to (2.center);
		\draw[line width=2pt] [in=180, out=0, looseness=1.00] (3.center) to (2.center);
		\draw[line width=2pt] [in=0, out=120, looseness=1.20] (1.center) to (4.center);
		\draw[line width=2pt] [in=0, out=180, looseness=1.00] (4.center) to (5.center);
	\end{pgfonlayer}
\end{tikzpicture}
    }
  \end{aligned}
}

\newcommand{\cocommute}[1]
{
  \begin{aligned}
    \resizebox{#1}{!}{
\begin{tikzpicture}
	\begin{pgfonlayer}{nodelayer}
		\node [style=none] (0) at (-2, -0) {};
		\node [style=circ] (1) at (-1.125, -0) {};
		\node [style=none] (2) at (-0.25, -0.5) {};
		\node [style=none] (3) at (1.25, 0.5) {};
		\node [style=none] (4) at (-0.25, 0.5) {};
		\node [style=none] (5) at (1.25, -0.5) {};
	\end{pgfonlayer}
	\begin{pgfonlayer}{edgelayer}
		\draw[line width=2pt] (0.center) to (1.center);
		\draw[line width=2pt] [in=180, out=-60, looseness=1.20] (1.center) to (2.center);
		\draw[line width=2pt] [in=0, out=180, looseness=1.00] (3.center) to (2.center);
		\draw[line width=2pt] [in=180, out=60, looseness=1.20] (1.center) to (4.center);
		\draw[line width=2pt] [in=180, out=0, looseness=1.00] (4.center) to (5.center);
	\end{pgfonlayer}
\end{tikzpicture}
    }
  \end{aligned}
}
\newcommand{\frobs}[1]
{
  \begin{aligned}
    \resizebox{#1}{!}{
\begin{tikzpicture}
	\begin{pgfonlayer}{nodelayer}
		\node [style=none] (0) at (-1.5, 0.5) {};
		\node [style=circ] (1) at (-0.75, 0.5) {};
		\node [style=none] (2) at (0.25, -0) {};
		\node [style=none] (3) at (0.25, 1) {};
		\node [style=circ] (4) at (1, -0.5) {};
		\node [style=none] (5) at (0, -0) {};
		\node [style=none] (6) at (1.75, -0.5) {};
		\node [style=none] (7) at (0, -1) {};
		\node [style=none] (8) at (1.75, 1) {};
		\node [style=none] (9) at (-1.5, -1) {};
	\end{pgfonlayer}
	\begin{pgfonlayer}{edgelayer}
		\draw[line width=2pt] [in=180, out=-60, looseness=1.20] (1) to (2.center);
		\draw[line width=2pt] [in=180, out=60, looseness=1.20] (1) to (3.center);
		\draw[line width=2pt] (0.center) to (1);
		\draw[line width=2pt] (6.center) to (4);
		\draw[line width=2pt] [in=0, out=120, looseness=1.20] (4) to (5.center);
		\draw[line width=2pt] [in=0, out=-120, looseness=1.20] (4) to (7.center);
		\draw[line width=2pt] (3.center) to (8.center);
		\draw[line width=2pt] (7.center) to (9.center);
	\end{pgfonlayer}
\end{tikzpicture}
    }
  \end{aligned}
}

\newcommand{\frobx}[1]
{
  \begin{aligned}
    \resizebox{#1}{!}{
\begin{tikzpicture}
	\begin{pgfonlayer}{nodelayer}
		\node [style=circ] (0) at (-0.5, -0) {};
		\node [style=none] (1) at (-1.5, -0.5) {};
		\node [style=none] (2) at (-1.5, 0.5) {};
		\node [style=circ] (3) at (0.5, -0) {};
		\node [style=none] (4) at (1.5, -0.5) {};
		\node [style=none] (5) at (1.5, 0.5) {};
	\end{pgfonlayer}
	\begin{pgfonlayer}{edgelayer}
		\draw[line width=2pt, in=0, out=-120, looseness=1.20] (0.center) to (1.center);
		\draw[line width=2pt, in=0, out=120, looseness=1.20] (0.center) to (2.center);
		\draw[line width=2pt, in=180, out=-60, looseness=1.20] (3) to (4.center);
		\draw[line width=2pt, in=180, out=60, looseness=1.20] (3) to (5.center);
		\draw[line width=2pt] (0) to (3);
	\end{pgfonlayer}
\end{tikzpicture}
    }
  \end{aligned}
}

\newcommand{\frobz}[1]
{
  \begin{aligned}
    \resizebox{#1}{!}{
\begin{tikzpicture}
	\begin{pgfonlayer}{nodelayer}
		\node [style=none] (0) at (1.75, 0.5) {};
		\node [style=circ] (1) at (1, 0.5) {};
		\node [style=none] (2) at (0, -0) {};
		\node [style=none] (3) at (0, 1) {};
		\node [style=circ] (4) at (-0.75, -0.5) {};
		\node [style=none] (5) at (0.25, -0) {};
		\node [style=none] (6) at (-1.5, -0.5) {};
		\node [style=none] (7) at (0.25, -1) {};
		\node [style=none] (8) at (-1.5, 1) {};
		\node [style=none] (9) at (1.75, -1) {};
	\end{pgfonlayer}
	\begin{pgfonlayer}{edgelayer}
		\draw[line width=2pt] [in=0, out=-120, looseness=1.20] (1) to (2.center);
		\draw[line width=2pt] [in=0, out=120, looseness=1.20] (1) to (3.center);
		\draw[line width=2pt] (0.center) to (1);
		\draw[line width=2pt] (6.center) to (4);
		\draw[line width=2pt] [in=180, out=60, looseness=1.20] (4) to (5.center);
		\draw[line width=2pt] [in=180, out=-60, looseness=1.20] (4) to (7.center);
		\draw[line width=2pt] (3.center) to (8.center);
		\draw[line width=2pt] (7.center) to (9.center);
	\end{pgfonlayer}
\end{tikzpicture}
    }
  \end{aligned}
}
\newcommand{\spec}[1]
{
  \begin{aligned}
    \resizebox{#1}{!}{
\begin{tikzpicture}
	\begin{pgfonlayer}{nodelayer}
		\node [style=none] (0) at (1.75, -0) {};
		\node [style=circ] (1) at (0.75, -0) {};
		\node [style=none] (2) at (0, -0.5) {};
		\node [style=none] (3) at (0, 0.5) {};
		\node [style=circ] (4) at (-0.75, -0) {};
		\node [style=none] (5) at (0, -0.5) {};
		\node [style=none] (6) at (-1.75, -0) {};
		\node [style=none] (7) at (0, 0.5) {};
	\end{pgfonlayer}
	\begin{pgfonlayer}{edgelayer}
		\draw[line width=2pt] (0.center) to (1.center);
		\draw[line width=2pt] [in=0, out=-120, looseness=1.20] (1.center) to (2.center);
		\draw[line width=2pt] [in=0, out=120, looseness=1.20] (1.center) to (3.center);
		\draw[line width=2pt] (6.center) to (4);
		\draw[line width=2pt] [in=180, out=-60, looseness=1.20] (4) to (5.center);
		\draw[line width=2pt] [in=180, out=60, looseness=1.20] (4) to (7.center);
	\end{pgfonlayer}
\end{tikzpicture}
    }
  \end{aligned}
}

\newcommand{\extral}[1]
{
  \begin{aligned}
    \resizebox{#1}{!}{
\begin{tikzpicture}
	\begin{pgfonlayer}{nodelayer}
		\node [style=none] (0) at (1.75, -0) {};
		\node [style=circ] (1) at (0.75, -0) {};
		\node [style=circ] (4) at (-0.75, -0) {};
		\node [style=none] (6) at (-1.75, -0) {};
	\end{pgfonlayer}
	\begin{pgfonlayer}{edgelayer}
	  \draw[line width=2pt] (1.center) to (4.center);
	\end{pgfonlayer}
\end{tikzpicture}
    }
  \end{aligned}
}

\newcommand{\extrar}[1]
{
  \begin{aligned}
    \resizebox{#1}{!}{
\begin{tikzpicture}
	\begin{pgfonlayer}{nodelayer}
		\node [style=none] (0) at (1.75, -0) {};
		\node [style=none] (6) at (-1.75, -0) {};
	\end{pgfonlayer}
\end{tikzpicture}
    }
  \end{aligned}
}

\title{The Algebra of Open and \\[1ex]
	Interconnected Systems}   %note \\[1ex] is a line break in the title
\author{Brendan Fong}     
\college{Hertford College}
\degree{Doctor of Philosophy in Computer Science} 
\degreedate{Trinity 2016}    

\begin{document}
\baselineskip=18pt plus1pt
\setcounter{secnumdepth}{3} %set the number of sectioning levels that get number and appear in the contents
\setcounter{tocdepth}{3}

\maketitle                 
\pagestyle{empty}

\begin{quotation}
\textit{For all those who have prepared food so I could eat and created homes so
  I could live over the past four years. You too have laboured to produce this; I
hope I have done your labours justice.}
\end{quotation}

\begin{abstract}
  Herein we develop category-theoretic tools for understanding network-style
  diagrammatic languages. The archetypal network-style diagrammatic language is
  that of electric circuits; other examples include signal flow graphs, Markov
  processes, automata, Petri nets, chemical reaction networks, and so on. The
  key feature is that the language is comprised of a number of \emph{components}
  with multiple (input/output) terminals, each possibly labelled with some type,
  that may then be connected together along these terminals to form a larger
  \emph{network}. The components form hyperedges between labelled vertices, and
  so a diagram in this language forms a hypergraph. We formalise the
  compositional structure by introducing the notion of a hypergraph category.
  Network-style diagrammatic languages and their semantics thus form hypergraph
  categories, and semantic interpretation gives a hypergraph functor. 

  The first part of this thesis develops the theory of hypergraph categories. In
  particular, we introduce the tools of decorated cospans and corelations.
  Decorated cospans allow straightforward construction of hypergraph categories
  from diagrammatic languages: the inputs, outputs, and their composition are
  modelled by the cospans, while the `decorations' specify the components
  themselves. Not all hypergraph categories can be constructed, however, through
  decorated cospans. Decorated corelations are a more powerful version that
  permits construction of all hypergraph categories and hypergraph functors.
  These are often useful for constructing the semantic categories of
  diagrammatic languages and functors from diagrams to the semantics. To
  illustrate these principles, the second part of this thesis details
  applications to linear time-invariant dynamical systems and passive linear
  networks.
\end{abstract}

\pagestyle{plain}
\begin{romanpages} 
\tableofcontents
\phantomsection
\addcontentsline{toc}{chapter}{Preface}
\chapter*{Preface}

This is a thesis in the mathematical sciences, with emphasis on the mathematics.
But before we get to the category theory, I want to say a few words about the
scientific tradition in which this thesis is situated.

Mathematics is the language of science. Twinned so intimately with physics, over
the past centuries mathematics has become a superb---indeed, unreasonably
effective---language for understanding planets moving in space, particles in a
vacuum, the structure of spacetime, and so on.  Yet, while Wigner speaks of the
unreasonable effectiveness of mathematics in the natural sciences \cite{Wig60},
equally eminent mathematicians, not least Gelfand, speak of the unreasonable
\emph{ineffectiveness} of mathematics in biology \cite{Mon07} and related
fields. Why such a difference?

A contrast between physics and biology is that while physical systems can often
be studied in isolation---the proverbial particle in a vacuum---, biological
systems are necessarily situated in their environment. A heart belongs in a
body, an ant in a colony. One of the first to draw attention to this contrast
was Ludwig von Bertalanffy, biologist and founder of general systems theory, who
articulated the difference as one between closed and open systems: 
\begin{quote}
  Conventional physics deals only with closed systems, i.e. systems which are
  considered to be isolated from their environment. \dots\ However, we find
  systems which by their very nature and definition are not closed systems.
  Every living organism is essentially an open system. It maintains itself in a
  continuous inflow and outflow, a building up and breaking down of components,
  never being, so long as it is alive, in a state of chemical and thermodynamic
  equilibrium but maintained in a so-called steady state which is distinct from
  the latter \cite{Ber68}.
\end{quote}
While the ambitious generality of general systems theory has proved difficult,
von Bertalanffy's philosophy has had great impact in his home field of biology,
leading to the modern field of systems biology. Half a century later, Dennis
Noble, another great pioneer of systems biology and the originator of the first
mathematical model of a working heart, describes the shift as one from reduction
to integration.
\begin{quote}
  Systems biology\dots\ is about putting together rather than taking apart,
  integration rather than reduction. It requires that we develop ways of
  thinking about integration that are as rigorous as our reductionist
  programmes, but different \cite{Nob06}.
%\dots\ It means changing our philosophy, in the full sense of the term 
\end{quote}
In this thesis we develop rigorous ways of thinking about integration or, as we
refer to it, interconnection.

Interconnection and openness are tightly related. Indeed, openness implies that
a system may be interconnected with its environment. But what is an environment
but comprised of other systems? Thus the study of open systems becomes the study
of how a system changes under interconnection with other systems.

To model this, we must begin by creating language to describe the
interconnection of systems. While reductionism hopes that phenomena can be
explained by reducing them to ``elementary units investigable independently of
each other'' \cite{Ber68}, this philosophy of integration introduces as an
additional and equal priority the investigation of the way these units are
interconnected. As such, this thesis is predicated on the hope that the meaning
of an expression in our new language is determined by the meanings of its
constituent expressions together with the syntactic rules combining them. This
is known as the principle of compositionality. 

Also commonly known as Frege's principle, the principle of compositionality both
dates back to Ancient Greek and Vedic philosophy, and is still the subject of
active research today \cite{Jan86,Sza13}. More recently, through the work of
Montague \cite{Mon70} in natural language semantics and Strachey and Scott
\cite{SS71} in programming language semantics, the principle of compositionality
has found formal expression as the dictum that the interpretation of a language
should be given by a homomorphism between an algebra of syntactic
representations and an algebra of semantic objects. We too shall follow this
route.

The question then arises: what do we mean by algebra? This mathematical question
leads us back to our scientific objectives: what do we mean by system? Here we
must narrow, or at least define, our scope. We give some examples. The
investigations of this thesis began with electrical circuits and their diagrams,
and we will devote significant time to exploring their compositional
formulation. We discussed biological systems above, and our notion of system
includes these, modelled say in the form of chemical reaction networks or Markov
processes, or the compartmental models of epidemiology, population biology, and
ecology. From computer science, we consider Petri nets, automata, logic
circuits, and similar.  More abstractly, our notion of system encompasses
matrices and systems of differential equations. 

Drawing together these notions of system are well-developed diagrammatic
representations based on network diagrams---that is, topological graphs. We call
these network-style diagrammatic languages. In the abstract, by system we shall
simply mean that which can be represented by a box with a collection of terminals,
perhaps of different types, through which it interfaces with the surroundings.
Concretely, one might envision a circuit diagram with terminals, such as
\[
\begin{tikzpicture}[circuit ee IEC, set resistor graphic=var resistor IEC graphic]
\node[contact] (I1) at (0,0) {};
\node[contact] (O1) at (3,0) {};
\node (input) at (-2,0) {\small{\textsf{terminal}}};
\node (output) at (5,0) {\small{\textsf{terminal}}};
\draw (I1) 	to [resistor] node [above]
{$1\Omega$} (O1);
\path[color=gray, very thick, shorten >=10pt, ->, >=stealth] (input) edge (I1);	
\path[color=gray, very thick, shorten >=10pt, ->, >=stealth] (output) edge (O1);
\end{tikzpicture}
\]
or
\[
\begin{tikzpicture}[circuit ee IEC, set resistor graphic=var resistor IEC graphic]
\node[contact] (I1) at (0,2) {};
\node[contact] (I2) at (0,0) {};
\coordinate (int1) at (2.83,1) {};
\coordinate (int2) at (5.83,1) {};
\node[contact] (O1) at (8.66,2) {};
\node[contact] (O2) at (8.66,0) {};
\node (input) at (-2,1) {\small{\textsf{terminals}}};
\node (output) at (10.66,1) {\small{\textsf{terminals}}};
\draw (I1) 	to [resistor] node [label={[label distance=2pt]85:{$1\Omega$}}] {} (int1);
\draw (I2)	to [resistor] node [label={[label distance=2pt]275:{$1\Omega$}}] {} (int1)
				to [resistor] node [label={[label distance=3pt]90:{$2\Omega$}}] {} (int2);
\draw (int2) 	to [resistor] node [label={[label distance=2pt]95:{$1\Omega$}}] {} (O1);
\draw (int2)		to [resistor] node [label={[label distance=2pt]265:{$3\Omega$}}] {} (O2);
\path[color=gray, very thick, shorten >=10pt, ->, >=stealth, bend left] (input) edge (I1);		\path[color=gray, very thick, shorten >=10pt, ->, >=stealth, bend right] (input) edge (I2);		
\path[color=gray, very thick, shorten >=10pt, ->, >=stealth, bend right] (output) edge (O1);
\path[color=gray, very thick, shorten >=10pt, ->, >=stealth, bend left] (output) edge (O2);
\end{tikzpicture}
\]
The algebraic structure of interconnection is then simply the structure that
results from the ability to connect terminals of one system with terminals of
another. This graphical approach motivates our language of interconnection:
indeed, these diagrams will be the expressions of our language.

We claim that the existence of a network-style diagrammatic language to
represent a system implies that interconnection is inherently important in
understanding the system. Yet, while each of these example notions of system are
well-studied in and of themselves, their compositional, or algebraic, structure
has received scant attention. In this thesis, we study an algebraic structure
called a hypergraph category, and argue that this is the relevant algebraic
structure for modelling interconnection of open systems. 

Given these pre-existing diagrammatic formalisms and our visual intuition,
constructing algebras of syntactic representations is thus rather
straightforward. The semantics and their algebraic structure are more subtle. 

In some sense our semantics is already given to us too: in studying these
systems as closed systems, scientists have already formalised the meaning of
these diagrams. But we have shifted from a closed perspective to an open one,
and we need our semantics to also account for points of interconnection.

Taking inspiration from Willems' behavioural approach \cite{Wi} and Deutsch's
constructor theory \cite{Deu}, in this thesis I advocate the following position.
First, at each terminal of an open system we may make measurements appropriate
to the type of terminal. Given a collection of terminals, the \emph{universum}
is then the set of all possible measurement outcomes. Each open system has a
collection of terminals, and hence a universum. The semantics of an open system
is the subset of measurement outcomes on the terminals that are permitted by the
system. This is known as the \emph{behaviour} of the system.

For example, consider a resistor of resistance $r$. This has two terminals---the
two ends of the resistor---and at each terminal, we may measure the potential
and the current. Thus the universum of this system is the set
$\mathbb{R}\oplus\mathbb{R}\oplus\mathbb{R}\oplus\mathbb{R}$, where the summands
represent respectively the potentials and currents at each of the two terminals.
The resistor is governed by Kirchhoff's current law, or conservation of charge,
and Ohm's law.  Conservation of charge states that the current flowing into one
terminal must equal the current flowing out of the other terminal, while Ohm's
law states that this current will be proportional to the potential difference,
with constant of proportionality $1/r$. Thus the behaviour of the resistor is
the set 
\[
  \big\{\big(\phi_1,\phi_2,
    -\tfrac1r(\phi_2-\phi_1),\tfrac1r(\phi_2-\phi_1)\big)\,\big\vert\,
    \phi_1,\phi_2 \in \mathbb{R}\big\}.
\]
Note that in this perspective a law such as Ohm's law is a mechanism for
partitioning \emph{behaviours} into possible and impossible
behaviours.\footnote{That is, given a universum $\mathcal U$ of trajectories, a
  behaviour of a system is an element of the power set $\mathcal P(\mathcal
U)$ representing all possible measurements of this system, and a law or principle
is an element of $\mathcal P(\mathcal P(\mathcal U))$ representing all possible
behaviours of a class of systems.}

Interconnection of terminals then asserts the identification of the variables at
the identified terminals. Fixing some notion of open system and subsequently an
algebra of syntactic representations for these systems, our approach, based on
the principle of compositionality, requires this to define an algebra of
semantic objects and a homomorphism from syntax to semantics. The first part of
this thesis develops the mathematical tools necessary to pursue this vision for
modelling open systems and their interconnection. 

The next goal is to demonstrate the efficacy of this philosophy in applications.
At core, this work is done in the faith that the right language allows deeper
insight into the underlying structure. Indeed, after setting up such a language
for open systems there are many questions to be asked: Can we find a sound and
complete logic for determining when two syntactic expressions have the same
semantics? Suppose we have systems that have some property, for example
controllability.  In what ways can we interconnect controllable systems so that
the combined system is also controllable? Can we compute the semantics of a
large system quicker by computing the semantics of subsystems and then composing
them?  If I want a given system to acheive a specified trajectory, can we
interconnect another system to make it do so? How do two different notions of
system, such as circuit diagrams and signal flow graphs, relate to each other?
Can we find homomorphisms between their syntactic and semantic algebras? In the
second part of this thesis we explore some applications in depth, providing
answers to questions of the above sort.

\subsection*{Outline of this thesis}
%\addcontentsline{toc}{section}{Organisation of this thesis}

This thesis is divided into two parts. Part \ref{part.maths}, comprising
Chapters \ref{ch.hypcats} to \ref{ch.deccorels}, focusses on mathematical
foundations. In it we develop the theory of hypergraph categories and a powerful
tool for constructing and manipulating them: decorated corelations.
Part~\ref{part.apps}, comprising Chapters \ref{ch.sigflow} to \ref{ch.further},
then discusses applications of this theory to examples of open systems.

The central refrain of this thesis is that the syntax and semantics of
network-style diagrammatic languages can be modelled by hypergraph categories.
These are introduced in Chapter \ref{ch.hypcats}. Hypergraph categories are
symmetric monoidal categories in which every object is equipped with the
structure of a special commutative Frobenius monoid in a way compatible with the
monoidal product. As we will rely heavily on properties of monoidal categories,
their functors, and their graphical calculus, we begin with a whirlwind review
of these ideas. We then provide a definition of hypergraph categories and their
functors, a strictification theorem, and an important example: the category of
cospans in a category with finite colimits.

Cospans are pairs of morphisms $X \to N \leftarrow Y$ with a common codomain.
In Chapter \ref{ch.deccospans} we introduce the idea of a decorated cospan, which
equips the apex $N$ with extra structure. Our motivating example is cospans of
finite sets decorated by graphs, as in the picture
\begin{center}
  \begin{tikzpicture}[auto,scale=2.15]
    \node[circle,draw,inner sep=1pt,fill=gray,color=gray]         (x) at (-1.4,-.43) {};
    \node at (-1.4,-.9) {$X$};
    \node[circle,draw,inner sep=1pt,fill]         (A) at (0,0) {};
    \node[circle,draw,inner sep=1pt,fill]         (B) at (1,0) {};
    \node[circle,draw,inner sep=1pt,fill]         (C) at (0.5,-.86) {};
    \node[circle,draw,inner sep=1pt,fill=gray,color=gray]         (y1) at (2.4,-.25) {};
    \node[circle,draw,inner sep=1pt,fill=gray,color=gray]         (y2) at (2.4,-.61) {};
    \node at (2.4,-.9) {$Y$};
    \path (B) edge  [bend right,->-] node[above] {0.2} (A);
    \path (A) edge  [bend right,->-] node[below] {1.3} (B);
    \path (A) edge  [->-] node[left] {0.8} (C);
    \path (C) edge  [->-] node[right] {2.0} (B);
    \path[color=gray, very thick, shorten >=10pt, shorten <=5pt, ->, >=stealth] (x) edge (A);
    \path[color=gray, very thick, shorten >=10pt, shorten <=5pt, ->, >=stealth] (y1) edge (B);
    \path[color=gray, very thick, shorten >=10pt, shorten <=5pt, ->, >=stealth] (y2) edge (B);
  \end{tikzpicture}
\end{center}
Here graphs are a proxy for expressions in a network-style diagrammatic
language. To give a bit more formal detail, let $\mathcal C$ be a category with
finite colimits, writing its coproduct $+$, and let $(\mathcal D, \otimes)$ be a
braided monoidal category. Decorated cospans provide a method of producing a
hypergraph category from a lax braided monoidal functor $F\colon (\mathcal C,+)
\to (\mathcal D, \otimes)$. The objects of these categories are simply the
objects of $\mathcal C$, while the morphisms are pairs comprising a cospan $X
\rightarrow N \leftarrow Y$ in $\mathcal C$ together with an element $I \to FN$
in $\mathcal D$---the so-called decoration. We will also describe how to
construct hypergraph functors between decorated cospan categories. In
particular, this provides a useful tool for constructing a hypergraph category
that captures the syntax of a network-style diagrammatic language.

Having developed a method to construct a category where the morphisms are
expressions in a diagrammatic language, we turn our attention to categories of
semantics. This leads us to the notion of a corelation, to which we devote
Chapter \ref{ch.corelations}. Given a factorisation system $(\mc E,\mc M)$ on a
category $\mc C$, we define a corelation to be a cospan $X \to N \leftarrow Y$
such that the copairing of the two maps, a map $X+Y \to N$, is a morphism in
$\mc E$. Factorising maps $X+Y \to N$ using the factorisation system leads to a
notion of equivalence on cospans, and this helps us describe when two diagrams
are equivalent. Like cospans, corelations form hypergraph categories.

In Chapter \ref{ch.deccorels} we decorate corelations. Like decorated cospans,
decorated corelations are corelations together with some additional structure on
the apex. We again use a lax braided monoidal functor to specify the sorts of
extra structure allowed. Moreover, decorated corelations too form the morphisms
of a hypergraph category. The culmination of our theoretical work is to show
that every hypergraph category and every hypergraph functor can be constructed
using decorated corelations. This implies that we can use decorated corelations
to construct a semantic hypergraph category for any network-style diagrammatic
language, as well as a hypergraph functor from its syntactic category that
interprets each diagram. We also discuss how the intuitions behind decorated
corelations guide construction of these categories and functors.

%The key of decorated corelations is that hypergraph structure
%requires some sort of uniformity in the composition rule, and it is easier to
%work by acknowledging this structure, defining them as algebras over some
%theory. 
%
%Hypergraph categories are really just $\Set$-valued lax symmetric monoidal
%functors: and these, being simply data structures, are often simpler to work
%with.

Having developed these theoretical tools, in the second part we turn to
demonstrating that they have useful application. Chapter \ref{ch.sigflow}
uses corelations to formalise signal flow diagrams representing linear
time-invariant discrete dynamical systems as morphisms in a category.
Our main result gives an intuitive sound and fully complete equational theory
for reasoning about these linear time-invariant systems. Using this framework,
we derive a novel structural characterisation of controllability, and
consequently provide a methodology for analysing controllability of networked
and interconnected systems.

Chapter \ref{ch.circuits} studies passive linear networks. Passive linear
networks are used in a wide variety of engineering applications, but the best
studied are electrical circuits made of resistors, inductors and capacitors. The
goal is to construct what we call the black box functor, a hypergraph functor
from a category of open circuit diagrams to a category of behaviours of
circuits. We construct the former as a decorated cospan category, with each morphism
a cospan of finite sets decorated by a circuit diagram on the apex. In this
category, composition describes the process of attaching the outputs of one
circuit to the inputs of another. The behaviour of a circuit is the relation it
imposes between currents and potentials at their terminals. The space of these
currents and potentials naturally has the structure of a symplectic vector
space, and the relation imposed by a circuit is a Lagrangian linear relation.
Thus, the black box functor goes from our category of circuits to the category
of symplectic vector spaces and Lagrangian linear relations. Decorated
corelations provide a critical tool for constructing these hypergraph categories
and the black box functor.

Finally, in Chapter \ref{ch.further} we mention two further research directions.
The first is the idea of a bound colimit, which aims to describe why epi-mono
factorisation systems are useful for constructing corelation categories of
semantics for open systems. The second research direction pertains to
applications of the black box functor for passive linear networks, discussing
the work of Jekel on the inverse problem for electric circuits \cite{Jek} and
the work of Baez, Fong, and Pollard on open Markov processes \cite{BFP,
Pol16}.

\subsection*{Related work}
%\addcontentsline{toc}{section}{Related work}

The work here is underpinned not just by philosophical precedent, but by a rich
tradition in mathematics, physics, and computer science. Here we make some
remarks on work on our broad theme of categorical network theory; more specific
references are included in each chapter. 

With its emphasis on composition, category theory is an attractive framework for
modelling open systems, and the present work is not the first attempt at develop
this idea. Notably from as early as the 1960s Goguen and Rosen both led efforts,
Goguen from a computer science perspective \cite{Go}, Rosen from
biology \cite{Ros12}. Goguen in particular promoted the idea, which we take up,
that composition of systems should be modelled by colimits \cite{Gog91}. This
manifests in our emphasis on cospans.

Indeed, cospans are well known as a formalism for making entities with an
arbitrarily designated `input end' and `output end' into the morphisms of a
category. This has roots in topological quantum field theory, where a particular
type of cospan known as a `cobordism' is used to describe pieces of spacetime
\cite{BL,BaezStay}. The general idea of using functors to associate algebraic
semantics to topological diagrams has also long been a technique associated with
topological quantum field theory, dating back to \cite{At}.  

Developed around the same time, the work of Joyal and Street showing the tight
connection between string diagrams and monoidal categories \cite{JS91,JS93} is
also a critical aspect of the work here. This rigorous link inspires our use of
hypergraph categories---a type of monoidal category---in modelling network-style
diagrammatic languages.

The work of Walters, Sabadini, and collaborators also makes use of cospans to
model interconnection of systems, including electric circuits, as well as
automata, Petri nets, transition systems, and Markov processes \cite{KSW,KSW2,
RSW05,RSW08,ASW}. Indeed, the definition of a hypergraph category is due to
Walters and Carboni \cite{Car91}. A difference and original contribution is our
present use of decorations. One paper deserving of particular mention is that of
Rosebrugh, Sabadini, and Walters on calculating colimits compositionally
\cite{RSW08}, which develops the beginnings of some of the present ideas on
corelations. Their discussion of automata and the proof of Kleene's theorem is
reminscent of our proof of the functoriality of the black box functor, and the
precise relationship deserves to be pinned down.

This link between automata and control systems, however, is not novel: Arbib
was the first to lay out this vision \cite{Arb65}. In fact, the categorical
approach to this task has been central from the beginning. Following Goguen
\cite{Gog72}, Arbib and Manes unified state-space based control theory and
automata theory in their study of machines in a category, showing that examples
including sequential machines, linear machines, stochastic automata, and tree
automata and their theories of reachability, observability, and realisation all
fall under the same general categorical framework \cite{AM74a,AM74b,AM80}.

For the most part, however, the above work starts with categories in which the
systems are objects, whereas here our systems are morphisms. The idea that
systems or processes should be modelled by morphisms in a category, with the
objects becoming interface specifications and composition interaction of
processes, has roots in Abramsky's work on interaction categories
\cite{Abr93,Abr94,AGN95}.  Like those here, interaction categories are symmetric
monoidal categories, with the diagrams providing an intuitive language for
composition of systems. The tight link between this work and spans has also been
evident from the beginning \cite{CS94}.

More broadly, the work of this thesis, talking as it does of semantics for
diagrams of interacting systems, and particularly semantics in terms of
relations, is couched in the language of concurrency theory and process algebra
\cite{Bae05}. This stems from the 1980s, with Milner's calculus of communicating
systems \cite{Mil80} and Hoare's communicating sequential processes
\cite{Hoa78}.

%Schweimeier categorical models of programming languages based on a graphical
%presentation \cite{Sch01}.

At the present time, the work of Bonchi, Soboc\'inski, and Zanasi on signal flow
diagrams and graphical linear algebra \cite{BSZ,BSZ2,BSZ3,Za} has heavily
influenced the work here. Indeed, Chapter \ref{ch.sigflow} is built upon the
foundation they created.

Spivak and collaborators use operads to study networked systems
diagrammatically. Although a more general framework, capable of studying
diagrammatic languages with more flexible syntax than that of hypergraph
categories, Spivak studies an operad of so-called wiring diagrams built from
cospans of labelled finite sets in a number of papers, including \cite{VSL,Sp,SSR}.

Finally, this thesis fits into the categorical network theory programme led by
Baez, together with Erbele, Pollard, Courser and others. In particular, the
papers \cite{BE,Erb16,BFP,Pol16,Cou16} on signal flow diagrams, open Markov
processes, and decorated cospans have been developed alongside and influenced
the work here.

\subsection*{Statement of work}
%\addcontentsline{toc}{section}{Statement of work}

The Examination Schools make the following request:
\begin{quote}
\emph{Where some part of the thesis is not solely the work of the candidate or
has been carried out in collaboration with one or more persons, the candidate
shall submit a clear statement of the extent of his or her own contribution.}
\end{quote}
I address this now. 

The first four chapters, on hypergraph categories and decorated corelations, are
my own work. (Chapter \ref{ch.deccospans} has been previously published as
\cite{Fon15}.) The applications chapters were developed with collaborators. 

Chapter \ref{ch.sigflow} arises from a weekly seminar with Paolo Rapisarda and
Pawe\l\ Soboc\'inski at Southampton in the Spring of 2015. The text is a minor
adaptation of that in the paper \cite{FRS16}. For that paper I developed the
corelation formalism, providing a first draft. Pawe\l\ provided much expertise
in signal flow graphs, significantly revising the text and contributing the
section on operational semantics. Paolo contributed comparisons to classical
methods in control theory.  A number of anonymous referees contributed helpful
and detailed comments.

Chapter \ref{ch.circuits} is joint work with John Baez; the majority of the text
is taken from our paper \cite{BF}. For that paper John supplied writing on
Dirichlet forms and the principle of minimum power that became the second
section of Chapter \ref{ch.circuits}, as well as parts of the next two sections.
I produced a first draft of the rest of the paper. We collaboratively revised
the text for publication.

\subsection*{Acknowledgements}
%\addcontentsline{toc}{section}{Acknowledgements}
This work could not have been completed alone, and I am grateful for the
assistance of so many people, indeed far more than I list here. Foremost
deserving of mention are my supervisors John Baez, Bob Coecke, and Rob Ghrist.
John's influence on this work cannot be understated; he has been tireless in
patiently explaining new mathematics to me, in sharing his vision for this great
project, and in supporting, reading, promoting, and encouraging my work. Bob is
the reason I started this work in the first place, and with Rob they have both
provided fantastic research environments, a generous amount of freedom to pursue
my passions and, most importantly, encouragement and faith.

Jamie Vicary, too, has always been available for advice and encouragement. My
coauthors, including John, but also Brandon Coya, Hugo Nava-Kopp, Blake Pollard,
Paolo Rapisarda, Pawe\l\ Soboc\'inski, have been teachers in the process.  And
valuable comments and conversations on the present work have come, not only from
all of the above, but also from Samson Abramsky, Marcelo Fiore, Sam Staton,
Bernhard Reinke, Dan Marsden, Rashmi Kumar, Aleks Kissinger, Stefano Gogioso,
Jason Erbele, Omar Camarena, and a number of anonymous referees.

I am grateful for the generous support of a number of institutions.  This work
was supported by the Clarendon Fund, Hertford College, and the Queen Elizabeth
Scholarships, Oxford, and completed at the University of Oxford, the Centre for
Quantum Technologies, Singapore, and the University of Pennsylvania.

Finally, I thank my partner, Stephanie, and my family---Mum, Dad, Justin,
Calvin, Tania---for their love and patience.

\end{romanpages}

\pagestyle{fancy}
\part{Mathematical Foundations} \label{part.maths}

\chapter[Hypergraph categories: the algebra of interconnection]{Hypergraph
categories: the algebra of interconnection} \label{ch.hypcats}

In this chapter we introduce hypergraph categories, giving a definition,
coherence theorem, graphical language, and examples. 

The first section, \textsection\ref{sec.interconnection}, motivates hypergraph
categories as a structure that captures interconnection in network-style
diagrammatic languages.  In \textsection\ref{sec.smcs}, we then provide a rapid
introduction to symmetric monoidal categories, with emphasis on their graphical
calculus and how this can be used to model diagrams in a network-style
diagrammatic language. This leads to the introduction of hypergraph categories
in \textsection\ref{sec.hypergraphs}, where we also prove that every hypergraph
category is equivalent, as a hypergraph category, to a strict hypergraph
category.  We then conclude this chapter with an exploration of a fundamental
example of hypergraph categories: categories of cospans
(\textsection\ref{sec.cospans}).

We assume basic familiarity with category theory and symmetric monoidal
categories; although we give a sparse overview of the latter for reference. A
proper introduction to both can be found in Mac Lane \cite{Mac98}.

\section{The algebra of interconnection} \label{sec.interconnection}

Our aim is to algebraicise network diagrams. A network diagram is built from
pieces like so:
\[
  \begin{tikzpicture}
    \node [thick, circle, draw] (0) at (0, -0) {};
    \node [style=none] (1) at (-0.75, 1.5) {};
    \node [style=none] (2) at (-1.75, -0) {};
    \node [style=none] (3) at (-0.75, -1.75) {};
    \node [style=none] (4) at (1.25, -1.5) {};
    \node [style=none] (5) at (1.75, 0.75) {};
    \draw (0) to (5);
    \draw [dashed] (0) to (4);
    \draw (0) to (3);
    \draw [line width=2pt, draw=gray] (0) to (2);
    \draw (0) to (1);
  \end{tikzpicture}
\]
These represent open systems, concrete or abstract; for example a resistor, a
chemical reaction, or a linear transformation. The essential feature, for
openness and for networking, is that the system may have terminals, perhaps of
different types, each one depicted by a line radiating from the central body.
In the case of a resistor each terminal might represent a wire, for chemical
reactions a chemical species, for linear transformations a variable in the
domain or codomain.  Network diagrams are formed by connecting terminals of
systems to build larger systems.

A network-style diagrammatic language is a collection of network diagrams
together with the stipulation that if we take some of these network diagrams,
and connect terminals of the same type in any way we like, then we form
another diagram in the collection.  The point of this chapter is that hypergraph
categories provide a precise formalisation of network-style diagrammatic
languages.  

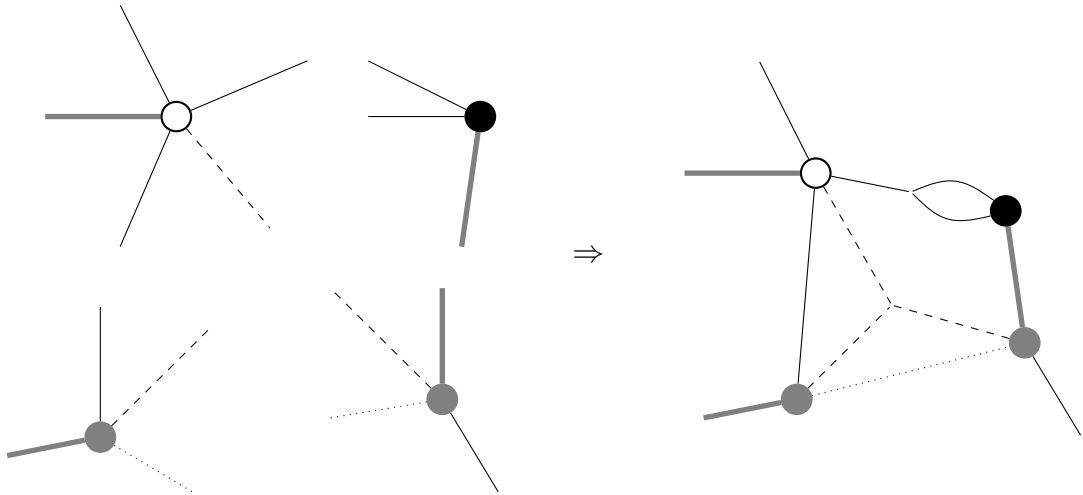
\begin{figure}
\[
\begin{aligned}
\begin{tikzpicture}
	\begin{pgfonlayer}{nodelayer}
		\node [thick, circle, draw] (0) at (0, -0) {};
		\node [style=none] (1) at (-0.75, 1.5) {};
		\node [style=none] (2) at (-1.75, -0) {};
		\node [style=none] (3) at (-0.75, -1.75) {};
		\node [style=none] (4) at (1.25, -1.5) {};
		\node [style=none] (5) at (1.75, 0.75) {};
		\node [thick, circle, draw=gray, fill=gray] (6) at (-1, -4.25) {};
		\node [style=none] (7) at (-1, -2.5) {};
		\node [style=none] (8) at (0.25, -5) {};
		\node [style=none] (9) at (-2.25, -4.5) {};
		\node [style=none] (10) at (0.5, -2.75) {};
		\node [thick, circle, draw=gray, fill=gray] (11) at (3.5, -3.75) {};
		\node [style=none] (12) at (3.5, -2.25) {};
		\node [style=none] (13) at (2, -4) {};
		\node [style=none] (14) at (4.25, -5) {};
		\node [thick, circle, draw, fill=black] (15) at (4, -0) {};
		\node [style=none] (16) at (2.5, 0.75) {};
		\node [style=none] (17) at (2.5, -0) {};
		\node [style=none] (18) at (3.75, -1.75) {};
		\node [style=none] (19) at (2, -2.25) {};
	\end{pgfonlayer}
	\begin{pgfonlayer}{edgelayer}
		\draw (0) to (5);
		\draw [dashed] (0) to (4);
		\draw (0) to (3);
		\draw [line width=2pt, draw=gray] (0) to (2);
		\draw (0) to (1);
		\draw (6) to (7);
		\draw [line width=2pt, draw=gray] (6) to (9);
		\draw [dashed] (6) to (10);
		\draw [dotted] (6) to (8);
		\draw [dotted] (11) to (13);
		\draw [line width=2pt, draw=gray] (11) to (12);
		\draw (11) to (14);
		\draw [line width=2pt, draw=gray] (15) to (18);
		\draw (15) to (17);
		\draw (15) to (16);
		\draw [dashed] (11) to (19);
	\end{pgfonlayer}
\end{tikzpicture}
\end{aligned}
\qquad
\Rightarrow
\qquad
\begin{aligned}
\begin{tikzpicture}
	\begin{pgfonlayer}{nodelayer}
		\node [draw, circle, thick] (0) at (0, -0) {};
		\node [style=none] (1) at (-0.75, 1.5) {};
		\node [style=none] (2) at (-1.75, -0) {};
		\node [style=none] (3) at (1.25, -0.25) {};
		\node [draw=gray, circle, thick, fill=gray] (4) at (-0.25, -3) {};
		\node [style=none] (5) at (1.25, -3.5) {};
		\node [style=none] (6) at (-1.5, -3.25) {};
		\node [draw=gray, circle, thick, fill=gray] (7) at (2.75, -2.25) {};
		\node [style=none] (8) at (3.5, -3.5) {};
		\node [draw, circle, thick, fill=black] (9) at (2.5, -0.5) {};
		\node [style=none] (10) at (1, -1.75) {};
	\end{pgfonlayer}
	\begin{pgfonlayer}{edgelayer}
		\draw (0) to (3);
		\draw [line width=2pt, draw=gray] (0) to (2);
		\draw (0) to (1);
		\draw [line width=2pt, draw=gray] (4) to (6);
		\draw [dotted] (4) to (7);
		\draw (7) to (8);
		\draw [dashed] (0) to (10);
		\draw [dashed] (4) to (10);
		\draw [dashed] (7) to (10);
		\draw [bend right, looseness=1.25] (9) to (3);
		\draw [bend left, looseness=1.25] (9) to (3);
		\draw [line width=2pt, draw=gray] (9) to (7);
		\draw (0) to (4);
	\end{pgfonlayer}
\end{tikzpicture}
\end{aligned}
\]
\caption{Interconnection of network diagrams. Note that we only connect
terminals of the same type, but we can connect as many as we like.}
\end{figure}

In jargon, a hypergraph category is a symmetric monoidal category in
which every object is equipped with a special commutative Frobenius monoid in a
way compatible with the monoidal product. We will walk through these terms in
detail, illustrating them with examples and a few theorems. 

The key data comprising a hypergraph category are its objects, morphisms,
composition rule, monoidal product, and Frobenius maps. Each of these model a
feature of network diagrams and their interconnection. The objects model the
terminal types, while the morphisms model the network diagrams themselves. The
composition, monoidal product, and Frobenius maps model different aspects of
interconnection: composition models the interconnection of two terminals of the
same type, the monoidal product models the network formed by taking two networks
without interconnecting any terminals, while the Frobenius maps model
multi-terminal interconnection.

These Frobenius maps are the distinguishing feature of hypergraph categories as
compared to other structured monoidal categories, and are crucial for
formalising the intuitive concept of network languages detailed above. In the
case of electric circuits the Frobenius maps model the `branching' of wires; in
the case when diagrams simply model an abstract system of equations and
terminals variables in these equations, the Frobenius maps allow variables to be
shared between many systems of equations.

Examining these correspondences, it is worthwhile to ask whether hypergraph
categories permit too much structure to be specified, given that interconnection
is now divided into three different aspects---composition, monoidal product, and
Frobenius maps---, and features such as domains and codomains of network
diagrams, rather than just a collection of terminals, exist. The answer is given
by examining the additional coherence laws that these data must obey. For
example, in the case of the domain and codomain, we shall see that hypergraph
categories are all compact closed categories, and so there is ultimately only a
formal distinction between domain and codomain objects. One way to think of
these data is as scaffolding. We could compare it to the use of matrices and
bases to provide language for talking about linear transformations and vector
spaces.  They are not part of the target structure, but nonetheless useful
paraphenalia for constructing it.

\begin{figure}
  \begin{center}
  \begin{tabular}{c|c}
    Networks & Hypergraph categories \\
    \hline 
    list of terminal types & object \\
    network diagram & morphism \\
    series connection & composition \\
    juxtaposition & monoidal product \\
    branching & Frobenius maps
  \end{tabular}
  \end{center}
  \caption{Corresponding features of networks and hypergraph categories.}
\end{figure}

Network languages are not only syntactic entities: as befitting the descriptor
`language', they typically have some associated semantics. Circuit diagrams, for
instance, not only depict wire circuits that may be constructed, they also
represent the electrical behaviour of that circuit. Such semantics considers the
circuits 
\[
  \begin{aligned}
  \begin{tikzpicture}[circuit ee IEC, set resistor graphic=var resistor IEC graphic]
    \node[contact] (I1) at (0,0) {};
    \node[circle, minimum width = 3pt, inner sep = 0pt, fill=black] (int) at (3,0) {};
    \node[contact] (O1) at (6,0) {};
    \draw (I1) 	to [resistor] node [label={[label distance=3pt]90:{$1 \Omega$}}] {} (int)
    to [resistor] node [label={[label distance=3pt]90:{$1 \Omega$}}] {} (O1);
  \end{tikzpicture}
  \end{aligned}
  \qquad
  \mbox{and}
  \qquad
  \begin{aligned}
  \begin{tikzpicture}[circuit ee IEC, set resistor graphic=var resistor IEC graphic]
    \node[contact] (I1) at (0,0) {};
    \node[contact] (O1) at (3,0) {};
    \draw (I1) 	to [resistor] node [label={[label distance=3pt]90:{$2 \Omega$}}]
    {} (O1);
  \end{tikzpicture}
  \end{aligned}
\]
the same, even though as `syntactic' diagrams they are distinct. A cornerstone
of the utility of the hypergraph formalism is the ability to also realise the
semantics of these diagrams as morphisms of another hypergraph category. This
`semantic' hypergraph category, as a hypergraph category, still permits the rich
`networking' interconnection structure, and a so-called hypergraph functor
implies that the syntactic category provides a sound framework for depicting
these morphisms. Network languages syntactically are often `free' hypergraph
categories, and much of the interesting structure lies in their functors to
their semantic hypergraph categories.

\section{Symmetric monoidal categories} \label{sec.smcs}
Suppose we have some tiles with inputs and outputs of various types like so:
\[
    \tikzset{every path/.style={line width=1.1pt}}
  \begin{tikzpicture}
	\begin{pgfonlayer}{nodelayer}
		\node [style=none] (0) at (-0.25, 0.375) {};
		\node [style=none] (1) at (0.5, 0.375) {};
		\node [style=none] (2) at (-0.25, -1.375) {};
		\node [style=none] (3) at (0.5, -1.375) {};
		\node [style=none] (4) at (0.5, 0.25) {};
		\node [style=none] (5) at (0.5, -1.25) {};
		\node [style=none] (6) at (1.25, 0.25) {};
		\node [style=none] (7) at (1.25, -1.25) {};
		\node [style=none] (8) at (0.125, -0.5) {$f$};
		\node [style=none] (9) at (1.5, 0.25) {$Y_1$};
		\node [style=none] (10) at (1.5, -1.25) {$Y_m$};
		\node [style=none] (11) at (1.25, -0.25) {};
		\node [style=none] (12) at (1.5, -0.25) {$Y_2$};
		\node [style=none] (13) at (0.5, -0.25) {};
		\node [style=none] (14) at (1, -0.75) {$\vdots$};
		\node [style=none] (15) at (-1, -1.25) {};
		\node [style=none] (16) at (-0.25, -1.25) {};
		\node [style=none] (17) at (-0.75, -0.75) {$\vdots$};
		\node [style=none] (18) at (-1.25, -1.25) {$X_n$};
		\node [style=none] (19) at (-0.25, -0.25) {};
		\node [style=none] (20) at (-1.25, 0.25) {$X_1$};
		\node [style=none] (21) at (-1, 0.25) {};
		\node [style=none] (22) at (-0.25, 0.25) {};
		\node [style=none] (23) at (-1, -0.25) {};
		\node [style=none] (24) at (-1.25, -0.25) {$X_2$};
	\end{pgfonlayer}
	\begin{pgfonlayer}{edgelayer}
		\draw (0.center) to (1.center);
		\draw (1.center) to (3.center);
		\draw (3.center) to (2.center);
		\draw (2.center) to (0.center);
		\draw (4.center) to (6.center);
		\draw (5.center) to (7.center);
		\draw (13.center) to (11.center);
		\draw (22.center) to (21.center);
		\draw (16.center) to (15.center);
		\draw (19.center) to (23.center);
	\end{pgfonlayer}
\end{tikzpicture}
\]
These tiles may vary in height and width. We can place these tiles above and
below each other, and to the left and right, so long as the inputs on the right
tile match the outputs on the left. Suppose also that some arrangements of tiles
are equal to other arrangements of tiles. How do we formalise this structure
algebraically? The theory of monoidal categories provides an answer. 

This is of relevance to us as hypergraph categories are first monoidal
categories, indeed symmetric monoidal categories.  Monoidal categories are
categories with two notions of composition: ordinary categorical composition and
monoidal composition, with the monoidal composition only associative and unital
up to natural isomorphism. They are the algebra of processes that may occur
simultaneously as well as sequentially.  First defined by B\'enabou and Mac Lane
in the 1960s \cite{Ben63, Mac63}, their theory and their links with graphical
representation are well explored \cite{JS91, Sel11}. We bootstrap on this, using
monoidal categories to define hypergraph categories, and so immediately arriving
at an understanding of how hypergraph categories formalise our network
languages. 

Moreover, symmetric monoidal functors play a key role in our framework for
defining and working with hypergraph categories: decorated cospans and
corelations. We thus use the next section to provide, for quick reference, a
definition of symmetric monoidal categories.

\subsection{Monoidal categories}
A \define{monoidal category} $(\mathcal C, \ot)$ consists of a category $\mathcal C$, a
functor $\ot: \mathcal C \times \mathcal C \to \mathcal C$, a distinguished object $I$, and natural
isomorphisms $\a_{A,B,C}: (A \ot B) \ot C \to A \ot (B \ot C)$,
$\rho_A: A \ot I  \to A$, and $\lambda_A: I \ot A \to A$ such that for all
$A,B,C,D$ in $\mc C$ the following two diagrams commute: 
\[
  \xymatrixcolsep{3pc}
  \xymatrix{
    \big((A \ot B) \ot C\big) \ot D \ar[d]_{\a_{A,B,C}\ot\idn_D} \ar[rr]^{\a_{(A\ot B),C,D}} 
    &&(A \ot B) \ot (C \ot D) \ar[d]^{\a_{A,B,(C\ot D)}} \\
    \big(A \ot (B\ot C)\big) \ot D \ar[r]_{\a_{A,(B\ot C),D}} 
    & A\ot\big((B \ot C)\ot D\big)\ar[r]_{\idn_A \ot \a_{B,C,D}}
    &A \ot \big(B \ot (C \ot D)\big)
  }
\]
\[
  \xymatrix{
    (A\ot I)\ot B  \ar[rr]^{\a_{A,I,B}} \ar[dr]_{\rho_{A}\ot \idn_B} && A \ot (I \ot B) \ar[dl]^{\idn_A\ot \lambda_B}\\
    & A \ot B \\
  }
\]
We call $\ot$ the \define{monoidal product}, $I$ the \define{monoidal unit},
$\alpha$ the \define{associator}, $\rho$ and $\lambda$ the \define{right} and
\define{left unitor} respectively. The associator and unitors are known
collectively as the \define{coherence maps}.

By Mac Lane's coherence theorem, these two axioms are equivalent to requiring
that `all formal diagrams'---that is, all diagrams in which the morphism are
built from identity morphisms and the coherence maps using composition and the
monoidal product---commute. Consequently, between any two products of the same
ordered list of objects up to instances of the monoidal unit, such as $((A \ot
I) \ot B) \ot C$ and $A \ot ((B \ot C) \ot (I \ot I))$, there is a unique
so-called \define{canonical} map. See Mac Lane \cite[Corollary of Theorem
VII.2.1]{Mac98} for a precise statement and proof.

A \define{lax monoidal functor} $(F, \varphi): (\mathcal C,\otimes) \to (\mathcal C',\boxtimes)$
between monoidal categories consists of a functor $F: \mathcal C \to \mathcal C'$, and natural
transformations $\varphi_{A,B}: FA \boxtimes FB \to F(A \ot B)$ and $\varphi_I:
I' \to FI$, such that for all $A,B,C \in \mathcal C$ the three diagrams
\[
  \xymatrixcolsep{4pc}
  \xymatrix{
    (FA \ot FB) \ot FC \ar[d]_{\a_{FA,FB,FC}} \ar[r]^{\varphi_{A,B} \ot \idn_{FC}} &
    F(A \ot B) \ot FC \ar[r]^{\varphi_{A\ot B,C}} & F((A \ot B) \ot C) \ar[d]^{F\a_{A,B,C}}\\
    FA \ot (FB \ot FC) \ar[r]_{\idn_{FA} \ot \varphi_{B,C}} & FA \ot F(B \ot C)
    \ar[r]_{\varphi_{A,B\ot C}} & F(A \ot (B \ot C))
  }
\]
\[
  \xymatrixcolsep{3pc}
  \xymatrixrowsep{3pc}
  \xymatrix{
    F(A) \ot I' \ar[d]_{\idn \ot \varphi_I} \ar[r]^{\rho} & F(A) \\
    F(A) \ot F(I) \ar[r]_{\varphi_{A,I}} & F(A \ot I) \ar[u]_{F\rho} 
  }
  \qquad
  \xymatrix{
    I' \ot F(A) \ar[d]_{\varphi_I \ot \idn} \ar[r]^{\lambda} & F(A) \\
    F(I) \ot F(A) \ar[r]_{\varphi_{I,A}} & F(I \ot A) \ar[u]_{F\lambda} 
  }
\]
commute. We further say a monoidal functor is a \define{strong monoidal functor}
if the $\varphi$ are isomorphisms, and a \define{strict monoidal functor} if the
$\varphi$ are identities. 

A \define{monoidal natural transformation} $\theta: (F,\varphi) \Rightarrow
(G,\gamma)$ between two monoidal functors $F$ and $G$ is a natural
transformation $\theta: F \Rightarrow G$ such that
\[
  \begin{aligned}
    \xymatrix{
      FI \ar[rr]^{\theta_I}&& GI \\
      & I' \ar[ul]^{\varphi_I} \ar[ur]_{\gamma_I}
    } 
  \end{aligned} 
  \qquad 
  \mbox{and}
  \qquad
  \begin{aligned}
    \xymatrixcolsep{3pc}
    \xymatrixrowsep{3pc}
    \xymatrix{
      FA \boxtimes FB \ar[r]^{\theta_A \ot \theta_B} \ar[d]_{\varphi_{A,B}} 
      & GA \boxtimes GB \ar[d]^{\gamma_{A,B}}\\
      F(A \ot B) \ar[r]_{\theta_{A\ot B}} & G(A \ot B)
    }
  \end{aligned} 
\]
commute for all objects $A,B$.

Two monoidal categories $\mc C, \mc D$ are \define{monoidally equivalent} if
there exist strong monoidal functors $F\maps \mc C \to \mc D$ and $G\maps \mc D
\to \mc C$ such that the composites $FG$ and $GF$ are monoidally naturally
isomorphic to the identity functors. (Note that identity functors are
immediately strict monoidal functors.)

\subsection{String diagrams}
A \define{strict monoidal category} category is a monoidal category in which the
associators and unitors are all identity maps. In this case then any two objects
that can be related by associators and unitors are equal, and so we may write
objects without parentheses and units without ambiguity. An equivalent statement
of Mac Lane's coherence theorem is that every monoidal category is monoidally
equivalent to strict monoidal category. 

Yet another equivalent statement of the coherence theorem is the existence of a
graphical calculus for monoidal categories. As discussed above, monoidal
categories figure strongly in our current investigations precisely because of
this. We leave the details to discussions elsewhere, such as Selinger's survey
\cite{Sel11} or the original work of Joyal and Street \cite[Theorem 1.2]{JS91}.
The main point is that we shall be free to assume our monoidal categories are
strict, writing $X_1 \otimes X_2 \otimes \dots \otimes X_n$ for objects in
$(\mathcal C,\otimes)$ without a care for parentheses. We then depict a morphism
$f\maps X_1 \otimes X_2 \otimes \dots \otimes X_n \to Y_1 \otimes Y_2 \otimes
\dots \otimes Y_n$ with the diagram:
\[
  f \quad = \quad
  \begin{aligned}
    \tikzset{every path/.style={line width=1.1pt}}
  \begin{tikzpicture}
	\begin{pgfonlayer}{nodelayer}
		\node [style=none] (0) at (-0.25, 0.375) {};
		\node [style=none] (1) at (0.5, 0.375) {};
		\node [style=none] (2) at (-0.25, -1.375) {};
		\node [style=none] (3) at (0.5, -1.375) {};
		\node [style=none] (4) at (0.5, 0.25) {};
		\node [style=none] (5) at (0.5, -1.25) {};
		\node [style=none] (6) at (1.25, 0.25) {};
		\node [style=none] (7) at (1.25, -1.25) {};
		\node [style=none] (8) at (0.125, -0.5) {$f$};
		\node [style=none] (9) at (1.5, 0.25) {$Y_1$};
		\node [style=none] (10) at (1.5, -1.25) {$Y_m$};
		\node [style=none] (11) at (1.25, -0.25) {};
		\node [style=none] (12) at (1.5, -0.25) {$Y_2$};
		\node [style=none] (13) at (0.5, -0.25) {};
		\node [style=none] (14) at (1, -0.75) {$\vdots$};
		\node [style=none] (15) at (-1, -1.25) {};
		\node [style=none] (16) at (-0.25, -1.25) {};
		\node [style=none] (17) at (-0.75, -0.75) {$\vdots$};
		\node [style=none] (18) at (-1.25, -1.25) {$X_n$};
		\node [style=none] (19) at (-0.25, -0.25) {};
		\node [style=none] (20) at (-1.25, 0.25) {$X_1$};
		\node [style=none] (21) at (-1, 0.25) {};
		\node [style=none] (22) at (-0.25, 0.25) {};
		\node [style=none] (23) at (-1, -0.25) {};
		\node [style=none] (24) at (-1.25, -0.25) {$X_2$};
	\end{pgfonlayer}
	\begin{pgfonlayer}{edgelayer}
		\draw (0.center) to (1.center);
		\draw (1.center) to (3.center);
		\draw (3.center) to (2.center);
		\draw (2.center) to (0.center);
		\draw (4.center) to (6.center);
		\draw (5.center) to (7.center);
		\draw (13.center) to (11.center);
		\draw (22.center) to (21.center);
		\draw (16.center) to (15.center);
		\draw (19.center) to (23.center);
	\end{pgfonlayer}
\end{tikzpicture}.
\end{aligned}
\]
Identity morphisms are depicted by so-called strings or wires:
\[
  \idn_X \quad = \quad
  \begin{aligned}
    \tikzset{every path/.style={line width=1.1pt}}
\begin{tikzpicture}
	\begin{pgfonlayer}{nodelayer}
		\node [style=none] (0) at (1.25, 0.25) {};
		\node [style=none] (1) at (1.5, 0.25) {$X$};
		\node [style=none] (2) at (-1.25, 0.25) {$X$};
		\node [style=none] (3) at (-1, 0.25) {};
	\end{pgfonlayer}
	\begin{pgfonlayer}{edgelayer}
		\draw (3.center) to (0.center);
	\end{pgfonlayer}
\end{tikzpicture}
\end{aligned}
\]
and the monoidal unit is not depicted at all:
\[
\idn_I\quad = \quad
  \begin{aligned}
    \tikzset{every path/.style={line width=1.1pt}}
\begin{tikzpicture}
		\node [style=none] (1) at (1.5, 0.25) {};
		\node [style=none] (2) at (-1.25, 0.25) {};
\end{tikzpicture}
\end{aligned}
\]
Composition of morphisms is depicted by connecting the relevant wires:
\[
    \tikzset{every path/.style={line width=1.1pt}}
  \begin{aligned}
    \begin{tikzpicture}
	\begin{pgfonlayer}{nodelayer}
		\node [style=none] (0) at (0.25, -0) {$Y_1$};
		\node [style=none] (1) at (0.5, -0) {};
		\node [style=none] (2) at (2.75, -0.75) {};
		\node [style=none] (3) at (3, -0.75) {$Y_2$};
		\node [style=none] (4) at (0.5, -0.75) {};
		\node [style=none] (5) at (2, -0.365) {};
		\node [style=none] (6) at (2.75, 0.25) {};
		\node [style=none] (7) at (3, 0.25) {$Z_1$};
		\node [style=none] (8) at (1.25, 0.375) {};
		\node [style=none] (9) at (1.25, -0) {};
		\node [style=none] (10) at (1.25, -0.365) {};
		\node [style=none] (11) at (2, 0.25) {};
		\node [style=none] (12) at (2, 0.375) {};
		\node [style=none] (13) at (1.625, -0) {$g$};
		\node [style=none] (14) at (2, -0.25) {};
		\node [style=none] (15) at (2.75, -0.25) {};
		\node [style=none] (16) at (3, -0.25) {$Z_2$};
		\node [style=none] (17) at (0.25, -0.75) {$Y_2$};
	\end{pgfonlayer}
	\begin{pgfonlayer}{edgelayer}
		\draw (4.center) to (2.center);
		\draw (8.center) to (12.center);
		\draw (5.center) to (10.center);
		\draw (11.center) to (6.center);
		\draw (12.center) to (5.center);
		\draw (10.center) to (8.center);
		\draw (14.center) to (15.center);
		\draw (1.center) to (9.center);
	\end{pgfonlayer}
\end{tikzpicture}
\end{aligned}
  \circ
  \begin{aligned}
  \begin{tikzpicture}
	\begin{pgfonlayer}{nodelayer}
		\node [style=none] (0) at (-0.25, 0.375) {};
		\node [style=none] (1) at (0.5, 0.375) {};
		\node [style=none] (2) at (-0.25, -.875) {};
		\node [style=none] (3) at (0.5, -.875) {};
		\node [style=none] (4) at (0.5, 0.125) {};
		\node [style=none] (5) at (1.25, 0.125) {};
		\node [style=none] (6) at (0.125, -0.25) {$f$};
		\node [style=none] (7) at (1.5, 0.125) {$Y_1$};
		\node [style=none] (8) at (1.25, -0.625) {};
		\node [style=none] (9) at (1.5, -0.625) {$Y_2$};
		\node [style=none] (10) at (0.5, -0.625) {};
		\node [style=none] (11) at (-0.25, -0.75) {};
		\node [style=none] (12) at (-1.25, -0.75) {$X_3$};
		\node [style=none] (13) at (-0.25, -0.25) {};
		\node [style=none] (14) at (-1.25, 0.25) {$X_1$};
		\node [style=none] (15) at (-1, 0.25) {};
		\node [style=none] (16) at (-0.25, 0.25) {};
		\node [style=none] (17) at (-1, -0.25) {};
		\node [style=none] (18) at (-1.25, -0.25) {$X_2$};
		\node [style=none] (19) at (-1, -0.75) {};
	\end{pgfonlayer}
	\begin{pgfonlayer}{edgelayer}
		\draw (0.center) to (1.center);
		\draw (1.center) to (3.center);
		\draw (3.center) to (2.center);
		\draw (2.center) to (0.center);
		\draw (4.center) to (5.center);
		\draw (10.center) to (8.center);
		\draw (16.center) to (15.center);
		\draw (13.center) to (17.center);
		\draw (11.center) to (19.center);
	\end{pgfonlayer}
\end{tikzpicture}
\end{aligned}
\quad = \quad
\begin{aligned}
\begin{tikzpicture}
	\begin{pgfonlayer}{nodelayer}
		\node [style=none] (0) at (-0.25, 0.375) {};
		\node [style=none] (1) at (0.5, 0.375) {};
		\node [style=none] (2) at (-0.25, -0.875) {};
		\node [style=none] (3) at (0.5, -0.875) {};
		\node [style=none] (4) at (0.5, 0.125) {};
		\node [style=none] (5) at (0.125, -0.25) {$f$};
		\node [style=none] (6) at (2.75, -0.625) {};
		\node [style=none] (7) at (3, -0.625) {$Y_2$};
		\node [style=none] (8) at (0.5, -0.625) {};
		\node [style=none] (9) at (-0.25, -0.75) {};
		\node [style=none] (10) at (-1.25, -0.75) {$X_3$};
		\node [style=none] (11) at (-0.25, -0.25) {};
		\node [style=none] (12) at (-1.25, 0.25) {$X_1$};
		\node [style=none] (13) at (-1, 0.25) {};
		\node [style=none] (14) at (-0.25, 0.25) {};
		\node [style=none] (15) at (-1, -0.25) {};
		\node [style=none] (16) at (-1.25, -0.25) {$X_2$};
		\node [style=none] (17) at (-1, -0.75) {};
		\node [style=none] (18) at (2, -0.25) {};
		\node [style=none] (19) at (2.75, 0.375) {};
		\node [style=none] (20) at (3, 0.375) {$Z_1$};
		\node [style=none] (21) at (1.25, 0.5) {};
		\node [style=none] (22) at (1.25, 0.125) {};
		\node [style=none] (23) at (1.25, -0.25) {};
		\node [style=none] (24) at (2, 0.375) {};
		\node [style=none] (25) at (2, 0.5) {};
		\node [style=none] (26) at (1.625, 0.125) {$g$};
		\node [style=none] (27) at (2, -0.125) {};
		\node [style=none] (28) at (2.75, -0.125) {};
		\node [style=none] (29) at (3, -0.125) {$Z_2$};
	\end{pgfonlayer}
	\begin{pgfonlayer}{edgelayer}
		\draw (0.center) to (1.center);
		\draw (1.center) to (3.center);
		\draw (3.center) to (2.center);
		\draw (2.center) to (0.center);
		\draw (8.center) to (6.center);
		\draw (14.center) to (13.center);
		\draw (11.center) to (15.center);
		\draw (9.center) to (17.center);
		\draw (21.center) to (25.center);
		\draw (18.center) to (23.center);
		\draw (24.center) to (19.center);
		\draw (25.center) to (18.center);
		\draw (23.center) to (21.center);
		\draw (27.center) to (28.center);
		\draw (4.center) to (22.center);
	\end{pgfonlayer}
\end{tikzpicture}
\end{aligned}
\]
while monoidal composition is just juxtaposition:
\[
    \tikzset{every path/.style={line width=1.1pt}}
  \begin{aligned}
    \begin{tikzpicture}
	\begin{pgfonlayer}{nodelayer}
		\node [style=none] (0) at (-0.25, 0.375) {};
		\node [style=none] (1) at (0.5, 0.375) {};
		\node [style=none] (2) at (-0.25, -0.375) {};
		\node [style=none] (3) at (0.5, -0.375) {};
		\node [style=none] (4) at (0.5, -0) {};
		\node [style=none] (5) at (1.25, -0) {};
		\node [style=none] (6) at (0.125, -0) {$h$};
		\node [style=none] (7) at (1.5, -0) {$Y_1$};
		\node [style=none] (8) at (-0.25, -0.25) {};
		\node [style=none] (9) at (-1.25, 0.25) {$X_1$};
		\node [style=none] (10) at (-1, 0.25) {};
		\node [style=none] (11) at (-0.25, 0.25) {};
		\node [style=none] (12) at (-1, -0.25) {};
		\node [style=none] (13) at (-1.25, -0.25) {$X_2$};
	\end{pgfonlayer}
	\begin{pgfonlayer}{edgelayer}
		\draw (0.center) to (1.center);
		\draw (3.center) to (2.center);
		\draw (2.center) to (0.center);
		\draw (4.center) to (5.center);
		\draw (11.center) to (10.center);
		\draw (8.center) to (12.center);
		\draw (1.center) to (3.center);
	\end{pgfonlayer}
\end{tikzpicture}
  \end{aligned}
  \ot \quad
  \begin{aligned}
    \begin{tikzpicture}
	\begin{pgfonlayer}{nodelayer}
		\node [style=none] (0) at (1.5, -0.75) {$Y_2$};
		\node [style=none] (1) at (-0.25, -1.375) {};
		\node [style=none] (2) at (1.25, -0.75) {};
		\node [style=none] (3) at (1.5, -1.25) {$Y_3$};
		\node [style=none] (4) at (1.25, -1.25) {};
		\node [style=none] (5) at (0.125, -1) {$k$};
		\node [style=none] (6) at (0.5, -0.75) {};
		\node [style=none] (7) at (0.5, -1.375) {};
		\node [style=none] (8) at (-0.25, -0.625) {};
		\node [style=none] (9) at (0.5, -1.25) {};
		\node [style=none] (10) at (0.5, -0.625) {};
	\end{pgfonlayer}
	\begin{pgfonlayer}{edgelayer}
		\draw (8.center) to (10.center);
		\draw (6.center) to (2.center);
		\draw (9.center) to (4.center);
		\draw (10.center) to (7.center);
		\draw (7.center) to (1.center);
		\draw (8.center) to (1.center);
	\end{pgfonlayer}
\end{tikzpicture}
  \end{aligned}
  \quad = \quad
  \begin{aligned}
    \begin{tikzpicture}
	\begin{pgfonlayer}{nodelayer}
		\node [style=none] (0) at (-0.25, 0.375) {};
		\node [style=none] (1) at (0.5, 0.375) {};
		\node [style=none] (2) at (-0.25, -0.375) {};
		\node [style=none] (3) at (0.5, -0.375) {};
		\node [style=none] (4) at (0.5, -0) {};
		\node [style=none] (5) at (1.25, -0) {};
		\node [style=none] (6) at (0.125, -0) {$h$};
		\node [style=none] (7) at (1.5, -0) {$Y_1$};
		\node [style=none] (8) at (-0.25, -0.25) {};
		\node [style=none] (9) at (-1.25, 0.25) {$X_1$};
		\node [style=none] (10) at (-1, 0.25) {};
		\node [style=none] (11) at (-0.25, 0.25) {};
		\node [style=none] (12) at (-1, -0.25) {};
		\node [style=none] (13) at (-1.25, -0.25) {$X_2$};
		\node [style=none] (14) at (0.5, -0.75) {};
		\node [style=none] (15) at (1.25, -0.75) {};
		\node [style=none] (16) at (0.125, -1) {$k$};
		\node [style=none] (17) at (0.5, -0.625) {};
		\node [style=none] (18) at (-0.25, -1.375) {};
		\node [style=none] (19) at (1.5, -1.25) {$Y_3$};
		\node [style=none] (20) at (0.5, -1.25) {};
		\node [style=none] (21) at (-0.25, -0.625) {};
		\node [style=none] (22) at (1.25, -1.25) {};
		\node [style=none] (23) at (1.5, -0.75) {$Y_2$};
		\node [style=none] (24) at (0.5, -1.375) {};
	\end{pgfonlayer}
	\begin{pgfonlayer}{edgelayer}
		\draw (0.center) to (1.center);
		\draw (3.center) to (2.center);
		\draw (2.center) to (0.center);
		\draw (4.center) to (5.center);
		\draw (11.center) to (10.center);
		\draw (8.center) to (12.center);
		\draw (21.center) to (17.center);
		\draw (14.center) to (15.center);
		\draw (20.center) to (22.center);
		\draw (1.center) to (3.center);
		\draw (21.center) to (18.center);
		\draw (17.center) to (24.center);
		\draw (24.center) to (18.center);
	\end{pgfonlayer}
\end{tikzpicture}
  \end{aligned}
\]

Only the `topology' of the diagrams matters: if two diagrams with the same
domain and codomain are equivalent up to isotopy, they represent the same
morphism.
On the other hand, two algebraic expressions might have the same diagrammatic
representation. For example, the equivalent diagrams
\[
    \tikzset{every path/.style={line width=1.1pt}}
\begin{aligned}
\begin{tikzpicture}
	\begin{pgfonlayer}{nodelayer}
		\node [style=none] (0) at (-0.25, 0.375) {};
		\node [style=none] (1) at (0.5, 0.375) {};
		\node [style=none] (2) at (-0.25, -0.875) {};
		\node [style=none] (3) at (0.5, -0.875) {};
		\node [style=none] (4) at (0.5, 0.125) {};
		\node [style=none] (5) at (0.125, -0.25) {$f$};
		\node [style=none] (6) at (2.75, -0.625) {};
		\node [style=none] (7) at (3, -0.625) {$Y_2$};
		\node [style=none] (8) at (0.5, -0.625) {};
		\node [style=none] (9) at (-0.25, -0.75) {};
		\node [style=none] (10) at (-1.25, -0.75) {$X_3$};
		\node [style=none] (11) at (-0.25, -0.25) {};
		\node [style=none] (12) at (-1.25, 0.25) {$X_1$};
		\node [style=none] (13) at (-1, 0.25) {};
		\node [style=none] (14) at (-0.25, 0.25) {};
		\node [style=none] (15) at (-1, -0.25) {};
		\node [style=none] (16) at (-1.25, -0.25) {$X_2$};
		\node [style=none] (17) at (-1, -0.75) {};
		\node [style=none] (18) at (2, -0.25) {};
		\node [style=none] (19) at (2.75, 0.375) {};
		\node [style=none] (20) at (3, 0.375) {$Z_1$};
		\node [style=none] (21) at (1.25, 0.5) {};
		\node [style=none] (22) at (1.25, 0.125) {};
		\node [style=none] (23) at (1.25, -0.25) {};
		\node [style=none] (24) at (2, 0.375) {};
		\node [style=none] (25) at (2, 0.5) {};
		\node [style=none] (26) at (1.625, 0.125) {$g$};
		\node [style=none] (27) at (2, -0.125) {};
		\node [style=none] (28) at (2.75, -0.125) {};
		\node [style=none] (29) at (3, -0.125) {$Z_2$};
		\node [style=none] (30) at (1.25, -1.75) {};
		\node [style=none] (31) at (2, -1.625) {};
		\node [style=none] (32) at (2.75, -1.125) {};
		\node [style=none] (33) at (3, -1.625) {$Y_3$};
		\node [style=none] (34) at (2.75, -1.625) {};
		\node [style=none] (35) at (1.625, -1.375) {$k$};
		\node [style=none] (36) at (3, -1.125) {$Y_2$};
		\node [style=none] (37) at (2, -1.125) {};
		\node [style=none] (38) at (2, -1.75) {};
		\node [style=none] (39) at (2, -1) {};
		\node [style=none] (40) at (1.25, -1) {};
	\end{pgfonlayer}
	\begin{pgfonlayer}{edgelayer}
		\draw (0.center) to (1.center);
		\draw (1.center) to (3.center);
		\draw (3.center) to (2.center);
		\draw (2.center) to (0.center);
		\draw (8.center) to (6.center);
		\draw (14.center) to (13.center);
		\draw (11.center) to (15.center);
		\draw (9.center) to (17.center);
		\draw (21.center) to (25.center);
		\draw (18.center) to (23.center);
		\draw (24.center) to (19.center);
		\draw (25.center) to (18.center);
		\draw (23.center) to (21.center);
		\draw (27.center) to (28.center);
		\draw (4.center) to (22.center);
		\draw (40.center) to (39.center);
		\draw (37.center) to (32.center);
		\draw (31.center) to (34.center);
		\draw (39.center) to (38.center);
		\draw (38.center) to (30.center);
		\draw (40.center) to (30.center);
	\end{pgfonlayer}
\end{tikzpicture}
\end{aligned}
=
\begin{aligned}
  \begin{tikzpicture}
	\begin{pgfonlayer}{nodelayer}
		\node [style=none] (0) at (0, 1) {};
		\node [style=none] (1) at (1, 1) {};
		\node [style=none] (2) at (0, -1.25) {};
		\node [style=none] (3) at (1, -1.25) {};
		\node [style=none] (4) at (1, -0.25) {};
		\node [style=none] (5) at (0.5, -0) {$f$};
		\node [style=none] (6) at (2.75, -0.5) {};
		\node [style=none] (7) at (3, -0.5) {$Y_2$};
		\node [style=none] (8) at (1, -0.75) {};
		\node [style=none] (9) at (0, -0.75) {};
		\node [style=none] (10) at (-1.25, -0.75) {$X_3$};
		\node [style=none] (11) at (0, -0.25) {};
		\node [style=none] (12) at (-1.25, 0.25) {$X_1$};
		\node [style=none] (13) at (-1, 0.25) {};
		\node [style=none] (14) at (0, 0.5) {};
		\node [style=none] (15) at (-1, -0.25) {};
		\node [style=none] (16) at (-1.25, -0.25) {$X_2$};
		\node [style=none] (17) at (-1, -0.75) {};
		\node [style=none] (18) at (2, -0.25) {};
		\node [style=none] (19) at (2.75, 0.375) {};
		\node [style=none] (20) at (3, 0.375) {$Z_1$};
		\node [style=none] (21) at (1.5, 0.25) {};
		\node [style=none] (22) at (1.375, -0) {};
		\node [style=none] (23) at (1.25, -0.25) {};
		\node [style=none] (24) at (2, 0.375) {};
		\node [style=none] (25) at (2, 0.5) {};
		\node [style=none] (26) at (1.75, -0) {$g$};
		\node [style=none] (27) at (2, -0.125) {};
		\node [style=none] (28) at (2.75, -0.125) {};
		\node [style=none] (29) at (3, -0.125) {$Z_2$};
		\node [style=none] (30) at (-0.5, -2.5) {};
		\node [style=none] (31) at (0.25, -1.75) {};
		\node [style=none] (32) at (2.75, -1.5) {};
		\node [style=none] (33) at (3, -1.75) {$Y_3$};
		\node [style=none] (34) at (2.75, -1.75) {};
		\node [style=none] (35) at (-0.125, -1.75) {$k$};
		\node [style=none] (36) at (3, -1.5) {$Y_2$};
		\node [style=none] (37) at (0.25, -1.5) {};
		\node [style=none] (38) at (0.25, -2.25) {};
		\node [style=none] (39) at (0.25, -1.375) {};
		\node [style=none] (40) at (-0.5, -1.375) {};
	\end{pgfonlayer}
	\begin{pgfonlayer}{edgelayer}
		\draw (0.center) to (1.center);
		\draw (1.center) to (3.center);
		\draw (3.center) to (2.center);
		\draw (2.center) to (0.center);
		\draw (8.center) to (6.center);
		\draw (14.center) to (13.center);
		\draw (11.center) to (15.center);
		\draw (9.center) to (17.center);
		\draw (21.center) to (25.center);
		\draw (18.center) to (23.center);
		\draw (24.center) to (19.center);
		\draw (25.center) to (18.center);
		\draw (23.center) to (21.center);
		\draw (27.center) to (28.center);
		\draw (4.center) to (22.center);
		\draw (40.center) to (39.center);
		\draw (37.center) to (32.center);
		\draw (31.center) to (34.center);
		\draw (39.center) to (38.center);
		\draw (38.center) to (30.center);
		\draw (40.center) to (30.center);
	\end{pgfonlayer}
\end{tikzpicture}
\end{aligned}
\]
read as all of the equivalent algebraic expressions 
\[
  ((g \ot \idn_{Y_2}) \ot k) \circ f = 
  (g \ot ((\idn_{Y_2} \ot k)) \circ \rho \circ (f \ot \idn_I) = 
  (g \ot \idn_{Y_2}) \circ f \circ (\idn_{X_1 \ot (X_2 \ot X_3)} \ot k) 
\]
and so on. The coherence theorem says that this does not matter: if two
algebraic expressions have the same diagrammatic representation, then the
algebraic expressions are equal. In more formal language, the graphical calculus
is sound and complete for the axioms of monoidal categories. Again, see
Joyal--Street for details \cite{JS91}.

The coherence theorem thus implies that the graphical calculi goes beyond
visualisations of morphisms: it can provide provide bona-fide proofs of
equalities of morphisms. As a general principle, string diagrams are often more
intuitive than the conventional algebraic language for understanding monoidal
categories.

\subsection{Symmetry}
A symmetric braiding in a monoidal category provides the ability to permute
objects or, equivalently, cross wires. We define symmetric monoidal categories
making use of the graphical notation outlined above, but introducing a new,
special symbol $\swap{.04\textwidth}$.

A \define{symmetric monoidal category} is a monoidal category $(\mc C,\ot)$
together with natural isomorphisms 
\[
    \tikzset{every path/.style={line width=1.1pt}}
    \xymatrixrowsep{0pt}
  \xymatrix{
\begin{tikzpicture}
	\begin{pgfonlayer}{nodelayer}
		\node [style=none] (0) at (1, 0.25) {$B$};
		\node [style=none] (1) at (0.5, -0.25) {};
		\node [style=none] (2) at (-1, 0.25) {$A$};
		\node [style=none] (3) at (0.5, 0.25) {};
		\node [style=none] (4) at (-0.5, -0.25) {};
		\node [style=none] (5) at (-1, -0.25) {$B$};
		\node [style=none] (6) at (1, -0.25) {$A$};
		\node [style=none] (7) at (-0.5, 0.25) {};
		\node [style=none] (8) at (-0.75, 0.25) {};
		\node [style=none] (9) at (-0.75, -0.25) {};
		\node [style=none] (10) at (0.75, 0.25) {};
		\node [style=none] (11) at (0.75, -0.25) {};
	\end{pgfonlayer}
	\begin{pgfonlayer}{edgelayer}
		\draw [in=180, out=0, looseness=1.00] (7.center) to (1.center);
		\draw [in=180, out=0, looseness=1.00] (4.center) to (3.center);
		\draw (7.center) to (8.center);
		\draw (4.center) to (9.center);
		\draw (3.center) to (10.center);
		\draw (1.center) to (11.center);
	\end{pgfonlayer}
\end{tikzpicture}
    \\
    \s_{A,B}: A \ot B \to B \ot A
  }
\]
such that
\[
    \tikzset{every path/.style={line width=1.1pt}}
  \begin{aligned}
    \begin{tikzpicture}
	\begin{pgfonlayer}{nodelayer}
		\node [style=none] (0) at (2, -0.25) {$B$};
		\node [style=none] (1) at (0.5, -0.25) {};
		\node [style=none] (2) at (-1, 0.25) {$A$};
		\node [style=none] (3) at (0.5, 0.25) {};
		\node [style=none] (4) at (-0.5, -0.25) {};
		\node [style=none] (5) at (-1, -0.25) {$B$};
		\node [style=none] (6) at (2, 0.25) {$A$};
		\node [style=none] (7) at (-0.5, 0.25) {};
		\node [style=none] (8) at (-0.75, 0.25) {};
		\node [style=none] (9) at (-0.75, -0.25) {};
		\node [style=none] (10) at (1.75, 0.25) {};
		\node [style=none] (11) at (1.75, -0.25) {};
		\node [style=none] (12) at (1.5, 0.25) {};
		\node [style=none] (13) at (1.5, -0.25) {};
	\end{pgfonlayer}
	\begin{pgfonlayer}{edgelayer}
		\draw [in=180, out=0, looseness=1.00] (7.center) to (1.center);
		\draw [in=180, out=0, looseness=1.00] (4.center) to (3.center);
		\draw (7.center) to (8.center);
		\draw (4.center) to (9.center);
		\draw [in=180, out=0, looseness=1.00] (3.center) to (13.center);
		\draw [in=180, out=-3, looseness=1.00] (1.center) to (12.center);
		\draw (13.center) to (11.center);
		\draw (12.center) to (10.center);
	\end{pgfonlayer}
\end{tikzpicture}
  \end{aligned}
\quad = \quad
  \begin{aligned}
\begin{tikzpicture}
	\begin{pgfonlayer}{nodelayer}
		\node [style=none] (0) at (1, -0.25) {$B$};
		\node [style=none] (1) at (-1, 0.25) {$A$};
		\node [style=none] (2) at (0.75, -0.25) {};
		\node [style=none] (3) at (-1, -0.25) {$B$};
		\node [style=none] (4) at (1, 0.25) {$A$};
		\node [style=none] (5) at (0.75, 0.25) {};
		\node [style=none] (6) at (-0.75, 0.25) {};
		\node [style=none] (7) at (-0.75, -0.25) {};
	\end{pgfonlayer}
	\begin{pgfonlayer}{edgelayer}
		\draw (5.center) to (6.center);
		\draw (2.center) to (7.center);
	\end{pgfonlayer}
\end{tikzpicture}
  \end{aligned}
\]
and
\[
    \tikzset{every path/.style={line width=1.1pt}}
  \begin{aligned}
    \begin{tikzpicture}
	\begin{pgfonlayer}{nodelayer}
		\node [style=none] (0) at (2, -0.25) {$C$};
		\node [style=none] (1) at (-1, 0.25) {$A$};
		\node [style=none] (2) at (-0.5, -0.25) {};
		\node [style=none] (3) at (-1, -0.25) {$B$};
		\node [style=none] (4) at (2, 0.25) {$B$};
		\node [style=none] (5) at (-0.5, 0.25) {};
		\node [style=none] (6) at (-0.75, 0.25) {};
		\node [style=none] (7) at (-0.75, -0.25) {};
		\node [style=none] (8) at (1.75, -0.25) {};
		\node [style=none] (9) at (1.75, 0.25) {};
		\node [style=none] (10) at (0.5, 0.25) {};
		\node [style=none] (11) at (1.5, -0.25) {};
		\node [style=none] (12) at (0.5, -0.25) {};
		\node [style=none] (13) at (-1, -0.75) {$C$};
		\node [style=none] (14) at (-0.75, -0.75) {};
		\node [style=none] (15) at (-0.5, -0.75) {};
		\node [style=none] (16) at (0.5, -0.75) {};
		\node [style=none] (17) at (1.5, -0.75) {};
		\node [style=none] (18) at (1.75, -0.75) {};
		\node [style=none] (19) at (2, -0.75) {$A$};
	\end{pgfonlayer}
	\begin{pgfonlayer}{edgelayer}
		\draw (11.center) to (8.center);
		\draw [in=180, out=0, looseness=1.00] (2.center) to (10.center);
		\draw [in=180, out=0, looseness=1.00] (5.center) to (12.center);
		\draw (10.center) to (9.center);
		\draw (14.center) to (15.center);
		\draw (15.center) to (16.center);
		\draw [in=180, out=0, looseness=1.00] (16.center) to (11.center);
		\draw [in=180, out=0, looseness=1.00] (12.center) to (17.center);
		\draw (17.center) to (18.center);
		\draw (6.center) to (5.center);
		\draw (2.center) to (7.center);
	\end{pgfonlayer}
\end{tikzpicture}
  \end{aligned}
\quad = \quad
  \begin{aligned}
\begin{tikzpicture}
	\begin{pgfonlayer}{nodelayer}
		\node [style=none] (0) at (1.35, 0.25) {$B\ot C$};
		\node [style=none] (1) at (0.5, -0.25) {};
		\node [style=none] (2) at (-1, 0.25) {$A$};
		\node [style=none] (3) at (0.5, 0.25) {};
		\node [style=none] (4) at (-0.5, -0.25) {};
		\node [style=none] (5) at (-1.35, -0.25) {$B\ot C$};
		\node [style=none] (6) at (1, -0.25) {$A$};
		\node [style=none] (7) at (-0.5, 0.25) {};
		\node [style=none] (8) at (-0.75, 0.25) {};
		\node [style=none] (9) at (-0.75, -0.25) {};
		\node [style=none] (10) at (0.75, 0.25) {};
		\node [style=none] (11) at (0.75, -0.25) {};
	\end{pgfonlayer}
	\begin{pgfonlayer}{edgelayer}
		\draw [in=180, out=0, looseness=1.00] (7.center) to (1.center);
		\draw [in=180, out=0, looseness=1.00] (4.center) to (3.center);
		\draw (7.center) to (8.center);
		\draw (4.center) to (9.center);
		\draw (3.center) to (10.center);
		\draw (1.center) to (11.center);
	\end{pgfonlayer}
\end{tikzpicture}
  \end{aligned}
\]
for all $A,B,C$ in $\mc C$.  We call $\s$ the \define{braiding}. We will also
talk, somewhat incidentally, of braided monoidal categories in the next
chapter; a \define{braided monoidal category} is a monoidal category with a
braiding that only obeys the latter axiom.

A \define{(lax/strong/strict) symmetric monoidal functor} is a
(lax/strong/strict) monoidal functor that further obeys
\[
\xymatrixcolsep{3pc}
\xymatrixrowsep{3pc}
\xymatrix{
FA \ot FB \ar[r]^{\varphi_{A,B}} \ar[d]_{\s'_{FA,FB}} & F(A \ot B)\ar[d]^{F\s_{A,B}}\\
FB \ot FA \ar[r]_{\varphi_{B,A}} & F(B \ot A)
}
\]
Morphisms between symmetric monoidal functors are simply monoidal natural
transformations. Thus two symmetric monoidal categories are \define{symmetric
monoidally equivalent} if they are monoidally equivalent by strong
\emph{symmetric} monoidal functors. If our categories are merely braided, we
refer to these functors as \define{braided monoidal functors}.

Roughly speaking, the coherence theorem for symmetric monoidal categories, with
respect to string diagrams, states that two morphisms in a symmetric monoidal
category are equal according to the axioms of symmetric monoidal categories if
and only if their diagrams are equal up to isotopy and applications of the
defining graphical identities above. See Joyal--Street \cite[Theorem
2.3]{JS91} for more precision and details.

\section{Hypergraph categories} \label{sec.hypergraphs}

Just as symmetric monoidal categories equip monoidal categories with precisely
enough extra structure to model crossing of strings in the graphical calculus,
hypergraph categories equip symmetric monoidal categories with precisely enough
extra structure to model multi-input multi-output interconnections of strings of
the same type. For this, we require each object to be equipped with a so-called
special commutative Frobenius monoid, which provides chosen maps to model this
interaction. These have a coherence result, known as the `spider theorem', that
says exactly how we use these chosen maps to describe the connection of strings
does not matter: all that matters is that the strings are connected. 

\subsection{Frobenius monoids}
A Frobenius monoid comprises a monoid and comonoid on the same object that
interact according to the so-called Frobenius law.
\begin{definition}
  A \define{special commutative Frobenius monoid} $(X,\mu,\eta,\delta,\epsilon)$
  in a symmetric monoidal category $(\mathcal C, \otimes)$ is an object $X$ of
  $\mathcal C$ together with maps 
\[
  \xymatrixrowsep{1pt}
  \xymatrixcolsep{20pt}
  \xymatrix{
    \mult{.075\textwidth} & & \unit{.075\textwidth} & & 
    \comult{.075\textwidth} & & \counit{.075\textwidth} \\
    \mu\maps X\otimes X \to X & & \eta\maps I \to X & & 
    \delta\maps X\to X \otimes X & & \epsilon\maps X \to I
  }
\]
obeying the commutative monoid axioms
\[
  \xymatrixrowsep{1pt}
  \xymatrixcolsep{25pt}
  \xymatrix{
    \assocl{.1\textwidth} = \assocr{.1\textwidth} & \unitl{.1\textwidth} =
    \idone{.1\textwidth} & \commute{.1\textwidth} = \mult{.07\textwidth} \\
    \textrm{(associativity)} & \textrm{(unitality)} & \textrm{(commutativity)}
  }
\]
the cocommutative comonoid axioms
\[
  \xymatrixrowsep{1pt}
  \xymatrixcolsep{25pt}
  \xymatrix{
    \coassocl{.1\textwidth} = \coassocr{.1\textwidth} & \counitl{.1\textwidth} =
    \idone{.1\textwidth} & \cocommute{.1\textwidth} = \comult{.07\textwidth} \\
    \textrm{(coassociativity)} & \textrm{(counitality)} &
    \textrm{(cocommutativity)}
  }
\]
and the Frobenius and special axioms
  \[
  \xymatrixrowsep{1pt}
  \xymatrixcolsep{25pt}
  \xymatrix{
    \frobs{.1\textwidth} = \frobx{.1\textwidth} = \frobz{.1\textwidth} & \spec{.1\textwidth} =
    \idone{.1\textwidth} \\
    \textrm{(Frobenius)} & \textrm{(special)} 
  }
  \]
  We call $\mu$ the \define{multiplication}, $\eta$ the \define{unit}, $\delta$
  the \define{comultiplication}, and $\epsilon$ the \define{counit}.
\end{definition}

Arising from work in representation theory \cite{BN37}, special commutative
Frobenius monoids were first formulated in this categorical form by Carboni and
Walters, under the name commutative separable algebras \cite{CW87}. The
Frobenius law and the special law were termed the S=X law and the diamond=1 law
respectively \cite{CW87,RSW05}.

Alternate axiomatisations are possible. In addition to the `upper' unitality law
above, the mirror image `lower' unitality law also holds, due to commutativity
and the naturality of the braiding. While we write two equations for the
Frobenius law, this is redundant: given the other axioms, the equality of any
two of the diagrams implies the equality of all three.  Further, note that a
monoid and comonoid obeying the Frobenius law is commutative if and only if it
is cocommutative. Thus while a commutative and cocommutative Frobenius monoid
might more properly be called a bicommutative Frobenius monoid, there is no
ambiguity if we only say commutative.

The common feature to these equations is that each side describes a different
way of using the generators to connect some chosen set of inputs to some chosen
set of outputs. This observation provides a coherence type result for special
commutative Frobenius monoids, known as the spider theorem.
\begin{theorem}
  Let $(X,\mu,\eta,\delta,\epsilon)$ be a special commutative Frobenius monoid,
  and let $f,g\maps X^{\ot n} \to X^{\ot m}$ be map constructed, using
  composition and the monoidal product, from $\mu$, $\eta$, $\delta$, $\epsilon$, the
  coherence maps and braiding, and the identity map on $X$. Then $f$ and $g$ are
  equal if and only if given their string diagrams in the above notation, there
  exists a bijection between the connected components of the two diagrams such
  that corresponding connected components connect the exact same sets of inputs
  and outputs.
\end{theorem}%
A corollary of Lack's work on distributive laws for props \cite{Lac04}, this
theorem was first explicitly formulated in the context of categorical quantum
mechanics.  See the textbook of Coecke and Kissinger \cite{CK} or paper by
Coecke, Paquette, and Pavlovic \cite{CPP} for further details.

\subsection{Hypergraph categories}

\begin{definition}
  A \define{hypergraph category} is a symmetric monoidal category $(\mc H,\ot)$
  in which each object $X$ is equipped with a special commutative Frobenius
  structure $(X,\mu_X,\delta_X,\eta_X,\epsilon_X)$ such that 
  \[
    \tikzset{every path/.style={line width=1.1pt}}
    \xymatrixcolsep{1.8ex}
    \xymatrixrowsep{1ex}
    \xymatrix{
    \begin{aligned}
      \begin{tikzpicture}[scale=.65]
	\begin{pgfonlayer}{nodelayer}
		\node [style=none] (0) at (-0.25, 0.5) {};
		\node [style=dot] (1) at (0.5, -0) {};
		\node [style=none] (2) at (-0.25, -0.5) {};
		\node [style=none] (3) at (1.25, -0) {};
		\node [style=none] (4) at (-1.5, 0.5) {$X \otimes Y$};
		\node [style=none] (5) at (-1.5, -0.5) {$X \otimes Y$};
		\node [style=none] (6) at (2, -0) {$X \otimes Y$};
		\node [style=none] (7) at (-0.75, 0.5) {};
		\node [style=none] (8) at (-0.75, -0.5) {};
	\end{pgfonlayer}
	\begin{pgfonlayer}{edgelayer}
		\draw [in=90, out=0, looseness=0.90] (0.center) to (1.center);
		\draw [in=-90, out=0, looseness=0.90] (2.center) to (1.center);
		\draw (1.center) to (3.center);
		\draw (7.center) to (0.center);
		\draw (8.center) to (2.center);
	\end{pgfonlayer}
\end{tikzpicture}
\end{aligned}
=
\begin{aligned}
  \begin{tikzpicture}[scale=.65]
	\begin{pgfonlayer}{nodelayer}
		\node [style=none] (0) at (-0.25, 0.5) {};
		\node [style=dot] (1) at (0.5, -0) {};
		\node [style=none] (2) at (-0.25, -0.5) {};
		\node [style=none] (3) at (1.25, -0) {};
		\node [style=none] (4) at (-1.75, 0.5) {$X$};
		\node [style=none] (5) at (-1.75, -1.25) {$X$};
		\node [style=none] (6) at (1.5, -0) {$X$};
		\node [style=none] (7) at (-1.5, 0.5) {};
		\node [style=none] (8) at (-1.5, -1.25) {};
		\node [style=none] (9) at (1.25, -1.75) {};
		\node [style=none] (10) at (-1.5, -2.25) {};
		\node [style=none] (11) at (1.5, -1.75) {$Y$};
		\node [style=none] (12) at (-1.5, -0.5) {};
		\node [style=none] (13) at (-0.25, -2.25) {};
		\node [style=dot] (14) at (0.5, -1.75) {};
		\node [style=none] (15) at (-0.25, -1.25) {};
		\node [style=none] (16) at (-1.75, -0.5) {$Y$};
		\node [style=none] (17) at (-1.75, -2.25) {$Y$};
	\end{pgfonlayer}
	\begin{pgfonlayer}{edgelayer}
		\draw [in=90, out=0, looseness=0.90] (0.center) to (1.center);
		\draw [in=-90, out=0, looseness=0.90] (2.center) to (1.center);
		\draw (1.center) to (3.center);
		\draw (7.center) to (0.center);
		\draw [in=180, out=0, looseness=1.00] (8.center) to (2.center);
		\draw [in=90, out=0, looseness=0.90] (15.center) to (14);
		\draw [in=-90, out=0, looseness=0.90] (13.center) to (14);
		\draw (14) to (9.center);
		\draw [in=180, out=0, looseness=1.00] (12.center) to (15.center);
		\draw (10.center) to (13.center);
	\end{pgfonlayer}
\end{tikzpicture}
\end{aligned}
& &   
\qquad
  \begin{aligned}
    \begin{tikzpicture}[scale=.65]
	\begin{pgfonlayer}{nodelayer}
		\node [style=dot] (0) at (-0.5, 0.5) {};
		\node [style=none] (1) at (1.5, 0.5) {$X\otimes Y$};
		\node [style=none] (2) at (0.75, 0.5) {};
	\end{pgfonlayer}
	\begin{pgfonlayer}{edgelayer}
		\draw (2.center) to (0.center);
	\end{pgfonlayer}
\end{tikzpicture}
  \end{aligned}
  = \qquad
  \begin{aligned}
    \begin{tikzpicture}[scale=.65]
	\begin{pgfonlayer}{nodelayer}
		\node [style=dot] (0) at (0, 0.5) {};
		\node [style=none] (1) at (1.5, 0.5) {$X$};
		\node [style=none] (2) at (1.25, 0.5) {};
		\node [style=none] (3) at (1.25, -0.25) {};
		\node [style=dot] (4) at (0, -0.25) {};
		\node [style=none] (5) at (1.5, -0.25) {$Y$};
	\end{pgfonlayer}
	\begin{pgfonlayer}{edgelayer}
		\draw (2.center) to (0.center);
		\draw (3.center) to (4.center);
	\end{pgfonlayer}
\end{tikzpicture}
  \end{aligned}
\\
    \begin{aligned}
\begin{tikzpicture}[scale=.65]
	\begin{pgfonlayer}{nodelayer}
		\node [style=none] (0) at (0.75, 0.5) {};
		\node [style=dot] (1) at (0, -0) {};
		\node [style=none] (2) at (0.75, -0.5) {};
		\node [style=none] (3) at (-0.75, -0) {};
		\node [style=none] (4) at (2, 0.5) {$X \otimes Y$};
		\node [style=none] (5) at (2, -0.5) {$X \otimes Y$};
		\node [style=none] (6) at (-1.5, -0) {$X \otimes Y$};
		\node [style=none] (7) at (1.25, 0.5) {};
		\node [style=none] (8) at (1.25, -0.5) {};
	\end{pgfonlayer}
	\begin{pgfonlayer}{edgelayer}
		\draw [in=90, out=180, looseness=0.90] (0.center) to (1.center);
		\draw [in=-90, out=180, looseness=0.90] (2.center) to (1.center);
		\draw (1.center) to (3.center);
		\draw (7.center) to (0.center);
		\draw (8.center) to (2.center);
	\end{pgfonlayer}
\end{tikzpicture}
\end{aligned}
=
\begin{aligned}
\begin{tikzpicture}[scale=.65]
	\begin{pgfonlayer}{nodelayer}
		\node [style=none] (0) at (0, 0.5) {};
		\node [style=dot] (1) at (-0.75, -0) {};
		\node [style=none] (2) at (0, -0.5) {};
		\node [style=none] (3) at (-1.5, -0) {};
		\node [style=none] (4) at (1.5, 0.5) {$X$};
		\node [style=none] (5) at (1.5, -1.25) {$X$};
		\node [style=none] (6) at (-1.75, -0) {$X$};
		\node [style=none] (7) at (1.25, 0.5) {};
		\node [style=none] (8) at (1.25, -1.25) {};
		\node [style=none] (9) at (-1.5, -1.75) {};
		\node [style=none] (10) at (1.25, -2.25) {};
		\node [style=none] (11) at (-1.75, -1.75) {$Y$};
		\node [style=none] (12) at (1.25, -0.5) {};
		\node [style=none] (13) at (0, -2.25) {};
		\node [style=dot] (14) at (-0.75, -1.75) {};
		\node [style=none] (15) at (0, -1.25) {};
		\node [style=none] (16) at (1.5, -0.5) {$Y$};
		\node [style=none] (17) at (1.5, -2.25) {$Y$};
	\end{pgfonlayer}
	\begin{pgfonlayer}{edgelayer}
		\draw [in=90, out=180, looseness=0.90] (0.center) to (1.center);
		\draw [in=-90, out=180, looseness=0.90] (2.center) to (1.center);
		\draw (1.center) to (3.center);
		\draw (7.center) to (0.center);
		\draw [in=0, out=180, looseness=1.00] (8.center) to (2.center);
		\draw [in=90, out=180, looseness=0.90] (15.center) to (14);
		\draw [in=-90, out=180, looseness=0.90] (13.center) to (14);
		\draw (14) to (9.center);
		\draw [in=0, out=180, looseness=1.00] (12.center) to (15.center);
		\draw (10.center) to (13.center);
	\end{pgfonlayer}
\end{tikzpicture}
\end{aligned}
& &
  \begin{aligned}
    \begin{tikzpicture}[scale=.65]
	\begin{pgfonlayer}{nodelayer}
		\node [style=dot] (0) at (1.5, 0.5) {};
		\node [style=none] (1) at (-0.5, 0.5) {$X\otimes Y$};
		\node [style=none] (2) at (0.25, 0.5) {};
	\end{pgfonlayer}
	\begin{pgfonlayer}{edgelayer}
		\draw (2.center) to (0.center);
	\end{pgfonlayer}
\end{tikzpicture}
  \end{aligned}
  \qquad \quad =
  \begin{aligned}
    \begin{tikzpicture}[scale=.65]
	\begin{pgfonlayer}{nodelayer}
		\node [style=dot] (0) at (1.5, 0.5) {};
		\node [style=none] (1) at (0, 0.5) {$X$};
		\node [style=none] (2) at (0.25, 0.5) {};
		\node [style=none] (3) at (0.25, -0.25) {};
		\node [style=dot] (4) at (1.5, -0.25) {};
		\node [style=none] (5) at (0, -0.25) {$Y$};
	\end{pgfonlayer}
	\begin{pgfonlayer}{edgelayer}
		\draw (2.center) to (0.center);
		\draw (3.center) to (4.center);
	\end{pgfonlayer}
\end{tikzpicture}
\qquad \quad
  \end{aligned}
}
\]
\end{definition}

Note that we do \emph{not} require these Frobenius morphisms to be natural in
$X$. While morphisms in a hypergraph category need not interact with the Frobenius
structure in any particular way, we do require functors between hypergraph
categories to preserve it.

\begin{definition}
A functor $(F,\varphi)$ of hypergraph categories, or \define{hypergraph
functor}, is a strong symmetric monoidal functor $(F,\varphi)\maps (\mc
H,\otimes) \to (\mc H',\boxtimes)$ such that for each object $X$ the following
diagrams commute:
\[
  \xymatrix{
    FX\boxtimes FX \ar[rr]^{\mu_{FX}} \ar[dr]_{\varphi} && FX \\
    & F(X \ot X) \ar[ur]_{F\mu_X}
  }
  \qquad
  \xymatrix{
    I' \ar[rr]^{\eta_{FX}} \ar[dr]_{\varphi_I} && FX \\
    & FI \ar[ur]_{F\eta_X}
  }
\]
\[
  \xymatrix{
    FX \ar[rr]^{\delta_{FX}} \ar[dr]_{F\delta_X} && FX \boxtimes FX\\
    & F(X \ot X) \ar[ur]_{\varphi^{-1}}
  }
  \qquad
  \xymatrix{
    FX \ar[rr]^{\epsilon_{FX}} \ar[dr]_{F\epsilon_X} && I' \\
    & FI \ar[ur]_{\varphi^{-1}}
  }
\]
\end{definition}

Equivalently, a strong symmetric monoidal functor $F$ is a hypergraph functor if
for every $X$ the special commutative Frobenius structure on $FX$ is
\[
  (FX,\enspace F\mu_X \circ \varphi_{X,X},\enspace  \varphi^{-1}_{X,X} \circ F\delta_X,\enspace  F\eta_X \circ
  \varphi_I,\enspace  \varphi_I^{-1} \circ F\epsilon_X).
\]

Just as monoidal natural transformations themselves are enough as morphisms
between symmetric monoidal functors, so too they suffice as morphisms between
hypergraph functors. Two hypergraph categories are \define{hypergraph
equivalent} if there exist hypergraph functors with monoidal natural
transformations to the identity functors. 

The term hypergraph category was introduced recently \cite{Fon15,Kis}, in
reference to the fact that these special commutative Frobenius monoids provide
precisely the structure required for their string diagrams to be directed graphs
with `hyperedges': edges connecting any number of inputs to any number of
outputs.  More precisely, write $\mc P(N)$ for the finite power set of $N$. Then
a directed hypergraph $(N,H,s,t)$ comprises a set $N$ of vertices, a set $H$ of
hyperedges, and source and target functions $s,t\maps H \to \mc P(N)$. We then
think of morphisms in a hypergraph category as hyperedges; for example, the
morphism
\[
  \begin{aligned}
  \begin{tikzpicture}
	\begin{pgfonlayer}{nodelayer}
		\node [style=none] (0) at (-0.25, 0.375) {};
		\node [style=none] (1) at (0.5, 0.375) {};
		\node [style=none] (2) at (-0.25, -.875) {};
		\node [style=none] (3) at (0.5, -.875) {};
		\node [style=none] (4) at (0.5, 0.125) {};
		\node [style=none] (5) at (1.25, 0.125) {};
		\node [style=none] (6) at (0.125, -0.25) {$f$};
		\node [style=none] (7) at (1.5, 0.125) {$Y_1$};
		\node [style=none] (8) at (1.25, -0.625) {};
		\node [style=none] (9) at (1.5, -0.625) {$Y_2$};
		\node [style=none] (10) at (0.5, -0.625) {};
		\node [style=none] (11) at (-0.25, -0.75) {};
		\node [style=none] (12) at (-1.25, -0.75) {$X_3$};
		\node [style=none] (13) at (-0.25, -0.25) {};
		\node [style=none] (14) at (-1.25, 0.25) {$X_1$};
		\node [style=none] (15) at (-1, 0.25) {};
		\node [style=none] (16) at (-0.25, 0.25) {};
		\node [style=none] (17) at (-1, -0.25) {};
		\node [style=none] (18) at (-1.25, -0.25) {$X_2$};
		\node [style=none] (19) at (-1, -0.75) {};
	\end{pgfonlayer}
	\begin{pgfonlayer}{edgelayer}
		\draw (0.center) to (1.center);
		\draw (1.center) to (3.center);
		\draw (3.center) to (2.center);
		\draw (2.center) to (0.center);
		\draw (4.center) to (5.center);
		\draw (10.center) to (8.center);
		\draw (16.center) to (15.center);
		\draw (13.center) to (17.center);
		\draw (11.center) to (19.center);
	\end{pgfonlayer}
\end{tikzpicture}
\end{aligned}
\]
is a hyperedge from the set $\{X_1,X_2,X_3\}$ to the set $\{Y_1,Y_2\}$. Roughly
speaking, two morphisms of a hypergraph category are equal by the axioms of a
hypergraph category if their string diagrams reduce, using the graphical
manipulations given by the symmetric monoidal category and Frobenius axioms, to
the same hypergraph. Details can be found in the work of Bonchi, Gadduci,
Kissinger, Soboci\'nski, and Zanasi \cite{BGKSZ}.

Like special commutative Frobenius monoids, hypergraph categories were first
defined by Walters and Carboni, under the name \emph{well-supported compact
closed category} (the first appearance in print is \cite{Car91}; see
\cite{RSW05} for a history).  In recent years they have been rediscovered a
number of times, also appearing under the names \emph{dgs-monoidal category}
\cite{GH98} and \emph{dungeon category} \cite{Mor14}. 

%Nonetheless, this flexibility is a strength of hypergraph categories. In the
%category of vector spaces, for example, a special commutative Frobenius monoid
%is a basis \cite{CoePavVic}. This allows us to talk about the hypergraph category of vector spaces with bases .

\subsection{Hypergraph categories are self-dual compact closed}
\label{ssec.compactclosed}

Note that if an object $X$ is equipped with a Frobenius monoid structure then
the maps 
\[
    \xymatrixrowsep{0pt}
    \xymatrix{
  \begin{aligned}
      \resizebox{.09\textwidth}{!}{
	\begin{tikzpicture}
	  \begin{pgfonlayer}{nodelayer}
	    \node [style=circ] (0) at (0.75, -0) {};
	    \node [style=circ] (1) at (0.125, -0) {};
	    \node [style=none] (2) at (-1, 0.5) {};
	    \node [style=none] (3) at (-1, -0.5) {};
	  \end{pgfonlayer}
	  \begin{pgfonlayer}{edgelayer}
	    \draw [line width=2pt] (0.center) to (1.center);
	    \draw [line width=2pt, in=0, out=120, looseness=1.20] (1.center) to (2.center);
	    \draw [line width=2pt, in=0, out=-120, looseness=1.20] (1.center) to (3.center);
	  \end{pgfonlayer}
	\end{tikzpicture} 
    }
  \end{aligned}
  & \quad \mbox{and} \quad&
  \begin{aligned}
      \resizebox{.09\textwidth}{!}{
	\begin{tikzpicture}
	  \begin{pgfonlayer}{nodelayer}
	    \node [style=circ] (0) at (-1, 0) {};
	    \node [style=circ] (1) at (-0.375, 0) {};
	    \node [style=none] (2) at (0.75, -0.5) {};
	    \node [style=none] (3) at (0.75, 0.5) {};
	  \end{pgfonlayer}
	  \begin{pgfonlayer}{edgelayer}
	    \draw [line width=2pt] (0.center) to (1.center);
	    \draw [line width=2pt, in=180, out=-60, looseness=1.20] (1.center) to (2.center);
	    \draw [line width=2pt, in=180, out=60, looseness=1.20] (1.center) to (3.center);
	  \end{pgfonlayer}
	\end{tikzpicture}
      } 
  \end{aligned} \\
      \epsilon \circ \mu\maps X \ot X \to I & &
      \delta \circ \eta\maps I \to X \ot X
    }
\]
obey both
\[
  \begin{aligned}
    \resizebox{3cm}{!}{
      \begin{tikzpicture}
	\begin{pgfonlayer}{nodelayer}
	  \node [style=circ] (0) at (-1.5, 0.5) {};
	  \node [style=circ] (1) at (-0.75, 0.5) {};
	  \node [style=none] (2) at (0.25, -0) {};
	  \node [style=none] (3) at (0.25, 1) {};
	  \node [style=circ] (4) at (1, -0.5) {};
	  \node [style=none] (5) at (0, -0) {};
	  \node [style=circ] (6) at (1.75, -0.5) {};
	  \node [style=none] (7) at (0, -1) {};
	  \node [style=none] (8) at (2.5, 1) {};
	  \node [style=none] (9) at (-2.5, -1) {};
	\end{pgfonlayer}
	\begin{pgfonlayer}{edgelayer}
	  \draw [line width=2pt, in=180, out=-60, looseness=1.20] (1) to (2.center);
	  \draw [line width=2pt, in=180, out=60, looseness=1.20] (1) to (3.center);
	  \draw [line width=2pt] (0.center) to (1);
	  \draw [line width=2pt] (6.center) to (4);
	  \draw [line width=2pt, in=0, out=120, looseness=1.20] (4) to (5.center);
	  \draw [line width=2pt, in=0, out=-120, looseness=1.20] (4) to (7.center);
	  \draw [line width=2pt] (3.center) to (8.center);
	  \draw [line width=2pt] (7.center) to (9.center);
	\end{pgfonlayer}
      \end{tikzpicture}
    }
  \end{aligned}
  \quad = \quad
  \begin{aligned}
    \resizebox{3cm}{!}{
      \begin{tikzpicture}
	\begin{pgfonlayer}{nodelayer}
	  \node [style=circ] (0) at (-0.5, -0) {};
	  \node [style=none] (1) at (-1.5, -0.5) {};
	  \node [style=circ] (2) at (-1.5, 0.5) {};
	  \node [style=circ] (3) at (0.5, -0) {};
	  \node [style=circ] (4) at (1.5, -0.5) {};
	  \node [style=none] (5) at (1.5, 0.5) {};
	  \node [style=none] (6) at (2.5, 0.5) {};
	  \node [style=none] (7) at (-2.5, -0.5) {};
	\end{pgfonlayer}
	\begin{pgfonlayer}{edgelayer}
	  \draw [line width=2pt, in=0, out=-120, looseness=1.20] (0.center) to (1.center);
	  \draw [line width=2pt, in=0, out=120, looseness=1.20] (0.center) to (2.center);
	  \draw [line width=2pt, in=180, out=-60, looseness=1.20] (3) to (4.center);
	  \draw [line width=2pt, in=180, out=60, looseness=1.20] (3) to (5.center);
	  \draw [line width=2pt] (0) to (3);
	  \draw [line width=2pt] (7.center) to (1.center);
	  \draw [line width=2pt] (5.center) to (6.center);
	\end{pgfonlayer}
      \end{tikzpicture}
    }
  \end{aligned}
  \quad = \quad
  \begin{aligned}
    \resizebox{2cm}{!}{
      \begin{tikzpicture}
	\begin{pgfonlayer}{nodelayer}
	  \node [style=none] (0) at (2, -0) {};
	  \node [style=none] (1) at (-2, -0) {};
	  \node [style=none] (2) at (0, -0.5) {};
	  \node [style=none] (3) at (0, 0.5) {};
	\end{pgfonlayer}
	\begin{pgfonlayer}{edgelayer}
	  \draw [line width=2pt](1.center) to (0.center);
	\end{pgfonlayer}
      \end{tikzpicture}
    }
  \end{aligned}
\]
and the reflected equations. Thus if an object carries a Frobenius monoid it is
also self-dual, and any hypergraph category is a fortiori self-dual compact
closed. 

We introduce the notation
  \[
    \tikzset{every path/.style={line width=1.1pt}}
    \begin{aligned}
      \begin{tikzpicture}[scale=.65]
	\begin{pgfonlayer}{nodelayer}
		\node [style=none] (0) at (0.75, 0.5) {};
		\node [style=none] (1) at (0, -0) {};
		\node [style=none] (2) at (0.75, -0.5) {};
		\node [style=none] (3) at (1.25, 0.5) {};
		\node [style=none] (4) at (1.25, -0.5) {};
	\end{pgfonlayer}
	\begin{pgfonlayer}{edgelayer}
		\draw [in=90, out=180, looseness=0.90] (0.center) to (1.center);
		\draw [in=-90, out=180, looseness=0.90] (2.center) to (1.center);
		\draw (3.center) to (0.center);
		\draw (4.center) to (2.center);
	\end{pgfonlayer}
\end{tikzpicture}
    \end{aligned}
    :=
    \begin{aligned}
      \begin{tikzpicture}[scale=.65]
	\begin{pgfonlayer}{nodelayer}
		\node [style=none] (0) at (0.75, 0.5) {};
		\node [style=dot] (1) at (0, -0) {};
		\node [style=none] (2) at (0.75, -0.5) {};
		\node [style=none] (3) at (1.25, 0.5) {};
		\node [style=none] (4) at (1.25, -0.5) {};
		\node [style=dot] (5) at (-0.5, -0) {};
	\end{pgfonlayer}
	\begin{pgfonlayer}{edgelayer}
		\draw [in=90, out=180, looseness=0.90] (0.center) to (1.center);
		\draw [in=-90, out=180, looseness=0.90] (2.center) to (1.center);
		\draw (3.center) to (0.center);
		\draw (4.center) to (2.center);
		\draw (5.center) to (1.center);
	\end{pgfonlayer}
\end{tikzpicture}
    \end{aligned}
    \qquad
    \qquad
    \begin{aligned}
      \begin{tikzpicture}[scale=.65]
	\begin{pgfonlayer}{nodelayer}
		\node [style=none] (0) at (0, 0.5) {};
		\node [style=none] (1) at (0.75, -0) {};
		\node [style=none] (2) at (0, -0.5) {};
		\node [style=none] (3) at (-0.5, 0.5) {};
		\node [style=none] (4) at (-0.5, -0.5) {};
	\end{pgfonlayer}
	\begin{pgfonlayer}{edgelayer}
		\draw [in=90, out=0, looseness=0.90] (0.center) to (1.center);
		\draw [in=-90, out=0, looseness=0.90] (2.center) to (1.center);
		\draw (3.center) to (0.center);
		\draw (4.center) to (2.center);
	\end{pgfonlayer}
\end{tikzpicture}
    \end{aligned}
    :=
    \begin{aligned}
      \begin{tikzpicture}[scale=.65]
	\begin{pgfonlayer}{nodelayer}
		\node [style=none] (0) at (0, 0.5) {};
		\node [style=dot] (1) at (0.75, -0) {};
		\node [style=none] (2) at (0, -0.5) {};
		\node [style=none] (3) at (-0.5, 0.5) {};
		\node [style=none] (4) at (-0.5, -0.5) {};
		\node [style=dot] (5) at (1.25, -0) {};
	\end{pgfonlayer}
	\begin{pgfonlayer}{edgelayer}
		\draw [in=90, out=0, looseness=0.90] (0.center) to (1.center);
		\draw [in=-90, out=0, looseness=0.90] (2.center) to (1.center);
		\draw (3.center) to (0.center);
		\draw (4.center) to (2.center);
		\draw (5.center) to (1.center);
	\end{pgfonlayer}
\end{tikzpicture}
    \end{aligned}
  \]
As in any self-dual compact closed category, mapping each morphism 
$
    \tikzset{every path/.style={line width=1.1pt}}
    \begin{aligned}
  \begin{tikzpicture}[scale=.65]
	\begin{pgfonlayer}{nodelayer}
		\node [style=none] (0) at (0.5, -0.5) {};
		\node [style=none] (1) at (-0.25, -0.5) {};
		\node [style=none] (2) at (-1, -0.5) {};
		\node [style=none] (3) at (-1.75, -0.5) {};
		\node [style=none] (4) at (-1, -0.125) {};
		\node [style=none] (5) at (-1, -0.875) {};
		\node [style=none] (6) at (-0.25, -0.875) {};
		\node [style=none] (7) at (-0.25, -0.125) {};
		\node [style=none] (8) at (-0.625, -0.5) {$f$};
		\node [style=none] (9) at (-2, -0.5) {$X$};
		\node [style=none] (10) at (0.75, -0.5) {$Y$};
	\end{pgfonlayer}
	\begin{pgfonlayer}{edgelayer}
		\draw (1.center) to (0.center);
		\draw (2.center) to (3.center);
		\draw (4.center) to (5.center);
		\draw (5.center) to (6.center);
		\draw (6.center) to (7.center);
		\draw (7.center) to (4.center);
	\end{pgfonlayer}
\end{tikzpicture}
    \end{aligned}
$
to its dual morphism
\[
  \tikzset{every path/.style={line width=1.1pt}}
  \begin{tikzpicture}[scale=.65]
    \begin{pgfonlayer}{nodelayer}
      \node [style=none] (0) at (0, 0.5) {};
      \node [style=none] (1) at (0.75, -0) {};
      \node [style=none] (2) at (0, -0.5) {};
      \node [style=none] (3) at (-2.5, 0.5) {};
      \node [style=none] (4) at (-0.25, -0.5) {};
      \node [style=none] (5) at (-2, -1) {};
      \node [style=none] (6) at (-1.25, -1.5) {};
      \node [style=none] (7) at (-1, -0.5) {};
      \node [style=none] (8) at (1.25, -1.5) {};
      \node [style=none] (9) at (-1.25, -0.5) {};
      \node [style=none] (10) at (-1, -0.125) {};
      \node [style=none] (11) at (-1, -0.875) {};
      \node [style=none] (12) at (-0.25, -0.875) {};
      \node [style=none] (13) at (-0.25, -0.125) {};
      \node [style=none] (14) at (-0.625, -0.5) {$f$};
      \node [style=none] (15) at (1.5, -1.5) {$X$};
      \node [style=none] (16) at (-2.75, 0.5) {$Y$};
    \end{pgfonlayer}
    \begin{pgfonlayer}{edgelayer}
      \draw [in=90, out=0, looseness=0.90] (0.center) to (1.center);
      \draw [in=-90, out=0, looseness=0.90] (2.center) to (1.center);
      \draw (3.center) to (0.center);
      \draw (4.center) to (2.center);
      \draw [in=90, out=180, looseness=0.90] (9.center) to (5.center);
      \draw [in=-90, out=180, looseness=0.90] (6.center) to (5.center);
      \draw (7.center) to (9.center);
      \draw (8.center) to (6.center);
      \draw (10.center) to (11.center);
      \draw (11.center) to (12.center);
      \draw (12.center) to (13.center);
      \draw (13.center) to (10.center);
    \end{pgfonlayer}
  \end{tikzpicture}
\]
further equips each hypergraph category with a so-called dagger functor---an
involutive contravariant endofunctor that is the identity on objects---such that
the category is a dagger compact category. Dagger compact categories were first
introduced in the context of categorical quantum mechanics \cite{AC04}, under the
name strongly compact closed category, and have been demonstrated to be a key
structure in diagrammatic reasoning and the logic of quantum mechanics.

Compactness introduces a tight relationship between composition and the monoidal
product of morphisms. Firstly, there is a one-to-one correspondence between
morphisms $X \to Y$ and morphisms $I \to X\ot Y$ given by taking 
$
    \tikzset{every path/.style={line width=1.1pt}}
    \begin{aligned}
  \begin{tikzpicture}[scale=.65]
	\begin{pgfonlayer}{nodelayer}
		\node [style=none] (0) at (0.5, -0.5) {};
		\node [style=none] (1) at (-0.25, -0.5) {};
		\node [style=none] (2) at (-1, -0.5) {};
		\node [style=none] (3) at (-1.75, -0.5) {};
		\node [style=none] (4) at (-1, -0.125) {};
		\node [style=none] (5) at (-1, -0.875) {};
		\node [style=none] (6) at (-0.25, -0.875) {};
		\node [style=none] (7) at (-0.25, -0.125) {};
		\node [style=none] (8) at (-0.625, -0.5) {$f$};
		\node [style=none] (9) at (-2, -0.5) {$X$};
		\node [style=none] (10) at (0.75, -0.5) {$Y$};
	\end{pgfonlayer}
	\begin{pgfonlayer}{edgelayer}
		\draw (1.center) to (0.center);
		\draw (2.center) to (3.center);
		\draw (4.center) to (5.center);
		\draw (5.center) to (6.center);
		\draw (6.center) to (7.center);
		\draw (7.center) to (4.center);
	\end{pgfonlayer}
\end{tikzpicture}
    \end{aligned}
$
to its so-called \define{name}
\[
    \tikzset{every path/.style={line width=1.1pt}}
  \begin{tikzpicture}[scale=.65]
	\begin{pgfonlayer}{nodelayer}
		\node [style=none] (0) at (0.5, -0.5) {};
		\node [style=none] (1) at (-0.25, -0.5) {};
		\node [style=none] (2) at (-2, -0) {};
		\node [style=none] (3) at (-1.25, -0.5) {};
		\node [style=none] (4) at (0.5, 0.5) {};
		\node [style=none] (5) at (-1, -0.5) {};
		\node [style=none] (6) at (-1.25, 0.5) {};
		\node [style=none] (7) at (-1, -0.125) {};
		\node [style=none] (8) at (-1, -0.875) {};
		\node [style=none] (9) at (-0.25, -0.875) {};
		\node [style=none] (10) at (-0.25, -0.125) {};
		\node [style=none] (11) at (-0.625, -0.5) {$f$};
		\node [style=none] (12) at (0.75, 0.5) {$X$};
		\node [style=none] (13) at (0.75, -0.5) {$Y$};
	\end{pgfonlayer}
	\begin{pgfonlayer}{edgelayer}
		\draw (1.center) to (0.center);
		\draw [in=90, out=180, looseness=0.90] (6.center) to (2.center);
		\draw [in=-90, out=180, looseness=0.90] (3.center) to (2.center);
		\draw (4.center) to (6.center);
		\draw (5.center) to (3.center);
		\draw (7.center) to (8.center);
		\draw (8.center) to (9.center);
		\draw (9.center) to (10.center);
		\draw (10.center) to (7.center);
	\end{pgfonlayer}
\end{tikzpicture}
\]
By compactness, we have the equation
\[
    \tikzset{every path/.style={line width=1.1pt}}
  \begin{aligned}
\begin{tikzpicture}[scale=.65]
	\begin{pgfonlayer}{nodelayer}
		\node [style=none] (0) at (0, -0.5) {};
		\node [style=none] (1) at (-0.25, -0.5) {};
		\node [style=none] (2) at (-2, -0) {};
		\node [style=none] (3) at (-1.25, -0.5) {};
		\node [style=none] (4) at (1.5, 0.5) {};
		\node [style=none] (5) at (-1, -0.5) {};
		\node [style=none] (6) at (-1.25, 0.5) {};
		\node [style=none] (7) at (-1, -0.125) {};
		\node [style=none] (8) at (-1, -0.875) {};
		\node [style=none] (9) at (-0.25, -0.875) {};
		\node [style=none] (10) at (-0.25, -0.125) {};
		\node [style=none] (11) at (-0.625, -0.5) {$f$};
		\node [style=none] (12) at (1.75, 0.5) {$X$};
		\node [style=none] (13) at (-2, -2) {};
		\node [style=none] (14) at (-0.25, -2.125) {};
		\node [style=none] (15) at (-0.25, -2.5) {};
		\node [style=none] (16) at (-1, -2.5) {};
		\node [style=none] (17) at (1.5, -2.5) {};
		\node [style=none] (18) at (-0.25, -2.875) {};
		\node [style=none] (19) at (-1.25, -2.5) {};
		\node [style=none] (20) at (0, -1.5) {};
		\node [style=none] (21) at (-0.625, -2.5) {$g$};
		\node [style=none] (22) at (-1, -2.125) {};
		\node [style=none] (23) at (-1.25, -1.5) {};
		\node [style=none] (24) at (-1, -2.875) {};
		\node [style=none] (25) at (1.75, -2.5) {$Z$};
		\node [style=none] (26) at (1, -1) {};
	\end{pgfonlayer}
	\begin{pgfonlayer}{edgelayer}
		\draw (1.center) to (0.center);
		\draw [in=90, out=180, looseness=0.90] (6.center) to (2.center);
		\draw [in=-90, out=180, looseness=0.90] (3.center) to (2.center);
		\draw (4.center) to (6.center);
		\draw (5.center) to (3.center);
		\draw (7.center) to (8.center);
		\draw (8.center) to (9.center);
		\draw (9.center) to (10.center);
		\draw (10.center) to (7.center);
		\draw (15.center) to (17.center);
		\draw [in=90, out=180, looseness=0.90] (23.center) to (13.center);
		\draw [in=-90, out=180, looseness=0.90] (19.center) to (13.center);
		\draw (20.center) to (23.center);
		\draw (16.center) to (19.center);
		\draw (22.center) to (24.center);
		\draw (24.center) to (18.center);
		\draw (18.center) to (14.center);
		\draw (14.center) to (22.center);
		\draw [in=90, out=0, looseness=0.90] (0.center) to (26.center);
		\draw [in=0, out=-90, looseness=0.90] (26.center) to (20.center);
	\end{pgfonlayer}
\end{tikzpicture}
  \end{aligned}
  \qquad
  =
  \qquad
  \begin{aligned}
\begin{tikzpicture}[scale=.65]
	\begin{pgfonlayer}{nodelayer}
		\node [style=none] (0) at (-0.25, -0.5) {};
		\node [style=none] (1) at (-2, -0) {};
		\node [style=none] (2) at (-1.25, -0.5) {};
		\node [style=none] (3) at (1.5, 0.5) {};
		\node [style=none] (4) at (-1, -0.5) {};
		\node [style=none] (5) at (-1.25, 0.5) {};
		\node [style=none] (6) at (-1, -0.125) {};
		\node [style=none] (7) at (-1, -0.875) {};
		\node [style=none] (8) at (-0.25, -0.875) {};
		\node [style=none] (9) at (-0.25, -0.125) {};
		\node [style=none] (10) at (-0.625, -0.5) {$f$};
		\node [style=none] (11) at (1.75, 0.5) {$X$};
		\node [style=none] (12) at (1, -0.125) {};
		\node [style=none] (13) at (1, -0.5) {};
		\node [style=none] (14) at (0.25, -0.5) {};
		\node [style=none] (15) at (1.5, -0.5) {};
		\node [style=none] (16) at (1, -.875) {};
		\node [style=none] (17) at (0.625, -0.5) {$g$};
		\node [style=none] (18) at (0.25, -0.125) {};
		\node [style=none] (19) at (0.25, -.875) {};
		\node [style=none] (20) at (1.75, -0.5) {$Z$};
	\end{pgfonlayer}
	\begin{pgfonlayer}{edgelayer}
		\draw [in=90, out=180, looseness=0.90] (5.center) to (1.center);
		\draw [in=-90, out=180, looseness=0.90] (2.center) to (1.center);
		\draw (3.center) to (5.center);
		\draw (4.center) to (2.center);
		\draw (6.center) to (7.center);
		\draw (7.center) to (8.center);
		\draw (8.center) to (9.center);
		\draw (9.center) to (6.center);
		\draw (13.center) to (15.center);
		\draw (18.center) to (19.center);
		\draw (19.center) to (16.center);
		\draw (16.center) to (12.center);
		\draw (12.center) to (18.center);
		\draw (0.center) to (14.center);
	\end{pgfonlayer}
\end{tikzpicture}
  \end{aligned}
\]
Here the right hand side is the name of the composite $f \circ g$, while the
left hand side is the monoidal product post-composed with the map
\[
    \tikzset{every path/.style={line width=1.1pt}}
\begin{tikzpicture}[scale=.65]
	\begin{pgfonlayer}{nodelayer}
		\node [style=none] (0) at (0, -0.5) {};
		\node [style=none] (1) at (1.5, 0.5) {};
		\node [style=none] (2) at (1.75, 0.5) {$X$};
		\node [style=none] (3) at (-0.5, -0.5) {$Y$};
		\node [style=none] (4) at (-0.25, -2.5) {};
		\node [style=none] (5) at (1.5, -2.5) {};
		\node [style=none] (6) at (0, -1.5) {};
		\node [style=none] (7) at (1.75, -2.5) {$Z$};
		\node [style=none] (8) at (1, -1) {};
		\node [style=none] (9) at (-0.25, 0.5) {};
		\node [style=none] (10) at (-0.5, 0.5) {$X$};
		\node [style=none] (11) at (-0.5, -2.5) {$Z$};
		\node [style=none] (12) at (-0.5, -1.5) {$Y$};
		\node [style=none] (13) at (-0.25, -1.5) {};
		\node [style=none] (14) at (-0.25, -0.5) {};
	\end{pgfonlayer}
	\begin{pgfonlayer}{edgelayer}
		\draw (4.center) to (5.center);
		\draw [in=90, out=0, looseness=0.90] (0.center) to (8.center);
		\draw [in=0, out=-90, looseness=0.90] (8.center) to (6.center);
		\draw (1.center) to (9.center);
		\draw (6.center) to (13.center);
		\draw (14.center) to (0.center);
	\end{pgfonlayer}
\end{tikzpicture}
\]
Thus this morphism, the product of a cap and two identity maps, enacts the
categorical composition on monoidal products of names. We will make liberal use
of this fact.

\subsection{Coherence}

The lack of naturality of the Frobenius maps in hypergraph categories affects
some common, often desirable, properties of structured categories. For example,
it is not always possible to construct a skeletal hypergraph category hypergraph
equivalent to a given hypergraph category: isomorphic objects may be equipped
with distinct Frobenius structures.  Similarly, a fully faithful, essentially
surjective hypergraph functor does not necessarily define a hypergraph
equivalence of categories. 

Nonetheless, in this section we prove that every hypergraph category is
hypergraph equivalent to a strict hypergraph category. This coherence result
will be important in proving that every hypergraph category can be constructed
using decorated corelations.

\begin{theorem} \label{thm.stricthypergraphs}
  Every hypergraph category is hypergraph equivalent to a strict hypergraph
  category. Moreover, the objects of this strict hypergraph category form a free
  monoid.
\end{theorem}
\begin{proof}
  Let $(\mc H,\ot)$ be a hypergraph category. Write $I$ for the monoidal unit in
  $\mc H$. As $\mc H$ is a fortiori a symmetric monoidal category, a standard
  construction (see Mac Lane \cite[Theorem XI.3.1]{Mac98}) gives an equivalent
  strict symmetric monoidal category $(\mc H_{\mathrm{str}}, \cdot)$ with
  finite lists $[x_1:\ldots:x_n]$ of objects of $\mc H$ as objects and as morphisms
  $[x_1:\ldots:x_n] \to [y_1:\ldots:y_m]$ those morphisms from $(((x_1 \ot x_2)
  \ot \dots) \ot x_n) \ot I \to (((y_1 \ot y_2) \ot \dots) \ot y_m) \ot I$ in $\mc H$.
  Composition is given by composition in $\mc H$.

  The monoidal structure on $\mc H_{\mathrm{str}}$ is given as follows. The
  monoidal product of objects in $\mc H_{\mathrm{str}}$ is given by
  concatenation $\cdot$ of lists; the monoidal unit is the empty list. Given a
  list $X = [x_1:\ldots: x_n]$ of objects in $\mc H$, write $PX = (((x_1 \ot x_2)
  \ot \dots) \ot x_n) \ot I$ for the `pre-parenthesised product of $X$' in $\mc H$ with
  all open parenthesis at the front.  Note that the empty list maps to the
  monoidal unit $I$. The monoidal product of two morphisms is
  given by their monoidal product in $\mc H$ pre- and post-composed with the
  necessary canonical maps: given $f\maps X \to Y$ and $g\maps Z \to W$, their
  product $f\cdot g\maps X\cdot Y \to Z \cdot W$ is 
  \[ 
    P(X \cdot Y) \stackrel{\sim}\longrightarrow PX \ot PY \stackrel{f \ot g}{\longrightarrow}
    PZ \ot PW \stackrel{\sim}\longrightarrow P(Z \cdot W).  
  \] 
  By design, the associators and unitors of $\mc H_{\mathrm{str}}$ are simply
  identity maps. The braiding $X \cdot Y \to Y \cdot X$ is given by the braiding
  $PX \ot PY \to PY \ot PX$ in $\mc H$, similarly pre- and post-composed with
  the necessary canonical maps. This defines a strict symmetric monoidal
  category \cite{Mac98}.

  To make $\mc H_{\mathrm{str}}$ into a hypergraph category, we make each
  object $[x_1:\ldots:x_n]$ into a special commutative Frobenius monoid in a
  similar way, pre- and post-composing lists of Frobenius maps with the
  necessary canonical maps. For example, the multiplication on $[x_1:\ldots:x_n]$
  is given by 
  \[
    \xymatrix{
      P([x_1:\ldots:x_n]\cdot[x_1:\,\ldots:x_n])
      \ar@{=}[d]
      \\ (((((((x_1 \ot x_2) \ot \dots) \ot x_n) \ot x_1) \ot x_2) \ot \dots)
      \ot x_n) \ot I
  \ar[d]^{\sim} \\
      ((((x_1 \ot x_1) \ot (x_2 \ot x_2)) \ot \dots) \ot (x_n \ot x_n)) \ot I
      \ar[d]^{(((\mu_{x_1} \ot \mu_{x_2}) \ot \dots ) \ot \mu_{x_n}) \ot 1_I}
      \\
      (((x_1 \ot x_2) \ot \dots ) \ot x_n) \ot I \ar@{=}[d]\\
      P([x_1:\ldots:x_n])
    }
  \]
  where the map labelled $\sim$ is the canonical map such that each pair of $x_i$'s
  remains in the same order. As the coherence maps are natural, each special
  commutative Frobenius monoid axiom for this data on $[x_1:\ldots:x_n]$ reduces
  to a list of the corresponding axioms for the objects $x_i$ in $\mc H$.
  Similarly, the coherence axioms and naturality of the coherence maps imply the
  Frobenius structure on the monoidal product of objects is given by the
  Frobenius structures on the factors in the required way. This defines a
  hypergraph category.
  \smallskip

  \begin{center}
    \begin{tabular}{| c | p{.65\textwidth} |}
      \hline
      \multicolumn{2}{|c|}{The strict hypergraph category $(\mc H_{\mathrm{str}},
      \cdot)$} \\
      \hline
      \textbf{objects} & finite lists $[x_1: \ldots: x_n]$ of objects of
      $\mathcal H$ \\ 
      \textbf{morphisms} & $\mathrm{hom}_{\mc H_{\mathrm{str}}}\big([x_1: \ldots:
      x_n],[y_1: \ldots: y_m]\big)$ \newline $= \mathrm{hom}_{\mc H}\big(((x_1 \ot
      \dots) \ot x_n) \ot I, ((y_1 \ot \dots) \ot y_m) \ot I\big)$\\ 
      \textbf{composition} & composition of corresponding maps in $\mc H$ \\
      \textbf{monoidal product} & concatenation of lists \\
      \textbf{coherence maps} & associators and unitors are strict; braiding is
      inherited from $\mc H$  \\
      \textbf{hypergraph maps} & lists of hypergraph maps in $\mc H$ \\
      \hline
    \end{tabular}
  \end{center}
  \smallskip

  Mac Lane's standard construction further gives a strong symmetric monoidal
  functors $P\maps \mc H_{\mathrm{str}} \to \mc H$ extending the map $P$ above,
  and $S\maps \mc H \to \mc H_{\mathrm{str}}$ sending $x \in \mc H$ to the
  string $[x]$ of length 1 in $\mc H_{\mathrm{str}}$. These extend to hypergraph
  functors.

  In detail, the functor $P$ is given on morphisms by taking a map in
  $\mathrm{hom}_{\mc H_{\mathrm{str}}}(X,Y)$ to the same map considered now as a
  map in $\mathrm{hom}_{\mc H}(PX,PY)$; its coherence maps are given by the
  canonical maps $PX \ot PY \to P(X\cdot Y)$. The functor $S$ is also easy to
  define: tensoring with the identity on $I$ gives a one-to-one correspondence
  between morphisms $x \to y$ in $\mc H$ and morphisms $[x] \to [y]$ in $\mc
  H_{\mathrm{str}}$, so $S$ is a monoidal embedding of $\mc H$ into $\mc
  H_{\mathrm{str}}$. By Mac Lane's proof of the coherence theorem for monoidal
  categories these are both strong monoidal functors. 
  
  By inspection these functors also preserve hypergraph structure: note in
  particular that the axioms interrelating the Frobenius structure and monoidal
  structure in $\mc H$ are crucial for the hypergraph-preserving nature of $P$.
  Hence $P$ and $S$ are hypergraph functors. As they already witness an
  equivalence of symmetric monoidal categories, thus $\mc H$ and $\mc
  H_{\mathrm{str}}$ are equivalent as hypergraph categories.
\end{proof}

\section{Example: cospan categories} \label{sec.cospans}

A central example of a hypergraph category is the category $\cospan(\mathcal C)$
of cospans in a category $\mathcal C$ with finite colimits. We call such
categories cospan categories. In the coming chapters we will successively
generalise this structure, first to decorated cospans, then corelations, then
decorated corelations. Each time, the new categories we construct will inherit
hypergraph structure from their relationship with cospan categories.

We first recall the basic definitions. Let $\mc C$ be a category with finite
colimits, writing the coproduct $+$. A \define{cospan}
\[
  \xymatrix{
    & N \\
    X \ar[ur]^{i} && Y \ar[ul]_{o}
  }
\]
from $X$ to $Y$ in $\mathcal C$ is a pair of morphisms with common codomain. We
refer to $X$ and $Y$ as the \define{feet}, and $N$ as the \define{apex}.  Given
two cospans $X \stackrel{i}{\longrightarrow} N \stackrel{o}{\longleftarrow} Y$
and $X \stackrel{i'}{\longrightarrow} N' \stackrel{o'}{\longleftarrow} Y$ with
the same feet, a \define{map of cospans} is a morphism $n\colon  N \to N'$ in
$\mathcal C$ between the apices such that
\[
  \xymatrix{
    & N \ar[dd]^n  \\
    X \ar[ur]^{i} \ar[dr]_{i'} && Y \ar[ul]_{o} \ar[dl]^{o'}\\
    & N'
  }
\]
commutes.

Cospans may be composed using the pushout from the common
foot: given cospans $X \stackrel{i_X}{\longrightarrow} N
\stackrel{o_Y}{\longleftarrow} Y$ and $Y \stackrel{i_Y}{\longrightarrow} M
\stackrel{o_Z}{\longleftarrow} Z$, their composite cospan is $X \xrightarrow{j_N
\circ i_X} N+_YM \xleftarrow{j_M\circ i_Z} Z$,
where the maps $j$ are given by the pushout square in the diagram
\[
  \xymatrix{
    && N+_YM \\
    & N \ar[ur]^{j_N} && M \ar[ul]_{j_M} \\
    \quad X \quad \ar[ur]^{i_X} && Y \ar[ul]_{o_Y} \ar[ur]^{i_Y} && \quad Z
    \quad \ar[ul]_{o_Z}.
  }
\]
This composition rule is associative up to isomorphism, and
so we may define a category, in fact a symmetric monoidal category,
$\mathrm{Cospan}(\mathcal C)$ with objects the objects of $\mathcal C$ and
morphisms isomorphism classes of cospans. 

The symmetric monoidal structure is `inherited' from $\mc C$. Indeed, we shall
consider any category $\mc C$ with finite colimits a symmetric monoidal category
as follows. Given maps $f \maps A \to C$, $g \maps B \to C$ with common
codomain, the universal property of the coproduct gives a unique map $A+B \to
C$. We call this the \define{copairing} of $f$ and $g$, and write it $[f,g]$. The
monoidal product on $\mc C$ is then given by the coproduct $+$, with monoidal
unit the initial object $\varnothing$ and coherence maps given by copairing the
appropriate identity, inclusion, and initial object maps. For example, the
braiding is given by $[\iota_X,\iota_Y]\maps X+Y \to Y+X$ where $\iota_X\maps X
\to Y+X$ and $\iota_Y\maps Y \to Y+X$ are the inclusion maps into the coproduct
$Y+X$.

The category $\mathrm{Cospan(\mathcal C)}$ inherits this symmetric monoidal
structure from $\mathcal C$ as follows. Call a subcategory $\mathcal C$ of a
category $\mathcal D$ \define{wide} if $\mathcal C$ contains all objects of
$\mathcal D$, and call a functor that is faithful and bijective-on-objects a
\define{wide embedding}.  Note then that we have a wide embedding
\[
  \mathcal C \hooklongrightarrow \mathrm{Cospan(\mathcal C)}
\]
that takes each object of $\mathcal C$ to itself as an object of
$\mathrm{Cospan(\mathcal C)}$, and each morphism $f\colon  X \to Y$ in $\mathcal
C$ to the cospan
\[
  \xymatrix{
    & Y \\
    X \ar[ur]^{f} && Y, \ar@{=}[ul]
  }
\]
where the extended `equals' sign denotes an identity morphism. Now since the
monoidal product $+\colon \mathcal C \times \mathcal C \to \mathcal C$ is left
adjoint to the diagonal functor, it preserves colimits, and so extends to a
functor $+\colon \mathrm{Cospan(\mathcal C)} \times \mathrm{Cospan(\mathcal C)}
\to \mathrm{Cospan(\mathcal C)}$. The coherence maps are just the images of the
coherence maps in $\mathcal C$ under this wide embedding; checking naturality is
routine, and clearly they still obey the required axioms.

\begin{remark}
  Recall that all finite colimits in a category $\mc C$ can be assembled from an
  initial object and the use of pushouts \cite{Mac98}. The monoidal category
  $(\cospan(\mc C),+)$ thus provides an excellent setting for describing finite
  colimits in $\mc C$: the initial object is the monoidal unit, and pushouts are
  given by composition of cospans. Rosebrugh, Sabadini, and Walters begin an
  exploration of the consequences of this idea in \cite{RSW08}.
  
  Recall also that coproduct in the category $\Set$ of sets and functions is the
  disjoint union of sets. Colimits often have this sort of intuitive
  interpretation of combining systems, and so are frequently useful for
  discussing interconnection of systems. This intuition is a cornerstone of the
  present work.
\end{remark}

The theory of symmetric monoidal categories of cospans is all well known,
following from theory originally developed for the dual concept: spans. First
discussed by Yoneda \cite{Yon54}, spans are pairs of morphisms $X \leftarrow N
\rightarrow Y$ with a common domain. B\'enabou introduced the bicategory of
objects, spans, and maps of spans in a category with finite limits as an example
of a symmetric monoidal bicategory \cite{Ben67}. The symmetric monoidal category
$\cospan(\mc C)$ is a decategorified version of the symmetric monoidal bicategory of
cospans. 

More recently, Lack proved that cospans and Frobenius monoids are intimately
related. Write $\FinSet$ for the category of finite sets and functions. This
category is a monoidal category with the coproduct $+$ as monoidal product.
Moreover, as a symmetric monoidal category $(\FinSet,+)$ has a strict skeleton.
In the following proposition $\FinSet$ denotes this strict skeleton.

\begin{proposition} \label{prop.cospanscfm}
  Special commutative Frobenius monoids in a strict symmetric monoidal category
  $\mc C$ are in one-to-one correspondence with strict symmetric monoidal functors
  $\cospan(\FinSet) \to \mc C$.
\end{proposition}
\begin{proof}
  The special commutative Frobenius monoid is given by the image of the one
  element set. See Lack \cite{Lac04}.
\end{proof}

Over the next few chapters we will further explore this deep link between
cospans and special commutative Frobenius monoids and hypergraph categories.
To begin, we detail a hypergraph structure on each cospan category $\cospan(\mc
C)$. 

Like the braiding, this hypergraph structure also comes from copairings of
identity morphisms.  Call cospans 
\[
  \xymatrix{
    & N \\
    X \ar[ur]^{i} && Y \ar[ul]_{o}
  }
  \qquad \xymatrix@R=8pt{\\\textrm{and}} \qquad 
  \xymatrix{
    & N \\
    Y \ar[ur]^{o} && X \ar[ul]_{i}
  }
\]
that are reflections of each other \define{opposite} cospans. Given any object
$X$ in $\mathcal C$, the copairing $[1_X,1_X]\colon  X + X \to X$ of two identity
maps on $X$, together with the unique map $!\colon  \varnothing \to X$ from the
initial object to $X$, define a monoid structure on $X$. Considering these
maps as morphisms in $\mathrm{Cospan(\mathcal C)}$, we may take them together
with their opposites to give a special commutative Frobenius structure on $X$.
Using the universal property of the coproduct, and noting that the Frobenius
maps and the braiding are both given by copairings of identity morphisms, it is
easily verified that this gives a hypergraph category.

Given $f \maps X \to Y$ in $\mc C$, abuse notation by writing $f \in
\mathrm{Cospan}(\mc C)$ for the cospan $X \stackrel{f}\to Y
\stackrel{1_Y}\leftarrow Y$, and $f^\opp$ for the cospan $Y \stackrel{1_Y}\to Y
\stackrel{f}\leftarrow X$. To summarise:

\begin{proposition}[Rosebrugh--Sabadini--Walters \cite{RSW08}]
  Given a category $\mc C$ with finite colimits, we may define a hypergraph
  category $\cospan(\mc C)$ as follows:
\smallskip 

  \begin{center}
  \begin{tabular}{ |c| p{.65\textwidth}|}
      \hline
      \multicolumn{2}{|c|}{The hypergraph category $(\mathrm{Cospan(\mc C)},+)$} \\
    \hline
    \textbf{objects} & the objects of $\mathcal C$ \\ 
    \textbf{morphisms} & isomorphism classes of cospans in
    $\mathcal C$\\ 
  \textbf{composition} & given by pushout \\
  \textbf{monoidal product} & the coproduct in $\mathcal C$ \\
  \textbf{coherence maps} & inherited from $(\mc C,+)$\\
  \textbf{hypergraph maps} & $\mu = [1,1]$, $\eta = !$,
  $\delta = [1,1]^\opp$, $\epsilon = !^\opp$. \\
      \hline
  \end{tabular}
\end{center}
\smallskip
\end{proposition}

We will often abuse our terminology and refer to cospans themselves as
morphisms in some cospan category $\mathrm{Cospan}(\mathcal C)$; we of course
refer instead to the isomorphism class of the said cospan.

It is not difficult to show that we in fact have a functor from the category
having categories with finite colimits as objects and colimit preserving
functors as morphisms to the category of hypergraph categories and hypergraph
functors.\footnote{Categories with finite colimits are known as \emph{finitely
cocomplete}, or (rarely) \emph{rex} or \emph{right exact}, while functors that
preserve finite colimits are known as \emph{finitely cocontinuous} or
\emph{right exact}. This extra layer of terminology is not necessary here, and
so we favour the more friendly, if occasionally a touch cumbersome, terms
above.} In the next chapter we show that this extends to a functor from a
category of so-called symmetric lax monoidal presheaves over categories with
finite colimits. This is known as the decorated cospan construction, and will
be instrumental in constructing hypergraph categories with network diagrams as
morphisms.

%Walters: cospan graph is the generic special commutative Frobenius monoid.  The
%free hypergraph category on a single object in the category of cospans in the
%category of finite sets.

%\begin{proposition}
%  The monoidal product in a hypergraph category is a coproduct for something if
%  and only if every morphism is a Frobenius monoid homomorphism. (ie take
%  subcategory of all objects, monoid maps, and monoid homs. This is cocomplete.)
%\end{proposition}
%Clearly not true for circuits for example.

\chapter{Decorated cospans: from closed systems to open} \label{ch.deccospans}

In the previous chapter we discussed hypergraph categories as a model for
network languages, and showed that cospans provide a method of constructing such
categories. In many situations, however, we wish to model our systems not just
with cospans, but `decorated' cospans, where the apex of each cospan is equipped
with some extra structure. In this chapter we detail a method for composing such
decorated cospans. 

The next section provides an introductory overview. Following this, we give a
formal definition of decorated cospan categories and show they are well defined
in \textsection\ref{sec:dcc}, and do the same for functors between decorated
cospan categories in \textsection\ref{sec:dcf}. While we first define decorated
cospans on any lax symmetric monoidal functor $F\maps (\mathcal C,+) \to
(\mathcal D,\otimes)$, where $\mc C$ has finite colimits and $+$ is the
coproduct, in \textsection\ref{sec.setdecs} we show that it is enough to
consider the case when $(\mc D, \otimes)$ is the category $(\Set,\times)$ of
sets and functions with monoidal structure given by the cartesian product. We
then conclude the chapter with some examples.

\section{From closed systems to open} \label{sec.closedtoopen}

The key idea in this chapter is that we may take our network diagrams, mark
points of interconnection, or terminals, using cospans, and connect these
terminals with others using pushouts. This marking of terminals defines a
boundary of the system, turning it from a closed system to an open system.
This process freely constructs a hypergraph category from a library of network
pieces.

It has been recognised for some time that spans and cospans provide an intuitive
framework for composing network diagrams \cite{KSW}, and the material we develop
here is a variant on this theme. Rather than use cospans in a category of
network diagrams, however, we define a category which models just the
interconnection structure, and then decorate cospans in this category with
additional structure. In the case of finite graphs, the intuition reflected is
this: given two graphs, we may construct a third by gluing chosen vertices of
the first with chosen vertices of the second. Our goal in this chapter is to
view this process as composition of morphisms in a category, in a way that also
facilitates the construction of a composition rule for any semantics associated
to the diagrams, and the construction of a functor between these two resulting
categories.

To see how this works, let us start with the following graph:
\begin{center}
  \begin{tikzpicture}[auto,scale=2.3]
    \node[circle,draw,inner sep=1pt,fill]         (A) at (0,0) {};
    \node[circle,draw,inner sep=1pt,fill]         (B) at (1,0) {};
    \node[circle,draw,inner sep=1pt,fill]         (C) at (0.5,-.86) {};
    \path (B) edge  [bend right,->-] node[above] {0.2} (A);
    \path (A) edge  [bend right,->-] node[below] {1.3} (B);
    \path (A) edge  [->-] node[left] {0.8} (C);
    \path (C) edge  [->-] node[right] {2.0} (B);
  \end{tikzpicture}
\end{center}
We shall work with labelled, directed graphs, as the additional data help
highlight the relationships between diagrams. Now, for this graph to be a
morphism, we must equip it with some notion of `input' and `output'. We do this by
marking vertices using functions from finite sets:
\begin{center}
  \begin{tikzpicture}[auto,scale=2.15]
    \node[circle,draw,inner sep=1pt,fill=gray,color=gray]         (x) at (-1.4,-.43) {};
    \node at (-1.4,-.9) {$X$};
    \node[circle,draw,inner sep=1pt,fill]         (A) at (0,0) {};
    \node[circle,draw,inner sep=1pt,fill]         (B) at (1,0) {};
    \node[circle,draw,inner sep=1pt,fill]         (C) at (0.5,-.86) {};
    \node[circle,draw,inner sep=1pt,fill=gray,color=gray]         (y1) at (2.4,-.25) {};
    \node[circle,draw,inner sep=1pt,fill=gray,color=gray]         (y2) at (2.4,-.61) {};
    \node at (2.4,-.9) {$Y$};
    \path (B) edge  [bend right,->-] node[above] {0.2} (A);
    \path (A) edge  [bend right,->-] node[below] {1.3} (B);
    \path (A) edge  [->-] node[left] {0.8} (C);
    \path (C) edge  [->-] node[right] {2.0} (B);
    \path[color=gray, very thick, shorten >=10pt, shorten <=5pt, ->, >=stealth] (x) edge (A);
    \path[color=gray, very thick, shorten >=10pt, shorten <=5pt, ->, >=stealth] (y1) edge (B);
    \path[color=gray, very thick, shorten >=10pt, shorten <=5pt, ->, >=stealth] (y2) edge (B);
  \end{tikzpicture}
\end{center}
Let $N$ be the set of vertices of the graph. Here the finite sets $X$, $Y$, and
$N$ comprise one, two, and three elements respectively, drawn as points, and the
values of the functions $X \to N$ and $Y \to N$ are indicated by the grey
arrows. This forms a cospan in the category of finite sets, one with the set at
the apex decorated by our given graph.

Given another such decorated cospan with input set equal to the output of the
above cospan
\begin{center}
  \begin{tikzpicture}[auto,scale=2.15]
    \node[circle,draw,inner sep=1pt,fill=gray,color=gray]         (y1) at (-1.4,-.25) {};
    \node[circle,draw,inner sep=1pt,fill=gray,color=gray]         (y2) at (-1.4,-.61) {};
    \node at (-1.4,-.9) {$Y$};
    \node[circle,draw,inner sep=1pt,fill]         (A) at (0,0) {};
    \node[circle,draw,inner sep=1pt,fill]         (B) at (1,0) {};
    \node[circle,draw,inner sep=1pt,fill]         (C) at (0.5,-.86) {};
    \node[circle,draw,inner sep=1pt,fill=gray,color=gray]         (z1) at (2.4,-.25) {};
    \node[circle,draw,inner sep=1pt,fill=gray,color=gray]         (z2) at (2.4,-.61) {};
    \node at (2.4,-.9) {$Z$};
    \path (A) edge  [->-] node[above] {1.7} (B);
    \path (C) edge  [->-] node[right] {0.3} (B);
    \path[color=gray, very thick, shorten >=10pt, shorten <=5pt, ->, >=stealth] (y1) edge (A);
    \path[color=gray, very thick, shorten >=10pt, shorten <=5pt, ->, >=stealth] (y2)
    edge (C);
    \path[color=gray, very thick, shorten >=10pt, shorten <=5pt, ->, >=stealth] (z1) edge (B);
    \path[color=gray, very thick, shorten >=10pt, shorten <=5pt, ->, >=stealth] (z2) edge (C);
  \end{tikzpicture}
\end{center}
composition involves gluing the graphs along the identifications
\begin{center}
  \begin{tikzpicture}[auto,scale=2]
    \node[circle,draw,inner sep=1pt,fill=gray,color=gray]         (x) at (-1.3,-.43) {};
    \node at (-1.3,-.9) {$X$};
    \node[circle,draw,inner sep=1pt,fill]         (A) at (0,0) {};
    \node[circle,draw,inner sep=1pt,fill]         (B) at (1,0) {};
    \node[circle,draw,inner sep=1pt,fill]         (C) at (0.5,-.86) {};
    \node[circle,draw,inner sep=1pt,fill=gray,color=gray]         (y1) at (2.3,-.25) {};
    \node[circle,draw,inner sep=1pt,fill=gray,color=gray]         (y2) at (2.3,-.61) {};
    \node at (2.3,-.9) {$Y$};
    \path (B) edge  [bend right,->-] node[above] {0.2} (A);
    \path (A) edge  [bend right,->-] node[below] {1.3} (B);
    \path (A) edge  [->-] node[left] {0.8} (C);
    \path (C) edge  [->-] node[right] {2.0} (B);
    \path[color=gray, very thick, shorten >=10pt, shorten <=5pt, ->, >=stealth] (x) edge (A);
    \path[color=gray, very thick, shorten >=10pt, shorten <=5pt, ->, >=stealth] (y1) edge (B);
    \path[color=gray, very thick, shorten >=10pt, shorten <=5pt, ->, >=stealth] (y2) edge (B);
    \node[circle,draw,inner sep=1pt,fill]         (A') at (3.6,0) {};
    \node[circle,draw,inner sep=1pt,fill]         (B') at (4.6,0) {};
    \node[circle,draw,inner sep=1pt,fill]         (C') at (4.1,-.86) {};
    \node[circle,draw,inner sep=1pt,fill=gray,color=gray]         (z1) at (5.9,-.25) {};
    \node[circle,draw,inner sep=1pt,fill=gray,color=gray]         (z2) at (5.9,-.61) {};
    \node at (5.9,-.9) {$Z$};
    \path (A') edge  [->-] node[above] {1.7} (B');
    \path (C') edge  [->-] node[right] {0.3} (B');
    \path[color=gray, very thick, shorten >=10pt, shorten <=5pt, ->, >=stealth] (y1) edge (A');
    \path[color=gray, very thick, shorten >=10pt, shorten <=5pt, ->, >=stealth] (y2)
    edge (C');
    \path[color=gray, very thick, shorten >=10pt, shorten <=5pt, ->, >=stealth] (z1) edge (B');
    \path[color=gray, very thick, shorten >=10pt, shorten <=5pt, ->, >=stealth] (z2) edge (C');
  \end{tikzpicture}
\end{center}
specified by the shared foot of the two cospans. This results in the decorated
cospan
\begin{center}
  \begin{tikzpicture}[auto,scale=2.15]
    \node[circle,draw,inner sep=1pt,fill=gray,color=gray]         (x) at (-1.4,-.43) {};
    \node at (-1.4,-.9) {$X$};
    \node[circle,draw,inner sep=1pt,fill]         (A) at (0,0) {};
    \node[circle,draw,inner sep=1pt,fill]         (B) at (1,0) {};
    \node[circle,draw,inner sep=1pt,fill]         (C) at (0.5,-.86) {};
    \node[circle,draw,inner sep=1pt,fill]         (D) at (2,0) {};
    \node[circle,draw,inner sep=1pt,fill=gray,color=gray]         (z1) at (3.4,-.25) {};
    \node[circle,draw,inner sep=1pt,fill=gray,color=gray]         (z2) at (3.4,-.61) {};
    \node at (3.4,-.9) {$Z$};
    \path (B) edge  [bend right,->-] node[above] {0.2} (A);
    \path (A) edge  [bend right,->-] node[below] {1.3} (B);
    \path (A) edge  [->-] node[left] {0.8} (C);
    \path (C) edge  [->-] node[right] {2.0} (B);
    \path (B) edge  [bend left,->-] node[above] {1.7} (D);
    \path (B) edge  [bend right,->-] node[below] {0.3} (D);
    \path[color=gray, very thick, shorten >=10pt, shorten <=5pt, ->, >=stealth] (x) edge (A);
    \path[color=gray, very thick, shorten >=10pt, shorten <=5pt, ->, >=stealth] (z1)
    edge (D);
    \path[bend left, color=gray, very thick, shorten >=10pt, shorten <=5pt, ->, >=stealth] (z2)
    edge (B);
  \end{tikzpicture}
\end{center}
The decorated cospan framework generalises this intuitive construction.

More precisely: fix a set $L$. Then given a finite set $N$, we may talk of the
collection of finite $L$-labelled directed multigraphs, to us just $L$-graphs
or simply graphs, that have $N$ as their set of vertices. Write such a graph
$(N,E,s,t,r)$, where $E$ is a finite set of edges, $s\colon E \to N$ and $t\colon E \to N$
are functions giving the source and target of each edge respectively, and $r\colon  E
\to L$ equips each edge with a label from the set $L$.  Next, given a function
$f\colon N \to M$, we may define a function from graphs on $N$ to graphs on $M$
mapping $(N,E,s,t,r)$ to $(M,E,f \circ s,f \circ t, r)$.  After dealing
appropriately with size issues, this gives a lax monoidal functor from
$(\FinSet,+)$ to $(\Set,\times)$.\footnote{Here $(\FinSet,+)$ is the monoidal
  category of finite sets and functions with disjoint union as monoidal
  product, and $(\Set,\times)$ is the category of sets and functions with
  cartesian product as monoidal product. One might ensure the collection of
  graphs forms a set in a number of ways. One such method is as follows: the
  categories of finite sets and finite graphs are essentially small; replace
  them with equivalent small categories. We then constrain the graphs
  $(N,E,s,t,r)$ to be drawn only from the objects of our small category of finite
  graphs.}  

Now, taking any lax monoidal functor $(F,\varphi)\colon  (\mathcal C,+) \to (\mathcal
D,\otimes)$ with $\mathcal C$ having finite colimits and coproduct written $+$,
the decorated cospan category associated to $F$ has as objects the objects of
$\mathcal C$, and as morphisms pairs comprising a cospan in $\mathcal C$
together with some morphism $I \to FN$, where $I$ is the unit in $(\mathcal
D,\otimes)$ and $N$ is the apex of the cospan. In the case of our graph
functor, this additional data is equivalent to equipping the apex $N$ of the
cospan with a graph. We thus think of our morphisms as having two distinct
parts: an instance of our chosen structure on the apex, and a cospan describing
interfaces to this structure. Our first theorem (Theorem \ref{thm:fcospans})
says that when $(\mathcal D,\otimes)$ is braided monoidal and $(F,\varphi)$ lax
braided monoidal, we may further give this data a composition rule and monoidal
product such that the resulting decorated cospan category is hypergraph.  

Suppose now we have two such lax monoidal functors; we then have two such
decorated cospan categories. Our second theorem (Theorem
\ref{thm:decoratedfunctors}) is that, given also a monoidal natural
transformation between these functors, we may construct a strict hypergraph
functor between their corresponding decorated cospan categories.  These natural
transformations can often be specified by some semantics associated to some type
of topological diagram. A trivial case of such is assigning to a finite graph
its number of vertices, but richer examples abound, including assigning to a
directed graph with edges labelled by rates its depicted Markov process, or
assigning to an electrical circuit diagram the current--voltage relationship
such a circuit would impose.

\section{Decorated cospan categories} \label{sec:dcc}
Our task in this section is to define decorated cospans and prove that they can
be used to construct hypergraph categories. Motivated by our discussion in the
previous section, we thus begin with the following definition.

\begin{definition} \label{def:fcospans}
  Let $\mathcal C$ be a category with finite colimits, and
  \[
    (F,\varphi)\colon  (\mathcal C,+) \longrightarrow (\mathcal D, \otimes)
  \]
  be a lax monoidal functor. We define a \define{decorated cospan}, or more
  precisely an $F$-decorated cospan, to be a pair 
  \[
    \left(
    \begin{aligned}
      \xymatrix{
	& N \\  
	X \ar[ur]^{i} && Y \ar[ul]_{o}
      }
    \end{aligned}
    ,
    \qquad
    \begin{aligned}
      \xymatrix{
	FN \\
	I \ar[u]_{s}
      }
    \end{aligned}
    \right)
  \]
  comprising a cospan $X \stackrel{i}\rightarrow N \stackrel{o}\leftarrow Y$ in
  $\mathcal C$ together with an element $I \stackrel{s}\rightarrow FN$ of
  the $F$-image $FN$ of the apex of the cospan. We shall call the element $I
  \stackrel{s}\rightarrow FN$ the \define{decoration} of the decorated
  cospan. A \define{morphism of decorated cospans}
  \[
    n\colon  \big(X \stackrel{i_X}\longrightarrow N \stackrel{o_Y}\longleftarrow
    Y,\enspace I \stackrel{s}\longrightarrow FN\big) \longrightarrow \big(X
    \stackrel{i'_X}\longrightarrow N' \stackrel{o'_Y}\longleftarrow Y,\enspace I
    \stackrel{s'}\longrightarrow FN'\big)
  \]
  is a morphism $n\colon  N \to N'$ of cospans such that $Fn \circ s = s'$.
\end{definition}

Observe that we may define a category with decorated cospans as objects and
morphisms of decorated cospans as morphisms.

\subsection{Composing decorated cospans}

Our first concern is to show that isomorphism classes of decorated cospans form
the \emph{morphisms} of a category. The objects in this category will be the
objects of $\mc C$.

\begin{proposition} \label{prop.composingdeccospans}
  There is a category $F\mathrm{Cospan}$ of $F$-decorated cospans, with objects
  the objects of $\mathcal C$, and morphisms isomorphism classes of
  $F$-decorated cospans. On representatives of the isomorphism classes,
  composition in this category is given by pushout of cospans in $\mathcal C$
  \[
    \xymatrix{
      && N+_YM \\
      & N \ar[ur]^{j_N} && M \ar[ul]_{j_M} \\
      \quad X \quad \ar[ur]^{i_X} && Y \ar[ul]_{o_Y} \ar[ur]^{i_Y} && \quad Z
      \quad \ar[ul]_{o_Z}
    }
  \]
  paired with the composite
  \[
    I \stackrel{\lambda^{-I}}\longrightarrow I \otimes I \stackrel{s \otimes
    t}\longrightarrow FN \otimes FM \stackrel{\varphi_{N,M}}\longrightarrow
    F(N+M) \stackrel{F[j_N,j_M]}\longrightarrow F(N+_YM)
  \]
  of the tensor product of the decorations with the $F$-image of the copairing
  of the pushout maps.
\end{proposition}

\begin{proof}
  The identity morphism on an object $X$ in a decorated cospan category is
  simply the identity cospan decorated as follows:
  \[
    \left(
    \begin{aligned}
      \xymatrix{
	& X \\  
	X \ar@{=}[ur] && X \ar@{=}[ul]
      }
    \end{aligned}
    ,
    \qquad
    \begin{aligned}
      \xymatrixrowsep{.5pc}
      \xymatrix{
	FX \\
	F\varnothing \ar[u]_{F!} \\
	I \ar[u]_{\varphi_I}
      }
    \end{aligned}
    \right).
  \]
  We must check that the composition defined is well defined on isomorphism
  classes, is associative, and, with the above identity maps, obeys the
  unitality axiom. These are straightforward, but lengthy, exercises in
  using the available colimits and monoidal structure to show that
  the relevant diagrams of decorations commute. We do this now.

\paragraph{Representation independence for composition of isomorphism
classes of decorated cospans.}
Let 
\[
  n\colon  \big(X \stackrel{i_X}\longrightarrow N \stackrel{o_Y}\longleftarrow Y,\enspace I
\stackrel{s}\longrightarrow FN\big) \longrightarrow \big(X \stackrel{i'_X}\longrightarrow N'
\stackrel{o'_Y}\longleftarrow Y,\enspace I \stackrel{s'}\longrightarrow FN'\big)
\]
and
\[
  m\colon  \big(Y \stackrel{i_Y}\longrightarrow M \stackrel{o_Z}\longleftarrow Z,\enspace I
\stackrel{t}\longrightarrow FM\big) \longrightarrow \big(Y \stackrel{i'_Y}\longrightarrow M'
\stackrel{o'_Z}\longleftarrow Z,\enspace I \stackrel{t'}\longrightarrow FM'\big)
\]
be isomorphisms of decorated cospans. We wish to show that the composite of the
decorated cospans on the left is isomorphic to the composite of the decorated
cospans on the right. As discussed in \textsection\ref{sec.cospans}, it is well
known that the composite cospans are isomorphic. It remains for us to check
the decorations agree too. Let $p\colon N+_YM \to N'+_YM'$ be the isomorphism
given by the universal property of the pushout and the isomorphisms $n\colon N
\to N'$ and $m\colon  M \to M'$. Then the two decorations in question are given
by the top and bottom rows of the following diagram.
\[
  \xymatrixrowsep{1pc}
  \xymatrixcolsep{1pc}
  \xymatrix{
    &&&& FN \otimes FM \ar[dd]^{Fn \otimes Fm}_\sim
    \ar[rr]^{\varphi_{N,M}} && F(N+M) \ar[dd]^{F(n+m)}_\sim
    \ar[rr]^{F[j_N,j_M]} && F(N+_YM) \ar[dd]^{Fp}_\sim \\ 
    I \ar[rr]^(.4){\lambda^{-1}} && I\otimes I \ar[urr]^{s \otimes t}
    \ar[drr]_{s' \otimes t'} &
    \qquad\textrm{\tiny(I)} && \textrm{\tiny(F)} && \textrm{\tiny(C)}\\ 
    &&&& FN' \otimes FM' \ar[rr]_{\varphi_{N',M'}} && F(N'+M')
    \ar[rr]_{F[j_{N'},j_{M'}]} && F(N'+_YM')
  }
\]
The triangle (I) commutes as $n$ and $m$ are morphisms of decorated cospans and
$- \otimes -$ is functorial, (F) commutes by the monoidality of $F$, and (C)
commutes by properties of colimits in $\mathcal C$ and the functoriality of $F$.
This proves the claim.

\paragraph{Associativity.}
Suppose we have morphisms
\[
  (X \stackrel{i_X}\longrightarrow N \stackrel{o_Y}\longleftarrow Y,\enspace I
  \stackrel{s}\longrightarrow FN),
\]
\[
  (Y \stackrel{i_Y}\longrightarrow M \stackrel{o_Z}\longleftarrow Z,\enspace I
  \stackrel{t}\longrightarrow FM), 
\]
\[
  (Z \stackrel{i_Z}\longrightarrow P \stackrel{o_W}\longleftarrow W,\enspace I
  \stackrel{u}\longrightarrow FP).
\]
It is well known that composition of isomorphism classes of cospans via
pushout of representatives is associative; this follows from the universal
properties of the relevant colimit. We must check that the pushforward of the
decorations is also an associative process. Write 
\[
  \tilde a\colon  (N+_YM)+_ZP \longrightarrow N+_Y(M+_ZP)
\]
for the unique isomorphism between the two pairwise pushouts constructions
from the above three cospans. Consider then the following diagram, with leftmost
column the decoration obtained by taking the composite of the first two morphisms
first, and the rightmost column the decoration obtained by taking the composite
of the last two morphisms first.
\[
  \xymatrixcolsep{-.58pc}
  \xymatrix{ 
    \scriptstyle F((N+_YM)+_ZP) \ar[rrrrrrrr]^{F\tilde a} &&&&&&&&
    \scriptstyle F(N+_Y(M+_ZP)) \\
    &&&& \textrm{\tiny(C)} \\
    \scriptstyle F((N+_YM)+P) \ar[uu]_{F[j_{N+_YM},j_P]} &&&&&&&& \scriptstyle
    F(N+(M+_ZP)) \ar[uu]^{F[j_N,j_{M+_ZP}]} \\
    && \scriptstyle F((N+M)+P) \ar[ull]_(.35)*+<6pt>_{\scriptstyle F([j_N,J_M]+1_P)}
    \ar[rrrr]^{Fa} &&&& \scriptstyle F(N+(M+P))
    \ar[urr]^(.35)*+<6pt>^{\scriptstyle F(1_N+[j_M,j_P])} \\
    & \textrm{\tiny(F2)} &&&&&& \textrm{\tiny(F3)} \\
    \scriptstyle F(N+_YM)\otimes FP \ar[uuu]^{\varphi_{N+_YM,P}} &&&& 
    \textrm{\tiny(F1)} &&&& \scriptstyle FN\otimes F(M+_ZP)
    \ar[uuu]_{\varphi_{N,M+_ZP}} \\
    && \scriptstyle F(N+M)\otimes FP \ar[ull]^(.6)*+<6pt>^{\scriptstyle
      F[j_N,j_M] \otimes 1_{FP}} \ar[uuu]_{\varphi_{N+M,P}} &&&& \scriptstyle
      FN \otimes F(M+P) \ar[urr]_(.6)*+<6pt>_{\scriptstyle 1_{FN} \otimes
    F[j_M,j_P]} \ar[uuu]^{\varphi_{N,M+P}} \\
    &&& \scriptstyle (FN \otimes FM) \otimes FP \ar[ul]^{\varphi_{N,M} \otimes
    1_{FP}} \ar[rr]^{a} && \scriptstyle FN \otimes (FM \otimes FP)
    \ar[ur]_{\phantom{1}1_{FN} \otimes \varphi_{M,P}} \\ 
    &&&& \textrm{\tiny(D2)} \\
    &&& \scriptstyle (I \otimes I) \otimes I \ar[uu]^{(s \otimes t) \otimes u}
    \ar[rr]^{a} && \scriptstyle I \otimes (I \otimes I) \ar[uu]_{s
    \otimes (t \otimes u)} \\
    &&&& \textrm{\tiny(D1)} \\
    &&&& \scriptstyle I \otimes I \ar[uul]^{\rho^{-1} \otimes 1} \ar[uur]_{1
      \otimes \lambda^{-1}} \\
      &&&& \scriptstyle I \ar[u]^{\lambda^{-1}}
  }
\]
This diagram commutes as (D1) is the triangle coherence equation for the
monoidal category $(\mathcal D,\otimes)$, (D2) is naturality for the associator
$a$, (F1) is the associativity condition for the monoidal functor $F$, (F2)
and (F3) commute by the naturality of $\varphi$, and (C) commutes as it is the
$F$-image of a hexagon describing the associativity of the pushout.  This
shows that the two decorations obtained by the two different orders of
composition of our three morphisms are equal up to the unique isomorphism
$\tilde a$ between the two different pushouts that may be obtained. Our
composition rule is hence associative.

\paragraph{Identity morphisms.} 
We shall show that the claimed identity morphism on $Y$, the decorated cospan
\[
  (Y \stackrel{1_Y}\longrightarrow Y \stackrel{1_Y}\longleftarrow Y,\enspace I
\xrightarrow{F!\circ \varphi_I} FY),
\]
is an identity for composition
on the right; the case for composition on the left is similar. The cospan in
this pair is known to be the identity cospan in $\mathrm{Cospan}(\mathcal C)$.
We thus need to check that, given a morphism 
\[
  (X \stackrel{i}\longrightarrow N
\stackrel{o}\longleftarrow Y,\enspace I \stackrel{s}\longrightarrow FN),
\] 
the composite of the product $s \otimes (F! \circ \varphi_I)$ with the $F$-image
of the copairing \linebreak $[1_N,o]\colon  N+Y \to N$ of the pushout maps is
again the same element $s$; this composite being, by definition, the decoration
of the composite of the given morphism and the claimed identity map. This is
shown by the commutativity of the diagram below, with the path along the lower
edge equal to the aforementioned pushforward.
\[
  \xymatrixcolsep{.96pc}
  \xymatrix{
    I \ar[rrrrrrrrrr]^s \ar[ddrr]_{\lambda_I^{-1}} &&&&&&&&&& FN \\
    &&&& \textrm{\tiny(D1)} &&&& \textrm{\tiny(F1)}\\
    && I \otimes I \ar[rr]^{s\otimes 1} \ar[ddrrrr]_{s \otimes (F! \circ
    \varphi_I)} && FN \otimes I \ar[uurrrrrr]^{\rho_{FN}} \ar[rr]^(.57){1_{FN}
  \otimes \varphi_I}
    && FN \otimes F\varnothing \ar[rr]^{\varphi_{N,\varnothing}} \ar[dd]_{1_{FN}
  \otimes F!} && F(N+\varnothing) \ar[uurr]^(.2){F\rho_N=F[1_N,!]}
      \ar[dd]_{F(1_N+!)} \quad \textrm{\tiny(C)}\\
      &&&&& \textrm{\tiny(D2)} && \textrm{\tiny(F2)}\\
      &&&&&& FN \otimes FY \ar[rr]_{\varphi_{N,M}} && F(N+Y)
      \ar[uuuurr]_{F[1_N,i_Y]}
  }
\]
This diagram commutes as each subdiagram commutes: (D1) commutes by the
naturality of $\rho$, (D2) by the functoriality of the monoidal product in
$\mathcal D$, (F1) by the unit axiom for the monoidal functor $F$, (F2) by the
naturality of $\varphi$, and (C) due to the properties of colimits in $\mathcal
C$ and the functoriality of $F$.

We thus have a category.
\end{proof}

\begin{remark}
While at first glance it might seem surprising that we can construct a
composition rule for decorations $s\colon  I\to FN$ and $t\colon  I \to FM$ just from
monoidal structure, the copair $[j_N,j_M]\colon  N+M \to N+_YM$ of the pushout maps
contains the data necessary to compose them. Indeed, this is the key insight of the
decorated cospan construction. To wit, the coherence maps for the lax monoidal
functor allow us to construct an element of $F(N+M)$ from the monoidal product
$s \otimes t$ of the decorations, and we may then post-compose with $F[j_N,j_M]$ to
arrive at an element of $F(N+_YM)$. The map $[j_N,j_M]$ encodes the
identification of the image of $Y$ in $N$ with the image of the same in $M$, and
so describes merging the `overlap' of the two decorations.
\end{remark}

\subsection{The hypergraph structure}
The main theorem of this chapter is that when \emph{braided} monoidal structure
is present, the category of decorated cospans is a hypergraph category, and
moreover one into which the category of `undecorated' cospans widely embeds.
Indeed, this embedding motivates the monoidal and hypergraph structures we put
on $F\mathrm{Cospan}$. This section is devoted to stating and proving this
theorem.

\begin{theorem} \label{thm:fcospans}
  Let $\mathcal C$ be a category with finite colimits, $(\mathcal D, \otimes)$ a
  braided monoidal category, and $(F,\varphi)\colon  (\mathcal C,+) \to (\mathcal D,
  \otimes)$ be a lax braided monoidal functor. Then we may give
  $F\mathrm{Cospan}$ a symmetric monoidal and hypergraph structure such that
  there is a wide embedding of hypergraph categories
  \[
    \mathrm{Cospan(\mathcal C)} \hooklongrightarrow F\mathrm{Cospan}.
  \]
\end{theorem}

We first prove a lemma.  Recall that the identity decorated cospan has apex
decorated by $I \stackrel{\varphi_I}\longrightarrow F\varnothing
\stackrel{F!}\longrightarrow FX$. Given any cospan $X \to N \leftarrow Y$, we
call the decoration $I \stackrel{\varphi_I}\longrightarrow F\varnothing
\stackrel{F!}\longrightarrow FN$ the \define{empty decoration} on $N$. For
example, in the case of cospans of finite sets decorated by graphs, the empty
decoration is simply the graph with no edges. Our lemma shows that empty
decorations indeed are empty of content: when composed with other decorations,
they have no impact.

\begin{lemma} \label{lem.emptydecorations}
  Let
\[
  (X \stackrel{i_X}\longrightarrow N
\stackrel{o_Y}\longleftarrow Y,\enspace I \stackrel{s}\longrightarrow FN),
\]
be a decorated cospan, and suppose we have an empty-decorated cospan 
\[
  (Y \stackrel{i_Y}\longrightarrow M \stackrel{o_Z}\longleftarrow Z,\enspace I
  \stackrel{\varphi\circ F!}\longrightarrow FM).
\]
Then the composite of these decorated cospans is 
\[
  \big(X \xrightarrow{j_N \circ i_X} N+_YM \xleftarrow{j_M \circ o_Z} Z,\enspace 
  I \xrightarrow{Fj_N \circ s} F(N+_YM)\big).
\]
In particular, the decoration on the composite is the decoration $s$ pushed
forward along the $F$-image of the map $j_N\colon N \to N+_YM$ to become an
$F$-decoration on $N+_YM$. The analogous statement also holds for composition
with an empty-decorated cospan on the left.
\end{lemma}
\begin{proof}
As is now familiar, a statement of this sort is proved by a large commutative
diagram:
\[
  \xymatrixcolsep{.4pc}
  \xymatrix{
    I \ar[rrrrrr]^s \ar[ddrr]_{\rho_I^{-1}} &&&&&& FN \ar[rrrr]^{Fj_N} &&&& F(N+_YM) \\
    &&& \textrm{\tiny(D1)} &&& \textrm{\tiny(F1)} && \textrm{\tiny(C1)} \\
    && I \otimes I \ar[rr]^{s\otimes 1} \ar[ddrrrr]_{s \otimes (F! \circ
    \varphi_I)} && FN \otimes I \ar[uurr]^{\rho_{FN}} \ar[rr]^{1_{FN}
  \otimes \varphi_I}
    && FN \otimes F\varnothing \ar[rr]^{\varphi_{N,\varnothing}} \ar[dd]_{1_{FN}
  \otimes F!} && F(N+\varnothing) \ar[uull]_{F\rho_N=F[1_N,!]}
  \ar[uurr]^{F[j_N,!]}
      \ar[dd]_{F(1_N+!)} \quad \textrm{\tiny(C2)}\\
      &&&&& \textrm{\tiny(D2)} && \textrm{\tiny(F2)}\\
      &&&&&& FN \otimes FM \ar[rr]_{\varphi_{N,M}} && F(N+M)
      \ar[uuuurr]_{F[j_N,j_M]}
  }
\]
Note that identity decorated cospans are empty decorated. The subdiagrams in
this diagram commute for the same reasons as their corresponding regions in the
diagram for identity morphisms (see proof of Proposition
\ref{prop.composingdeccospans}). 
\end{proof}

\begin{proof}[Proof of Theorem \ref{thm:fcospans}]
We first prove that we have a wide embedding of categories, and then use this to
transfer monoidal and hypergraph structure onto $F\mathrm{Cospan}$.

\paragraph{Embedding.} 
Define a functor
\[
  \mathrm{Cospan(\mathcal C)} \hooklongrightarrow F\mathrm{Cospan}.
\]
mapping each object of $\mathrm{Cospan(\mathcal C)}$ to itself as an object
of $F\mathrm{Cospan}$, and each cospan in $\mathcal C$ to the same cospan
decorated with the empty decoration on its apex. 

To check this is functorial, it suffices to check that the composite of two
empty-decorated cospans is again empty-decorated. This is an immediate
consquence of Lemma \ref{lem.emptydecorations}. Glancing at the definition, it
is also clear that this functor is faithful and bijective-on-objects. We thus
have a wide embedding $\mathrm{Cospan(\mathcal{C})} \hookrightarrow
F\mathrm{Cospan}$.

\paragraph{Monoidal structure.} 

We define the monoidal product of objects $X$ and $Y$ of $F\mathrm{Cospan}$ to
be their coproduct $X+Y$ in $\mathcal C$, and define the monoidal product of
decorated cospans 
\[
  (X \stackrel{i_X}\longrightarrow N
  \stackrel{o_Y}\longleftarrow Y,\enspace I \stackrel{s}\longrightarrow FN)
  \qquad \mbox{and} \qquad
  (X' \stackrel{i_{X'}}\longrightarrow N' \stackrel{o_{Y'}}\longleftarrow
  Y',\enspace I \stackrel{t}\longrightarrow FN')
\]
to be 
\[
  \left(
  \begin{aligned}
    \xymatrix{
      & N+N' \\  
      X+X' \ar[ur]^{i_X+i_{X'}} && Y+Y' \ar[ul]_{o_Y+o_{Y'}}
    }
  \end{aligned}
  ,
  \qquad
  \begin{aligned}
    \xymatrixrowsep{.8pc}
    \xymatrix{
      F(N+N') \\
      FN \otimes FN' \ar[u]_{\varphi_{N,N'}}\\
      I \otimes I \ar[u]_{s \otimes t} \\
      I \ar[u]_{\lambda^{-1}}
    }
  \end{aligned}
  \right).
\]
Using the braiding in $\mathcal D$, we can show that this proposed monoidal
product is functorial. 

Indeed, suppose we have decorated cospans
\[
  (X \stackrel{i_X}\longrightarrow N \stackrel{o_Y}\longleftarrow Y,\enspace I
  \stackrel{s}\longrightarrow FN),
  \qquad
  (Y \stackrel{i_Y}\longrightarrow M \stackrel{o_Z}\longleftarrow Z,\enspace I
  \stackrel{t}\longrightarrow FM), 
\]
\[
  (U \stackrel{i_U}\longrightarrow P \stackrel{o_V}\longleftarrow V,\enspace I
  \stackrel{u}\longrightarrow FP),
  \qquad
  (V \stackrel{i_V}\longrightarrow Q \stackrel{o_W}\longleftarrow W,\enspace I
  \stackrel{v}\longrightarrow FQ). 
\]
We must check the so-called interchange law: that the composite of the
column-wise monoidal products is equal to the monoidal product of the row-wise
composites.

Again, for the cospans we take this equality as familiar fact. Write 
\[
  b\colon  (N+P)+_{(Y+V)}(M+Q) \stackrel{\sim}{\longrightarrow} (N+_YM)+(P+_{V}Q).
\]
for the isomorphism between the two resulting representatives of the isomorphism
class of cospans. The two resulting decorations are then given by the leftmost 
and rightmost columns respectively of the diagram below.
\[
  \xymatrixcolsep{1.4pc}
  \xymatrixrowsep{1.2pc}
  \xymatrix{ 
    \scriptstyle F((N+P)+_{(Y+V)}(M+Q)) \ar[rrrr]^{Fb}_{\sim} &&&&
    \scriptstyle F((N+_YM)+(P+_VQ)) \\
    & \textsc{\tiny(C)} \\
    \scriptstyle F((N+P)+(M+Q)) \ar[uu]^{F[j_{N+P},j_{M+Q}]} \ar@{.>}[uurrrr] &&&
    \textsc{\tiny(F2)} & \scriptstyle F(N+_YM)\otimes F(P+_VQ)
    \ar[uu]_{\varphi_{N+_YM,P+_VQ}} \\
    \\
    \scriptstyle F(N+P)\otimes F(M+Q) \ar[uu]^{\varphi_{N+P,M+Q}} &&&&
    \scriptstyle F(N+M)\otimes F(P+Q) \ar@{.>}[uullll]
    \ar[uu]_{F[j_N,j_M] \otimes F[j_P,j_Q]} \\
    && \textsc{\tiny(F1)} \\
    \scriptstyle (FN \otimes FP) \otimes (FM \otimes FQ) \ar[uu]^{\varphi_{N,P} \otimes
    \varphi_{M,Q}} \ar@{.>}[rrrr] &&&& \scriptstyle (FN \otimes FM) \otimes (FP \otimes FQ)
    \ar[uu]_{\varphi_{N,M} \otimes \varphi_{P,Q}} \\
    && \textsc{\tiny(D)} \\
    && \scriptstyle (I\otimes I)\otimes(I\otimes I) \ar[uull]^{(s\otimes u) \otimes (t \otimes
    v)} \ar[uurr]_{(s \otimes t) \otimes (u \otimes v)} \\
    && \scriptstyle I \ar[u]_{(\lambda^{-1}\otimes\lambda^{-1})\circ\lambda^{-1}}
  }
\]
These two decorations are related by the isomorphism $b$ as the diagram
commutes. We argue this more briefly than before, as the basic structure of
these arguments is now familiar to us. Briefly then, there exist dotted arrows
of the above types such that the subdiagram (D) commutes by the naturality of
the associators and braiding in $\mathcal D$, (F1) commutes by the coherence
diagrams for the braided monoidal functor $F$, (F2) commutes by the naturality of
the coherence map $\varphi$ for $F$, and (C) commutes by the properties of
colimits in $\mathcal C$ and the functoriality of $F$. 

Using now routine methods, it also is straightforward to show that the monoidal
product of identity decorated cospans on objects $X$ and $Y$ is the identity
decorated cospan on $X+Y$; for the decorations this amounts to the observation
that the monoidal product of empty decorations is again an empty decoration.

Choosing associator, unitors, and braiding in $F\mathrm{Cospan}$ to be the
images of those in $\mathrm{Cospan(\mathcal{C})}$, we have a symmetric
monoidal category. These transformations remain natural transformations when
viewed in the category of $F$-decorated cospans as they have empty
decorations.

We consider the case of the left unitor in detail; the naturality of the right
unitor, associator, and braiding follows similarly, using the relevant axiom
where here we use the left unitality axiom. 

Given a decorated cospan $(X \stackrel{i}\longrightarrow N
\stackrel{o}\longleftarrow Y,\enspace I \stackrel{s}\longrightarrow FN)$, we
must show that the square of decorated cospans
\[
  \xymatrix{
    X+\varnothing \ar[r]^{i+1} \ar[d]_{\lambda_{\mathcal C}} & N+\varnothing & Y+\varnothing
    \ar[l]_{i+1} \ar[d]^{\lambda_{\mathcal C}} \\
    X \ar[r]_{i} & N & Y \ar[l]^{o} 
  }
\]
commutes, where the $\lambda_{\mathcal C}$ are the maps of the left unitor in
$\mathcal C$ considered as empty-decorated cospans, the top cospan has
decoration 
\[
  I \stackrel{\lambda^{-1}_{\mathcal D}}{\longrightarrow} I \otimes I \stackrel{s \otimes
  \varphi_I}\longrightarrow FN \otimes F\varnothing
  \stackrel{\varphi_{N,\varnothing}}\longrightarrow F(N+\varnothing),
\]
and the lower cospan has decoration $I \stackrel{s}{\longrightarrow} FN$. 

Now as the $\lambda$ are isomorphisms in $\mathcal C$ and have empty
decorations, the composite through the upper right corner is the decorated
cospan
\[
  \big(X+\varnothing \xrightarrow{i+1} N+\varnothing
  \xleftarrow{(o+1)\circ \lambda^{-1}_{\mathcal C}} Y,\enspace I
  \xrightarrow{\varphi_{N,\varnothing} \circ (s \otimes \varphi_1) \circ
  \lambda^{-1}_{\mathcal D}} F(N+\varnothing)\big).
\]
The composite through the lower left corner is the decorated cospan
\[
  (X+\varnothing \stackrel{[i,!]}\longrightarrow N
\stackrel{o}\longleftarrow Y,\enspace 1 \stackrel{s}\longrightarrow FN).
\]
Then $\lambda_{\mathcal C}\colon  N+\varnothing \rightarrow N$ gives an isomorphism
between these two cospans, and the naturality of the left unitor and the left
unitality axiom in $\mathcal D$ imply that $\lambda_{\mc C}$ is in fact an
isomorphism of decorated cospans:
\[
  \xymatrix{
    I \ar[rrr]^{s} \ar[d]_{\lambda^{-1}_{\mathcal D}} &&& FN \\
    I \otimes I \ar[r]_{s \otimes 1} & FN \otimes I \ar[r]_{1 \otimes \varphi_I}
    \ar[urr]^{\lambda_{\mathcal D}} & FN \otimes F\varnothing
    \ar[r]_{\varphi_{N,\varnothing}} & F(N+\varnothing).
    \ar[u]_{F\lambda_{\mathcal C}}
  }
\]
Thus $\lambda$ is indeed a natural transformation; again the naturality of the
remaining coherence maps can be proved similarly. Moreover, the coherence maps
obey the required coherence laws as they are images of maps that obey these laws
in $\mathrm{Cospan(\mathcal{C})}$. 

\paragraph{Hypergraph structure.}
Similarly, to arrive at the hypergraph structure on \linebreak
$F\mathrm{Cospan}$, we simply equip each object $X$ with the image of the
special commutative Frobenius monoid specified by the hypergraph structure of
$\mathrm{Cospan(\mathcal{C})}$. The axioms of hypergraph structure follow from
the functoriality of our embedding. Moreover, it is evident that this choice of
structures implies the above wide embedding is a hypergraph functor.
\end{proof}

In summary, the hypergraph category $F\mathrm{Cospan}$ comprises:
\smallskip

\begin{center}
  \begin{tabular}{| c | p{.65\textwidth} |}
    \hline
    \multicolumn{2}{|c|}{The hypergraph category $(F\mathrm{Cospan},+)$} \\
    \hline
    \textbf{objects} & the objects of $\mathcal C$ \\ 
    \textbf{morphisms} & isomorphism classes of $F$-decorated cospans in
    $\mathcal C$\\ 
    \textbf{composition} & given by pushout \\
    \textbf{monoidal product} & the coproduct in $\mathcal C$ \\
    \textbf{coherence maps} &  maps from $\cospan(\mc C)$ with empty decoration \\
    \textbf{hypergraph maps} & maps from $\cospan(\mc C)$ with empty decoration
    \\
    \hline
  \end{tabular}
\end{center}
\smallskip

Note that if the monoidal unit in $(\mathcal D,\otimes)$ is the initial object,
then each object only has one possible decoration: the empty decoration. This
immediately implies the following corollary.
\begin{corollary}
  Let $1_{\mathcal C}\colon (\mathcal C,+) \to (\mathcal C,+)$ be the identity functor
  on a category $\mathcal C$ with finite colimits. Then
  $\mathrm{Cospan}(\mathcal C)$ and $1_{\mathcal C}\mathrm{Cospan}$ are
  isomorphic as hypergraph categories.
\end{corollary}

Thus we see that there is always a hypergraph functor between decorated cospan
categories $1_{\mathcal C}\mathrm{Cospan} \rightarrow F\mathrm{Cospan}$. This
provides an example of a more general way to construct hypergraph functors
between decorated cospan categories. We detail this in the next section. First
though, we conclude this section by remarking on the bicategory of decorated
cospans.

\subsection{A bicategory of decorated cospans} \label{ssec.bicatdeccospan}
As shown by B\'enabou \cite{Ben67}, cospans are most naturally thought of as
1-morphisms in a bicategory.  To obtain the category we are calling
$\mathrm{Cospan}(\mc C)$, we `decategorify' this bicategory by discarding
2-morphisms and identifying isomorphic 1-morphisms. This suggests that decorated
cospans might also be most naturally thought of as 1-morphisms in a bicategory.
Indeed this is the case.  

Recall that a morphism of $F$-decorated cospans is a morphism $n$ in $\mc C$ such
that the diagrams
\[
  \begin{aligned}
    \xymatrix{
      & N \ar[dd]^n \\  
      X \ar[ur]^{i} \ar[dr]_{i'} && Y \ar[ul]_{o} \ar[dl]^{o'}\\
      & N'
    }
  \end{aligned}
  \qquad \mbox{and}
  \qquad
  \begin{aligned}
    \xymatrix{
      & FN \ar[dd]^{Fn} \\
      1 \ar[ur]^{s} \ar[dr]_{s'} \\
      & FN'
    }
  \end{aligned}
\]
commute. Then Courser has proved the following proposition.
\begin{proposition}
  There exists a symmetric monoidal bicategory with:
\begin{center}
  \begin{tabular}{| c | p{.65\textwidth} |}
    \hline
    \multicolumn{2}{|c|}{The symmetric monoidal bicategory $(F\mathrm{Cospan},+)$} \\
    \hline
    \textbf{objects} & the objects of $\mathcal C$ \\ 
    \textbf{morphisms} & $F$-decorated cospans in
    $\mathcal C$\\ 
    \textbf{2-morphisms} & morphisms of $F$-decorated cospans in $\mc C$ \\
    \hline
  \end{tabular}
\end{center}
\end{proposition}
Further details can be found in Courser's paper \cite{Cou16}. These 2-morphisms
could be employed to model transformations, coarse grainings, rewrite rules, and
so on of open systems and their representations.

\section{Functors between decorated cospan categories} \label{sec:dcf}

Decorated cospans provide a setting for formulating various operations that we
might wish to enact on the decorations, including the composition of these
decorations, both sequential and monoidal, as well as dagger, dualising, and
other operations afforded by the hypergraph structure. In this section we
observe that these operations are formulated in a systematic way, so that
transformations of the decorating structure---that is, monoidal transformations
between the lax monoidal functors defining decorated cospan categories---respect
these operations. 

\begin{theorem} \label{thm:decoratedfunctors}
  Let $\mathcal C$, $\mathcal C'$ be categories with finite colimits, abusing
  notation to write the coproduct in each category $+$, and $(\mathcal D,
  \otimes)$, $(\mathcal D',\boxtimes)$ be braided monoidal categories. Further let
  \[
    (F,\varphi)\colon  (\mathcal C,+) \longrightarrow (\mathcal D,\otimes)
    \qquad \mbox{and} \qquad
    (G,\gamma)\colon  (\mathcal C',+) \longrightarrow (\mathcal D',\boxtimes)
  \]
  be lax braided monoidal functors. This gives rise to decorated cospan
  categories $F\mathrm{Cospan}$ and $G\mathrm{Cospan}$. 

  Suppose then that we have a finite colimit-preserving functor $A\colon  \mathcal C
  \to \mathcal C'$ with accompanying natural isomorphism $\alpha\colon  A(-)+A(-)
  \Rightarrow A(-+-)$, a lax monoidal functor $(B,\beta)\colon  (\mathcal D, \otimes)
  \to (\mathcal D', \boxtimes)$, and a monoidal natural transformation $\theta\colon 
  (B \circ F, B\varphi\circ\beta) \Rightarrow (G \circ A, G\alpha\circ\gamma)$.
  This may be depicted by the diagram:
  \[
    \xymatrixcolsep{3pc}
    \xymatrixrowsep{3pc}
    \xymatrix{
      (\mathcal C,+) \ar^{(F,\varphi)}[r] \ar_{(A,\alpha)}[d] \drtwocell
      \omit{_\:\theta} & (\mathcal D,\otimes) \ar^{(B,\beta)}[d]  \\
      (\mathcal C',+) \ar_{(G,\gamma)}[r] & (\mathcal D',\boxtimes).
    }
  \]

  Then we may construct a hypergraph functor 
  \[
    (T, \tau)\colon  F\mathrm{Cospan} \longrightarrow G\mathrm{Cospan}
  \]
  mapping objects $X \in F\mathrm{Cospan}$ to $AX \in G\mathrm{Cospan}$, and
  morphisms 
  \[
    \left(
    \begin{aligned}
      \xymatrix{
	& N \\  
	X \ar[ur]^{i} && Y \ar[ul]_{o}
      }
    \end{aligned}
    ,
    \quad
    \begin{aligned}
      \xymatrix{
	FN \\
	I_{\mathcal D} \ar[u]_{s}
      }
    \end{aligned}
    \right)
    \qquad
    to
    \qquad
    \left(
    \begin{aligned}
      \xymatrix{
	& AN \\  
	AX \ar[ur]^{Ai} && AY \ar[ul]_{Ao}
      }
    \end{aligned}
    ,
    \quad
    \begin{aligned}
      \xymatrixrowsep{.8pc}
      \xymatrix{
	GAN \\
	BFN \ar[u]_{\theta_N}\\
	BI_{\mathcal D} \ar[u]_{Bs} \\
	I_{\mathcal D'} \ar[u]_{\beta_I}
      }
    \end{aligned}
    \right).
  \]
  Moreover, $(T,\tau)$ is a strict monoidal functor if and only if $(A,\alpha)$
  is.
\end{theorem}

\begin{proof}
  We must prove that $(T,\tau)$ is a functor, is strong symmetric monoidal, and
  that it preserves the special commutative Frobenius structure on each object.

  \paragraph{Functoriality.}
  Checking the functoriality of $T$ is again an exercise in applying the
  properties of structure available---in this case the colimit-preserving nature
  of $A$ and the monoidality of $(\mathcal D,\boxtimes)$, $(B,\beta)$, and
  $\theta$---to show that the relevant diagrams of decorations commute. 

  In detail, let 
  \[
    (X \stackrel{i_X}\longrightarrow N \stackrel{o_Y}\longleftarrow Y, \enspace I
    \stackrel{s}\longrightarrow FN)
    \qquad \mbox{and} \qquad
    (Y \stackrel{i_Y}\longrightarrow M \stackrel{o_Z}\longleftarrow Z, \enspace I
    \stackrel{t}\longrightarrow FM), 
  \]
  be morphisms in $F\mathrm{Cospan}$. As the composition of the cospan part is by
  pushout in $\mathcal C$ in both cases, and as $T$ acts as the colimit preserving
  functor $A$ on these cospans, it is clear that $T$ preserves composition of
  isomorphism classes of cospans. Write
  \[
    c\colon  AX+_{AY}AZ \stackrel\sim\longrightarrow A(X+_YZ)
  \]
  for the isomorphism from the cospan obtained by composing the $A$-images of the
  above two decorated cospans to the cospan obtained by taking the $A$-image of their
  composite. To see that this extends to an isomorphism of decorated cospans,
  observe that the decorations of these two cospans are given by the rightmost and
  leftmost columns respectively in the following diagram:
  \[
    \xymatrixcolsep{.5pc}
    \xymatrixrowsep{.9pc}
    \xymatrix{ 
      GA(N+_YM) &&&&&&&& \ar[llllllll]_{Gc}^{\sim} G(AN+_{AY}AM) \\
      &&&&& \textsc{\tiny(A)} \\
      BF(N+_YM) \ar[uu]^{\theta_{N+YM}} && \textsc{\tiny(T2)} && GA(N+M)
      \ar[uullll]_{GA[j_N,j_M]} &&&& G(AN+AM) \ar[uu]_{G[j_{AN},j_{AM}]}
      \ar[llll]^{\sim}_{G\alpha} \\
      \\
      BF(N+M) \ar[uu]^{BF[j_N,j_M]} \ar[uurrrr]_{\theta_{N,M}} &&&&
      \textsc{\tiny(T1)} &&&& GAN \boxtimes GAM \ar[uu]_{\gamma_{AN,AM}} \\
      \\
      B(FN \otimes FM) \ar[uu]^{B\varphi_{N,M}} &&&&&&&& BFN \boxtimes BFM
      \ar[uu]_{\theta_N \boxtimes \theta_M} \ar[llllllll]_{\beta_{FN,FM}} \\
      &&&& \textsc{\tiny(B2)} \\
      B(I_{\mathcal D} \otimes I_{\mathcal D}) \ar[uu]^{B(s \otimes t)} &&&&&&&&
      BI_{\mathcal D} \boxtimes BI_{\mathcal D} \ar[uu]_{Bs \boxtimes Bt}
      \ar[llllllll]_{\beta_{I,I}} \\
      &&& \textsc{\tiny(B1)} &&&& \textsc{\tiny(D2)} \\
      BI_{\mathcal D} \ar[uu]^{\beta\lambda^{-1}_I} \ar[rrrr]^{\lambda^{-1}_{BI}}
      &&&& I_{\mathcal D'} \boxtimes BI_{\mathcal D} \ar[uurrrr]^{\beta_I \boxtimes 1}
      &&&& I_{\mathcal D'} \boxtimes I_{\mathcal D'} \ar[uu]^{\beta_I \boxtimes \beta_I}
      \ar[llll]_{1 \boxtimes \beta_I} \\
      &&&& \textsc{\tiny(D1)} \\
      &&&& I_{\mathcal D'} \ar[uullll]^{\beta_I} \ar[uurrrr]_{\lambda^{-1}_I}
    }
  \]
  From bottom to top, \textsc{(D1)} commutes by the naturality of $\lambda$,
  \textsc{(D2)} by the functoriality of the monoidal product $-\boxtimes-$,
  \textsc{(B1)} by the unit law for $(B,\beta)$, \textsc{(B2)} by the
  naturality of $\beta$, \textsc{(T1)} by the monoidality of the natural
  transformation $\theta$, \textsc{(T2)} by the naturality of $\theta$, and
  \textsc{(A)} by the colimit preserving property of $A$ and the functoriality of
  $G$.

We must also show that identity morphisms are mapped to identity morphisms. Let 
\[
  (X \stackrel{1_X}\longrightarrow X \stackrel{1_X}\longleftarrow X, \enspace I
  \stackrel{F!\circ \varphi_I}\longrightarrow FX)
\]
be the identity morphism on some object $X$ in the category of $F$-decorated
cospans. Now this morphism has $T$-image
\[
  (AX \xrightarrow{1_{AX}} AX \xleftarrow{1_{AX}} AX, \enspace 
  I \xrightarrow{\theta_X \circ B(F!\circ \varphi_I) \circ \beta_I} GAX).
\]
But we have the following diagram
\[
  \xymatrixcolsep{2pc}
  \xymatrixrowsep{.3pc}
  \xymatrix{ 
    &&&& BF\varnothing_{\mathcal C} \ar[ddrr]^{BF!}
    \ar[dddd]^{\theta_\varnothing} \\
    \\
    && BI_{\mathcal D} \ar[uurr]^{B\varphi_I} &&&& BFX \ar[ddrr]^{\theta_X} \\
    &&& \textsc{\tiny (T1)} && \textsc{\tiny (T2)} \\
    I_{\mathcal D'} \ar[uurr]^{\beta_I} \ar[rr]_{\gamma_I}
    \ar[ddddrrrr]_{\gamma_I} && G\varnothing_{\mathcal C'} \ar[rr]_{G\alpha_I}
    && GA\varnothing_{\mathcal C} \ar[rrrr]_{GA!} &&&& GAX \\
    &&& \textsc{\tiny (A1)} && \textsc{\tiny (A2)} \\
    \\
    \\
    &&&& G\varnothing_{\mathcal C'} \ar[uuuu]^{\sim}_{G!} \ar[uuuurrrr]_{G!}
  }
\]
Here \textsc{(A1)} and \textsc{(A2)} commute by the fact $A$ preserves colimits,
\textsc{(T1)} commutes by the unit law for the monoidal natural transformation
$\theta$, and \textsc{(T2)} commutes by the naturality of $\theta$.

Thus we have the equality of decorations $\theta_X \circ B(F! \circ \varphi_I)
\circ \beta_I = G! \circ \gamma_I\colon  I \to GAX$, and so $T$ sends identity
morphisms to identity morphisms.

\paragraph{Monoidality.} The coherence maps of the functor $T$ are given by the
coherence maps for the monoidal functor $A$, viewed now as cospans with the
empty decoration. That is, we define the coherence maps $\tau$ to be the
collection of isomorphisms
\[
  \tau_I = 
  \left(
  \begin{aligned}
    \xymatrix{
      & A\varnothing_{\mathcal C} \\  
      \varnothing_{\mathcal C'} \ar[ur]^{\alpha_I} && A\varnothing_{\mathcal
      C} \ar@{=}[ul]
    }
  \end{aligned}
  ,
  \qquad
  \begin{aligned}
    \xymatrixrowsep{.9pc}
    \xymatrix{
      GA\varnothing_{\mathcal C} \\
      G\varnothing_{\mathcal C'} \ar[u]_{G!} \\
      I_{\mathcal D'} \ar[u]_{\gamma_I}
    }
  \end{aligned}
  \right),
\]
\[
  \tau_{X,Y}=
  \left(
  \begin{aligned}
    \xymatrix{
      & A(X+Y) \\  
      AX+AY \ar[ur]^{\alpha_{X,Y}} && A(X+Y) \ar@{=}[ul]
    }
  \end{aligned}
  ,
  \qquad
  \begin{aligned}
    \xymatrixrowsep{.9pc}
    \xymatrix{
      GA(X+Y) \\
      G\varnothing_{\mathcal C'} \ar[u]_{G!} \\
      I_{\mathcal D} \ar[u]_{\gamma_I}
    }
  \end{aligned}
  \right),
\]
where $X$, $Y$ are objects of $G\mathrm{Cospan}$. As $(A,\alpha)$ is already
strong symmetric monoidal and $\tau$ merely views these maps in $\mathcal C$
as empty-decorated cospans in $G\mathrm{Cospan}$, $\tau$ is natural in $X$ and
$Y$, and obeys the required coherence axioms for $(T,\tau)$ to also be strong
symmetric monoidal. 

Indeed, the monoidality of a functor $(T,\tau)$ has two aspects: the naturality
of the transformation $\tau$, and the coherence axioms. We discuss the former;
since $\tau$ is just an empty-decorated version of $\alpha$, the latter then
immediately follow from the coherence of $\alpha$.

The naturality of $\tau$ may be proved via the same method as that employed for
the naturality of the coherence maps of decorated cospan categories: we first
use the composition of empty decorations to compute the two paths around the
naturality square, and then use the naturality of the coherence map $\alpha$ to
show that these two decorated cospans are isomorphic.

In slightly more detail, suppose we have decorated cospans
\[
  (X \stackrel{i_X}\longrightarrow N
  \stackrel{o_Z}\longleftarrow Z,\enspace I \stackrel{s}\longrightarrow FN) \quad
  \textrm{and} \quad (Y \stackrel{i_Y}\longrightarrow M
  \stackrel{o_W}\longleftarrow W,\enspace I \stackrel{t}\longrightarrow FM).
\]
Then naturality demands that the cospans
\[
  AX+AY \xrightarrow{Ai_X+Ai_Y} AN+AM \xleftarrow{(o_Z+o_W)\circ\alpha^{-1}}
  A(Z+W)
\]
and
\[
  AX+AY \xrightarrow{A(i_X+i_Y)\circ\alpha} A(N+M) \xleftarrow{A(o_Z+o_W)}
  A(Z+W)
\]
are isomorphic as decorated cospans, with decorations the top and bottom rows of
the diagram below respectively.
\[
  \xymatrixcolsep{3pc}
  \xymatrixrowsep{.7pc}
  \xymatrix{
    & \scriptstyle  \scriptstyle I \otimes I \ar[r]^(.4){\beta_I \otimes \beta_I} & \scriptstyle  BI \otimes BI \ar[r]^(.4){Bs
    \otimes Bt} & \scriptstyle  BFN \otimes BFM \ar[r]^{\theta_N \otimes \theta_M} & \scriptstyle  GAN \otimes
    GAM \ar[r]^{\gamma_{AN,AM}} & \scriptstyle  G(AN+AM) \ar[dd]^{G\alpha_{N,M}} \\
    \scriptstyle I \ar[ur]^{\lambda^{-1}} \ar[dr]_{\beta_I} \\
    & \scriptstyle  BI \ar[r]_(.4){B\lambda^{-1}} & \scriptstyle  B(I \otimes I) \ar[r]_(.4){B(s\otimes t)} & \scriptstyle  B(FN
    \otimes FM) \ar[r]_{B\varphi_{N,M}} & \scriptstyle  B(F(N+M))
    \ar[r]_{\theta_{N,M}} & \scriptstyle  G(A(N+M)).
  }
\]
As it is a subdiagram of the large functoriality commutative diagram, this
diagram commutes. The diagrams required for $\alpha_{N,M}$ to be a morphism of
cospans also commute, so our decorated cospans are indeed isomorphic. This
proves $\tau$ is a natural tranfomation.  

Moreover, as $A$ is coproduct-preserving and the Frobenius structures on
$F\mathrm{Cospan}$ and $G\mathrm{Cospan}$ are built using various copairings
of the identity map, $(T,\tau)$ preserves the hypergraph structure.

Finally, it is straightforward to observe that the maps $\tau$ are identity
maps if and only if the maps $\alpha$ are, so $(T,\tau)$ is a strict monoidal
functor if and only $(A,\alpha)$ is.
\end{proof}

When the decorating structure comprises some notion of topological diagram, such
as a graph, these natural transformations $\theta$ might describe some semantic
interpretation of the decorating structure. In this setting the above theorem
might be used to construct semantic functors for the decorated cospan category
of diagrams. It has the limitation, however, of only being able to construct
semantic functors to other decorated cospan categories. Note that this idea of
semantic functor is loosely inspired by, but not strictly the same as, Lawvere's
functorial semantics \cite{Law}.

\section{Decorations in $\Set$ are general} \label{sec.setdecs}

So far we have let decorations lie in any braided monoidal category. While this
gives us greater choice in the functors we may use for decorated cospan
categories, if we are interested in constructing a particular hypergraph
category as a decorated cospan category, it is more general than we need. In
this section, we prove an observation of Sam Staton that it is general enough to
let decorations lie in the symmetric monoidal category $(\Set,\times)$ of sets,
functions, and the cartesian product. 

The key observation is that decorated cospans only make use of the sets of
monoidal elements $x\maps I \to X$ in the decorating category $\mc D$. To
extract this information, use the covariant hom functor $\mathcal D(I,-)\maps (\mc
D,\ot) \to (\Set,\times)$. This takes each object $X$ in $\mc D$ to the homset
$\mc D(I,X)$, and each morphism $f\maps X \to Y$ to the function 
\begin{align*}
  \mc D(I,f)\maps \mc D(I,X) &\longrightarrow \mc D(I,Y); \\ 
  x &\longmapsto f \circ x.
\end{align*}
Write $1=\{\bullet\}$ for the monoidal unit of $(\Set,\times)$.

\begin{proposition} \label{prop.monglobalsecs}
  Let $(\mathcal D, \otimes)$ be a braided monoidal category. Define maps  
  \begin{align*}
    d_1\maps 1 &\longrightarrow \mc D(I,I); \\
    \bullet &\longmapsto \idn_I
  \end{align*}
  and
  \begin{align*}
    d_{X,Y} \maps \mc D(I,X) \times \mc D(I,Y) &\longrightarrow \mc D(I, X \ot
    Y);\\
    (x,y) &\longmapsto (I \xrightarrow{\rho^{-1}} I \ot I \xrightarrow{x \ot y} X \ot
    Y).
  \end{align*}
  These are natural in $X$ and $Y$. 

  Furthermore, with these as coherence maps, the covariant hom functor on $I$,
  $\mathcal D(I,-)\maps (\mc D,\ot) \to (\Set,\times)$, is a lax braided monoidal
  functor. 
\end{proposition}
\begin{proof}
  The naturality of the maps $d_{X,Y}$ follows from the naturality of $\rho$.
  The coherence axioms immediately follow from the coherence of the braided
  monoidal category $\mc D$.
\end{proof}

Composing a decorating functor $F$ with this hom functor produces an isomorphic
decorated cospan category. Note that we write $\mc D(I,F-) = \mc D(I,-) \circ
F$.

\begin{proposition} \label{prop.setdecorations}
  Let $F\maps (\mathcal C, +) \to (\mathcal D, \otimes)$ be a braided lax
  monoidal functor. Then $F\mathrm{Cospan}$ and $\mathcal D(I,F-)\mathrm{Cospan}$ are isomorphic as hypergraph categories.
\end{proposition}
\begin{proof}
  We have the commutative-by-definition triangle of braided lax monoidal
  functors
  \[
    \xymatrixrowsep{2ex}
    \xymatrix{
      && (\mathcal D,\otimes) \ar[dd]^{\mathcal D(I,-)} \\
      (\mathcal C,+) \ar[urr]^{F} \ar[drr]_{\mathcal D(I,F-)} \\
      && (\Set, \times)
    }
  \]
  By Theorem \ref{thm:decoratedfunctors}, this gives a strict hypergraph functor
  $F\mathrm{Cospan} \to \mathcal D(I,F-)\mathrm{Cospan}$. It is easily
  observed that this functor is bijective on objects and on morphisms. 
\end{proof}

Similarly, taking $\Set$ to be our decorating category does not impinge upon the
functors that can be constructed using decorated cospans.
\begin{proposition}
  Let $T\maps F\mathrm{Cospan} \to G\mathrm{Cospan}$ be a decorated cospans
  functor. Then there is a decorated cospans functor $U$ such that the square of
  hypergraph functors 
  \[
    \xymatrixcolsep{3pc}
    \xymatrixrowsep{3pc}
    \xymatrix{
      F\mathrm{Cospan} \ar[r]^T \ar[d]_{\sim} & G\mathrm{Cospan} \ar[d]^{\sim}\\ 
      \mathcal D(I,F-)\mathrm{Cospan} \ar[r]^U & \mathcal D'(I,G-)\mathrm{Cospan}
    }
  \]
  commutes, where the vertical maps are the isomorphisms given by the previous
  proposition.
\end{proposition}
\begin{proof}
  Write 
  \[
    \xymatrixcolsep{3pc}
    \xymatrixrowsep{3pc}
    \xymatrix{
      (\mathcal C,+) \ar^{(F,\varphi)}[r] \ar_{(A,\alpha)}[d] \drtwocell
      \omit{_\:\theta} & (\mathcal D,\otimes) \ar^{(B,\beta)}[d]  \\
      (\mathcal C',+) \ar_{(G,\gamma)}[r] & (\mathcal D',\boxtimes).
    }
  \]
  for the monoidal natural transformation yielding $T$, and write
  \begin{align*}
    \overline\beta_X\maps \mc D(I,X) &\longrightarrow \mc D'(I,BX); \\
    (I_{\mc D} \xrightarrow{x} X) &\longmapsto (I_{\mc D'} \xrightarrow{\beta_I}
    BI_{\mc D} \xrightarrow{Bx} BX)
  \end{align*}
  for all $X \in \mc D$. This defines a monoidal natural transformation
  $\overline\beta\maps \mc D(I,-) \Rightarrow \mc D'(I,B-)$, with the
  monoidality following from the monoidality of the functor $(B,\beta)$.

  Next, define $U$ as the functor resulting from the decorated cospan
  construction applied to the horizontal composition of monoidal natural
  transformations
  \[
    \xymatrixcolsep{3pc}
    \xymatrixrowsep{3pc}
    \xymatrix{
      (\mathcal C,+) \ar^{(F,\varphi)}[r] \ar_{(A,\alpha)}[d] \drtwocell
      \omit{_\:\theta} & (\mathcal D,\otimes) \ar[r]^{\mc D(I,-)}
      \ar^{(B,\beta)}[d] \drtwocell \omit{_\:\overline{\beta}} & (\Set,\times)
      \ar@{=}[d] \\
      (\mathcal C',+) \ar_{(G,\gamma)}[r] & (\mathcal D',\boxtimes) \ar[r]_{\mc
      D'(I,-)} & (\Set,\times).
    }
  \]
  To see that the required square of hypergraph functors commutes, observe that
  both functors $F\mathrm{Cospan} \to \mc D'(I,G-)\mathrm{Cospan}$ take
  objects $X$ to $AX$ and morphisms 
  \[
    (X \stackrel{i}\longrightarrow N \stackrel{o}\longleftarrow Y, \enspace I
    \stackrel{s}\longrightarrow FN)
  \]
  to 
  \[
    \big(AX \stackrel{Ai}\longrightarrow AN
    \stackrel{Ao}\longleftarrow AY, \enspace \theta_N \circ Bs\circ \beta_I\in \mc
    D'(I,GAN)\big). \qedhere
  \]
\end{proof}

\begin{remark}
Note that the braided monoidal categories $(\mc C,+)$ and $(\Set,\times)$ are
both symmetric monoidal categories. As symmetric monoidal functors are just
braided monoidal functors between symmetric monoidal categories, when working
with the decorating category $\Set$ we may refer to the decorating functors $F$
as symmetric lax monoidal, rather than merely braided lax monoidal.
\end{remark}

\begin{remark}
  When $\mc C$ has finite colimits, call a lax symmetric monoidal functor $(\mc
  C,+) \to (\Set,\times)$ a lax symmetric monoidal copresheaf on $\mc C$. A
  morphism between lax symmetric monoidal copresheaves $F$ and $G$ on categories
  $\mc C$ and $\mc C'$ with finite colimits is a colimit preserving functor $A
  \maps F \to G$ together with a monoidal natural transformation $\theta\maps F
  \Rightarrow G \circ A$. This forms a category.

  The work in this chapter implies there exists a full and faithful functor from
  this category to the category of hypergraph categories, taking a lax symmetric
  monoidal copresheaf on a category with finite colimits to its decorated cospans
  category.
\end{remark}

\section{Examples} \label{sec:ex}
In this final section we outline two constructions of decorated cospan
categories, based on labelled graphs and linear subspaces respectively, and a
functor between these two categories interpreting each graph as an electrical
circuit. We shall see that the decorated cospan framework allows us to take a
notion of closed system and construct a corresponding notion of open or
composable system, together with semantics for these systems. 

This electrical circuits example is the motivating application for the
decorated cospan construction, and its shortcomings motivate the further
theoretical development of decorated cospans in the next chapter. Armed with
these additional tools, we will return to the application for a full discussion
in Chapter \ref{ch.circuits}.

\subsection{Labelled graphs} \label{ssec.exlabelledgraphs}

Recall that a \define{$[0,\infty)$-graph} $(N,E,s,t,r)$ comprises a finite set
$N$ of vertices (or nodes), a finite set $E$ of edges, functions $s,t\colon  E \to N$
describing the source and target of each edge, and a function $r\colon  E \to
[0,\infty)$ labelling each edge. The decorated cospan framework allows us to
construct a category with, roughly speaking, these graphs as morphisms. More
precisely, our morphisms will consist of these graphs, together with subsets of
the nodes marked, with multiplicity, as `input' and `output' connection points.

Pick small categories equivalent to the category of $[0,\infty)$-graphs such
that we may talk about the set of all $[0,\infty)$-graphs on each finite set
$N$.  Then we may consider the functor
\[
  \lgraph\colon  (\FinSet,+) \longrightarrow (\Set,\times)
\]
taking a finite set $N$ to the set $\lgraph(N)$ of $[0,\infty)$-graphs
$(N,E,s,t,r)$ with set of nodes $N$. On
morphisms let it take a function $f\colon N \to M$ to the function that pushes
labelled graph structures on a set $N$ forward onto the set $M$:
\begin{align*}
  \lgraph(f)\colon  \lgraph(N) &\longrightarrow
  \lgraph(M); \\
  (N,E,s,t,r) &\longmapsto (M,E,f \circ s, f \circ t, r).
\end{align*}
As this map simply acts by post-composition, our map $\lgraph$ is indeed
functorial.

We then arrive at a lax braided monoidal functor $(\lgraph,\zeta)$ by equipping
this functor with the natural transformation 
\begin{align*}
  \zeta_{N,M}\colon  \lgraph(N) \times \lgraph(M)
  &\longrightarrow \lgraph(N+M); \\
  \big( (N,E,s,t,r), (M,F,s',t',r') \big) &\longmapsto
  \big(N+M,E+F,s+s',t+t',[r,r']\big),
\end{align*}
together with the unit map
\begin{align*}
  \zeta_1\colon  1=\{\bullet\} &\longrightarrow \lgraph(\varnothing); \\
  \bullet &\longmapsto
  (\varnothing,\varnothing,!,!,!),
\end{align*}
where we remind ourselves
that we write $[r,r']$ for the copairing of the functions $r$ and $r'$. The
naturality of this collection of morphisms, as well as the coherence laws for
lax braided monoidal functors, follow from the universal property of the coproduct.

Theorem \ref{thm:fcospans} thus allows us to construct a hypergraph category
$\mathrm{GraphCospan}$.  For an intuitive visual understanding of the morphisms
of this category and its composition rule, see
\textsection\ref{sec.closedtoopen}. The
main idea is that composition glues together any pair of terminal that have the same
preimage under the cospan maps. An identity morphism is the empty decorated
identity cospan or, equivalently, a graph with no edges in which every vertex is
both an input and an output:
\[
  \begin{tikzpicture}[circuit ee IEC]
    \node[circle,draw,inner sep=1pt,fill=gray,color=gray]         (y1) at
    (-1.6,1.2) {};
    \node[circle,draw,inner sep=1pt,fill=gray,color=gray]         (y2) at
    (-1.6,0) {};
    \node at (-1.6,-.5) {$X$};
    \node[contact]         (A) at (0,1.2) {};
    \node[contact]         (B) at (0,0) {};
    \node[circle,draw,inner sep=1pt,fill=gray,color=gray]         (z1) at
    (1.6,1.2) {};
    \node[circle,draw,inner sep=1pt,fill=gray,color=gray]         (z2) at
    (1.6,0) {};
    \node at (1.6,-.5) {$X$};
    \path[color=gray, very thick, shorten >=10pt, shorten <=5pt, ->, >=stealth] (y1) edge (A);
    \path[color=gray, very thick, shorten >=10pt, shorten <=5pt, ->, >=stealth] (y2)
    edge (B);
    \path[color=gray, very thick, shorten >=10pt, shorten <=5pt, ->, >=stealth] (z1) edge (A);
    \path[color=gray, very thick, shorten >=10pt, shorten <=5pt, ->, >=stealth] (z2) edge (B);
  \end{tikzpicture}
\]
The Frobenius maps are also empty decorated, with each node in the apex marked
as input or output perhaps multiple or no times, as appropriate. For example,
the Frobenius multiplication on a one element set $Y$ is given by the decorated
cospan:
\[
  \begin{tikzpicture}[circuit ee IEC]
    \node[circle,draw,inner sep=1pt,fill=gray,color=gray]         (y1) at
    (-1.6,1.2) {};
    \node[circle,draw,inner sep=1pt,fill=gray,color=gray]         (y2) at
    (-1.6,0) {};
    \node at (-1.6,-.5) {$Y+Y$};
    \node[contact]         (A) at (0,.6) {};
    \node[circle,draw,inner sep=1pt,fill=gray,color=gray]         (z1) at
    (1.6,.6) {};
    \node at (1.6,-.5) {$Y$};
    \path[color=gray, very thick, shorten >=10pt, shorten <=5pt, ->, >=stealth] (y1) edge (A);
    \path[color=gray, very thick, shorten >=10pt, shorten <=5pt, ->, >=stealth] (y2)
    edge (A);
    \path[color=gray, very thick, shorten >=10pt, shorten <=5pt, ->, >=stealth] (z1) edge (A);
  \end{tikzpicture}
\]

\subsection{Linear relations} 
Another example of a decorated cospan category arising from a functor $(\FinSet,+)
\to (\Set, \times)$ is closely related to the category of linear relations. Here
we decorate each cospan in $\Set$ with a linear subspace of $\R^N \oplus
(\R^N)^\ast$, the sum of the vector space generated by the apex $N$ over $\R$
and its vector space dual.

First let us recall some facts about relations. Let $R \subseteq X\times Y$ be
a relation; we write this also as $R\colon  X \to Y$. The opposite relation $R^\opp\colon 
Y\to X$, is the subset $R^\opp \subseteq Y\times X$ such that $(y,x) \in R^\opp$
if and only if $(x,y) \in R$. We say that the image of a subset $S \subseteq X$
under a relation $R\colon  X \to Y$ is the subset of all elements of the codomain $Y$
related by $R$ to an element of $S$. Note that if $X$ and $Y$ are vector spaces
and $S$ and $R$ are both linear subspaces, then the image $R(S)$ of $S$ under
$R$ is again a linear subspace.

Now any function $f\colon N \to M$ induces a linear map $f^\ast\colon  \R^M \to
\R^N$ by precomposition. This linear map $f^\ast$ itself induces a dual map
$f_\ast\colon (\R^N)^\ast \to (\R^M)^\ast$ by precomposition. Furthermore
$f^\ast$ has, as a linear relation $f^\ast \subseteq \R^M \oplus \R^N$, an
opposite linear relation $(f^\ast)^\opp\colon  \R^N \to \R^M$.  Define the
functor 
\[
  \linsub\colon  (\FinSet,+) \longrightarrow (\Set,\times)
\]
taking a finite set $N$ to set of linear subspaces of the vector space $\R^N
\oplus (\R^N)^\ast$, and taking a function $f\colon  N \to M$ to the function
$\linsub(N) \to \linsub(M)$ induced by the sum of these two relations:
\begin{align*}
  \linsub(f)\colon  \linsub(N) &\longrightarrow \linsub(M); \\
  L &\longmapsto \big((f^\ast)^\opp \oplus f_\ast\big)(L).
\end{align*}
The above operations on $f$ used in the construction of this map are functorial,
and so it is readily observed that $\linsub$ is indeed a functor.

It is moreover lax braided monoidal as the sum of a linear subspace of $\R^N
\oplus (\R^N)^\ast$ and a linear subspace of $\R^M \oplus (\R^M)^\ast$ may be
viewed as a subspace of $\R^N \oplus (\R^N)^\ast \oplus \R^M \oplus (\R^M)^\ast
\cong \R^{N+M} \oplus (\R^{N+M})^\ast \cong \R^{M+N} \oplus (\R^{M+N})^\ast$,
and the empty subspace is a linear subspace of each $\R^N \oplus (\R^N)^\ast$.

We thus have a hypergraph category $\mathrm{LinSubCospan}$.

\subsection{Electrical circuits} \label{ssec.exelectricalcircuits}
Electrical circuits and their diagrams are the motivating application for the
decorated cospan construction. Specialising to the case of networks of linear
resistors, we detail here how we may use the category $\mathrm{LinSubCospan}$ to
provide semantics for the morphisms of $\mathrm{GraphCospan}$ as diagrams of
networks of linear resistors.

Intuitively, after choosing a unit of resistance, say ohms ($\Omega$), each
$[0,\infty)$-graph can be viewed as a network of linear resistors, with the
$[0,\infty)$-graph of \textsection\ref{sec.closedtoopen} now more commonly
depicted as
\begin{center}
  \begin{tikzpicture}[circuit ee IEC, set resistor graphic=var resistor IEC graphic]
    \node[contact]         (A) at (0,0) {};
    \node[contact]         (B) at (3,0) {};
    \node[contact]         (C) at (1.5,-2.6) {};
    \coordinate         (ua) at (.5,.25) {};
    \coordinate         (ub) at (2.5,.25) {};
    \coordinate         (la) at (.5,-.25) {};
    \coordinate         (lb) at (2.5,-.25) {};
    \path (A) edge (ua);
    \path (A) edge (la);
    \path (B) edge (ub);
    \path (B) edge (lb);
    \path (ua) edge  [resistor] node[label={[label distance=1pt]90:{$0.2\Omega$}}] {} (ub);
    \path (la) edge  [resistor] node[label={[label distance=1pt]270:{$1.3\Omega$}}] {} (lb);
    \path (A) edge  [resistor] node[label={[label distance=2pt]180:{$0.8\Omega$}}] {} (C);
    \path (C) edge  [resistor] node[label={[label distance=2pt]0:{$2.0\Omega$}}] {} (B);
  \end{tikzpicture}
\end{center}
$\mathrm{GraphCospan}$ may then be viewed as a category with morphisms
circuits of linear resistors equipped with chosen input and output terminals.

The suitability of this language is seen in the way the different categorical
structures of $\mathrm{GraphCospan}$ capture different operations that can be
performed with circuits. To wit, the sequential composition expresses the fact
that we can connect the outputs of one circuit to the inputs of the next, while
the monoidal composition models the placement of circuits side-by-side.
Furthermore, the symmetric monoidal structure allows us reorder input and output
wires, the compactness captures the interchangeability between input and
output terminals of circuits---that is, the fact that we can choose any input
terminal to our circuit and consider it instead as an output terminal, and vice
versa---and the Frobenius structure expresses the fact that we may wire any
node of the circuit to as many additional components as we like.

Moreover, Theorem \ref{thm:decoratedfunctors} provides semantics. Each node in a
network of resistors can be assigned an electric potential and a net current
outflow at that node, and so the set $N$ of vertices of a $[0,\infty)$-graph can
be seen as generating a space $\R^N \oplus (\R^N)^\ast$ of electrical states of
the network. We define a natural transformation 
\[
  \res\colon  \lgraph \Longrightarrow \linsub
\]
mapping each $[0,\infty)$-graph on $N$, viewed as a network of resistors, to the
linear subspace of $\R^N \oplus (\R^N)^\ast$ of electrical states permitted by
Ohm's law.\footnote{Note that these states need not obey Kirchhoff's current
law.} In detail, let $\psi \in \R^N$. We define the power $Q\colon  \R^N \to \R$
corresponding to a $[0,\infty)$-graph $(N,E,s,t,r)$ to be the function
\[
  Q(\psi) = \sum_{e \in E} \frac1{r(e)}
  \Big(\psi\big(t(e)\big)-\psi\big(s(e)\big)\Big)^2.
\]
Then the states of a network of resistors are given by a potential $\phi$ on the
nodes and the gradient of the power at this potential:
\begin{align*}
  \res_N\colon  \lgraph(N) &\longrightarrow \linsub(N)\\
  (N,E,s,t,r) &\longmapsto \{(\phi,\nabla Q_\phi) \mid \phi \in \R^N\}.
\end{align*}
This defines a monoidal natural transformation. Hence, by Theorem
\ref{thm:decoratedfunctors}, we obtain a hypergraph functor
$\mathrm{GraphCospan} \to \mathrm{LinSubCospan}$. 

The semantics provided by this functor match the standard interpretation of
networks of linear resistors. The maps of the Frobenius monoid take on the
interpretation of perfectly conductive wires, forcing the potentials at all
nodes they connect to be equal, and the sum of incoming currents to equal the
sum of outgoing currents---precisely the behaviour implied by Kirchhoff's laws.
More generally, let $(X \stackrel{i}{\rightarrow} N
\stackrel{o}{\leftarrow} Y, \, (N,E,s,t,r))$ be a morphism of
$\mathrm{GraphCospan}$, with $Q$ the power function corresponding to the graph
$\Gamma = (N,E,s,t,r)$. The image of this decorated cospan in
$\mathrm{LinSubCospan}$ is the decorated cospan $(X \stackrel{i}{\rightarrow} N
\stackrel{o}{\leftarrow} Y, \, \{(\phi,\nabla Q_\phi) \mid \phi \in \R^N\})$.
Then it is straightforward to check that the subspace 
\[
  \linsub[i,o]\big(\res_N(\Gamma)\big)\subseteq \R^{X+Y} \oplus (\R^{X+Y})^\ast
\]
is the subspace of electrical states on the terminals $X+Y$ such that currents
and potentials can be chosen across the network of resistors $(N,E,s,t,r)$ that
obey Ohm's and, on its interior, Kirchhoff's laws. In particular, after passing
to a subspace of the terminals in this way, composition in
$\mathrm{LinSubCospan}$ corresponds to enforcing Kirchhoff's laws on the shared
terminals of the two networks. 

This behaviour at the terminals $X+Y$ is often all we are interested in: for
example, in a large electrical network, substituting a subcircuit for a
different subcircuit with the same terminal behaviour will not affect the rest
of the network. Yet this terminal behaviour is not yet encoded directly in the
categorical structure. The images of circuits in $\mathrm{GraphCospan}$ in
$\mathrm{LinSubCospan}$ are cospans decorated by a subspace of $\R^N \oplus
(\R^N)^\ast$, keeping track of \emph{all} the internal behaviour as well. This
can be undesirable, for reasons of information compression as well as reasoning
about equivalence. The next chapter addresses this issue by introducing corelations.

\chapter{Corelations: a tool for black boxing} \label{ch.corelations}

As we have seen, cospans are a useful tool for describing finite colimits, and
so for describing the interconnection of systems. In the previous chapter
we defined decorated cospans to take advantage of this fact and provide
a tool for composing structures that had no inherent composition law, like
graphs and subspaces. In many situations, however, the colimit contains more
information than we care about.  Rather than concerning ourselves with all the
internal structure of a system, we only find a certain aspect of it
relevant---roughly, the part that affects what happens at the boundary. This is
better modelled by corelations. In this chapter we develop the theory of
corelations.

As usual, we begin by giving an intuitive overview of the aims of this chapter
in \textsection\ref{sec.blackboxing}. We then give the formal details of
corelation categories (\textsection\ref{sec.corels}) and their functors
(\textsection\ref{sec.corelfunctors}), before concluding in
\textsection\ref{sec.corelexs} with two important
examples: equivalence relations as epi-mono corelations in $(\Set,+)$, and
linear relations as epi-mono corelations in $(\Vect,\oplus)$. This sets us up
for a decorated version of this theory in Chapter \ref{ch.deccorels}.

\section{The idea of black boxing} \label{sec.blackboxing}

Thus far we have argued that network languages should be modelled using
hypergraph categories, and shown that cospans provide a good language for
talking about interconnection. We then developed the theory of decorated
cospans, which allows us to take diagrams, mark `inputs' and `outputs' using
cospans, and then compose these diagrams using pushouts. This turns a notion of
closed system into a notion of open one.

A serious limitation of using cospans alone, however, is that cospans
indiscrimately accumulate information. For example, suppose we consider the
morphisms of $\mathrm{GraphCospan}$ as open electrical circuits as in
\textsection\ref{ssec.exelectricalcircuits}. Pursuing this, let us depict a
graph decorated cospan from $(X \xrightarrow{i} N \xleftarrow{o} Y, \enspace
(N,E,s,t,r))$ by 
\[
\resizebox{.35\textwidth}{!}{
    \tikzset{every path/.style={line width=1.1pt}}
  \begin{tikzpicture}[circuit ee IEC, set resistor graphic=var resistor IEC graphic]
	\begin{pgfonlayer}{nodelayer}
		\node [style=dot] (0) at (-2.5, 0.75) {};
		\coordinate (1) at (-2, -2) {};
		\coordinate (2) at (1.5, 2) {};
		\coordinate (A) at (-2, 0.75) {};
		\node [style=dot] (4) at (2, 1.25) {};
		\coordinate (B) at (1.5, 0.75) {};
		\node [style=dot] (6) at (2, 0.25) {};
		\coordinate (ub) at (0.75, 1) {};
		\coordinate (la) at (-1.25, 0.5) {};
		\coordinate (C) at (-0.25, -1.375) {};
		\coordinate (lb) at (0.75, 0.5) {};
		\coordinate (ua) at (-1.25, 1) {};
	\end{pgfonlayer}
	\begin{pgfonlayer}{edgelayer}
		\draw (0.center) to (A);
		\draw (6) to (B);
		\draw (4) to (B);
    \path (A) edge (ua);
    \path (A) edge (la);
    \path (B) edge (ub);
    \path (B) edge (lb);
    \path (ua) edge  [resistor] (ub);
    \path (la) edge  [resistor] (lb);
    \path (A) edge  [resistor]  (C);
    \path (C) edge  [resistor]  (B);
	\end{pgfonlayer}
	\begin{pgfonlayer}{background}
	  \filldraw [fill=black!5!white, draw=black!40!white] (1) rectangle (2);
	\end{pgfonlayer}
\end{tikzpicture}
}
\]
where the bullets $\bullet$ on the left represent the elements of the set $X$,
those on the right represent the elements of $Y$, and we draw a line from $x \in
X$ or $y \in Y$ to a node $n$ in the graph when $i(x) = n$ or $o(y)=n$.  We omit
the labels (resistances) on the edges as these are not essential to the main
point here.  

We also depict an example of composition of these open circuits using decorated
cospans:
\[
\begin{aligned}
\resizebox{.45\textwidth}{!}{
    \tikzset{every path/.style={line width=1.1pt}}
  \begin{tikzpicture}[circuit ee IEC, set resistor graphic=var resistor IEC graphic]
	\begin{pgfonlayer}{nodelayer}
		\node [style=dotbig] (0) at (-2.25, 4) {};
		\coordinate (1) at (-1.75, 1.25) {};
		\coordinate (2) at (1.75, 5.25) {};
		\coordinate (3) at (-1.75, 4) {};
		\node [style=dotbig] (4) at (2.25, 4.5) {};
		\coordinate (5) at (1.75, 4) {};
		\node [style=dotbig] (6) at (2.25, 3.5) {};
		\coordinate (7) at (1, 4.25) {};
		\coordinate (8) at (-1, 3.75) {};
		\coordinate (9) at (0, 1.75) {};
		\coordinate (10) at (1, 3.75) {};
		\coordinate (11) at (-1, 4.25) {};
		\node [style=dotbig] (12) at (7.25, 4) {};
		\coordinate (13) at (-3, -1.5) {};
		\coordinate (14) at (6, 2) {};
		\coordinate (15) at (3.5, 2) {};
		\coordinate (16) at (4.75, -0) {};
		\coordinate (17) at (3, 3.5) {};
		\coordinate (18) at (3, 4.5) {};
		\node [style=dotbig] (19) at (2.5, -1.5) {};
		\coordinate (20) at (6.75, 4) {};
		\node [style=dotbig] (21) at (7.25, -1.5) {};
		\node [style=dotbig] (22) at (-6.5, -1) {};
		\node [style=dotbig] (23) at (7.25, 1.5) {};
		\node [style=dotbig] (24) at (-6.5, 4.5) {};
		\node [style=dotbig] (25) at (-6.5, 3) {};
		\coordinate (26) at (-6, -2) {};
		\coordinate (27) at (-3, 5.25) {};
		\coordinate (28) at (1.75, 1) {};
		\coordinate (29) at (-1.75, -2) {};
		\coordinate (30) at (3, -2) {};
		\coordinate (31) at (-6, 4.5) {};
		\coordinate (32) at (-6, 3) {};
		\coordinate (33) at (-3, 4) {};
		\coordinate (34) at (6.75, 2.5) {};
		\coordinate (35) at (3, 2.75) {};
		\coordinate (36) at (6.75, 5.25) {};
		\node [style=dotbig] (37) at (7.25, 0.5) {};
		\node [style=dotbig] (38) at (-2.25, -1.5) {};
		\node [style=dotbig] (39) at (-2.5, 4) {};
		\node [style=dotbig] (40) at (2.5, 4.5) {};
		\node [style=dotbig] (41) at (2.5, 3.5) {};
		\node [style=dotbig] (42) at (-2.25, 0.5) {};
		\node [style=dotbig] (43) at (-2.5, -1.5) {};
		\node [style=dotbig] (44) at (-2.5, 0.5) {};
		\node [style=dotbig] (45) at (2.25, -1.5) {};
		\node [style=dotbig] (46) at (2.25, -0) {};
		\node [style=dotbig] (47) at (2.5, -0) {};
		\node [style=dotbig] (48) at (-2.5, -0.5) {};
		\node [style=dotbig] (49) at (-2.25, -0.5) {};
		\coordinate (50) at (-6, -1) {};
		\coordinate (51) at (-5, -0) {};
		\coordinate (52) at (-3, 0.5) {};
		\coordinate (53) at (-3, -0.5) {};
		\coordinate (54) at (-1.75, 0.5) {};
		\coordinate (55) at (-1.75, -0.5) {};
		\coordinate (56) at (0, -0) {};
		\coordinate (57) at (1.75, -0) {};
		\coordinate (58) at (3, -0) {};
		\coordinate (59) at (-5.75, 1.25) {};
		\coordinate (60) at (4.75, 1) {};
		\coordinate (61) at (6.75, 1.5) {};
		\coordinate (62) at (6.75, 0.5) {};
		\coordinate (63) at (-1.75, -1.5) {};
		\coordinate (64) at (1.75, -1.5) {};
		\coordinate (65) at (3, -1.5) {};
		\coordinate (66) at (6.75, -1.5) {};
		\coordinate (67) at (-5.5, 1.75) {};
		\coordinate (68) at (-5.25, 1) {};
		\coordinate (69) at (-3.75, 2.5) {};
		\coordinate (70) at (-3.5, 1.75) {};
		\coordinate (71) at (-3.25, 2.25) {};
	\end{pgfonlayer}
	\begin{pgfonlayer}{edgelayer}
		\path (50) edge [resistor] (13);
		\draw (24) to (31);
		\draw (61) to (23);
		\draw (62) to (37);
		\draw (25) to (32);
		\draw (22) to (50);
		\draw (33) to (39);
		\draw (0) to (3);
		\draw (52) to (44);
		\draw (42) to (54);
		\draw (53) to (48);
		\draw (49) to (55);
		\path (31) edge [resistor] (33);
		\path (32) edge [resistor] (33);
		\draw (3) to (11);
		\draw (3) to (8);
		\path (11) edge [resistor] (7);
		\path (8) edge [resistor] (10);
		\draw (7) to (5);
		\draw (10) to (5);
		\draw (5) to (4);
		\draw (5) to (6);
		\draw (41) to (17);
		\draw (40) to (18);
		\path (18) edge [resistor] (20);
		\path (17) edge [resistor] (20);
		\draw (20) to (12);
		\path (15) edge [resistor] (14);
		\path (54) edge [resistor] (56);
		\path (55) edge [resistor] (56);
		\path (56) edge [resistor] (57);
		\path (51) edge [resistor] (52);
		\path (51) edge [resistor] (53);
		\draw (57) to (46);
		\draw (47) to (58);
		\path (58) edge [resistor] (16);
		\path (3) edge [resistor] (9);
		\path (9) edge [resistor] (5);
		\draw (13) to (43);
		\path (60) edge [resistor] (61);
		\path (60) edge [resistor] (62);
		\draw (38) to (63);
		\path (63) edge [resistor] (64);
		\draw (64) to (45);
		\draw (19) to (65);
		\path (65) edge [resistor] (66);
		\draw (66) to (21);
		\path (67) edge [resistor] (69);
		\path (68) edge [resistor] (70);
		\draw (59) to (67);
		\draw (59) to (68);
		\draw (69) to (71);
		\draw (71) to (70);
	\end{pgfonlayer}
	\begin{pgfonlayer}{background}
	  \filldraw [fill=black!5!white, draw=black!40!white] (1) rectangle (2);
	  \filldraw [fill=black!5!white, draw=black!40!white] (26) rectangle (27);
	  \filldraw [fill=black!5!white, draw=black!40!white] (29) rectangle (28);
	  \filldraw [fill=black!5!white, draw=black!40!white] (35) rectangle (36);
	  \filldraw [fill=black!5!white, draw=black!40!white] (30) rectangle (34);
	\end{pgfonlayer}
\end{tikzpicture}
}
\end{aligned}
\quad
\mapsto
\enspace
\begin{aligned}
\resizebox{.4\textwidth}{!}{
    \tikzset{every path/.style={line width=1.1pt}}
  \begin{tikzpicture}[circuit ee IEC, set resistor graphic=var resistor IEC graphic]
	\begin{pgfonlayer}{nodelayer}
		\coordinate (0) at (-1.75, 4) {};
		\coordinate (1) at (1.75, 4) {};
		\coordinate (2) at (1, 4.25) {};
		\coordinate (3) at (-1, 3.75) {};
		\coordinate (4) at (0, 1.75) {};
		\coordinate (5) at (1, 3.75) {};
		\coordinate (6) at (-1, 4.25) {};
		\node [style=dotbig] (7) at (6.25, 4) {};
		\coordinate (8) at (5, 2) {};
		\coordinate (9) at (2.5, 2) {};
		\coordinate (10) at (3.75, -0) {};
		\coordinate (11) at (2.5, 3.75) {};
		\coordinate (12) at (2.5, 4.25) {};
		\coordinate (13) at (5.75, 4) {};
		\node [style=dotbig] (14) at (6.25, -1.5) {};
		\node [style=dotbig] (15) at (-5.5, -1) {};
		\node [style=dotbig] (16) at (6.25, 1.5) {};
		\node [style=dotbig] (17) at (-5.5, 4.5) {};
		\node [style=dotbig] (18) at (-5.5, 3) {};
		\coordinate (19) at (-5, -2) {};
		\coordinate (20) at (-5, 4.5) {};
		\coordinate (21) at (-5, 3) {};
		\coordinate (22) at (5.75, 5.25) {};
		\node [style=dotbig] (23) at (6.25, 0.5) {};
		\coordinate (24) at (-5, -1) {};
		\coordinate (25) at (-4, -0) {};
		\coordinate (26) at (-1.75, 0.5) {};
		\coordinate (27) at (-1.75, -0.5) {};
		\coordinate (28) at (0, -0) {};
		\coordinate (29) at (1.75, -0) {};
		\coordinate (30) at (-4.75, 1.25) {};
		\coordinate (31) at (3.75, 1) {};
		\coordinate (32) at (5.75, 1.5) {};
		\coordinate (33) at (5.75, 0.5) {};
		\coordinate (34) at (-1.75, -1.5) {};
		\coordinate (35) at (1.75, -1.5) {};
		\coordinate (36) at (5.75, -1.5) {};
		\coordinate (37) at (-4.25, 1.75) {};
		\coordinate (38) at (-4.25, 1) {};
		\coordinate (39) at (-2.75, 2.5) {};
		\coordinate (40) at (-2.5, 1.75) {};
		\coordinate (41) at (-2.25, 2.25) {};
		\coordinate (42) at (5, 4.25) {};
		\coordinate (43) at (5, 3.75) {};
	\end{pgfonlayer}
	\begin{pgfonlayer}{edgelayer}
		\draw (17) to (20);
		\draw (18) to (21);
		\draw (15) to (24);
		\draw (13) to (7);
		\draw (36) to (14);
		\draw (32) to (16);
		\draw (33) to (23);
		\draw (1) to (12);
		\draw (1) to (11);
		\draw (2) to (1);
		\draw (5) to (1);
		\draw (0) to (6);
		\draw (0) to (3);
		\draw (30) to (37);
		\draw (30) to (38);
		\draw (39) to (41);
		\draw (41) to (40);
		\draw (42) to (13);
		\draw (43) to (13);
		\path (6) edge [resistor] (2);
		\path (3) edge [resistor] (5);
		\path (9) edge [resistor] (8);
		\path (26) edge [resistor] (28);
		\path (27) edge [resistor] (28);
		\path (28) edge [resistor] (29);
		\path (0) edge [resistor] (4);
		\path (4) edge [resistor] (1);
		\path (31) edge [resistor] (32);
		\path (31) edge [resistor] (33);
		\path (34) edge [resistor] (35);
		\path (37) edge [resistor] (39);
		\path (38) edge [resistor] (40);
		\path (20) edge [resistor] (0);
		\path (21) edge [resistor] (0);
		\path (25) edge [resistor] (26);
		\path (25) edge [resistor] (27);
		\path (24) edge [resistor] (34);
		\path (29) edge [resistor] (10);
		\path (35) edge [resistor] (36);
		\path (12) edge [resistor] (42);
		\path (11) edge [resistor] (43);
	\end{pgfonlayer}
	\begin{pgfonlayer}{background}
	  \filldraw [fill=black!5!white, draw=black!40!white] (19) rectangle (22);
	\end{pgfonlayer}
\end{tikzpicture}
}
\end{aligned}
\]
Note in particular that the composite of these open circuits contains a unique
resistor for every resistor in the factors. If we are interested in describing
the syntax of a diagrammatic language, then this is useful: composition builds
given expressions into a larger one. If we are only interested in the semantics,
however, this is often unnecessary and thus often wildly inefficient.

Indeed, suppose our semantics for open circuits is given by the information that
can be gleaned by connecting other open circuits, such as measurement devices,
to the terminals. In these semantics we consider two open circuits equivalent
if, should they be encased, but for their terminals, in a black box
\[
\resizebox{.4\textwidth}{!}{
    \tikzset{every path/.style={line width=1.1pt}}
  \begin{tikzpicture}
    \begin{pgfonlayer}{nodelayer}
		\node [style=dotbig] (14) at (-5.5, 4.5) {};
		\node [style=dotbig] (15) at (-5.5, 3) {};
		\node [style=dotbig] (12) at (-5.5, -1) {};
		\node [style=dotbig] (7) at (6.25, 4) {};
		\node [style=dotbig] (13) at (6.25, 1.5) {};
		\node [style=dotbig] (20) at (6.25, 0.5) {};
		\node [style=dotbig] (11) at (6.25, -1.5) {};
		\coordinate (17) at (-5, 4.5) {};
		\coordinate (18) at (-5, 3) {};
		\coordinate (21) at (-5, -1) {};
		\coordinate (10) at (5.75, 4) {};
		\coordinate (23) at (5.75, 1.5) {};
		\coordinate (24) at (5.75, 0.5) {};
		\coordinate (27) at (5.75, -1.5) {};
		\coordinate (16) at (-5, -2) {};
		\coordinate (19) at (5.75, 5.25) {};
	\end{pgfonlayer}
	\begin{pgfonlayer}{edgelayer}
		\draw (14) to (17);
		\draw (15) to (18);
		\draw (12) to (21);
		\draw (10) to (7);
		\draw (27) to (11);
		\draw (23) to (13);
		\draw (24) to (20);
	\end{pgfonlayer}
	\begin{pgfonlayer}{background}
	  \filldraw [fill=black!80!white, draw=black!40!white] (16) rectangle (19);
	\end{pgfonlayer}
\end{tikzpicture}
}
\]
we would be unable to distinguish them through our electrical investigations. In
this case, at the very least, the previous circuit is equivalent to the circuit
\[
\resizebox{.4\textwidth}{!}{
    \tikzset{every path/.style={line width=1.1pt}}
  \begin{tikzpicture}[circuit ee IEC, set resistor graphic=var resistor IEC graphic]
	\begin{pgfonlayer}{nodelayer}
		\coordinate (0) at (-1.75, 4) {};
		\coordinate (1) at (1.75, 4) {};
		\coordinate (2) at (1, 4.25) {};
		\coordinate (3) at (-1, 3.75) {};
		\coordinate (4) at (0, 1.75) {};
		\coordinate (5) at (1, 3.75) {};
		\coordinate (6) at (-1, 4.25) {};
		\node [style=dotbig] (7) at (6.25, 4) {};
		\coordinate (8) at (2.5, 3.75) {};
		\coordinate (9) at (2.5, 4.25) {};
		\coordinate (10) at (5.75, 4) {};
		\node [style=dotbig] (11) at (6.25, -1.5) {};
		\node [style=dotbig] (12) at (-5.5, -1) {};
		\node [style=dotbig] (13) at (6.25, 1.5) {};
		\node [style=dotbig] (14) at (-5.5, 4.5) {};
		\node [style=dotbig] (15) at (-5.5, 3) {};
		\coordinate (16) at (-5, -2) {};
		\coordinate (17) at (-5, 4.5) {};
		\coordinate (18) at (-5, 3) {};
		\coordinate (19) at (5.75, 5.25) {};
		\node [style=dotbig] (20) at (6.25, 0.5) {};
		\coordinate (21) at (-5, -1) {};
		\coordinate (22) at (3.75, 1) {};
		\coordinate (23) at (5.75, 1.5) {};
		\coordinate (24) at (5.75, 0.5) {};
		\coordinate (25) at (-1.75, -1.5) {};
		\coordinate (26) at (1.75, -1.5) {};
		\coordinate (27) at (5.75, -1.5) {};
		\coordinate (28) at (5, 4.25) {};
		\coordinate (29) at (5, 3.75) {};
	\end{pgfonlayer}
	\begin{pgfonlayer}{edgelayer}
		\draw (14) to (17);
		\draw (15) to (18);
		\draw (12) to (21);
		\draw (10) to (7);
		\draw (27) to (11);
		\draw (23) to (13);
		\draw (24) to (20);
		\draw (1) to (9);
		\draw (1) to (8);
		\draw (2) to (1);
		\draw (5) to (1);
		\draw (0) to (6);
		\draw (0) to (3);
		\draw (28) to (10);
		\draw (29) to (10);
		\path (6) edge [resistor] (2);
		\path (3) edge [resistor] (5);
		\path (0) edge [resistor] (4);
		\path (4) edge [resistor] (1);
		\path (22) edge [resistor] (23);
		\path (22) edge [resistor] (24);
		\path (25) edge [resistor] (26);
		\path (17) edge [resistor] (0);
		\path (18) edge [resistor] (0);
		\path (21) edge [resistor] (25);
		\path (26) edge [resistor] (27);
		\path (9) edge [resistor] (28);
		\path (8) edge [resistor] (29);
	\end{pgfonlayer}
	\begin{pgfonlayer}{background}
	  \filldraw [fill=black!5!white, draw=black!40!white] (16) rectangle (19);
	\end{pgfonlayer}
\end{tikzpicture}
}
\]
where we have removed circuitry not connected to the terminals.  Moreover, this
second circuit is a much more efficient representation, as it does not model
inaccessible, internal structure. If we wish to construct a hypergraph category
modelling the semantics of open circuits, we require circuit representations and
a composition rule that only retain the information relevant to the black boxed
circuit.

Dealing with this problem fully requires further discussion of the semantics of
circuit diagrams, and the introduction of structures more flexible than graphs,
such as linear relations.  We deal with this in depth in Chapter
\ref{ch.circuits}. In this chapter, however, we are able to contribute an
important piece of the puzzle: corelations.

Corelations allow us to pursue a notion of composition that discards extraneous
information as we compose our systems. Consider, for example, the category
$\cospan(\FinSet)$ of cospans in the category of finite sets and functions. A
morphism is then a finite set $N$ together with functions $X \to N$ and $Y \to
N$. We depict these like so:
\begin{center}
  \begin{tikzpicture}[auto,scale=2]
    \node[circle,draw,inner sep=1pt,fill=gray,color=gray] (x1) at (-1.5,.2) {};
    \node[circle,draw,inner sep=1pt,fill=gray,color=gray] (x2) at (-1.5,-.2) {};
    \node at (-1.5,-.7) {$X$};
    \node[circle,draw,inner sep=1pt,fill]         (A) at (0,.4) {};
    \node[circle,draw,inner sep=1pt,fill]         (B) at (0,.133) {};
    \node[circle,draw,inner sep=1pt,fill]         (C) at (0,-.133) {};
    \node[circle,draw,inner sep=1pt,fill]         (D) at (0,-.4) {};
    \node at (0,-.7) {$N$};
    \node[circle,draw,inner sep=1pt,fill=gray,color=gray] (y1) at (1.5,.3) {};
    \node[circle,draw,inner sep=1pt,fill=gray,color=gray] (y2) at (1.5,.1) {};
    \node[circle,draw,inner sep=1pt,fill=gray,color=gray] (y3) at (1.5,-.1) {};
    \node[circle,draw,inner sep=1pt,fill=gray,color=gray] (y4) at (1.5,-.3) {};
    \node at (1.5,-.7) {$Y$};
    \path[color=gray, very thick, shorten >=10pt, shorten <=5pt, ->, >=stealth]
    (x1) edge (C);
    \path[color=gray, very thick, shorten >=10pt, shorten <=5pt, ->, >=stealth]
    (x2) edge (C);
    \path[color=gray, very thick, shorten >=10pt, shorten <=5pt, ->, >=stealth]
    (y1) edge (B);
    \path[color=gray, very thick, shorten >=10pt, shorten <=5pt, ->, >=stealth]
    (y2) edge (C);
    \path[color=gray, very thick, shorten >=10pt, shorten <=5pt, ->, >=stealth]
    (y3) edge (D);
    \path[color=gray, very thick, shorten >=10pt, shorten <=5pt, ->, >=stealth]
    (y4) edge (D);
  \end{tikzpicture}
\end{center}
Here $X$ is a two element set, while $N$ and $Y$ are four element sets.

Suppose we have a pair of cospans $X \to N \leftarrow Y$, $Y \to M \leftarrow
Z$. By definition, their composite has apex the pushout $N+_YM$ which, roughly
speaking, is the union of $N$ and $M$ with two points identified if they are
both images of the same element of $Y$. For example, the following pair of
cospans:
\begin{center}
  \begin{tikzpicture}[auto,scale=2]
    \node[circle,draw,inner sep=1pt,fill=gray,color=gray] (x1) at (-1.5,.2) {};
    \node[circle,draw,inner sep=1pt,fill=gray,color=gray] (x2) at (-1.5,-.2) {};
    \node at (-1.5,-.8) {$X$};
    \node[circle,draw,inner sep=1pt,fill]         (A) at (0,.4) {};
    \node[circle,draw,inner sep=1pt,fill]         (B) at (0,.133) {};
    \node[circle,draw,inner sep=1pt,fill]         (C) at (0,-.133) {};
    \node[circle,draw,inner sep=1pt,fill]         (D) at (0,-.4) {};
    \node at (0,-.8) {$N$};
    \node[circle,draw,inner sep=1pt,fill=gray,color=gray] (y1) at (1.5,.3) {};
    \node[circle,draw,inner sep=1pt,fill=gray,color=gray] (y2) at (1.5,.1) {};
    \node[circle,draw,inner sep=1pt,fill=gray,color=gray] (y3) at (1.5,-.1) {};
    \node[circle,draw,inner sep=1pt,fill=gray,color=gray] (y4) at (1.5,-.3) {};
    \node at (1.5,-.8) {$Y$};
    \node[circle,draw,inner sep=1pt,fill]         (A') at (3,.5) {};
    \node[circle,draw,inner sep=1pt,fill]         (B') at (3,.25) {};
    \node[circle,draw,inner sep=1pt,fill]         (C') at (3,0) {};
    \node[circle,draw,inner sep=1pt,fill]         (D') at (3,-.25) {};
    \node[circle,draw,inner sep=1pt,fill]         (E') at (3,-.5) {};
    \node at (3,-.8) {$M$};
    \node[circle,draw,inner sep=1pt,fill=gray,color=gray] (z1) at (4.5,.2) {};
    \node[circle,draw,inner sep=1pt,fill=gray,color=gray] (z2) at (4.5,-.2) {};
    \node at (4.5,-.8) {$Z$};
    \path[color=gray, very thick, shorten >=10pt, shorten <=5pt, ->, >=stealth]
    (x1) edge (C);
    \path[color=gray, very thick, shorten >=10pt, shorten <=5pt, ->, >=stealth]
    (x2) edge (C);
    \path[color=gray, very thick, shorten >=10pt, shorten <=5pt, ->, >=stealth]
    (y1) edge (B);
    \path[color=gray, very thick, shorten >=10pt, shorten <=5pt, ->, >=stealth]
    (y2) edge (C);
    \path[color=gray, very thick, shorten >=10pt, shorten <=5pt, ->, >=stealth]
    (y3) edge (D);
    \path[color=gray, very thick, shorten >=10pt, shorten <=5pt, ->, >=stealth]
    (y4) edge (D);
    \path[color=gray, very thick, shorten >=10pt, shorten <=5pt, ->, >=stealth]
    (y1) edge (A');
    \path[color=gray, very thick, shorten >=10pt, shorten <=5pt, ->, >=stealth]
    (y2) edge (B');
    \path[color=gray, very thick, shorten >=10pt, shorten <=5pt, ->, >=stealth]
    (y3) edge (C');
    \path[color=gray, very thick, shorten >=10pt, shorten <=5pt, ->, >=stealth]
    (y4) edge (D');
    \path[color=gray, very thick, shorten >=10pt, shorten <=5pt, ->, >=stealth]
    (z1) edge (B');
    \path[color=gray, very thick, shorten >=10pt, shorten <=5pt, ->, >=stealth]
    (z2) edge (E');
  \end{tikzpicture}
\end{center}
becomes
\begin{center}
  \begin{tikzpicture}[auto,scale=2]
    \node[circle,draw,inner sep=1pt,fill=gray,color=gray] (x1) at (1.5,.2) {};
    \node[circle,draw,inner sep=1pt,fill=gray,color=gray] (x2) at (1.5,-.2) {};
    \node at (1.5,-.8) {$X$};
    \node[circle,draw,inner sep=1pt,fill]         (A) at (3,.5) {};
    \node[circle,draw,inner sep=1pt,fill]         (B) at (3,.25) {};
    \node[circle,draw,inner sep=1pt,fill]         (C) at (3,0) {};
    \node[circle,draw,inner sep=1pt,fill]         (D) at (3,-.25) {};
    \node[circle,draw,inner sep=1pt,fill]         (E) at (3,-.5) {};
    \node at (3,-.8) {$N+_YM$};
    \node[circle,draw,inner sep=1pt,fill=gray,color=gray] (z1) at (4.5,.2) {};
    \node[circle,draw,inner sep=1pt,fill=gray,color=gray] (z2) at (4.5,-.2) {};
    \node at (4.5,-.8) {$Z$};
    \path[color=gray, very thick, shorten >=10pt, shorten <=5pt, ->, >=stealth]
    (x1) edge (C);
    \path[color=gray, very thick, shorten >=10pt, shorten <=5pt, ->, >=stealth]
    (x2) edge (C);
    \path[color=gray, very thick, shorten >=10pt, shorten <=5pt, ->, >=stealth]
    (z1) edge (C);
    \path[color=gray, very thick, shorten >=10pt, shorten <=5pt, ->, >=stealth]
    (z2) edge (E);
  \end{tikzpicture}
\end{center}
Here we see essentially the same phenomenon as we described for circuits above:
the apex of the cospan is much larger than the image of the maps from the feet.

Corelations address this with what is known as a $(\mc E,\mc M)$-factorisation
system. A factorisation system comprises subcategories $\mc E$ and $\mc M$ of
$\mc C$ such that every morphism in $\mc C$ factors, in a coherent way, as the
composite of a morphism in $\mc E$ followed by a morphism in $\mc M$. An
example, known as the epi-mono factorisation system on $\Set$, is yielded by the
observation that every function can be written as a surjection followed by an
injection.

Corelations, or more precisely $(\mc E,\mc M)$-corelations, are cospans $X
\to N \leftarrow Y$ such that the copairing $X+Y \to N$ of the two maps is an
element of the first factor $\mc E$ of the factorisation system. Composition of
corelations proceeds first as composition of cospans, but then takes only the
so-called $\mc E$-part of the composite cospan, to ensure the composite is again
a corelation. If we take the $\mc E$-part of a cospan $X \to N \leftarrow Y$, we
write the new apex $\overline{N}$, and so the resulting corelation $X \to
\overline{N} \leftarrow Y$. 

Mapping the above two cospans to epi-mono corelations in $\FinSet$ they become 
\begin{center}
  \begin{tikzpicture}[auto,scale=2]
    \node[circle,draw,inner sep=1pt,fill=gray,color=gray] (x1) at (-1.5,.2) {};
    \node[circle,draw,inner sep=1pt,fill=gray,color=gray] (x2) at (-1.5,-.2) {};
    \node at (-1.5,-.8) {$X$};
    \node[circle,draw,inner sep=1pt,fill]         (B) at (0,.3) {};
    \node[circle,draw,inner sep=1pt,fill]         (C) at (0,0) {};
    \node[circle,draw,inner sep=1pt,fill]         (D) at (0,-.3) {};
    \node at (0,-.8) {$\overline{N}$};
    \node[circle,draw,inner sep=1pt,fill=gray,color=gray] (y1) at (1.5,.3) {};
    \node[circle,draw,inner sep=1pt,fill=gray,color=gray] (y2) at (1.5,.1) {};
    \node[circle,draw,inner sep=1pt,fill=gray,color=gray] (y3) at (1.5,-.1) {};
    \node[circle,draw,inner sep=1pt,fill=gray,color=gray] (y4) at (1.5,-.3) {};
    \node at (1.5,-.8) {$Y$};
    \node[circle,draw,inner sep=1pt,fill]         (A') at (3,.5) {};
    \node[circle,draw,inner sep=1pt,fill]         (B') at (3,.25) {};
    \node[circle,draw,inner sep=1pt,fill]         (C') at (3,0) {};
    \node[circle,draw,inner sep=1pt,fill]         (D') at (3,-.25) {};
    \node[circle,draw,inner sep=1pt,fill]         (E') at (3,-.5) {};
    \node at (3,-.8) {$M=\overline{M}$};
    \node[circle,draw,inner sep=1pt,fill=gray,color=gray] (z1) at (4.5,.2) {};
    \node[circle,draw,inner sep=1pt,fill=gray,color=gray] (z2) at (4.5,-.2) {};
    \node at (4.5,-.8) {$Z$,};
    \path[color=gray, very thick, shorten >=10pt, shorten <=5pt, ->, >=stealth]
    (x1) edge (C);
    \path[color=gray, very thick, shorten >=10pt, shorten <=5pt, ->, >=stealth]
    (x2) edge (C);
    \path[color=gray, very thick, shorten >=10pt, shorten <=5pt, ->, >=stealth]
    (y1) edge (B);
    \path[color=gray, very thick, shorten >=10pt, shorten <=5pt, ->, >=stealth]
    (y2) edge (C);
    \path[color=gray, very thick, shorten >=10pt, shorten <=5pt, ->, >=stealth]
    (y3) edge (D);
    \path[color=gray, very thick, shorten >=10pt, shorten <=5pt, ->, >=stealth]
    (y4) edge (D);
    \path[color=gray, very thick, shorten >=10pt, shorten <=5pt, ->, >=stealth]
    (y1) edge (A');
    \path[color=gray, very thick, shorten >=10pt, shorten <=5pt, ->, >=stealth]
    (y2) edge (B');
    \path[color=gray, very thick, shorten >=10pt, shorten <=5pt, ->, >=stealth]
    (y3) edge (C');
    \path[color=gray, very thick, shorten >=10pt, shorten <=5pt, ->, >=stealth]
    (y4) edge (D');
    \path[color=gray, very thick, shorten >=10pt, shorten <=5pt, ->, >=stealth]
    (z1) edge (B');
    \path[color=gray, very thick, shorten >=10pt, shorten <=5pt, ->, >=stealth]
    (z2) edge (E');
  \end{tikzpicture}
\end{center}
with composite
\begin{center}
  \begin{tikzpicture}[auto,scale=2]
    \node[circle,draw,inner sep=1pt,fill=gray,color=gray] (x1) at (1.5,.2) {};
    \node[circle,draw,inner sep=1pt,fill=gray,color=gray] (x2) at (1.5,-.2) {};
    \node at (1.5,-.45) {$X$};
    \node[circle,draw,inner sep=1pt,fill]         (A) at (3,.2) {};
    \node[circle,draw,inner sep=1pt,fill]         (B) at (3,-.2) {};
    \node at (3,-.45) {$\overline{N+_YM}$};
    \node[circle,draw,inner sep=1pt,fill=gray,color=gray] (z1) at (4.5,.2) {};
    \node[circle,draw,inner sep=1pt,fill=gray,color=gray] (z2) at (4.5,-.2) {};
    \node at (4.5,-.45) {$Z$.};
    \path[color=gray, very thick, shorten >=10pt, shorten <=5pt, ->, >=stealth]
    (x1) edge (A);
    \path[color=gray, very thick, shorten >=10pt, shorten <=5pt, ->, >=stealth]
    (x2) edge (A);
    \path[color=gray, very thick, shorten >=10pt, shorten <=5pt, ->, >=stealth]
    (z1) edge (A);
    \path[color=gray, very thick, shorten >=10pt, shorten <=5pt, ->, >=stealth]
    (z2) edge (B);
  \end{tikzpicture}
\end{center}
Note that the apex of the composite corelation is the subset of the apex of the
composite cospan comprising those elements that are in the image of the maps
from the feet. The intuition, again, is that composition of corelations discards
irrelevant information---of course, exactly what it discards depends on our
choice of factorisation system.

In this chapter we show that given a category $\mc C$ with finite colimits and a
factorisation system $(\mc E,\mc M)$, if $\mc M$ obeys a condition known as
`stability under pushout', then corelations in $\mc C$ form a hypergraph
category. We also show that given a colimit-preserving functor $A$ between such
categories $\mathcal C$, $\mathcal C'$ with factorisation systems $(\mathcal E,
\mathcal M)$, $(\mathcal E', \mathcal M')$, $A$ induces a hypergraph functor
between their corelation categories if the image of $\mathcal M$ lies in
$\mathcal M'$.

\section{Corelations} \label{sec.corels}

Given sets $X$, $Y$, a relation $X \to Y$ is a subset of the product $X
\times Y$. Note that by the universal property of the product, spans $X
\leftarrow N \to Y$ are in one-to-one correspondence with functions $N \to X
\times Y$. When this map is monic, we say that the span is \emph{jointly monic}.
More abstractly then, we might say a relation is an isomorphism class of jointly
monic spans in the category of sets. Here we generalise the dual concept: these
are our so-called corelations.

Relations and, equivalently, multi-valued functions have long been a structure
of mathematical interest, with their formalisation going back to De Morgan
\cite{DeM60} and Pierce \cite{Pie70} in the nineteenth century. One hundred
years later, their category theoretic generalisation as subobjects of pairwise
products was introduced by Puppe \cite{Pup62} and Mac Lane \cite{Mac63a} in the
setting of abelian categories, and Barr \cite{Bar70} for regular categories.
Subsequently, Klein \cite{Kle70} provided conditions under which these relations
can be given an associative composition rule---albeit conditions slightly more
restrictive than that which we use here---and by now the categorical theory of
relations is well studied \cite{FS90, Mil00, JW00}.  The key insight is the
definition of a factorisation system. We begin this section by introducing this
idea, before developing the theory of not just categories of relations, but
monoidal categories of (co)relations.

\subsection{Factorisation systems}
The relevant properties of jointly monic spans come from the fact that
monomorphisms form one half of a factorisation system. A factorisation system
allows any morphism in a category to be factored into the composite of two
morphisms in a coherent way. This subsection introduces factorisation systems
and monoidal factorisation systems.

\begin{definition}
  A \define{factorisation system} $(\mathcal E,\mathcal M)$ in a category
  $\mathcal C$ comprises subcategories $\mathcal E$, $\mathcal M$ of $\mathcal
  C$ such that
  \begin{enumerate}[(i)]
    \item $\mathcal E$ and $\mathcal M$ contain all isomorphisms of $\mathcal
      C$.
    \item  every morphism $f \in \mathcal C$ admits a factorisation $f=m \circ
      e$, $e \in \mathcal E$, $m \in \mathcal M$.
\item given morphisms $f,f'$, with factorisations $f = m \circ e$, $f' = m' \circ
  e'$ of the above sort, for every $u$, $v$ such that the square
  \[
    \xymatrixcolsep{3pc}
    \xymatrixrowsep{3pc}
    \xymatrix{
       \ar[r]^f \ar[d]_u &  \ar[d]^v \\
       \ar[r]_{f'} & 
    }
  \]
  commutes, there exists a unique morphism $s$ such that
  \[
    \xymatrixcolsep{3pc}
    \xymatrixrowsep{3pc}
    \xymatrix{
      \ar[r]^e \ar[d]_u & \ar[r]^m \ar@{.>}[d]^{\exists! s} &  \ar[d]^v \\
       \ar[r]_{e'}& \ar[r]_{m'} & 
    }
  \]
  commutes.
  \end{enumerate}
\end{definition}

\begin{examples} \label{ex.factsysts}\ 

  \begin{itemize}
    \item Write $\mathcal I_{\mathcal C}$ for the wide subcategory of
      $\mathcal C$ containing exactly the isomorphisms of $\mathcal C$. Then
      $(\mathcal I_{\mathcal C}, \mathcal C)$ and $(\mathcal C, \mathcal
      I_{\mathcal C})$ are both factorisation systems in $\mathcal C$. While
      these may seem too trivial to mention, we will see they are of central
      importance in what follows.
    
    \item The prototypical example of a factorisation system is the epi-mono
      factorisation system $(\mathrm{Sur},\mathrm{Inj})$ in $\Set$. Here we
      write $\mathrm{Sur}$ for the subcategory of surjections in $\Set$, and
      $\mathrm{Inj}$ for the subcategory of injections. 
      
      Recall that a split monomorphism is a map $m\colon X\to Y$ such that there
      exists a one-sided inverse, i.e. a map $m'\colon Y\to X$ such that
      $m'm=\idn_X$. Observe that, assuming the axiom of choice, all monos in
      $\Set$ split.  One way of proving that the above is a factorisation system
      on $\Set$ is via the more general fact, true in any category: if every
      arrow can be factorised as an epi followed by a split mono, then
      epimorphisms and split monomorphisms form the factors of a factorisation
      system.  The only non-trivial part to check is the uniqueness condition:
      given epis $e_1,e_2$, split monos $m_1,m_2$, and commutative diagram
      \[
	\xymatrixcolsep{3pc}
	\xymatrixrowsep{3pc}
	\xymatrix{
	  \ar[r]^{e_1} \ar[d]_u & \ar[r]^{m_1} \ar@{.>}[d]^{\exists! t} &
	  \ar[d]^v \\
	  \ar[r]_{e_2}& \ar[r]_{m_2} & 
	}
      \]
      we must show that there is a unique $t$ that makes the diagram commute.
      Indeed let $t= m_2'vm_1$ where $m_2'$ satisfies $m_2'm_2=id$. 
      To see that the right square commutes, observe
      \[
	m_2 t e_1 =  m_2 m_2' v m_1 e_1 = m_2 m_2' m_2 e_2 u = m_2 e_2 u = v m_1 e_1
      \]
      and since $e_1$ is epi we have $m_2 t = v m_1$. For the left square,
      \[
	t e_1 = m_2' v m_1 e_1 = m_2' m_2 e_2 u = e_2 u.
      \] 
      Uniqueness is immediate, since, $e_1$ is epi and $m_2$ is mono. 
    
    \item Recall that a regular epimorphism is an epimorphism that is the
      coequaliser for some pair of parallel morphisms. (Dually, a regular
      monomorphism is an equaliser for some pair of parallel morphisms.) In any
      so-called regular category the regular epimorphisms and monomorphisms form
      a factorisation system.  Examples of regular categories include $\Set$,
      toposes, and abelian categories. Regular categories were introduced by
      Barr and Grillet; for more details see \cite{Bar71,Gri71}.
\end{itemize}
More details and further examples can be found in \cite[\textsection 14]{AHS}.
\end{examples}

As we are concerned with building monoidal categories of corelations, it will be
important that our factorisation systems are monoidal factorisation systems.

\begin{definition}
Call a factorisation system $(\mc E,\mc M)$ in a monoidal category $(\mc C,\ot)$
a \define{monoidal factorisation system} if $(\mc E,\ot)$ is a monoidal
category.
\end{definition}

One might wonder why $\mc M$ does not appear in the above definition. To give a
touch more intuition for this definition, we quote a theorem of Ambler. Recall a
symmetric monoidal closed category is one in which each functor $- \ot X$ has a
specified right adjoint $[X,-]$. 
\begin{proposition}
  Let $(\mc E,\mc M)$ be a factorisation system in a symmetric monoidal
  closed category $(\mc C,\ot)$. Then the following are equivalent:
  \begin{enumerate}[(i)]
    \item $(\mc E,\mc M)$ is a monoidal factorisation system.
    \item $\mc E$ is closed under $- \ot X$ for all $X \in \mc C$.
    \item $\mc M$ is closed under $[X,-]$ for all $X \in \mc C$.
\end{enumerate}
\end{proposition}
\begin{proof}
  See Ambler for proof and further details \cite[Lemma 5.2.2]{Am}.
\end{proof}

We need not worry too much, however, about the distinction between factorisation
systems and monoidal factorisation systems. The reason is that for our
purposes---where the underlying category $\mc C$ has finite colimits and the
monoidal product is the coproduct---\emph{all} factorisation systems are
monoidal factorisation systems. This is implied by the following lemma.

\begin{lemma} \label{lem.monfact}
  Let $\mc C$ be a category with finite coproducts, and let $(\mc E, \mc M)$ be a
  factorisation system on $\mc C$. Then $(\mc E,+)$ is a symmetric monoidal
  category.
\end{lemma}
\begin{proof}
  The only thing to check is that $\mc E$ is closed under $+$. That is, given
  $f\maps A \to B$ and $g\maps C \to D$ in $\mc E$, we wish to show that
  $f+g\maps A+C \to B+D$, defined in $\mc C$, is also a morphism in $\mc E$. 

  Let $f+g$ have factorisation $A+C \stackrel{e}\longrightarrow \overline{B+D}
  \stackrel{m}\longrightarrow B+D$, where $e \in \mc E$ and $m \in \mc
  M$. We will prove that $m$ is an isomorphism. To construct an inverse, recall
  that by definition, as $f$ and $g$ lie in $\mc E$, there exist morphisms
  $x\maps B \to \overline{B+D}$ and $y\maps D \to \overline{B+D}$ such that
  \[ \label{eq.coreltensor}
    \xymatrixcolsep{2pc}
    \xymatrixrowsep{2pc}
    \xymatrix{
      A \ar[r]^f \ar[d] & B \ar@{=}[r] \ar@{.>}[d]^x & B
      \ar[d] \\
      A+C \ar[r]_{e}&\overline{B+D} \ar[r]_{m} & B+D
    }
    \qquad \mbox{and} \qquad
    \xymatrixcolsep{2pc}
    \xymatrixrowsep{2pc}
    \xymatrix{
      C \ar[r]^g \ar[d] & D \ar@{=}[r] \ar@{.>}[d]^y & D
      \ar[d] \\
      A+C \ar[r]_{e}&\overline{B+D} \ar[r]_{m} & B+D
    }
    \tag{1}
  \]
  The copairing $[x,y]$ is an inverse to $m$. 
  
  Indeed, taking the coproduct of the top rows of the two diagrams above and the
  copairings of the vertical maps gives the commutative diagram
  \[
    \xymatrix{
      A+C \ar[r]^{f+g} \ar@{=}[d] & B+D \ar@{=}[r] \ar[d]_{[x,y]} & B+D \ar@{=}[d] \\
      A+C \ar[r]^{e} & \overline{B+D} \ar[r]^{m} & B+D
    }
  \]
  Reading the right-hand square immediately gives $m \circ [x,y] =1$.
  
  Conversely, to see that $[x,y] \circ m = 1$, remember that by definition $f+g
  = m \circ e$. So the left-hand square above implies that
  \[
    \xymatrixcolsep{2pc}
    \xymatrixrowsep{2pc}
    \xymatrix{
      A+C \ar[r]^e \ar@{=}[d] & \overline{B+D} \ar[d]^{[x,y] \circ m} \\
      A+C \ar[r]_{e}&\overline{B+D} 
    }
  \]
  commutes. But by the universal property of factorisation systems, there is a
  unique map $\overline{B+D} \to \overline{B+D}$ such that this diagram
  commutes, and clearly the identity map also suffices. Thus $[x,y] \circ m =
  1$.
\end{proof}

\subsection{Corelations}
Now that we have introduced factorisation systems, observe that relations are
just spans $X \leftarrow N \to Y$ in $\Set$ such that $N \to X \times Y$ is an
element of $\mathrm{Inj}$, the right factor in the factorisation system
$(\mathrm{Sur},\mathrm{Inj})$. Relations may thus be generalised as spans such
that the span maps jointly belong to some class $\mc M$ of an $(\mc E,\mc
M)$-factorisation system. We define corelations in the dual manner.

\begin{definition}
  Let $\mathcal C$ be a category with finite colimits, and let $(\mathcal E,
  \mathcal M)$ be a factorisation system on $\mathcal C$. An $(\mathcal
  E,\mathcal M)$\define{-corelation} $X \to Y$ is a cospan $X
  \stackrel{i}\longrightarrow N \stackrel{o}\longleftarrow Y$ in $\mc C$ such
  that the copairing $[i,o]\maps X+Y \to N$ lies in $\mathcal E$.
\end{definition}

When the factorisation system is clear from context, we simply call $(\mathcal
E,\mathcal M)$-corelations `corelations'.

We also say that a cospan $X \stackrel{i}\longrightarrow N
\stackrel{o}\longleftarrow Y$ with the property that the copairing $[i,o]\maps
X+Y \to N$ lies in $\mathcal E$ is \define{jointly} $\mathcal E$\define{-like}.
Note that if a cospan is jointly $\mc E$-like then so are all isomorphic
cospans. Thus the property of being a corelation is closed under isomorphism of
cospans, and we again are often lazy with our language, referring to both
jointly $\mc E$-like cospans and their isomorphism classes as corelations. 

If $f\maps A \to N$ is a morphism with factorisation $f = m \circ e$, write
$\overline N$ for the object such that $e\maps A \to \overline N$ and $m\maps
\overline N \to N$. Now, given a cospan $X \stackrel{i_X}{\longrightarrow} N
\stackrel{o_Y}{\longleftarrow} Y$, we may use the factorisation system to write
the copairing $[i_X,o_Y]\maps X+Y \to N$ as
\[
  X+Y \stackrel{e}{\longrightarrow} \overline{N} \stackrel{m}{\longrightarrow}
  N.
\]
From the universal property of the coproduct, we also have maps $\iota_X\maps X
\to X+Y$ and $\iota_Y\maps Y \to X+Y$. We then call the corelation 
\[
  X \stackrel{e\circ \iota_X}{\longrightarrow} \overline{N} \stackrel{e \circ
  \iota_Y}{\longleftarrow} Y
\]
the $\mathcal E$\define{-part} of the above cospan. On occasion we will also
write $e\maps X+Y \to \overline N$ for the same corelation.

\begin{examples} \label{ex.corels}
  Many examples of corelations are already familiar.
  \begin{itemize}
    \item For the morphism-isomorphism factorisation system $(\mc C,\mc I_{\mc
      C})$, corelations are just cospans.

    \item For the isomorphism-morphism factorisation $(\mc I_{\mc C}, \mc C)$,
      jointly $\mc I_{\mc C}$-like cospans $X \to Y$ are simply isomorphisms
      $X+Y \stackrel\sim\to N$. Thus there is a unique isomorphism class of
      corelations between any two objects.

    \item Note that the category $\Set$ has finite colimits and an epi-mono
      factorisation system $(\mathrm{Sur},\mathrm{Inj})$. Epi-mono corelations
      from $X \to Y$ in $\mathrm{Set}$ are surjective functions $X+Y \to N$;
      thus their isomorphism classes are partitions, or equivalence relations on
      $X+Y$. 

      While this compositional structure on equivalence relations has not the
      prominence of that on relations, these corelations have been recognised as
      an important structure. Ellerman gives a detailed treatment from a logic
      viewpoint in \cite{Ell14}, while basic category theoretic aspects can be
      found in Lawvere and Rosebrugh \cite{LR}.  Note that in these and in other
      sources, including \cite{CF,BF}, the term corelation is used to solely
      refer to these $(\mathrm{Sur},\mathrm{Inj})$-corelations in $\Set$, while
      here we use the term corelation primarily in the generalised sense.
  \end{itemize}
\end{examples}

Our next task is to define a composition rule on corelations.

\subsection{Categories of corelations} \label{ssec.corelcats}
We begin this subsection by defining a composition rule on isomorphism classes
of corelations. The end goal, however, is to define a hypergraph category in
which the morphism are corelations. Here we will explain how to define such a
category, introducing and exploring the important condition that $\mc M$ is
stable under pushout. We leave the proof that we have truly defined a hypergraph
category for the next subsection.

We compose corelations by taking the $\mathcal E$-part of their composite
cospan. That is, given corelations $X \stackrel{i_X}{\longrightarrow} N
\stackrel{o_Y}{\longleftarrow} Y$ and $Y \stackrel{i_Y}{\longrightarrow} M
\stackrel{o_Z}{\longleftarrow} Z$, their composite is given by the cospan $X
\xrightarrow{e\circ\iota_X} \overline{N+_YM} \xleftarrow{e \circ \iota_Z} Z$ in the commutative diagram
\[
\includegraphics{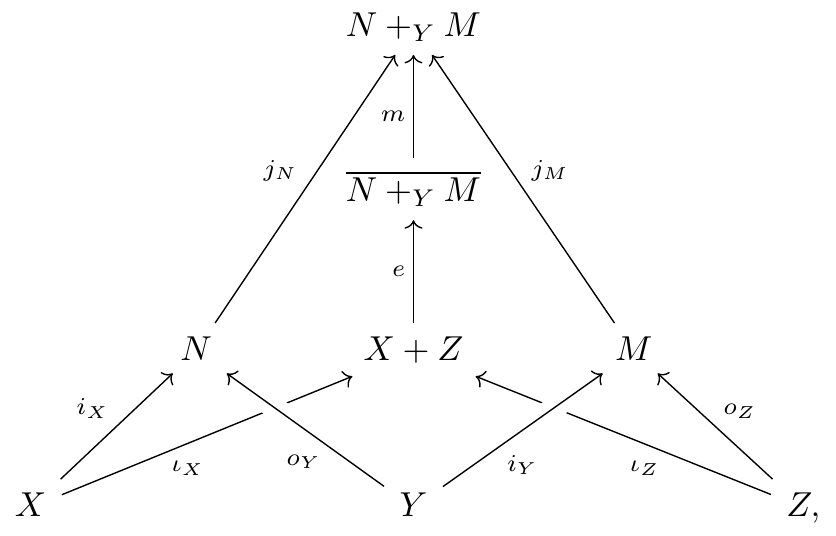}
\]
%\[
%  \begin{tikzcd}[row sep=7ex,column sep=7ex]
%    && N+_YM \\
%    && \overline{N+_YM} \arrow[u,pos=.4,"m"] \\
%    & N \arrow[uur,"j_N"] & X+Z \arrow[from=dll,pos=.35,swap,"\iota_X"]
%    \arrow[from=drr,pos=.35,"\iota_Z"]
%    \arrow[u,"e"] & 
%    M \arrow[uul,swap,"j_M"] \\
%    X \arrow[ur,"i_X"] && Y
%    \arrow[ul,crossing over,pos=.35,"o_Y"] \arrow[ur,crossing
%    over,pos=.35,swap,"i_Y"] && Z,
%    \arrow[ul,swap,"o_Z"]
%  \end{tikzcd}
%\]
where $m \circ e$ is the $(\mc E,\mc M)$-factorisation of $[j_N\circ i_X,j_M
\circ o_Z]\maps X+Z
\to N+_YM$. 

It is well known that this composite is unique up to isomorphism, and that when
$\mc M$ is well behaved it defines a category with isomorphism classes of
corelations as morphisms. For instance, a bicategorical version of the dual
theorem, for spans and relations, can be found in \cite{JW00}. Nonetheless, for
the sake of completeness we will explain all the details and sketch our own
argument here. The first fact to show is that, as we have just stated, the
composite of corelations is unique up to isomorphism.

\begin{proposition} \label{prop.corelcomp}
  Let $\mc C$ be a category with finite colimits and with a factorisation system
  $(\mc E,\mc M)$. Then the above is a well-defined composition rule on
  isomorphism classes of corelations.
\end{proposition}
\begin{proof}
  Let
  $(X \stackrel{i_X}{\longrightarrow} N \stackrel{o_Y}{\longleftarrow} Y$,
  $X \stackrel{i_X'}{\longrightarrow} N' \stackrel{o_Y'}{\longleftarrow} Y)$
%  \[
%    X \stackrel{i_X}{\longrightarrow} N \stackrel{o_Y}{\longleftarrow} Y 
%    \qquad \mbox{and} \qquad 
%    X \stackrel{i_X'}{\longrightarrow} N' \stackrel{o_Y'}{\longleftarrow} Y
%  \]
  and
  $(Y \stackrel{i_Y}{\longrightarrow} M \stackrel{o_Z}{\longleftarrow} Z$, $Y
  \stackrel{i_Y'}{\longrightarrow} M' \stackrel{o_Z'}{\longleftarrow} Z)$
%  \[
%    Y \stackrel{i_Y}{\longrightarrow} M \stackrel{o_Z}{\longleftarrow} Z 
%    \qquad \mbox{and} \qquad 
%    Y \stackrel{i_Y'}{\longrightarrow} M' \stackrel{o_Z'}{\longleftarrow} Z
%  \]
  be pairs of isomorphic jointly $\mathcal E$-like cospans. By Proposition
  \ref{prop.composingdeccospans} their composites \emph{as cospans} 
  % $(X \longrightarrow N+_YM \longleftarrow Z)$ and $(X
  %\longrightarrow N'+_YM' \longleftarrow Z)$
%  \[
%    X \longrightarrow N+_YM \longleftarrow Z 
%    \qquad \mbox{and} \qquad 
%    X \longrightarrow N'+_YM' \longleftarrow Z
%  \]
  are isomorphic via an isomorphism $p\maps N+_YM \to N'+_YM'$. The
  factorisation system then gives an isomorphism $s$ such that the diagram
  \[
    \xymatrixcolsep{3pc}
    \xymatrixrowsep{2pc}
    \xymatrix{
      X+Z \ar[r]^e \ar@{=}[d] & \overline{N+_YM} \ar[r]^m \ar@{.>}[d]^{s}_\sim & N+_YM
      \ar[d]_\sim^{p} \\
      X+Z \ar[r]_{e'}&\overline{N'+_YM'} \ar[r]_{m'} & N'+_YM'
    }
  \]
  commutes. Thus $s$ is an isomorphism of the composite corelations.
\end{proof}

As we have said, this composition rule only gives a category when $\mc M$ is
well behaved. The reason is that composition of corelations is \emph{not}
associative in general. It is, however, associative when $\mathcal M$ is stable
under pushout. We now define this, provide some examples, and prove a crucial
lemma.

\begin{definition}
  Given a category $\mc C$, we say that a subcategory $\mc M$ is \define{stable
  under pushout} if for every pushout square
  \[
    \xymatrixcolsep{3pc}
    \xymatrixrowsep{3pc}
    \xymatrix{
      \ar[r]^j & \\
      \ar[u] \ar[r]^m &  \ar[u]
    }
  \]
  such that $m \in \mathcal M$, we also have that $j \in \mathcal M$. 
\end{definition}

\begin{examples}
  There are many examples of $(\mc E,\mc M)$-factorisation systems with $\mc M$
  stable under pushout, including the factorisation systems $(\mc I_{\mc C},\mc
  C)$ and $(\mc C,\mc I_{\mc C})$ in $\mc C$, and $(\mathrm{Sur},\mathrm{Inj})$
  in $\Set$.

  This last example generalises to any topos. Indeed, Lack and Soboci\'nski
  showed that monomorphisms are stable under pushout in any adhesive category
  \cite{LS04}.  Since any topos is both a regular category and an adhesive
  category \cite{LS06,Lac11}, the regular epimorphism-monomorphism factorisation
  system in any topos is an $(\mc E,\mc M)$-factorisation system with $\mc M$
  stable under pushout.
  
  Another example is the dual of any regular category. Such a category is known
  as a coregular category, and is by definition a category that has finite
  colimits and an epimorphism-regular monomorphism factorisation system with
  regular monomorphisms stable under pushout. Examples of these include the
  category of topological spaces and continuous maps, as well as $\Set^\opp$,
  any cotopos, and so on.
\end{examples}

Stability under pushout is a powerful property. A key corollary, both for
associativity and in general, is that it implies $\mc M$ is also closed under
$+$. 
%For our proof sketch we rely on the following useful lemma.

\begin{lemma} \label{lem.mcoproductsmc}
  Let $\mathcal C$ be a category with finite colimits, and let $\mathcal M$ be a
  subcategory of $\mathcal C$ stable under pushouts and containing all
  isomorphisms. Then $(\mc M,+)$ is a symmetric monoidal category.
\end{lemma}
\begin{proof}
  It is enough to show that for all morphisms $m,m' \in \mc M$ we have $m+m'$ in
  $\mc M$. Since $\mc M$ contains all isomorphisms, the coherence maps are
  inherited from $\mc C$. The required axioms---the functoriality of the tensor
  product, the naturality of the coherence maps, and the coherence laws---are
  also inherited as they hold in $\mc C$.

  To see $m+m'$ is in $\mc M$, simply observe that we have the pushout square
  \[
  \xymatrixcolsep{3pc}
  \xymatrixrowsep{2pc}
    \xymatrix{
      A+C \ar[r]^{m+1} & B+C \\
      A \ar[r]^m \ar[u]^{\iota} & B \ar[u]_{\iota} \\
    }
  \]
  in $\mc C$. As $\mc M$ is stable under pushout, $m+1 \in \mc M$. Similarly,
  $1+m' \in \mc M$. Thus their composite $m+m'$ lies in $\mc M$, as required.
\end{proof}

An analogous argument shows that pushouts of maps $m+_Ym'$ also lie in $\mc M$.
Using this lemma it is not difficult to show associativity---the key point is
that factorisation `commutes' with pushouts, and that we have a category
$\corel_{(\mc E,\mc M)}(\mc C)$. Again, this is all well known, and can be found
in \cite{JW00}. We will incidentally reprove these facts in the following,
while pursuing richer structure.

Indeed, for modelling networks, we require not just a category, but a hypergraph
category. Corelation categories come equipped with this extra structure.  Recall
that we gave decorated cospan categories a hypergraph structure by defining a
wide embedding $\mathrm{Cospan}(\mc C) \hookrightarrow F\mathrm{Cospan}$, via
which $F\mathrm{Cospan}$ inherited the coherence and Frobenius maps (Theorem
\ref{thm:fcospans}). We will argue similarly here, after showing that the map
\[
  \cospan(\mc C) \longrightarrow \corel(\mc C)
\]
taking each cospan to its jointly $\mc E$-like part is functorial. Indeed, we
define the coherence and Frobenius maps of $\corel(\mc C)$ to be their image
under this map. For the monoidal product we again use the coproduct in $\mc C$;
the monoidal product of two corelations is their monoidal product as cospans.

\begin{theorem} \label{thm.cospantocorel}
  Let $\mathcal C$ be a category with finite colimits, and let $(\mathcal E,
  \mathcal M)$ be a factorisation system on $\mathcal C$ such that $\mathcal M$
  is stable under pushout. Then there exists a hypergraph category
  $\mathrm{Corel}_{(\mc E,\mc M)}(\mathcal C)$ with 
  \smallskip

  \begin{center}
    \begin{tabular}{| c | p{.65\textwidth} |}
      \hline
      \multicolumn{2}{|c|}{The hypergraph category $(\mathrm{Corel}_{(\mc E,\mc M)}(\mc C),+)$} \\
      \hline
      \textbf{objects} & the objects of $\mathcal C$ \\ 
      \textbf{morphisms} & isomorphism classes of $(\mc E,\mc M)$-corelations in $\mathcal C$\\ 
      \textbf{composition} & given by the $\mc E$-part of pushout \\
      \textbf{monoidal product} & the coproduct in $\mathcal C$ \\
      \textbf{coherence maps} & inherited from $\cospan(\mc C)$  \\
      \textbf{hypergraph maps} & inherited from $\cospan(\mc C)$ \\
      \hline
    \end{tabular}
  \end{center}  
  \smallskip
\end{theorem}

Again, we will drop explicit reference to the factorisation system when context allows, simply writing $\mathrm{Corel}(\mc C)$.

\begin{examples} \label{ex.corelcats}
In each factorisation system of Examples \ref{ex.corels} the right factor
$\mathcal M$ is stable under pushout. The hypergraph category 
$\corel_{(\mc C,\mc I_{\mc C})}(\mc C)$ is just the hypergraph category of
cospans in $\mc C$. In the hypergraph category $\corel_{(\mc I_{\mc C},\mc
C)}(\mc C)$, the only morphism between any two objects $X$ and $Y$ is the
isomorphism class of the corelation $X \xrightarrow{\iota_X} X+Y
\xleftarrow{\iota_Y} Y$ corresponding to the identity map $X+Y \to X+Y$ in $\mc
I_{\mc C}$. Thus $\corel_{(\mc I_{\mc C},\mc C)}(\mc C)$ is the indiscrete
category on the objects of $\mc C$. 

We discussed epi-mono corelations in $\FinSet$ informally in our motivation
section \textsection\ref{sec.blackboxing}. These are equivalence relations
between finite sets. This example is perhaps the most instructive for
black-boxing open systems, and we will return to it in
\textsection\ref{ssec.equivrels}.
\end{examples}

Proposition \ref{prop.corelcomp} shows that our composition rule is a
well-defined function; Lemma \ref{lem.monfact} shows likewise for the monoidal
product $+\maps \corel(\mc C)\times \corel(\mc C) \to \corel(\mc C)$. Thus we
have the required data for a hypergraph category. It remains to check a
number of axioms: associativity and unitality of the categorical composition,
functoriality of the monoidal product, naturality of the coherence maps, the
coherence axioms for symmetric monoidal categories, the Frobenius laws.

Our strategy for this will be to show that the surjective function from cospans
to corelations defined by taking a cospan to its jointly $\mc E$-part preserves
both composition and the monoidal product. This then implies that to evaluate an
expression in the monoidal category of corelations, we may simply evaluate it in
the monoidal category of cospans, and then take the $\mc E$-part. Thus if an
equation is true for cospans, it is true for corelations.

Instead of proving just this, however, we will prove a generalisation regarding
an analogous map between any two corelation categories. Such a map exists
whenever we have two corelation categories $\corel_{(\mc E,\mc M)}(\mc C)$ and
$\corel_{(\mc E',\mc M')}(\mc C')$ and a colimit
preserving functor $A\maps \mc C \to \mc C'$ such that the image of $\mc M$ lies
in $\mc M'$. As $(\mc C,\mc I_{\mc C})$-corelations are just cospans, this
reduces to the desired special case by taking the domain to be the category of
$(\mc C,\mc I_{\mc C})$-corelations, $\mc C'$ to be equal to $\mc C$, and $A$ to
be the identity functor. But the generality is not spurious: it has the
advantage of proving the existence of a class of hypergraph functors between
corelation categories in the same fell swoop.

Although a touch convoluted, this strategy is worth the pause for thought. We
will use it once again for \emph{decorated} corelations, to great economy.

\section{Functors between corelation categories} \label{sec.corelfunctors}
We have seen that to construct a functor between cospan categories one may start
with a colimit-preserving functor between the underlying categories. Corelations
are cospans where we forget the $\mc M$-part of each cospan. Hence for functors
between corelation categories, we require not just a colimit-preserving functor
but, loosely speaking, also that we don't forget too much in the domain category
compared to the codomain category.

We devote the next few pages to proving the following proposition. Along the way
we prove, as promised, that corelation categories are well-defined hypergraph
categories.

\begin{proposition} \label{prop.corelfunctors}
  Let $\mathcal C$, $\mathcal C'$ have finite colimits and respective
  factorisation systems $(\mathcal E, \mathcal M)$, $(\mathcal E', \mathcal M')$,
  such that $\mathcal M$ and $\mathcal M'$ are stable under pushout. Further let
  $A\maps \mathcal C \to \mathcal C'$ be a functor that preserves finite colimits
  and such that the image of $\mathcal M$ lies in $\mathcal M'$.

  Then we may define a hypergraph functor $\square\maps \corel(\mathcal C) \to
  \corel(\mathcal C')$ sending each object $X$ in $\corel(\mathcal C)$ to $AX$ in
  $\corel(\mc C')$ and each corelation 
  \[
    X \stackrel{i_X}{\longrightarrow} N \stackrel{o_Y}{\longleftarrow} Y 
  \]
  to the $\mc E'$-part
  \[
    AX \xrightarrow{e'\circ\iota_{AX}} \overline{AN}
    \xleftarrow{e'\circ\iota_{AY}} AY.
  \]
  of the image cospan. The coherence maps are the $\mc E'$-part
  $\overline{\kappa_{X,Y}}$ of the isomorphisms $\kappa_{X,Y}\maps AX+AY \to
  A(X+Y)$ given as $A$ preserves colimits.
\end{proposition}

As discussed, we still have to prove that $\corel(\mc C)$ is a hypergraph
category. We address this first with two lemmas regarding these proposed
functors.

\begin{lemma} \label{lem.corelfuncomposition}
  The above function $\square\maps \corel(\mc C) \to \corel(\mc C')$ preserves
  composition.
\end{lemma}
\begin{proof}
  Let $f = (X \longrightarrow N \longleftarrow Y)$ and $g= (Y \longrightarrow M
  \longleftarrow Z)$ be corelations in $\mathcal C$. By definition, the
  corelations $\square(g) \circ \square(f)$ and $\square(g \circ f)$ are given
  by the first arrows in the top and bottom row respectively of the diagram:
  \[ \label{diag.eparts}
    \begin{aligned}
      \xymatrixcolsep{5.5pc}
      \xymatrixrowsep{2pc}
      \xymatrix{
	\scriptstyle AX+AZ \ar[r]^{\mc E'} \ar@{=}[d] & \scriptstyle \overline{\overline{AN}+_{AY}\overline{AM}}
	\ar[r]^{\mc M'} \ar@{<.>}[d]^{n} & \scriptstyle \overline{AN}+_{AY}\overline{AM}
	\ar[r]^{m'_{AN}+_{AY}m'_{AM}} & \scriptstyle
	AN+_{AY}AM \\
	\scriptstyle AX+AZ \ar[r]^{\mc E'} & \scriptstyle \overline{A(\overline{N+_YM})} \ar[r]^{\mc M'} & \scriptstyle
	A(\overline{N+_YM}) \ar[r]^{Am_{N+_YM}} & \scriptstyle A(N+_YM) \ar@{<->}[u]_{\sim}
      }
    \end{aligned}
    \tag{$\ast$}
  \]
  The morphisms labelled $\mc E'$ lie in $\mc E'$, and similarly for $\mc M'$;
  these are given by the factorisation system on $\mc C'$.  The maps
  $Am_{N+_YM}$ and $m'_{AN}+_{AY}m'_{AM}$ lie in $\mc M'$ too: $Am_{N+_YM}$ as
  it is in the image of $\mc M$, and $m'_{AN}+_{AY}m'_{AM}$ as $\mc M'$ is
  stable under pushout. 

  Moreover, the diagram commutes as both maps $AX+AZ \to AN+_{AY}AM$ compose to
  that given by the pushout of the images of $f$ and $g$ over $AY$.  Thus the
  diagram represents two $(\mc E', \mc M')$ factorisations of the same morphism,
  and there exists an isomorphism $n$ between the corelations $\square(g) \circ
  \square(f)$ and $\square(g\circ f)$. This proves that $\square$ preserves
  composition.
\end{proof}
This first lemma allows us to verify the associativity and unit laws for
$\corel(\mc C)$.
\begin{corollary}
  $\corel(\mc C)$ is well defined as a category.
\end{corollary}
\begin{proof}
  Consider the case of Proposition \ref{prop.corelfunctors} with $\mc C = \mc
  C'$, $(\mc E,\mc M) = (\mc C, \mc I_{\mc C})$, and $A = 1_{\mc C}$. Then the
  domain of $\square$ is $\cospan(\mc C)$ by definition. In this case, the
  function $\square\maps \cospan(\mc C) \to \corel(\mc C)$ is
  bijective-on-objects and surjective-on-morphisms. Thus to compute the
  composite of any two corelations, we may consider them as cospans, compute
  their composite \emph{as cospans}, and then take the $\mc E$-part of the
  result. Since composition of cospans is associative and unital, so is
  composition of corelations, with the identity corelation just the image of the
  identity cospan.
\end{proof}

Note that the identity in $\corel(\mc C)$ may not be the identity cospan itself.
For example, with the factorisation system $(\mc I_{\mc C}, \mc C)$ the $\mc
I_{\mc C}$-part of the identity cospan is simply $X \xrightarrow{\iota_{X_1}}
X+X \xleftarrow{\iota_{X_2}} X$, where $\iota_{X_i}$ is the inclusion of
$X$ into the $X_i$ factor of the coproduct $X+X$.

This first lemma is also useful in proving the second important lemma: the
naturality of $\overline{\kappa}$.

\begin{lemma} \label{lem.corelfunmonoidal}
  The maps $\overline{\kappa_{X,Y}}$, as defined in Proposition
  \ref{prop.corelfunctors}, are natural.
\end{lemma}
\begin{proof}
  Let $f = (X \longrightarrow N \longleftarrow Y)$, $g= (Z \longrightarrow M
  \longleftarrow W)$ be corelations in $\mc C$. We wish to show that
  \[
    \xymatrixcolsep{4pc}
    \xymatrixrowsep{3pc}
    \xymatrix{
      AX+AY \ar[r]^{\square(f)+\square(g)}
      \ar[d]_{\overline{\kappa_{X,Y}}} & 
      AZ+AW \ar[d]^{\overline{\kappa_{Z,W}}} \\
      A(X+Y) \ar[r]^{\square(f+g)} & A(Z+W)
    }
  \]
  commutes in $\corel(\mc C')$. 

  Consider the following commutative diagram in $\mc C'$, with the outside
  square equivalent to the naturality square for the coherence maps of the
  monoidal functor \linebreak $\cospan(\mc C) \to \cospan(\mc C')$:
  \[ \label{diag.natural}
    \begin{aligned}
      \xymatrixcolsep{4pc}
      \xymatrixrowsep{2.5pc}
      \xymatrix{
	(AX+AY)+(AZ+AW) \ar[r]^(.65){\mc E'+\mc E'}
	\ar[d]_{\kappa_{X,Y}+\kappa_{Z,W}} & 
	\overline{AN}+\overline{AM} \ar[r]^{\mc M'+\mc M'} \ar@{.>}[d]^{p} & 
	AN+AM \ar[d]^{\kappa_{N,M}}\\
	A(X+Y)+A(Z+W) \ar[r]^{\mc E'} & \overline{A(N+M)} \ar[r]^{\mc M'} & A(N+M)
      }
    \end{aligned}
    \tag{$\#$}
  \]
  We have factored the top edge as the coproduct of the respective
  factorisations of $f$ and $g$, and the bottom edge simply as the factorisation
  of the coproduct $f+g$. 

  Note that by Lemma \ref{lem.monfact} the coproduct of two maps in $\mc E'$ is
  again in $\mc E'$, while Lemma \ref{lem.mcoproductsmc} implies the same for
  $\mc M'$. Thus the top edge is an $(\mc E',\mc M')$-factorisation, and the
  uniqueness of factorisations gives the isomorphism $n$. 
  Given that the map reducing cospans to corelations is functorial, the
  commutative square
  \[
    \xymatrixcolsep{2.5pc}
    \xymatrixrowsep{2.5pc}
    \xymatrix{
      (AX+AY)+A(Z+W) \ar[r]^{1+\kappa_{Z,W}^{-1}} \ar@{=}[d] & (AX+AY)+(AZ+AW)
      \ar[r]^(.65){\mc E'+\mc E'} & 
      \overline{AN}+\overline{AM} \ar[d]^{n} \\
      (AX+AY)+A(Z+W) \ar[r]^{\kappa_{X,Y}+1} & A(X+Y)+A(Z+W) \ar[r]^(.6){\mc E'} & 
      \overline{A(N+M)}
    }
  \]
  then implies the naturality of the maps $\overline{\kappa}$.
\end{proof}

These lemmas now imply that $\corel(\mc C)$ is a well-defined hypergraph
category.
\begin{proof}[Proof of Theorem \ref{thm.cospantocorel}]
  To complete the proof then, again consider the case of Proposition
  \ref{prop.corelfunctors} with $\mc C = \mc C'$, $(\mc E,\mc M) = (\mc C, \mc
  I_{\mc C})$, and $A = 1_{\mc C}$. Note that by definition this function maps
  the coherence and hypergraph maps of $\cospan(\mc C)$ onto the corresponding
  maps of $\corel(\mc C)$. As $\cospan(\mc C)$ is a hypergraph, and $\square$
  preserves composition and respects the monoidal and hypergraph structure,
  $\corel(\mc C)$ is also a hypergraph category. 
  
  For instance, suppose we want to check the functoriality of the monoidal
  product $+$. We then wish to show $(g \circ f) + (k \circ h) = (g + k) \circ
  (f + h)$ for corelations of the appropriate types.  But $\square$ preserves
  composition, and the naturality of $\kappa$, here the identity map, implies
  that for any two cospans the $\mc E$-part of their coproduct is equal to the
  coproduct of their $\mc E$-parts. Thus we may compute these two expressions by
  viewing $f$, $g$, $h$, and $k$ as cospans, evaluating them in the category of
  cospans, and then taking their $\mc E$-parts. Since the equality holds in the
  category of cospans, it holds in the category of corelations.
\end{proof}

\begin{corollary}
  There is a strict hypergraph functor 
  \[
    \square\maps \mathrm{Cospan}(\mathcal C) \longrightarrow \mathrm{Corel}(\mathcal C)
  \]
  that takes each object of $\cospan(\mathcal C)$ to itself as an object of
  $\corel(\mathcal C)$ and each cospan to its $\mathcal E$-part.
\end{corollary}

  Finally, we complete the proof that $\square$ is always a hypergraph functor.

\begin{proof}[Proof of Proposition \ref{prop.corelfunctors}] 
  We show $\square$ is a functor, a symmetric monoidal functor, and then finally
  a hypergraph functor.

  \paragraph{Functoriality.} First, recall that $\square$ preserves composition
  (Lemma \ref{lem.corelfuncomposition}). Thus to prove $\square$ is a functor it
  remains to show identities are mapped to identities. The general idea for this
  and for similar axioms is to recall that the special maps are given by reduced
  versions of particular colimits, and that $(\mc E',\mc M')$ reduces maps more
  than $(\mc E,\mc M)$. 

  In this case, recall the identity corelation is given by the $\mc E$-part $X+X
  \to \overline{X}$ of $[1,1]\maps X+X \to X$. Thus the image of the identity on
  $X$ and the identity on $AX$ are given by the top and bottom rows of the
  commuting square
  \[
    \xymatrixcolsep{4pc}
    \xymatrixrowsep{2pc}
    \xymatrix{
      A(X+X) \ar[d]^{\kappa^{-1}}_{\sim} \ar[r]^{\mc E'} &
      \overline{A\overline{X}} \ar[r]^{\mc M'} \ar@{.>}[d]^{n} &
      A\overline{X} \ar[r]^{A\mc M} & AX \ar@{=}[d]\\
      AX+AX \ar[r]^{\mc E'} & \overline{AX} \ar[rr]^{\mc M'} && AX
    }
  \]
  The outside square commutes as we know $A$ maps the identity cospan of $\mc C$
  to the identity cospan of $\mc C'$. The top row is the image under $A$ of the
  identity cospan in $\mc C$, factored first in $\mc C$, and then in $\mc C'$.
  The bottom row is just the factored identity cospan on $AX$ in $\mc C'$. As
  $A$ maps $\mc M$ into $\mc M'$, the map marked $A\mc M$ lies in $\mc M'$. Thus
  both rows are $(\mc E',\mc M')$-factorisations, and so we have the isomorphism
  $n$. Thus $\square$ preserves identities.

  \paragraph{Strong monoidality.} We proved in Lemma \ref{lem.corelfunmonoidal} that
    our proposed coherence maps are natural. The rest of the properties follow
    from the composition preserving map $\cospan(\mc C') \to \corel(\mc C')$.
    Since the $\kappa$ obey all the required axioms as cospans, they obey them
    as corelations too.

  \paragraph{Hypergraph structure.} The proof of preservation of the hypergraph
  structure follows the same pattern as the identity maps. 
\end{proof}

\begin{example}
  In the notation of Proposition \ref{prop.corelfunctors}, note that if both
  $(\mathcal E, \mathcal M)$ and $(\mathcal E', \mathcal M')$ are epi-split mono
  factorisations, then we always have that $A(\mathcal M) \subseteq \mathcal
  M'$. Indeed, if an (one-sided) inverse exists in the domain category, it
  exists in the codomain category. Thus colimit-preserving functors between
  categories with finite colimits and epi-split mono factorisation systems also
  induce a functor between the epi-split mono corelation categories. We will use
  this in Chapter \ref{ch.sigflow}.
\end{example}

\begin{remark} \label{rem.corelposet}
  On any category $\mc C$ with finite colimits, reverse inclusions of the right factor
  $\mc M$ defines a partial order on the set of factorisation systems $(\mc E,\mc
  M)$ with $\mc M$ stable under pushout. That is, we write $(\mc E,\mc M) \ge (\mc
  E,\mc M')$ whenever $\mc M \subseteq \mc M'$.  The trivial factorisation
  systems $(\mc C,\mc I_{\mc C})$ and $(\mc I_{\mc C},\mc C)$ are the top and
  bottom elements of this poset respectively.

  Corelation categories realise this poset as a subcategory of the category of
  hypergraph categories. One way to understand this is that corelations are
  cospans with the $\mc M$-part `forgotten'. Using a morphism-isomorphism
  factorisation system nothing is forgotten, so these corelations are just
  cospans. Using the isomorphism-morphism factorisation system everything is
  forgotten, so there is a unique corelation between any two objects.

  We can construct a hypergraph functor between two corelation categories if the
  codomain forgets more than the domain: i.e. if the codomain is less than
  the domain in the poset. In particular, this implies there is always a
  hypergraph functor from the cospan category $\corel_{(\mc C,\mc I_{\mc
  C})}(\mc C)= \cospan(\mc C)$ to any other corelation category $\corel_{(\mc
  E,\mc M)}(\mc C)$, and from $\corel_{(\mc E,\mc M)}(\mc C)$ any corelation
  category to the indiscrete category $\corel_{(\mc I_{\mc C},\mc C)}(\mc C)$ on
  the objects of $\mc C$.
\end{remark}

\section{Examples} \label{sec.corelexs}
We conclude this chapter with two examples of corelation categories. These both
play a central role in the applications of Part \ref{part.apps}: the first as
an algebra of ideal wires, and the second as semantics for signal flow graphs.

\subsection{Equivalence relations as corelations in $\Set$}
\label{ssec.equivrels}
   
As we saw in Examples \ref{ex.corelcats}, an epi-mono corelation $X \to Y$ in
$\Set$ is an equivalence relation on $X+Y$. We might depict these as follows
\[
  \begin{tikzpicture}[circuit ee IEC]
	\begin{pgfonlayer}{nodelayer}
		\node [contact, outer sep=5pt] (0) at (-2, 1) {};
		\node [contact, outer sep=5pt] (1) at (-2, 0.5) {};
		\node [contact, outer sep=5pt] (2) at (-2, -0) {};
		\node [contact, outer sep=5pt] (3) at (-2, -0.5) {};
		\node [contact, outer sep=5pt] (4) at (-2, -1) {};
		\node [contact, outer sep=5pt] (5) at (1, 1.25) {};
		\node [contact, outer sep=5pt] (6) at (1, 0.75) {};
		\node [contact, outer sep=5pt] (7) at (1, 0.25) {};
		\node [contact, outer sep=5pt] (8) at (1, -0.25) {};
		\node [contact, outer sep=5pt] (9) at (1, -0.75) {};
		\node [contact, outer sep=5pt] (10) at (1, -1.25) {};
		\node [style=none] (11) at (-2.75, -0) {$X$};
		\node [style=none] (12) at (1.75, -0) {$Y$};
	\end{pgfonlayer}
	\begin{pgfonlayer}{edgelayer}
		\draw [rounded corners=5pt, dashed] 
   (node cs:name=0, anchor=north west) --
   (node cs:name=1, anchor=south west) --
   (node cs:name=6, anchor=south east) --
   (node cs:name=5, anchor=north east) --
   cycle;
		\draw [rounded corners=5pt, dashed] 
   (node cs:name=2, anchor=north west) --
   (node cs:name=3, anchor=south west) --
   (node cs:name=3, anchor=south east) --
   (node cs:name=2, anchor=north east) --
   cycle;
		\draw [rounded corners=5pt, dashed] 
   (node cs:name=4, anchor=north west) --
   (node cs:name=4, anchor=south west) --
   (node cs:name=10, anchor=south east) --
   (node cs:name=9, anchor=north east) --
   cycle;
   		\draw [rounded corners=5pt, dashed] 
   (node cs:name=7, anchor=north west) --
   (node cs:name=7, anchor=south west) --
   (node cs:name=7, anchor=south east) --
   (node cs:name=7, anchor=north east) --
   cycle;
   		\draw [rounded corners=5pt, dashed] 
   (node cs:name=8, anchor=north west) --
   (node cs:name=8, anchor=south west) --
   (node cs:name=8, anchor=south east) --
   (node cs:name=8, anchor=north east) --
   cycle;
	\end{pgfonlayer}
\end{tikzpicture}
\]
Here we have a corelation from a set $X$ of five elements to a set $Y$ of six
elements. Elements belonging to the same equivalence class of $X+Y$ are grouped
(`connected') by a dashed line.

Composition of corelations first takes the transitive closure of the two
partitions (the pushout in $\Set$), before restricting the partition to the new
domain and codomain (restricting to the jointly epic part). For example,
suppose in addition to the corelation $\alpha\maps X \to Y$ above we have
another corelation $\beta\maps Y \to Z$
\[
\begin{tikzpicture}[circuit ee IEC]
	\begin{pgfonlayer}{nodelayer}
		\node [style=none] (0) at (-2.75, -0) {$Y$};
		\node [style=none] (1) at (1.75, 0) {$Z$};
		\node [contact, outer sep=5pt] (2) at (-2, 1.25) {};
		\node [contact, outer sep=5pt] (3) at (-2, 0.75) {};
		\node [contact, outer sep=5pt] (4) at (-2, 0.25) {};
		\node [contact, outer sep=5pt] (5) at (-2, -0.25) {};
		\node [contact, outer sep=5pt] (6) at (-2, -0.75) {};
		\node [contact, outer sep=5pt] (7) at (-2, -1.25) {};
		\node [contact, outer sep=5pt] (8) at (1, 1) {};
		\node [contact, outer sep=5pt] (9) at (1, 0.5) {};
		\node [contact, outer sep=5pt] (10) at (1, -0) {};
		\node [contact, outer sep=5pt] (11) at (1, -0.5) {};
		\node [contact, outer sep=5pt] (12) at (1, -1) {};
	\end{pgfonlayer}
		\draw [rounded corners=5pt, dashed] 
   (node cs:name=2, anchor=north west) --
   (node cs:name=3, anchor=south west) --
   (node cs:name=8, anchor=south east) --
   (node cs:name=8, anchor=north east) --
   cycle;
		\draw [rounded corners=5pt, dashed] 
   (node cs:name=4, anchor=north west) --
   (node cs:name=4, anchor=south west) --
   (node cs:name=4, anchor=south east) --
   (node cs:name=4, anchor=north east) --
   cycle;
		\draw [rounded corners=5pt, dashed] 
   (node cs:name=5, anchor=north west) --
   (node cs:name=6, anchor=south west) --
   (node cs:name=11, anchor=south east) --
   (node cs:name=10, anchor=north east) --
   cycle;
		\draw [rounded corners=5pt, dashed] 
   (node cs:name=7, anchor=north west) --
   (node cs:name=7, anchor=south west) --
   (node cs:name=12, anchor=south east) --
   (node cs:name=12, anchor=north east) --
   cycle;
		\draw [rounded corners=5pt, dashed] 
   (node cs:name=9, anchor=north west) --
   (node cs:name=9, anchor=south west) --
   (node cs:name=9, anchor=south east) --
   (node cs:name=9, anchor=north east) --
   cycle;
\end{tikzpicture}
\]
Then the composite $\beta\circ\alpha$ of our two corelations is given by
\vspace{-1ex}
\[
  \begin{aligned}
\begin{tikzpicture}[circuit ee IEC]
	\begin{pgfonlayer}{nodelayer}
		\node [contact, outer sep=5pt] (-2) at (1, 1.25) {};
		\node [contact, outer sep=5pt] (-1) at (1, 0.75) {};
		\node [contact, outer sep=5pt] (0) at (1, 0.25) {};
		\node [contact, outer sep=5pt] (1) at (1, -0.25) {};
		\node [contact, outer sep=5pt] (2) at (1, -0.75) {};
		\node [contact, outer sep=5pt] (3) at (1, -1.25) {};
		\node [style=none] (4) at (-2.75, -0) {$X$};
		\node [style=none] (5) at (4.75, -0) {$Z$};
		\node [contact, outer sep=5pt] (6) at (-2, 1) {};
		\node [contact, outer sep=5pt] (7) at (-2, -0.5) {};
		\node [contact, outer sep=5pt] (8) at (-2, 0.5) {};
		\node [contact, outer sep=5pt] (9) at (-2, -0) {};
		\node [contact, outer sep=5pt] (10) at (-2, -1) {};
		\node [contact, outer sep=5pt] (11) at (4, -0) {};
		\node [contact, outer sep=5pt] (12) at (4, -1) {};
		\node [contact, outer sep=5pt] (13) at (4, -0.5) {};
		\node [contact, outer sep=5pt] (14) at (4, 0.5) {};
		\node [contact, outer sep=5pt] (19) at (4, 1) {};
		\node [style=none] (20) at (1, -1.75) {$Y$};
		\node [style=none] (21) at (1, 1.75) {\phantom{$Y$}};
	\end{pgfonlayer}
	\begin{pgfonlayer}{edgelayer}
		\draw [rounded corners=5pt, dashed] 
   (node cs:name=6, anchor=north west) --
   (node cs:name=8, anchor=south west) --
   (node cs:name=-1, anchor=south east) --
   (node cs:name=-2, anchor=north east) --
   cycle;
		\draw [rounded corners=5pt, dashed] 
   (node cs:name=9, anchor=north west) --
   (node cs:name=7, anchor=south west) --
   (node cs:name=7, anchor=south east) --
   (node cs:name=9, anchor=north east) --
   cycle;
		\draw [rounded corners=5pt, dashed] 
   (node cs:name=10, anchor=north west) --
   (node cs:name=10, anchor=south west) --
   (node cs:name=3, anchor=south east) --
   (node cs:name=2, anchor=north east) --
   cycle;
		\draw [rounded corners=5pt, dashed] 
   (node cs:name=-2, anchor=north west) --
   (node cs:name=-1, anchor=south west) --
   (node cs:name=19, anchor=south east) --
   (node cs:name=19, anchor=north east) --
   cycle;
		\draw [rounded corners=5pt, dashed] 
   (node cs:name=0, anchor=north west) --
   (node cs:name=0, anchor=south west) --
   (node cs:name=0, anchor=south east) --
   (node cs:name=0, anchor=north east) --
   cycle;
		\draw [rounded corners=5pt, dashed] 
   (node cs:name=1, anchor=north west) --
   (node cs:name=1, anchor=south west) --
   (node cs:name=1, anchor=south east) --
   (node cs:name=1, anchor=north east) --
   cycle;
		\draw [rounded corners=5pt, dashed] 
   (node cs:name=1, anchor=north west) --
   (node cs:name=2, anchor=south west) --
   (node cs:name=13, anchor=south east) --
   (node cs:name=11, anchor=north east) --
   cycle;
		\draw [rounded corners=5pt, dashed] 
   (node cs:name=3, anchor=north west) --
   (node cs:name=3, anchor=south west) --
   (node cs:name=12, anchor=south east) --
   (node cs:name=12, anchor=north east) --
   cycle;
		\draw [rounded corners=5pt, dashed] 
   (node cs:name=14, anchor=north west) --
   (node cs:name=14, anchor=south west) --
   (node cs:name=14, anchor=south east) --
   (node cs:name=14, anchor=north east) --
   cycle;
	\end{pgfonlayer}
\end{tikzpicture}
\end{aligned}
\:
  =
\:
\begin{aligned}
\begin{tikzpicture}[circuit ee IEC]
	\begin{pgfonlayer}{nodelayer}
		\node [style=none] (0) at (-2.75, -0) {$X$};
		\node [style=none] (1) at (1.75, -0) {$Z$};
		\node [contact, outer sep=5pt] (2) at (-2, 1) {};
		\node [contact, outer sep=5pt] (3) at (-2, -0.5) {};
		\node [contact, outer sep=5pt] (4) at (-2, 0.5) {};
		\node [contact, outer sep=5pt] (5) at (-2, -0) {};
		\node [contact, outer sep=5pt] (6) at (-2, -1) {};
		\node [contact, outer sep=5pt] (7) at (1, -0) {};
		\node [contact, outer sep=5pt] (8) at (1, -1) {};
		\node [contact, outer sep=5pt] (9) at (1, -0.5) {};
		\node [contact, outer sep=5pt] (10) at (1, 0.5) {};
		\node [contact, outer sep=5pt] (13) at (1, 1) {};
		\node [style=none] (20) at (1, -1.75) {\phantom{$Y$}};
		\node [style=none] (21) at (1, 1.75) {\phantom{$Y$}};
	\end{pgfonlayer}
	\begin{pgfonlayer}{edgelayer}
		\draw [rounded corners=5pt, dashed] 
   (node cs:name=2, anchor=north west) --
   (node cs:name=4, anchor=south west) --
   (node cs:name=13, anchor=south east) --
   (node cs:name=13, anchor=north east) --
   cycle;
		\draw [rounded corners=5pt, dashed] 
   (node cs:name=5, anchor=north west) --
   (node cs:name=3, anchor=south west) --
   (node cs:name=3, anchor=south east) --
   (node cs:name=5, anchor=north east) --
   cycle;
		\draw [rounded corners=5pt, dashed] 
   (node cs:name=6, anchor=north west) --
   (node cs:name=6, anchor=south west) --
   (node cs:name=8, anchor=south east) --
   (node cs:name=7, anchor=north east) --
   cycle;
		\draw [rounded corners=5pt, dashed] 
   (node cs:name=10, anchor=north west) --
   (node cs:name=10, anchor=south west) --
   (node cs:name=10, anchor=south east) --
   (node cs:name=10, anchor=north east) --
   cycle;
	\end{pgfonlayer}
\end{tikzpicture}
\end{aligned}
\]
Informally, this captures the idea that two elements of $X+Z$ are `connected' if
we may travel from one to the other staying within connected components of
$\alpha$ and $\beta$.
  
For a greater resemblance of the diagrams in the motivating comments of
\textsection\ref{sec.blackboxing}, epi-mono corelations in $\Set$ can also be
visualised as terminals connected by junctions of ideal wires. We draw these by
marking each equivalence class with a point (the `junction'), and then
connecting each element of the domain and codomain to their equivalence class
with a `wire'. The junction serves as the apex of the corelation, the terminals
as the feet, and the wires depict a function from the feet to the apex.
Composition then involves collapsing connected junctions down to a point.
\vspace{-1ex}
\[
  \begin{aligned}
\begin{tikzpicture}[circuit ee IEC]
	\begin{pgfonlayer}{nodelayer}
		\node [contact, outer sep=5pt] (6) at (-2, 1) {};
		\node [contact, outer sep=5pt] (7) at (-2, -0.5) {};
		\node [contact, outer sep=5pt] (8) at (-2, 0.5) {};
		\node [contact, outer sep=5pt] (9) at (-2, -0) {};
		\node [contact, outer sep=5pt] (10) at (-2, -1) {};
		\node [style=none] (15) at (-0.5, 0.875) {};
		\node [style=none] (28) at (-0.5, 0.25) {};
		\node [style=none] (16) at (-0.5, -0.125) {};
		\node [style=none] (29) at (-0.5, -0.375) {};
		\node [style=none] (17) at (-0.5, -1) {};
		\node [contact, outer sep=5pt] (-2) at (1, 1.25) {};
		\node [contact, outer sep=5pt] (-1) at (1, 0.75) {};
		\node [contact, outer sep=5pt] (0) at (1, 0.25) {};
		\node [contact, outer sep=5pt] (1) at (1, -0.25) {};
		\node [contact, outer sep=5pt] (2) at (1, -0.75) {};
		\node [contact, outer sep=5pt] (3) at (1, -1.25) {};
		\node [style=none] (18) at (2.5, -1.125) {};
		\node [style=none] (21) at (2.5, 1) {};
		\node [style=none] (22) at (2.5, -0.375) {};
		\node [style=none] (23) at (2.5, 0.475) {};
		\node [style=none] (24) at (2.5, 0.25) {};
		\node [contact, outer sep=5pt] (19) at (4, 1) {};
		\node [contact, outer sep=5pt] (14) at (4, 0.5) {};
		\node [contact, outer sep=5pt] (11) at (4, -0) {};
		\node [contact, outer sep=5pt] (13) at (4, -0.5) {};
		\node [contact, outer sep=5pt] (12) at (4, -1) {};
		\node [style=none] (4) at (-2.75, -0) {$X$};
		\node [style=none] (5) at (4.75, -0) {$Z$};
		\node [style=none] (20) at (1, -1.75) {$Y$};
		\node [style=none] (30) at (1, 1.75) {\phantom{$Y$}};
	\end{pgfonlayer}
	\begin{pgfonlayer}{edgelayer}
		\draw [thick] (6.center) to (15.center);
		\draw [thick] (8.center) to (15.center);
		\draw [thick] (-2.center) to (15.center);
		\draw [thick] (-1.center) to (15.center);
		\draw [thick] (9.center) to (16.center);
		\draw [thick] (7.center) to (16.center);
		\draw [thick] (10.center) to (17.center);
		\draw [thick] (17.center) to (2.center);
		\draw [thick] (17.center) to (3.center);
		\draw [thick] (3.center) to (18.center);
		\draw [thick] (18.center) to (12.center);
		\draw [thick] (-2.center) to (21.center);
		\draw [thick] (-1.center) to (21.center);
		\draw [thick] (21.center) to (19.center);
		\draw [thick] (1.center) to (22.center);
		\draw [thick] (2.center) to (22.center);
		\draw [thick] (22.center) to (11.center);
		\draw [thick] (22.center) to (13.center);
		\draw [thick] (23.center) to (14.center);
		\draw [thick] (28.center) to (0.center);
		\draw [thick] (0.center) to (24.center);
		\draw [thick] (29.center) to (1.center);
		\draw [rounded corners=5pt, dashed, color=gray] 
   (node cs:name=6, anchor=north west) --
   (node cs:name=8, anchor=south west) --
   (node cs:name=-1, anchor=south east) --
   (node cs:name=-2, anchor=north east) --
   cycle;
		\draw [rounded corners=5pt, dashed, color=gray] 
   (node cs:name=9, anchor=north west) --
   (node cs:name=7, anchor=south west) --
   (node cs:name=7, anchor=south east) --
   (node cs:name=9, anchor=north east) --
   cycle;
		\draw [rounded corners=5pt, dashed, color=gray] 
   (node cs:name=10, anchor=north west) --
   (node cs:name=10, anchor=south west) --
   (node cs:name=3, anchor=south east) --
   (node cs:name=2, anchor=north east) --
   cycle;
		\draw [rounded corners=5pt, dashed, color=gray] 
   (node cs:name=-2, anchor=north west) --
   (node cs:name=-1, anchor=south west) --
   (node cs:name=19, anchor=south east) --
   (node cs:name=19, anchor=north east) --
   cycle;
		\draw [rounded corners=5pt, dashed, color=gray] 
   (node cs:name=0, anchor=north west) --
   (node cs:name=0, anchor=south west) --
   (node cs:name=0, anchor=south east) --
   (node cs:name=0, anchor=north east) --
   cycle;
		\draw [rounded corners=5pt, dashed, color=gray] 
   (node cs:name=1, anchor=north west) --
   (node cs:name=1, anchor=south west) --
   (node cs:name=1, anchor=south east) --
   (node cs:name=1, anchor=north east) --
   cycle;
		\draw [rounded corners=5pt, dashed, color=gray] 
   (node cs:name=1, anchor=north west) --
   (node cs:name=2, anchor=south west) --
   (node cs:name=13, anchor=south east) --
   (node cs:name=11, anchor=north east) --
   cycle;
		\draw [rounded corners=5pt, dashed, color=gray] 
   (node cs:name=3, anchor=north west) --
   (node cs:name=3, anchor=south west) --
   (node cs:name=12, anchor=south east) --
   (node cs:name=12, anchor=north east) --
   cycle;
		\draw [rounded corners=5pt, dashed, color=gray] 
   (node cs:name=14, anchor=north west) --
   (node cs:name=14, anchor=south west) --
   (node cs:name=14, anchor=south east) --
   (node cs:name=14, anchor=north east) --
   cycle;
	\end{pgfonlayer}
\end{tikzpicture}
\end{aligned}
\:
  =
\:
\begin{aligned}
\begin{tikzpicture}[circuit ee IEC]
	\begin{pgfonlayer}{nodelayer}
		\node [style=none] (0) at (-2.75, -0) {$X$};
		\node [style=none] (1) at (1.75, -0) {$Z$};
		\node [contact, outer sep=5pt] (2) at (-2, 1) {};
		\node [contact, outer sep=5pt] (3) at (-2, -0.5) {};
		\node [contact, outer sep=5pt] (4) at (-2, 0.5) {};
		\node [contact, outer sep=5pt] (5) at (-2, -0) {};
		\node [contact, outer sep=5pt] (6) at (-2, -1) {};
		\node [contact, outer sep=5pt] (7) at (1, -0) {};
		\node [contact, outer sep=5pt] (8) at (1, -1) {};
		\node [contact, outer sep=5pt] (9) at (1, -0.5) {};
		\node [contact, outer sep=5pt] (10) at (1, 0.5) {};
		\node [style=none] (11) at (-0.5, 0.875) {};
		\node [style=none] (12) at (-0.5, 0.3) {};
		\node [contact, outer sep=5pt] (13) at (1, 1) {};
		\node [style=none] (14) at (-0.5, -0.2) {};
		\node [style=none] (15) at (-0.5, -0.6) {};
	\end{pgfonlayer}
	\begin{pgfonlayer}{edgelayer}
		\draw [thick] (2.center) to (11.center);
		\draw [thick] (4.center) to (11.center);
		\draw [thick] (11.center) to (13.center);
		\draw [thick] (5.center) to (14.center);
		\draw [thick] (3.center) to (14.center);
		\draw [thick] (15.center) to (7.center);
		\draw [thick] (15.center) to (9.center);
		\draw [thick] (6.center) to (15.center);
		\draw [thick] (15.center) to (8.center);
		\draw [thick] (12.center) to (10.center);
		\draw [rounded corners=5pt, dashed, color=gray] 
   (node cs:name=2, anchor=north west) --
   (node cs:name=4, anchor=south west) --
   (node cs:name=13, anchor=south east) --
   (node cs:name=13, anchor=north east) --
   cycle;
		\draw [rounded corners=5pt, dashed, color=gray] 
   (node cs:name=5, anchor=north west) --
   (node cs:name=3, anchor=south west) --
   (node cs:name=3, anchor=south east) --
   (node cs:name=5, anchor=north east) --
   cycle;
		\draw [rounded corners=5pt, dashed, color=gray] 
   (node cs:name=10, anchor=north west) --
   (node cs:name=10, anchor=south west) --
   (node cs:name=10, anchor=south east) --
   (node cs:name=10, anchor=north east) --
   cycle;
		\draw [rounded corners=5pt, dashed, color=gray] 
   (node cs:name=6, anchor=north west) --
   (node cs:name=6, anchor=south west) --
   (node cs:name=8, anchor=south east) --
   (node cs:name=7, anchor=north east) --
   cycle;
	\end{pgfonlayer}
\end{tikzpicture}
\end{aligned}
\]
The composition law captures the idea that connectivity is all that
matters: as long as the wires are `ideal', the exact path does not matter. 

In Coya--Fong \cite{CF} we formalise this idea by saying that corelations are
the prop for extraspecial commutative Frobenius monoids. An \define{extraspecial
commutative Frobenius monoid} $(X,\mu,\eta,\delta,\epsilon)$ in a monoidal
category $(\mathcal C, \otimes)$ is a special commutative Frobenius monoid that
further obeys the extra law
  \[
    \extral{.1\textwidth} = \extrar{.1\textwidth}
  \]

Two morphisms built from the generators of an extraspecial commutative Frobenius
monoid are equal and if and only if their diagrams impose the same connectivity
relations on the disjoint union of the domain and codomain. This is an extension
of the spider theorem for special commutative Frobenius monoids.

\subsection{Linear relations as corelations in $\Vect$} \label{ssec.linrel}

Recall that a linear relation $L\maps U \leadsto V$ is a subspace $L \subseteq U
\oplus V$. We compose linear relations as we do relations, and vector spaces and
linear relations form a category $\LinRel$. It is straightforward to show that
this category can be constructed as the category of relations in the category
$\Vect$ of vector spaces and linear maps with respect to epi-mono
factorisations: monos in $\Vect$ are simply injective linear maps, and hence
subspace inclusions. We show that they may also be constructed as corelations in
$\Vect$ with respect to epi-mono factorisations.

If we restrict to the full subcategory $\FinVect$ of finite dimensional vector
spaces duality makes this easy to see: after picking a basis for each vector
space the transpose yields an equivalence of $\FinVect$ with its opposite
category, so the category of $(\mathcal E,\mathcal M)$-corelations (jointly epic
cospans) is isomorphic to the category of $(\mathcal E,\mathcal M)$-relations
(jointly monic spans) in $\FinVect$. This fact has been fundamental in work on
finite dimensional linear systems and signal flow diagrams \cite{BE,BSZ,FRS16}.

We prove the general case in detail. To begin, note $\mathrm{Vect}$ has an
epi-mono factorisation system with monos stable under pushouts. This
factorisation system is inherited from $\Set$: the epimorphisms in $\Vect$ are
precisely the surjective linear maps, the monomorphisms are the injective
linear maps, and the image of a linear map is always a subspace of the
codomain, and so itself a vector space. Monos are stable under pushout as the
pushout of a diagram $V \xleftarrow{f} U \xrightarrow{m} W$ is $V \oplus
W/\im[f\; -m]$. The map $m'\maps V \to V \oplus W/\im[f\; -m]$ into the pushout
has kernel $f(\ker m)$. Thus when $m$ is a monomorphism, $m'$ is too.

Thus we have a category of corelations $\corel(\Vect)$. We show that the map
$\corel(\Vect) \to \LinRel$ sending each vector space to itself and each
corelation
\[
  U \stackrel{f}\longrightarrow A \stackrel{g}\longleftarrow V
\]
to the linear subspace $\ker[f\;-g]$ is a full, faithful, and
bijective-on-objects functor.

Indeed, corelations $U \xrightarrow{f} A \xleftarrow{g} V$ are in one-to-one
correspondence with surjective linear maps $U\oplus V \to A$, which are in
turn, by the isomorphism theorem, in one-to-one correspondence with subspaces
of $U\oplus V$. These correspondences are described by the kernel construction
above. Thus our map is evidently full, faithful, and bijective-on-objects. It
also maps identities to identities. It remains to check that it preserves
composition.

Suppose we have corelations $U \xrightarrow{f} A \xleftarrow{g} V$
and $V \xrightarrow{h} B \xleftarrow{k} W$. Then their pushout is given by
$P=A \oplus B/\im[g\;-h]$, and we may draw the pushout diagram
\[
  \xymatrix{
    U \ar[dr]_{f} & & V \ar[dl]^{g}  
    \ar[dr]_{h} & & W \ar[dl]^{k} {} 
    \\
    & A \ar[dr]_{\iota_A} & & B \ar[dl]^{\iota_B}  \\
    & & P {\save*!<0cm,-.5cm>[dl]@^{|-}\restore}
  }
\]
We wish to show the equality of relations
\[
  \ker[f\;-g];\ker[h\;-k] = \ker[\iota_A f\; -\iota_B g].
\]
Now $(\mathbf{u},\mathbf{w}) \in U \oplus W$ lies in the composite relation
$\ker[f\;-g];\ker[h\;-k]$ iff there exists $\mathbf{v} \in V$ such that
$f\mathbf{u} = g\mathbf{v}$ and $h\mathbf{v} = k\mathbf{w}$. But as $P$ is the
pushout, this is true iff 
\[
  \iota_A f \mathbf{u} = \iota_A g \mathbf{v} = \iota_B h \mathbf{v} =
  \iota_B k \mathbf{w}.
\]
This in turn is true iff $(\mathbf{u}, \mathbf{w}) \in \ker[\iota_Af\;
-\iota_Bk]$, as required. 

This corelations perspective is important as it fits the relational picture into
our philosophy of black boxing. In Chapter \ref{ch.sigflow} we will see it is
the corelation construction of $\LinRel$ that correctly generalises to the case
of non-controllable systems.
\smallskip

In the next chapter we discuss decorations on corelations.

\chapter{Decorated corelations: black-boxed open systems} \label{ch.deccorels}

When enough structure is available to us, we may decorate corelations too.
Furthermore, and key to the idea of `black-boxing', we get a hypergraph functor
from decorated cospans to decorated corelations. This assists with constructing
hypergraph categories of semantics for network-style diagrammatic languages.
Indeed, we will prove that every hypergraph category can be constructed using
decorated corelations, and so can every hypergraph functor.

In the next section we motivate this chapter by discussing how to construct the
category of matrices using decorations, illustrating the shortcomings of
decorated cospans and how they are overcome by decorated corelations. As usual,
we then use the subsequent two sections to get into the technical details,
defining decorated corelations (\textsection\ref{sec:dcorc}) and their functors
(\textsection\ref{sec.dcorf}). This sets us up to state and prove the climactic
theorem of Part \ref{part.maths}: roughly, decorated corelations are powerful
enough to construct all hypergraph categories and functors. We do this in
\textsection\ref{sec.allhypergraphs}. Finally, in \textsection\ref{sec:excor},
we give some examples.

\section{Black-boxed open systems} \label{sec.blackboxedsystems}

Suppose we have devices built from paths that take the signal at some input,
amplify it, and deliver it to some output. For simplicity let these signals be
real numbers, and amplication be linear: we just multiply by some fixed scalar.
We depict an example device like so:
\[
    \tikzset{every path/.style={line width=.8pt}}
\begin{tikzpicture}
	\begin{pgfonlayer}{nodelayer}
		\node [style=sdot] (0) at (-2.5, 1.5) {};
		\node [style=sdot] (1) at (-2.5, -0) {};
		\node [style=sdot] (2) at (-2.5, -1.5) {};
		\node [style=amp] (3) at (-0.75, 2) {$5$};
		\node [style=amp] (4) at (-0.75, 1.25) {$1$};
		\node [style=amp] (5) at (-0.75, -1) {$-2$};
		\node [style=amp] (6) at (-0.75, 0.25) {$2.1$};
		\node [style=amp] (7) at (-0.75, -0.25) {$-0.4$};
		\node [style=sdot] (8) at (1, -0.5) {};
		\node [style=sdot] (9) at (1, -1.5) {};
		\node [style=sdot] (10) at (1, 0.5) {};
		\node [style=sdot] (11) at (1, 1.5) {};
	\end{pgfonlayer}
	\begin{pgfonlayer}{edgelayer}
		\draw (3.west) to (0.center);
		\draw (4.west) to (0.center);
		\draw (6.west) to (1.center);
		\draw (7.west) to (1.center);
		\draw (5.west) to (2.center);
		\draw (3) to (11);
		\draw (4) to (11);
		\draw (6) to (10);
		\draw (7) to (8);
		\draw (5) to (8);
	\end{pgfonlayer}
\end{tikzpicture}
\]
Here there are three inputs, four outputs, and five paths. Formally, we might
model these devices as finite sets of inputs $X$, outputs $Y$, and paths $N$,
together with functions $i\maps N \to X$ and $o\maps N \to Y$ describing the
start and end of each path, and a function $s\maps N \to \R$ describing the
amplification along it. In other words, these are cospans $X \to N \leftarrow Y$
in $\FinSet^\opp$ decorated by `scalar assignment' functions $N \to \R$. This
suggests a decorated cospans construction.

Such a construction can be obtained from the lax symmetric monoidal functor
$\mathbb R^{(-)}\maps \FinSet^\opp \to \Set$ taking a finite set $N$ to the set
$\R^N$ of functions $s\maps N \to \R$, and opposite functions $f^\opp\maps N \to
M$ to the map taking $s\maps N \to \R$ to $s \circ f \maps M \to \R$. Note that
the coproduct in $\FinSet^\opp$ is the cartesian product of sets. The coherence
maps of the functor which, recall, are critical for the composition of the
decorations, are given by $\varphi_{N,M}\maps \R^N \times \R^M \to \R^{N\times
M}$, taking $(s,t) \in \R^N \times \R^M$ to the function $s \cdot t\maps N
\times M \to \R$ defined by pointwise multiplication in $\R$. 

Composition in this decorated cospan category is thus given by the
multiplication in $\R$. In detail, given decorated cospans $(X
\xrightarrow{i_X^\opp} N \xleftarrow{o_Y^\opp} Y, N \xrightarrow{s} \R)$ and $(Y
\xrightarrow{i_Y^\opp} M \xleftarrow{o_Z^\opp} Z, M \xrightarrow{t} \R)$, the
composite has a path from $x \in X$ to $z \in Z$ for every triple $(y,n,m)$
where $y \in Y$, $n \in N$, and $m \in M$, such that $n$ is a path from $x$ to
$y$ and $m$ is a path from $y$ to $z$. The scalar assigned to this path is the
product of those assigned to $n$ and $m$. For example, we have the following
composite
\[
    \tikzset{every path/.style={line width=.8pt}}
    \begin{aligned}
\begin{tikzpicture}
	\begin{pgfonlayer}{nodelayer}
		\node [style=sdot] (0) at (-2.5, 1.5) {};
		\node [style=sdot] (1) at (-2.5, -0) {};
		\node [style=sdot] (2) at (-2.5, -1.5) {};
		\node [style=amp] (3) at (-0.75, 2) {$5$};
		\node [style=amp] (4) at (-0.75, 1.25) {$1$};
		\node [style=amp] (5) at (-0.75, -1) {$-2$};
		\node [style=amp] (6) at (-0.75, 0.25) {$2.1$};
		\node [style=amp] (7) at (-0.75, -0.25) {$-0.4$};
		\node [style=sdot] (8) at (1, -0.5) {};
		\node [style=sdot] (9) at (4.5, 0.5) {};
		\node [style=sdot] (10) at (1, -1.5) {};
		\node [style=sdot] (11) at (4.5, -1.5) {};
		\node [style=amp] (12) at (2.75, -0.5) {$-1$};
		\node [style=amp] (13) at (2.75, 1.25) {$3$};
		\node [style=amp] (14) at (2.75, -1.5) {$-2.3$};
		\node [style=amp] (15) at (2.75, 2) {$1$};
		\node [style=sdot] (16) at (4.5, 1.5) {};
		\node [style=amp] (17) at (2.75, -0) {$3$};
		\node [style=sdot] (18) at (4.5, -0.5) {};
		\node [style=sdot] (19) at (1, 0.5) {};
		\node [style=sdot] (20) at (1, 1.5) {};
	\end{pgfonlayer}
	\begin{pgfonlayer}{edgelayer}
		\draw (3.west) to (0.center);
		\draw (4.west) to (0.center);
		\draw (6.west) to (1.center);
		\draw (7.west) to (1.center);
		\draw (5.west) to (2.center);
		\draw (16) to (15);
		\draw (16) to (13);
		\draw (18) to (12);
		\draw (14) to (11);
		\draw (10) to (14);
		\draw (15) to (20);
		\draw (13) to (20);
		\draw (3) to (20);
		\draw (4) to (20);
		\draw (6) to (19);
		\draw (7) to (8);
		\draw (5) to (8);
		\draw (12) to (8);
		\draw (19) to (17);
		\draw (17) to (18);
	\end{pgfonlayer}
\end{tikzpicture}
\end{aligned}
\qquad = \qquad
\begin{aligned}
\begin{tikzpicture}
	\begin{pgfonlayer}{nodelayer}
		\node [style=sdot] (0) at (-2.5, 1.5) {};
		\node [style=sdot] (1) at (-2.5, -0) {};
		\node [style=sdot] (2) at (-2.5, -1.5) {};
		\node [style=amp] (3) at (-0.75, 1.75) {$15$};
		\node [style=amp] (4) at (-0.75, 1.25) {$1$};
		\node [style=amp] (5) at (-0.75, -1) {$2$};
		\node [style=amp] (6) at (-0.75, -0) {$6.3$};
		\node [style=amp] (7) at (-0.75, -0.5) {$0.4$};
		\node [style=sdot] (8) at (1, -0.5) {};
		\node [style=sdot] (9) at (1, -1.5) {};
		\node [style=sdot] (10) at (1, 0.5) {};
		\node [style=sdot] (11) at (1, 1.5) {};
		\node [style=amp] (12) at (-0.75, 2.25) {$5$};
		\node [style=amp] (13) at (-0.75, 0.75) {$3$};
	\end{pgfonlayer}
	\begin{pgfonlayer}{edgelayer}
		\draw (3.west) to (0.center);
		\draw (4.west) to (0.center);
		\draw (6.west) to (1.center);
		\draw (7.west) to (1.center);
		\draw (5.west) to (2.center);
		\draw (3) to (11);
		\draw (4) to (11);
		\draw (7) to (8);
		\draw (5) to (8);
		\draw (0) to (12);
		\draw (12) to (11);
		\draw (0) to (13);
		\draw (13) to (11);
		\draw (6) to (8);
	\end{pgfonlayer}
\end{tikzpicture}
\end{aligned}
\]
There are four paths between the top-most element $x_1$ of the domain and the
top-most element $z_1$ of the codomain: we may first take the path that
amplifies by $5\times$ and then the path that amplifies by $1\times$ for a total
amplification of $5\times$, or $5\times$ and $3\times$ for $15\times$, and so
on. This means we end up with four elements relating $x_1$ and $z_1$ in the
composite. The apex of the composite is in fact given by the pullback $N
\times_Y M$ of the cospan $N \xrightarrow{o_Y} Y \xleftarrow{i_Y} M$ in
$\FinSet$.

Here we again see the problem of decorated cospans: the composite of the above
puts decorations on $N\times_Y M$, which can be of much larger cardinality than
$N$ and $M$. We wish to avoid the size of our decorated cospan from growing so
fast. Moreover, from our open systems perspective, we care not about the path
but by the total amplification of the signal from some chosen input to some
chosen output. The intuition is that if we black box the system, then we cannot
tell what paths the signal took through the system, only the total amplification
from input to output.

We thus want to restrict our apex to contain at most one point for each
input--output pair $(x,y)$. We do this by pushing the decoration along the
surjection $e$ in the epi-mono factorisation of the function $N\times_YM
\xrightarrow{e} \overline{N\times_YM} \xrightarrow{m} X \times Z$. Put another
way, we want the category of decorated corelations, not decorated cospans.

Represented as decorated corelations, the above composite becomes
\[
  \tikzset{every path/.style={line width=.8pt}}
  \begin{aligned}
    \begin{tikzpicture}
      \begin{pgfonlayer}{nodelayer}
	\node [style=sdot] (0) at (-2.5, 1.5) {};
	\node [style=sdot] (1) at (-2.5, -0) {};
	\node [style=sdot] (2) at (-2.5, -1.5) {};
	\node [style=amp] (3) at (-0.75, 1.5) {$6$};
	\node [style=amp] (4) at (-0.75, -1) {$-2$};
	\node [style=amp] (5) at (-0.75, 0.25) {$2.1$};
	\node [style=amp] (6) at (-0.75, -0.25) {$-0.4$};
	\node [style=sdot] (7) at (1, -0.5) {};
	\node [style=sdot] (8) at (4.5, 0.5) {};
	\node [style=sdot] (9) at (1, -1.5) {};
	\node [style=sdot] (10) at (4.5, -1.5) {};
	\node [style=amp] (11) at (2.75, -0.5) {$-1$};
	\node [style=amp] (12) at (2.75, -1.5) {$-2.3$};
	\node [style=amp] (13) at (2.75, 1.5) {$4$};
	\node [style=sdot] (14) at (4.5, 1.5) {};
	\node [style=amp] (15) at (2.75, -0) {$3$};
	\node [style=sdot] (16) at (4.5, -0.5) {};
	\node [style=sdot] (17) at (1, 0.5) {};
	\node [style=sdot] (18) at (1, 1.5) {};
      \end{pgfonlayer}
      \begin{pgfonlayer}{edgelayer}
	\draw (3.west) to (0.center);
	\draw (5.west) to (1.center);
	\draw (6.west) to (1.center);
	\draw (4.west) to (2.center);
	\draw (14) to (13);
	\draw (16) to (11);
	\draw (12) to (10);
	\draw (9) to (12);
	\draw (13) to (18);
	\draw (3) to (18);
	\draw (5) to (17);
	\draw (6) to (7);
	\draw (4) to (7);
	\draw (11) to (7);
	\draw (17) to (15);
	\draw (15) to (16);
      \end{pgfonlayer}
    \end{tikzpicture}
  \end{aligned}
  \qquad = \qquad
  \begin{aligned}
    \begin{tikzpicture}
      \begin{pgfonlayer}{nodelayer}
	\node [style=sdot] (0) at (-2.5, 1.5) {};
	\node [style=sdot] (1) at (-2.5, -0) {};
	\node [style=sdot] (2) at (-2.5, -1.5) {};
	\node [style=amp] (3) at (-0.75, 1.5) {$24$};
	\node [style=amp] (4) at (-0.75, -1) {$2$};
	\node [style=amp] (5) at (-0.75, -0.25) {$6.7$};
	\node [style=sdot] (6) at (1, -0.5) {};
	\node [style=sdot] (7) at (1, -1.5) {};
	\node [style=sdot] (8) at (1, 0.5) {};
	\node [style=sdot] (9) at (1, 1.5) {};
      \end{pgfonlayer}
      \begin{pgfonlayer}{edgelayer}
	\draw (3.west) to (0.center);
	\draw (5.west) to (1.center);
	\draw (4.west) to (2.center);
	\draw (3) to (9);
	\draw (5) to (6);
	\draw (4) to (6);
      \end{pgfonlayer}
    \end{tikzpicture}
  \end{aligned}
\]
Note that composite is not simply the composite as decorated cospans, but the
composite decorated cospan reduced to a decorated corelation. In
\textsection\ref{ssec.matrices}, we will show that this decorated corelations
category is equivalent to the category of real vector spaces and linear maps,
with monoidal product the tensor product.

It is not a trivial fact that the above composition rule for decorated
corelations defines a category. Indeed, the reason that it is possible to push
the decoration along the surjection $e$ is that the lax symmetric monoidal
functor $\R^{(-)}\maps \FinSet^\opp \to \Set$ extends to a lax symmetric
monoidal functor $(\FinSet^\opp;\mathrm{Sur}^\opp) \to \Set$. Here
$\FinSet^\opp;\mathrm{Sur}^\opp$ is the subcategory of $\cospan(\FinSet^\opp)$
comprising cospans of the form $\xrightarrow{f^\opp}\xleftarrow{e^\opp}$, where
$f$ is any function but $e$ must be a surjection.

More generally, given a category $\mc C$ with finite colimits and a subcategory
$\mc M$ stable under pushouts, we may construct a symmetric monoidal category
$\mc C;\mc M^\opp$ with isomorphisms classes of cospans of the form
$\xrightarrow{f}\xleftarrow{m}$, where $f \in \mc C$, $m \in \mc M$, as
morphisms. The monoidal product is again derived from the coproduct in $\mc C$.

We prove two key theorems in this section. The first, Theorem \ref{thm.fcorel},
is that these decorated corelations form a hypergraph category. That is, given a
category $\mc C$ with finite colimits, factorisation system $(\mc E,\mc M)$ such
that $\mc M$ is stable under pushouts, and a lax symmetric monoidal functor 
\[
  F\maps \mc C; \mc M^\opp \longrightarrow \Set,
\]
define a decorated corelation to be, as might be expected, an $(\mc E,\mc
M)$-corelation $X \to N \leftarrow Y$ in $\mc C$ together with an element of
$FN$. Then there is a hypergraph category $F\mathrm{Corel}$ with the objects of
$\mc C$ as objects and isomorphism classes of decorated corelations as
morphisms. As usual, functors between these so-called decorated corelations
categories can be defined from natural transformations between the decorating
functors.

The second key theorem is that every hypergraph category and hypergraph
functor can be constructed in this way, yielding an equivalence of categories
(Theorem \ref{thm.equivhypdccor}). This shows our decorated corelations
construction is as general as we need for the study of hypergraph categories,
and hence network-style diagrammatic languages.

\section{Decorated corelations} \label{sec:dcorc}

Decorating cospans requires more than just choosing a set of decorations for
each apex: for composition, we need to describe how these decorations transfer
along the copairing of pushout maps $[j_N,j_M]\maps N+M \to N+_YM$. Thus to
construct a decorated cospan category we need, as we have seen, not merely a
function from the objects of $\mc C$ to $\Set$, but a lax symmetric monoidal
functor $(\mc C,+) \to (\Set,\times)$. 

Similarly, decorating $(\mc E,\mc M)$-corelations requires still more
information: we now further need to know how to transfer decorations backwards along the
morphisms $N+_YM \xleftarrow{m} \overline{N+_YM}$. We thus introduce the
symmetric monoidal category $\mc C;\mc M^\opp$ with morphisms isomorphism
classes of cospans of the form $\xrightarrow{f}\xleftarrow{m}$, where $f \in \mc
C$ and $m \in \mc M$.  For constructing categories of decorated $(\mc E,\mc
M)$-corelations, we then require a lax symmetric monoidal functor $F$ from $\mc
C;\mc M^\opp$ to $\Set$. 

This section is structured in the same way as \textsection\ref{sec.corels}.
First, in \textsection\ref{ssec.rightads}, we introduce the categories $\mc
C;\mc M^\opp$, a prerequisite for the definition of decorated corelations. Next,
in \textsection\ref{ssec.deccorel}, we define decorated corelations, and then,
in \textsection\ref{ssec.deccorelcats} the hypergraph category of decorated
corelations. As in the previous chapter, we defer the proof that our hypergraph
category is a well-defined category until our discussion of functors in the
following section, \textsection\ref{sec.dcorf}.

\subsection{Adjoining right adjoints} \label{ssec.rightads}

Suppose we have a cospan $X+Y \to N$ with a decoration on $N$. Reducing this to
a corelation requires us to factor this to $X+Y \stackrel{e}\to \overline{N}
\stackrel{m}\to N$. To define a category of decorated corelations, then, we must
specify how to take decoration on $N$ and `pull it back' along $m$ to a decoration on
$\overline{N}$.

For decorated cospans, it is enough to have a functor $F$ from a category $\mc
C$ with finite colimits; the image $Ff$ of morphisms $f$ in $\mc C$ describes
how to move decorations forward along $f$. In this subsection we explain how to
expand $\mc C$ to include morphisms $m^\opp$ for each $m \in \mc C$, so that the
image $Fm^\opp$ of $m^\opp$ describes how to move morphisms backwards along $m$.
The construction of this expanded category $\mc C;\mc M^\opp$ is a simple
consequence of the stability of $\mc M$ under pushouts.

\begin{proposition}
  Let $\mathcal C$ be a category with finite colimits, and let $\mathcal M$ be a
  subcategory of $\mathcal C$ stable under pushouts. Then we define the category
  $\mathcal C; \mathcal M^\opp$ as follows  
  \smallskip 

  \begin{center}
    \begin{tabular}{| c | p{.65\textwidth} |}
      \hline
      \multicolumn{2}{|c|}{The symmetric monoidal category $(\mc C;\mc M^\opp,+)$} \\
      \hline
      \textbf{objects} & the objects of $\mathcal C$ \\ 
      \textbf{morphisms} & isomorphism classes of cospans of the form
      $\stackrel{c}\rightarrow \stackrel{m}\leftarrow$, where $c$ lies in
      $\mathcal C$ and $m$ in $\mathcal M$\\ 
      \textbf{composition} & given by pushout \\
      \textbf{monoidal product} & the coproduct in $\mathcal C$ \\
      \textbf{coherence maps} & the coherence maps in $\mc C$ \\
      \hline
    \end{tabular}
  \end{center}
  \smallskip 
\end{proposition}

\begin{proof}
  Our data is well defined: composition because $\mc M$ is stable under
  pushouts, and monoidal composition by Lemma \ref{lem.mcoproductsmc}. The
  coherence laws follow as this is a symmetric monoidal subcategory of
  $\cospan(\mc C)$. 
\end{proof}

\begin{remark}
As we state in the proof, the category $\mc C;\mc M^\opp$ is a subcategory of
$\cospan(\mc C)$. We can in fact view it as a sub-bicategory of the bicategory
of cospans in $\mc C$, where the 2-morphisms given by maps of cospans. In this
bicategory every morphism of $\mc M$ has a right adjoint.
\end{remark}

\begin{examples} 
  The factorisation systems introduced in Examples \ref{ex.factsysts} and
  considered again in Examples \ref{ex.corelcats} all have $\mc M$ stable under
  pushout. This yields the following examples.
  \begin{itemize}
    \item $\mathcal C; \mathcal C^\opp$ is by definition equal to
      $\cospan(\mathcal C)$.
    \item $\mathcal C;\mathcal I_{\mathcal C}^\opp$ is isomorphic to $\mathcal C$.
    \item $\Set;\mathrm{Inj}^\opp$ is the category with sets as objects and partial
      functions as morphisms.
  \end{itemize}
\end{examples}

The following lemma details how to construct functors between this type of
category.

\begin{lemma} \label{lem.madjointsfunctor}
  Let $\mathcal C$, $\mathcal C'$ be categories with finite colimits, and let
  $\mathcal M$, $\mathcal M'$ be subcategories of $\mc C$, $\mc C'$ respectively
  each stable under pushouts. Let $A\maps \mathcal C \to \mc C'$ be functor that
  preserves colimits and such that the image of $\mc M$ lies in $\mc M'$. Then
  $A$ extends to a symmetric strong monoidal functor
  \[
    A\maps \mc C;\mc M^\opp \longrightarrow \mc C'; \mc M'^\opp.
  \]
  mapping $X$ to $AX$ and $\stackrel{c}\rightarrow \stackrel{m}\leftarrow$ to
  $\stackrel{Ac}\rightarrow \stackrel{Am}\leftarrow$.
\end{lemma}
\begin{proof}
  Note $A(\mc M) \subseteq \mc M'$, so $\stackrel{Ac}\rightarrow
  \stackrel{Am}\leftarrow$ is indeed a morphism in $\mc C';\mc M'^\opp$. This is
  then a restriction and corestriction of the usual functor $\cospan(\mc C) \to
  \cospan(\mc C')$ to the above domain and codomain.
\end{proof}

Note that a similar construction giving subcategories of cospan categories could
be defined more generally using any two isomorphism-containing wide
subcategories stable under pushout.  The above, however, suffices for decorated
corelations.

\subsection{Decorated corelations} \label{ssec.deccorel}

As we have said, decorated corelations are constructed from a lax symmetric
monoidal functor from $\mc C;\mc M^\opp$ to $\Set$. In this subsection we define
decorated corelations and give a composition rule for them, showing that this
composition rule is well defined up to isomorphism.

\begin{definition}
  Let $\mathcal C$ be a category with finite colimits, $(\mathcal E, \mathcal
  M)$ be a factorisation system on $\mathcal C$ with $\mc M$ stable under
  pushout, and 
  \[
    F: (\mathcal C;\mathcal M^\opp,+) \longrightarrow (\Set, \times)
  \]
  be a lax symmetric monoidal functor.  We define an $F$-\define{decorated
  corelation} to a pair
  \[
    \left(
    \begin{aligned}
      \xymatrix{
	& N \\  
	X \ar[ur]^{i} && Y \ar[ul]_{o}
      }
    \end{aligned}
    ,
    \qquad
    \begin{aligned}
      \xymatrix{
	FN \\
	1 \ar[u]_{s}
      }
    \end{aligned}
    \right)
  \]
  where the cospan is jointly $\mathcal E$-like. A morphism of decorated
  corelations is a morphism of decorated cospans between two decorated
  corelations.  As usual, we will be lazy about the distinction between a
  decorated corelation and its isomorphism class.
\end{definition}

Suppose we have decorated corelations
\[
  \left(
  \begin{aligned}
    \xymatrix{
      & N \\  
      X \ar[ur]^{i_X} && Y \ar[ul]_{o_Y}
    }
  \end{aligned}
  ,
  \qquad
  \begin{aligned}
    \xymatrix{
      FN \\
      1 \ar[u]_{s}
    }
  \end{aligned}
  \right)
  \qquad
  \mbox{and}
  \qquad
  \left(
  \begin{aligned}
    \xymatrix{
      & M \\  
      Y \ar[ur]^{i_Y} && Z \ar[ul]_{o_Z}
    }
  \end{aligned}
  ,
  \qquad
  \begin{aligned}
    \xymatrix{
      FM \\
      1 \ar[u]_{t}
    }
  \end{aligned}
  \right).
\]
Then, recalling the notation introduced in
\textsection\ref{ssec.corelcats}, their composite is given by the composite
corelation
\[
  \xymatrix{
    & \overline{N+_YM} \\  
    X \ar[ur]^{e\circ \iota_X} && Z \ar[ul]_{e \circ \iota_Z}
  }
\]
paired with the decoration
\[
  1 \xrightarrow{\varphi_{N,M}\circ\langle s,t\rangle} F(N+M)
  \xrightarrow{F[j_N,j_M]} F(N+_YM) \xrightarrow{F(m^\opp)} F(\overline{N+_YM}).
\]
As composition of corelations and decorated cospans are both well defined up to
isomorphism, this too is well defined up to isomorphism. 

\begin{proposition} \label{prop.deccorelcomp}
  Let $\mc C$ be a category with finite colimits and a factorisation system
  $(\mc E,\mc M)$ with $\mc M$ stable under pushout. Then the above is a
  well-defined composition rule on isomorphism classes of decorated corelations.
\end{proposition}
\begin{proof}
Let 
\[
  \big(X \stackrel{i_X}\longrightarrow N \stackrel{o_Y}\longleftarrow Y,\enspace 1
\stackrel{s}\longrightarrow FN\big) \stackrel\sim\longrightarrow \big(X \stackrel{i'_X}\longrightarrow N'
\stackrel{o'_Y}\longleftarrow Y,\enspace 1 \stackrel{s'}\longrightarrow FN'\big)
\]
and
\[
  \big(Y \stackrel{i_Y}\longrightarrow M \stackrel{o_Z}\longleftarrow Z,\enspace 1
\stackrel{t}\longrightarrow FM\big) \stackrel\sim\longrightarrow \big(Y \stackrel{i'_Y}\longrightarrow M'
\stackrel{o'_Z}\longleftarrow Z,\enspace 1 \stackrel{t'}\longrightarrow FM'\big)
\]
be isomorphisms of decorated corelations. We wish to show that the composite of the
decorated corelations on the left is isomorphic to the composite of the decorated
corelations on the right. 

As discussed in Proposition \ref{prop.deccorelcomp}, the composites of the
underlying corelations are isomorphic, via an isomorphism $s$ which exists by the
factorisation system. We need to show this $s$ is an isomorphism of decorations.
This is a matter of showing the diagram 
\[
    \xymatrixcolsep{3pc}
    \xymatrixrowsep{1pc}
 \xymatrix{
   && F(N+_YM) \ar[r]^{Fm^\opp} \ar[dd]_\sim^{Fp} & F(\overline{N+_YM})
   \ar[dd]_\sim^{Fs} \\
   1 \ar[urr]^(.4){F[j_N,j_M]\circ\varphi_{N,M}\circ \langle s,t \rangle\qquad}
   \ar[drr]_(.4){F[j_{N'},j_{M'}]\circ\varphi_{N',M'}\circ \langle s',t'
 \rangle\qquad\quad} \\
   && F(N'+_YM') \ar[r]_{Fm'^\opp} & F(\overline{N'+_YM'})\\
 }
\]
The triangle commutes as composition of decorated cospans is well defined
(Proposition \ref{prop.composingdeccospans}), while the square commutes as
composition of corelations is well defined (Proposition
\ref{prop.deccorelcomp}).
\end{proof}

\begin{remark}
  Note that we have chosen to define decorated corelations only for lax symmetric
  monoidal functors 
  \[
    (\mathcal C;\mathcal M^\opp,+) \longrightarrow (\Set, \times),
  \]
  and not more generally for lax braided monoidal functors 
  \[
    (\mathcal C;\mathcal M^\opp,+) \longrightarrow (\mc D, \otimes)
  \]
  for braided monoidal categories $\mc D$. While we could use the latter type of
  functor for our constructions, a similar arguments to those in Section
  \ref{sec.setdecs} show we gain no extra generality. On the other hand, keeping
  track of this possibly varying category $\mc D$ in the following distracts
  from the main insights. We thus merely remark that it is possible to make the
  more general definition, and leave it at that.
\end{remark}

\subsection{Categories of decorated corelations} \label{ssec.deccorelcats}

In this subsection we define the hypergraph category $F\mathrm{Corel}$ of
decorated corelations. Having defined decorated corelations and their
composition in the previous subsection, the key question to address is the
provenance of the monoidal and hypergraph structure. 

Recall, from \textsection\ref{ssec.corelcats}, that to define the monoidal and hypergraph
structure on categories of corelations, we used functors $\mathrm{Cospan}(\mc C)
\to \mathrm{Corel}(\mc C)$, leveraging the monoidal and hypergraph structure on
cospan categories. In analogy, here we leverage the same fact for decorated
cospans, this time using a structure preserving map 
\[
  \square\maps F\mathrm{Cospan} \longrightarrow F\mathrm{Corel}.
\]
Here $F\mathrm{Cospan}$ denotes the decorated cospan category constructed from
the restriction of the functor $F\maps \mc C;\mc M^\opp \to \Set$ to the domain
$\mc C$. 

The monoidal product of two decorated corelations is their monoidal product as
decorated cospans. To define the coherence maps for this monoidal product, as
well as the coherence maps, we introduce the notion of a restricted decoration.

Given a cospan $X \to N \leftarrow Y$, write $m\maps \overline{N} \to N$ for the
$\mc M$ factor of the copairing $X+Y \to N$. The map $\square$ takes a
decorated cospan 
\[
  (X \stackrel{i}\longrightarrow N \stackrel{o}\longleftarrow Y, \enspace 1
    \stackrel{s}\longrightarrow FN)
\]
to the decorated corelation 
\[
  (X \stackrel{\overline{i}}\longrightarrow \overline N
  \stackrel{\overline{o}}{\longleftarrow} Y, \enspace 1 \xrightarrow{Fm^\opp \circ
  s} F\overline{N}),
\]
where the corelation is given by the jointly $\mc E$-part of the cospan, and the
decoration is given by composing $s$ with the $F$-image $Fm^\opp\maps FN \to
F\overline{N}$ of the map $N \stackrel{1_N}\to N \stackrel{m}\leftarrow
\overline{N}$ in $\mc C;\mc M^\opp$. This is well-defined up to isomorphism of
decorated corelations. We call $Fm^\opp \circ s$ the \define{restricted
decoration} of the decoration on the cospan $(X \to N \leftarrow Y, \enspace 1
\stackrel{s}\to FN)$.

We then make the following definition.

\begin{theorem} \label{thm.fcorel}
  Let $\mathcal C$ be a category with finite colimits and factorisation system
  $(\mathcal E, \mathcal M)$ with $\mathcal M$ stable under pushout, and let 
  \[
    F: (\mathcal C;\mathcal M^\opp,+) \longrightarrow (\Set, \times)
  \]
  be a lax symmetric monoidal functor.  Then we may define 
  \smallskip 

  \begin{center}
    \begin{tabular}{| c | p{.65\textwidth} |}
      \hline
      \multicolumn{2}{|c|}{The hypergraph category $(F\mathrm{Corel},+)$} \\
      \hline
      \textbf{objects} & the objects of $\mathcal C$ \\ 
      \textbf{morphisms} & isomorphism classes of $F$-decorated corelations in
      $\mathcal C$\\ 
      \textbf{composition} & given by $\mc E$-part of pushout with restricted
      decoration  \\
      \textbf{monoidal product} & the coproduct in $\mathcal C$  \\
      \textbf{coherence maps} & maps from $\cospan(\mc C)$ with restricted empty
      decoration \\
      \textbf{hypergraph maps} & maps from $\cospan(\mc C)$ with restricted empty
      decoration \\
      \hline
    \end{tabular}
  \end{center}
  \smallskip 
\end{theorem}

Similar to Theorem~\ref{thm.cospantocorel} defining the hypergraph category
$\corel(\mc C)$, we have now specified well-defined data and just need to check a
collection of coherence axioms. As before, we prove this in the next section,
alongside a theorem regarding functors between decorated corelation categories.

\begin{examples} \label{ex.cospansascorels}
  Recall that corelations with respect to the trivial \linebreak
  morphism--isomorphism factorisation system $(\mc C, \mc I_{\mc C})$ are simply
  cospans. Using this factorisation system, we see that decorated cospans are a
  special case of decorated corelations.
\end{examples}

\begin{example} \label{ex.undeccorel}
  `Undecorated' corelations are also a special case of decorated
  corelations: they are corelations decorated by the functor 
  \[
    1\maps \mc C;\mc M^\opp \to \Set
  \] 
  that maps each object to the one element set $1$, and each morphism to the
  identity function on $1$. This is a symmetric monoidal functor with the
  coherence maps all also the identity function on $1$.
\end{example}

%Associativity: To take a decoration on $A+B$ to one on $A+_C\overline B$ we may
%either reduce to the $\mathcal E$-part of $B$ and then pushout over $C$, or
%pushout over $C$ and then reduce to the $\mathcal E$ part of $B$. This lemma
%implies that both processes result in the same decoration.  i

\section{Functors between decorated corelation categories} \label{sec.dcorf}
In this section we show how to construct hypergraph functors between decorated
corelation categories. The construction of these functors holds no surprises:
their requirements combine the requirements of corelations and decorated
cospans. In the process of proving that our construction gives well-defined
hypergraph functors, we complete the necessary prerequisite proof that decorated
corelation categories are well-defined hypergraph categories.

Recall that Lemma \ref{lem.madjointsfunctor} says that, when the image of $\mc
M$ lies in $\mc M'$, we can extend a colimit-preserving functor $\mc C \to \mc
C'$ to a symmetric monoidal functor $\mc C;\mc M^\opp \to \mc C';\mc M'^\opp$ .

\begin{proposition}\label{prop.deccorelfunctors}
  Let $\mathcal C$, $\mathcal C'$ have finite colimits and respective
factorisation systems $(\mathcal E, \mathcal M)$, $(\mathcal E', \mathcal M')$,
such that $\mathcal M$ and $\mathcal M'$ are stable under pushout, and suppose
that we have lax symmetric monoidal functors
\[
  F: (\mathcal C;\mathcal M^\opp,+) \longrightarrow (\Set, \times)
\]
and
\[
  G: (\mathcal C';\mathcal M'^\opp,+) \longrightarrow (\Set, \times).
\]

Further let $A\maps \mathcal C \to \mathcal C'$ be a functor that preserves
finite colimits and such that the image of $\mathcal M$ lies in $\mathcal M'$.
This functor $A$ extends to a symmetric monoidal functor $\mc C;\mc M^\opp \to
\mc C';\mc M'^\opp$.

Suppose we have a monoidal natural transformation $\theta$:
\[
  \xymatrixrowsep{2ex}
  \xymatrix{
    \mc C; \mc M^\opp \ar[dd]_{A} \ar[drr]^F  \\
    &\twocell \omit{_\:\theta}& \Set \\
    \mc C'; \mc {M'}^\opp \ar[urr]_{G} 
  }
\]

Then we may define a hypergraph functor $T\maps F\mathrm{Corel} \to
G\mathrm{Corel}$ sending each object $X \in F\mathrm{Corel}$ to $AX \in
G\mathrm{Corel}$ and each decorated corelation 
  \[
    \left(
    \begin{aligned}
      \xymatrix{
	& N \\  
	X \ar[ur]^{i} && Y \ar[ul]_{o}
      }
    \end{aligned}
    ,
    \qquad
    \begin{aligned}
      \xymatrix{
	FN \\
	1 \ar[u]_{s}
      }
    \end{aligned}
    \right)
  \]  
to
  \[
    \left(
    \begin{aligned}
      \xymatrix{
	& \overline{AN} \\  
	AX \ar[ur]^{e'\circ\iota_{AX}} && AY \ar[ul]_{e'\circ\iota_{AY}}
      }
    \end{aligned}
    ,
    \qquad
    \begin{aligned}
      \xymatrixrowsep{1.5ex}
      \xymatrix{
	G\overline{AN} \\
	GAN \ar[u]_{Gm_{AN}^\opp}\\
        FN \ar[u]_{\theta_N}\\
	1 \ar[u]_{s}
      }
    \end{aligned}
    \right)
  \]  
  The coherence maps $\overline{\kappa_{X,Y}}$ are given by the coherence maps of
  $A$ with the restricted empty decoration.
\end{proposition}
\begin{proof}[Proof of Theorem \ref{thm.fcorel} and Proposition
  \ref{prop.deccorelfunctors}]
  In the proof of Theorem \ref{thm.cospantocorel} and Proposition
  \ref{prop.corelfunctors} we proved that the map 
  \[
    \square\maps \corel(\mc C) \longrightarrow \corel(\mc C')
  \]
  preserved composition and had natural coherence maps. Specialising to the case
  when $\corel(\mc C)=\cospan(\mc C')$, we saw that this bijective-on-objects,
  surjective-on-morphisms, composition and monoidal product preserving map
  proved $\corel(\mc C')$ is a hypergraph category, and it immediately followed
  that $\square$ is a hypergraph functor.

  The analogous argument holds here: we simply need to prove
  \[
    \square\maps F\corel \longrightarrow G\corel
  \]
  preserves composition and has natural coherence maps. Theorem \ref{thm.fcorel}
  then follows from examining the map $F\cospan \to F\corel$ obtained by
  choosing $\mc C = \mc C'$, $(\mc E,\mc M) = (\mc C', \mc I_{\mc C'})$, $F$ the
  restriction of $G$ to $\mc C'$, $A$ the identity functor on $\mc C'$, and
  $\theta$ the identity natural transformation. Subsequently Proposition
  \ref{prop.deccorelfunctors} follows from noting that all the axioms hold for
  the corresponding maps in $G\mathrm{Cospan}$.
  
  \paragraph{$\square$ preserves composition.} Suppose we have decorated corelations
  \[
    f=(X \stackrel{i_X}{\longrightarrow} N \stackrel{o_Y}{\longleftarrow} Y,
    \enspace 1 \stackrel{s}{\to} FN)
    \qquad
    \mbox{and}
    \qquad 
    g=(Y \stackrel{i_Y}{\longrightarrow} M \stackrel{o_Y}{\longleftarrow} Z,
    \enspace 1 \stackrel{t}{\to} FM)
  \]
%  \[
%    \left(
%    \begin{aligned}
%      \xymatrix{
%	& N \\  
%	X \ar[ur]^{i} && Y \ar[ul]_{o}
%      }
%    \end{aligned}
%    ,
%    \qquad
%    \begin{aligned}
%      \xymatrix{
%	FN \\
%	1 \ar[u]_{s}
%      }
%    \end{aligned}
%    \right)
%    \qquad
%    \mbox{and}
%    \qquad
%    \left(
%    \begin{aligned}
%      \xymatrix{
%	& M \\  
%	Y \ar[ur]^{i} && Z \ar[ul]_{o}
%      }
%    \end{aligned}
%    ,
%    \qquad
%    \begin{aligned}
%      \xymatrix{
%	FM \\
%	1 \ar[u]_{s}
%      }
%    \end{aligned}
%    \right).
%  \]
  We know the functor $\square$ preserves composition on the cospan part; this
  is precisely the content of Proposition \ref{prop.corelfunctors}. It remains to
  check that $\square( g \circ f)$ and $\square g \circ \square f$ have
  isomorphic decorations. This is expressed by the commutativity of the
  following diagram:
  \[
    \xymatrixrowsep{1.1pc}
    \xymatrixcolsep{.6pc}
    \xymatrix{
      \scriptstyle G\overline{A(\overline{N+_YM})} \ar[rrrrrr]^{Gn} &&&&&&
      \scriptstyle G\overline{(\overline{AN}+_{AY}\overline{AM})}\\
      \\
      \scriptstyle GA(\overline{N+_YM}) \ar[uu]^{Gm^\opp_{A(\overline{N+_YM})}} &&&
      \textsc{\tiny($\ast\ast$)} &&& 
      \scriptstyle G(\overline{AN}+_{AY}\overline{AM})
      \ar[uu]_{Gm^\opp_{\overline{AN}+_{AY}\overline{AM}}} \\
      \\
      \scriptstyle F(\overline{N+_YM}) \ar[uu]^{\theta_{\overline{N+_YM}}} &
      %\textsc{\tiny()} & 
      &
      \scriptstyle GA(N+_YM) \ar[uull]_{GAm^\opp_{N+_YM}} && 
      \scriptstyle G(AN+_{AY}AM) \ar[ll]_{G\sim} \ar[uurr]^{G(m_{AN}^\opp
      +_{AY}m_{AM}^\opp)\phantom{spac}}
      & \textsc{\tiny(c)} & 
      \scriptstyle G(\overline{AN}+\overline{AM})
      \ar[uu]_{G[j_{\overline{AN}},j_{\overline{AM}}]} 
      \\
      &\textsc{\tiny(tn)}&& \textsc{\tiny(a)} 
      \\
      \scriptstyle F(N+_YM) \ar[uu]^{Fm_{N+_YM}^\opp}
      %\ar[uurr]_{\theta_{N+_YM}} & \textsc{\tiny()} & 
      &&
      \scriptstyle GA(N+M)\ar[uu]_{GA[j_N,j_M]} && 
      \scriptstyle G(AN+AM) \ar[uu]^{G[j_{AN},j_{AM}]} \ar[ll]^{G\alpha_{N,M}}
      \ar[uurr]^(.6){G(m_{AN}^\opp +m_{AM}^\opp)\phantom{s}} & \textsc{\tiny(gm)} & 
      \scriptstyle G\overline{AN} \times G\overline{AM}
      \ar[uu]_{\gamma_{\overline{AN},\overline{AM}}} \\
      \\
      \scriptstyle F(N+M) \ar[uu]^{F[j_N,j_M]} \ar[uurr]_{\theta_{N+M}} &&&
      \textsc{\tiny(tm)} &&& 
      \scriptstyle GAN \times GAM \ar[uu]_{Gm_{AN}^\opp \times Gm_{AM}^\opp}
      \ar[uull]^{\gamma_{AN,AM}} \\
      \\
      &&&
      \scriptstyle FN \times FM \ar[uulll]^{\varphi_{N,M}}
      \ar[uurrr]_{\theta_N \times \theta_M} \\\\
      &&&
      \scriptstyle 1 \ar[uu]_{\rho_1\circ (s \times t)}
    }
  \]
  This diagram does indeed commute. To check this, first observe that \textsc{(tm)}
  commutes by the monoidality of $\theta$, \textsc{(gm)} commutes by the
  monoidality of $G$, and \textsc{(tn)} commutes by the naturality of $\theta$.
  The remaining three diagrams commute as they are $G$-images of diagrams that
  commute in $\mc C';\mc M'^\opp$. Indeed, \textsc{(a)} commutes since $A$ preserves
  colimits and $G$ is functorial, \textsc{(c)} commutes as it is the $G$-image
  of a pushout square in $\mc C'$, so 
  \[
    \xleftarrow{m_{AN}+m_{AM}}
    \xrightarrow{[j_{\overline{AN}},j_{\overline{AM}}]}
    \quad 
    \textrm{and}
    \quad
    \xrightarrow{[j_{AN},j_{AM}]}
    \xleftarrow{m_{AN}+_{AY}m_{AM}} 
  \]
  are equal as morphisms of $\mc C';\mc M'^\opp$, and \textsc{($\ast\ast$)}
  commutes as it is the $G$-image of the right-hand subdiagram of
  (\ref{diag.eparts}) used to define $n$ in the proof of Lemma
  \ref{lem.corelfuncomposition}.

  \paragraph{Coherence maps are natural.}
  Let $f = (X \longrightarrow N \longleftarrow Y, \enspace 1 \to FN)$, $g= (Z
  \longrightarrow M \longleftarrow W, \enspace 1 \to FM)$ be $F$-decorated
  corelations in $\mc C$. We wish to show that
  \[
    \xymatrixcolsep{4pc}
    \xymatrixrowsep{2pc}
    \xymatrix{
      AX+AY \ar[r]^{\square f+\square g}
      \ar[d]_{\overline{\kappa_{X,Y}}} & 
      AZ+AW \ar[d]^{\overline{\kappa_{Z,W}}} \\
      A(X+Y) \ar[r]^{\square(f+g)} & A(Z+W)
    }
  \]
  commutes in $G\mathrm{Corel}$, where the coherence maps are given by
  \[
    \overline{\kappa_{X,Y}}=          
    \left(
    \begin{aligned}
      \xymatrix{
	& \overline{A(X+Y)} \\  
	AX+AY \ar[ur] && A(X+Y) \ar[ul]
      }
    \end{aligned}
    ,
    \qquad
    \begin{aligned}
      \xymatrixrowsep{1.4ex}
      \xymatrix{
	G(\overline{A(X+Y)}) \\
	GA(X+Y) \ar[u]_{Gm_{AX+AY}^\opp} \\
	G\varnothing \ar[u]_{G!} \\
	1 \ar[u]_{\gamma_1}
      }
    \end{aligned}
    \right).
  \]
  Lemma \ref{lem.corelfunmonoidal} shows that the composites of corelations
  agree. It remains to check that the decorations also agree.

  Here Lemma \ref{lem.emptydecorations} is helpful. Since $\square$ is
  composition preserving, we can replace the $\overline{\kappa}$ with the empty
  decorated coherence maps $\kappa$ of $G\mathrm{Cospan}$, and compute these
  composites in $G\mathrm{Cospan}$, before restricting to the $\mc E'$-parts.
  Lemma \ref{lem.emptydecorations} then implies that the restricted empty
  decorations on the isomorphisms $\overline{\kappa}$ play no role in
  determining the composite decorations. It is thus enough to prove that the
  decorations of $\square f + \square g$ and $\square(f+g)$ are the same up to
  the isomorphism $p\maps G(\overline{AN} +\overline{AM}) \to
  G\overline{A(N+M)}$ between their apices, as defined in the diagram
  (\ref{diag.natural}) in the proof of Lemma \ref{lem.corelfunmonoidal}.

  This comes down to proving the following diagram commutes:
  \[
    \xymatrixrowsep{.8pc}
    \xymatrixcolsep{.8pc}
    \xymatrix{
      &&&& 
      GAN \times GAM \ar[rrdd]^{\gamma} \ar[rr]^{Gm \times Gm} && 
      G\overline{AN} \times G\overline{AM} \ar[rr]^{\gamma} && 
      G(\overline{AN}+\overline{AM}) \ar[dddd]^{Gp}_\sim \\ 
      &&&&&& 
      \textrm{\tiny(G)} \\
      1 \ar[rr]^(.4){\langle s,t \rangle} && 
      FN\times FM \ar[uurr]^{\theta} \ar[ddrr]_{\varphi} && 
      \textrm{\tiny(T)} && 
      G(AN+AM) \ar[uurr]^{G(m+m)} \ar[dd]_{G\kappa} \\
      &&&&&&& 
      \textrm{\tiny(\#\#)}\\ 
      &&&& 
      F(N+M) \ar[rr]_{\theta} && 
      GA(N+M) \ar[rr]_{Gm} && 
      G\overline{A(N+M)}
    }
  \]
  This is straightforward to check: (T) commutes by the monoidality of $\theta$,
  (G) by the monoidality of $G$, and (\#\#) as it is the $G$-image of the
  rightmost square in (\ref{diag.natural}).
\end{proof}

In particular, we get a hypergraph functor from the category of $F$-decorated
cospans to the category of $F$-decorated corelations. In our applications, this
is often the key aspect of constructing `black box' or semantic functors.

%  \[
%    \xymatrixrowsep{1pc}
%    \xymatrixcolsep{1pc}
%    \xymatrix{
%      &&&& F(N+M) \ar[dd]^{F((m_N +m_M)^\opp)} \ar[rr]^{F[j_N,j_M]} 
%      && F(N+_YM) \ar[dd]^{F((m_N+_Ym_M)^\opp)} \ar[rr]^{F((m_{N+_YM})^\opp)} 
%      && F(\overline{N+_YM}) \ar[dd]^{Fn}_\sim \\ 
%      1 \ar[rr]^(.4){(s\times t)\circ\lambda^{-1}} 
%      && FN\times FM \ar[urr]^{\varphi} \ar[dr]_{Fm_N^\opp\times Fm_M^\opp} &
%      \qquad\textrm{\tiny(I)}% && \textrm{\tiny(F)} && \textrm{\tiny(C)}
%      \\ 
%      &&& F\overline{N} \times F\overline{M} \ar[r]_{\varphi} 
%      & F(\overline{N}+\overline{M})
%      \ar[rr]_{F[j_{\overline{N}},j_{\overline{M}}]} 
%      && F(\overline{N}+_Y\overline{M})
%      \ar[rr]_{F((m_{\overline{N}+_Y\overline{M}})^\opp)} 
%      && F(\overline{\overline{N}+_Y\overline{M}})
%    }
%  \]
%  Here $n$ is the isomorphism from the proof of Proposition
%  \ref{prop.corelfunctors}.
%  The leftmost square (I) commutes by the naturality of $\varphi$, the central
%  square commutes as it is the $F$-image of a pushout square in $\mathcal C$,
%  and rightmost square commutes as it is the $F$-image of the rightmost square
%  in the commutative diagram (\ref{diag.eparts}) in $\mathcal M^\opp$.

\begin{corollary}
  Let $\mathcal C$ be a category with finite colimits, and let $(\mathcal E,
  \mathcal M)$ be a factorisation system on $\mathcal C$. Suppose that we also
  have a lax monoidal functor
  \[
    F: (\mathcal C;\mathcal M^\opp,+) \longrightarrow (\Set, \times).
  \]
  Then we may define a category $F\mathrm{Corel}$ with objects the objects of
  $\mathcal C$ and morphisms isomorpism classes of $F$-decorated corelations.

  Write also $F$ for the restriction of $F$ to the wide subcategory $\mathcal
  C$ of $\mathcal C;\mathcal M^\opp$. We can thus also obtain the category
  $F\mathrm{Cospan}$ of
  $F$-decorated cospans. We moreover have a functor 
  \[
    F\mathrm{Cospan} \to F\mathrm{Corel}
  \]
  which takes each object of $F\mathrm{Cospan}$ to itself as an object of
  $F\mathrm{Corel}$, and each decorated cospan
  \[
    \left(
    \begin{aligned}
      \xymatrix{
	& N \\  
	X \ar[ur]^{i} && Y \ar[ul]_{o}
      }
    \end{aligned}
    ,
    \qquad
    \begin{aligned}
      \xymatrix{
	FN \\
	1 \ar[u]_{s}
      }
    \end{aligned}
    \right)
  \]  
  to its jointly $\mathcal E$-part
  \[
    \xymatrix{
      & \overline{N} \\  
      X \ar[ur]^{e\circ \iota_X} && Y \ar[ul]_{e\circ\iota_Y}
    }
  \]
  decorated by the composite
  \[
    \xymatrix{
      1 \ar[r]^s & FN \ar[r]^{Fm_N^\opp} & F\overline{N}.
    }
  \]
\end{corollary}

\section{All hypergraph categories are decorated corelation categories}
\label{sec.allhypergraphs}

Not all hypergraph categories are decorated \emph{cospan} categories. To see
this, we can count so-called scalars: morphisms from the monoidal unit
$\varnothing$ to itself. In a decorated cospan category, the set of morphisms
from $X$ to $Y$ always comprises all decorated cospans $(X \to N \leftarrow
Y,\enspace 1 \to FN)$. Now for any object $N$ in the underlying category $\mc
C$, there is a unique morphism $\varnothing \to N$. This means that the
morphisms $\varnothing \to \varnothing$ are indexed by (isomorphism classes of)
elements of $FN$, ranging over $N$.

Suppose we have a decorated cospan category with a unique morphism $\varnothing
\to \varnothing$. By the previous paragraph, and replacing $\mc C$ with an
equivalent skeletal category, this implies there is only one object $N$ such
that $FN$ is nonempty. But $FN$ must always contain at least one element, the
empty decoration $1 \xrightarrow{\varphi_I} F\varnothing \xrightarrow{F!} FN$.
This implies there is only one object $N$ in $\mc C$: the object $\varnothing$.
Thus $\mc C$ must be the one object discrete category, and $F\maps \mc C \to
\Set$ is the functor that sends the object of $\mc C$ to the one element set
$1$.

Hence any decorated cospan category with a unique morphism $\varnothing \to
\varnothing$ is the one object discrete category. But the hypergraph
category of finite sets and equivalence relations, discussed in Subsection
\ref{ssec.equivrels}, is a nontrivial category with a unique morphism
$\varnothing \to \varnothing$.  Thus it cannot be constructed as a decorated
cospan category.

Decorated corelation categories, however, are more powerful. In this section we
show that all hypergraph categories are decorated corelation categories, and
provide some examples for intuition.

\subsection{Representing a morphism with its name}
Suppose we are given a hypergraph category $\mc H$ and wish to construct a
hypergraph equivalent decorated corelation category. The main idea to is to take
advantage of the compact closed structure: recall from
\textsection\ref{ssec.compactclosed} that morphisms $X \to Y$ in a hypergraph
category are in one-to-one correspondence with their so-called names $I \to X\ot
Y$. We shall also witness the great utility of the seemingly trivial
isomorphism-morphism factorisation system. 

We start with the wide subcategory of $\mc H$ comprising, roughly speaking, just
the coherence and hypergraph maps. Then, choosing the isomorphism-morphism
factorisation system on this subcategory, we use the restriction of hom functor
$\mc H(I,-)$ to this subcategory to decorate each corelation $X \to X \ot Y
\leftarrow Y$ with the monoidal elements of $X\ot Y$.  The morphisms in the
resulting category are cospans $X \xrightarrow{1_X \otimes \eta_Y} X\otimes Y
\xleftarrow{\eta_X \otimes 1_Y} Y$ decorated by a map $I \to X \ot Y$. This
recovers the original hypergraph category.

\begin{theorem}\label{thm.hypdeccorcats}
  Every hypergraph category is hypergraph equivalent to a decorated corelation
  category. That is, given a hypergraph category $(\mc H,\otimes)$, there exists
  a lax symmetric monoidal functor $F$ such that $(\mc H,\otimes)$ is hypergraph
  equivalent to $F\mathrm{Corel}$.
\end{theorem}
\begin{proof}
  Let $(\mc H,\ot)$ be a hypergraph category. By Theorem
  \ref{thm.stricthypergraphs}, we have a hypergraph equivalent strict hypergraph
  category $(\mc H_{\mathrm{str}},\cdot)$ with objects finite lists of objects
  in $\mc H$. We will build $\mc H_{\mathrm{str}}$ as a decorated corelation
  category.

  Write $\mc O$ for the collection of objects in $\mc H$.  Recall that we may
  write $\FinSet$ for the skeletal strict symmetric monoidal category of finite
  sets and functions, and here let us further denote the objects of $\FinSet$ as
  $n= \{0,1,\dots,n-1\}$ for $n \in \mathbb N$, where $0$ is the empty set. We
  then write $\FinSet_{\mc O}$ for the category of $\mc O$-labelled finite sets
  (i.e.  an object is a finite set $n$ with each element in $n$ labelled by some
  object in $\mc H$, and a morphism is a function that preserves labels). Note
  that, due to the chosen ordering on elements of each finite set $n$, the
  objects of $\FinSet_{\mc O}$ are in one-to-one correspondence with finite
  lists of objects of $\mc H$, and hence can be considered the same as the
  objects of $\mc H_{\mathrm{str}}$.
  
  Next, observe that the category $\FinSet_{\mc O}$ has finite colimits: the
  colimit is just the colimit set in $\FinSet$ with each element labelled by the
  label of any element that maps to it in the colimit diagram. This labelling
  exists and is unique by the universal property of the colimit. Thus we view
  $\FinSet_{\mc O}$ as a symmetric monoidal category with the coproduct, written
  $+$ as per usual, as its monoidal product.  

  We define a strict hypergraph functor that is a wide embedding of categories
  \[
    H\maps (\cospan(\FinSet_{\mc O}),+) \longrightarrow (\mc
    H_{\mathrm{str}},\cdot),
  \]
  whose image is the hypergraph structure of $\mc H_{\mathrm{str}}$. More
  precisely, $H$ sends each object of $\cospan(\FinSet_{\mc O})$ to the same as
  an object of $\mc H_{\mathrm{str}}$.  As the morphisms of $\FinSet$---that is,
  all functions between finite sets---can be generated via composition and
  coproduct from the unique functions $\mu \maps 2 \to 1$, $\eta \maps 0 \to 1$,
  and the braiding $2 \to 2$ in $\FinSet$, the morphisms of $\FinSet_{\mc O}$
  are generated as a symmetric monoidal category by the unique morphisms
  $\mu_x\maps [x:x]\to [x]$ and $\eta_x\maps \varnothing \to [x]$, where $x$
  ranges over objects of $\mc H$.  This then implies the morphisms of
  $\cospan(\FinSet_{\mc O})$ are generated as a symmetric monoidal category by
  the cospans $\mu_x$, $\eta_x$, $\delta_x = \mu_x^\opp$, and $\epsilon_x =
  \eta_x^\opp$. The functor $H$ maps these generators to the corresponding
  Frobenius maps on $[x]$ in $\mc H_{\mathrm{str}}$. Since $\cospan(\FinSet)$ is
  the `generic special commutative monoid' (see Proposition
  \ref{prop.cospanscfm}), it is straightforward to check this defines a
  hypergraph functor.

  Next, recall that the hom functor $\mc H_{\mathrm{str}}(I,-)\maps (\mc
  H_{\mathrm{str}},\cdot) \to (\Set,\times)$ is a lax symmetric monoidal functor taking each
  object $X$ of $\mc H_{\mathrm{str}}$ to the homset $\mc H_{\mathrm{str}}(I,X)$
  (Proposition \ref{prop.monglobalsecs}). Write $F$ for the composite of these
  two functors: 
  \[
    F\maps (\cospan(\FinSet_{\mc O}),+) \stackrel{H}\longrightarrow (\mc
    H_{\mathrm{str}},\cdot) \xrightarrow{\mc H_{\mathrm{str}}(I,-)}
    (\Set,\times).
  \]
  This is lax symmetric monoidal functor mapping a finite list $X$ of objects in
  $\mc H$ to the homset $\mc H_{\mathrm{str}}(I,X)$, and a cospan between lists
  of objects in $\mc H$ to the corresponding Frobenius map.

  Now consider $\FinSet_{\mc O}$ with an isomorphism-morphism factorisation
  system. This implies $F$ defines a decorated corelation category
  $F\mathrm{Corel}$ with the objects of $\mc H_{\mathrm{str}}$ as objects, and
  the unique isomorphism-morphism corelation $X \xrightarrow{\iota_X} X+Y
  \xleftarrow{\iota_Y} Y$ from $X \to Y$ decorated by some morphism $s \in \mc
  H_{\mathrm{str}}(I,X\cdot Y)$ as morphisms $X \to Y$. We complete this proof
  by showing $(F\mathrm{Corel},+)$ and $(\mc H_{\mathrm{str}},\cdot)$ are
  isomorphic as hypergraph categories.
  
  First, let us examine composition in $F\mathrm{Corel}$. As morphisms $X \to Y$
  are completely specified by their decoration $s \in \mc
  H_{\mathrm{str}}(I,X\cdot Y)$, we will abuse terminology and refer to the
  decorations themselves as the morphisms.  Given morphisms $s \in \mc
  H_{\mathrm{str}}(I,X\cdot Y)$ and $t \in \mc H_{\mathrm{str}}(I,Y\cdot Z)$ in
  $F\mathrm{Corel}$, composition is given by the map 
  \[
    H_{\mathrm{str}}(I,X\cdot Y\cdot Y\cdot Z) \to H_{\mathrm{str}}(I,X\cdot Z)
  \]
  arising as the $F$-image of the cospan $X+Y+Y+Z
  \xrightarrow{[j_{X+Y},j_{Y+Z}]} X+Y+Z \xleftarrow{[\iota_X,\iota_Z]} X+Z$,
  where these maps come from the pushout square
  \[
    \xymatrix{
      && X+Y+Z \\
      & X+Y \ar[ur]^{j_{X+Y}} && Y+Z \ar[ul]_{j_{Y+Z}} \\
      X \ar[ur]^{\iota_X} && Y \ar[ul]_{\iota_Y} \ar[ur]^{\iota_Y} && Z
      \ar[ul]_{\iota_Z}
    }
  \]
  and the trivial isomorphism-morphism factorisation $X+Z = X+Z
  \xrightarrow{[\iota_X,\iota_Z]} X+Y+Z$ in $\FinSet_{\mc O}$.
  
  In terms of string diagrams in $\mc H_{\mathrm{str}}$, this means composing
  the maps 
  \[
    \tikzset{every path/.style={line width=1.1pt}}
    \begin{tikzpicture}
      \begin{pgfonlayer}{nodelayer}
	\node [style=none] (0) at (-0.25, 0.375) {};
	\node [style=none] (1) at (0.5, 0.375) {};
	\node [style=none] (2) at (-0.25, -0.375) {};
	\node [style=none] (3) at (0.5, -0.375) {};
	\node [style=none] (4) at (0.5, 0.25) {};
	\node [style=none] (5) at (0.5, -0.25) {};
	\node [style=none] (6) at (1.25, 0.25) {};
	\node [style=none] (7) at (1.25, -0.25) {};
	\node [style=none] (8) at (0.125, -0) {$s$};
	\node [style=none] (9) at (1.5, 0.25) {$X$};
	\node [style=none] (10) at (1.5, -0.25) {$Y$};
      \end{pgfonlayer}
      \begin{pgfonlayer}{edgelayer}
	\draw (0.center) to (1.center);
	\draw (1.center) to (3.center);
	\draw (3.center) to (2.center);
	\draw (2.center) to (0.center);
	\draw (4.center) to (6.center);
	\draw (5.center) to (7.center);
      \end{pgfonlayer}
    \end{tikzpicture}
    \qquad 
    \qquad
    \begin{tikzpicture}
      \begin{pgfonlayer}{nodelayer}
	\node [style=none] (0) at (-0.25, 0.375) {};
	\node [style=none] (1) at (0.5, 0.375) {};
	\node [style=none] (2) at (-0.25, -0.375) {};
	\node [style=none] (3) at (0.5, -0.375) {};
	\node [style=none] (4) at (0.5, 0.25) {};
	\node [style=none] (5) at (0.5, -0.25) {};
	\node [style=none] (6) at (1.25, 0.25) {};
	\node [style=none] (7) at (1.25, -0.25) {};
	\node [style=none] (8) at (0.125, -0) {$t$};
	\node [style=none] (9) at (1.5, 0.25) {$Y$};
	\node [style=none] (10) at (1.5, -0.25) {$Z$};
      \end{pgfonlayer}
      \begin{pgfonlayer}{edgelayer}
	\draw (0.center) to (1.center);
	\draw (1.center) to (3.center);
	\draw (3.center) to (2.center);
	\draw (2.center) to (0.center);
	\draw (4.center) to (6.center);
	\draw (5.center) to (7.center);
      \end{pgfonlayer}
    \end{tikzpicture}
  \]
  with the Frobenius map
  \[
    \tikzset{every path/.style={line width=1.1pt}}
    \begin{aligned}
      \begin{tikzpicture}
	\begin{pgfonlayer}{nodelayer}
	  \node [style=none] (0) at (-0.125, 0.75) {};
	  \node [style=none] (1) at (-0.125, 0.25) {};
	  \node [style=none] (2) at (-0.125, -0.25) {};
	  \node [style=none] (3) at (-0.125, -0.75) {};
	  \node [style=none] (4) at (0.5, -0) {};
	  \node [style=none] (5) at (1, 0.75) {};
	  \node [style=none] (6) at (1, -0.75) {};
	  \node [style=none] (8) at (-0.375, 0.75) {$X$};
	  \node [style=none] (9) at (-0.375, 0.25) {$Y$};
	  \node [style=none] (10) at (-0.375, -0.25) {$Y$};
	  \node [style=none] (11) at (-0.375, -0.75) {$Z$};
	  \node [style=none] (12) at (1.25, 0.75) {$X$};
	  \node [style=none] (13) at (1.25, -0.75) {$Z$};
	\end{pgfonlayer}
	\begin{pgfonlayer}{edgelayer}
	  \draw (0.center) to (5.center);
	  \draw [in=90, out=0, looseness=1.00] (1.center) to (4.center);
	  \draw [in=-90, out=0, looseness=1.00] (2.center) to (4.center);
	  \draw (3.center) to (6.center);
	\end{pgfonlayer}
      \end{tikzpicture}
    \end{aligned}
    \quad
    =
    \quad
    \begin{aligned}
      \begin{tikzpicture}
	\begin{pgfonlayer}{nodelayer}
	  \node [style=none] (0) at (-0.125, 0.75) {};
	  \node [style=none] (1) at (-0.125, 0.25) {};
	  \node [style=none] (2) at (-0.125, -0.25) {};
	  \node [style=none] (3) at (-0.125, -0.75) {};
	  \node [style=none] (4) at (0.5, -0) {};
	  \node [style=none] (5) at (1.25, 0.75) {};
	  \node [style=none] (6) at (1.25, -0.75) {};
	  \node [style=circ2] (7) at (1, -0) {};
	  \node [style=none] (8) at (-0.375, 0.75) {$X$};
	  \node [style=none] (9) at (-0.375, 0.25) {$Y$};
	  \node [style=none] (10) at (-0.375, -0.25) {$Y$};
	  \node [style=none] (11) at (-0.375, -0.75) {$Z$};
	  \node [style=none] (12) at (1.5, 0.75) {$X$};
	  \node [style=none] (13) at (1.5, -0.75) {$Z$};
	\end{pgfonlayer}
	\begin{pgfonlayer}{edgelayer}
	  \draw (0.center) to (5.center);
	  \draw [in=90, out=0, looseness=1.00] (1.center) to (4.center);
	  \draw [in=-90, out=0, looseness=1.00] (2.center) to (4.center);
	  \draw (3.center) to (6.center);
	  \draw (4.center) to (7.center);
	\end{pgfonlayer}
      \end{tikzpicture}
    \end{aligned}
  \]
  to get 
  \[
    \tikzset{every path/.style={line width=1.1pt}}
    \begin{aligned}
      \begin{tikzpicture}
	\begin{pgfonlayer}{nodelayer}
	  \node [style=none] (0) at (-0.25, 0.375) {};
	  \node [style=none] (1) at (0.5, 0.375) {};
	  \node [style=none] (2) at (-0.25, -0.375) {};
	  \node [style=none] (3) at (0.5, -0.375) {};
	  \node [style=none] (4) at (0.5, 0.25) {};
	  \node [style=none] (5) at (0.5, -0.25) {};
	  \node [style=none] (6) at (1.25, 0.25) {};
	  \node [style=none] (7) at (1.25, -0.25) {};
	  \node [style=none] (8) at (0.125, -0) {$t\circ s$};
	  \node [style=none] (9) at (1.5, 0.25) {$X$};
	  \node [style=none] (10) at (1.5, -0.25) {$Z$};
	\end{pgfonlayer}
	\begin{pgfonlayer}{edgelayer}
	  \draw (0.center) to (1.center);
	  \draw (1.center) to (3.center);
	  \draw (3.center) to (2.center);
	  \draw (2.center) to (0.center);
	  \draw (4.center) to (6.center);
	  \draw (5.center) to (7.center);
	\end{pgfonlayer}
      \end{tikzpicture}
    \end{aligned}
    \quad = 
    \quad
    \begin{aligned}
      \begin{tikzpicture}
	\begin{pgfonlayer}{nodelayer}
	  \node [style=none] (0) at (-0.875, 0.875) {};
	  \node [style=none] (1) at (-0.125, 0.875) {};
	  \node [style=none] (2) at (-0.125, 0.125) {};
	  \node [style=none] (3) at (-0.875, 0.125) {};
	  \node [style=none] (4) at (-0.875, -0.125) {};
	  \node [style=none] (5) at (-0.125, -0.125) {};
	  \node [style=none] (6) at (-0.125, -0.875) {};
	  \node [style=none] (7) at (-0.875, -0.875) {};
	  \node [style=none] (8) at (-0.125, 0.75) {};
	  \node [style=none] (9) at (-0.125, 0.25) {};
	  \node [style=none] (10) at (-0.125, -0.25) {};
	  \node [style=none] (11) at (-0.125, -0.75) {};
	  \node [style=none] (12) at (0.5, -0) {};
	  \node [style=none] (13) at (0.75, 0.75) {};
	  \node [style=none] (14) at (0.75, -0.75) {};
	  \node [style=none] (15) at (-0.5, 0.5) {$s$};
	  \node [style=none] (16) at (-0.5, -0.5) {$t$};
	  \node [style=none] (17) at (1, 0.75) {$X$};
	  \node [style=none] (18) at (1, -0.75) {$Z$};
	\end{pgfonlayer}
	\begin{pgfonlayer}{edgelayer}
	  \draw (0.center) to (1.center);
	  \draw (1.center) to (2.center);
	  \draw (2.center) to (3.center);
	  \draw (3.center) to (0.center);
	  \draw (4.center) to (5.center);
	  \draw (5.center) to (6.center);
	  \draw (6.center) to (7.center);
	  \draw (7.center) to (4.center);
	  \draw (8.center) to (13.center);
	  \draw [in=90, out=0, looseness=1.00] (9.center) to (12.center);
	  \draw [in=-90, out=0, looseness=1.00] (10.center) to (12.center);
	  \draw (11.center) to (14.center);
	\end{pgfonlayer}
      \end{tikzpicture}
    \end{aligned}
  \]
  in $\mc H_{\mathrm{str}}(I,X\cdot Z)$. 

  The monoidal product is given by
  \[
    \tikzset{every path/.style={line width=1.1pt}}
    \begin{aligned}
      \begin{tikzpicture}
	\begin{pgfonlayer}{nodelayer}
	  \node [style=none] (0) at (-0.25, 0.375) {};
	  \node [style=none] (1) at (0.5, 0.375) {};
	  \node [style=none] (2) at (-0.25, -0.375) {};
	  \node [style=none] (3) at (0.5, -0.375) {};
	  \node [style=none] (4) at (0.5, 0.25) {};
	  \node [style=none] (5) at (0.5, -0.25) {};
	  \node [style=none] (6) at (1.25, 0.25) {};
	  \node [style=none] (7) at (1.25, -0.25) {};
	  \node [style=none] (8) at (0.125, -0) {$s$};
	  \node [style=none] (9) at (1.5, 0.25) {$X$};
	  \node [style=none] (10) at (1.5, -0.25) {$Y$};
	\end{pgfonlayer}
	\begin{pgfonlayer}{edgelayer}
	  \draw (0.center) to (1.center);
	  \draw (1.center) to (3.center);
	  \draw (3.center) to (2.center);
	  \draw (2.center) to (0.center);
	  \draw (4.center) to (6.center);
	  \draw (5.center) to (7.center);
	\end{pgfonlayer}
      \end{tikzpicture}
    \end{aligned}
    \quad 
    +
    \quad
    \begin{aligned}
      \begin{tikzpicture}
	\begin{pgfonlayer}{nodelayer}
	  \node [style=none] (0) at (-0.25, 0.375) {};
	  \node [style=none] (1) at (0.5, 0.375) {};
	  \node [style=none] (2) at (-0.25, -0.375) {};
	  \node [style=none] (3) at (0.5, -0.375) {};
	  \node [style=none] (4) at (0.5, 0.25) {};
	  \node [style=none] (5) at (0.5, -0.25) {};
	  \node [style=none] (6) at (1.25, 0.25) {};
	  \node [style=none] (7) at (1.25, -0.25) {};
	  \node [style=none] (8) at (0.125, -0) {$t$};
	  \node [style=none] (9) at (1.5, 0.25) {$Z$};
	  \node [style=none] (10) at (1.5, -0.25) {$W$};
	\end{pgfonlayer}
	\begin{pgfonlayer}{edgelayer}
	  \draw (0.center) to (1.center);
	  \draw (1.center) to (3.center);
	  \draw (3.center) to (2.center);
	  \draw (2.center) to (0.center);
	  \draw (4.center) to (6.center);
	  \draw (5.center) to (7.center);
	\end{pgfonlayer}
      \end{tikzpicture}
    \end{aligned}
    \quad = \quad 
    \begin{aligned}
      \begin{tikzpicture}
	\begin{pgfonlayer}{nodelayer}
	  \node [style=none] (0) at (-0.875, 0.875) {};
	  \node [style=none] (1) at (-0.125, 0.875) {};
	  \node [style=none] (2) at (-0.125, 0.125) {};
	  \node [style=none] (3) at (-0.875, 0.125) {};
	  \node [style=none] (4) at (-0.875, -0.125) {};
	  \node [style=none] (5) at (-0.125, -0.125) {};
	  \node [style=none] (6) at (-0.125, -0.875) {};
	  \node [style=none] (7) at (-0.875, -0.875) {};
	  \node [style=none] (8) at (-0.125, 0.75) {};
	  \node [style=none] (9) at (-0.125, 0.25) {};
	  \node [style=none] (10) at (-0.125, -0.25) {};
	  \node [style=none] (11) at (-0.125, -0.75) {};
	  \node [style=none] (12) at (1, 0.75) {};
	  \node [style=none] (13) at (1, 0.25) {};
	  \node [style=none] (14) at (1, -0.25) {};
	  \node [style=none] (15) at (1, -0.75) {};
	  \node [style=none] (16) at (-0.5, 0.5) {$s$};
	  \node [style=none] (17) at (-0.5, -0.5) {$t$};
	  \node [style=none] (18) at (1.25, 0.75) {$X$};
	  \node [style=none] (19) at (1.25, 0.25) {$Z$};
	  \node [style=none] (20) at (1.25, -0.25) {$Y$};
	  \node [style=none] (21) at (1.25, -0.75) {$W.$};
	\end{pgfonlayer}
	\begin{pgfonlayer}{edgelayer}
	  \draw (0.center) to (1.center);
	  \draw (1.center) to (2.center);
	  \draw (2.center) to (3.center);
	  \draw (3.center) to (0.center);
	  \draw (4.center) to (5.center);
	  \draw (5.center) to (6.center);
	  \draw (6.center) to (7.center);
	  \draw (7.center) to (4.center);
	  \draw (8.center) to (12.center);
	  \draw [in=180, out=0, looseness=1.00] (9.center) to (14.center);
	  \draw [in=180, out=0, looseness=1.00] (10.center) to (13.center);
	  \draw (11.center) to (15.center);
	\end{pgfonlayer}
      \end{tikzpicture}
    \end{aligned}
  \]

  Taking a hint from the compact closed structure, the isomorphism between
  $F\mathrm{Corel}$ and $\mc H_{\mathrm{str}}$ is then clear: the functors act as
  the identity on objects, and on morphisms take $f: X \to Y$ in $\mc
  H_{\mathrm{str}}$ to its name $\hat f: I \to X\cdot Y$ as a morphism of
  $F\mathrm{Corel}$, and vice versa. 
  
  It is straightforward to check that these are strict hypergraph functors. To
  demonstrate the most subtle aspect, the Frobenius structure, we consider the
  multiplication on an object $X$ of $F\mathrm{Corel}$. To obtain this, we start
  with the multiplication in $F\mathrm{Cospan}$, the cospan $X+X
  \xrightarrow{[1,1]} X \xleftarrow{1} X$ decorated with the empty decoration,
  which in this case is $\eta_X \in \mc H_{\mathrm{str}}(I,X)$, and restrict this empty
  decoration along the map $X \xleftarrow{[1,1,1]} X+X+X$. Thus the
  multiplication on $X$ in $F\mathrm{Corel}$ is 
  \[
    \tikzset{every path/.style={line width=1.1pt}}
    \begin{aligned}
      \begin{tikzpicture}[scale=.65]
	\begin{pgfonlayer}{nodelayer}
	  \node [style=none] (0) at (0.75, 0.5) {};
	  \node [style=dot] (1) at (0, -0) {};
	  \node [style=none] (2) at (0.75, -0.5) {};
	  \node [style=none] (3) at (1.5, 0.5) {$X$};
	  \node [style=none] (4) at (1.25, 0.5) {};
	  \node [style=none] (5) at (1.25, -0.5) {};
	  \node [style=dot] (6) at (-1.25, -0.5) {};
	  \node [style=none] (7) at (-0.5, -1) {};
	  \node [style=none] (8) at (1.25, -1) {};
	  \node [style=none] (9) at (-0.5, -0) {};
	  \node [style=dot] (10) at (-2, -0.5) {};
	  \node [style=none] (11) at (1.5, -0.5) {$X$};
	  \node [style=none] (12) at (1.5, -1) {$X$};
	\end{pgfonlayer}
	\begin{pgfonlayer}{edgelayer}
	  \draw [in=90, out=180, looseness=0.90] (0.center) to (1.center);
	  \draw [in=-90, out=180, looseness=0.90] (2.center) to (1.center);
	  \draw (4.center) to (0.center);
	  \draw (5.center) to (2.center);
	  \draw [in=90, out=180, looseness=0.90] (9.center) to (6);
	  \draw [in=-90, out=180, looseness=0.90] (7.center) to (6);
	  \draw (6) to (10.center);
	  \draw (8.center) to (7.center);
	  \draw (9.center) to (1);
	\end{pgfonlayer}
      \end{tikzpicture}
    \end{aligned}
    \in \mc H_{\mathrm{str}}(I,X \cdot X \cdot X),
  \]
  which corresponds under our isomorphism to the map $\mu_X\maps X \cdot X \to X
  \in \mc H_{\mathrm{str}}$, as required.
\end{proof}

Similarly, every hypergraph functor can be recovered using decorated
corelations. The key point here is again the compact closed structure: a
hypergraph functor $(T,\tau) \maps (\mc H,\ot) \to (\mc H',\boxtimes)$ takes
morphisms $f\maps X \to Y$ to morphisms $Tf\maps TX \to TY$, and hence induces a
function from the set of names $I \to X \ot Y$ to the set of names
$I' \to TX \boxtimes TY$.

\begin{theorem} \label{thm.hypdeccorfunctors}
  Every hypergraph functor can be constructed as a decorated corelation functor.
  
  More precisely, let $(\mc H,\ot)$, $(\mc H',\boxtimes)$ be hypergraph
  categories and $(T,\tau)\maps \mc H \to \mc H'$ be a hypergraph functor. By
  the previous theorem, there exist lax symmetric monoidal functors
  \[
    F\maps \cospan(\FinSet_{\mc O}) \to \Set \quad \mbox{and} \quad F'\maps
    \cospan(\FinSet_{\mc O'}) \to \Set
  \]
  such that we have isomorphisms $\mc H_{\mathrm{str}} \cong F\mathrm{Corel}$
  and $\mc H_{\mathrm{str}}' \cong F'\mathrm{Corel}$.
  
  There exists a colimit-preserving functor $A\maps \FinSet_{\mc O} \to
  \FinSet_{\mc O'}$ and a monoidal natural transformation
  \[
    \xymatrixrowsep{2ex}
    \xymatrix{
      \cospan(\FinSet_{\mc O}) \ar[dd]_{A} \ar[drr]^F  \\
      &\twocell \omit{_\:\theta}& \Set \\
      \cospan(\FinSet_{\mc O'}) \ar[urr]_{F'} 
    }
  \]
  such that the resulting decorated corelation functor makes the
  diagram of hypergraph functors
  \[
    \xymatrix{
      \mc H \ar[d]_\sim \ar[r]^{T} & \mc H' \ar[d]^\sim \\
      F\mathrm{Corel} \ar[r] & F'\mathrm{Corel}
    }
  \]
  commute.
\end{theorem}
\begin{proof}
  On objects, define the functor $A\maps \cospan(\FinSet_{\mc O}) \to
  \cospan(\FinSet_{\mc O'})$ to take each $\mc O$-labelled set $X=[x_1:\ldots:x_n]$,
  written here as a string of elements of $\mc O$, to the $\mc O'$-labelled set
  $TX=[Tx_1:\ldots:Tx_n]$ comprising the same set but changing each label $x_i$ to
  $Tx_i$. On morphisms, let it take a function between $\mc O$-labelled sets to
  the same function between the underlying sets, noting that the function now
  preserves the $\mc O'$-labels. This functor $A$ clearly preserves colimits:
  again, colimits in $\FinSet_{\mc O}$ are just colimits in $\FinSet$ with the
  labels inherited as described above.

  Let $X=[x_1:\ldots:x_n]$ be an $\mc O$-labelled set, and recall that we write
  $PX= (((x_1\ot x_2) \ot \dots )\ot x_n) \ot I$ for the monoidal product of
  this string in $\mc H$ with all open parentheses at the front. We define
  \begin{align*}
    \theta_X\maps \mc H_{\mathrm{str}}(I,X) = \mc H(I,PX) &\longrightarrow \mc
    H'_{\mathrm{str}}(I', TX)= \mc H'(I',PTX); \\
    \Big(I \stackrel{s}\to PX\Big) &\longmapsto \Big(I' \xrightarrow{\tau_I} TI
    \xrightarrow{Ts} TPX \xrightarrow{\tau_\ast} PTX\Big), 
  \end{align*}
  where $\tau_\ast$ is the appropriate composite of coherence isomorphisms of
  $T$. This collection of maps $\theta$ is natural because (i) functions in
  $\FinSet_{\mc O}$ act as the Frobenius maps on the homsets $\mc H(I,X)$ and
  $\mc H'(I',TX)$, and (ii) the functor $T$ is hypergraph and thus sends,
  loosely speaking, Frobenius maps to Frobenius maps. Moreover $\theta$ is
  monoidal as $T$ is monoidal.  Thus $\theta$ defines a monoidal natural
  transformation.

  The hypergraph functor $F\mathrm{Corel} \to F'\mathrm{Corel}$ induced by
  $\theta$ is by definition the map taking the decoration $s \in \mc
  H(I,P(X\cdot Y))$ to the decoration $\tau_\ast \circ Ts \circ
  \tau_I \in \mc H'(I',PT(X \cdot Y))$. It is immediate that the above square
  commutes. We thus say that $T$ can be constructed as a decorated corelation
  functor.
\end{proof}

\subsection{A categorical equivalence}

Choose Grothendieck universes so that we may talk about the category $\Set$ now
only of all so-called small sets, and the category $\mathrm{Cat}$ of all small
categories. If we restrict our attention to small hypergraph categories, we may
summarise these results as a categorical equivalence. 

Indeed, given some fixed object set $\mc O$, we have a category of lax symmetric
monoidal functors and monoidal natural transformations
\[
  \mathrm{LaxSymMon}(\cospan(\FinSet_{\mc O}),\Set).
\]
This is equivalent to some subcategory of the category $\mathrm{HypCat}$ of
hypergraph categories. To get a category equivalent to $\mathrm{HypCat}$, we
must patch together these functor categories by varying the object set and
specifying morphisms from objects in one category to objects in another. For
this we use the Grothendieck construction. 

Given a (contravariant) functor $S\maps \mc B \to \mathrm{Cat}$ from some
category $\mc B$ to $\mathrm{Cat}$, we define the \define{(contravariant)
Grothendieck construction} $\mc B \int S$ to be the category with pairs
$(\mc O,F)$ where $\mc O$ is an object of $\mc B$ and $F$ is an object of $S\mc
O$ as objects, and pairs 
\[
  (f,\theta)\maps (\mc O,F) \longrightarrow (\mc O',F')
\]
where $f\maps \mc O \to \mc O'$ is a morphism in $\mc B$ and $\theta\maps F \to
Sf(F')$ is a morphism in $S\mc O$ as morphisms. Given another morphism
$(g,\zeta)\maps (\mc O', F') \to (\mc O'',F'')$, the composite morphism in the
pair $(g \circ f, Sf\zeta \circ \theta)$.

Now, define the functor 
\[
  \mathrm{LaxSymMon}(\cospan(\FinSet_{-}), \Set)\maps \Set \longrightarrow
  \mathrm{Cat}
\]
as follows. On objects let it map a set $\mc O$ to the lax symmetric monoidal
functor category $\mathrm{LaxSymMon}(\cospan(\FinSet_{\mc O}), \Set)$. For
morphisms, we noted at the beginning of the proof of Theorem
\ref{thm.hypdeccorfunctors} that a function $r\maps \mc O \to \mc O'$ from one
set of labels to another induces a functor $R\maps \cospan(\FinSet_{\mc O}) \to
\cospan(\FinSet_{\mc O'})$ taking each $\mc O$-labelled set $N \stackrel{l}\to
\mc O$ to the $\mc O'$-labelled set $N \xrightarrow{r \circ l} \mc O'$. This in
turn defines a functor 
\begin{align*}
 \mathrm{LaxSymMon}(\cospan(\FinSet_{\mc O'}), \Set) &\longrightarrow
\mathrm{LaxSymMon}(\cospan(\FinSet_{\mc O}), \Set); \\
\Big(F'\maps
\cospan(\FinSet_{\mc O'}) \to \Set\Big) &\longmapsto \Big(R\circ F'\maps  \cospan(\FinSet_{\mc O})
\to \Set\Big).
\end{align*}
by post-composition with $R$. Note that
$\mathrm{LaxSymMon}(\cospan(\FinSet_{-}), \Set)$ is a well-defined functor: in
particular, it preserves composition on the nose.

The Grothendieck construction 
\[
 \Set\int \mathrm{LaxSymMon}(\cospan(\FinSet_{-}), \Set)
\]
thus gives a category where the objects are some label set $\mc O$ together with
an object in $\mathrm{LaxSymMon}(\cospan(\FinSet_{\mc O}), \Set)$---that is,
with a lax symmetric monoidal functor $F\maps (\cospan(\FinSet_{\mc O}),+) \to
(\Set,\times)$. The morphisms $(\mc O, F) \to (\mc O,F')$ are functions of label
sets $r: \mc O \to \mc O'$ together with a natural transformation $\theta\maps
R \circ F'\Rightarrow F$.

In the case of small categories, it is not difficult to show that Theorems
\ref{thm.hypdeccorcats} and \ref{thm.hypdeccorfunctors} imply:

\begin{theorem} \label{thm.equivhypdccor}
  There is an equivalence of categories
\[
  \mathrm{HypCat} \cong \Set\int \mathrm{LaxSymMon}(\cospan(\FinSet_{-}), \Set).
\]
\end{theorem}

The 2-categorical version of the above equivalence is also the subject of a
forthcoming paper by Vagner, Spivak, and Schultz, from their operadic
perspective \cite{VSS}.

\begin{remark}
  It is prudent, to wonder, in the definition of hypergraph category, why it is
  useful to define a structure---the Frobenius maps---on a category that does
  not interact with the morphisms of the category. This affects, for example,
  invariance under equivalence of symmetric monoidal categories: given two
  equivalent symmetric monoidal categories, the hypergraph structures on these
  categories need not be in one-to-one correspondence. The Frobenius maps,
  however, are seen by the monoidal product and by hypergraph functors. Theorem
  \ref{thm.equivhypdccor} provides a more invariant definition of the category
  of hypergraph categories.
\end{remark}

\subsection{Factorisations as decorations}

We have seen that every hypergraph category can be constructed as a decorated
corelations category. More precisely, we have seen that every hypergraph
category can be constructed as a decorated corelation category with the
factorisation system the trivial isomorphism-morphism factorisation system. But
we can also use other factorisation systems to construct decorated corelation
categories, and these are also hypergraph categories. This implies that we might
have multiple decorated corelation constructions for the same hypergraph
category.  How do these constructions relate to each other?

Recall from Remark \ref{rem.corelposet} that for any category $\mc C$ with
finite colimits there exists a poset of factorisation systems $(\mc E,\mc M)$
with $\mc M$ stable under pushout, where the order is given by reverse inclusion
on $\mc M$. Suppose we choose a factorisation system $(\mc E,\mc M)$ and have a
lax symmetric monoidal functor
\[
  F\maps \mc C;\mc M^\opp \longrightarrow \Set.
\]
This defines a decorated corelation category. Although we will not prove it
here, it is in fact possible to construct a hypergraph equivalent decorated
corelation category using any factorisation system $(\mc E',\mc M')$ less than
$(\mc E,\mc M)$. Roughly, the idea is that although moving to a smaller
factorisation system results in fewer corelations, we may store any lost
information in the decorations by altering the functor $F$ in the right way.
The proof of Theorem \ref{thm.hypdeccorcats} shows how to do this for the
factorisation system at the bottom of this poset, the isomorphism-morphism
factorisation system $(\mc I_{\mc C},\mc M)$.

For illustration, we example this interaction for the simplest hypergraph
category: $\cospan(\FinSet)$, the free hypergraph category on a single object.

\begin{example}
  As per Example \ref{ex.undeccorel}, $\cospan(\FinSet)$ is the hypergraph
  category of undecorated morphism-isomorphism corelations in $\FinSet$.
  Theorem \ref{thm.hypdeccorcats} shows it is also the hypergraph category of
  function-decorated isomorphism-morphism corelations in $\FinSet$.   
  
  The construction in Theorem \ref{thm.hypdeccorcats} proceeds as follows.
  First, we construct a lax symmetric monoidal functor $F\maps \cospan(\FinSet)
  \to \Set$. On objects $F$ takes each finite set $X$ to the set of isomorphism
  classes of cospans $\varnothing \to D \leftarrow X$ from the monoidal unit
  $\varnothing$ to $X$ or, equivalently, the set $FX$ of functions $s\maps X \to
  D$, where a unique codomain $D$ is chosen for each finite cardinality. On
  morphisms, $F$ takes a cospan $X \stackrel{f}\to N \stackrel{g}\leftarrow Y$
  to the function $FX \to FY$ that maps $s\maps X \to D$ in $FX$ to the function
  $Y \to N+_XD$ in $FY$ given by
  \[
    \xymatrix{
      & X \ar[r]^s \ar[d]_f & D \ar[d] \\
      Y \ar[r]^g  & N \ar[r] & N+_XD
    }
  \]
  where the square is a pushout square. For coherence maps, the map
  $\varphi_1\maps 1 \to F\varnothing$ takes the unique element of $1$ to the
  unique function $!\maps \varnothing \to \varnothing$, while $\varphi_{X,Y}$
  takes a pair of functions $a\maps X \to D$, $b\maps Y \to E$ to $a+b\maps X+Y
  \to D+E$. This defines a lax symmetric monoidal functor $(F,\varphi)$.

  A decorated \emph{cospan} in $\FinSet$ with respect to $(F,\varphi)$ is then a
  cospan of finite sets $X \to N \leftarrow Y$ together with a function of
  finite sets $N \to D$. Using the isomorphism-morphism factorisation, a
  decorated \emph{corelation} is thus a cospan $X \xrightarrow{\iota_X} X+Y
  \xleftarrow{\iota_Y} Y$ together with a function $X+Y \to D$. As there is a
  unique isomorphism class of $\mc I_{\FinSet}$-like cospans $X
  \xrightarrow{\iota_X} X+Y \xleftarrow{\iota_Y} Y$, a decorated corelation is
  thus specified by its decoration $X+Y \to D$ alone. Note that
  maps $X+Y \to D$ are in one-to-one correspondence with cospans $X \to D
  \leftarrow Y$ in $\FinSet$ via the coproduct inclusion maps.

  As observed in the proof of Theorem \ref{thm.hypdeccorcats}, the hypergraph
  structure on $F\mathrm{Corel}$ agrees with that on $\cospan(\FinSet)$ via
  this correspondence; the multiplication \linebreak $\mu_X\maps X+X \to X$ in
  $F\mathrm{Corel}$ is simply given by the decoration $(X+X)+X \to X$, and so
  on. The intuition is that $F$ takes the `factored out part' of the corelation
  and puts it into the decoration. 

  This correspondence between maps $X+Y \to D$ and cospans $X \to D\leftarrow Y$
  suggests an isomorphism. Indeed, the equivalence given by Theorem
  \ref{thm.hypdeccorcats} is precisely this. We can construct one direction, the
  one from the smaller to larger $\mc M$---that is, from $\mc M = \mc
  I_{\FinSet}$ to $\mc M = \FinSet$---as a decorated corelations functor.  Note
  that the identity on $\FinSet$ maps the subcategory $\mc I_{\FinSet}$ into
  $\FinSet$, and so extends to a morphism $\FinSet \to \cospan(\FinSet)$. Also
  recall (from say Example \ref{ex.undeccorel}) that the `undecorated' cospan
  category $\cospan(\FinSet)$ is equal to the decorated cospan category given by
  the functor $1\maps \FinSet \to \Set$ mapping each finite set to some chosen
  one element set $1$, and each morphism to the identity morphism on $1$. Define
  monoidal natural transformation 
  \[
    \xymatrixrowsep{2ex}
    \xymatrix{
      \FinSet = \FinSet; \mc I_{\FinSet}^\opp \ar[dd]_{\iota} \ar[drr]^(.65)1  \\
      &\twocell \omit{_\:\theta}& \Set \\
      \cospan(\FinSet) = \FinSet; \FinSet^\opp \ar[urr]_(.65){G} 
    }
  \]
  with each $\theta_X \maps 1 \to GX$ mapping the unique element to the identity
  function $1_X\maps X \to X$.  This gives the hypergraph functor we expect,
  mapping the undecorated cospan $X \to N \leftarrow Y$ to the trivial cospan $X
  \to X+Y \leftarrow Y$ decorated by $X+Y \to N$. It is now routine to verify
  this is an isomorphism.
\end{example}

The previous example extends to any category $\mc C$ with finite colimits: the
hypergraph category $\cospan(\mc C)$ can always be constructed as (i)
trivially decorated $(\mc C, \mc I_{\mc C})$-corelations, or (ii) $(\mc I_{\mc
C}, \mc C)$-corelations decorated by equivalence classes of morphisms with
domain the apex of the corelation. Moreover, the isomorphism of these
hypergraph categories is given by the analogous monoidal natural transformation
between the decorating functors.

More general still, a category of trivially decorated $(\mc E, \mc
M)$-corelations in $\mc C$ can always be constructed also as $(\mc I_{\mc C},
\mc C)$-corelations decorated by equivalence classes of morphisms in $\mc E$
with domain the apex of the corelation, and the isomorphism of these
hypergraph categories a decorated corelations functor.

Most generally, we can still perform this construction on decorated corelation
categories: Theorem \ref{thm.hypdeccorcats} implies any category of $(\mc E,\mc
M)$-decorated corelations can be constructed also as $(\mc I_{\mc C}, \mc
C)$-corelations decorated by codomain decorated morphisms in $\mc E$. Given some
lax symmetric monoidal functor
\[
  F\maps \enspace \mc C;\mc M^\opp \longrightarrow \Set,
\]
the decorated $(\mc I_{\mc C}, \mc C)$-corelations are specified by the functor
\[
  F'\maps \cospan(\mc C) \longrightarrow \Set
\]
taking any object $Z \in \cospan(\mc C)$ to the pair $(Z \stackrel{e}\to N,
\enspace 1\stackrel{s}\to FN)$, where $e$ is a morphism in $\mc E$.

What is the utility of this variety of constructions? Different constructions
suit different purposes. At the top of the poset, consider the
morphism-isomorphism factorisation system: these give decorated cospan
categories. Decorated cospan categories provide, as we saw for open electric
circuits, an intuitive way to construct a `syntactic' hypergraph category from
some notion of network-style diagrammatic language. Moreover, they are decorated
corelation categories with $\mc M = \mc I_{\mc C}$ as small as possible. This
makes it easy to use decorated corelations to construct hypergraph
functors---such as those describing the semantics of diagrams---from a decorated
cospan category to another category.

To do this, however, we need to construct hypergraph categories of semantics.
Here epi-mono corelations are a useful tool. Indeed, some hypergraph categories
are very naturally constructed as corelation (or, dually, relation) categories,
such as the categories of equivalence relations, relations, or linear relations.
Here the factorisation system has an intuitive interpretation, such as the
epimorphisms retaining only the structure in the apex that is `accessible' or
`mapped onto' by the feet/boundary. Decorating these corelations retains the
same sort of intuition. 

Finally, at the bottom of the poset, we have the isomorphism-morphism
factorisation system. These give rise to the decorated corelation categories
constructed in Theorem \ref{thm.hypdeccorcats}, in which the factorisation
system is the trivial isomorphism-morphism one and all interesting structure is
carried by the apex of the decorated corelation. Note that these have $\mc M =
\mc C$ as large as possible, and also large sets of decorations. This makes it
easy to construct functors into a given hypergraph category. Thus we might first
define a semantic category as an epi-mono corelation category to match our
intuitions, and then later use an equivalent definition as a
isomorphism-morphism corelation category to facilitate construction of semantic
functors.

Indeed, the ability to construct functors from one hypergraph category to
another is essential for the understanding of hypergraph categories as
network-type diagrammatic languages, with functors not just giving rise to
notions of semantics, but hence also notions of equivalence of diagrams, and
reasoning tools.

We will illustrate these principles in greater depth in Part \ref{part.apps},
which delves into applications of this philosophy and framework. We conclude
this section and Part \ref{part.maths} with two examples of a more abstract
nature.

\section{Examples} \label{sec:excor}
We give two extended examples. Our first example is to revisit the matrix
example from the introduction, having now developed the necessary material. Our
second example is to revisit the category of linear relations once more, showing
that we can also construct it as a decorated corelation category. 

\subsection{Matrices} \label{ssec.matrices}

Let $R$ be a commutative rig.\footnote{Also known as a semiring, a rig is a ring
  without \emph{n}egatives.} In this subsection we will
construct matrices over $R$ as decorated corelations over $\FinSet^\opp$. 

In $\FinSet^\opp$ the coproduct is the Cartesian product $\times$ of sets, the
initial object is the one element set $1$, and cospans are spans in $\FinSet$.
The notation will thus be less confusing if we talk of decorated spans on
$(\FinSet,\times)$ given by the contravariant lax monoidal functor
\begin{align*}
  R^{(-)}: (\mathrm{FinSet},\times) &\longrightarrow (\Set,\times); \\
  N &\longmapsto R^N \\
  \Big(f\maps N \to M\Big) &\longmapsto \Big(R^f\maps R^M \to R^N; v \mapsto v \circ
  f\big).
\end{align*}
The coherence maps $\varphi_{N,M}\maps R^N \times R^M \to R^{N\times M}$ take a
pair $(s,t)$ of maps $s\maps N \to R$, $t\maps M \to R$ to the pointwise product
$s\cdot t \maps N\times M \to R; (n,m) \mapsto s(n) \cdot t(m)$. The unit
coherence map $\varphi_1\maps 1 \to R^1$ sounds almost tautological: it takes
the unique element of the one element set $1$ to the function $1 \to R$ that
maps the unique element of the one element set to the multiplicative identity
$1_R$ of the rig $R$.  As described in Section \ref{sec.blackboxedsystems},
$R^{(-)}\mathrm{Cospan}$ can be considered as the category of `multivalued
matrices' over $R$, and $R^{(-)}\mathrm{Corel}$ the category of matrices over
$R$.

Just as the coherence map $\varphi_1$ gives the unit for the multiplication, it
is the coherence maps $\varphi_{N,M}$ that enact multiplication of scalars:
the composite of decorated spans $(X \xleftarrow{i_X} N \xrightarrow{o_Y} Y,\enspace N
\xrightarrow{s} R)$ and $(Y \xleftarrow{i_Y} M \xrightarrow{o_Z} Z,\enspace M
\xrightarrow{t} R)$ is the span $X \leftarrow N\times_YM \rightarrow Z$ decorated
by the map
\[
  N \times_YM \hooklongrightarrow N \times M \xrightarrow{\varphi_{N,M}(s,t) = s \cdot t}
  R,
\]
where the inclusion from $N \times_YM$ into $N \times M$ is that given by the
categorical product. We discussed the intuition for this composition rule, in
terms of paths between elements of $X$ and those of $Z$, in Section
\ref{sec.blackboxedsystems}. 

As $\varphi_1$ selects the multiplicative unit $1_R$ of $R$, the empty
decoration on any set $N$ is the function that sends every element of $N$ to
$1_R$. This implies the identity decorated span on $X = \{x_1,\dots, x_n\}$ is
that represented by the diagram
\[
    \tikzset{every path/.style={line width=.8pt}}
\begin{tikzpicture}
	\begin{pgfonlayer}{nodelayer}
		\node [style=sdot] (0) at (1.75, 0.5) {};
		\node [style=sdot] (1) at (-1.75, -1) {};
		\node [style=sdot] (2) at (1.75, -1) {};
		\node [style=none] (3) at (0, -0.25) {$\vdots$};
		\node [style=amp] (4) at (0, -1) {$1$};
		\node [style=amp] (5) at (0, 1) {$1$};
		\node [style=sdot] (6) at (1.75, 1) {};
		\node [style=amp] (7) at (0, 0.5) {$1$};
		\node [style=sdot] (8) at (-1.75, 0.5) {};
		\node [style=sdot] (9) at (-1.75, 1) {};
		\node [style=none] (10) at (1.75, -0.25) {$\vdots$};
		\node [style=none] (11) at (-1.75, -0.25) {$\vdots$};
		\node [style=none] (12) at (-2.125, 1) {$x_1$};
		\node [style=none] (13) at (-2.125, 0.5) {$x_2$};
		\node [style=none] (14) at (-2.125, -1) {$x_n$};
		\node [style=none] (15) at (2.125, 1) {$x_1$};
		\node [style=none] (16) at (2.125, 0.5) {$x_2$};
		\node [style=none] (17) at (2.125, -1) {$x_n$};
	\end{pgfonlayer}
	\begin{pgfonlayer}{edgelayer}
		\draw (6) to (5);
		\draw (4) to (2);
		\draw (1) to (4);
		\draw (5) to (9);
		\draw (8) to (7);
		\draw (7) to (0);
	\end{pgfonlayer}
\end{tikzpicture}
\]
while the Frobenius multiplication and unit are 
\[
    \tikzset{every path/.style={line width=.8pt}}
  \begin{aligned}
\begin{tikzpicture}
	\begin{pgfonlayer}{nodelayer}
		\node [style=sdot] (0) at (1.75, -0.5) {};
		\node [style=sdot] (1) at (1.75, -2.25) {};
		\node [style=none] (2) at (0, -1.25) {$\vdots$};
		\node [style=amp] (3) at (0, -2.25) {$1$};
		\node [style=amp] (4) at (0, 1.5) {$1$};
		\node [style=sdot] (5) at (1.75, 1.5) {};
		\node [style=amp] (6) at (0, -0.5) {$1$};
		\node [style=sdot] (7) at (-1.75, -0.5) {};
		\node [style=none] (8) at (1.75, -1.25) {$\vdots$};
		\node [style=none] (9) at (-1.75, -1.25) {$\vdots$};
		\node [style=sdot] (10) at (-1.75, -2.25) {};
		\node [style=sdot] (11) at (-1.75, 1.5) {};
		\node [style=none] (12) at (2.125, 1.5) {$x_1$};
		\node [style=none] (13) at (2.125, -0.5) {$x_2$};
		\node [style=none] (14) at (2.125, -2.25) {$x_n$};
		\node [style=none] (15) at (-2.5, 1.5) {$(x_1,x_1)$};
		\node [style=none] (16) at (-2.5, -0.5) {$(x_2,x_2)$};
		\node [style=none] (17) at (-2.5, -2.25) {$(x_n,x_n)$};
		\node [style=none] (18) at (-2.5, 1) {$(x_1,x_2)$};
		\node [style=sdot] (19) at (-1.75, 1) {};
		\node [style=none] (21) at (-1.75, 0.5) {$\vdots$};
		\node [style=none] (22) at (1.75, 0.75) {$\vdots$};
		\node [style=sdot] (23) at (-1.75, -0) {};
		\node [style=none] (24) at (-2.5, -0) {$(x_2,x_1)$};
		\node [style=sdot] (25) at (-1.75, -1.75) {};
		\node [style=none] (26) at (-2.7, -1.75) {$(x_n,x_{n-1})$};
	\end{pgfonlayer}
	\begin{pgfonlayer}{edgelayer}
		\draw (5) to (4);
		\draw (3) to (1);
		\draw (7) to (6);
		\draw (6) to (0);
		\draw (10) to (3);
		\draw (4) to (11);
	\end{pgfonlayer}
\end{tikzpicture}
  \end{aligned}
  \qquad \mbox{and} \qquad
  \begin{aligned}
\begin{tikzpicture}
	\begin{pgfonlayer}{nodelayer}
		\node [style=sdot] (0) at (1.75, 0.5) {};
		\node [style=sdot] (1) at (1.75, -1) {};
		\node [style=none] (2) at (0, -0.25) {$\vdots$};
		\node [style=amp] (3) at (0, -1) {$1$};
		\node [style=amp] (4) at (0, 1) {$1$};
		\node [style=sdot] (5) at (1.75, 1) {};
		\node [style=amp] (6) at (0, 0.5) {$1$};
		\node [style=sdot] (7) at (-1.75, -0) {};
		\node [style=none] (8) at (1.75, -0.25) {$\vdots$};
		\node [style=none] (9) at (2.125, 1) {$x_1$};
		\node [style=none] (10) at (2.125, 0.5) {$x_2$};
		\node [style=none] (11) at (2.125, -1) {$x_n$};
	\end{pgfonlayer}
	\begin{pgfonlayer}{edgelayer}
		\draw (5) to (4);
		\draw (3) to (1);
		\draw (7) to (6);
		\draw (6) to (0);
		\draw (7) to (4);
		\draw (7) to (3);
	\end{pgfonlayer}
\end{tikzpicture}
  \end{aligned}
\]
respectively, with the comultiplication and counit the mirror images.

These morphisms are multivalued matrices in the following sense: the
cardinalities of the domain $X$ and the codomain $Y$ give the dimensions of the
matrix, and the apex $N$ indexes its entries. If $n \in N$ maps to $x \in X$ and
$y \in Y$, we say there is an entry of value $s(n) \in \R$ in the $x$th row and
$y$th column of the matrix. It is multivalued in the sense that there may be
multiple entries in any position $(x,y)$ of the matrix.

To construct matrices proper, and not just multivalued matrices, as decorated
relations, we extend $R^{(-)}$ to the contravariant functor
\[
  R^{(-)}: (\mathrm{Span(FinSet)},\times) \longrightarrow (\Set,\times)
\]
mapping now a span $N \stackrel{f}\leftarrow A \stackrel{g}\to M$ to the
function
\begin{align*}
  R^{f^\opp;g}\maps R^M &\longrightarrow R^N; \\
  v &\longmapsto \Big(n \mapsto \sum_{a \in f^{-1}(n)} v\circ g(a)\Big).
\end{align*}
It is simply a matter of computation to check this is functorial.

Decorated corelations in this category then comprise trivial spans $X
\xleftarrow{\pi_X} X \times Y \xrightarrow{\pi_Y} Y$, where $\pi$ is the
projection given by the categorical product, together with a decoration $X\times
Y \to R$. Such morphisms give a value of $R$ for each pair $(x,y) \in X \times
Y$, and thus are trivially in one-to-one correspondence with $\lvert X \rvert
\times \lvert Y\rvert$-matrices. 

The map $R^{(-)}\mathrm{Cospan} \to R^{(-)}\mathrm{Corel}$ transports the
decoration $N\times_YM \to R$ along the function $N \times_YM \to N \times M$
that identifies elements over the same pair $(x,y)$. In terms of the multivalued
matrices, this sums over (the potentially empty) set of entries over $(x,y)$ to
create a single entry. It is thus easily observed that composition in this
category is matrix multiplication. Moreover, it is not difficult to check that
the monoidal product is the Kroenecker product of matrices, and thus that
$R^{(-)}\mathrm{Corel}$ is monoidally equivalent to the monoidal category of
$(\FinVect, \otimes)$ of finite dimensional vector spaces, linear maps, and the
tensor product.

%Many aspects of this example are `atypical' for the intuition we have been
%working towards.  Note that the monoidal product here is the tensor product of
%matrices, not the biproduct. Indeed, there is no special commutative Frobenius
%algebra in $\Vect$ if we use the biproduct, but if we use the tensor product
%then these correspond to orthonormal bases (Vicary), as they do here.  

Note that $R^X$ is always an $R$-module, and $R^f$ a homomorphism of
$R$-modules. Thus we could take decorations here in the category $R\mathrm{Mod}$
of $R$-modules, rather than the category $\Set$. While Proposition
\ref{prop.setdecorations} shows that the resulting decorated cospan category
would be isomorphic, this may hint at an enriched version of the theory.

\subsection{Two constructions for linear relations}

We saw in \textsection\ref{ssec.linrel} that linear relations are epi-mono
corelations in $\Vect$. As linear relations thus form a hypergraph category
$\LinRel$, we can also give a decorated corelations construction. 

Indeed, Theorem \ref{thm.hypdeccorcats} shows that from the hom functor
$\LinRel(0,-)$ on the monoidal unit of the hypergraph category $\LinRel$, we can
build the functor
\[
  \mathrm{Lin} \maps\mathrm{Cospan}(\mathrm{FinSet}) \longrightarrow \mathrm{Set}
\]
taking a finite set $N$ to the set $\mathrm{Lin}(N)$ of linear subspaces of the
vector space $k^N$. Moreover, the image $\mathrm{Lin}(f)$ of a function $f\maps
N \to M$ maps a subspace $L \subseteq k^N$ to $\{v \mid v\circ f \in L\}
\subseteq k^M$, while the image $\mathrm{Lin}(f^\opp)$ of an opposite function
$g^\opp: N \to M$ maps a subspace $L \subseteq k^N$ to $\{v = u \circ g \mid u
\in L\} \subseteq k^M$. 

Given this functor, it can be shown that $\mathrm{LinCospan}$ is the category of
cospans decorated by subspaces, while $\mathrm{LinCorel}$ is the category of
linear relations. 

This pair of constructions is important for circuits work \cite{BF,BSZ}. Recall
the functor $\linsub\maps \FinSet \to \Set$ of Section \ref{sec:ex}, which takes
a finite set $X$ to the set of subspaces of $\R^X \oplus (\R^X)^*$.  This may be
extended to a functor $\cospan(\FinSet) \to \Set$, giving a decorated
corelations category of linear relations where the objects are the direct sum of
a vector space and its dual. Recall also that in Section \ref{sec:ex} we
discussed a functor 
\[
  \mathrm{GraphCospan} \to \mathrm{LinSubCospan}
\]
interpreting labelled graphs as linear subspaces. Composing this with the
quotient functor $\mathrm{LinSubCospan} \to \mathrm{LinSubCorel}$, this gives a
compositional linear relations semantics for circuit diagrams. This hints at the
material we explore in depth in Chapter \ref{ch.circuits}.

\part{Applications} \label{part.apps}

\chapter{Signal flow diagrams} \label{ch.sigflow}
In this chapter we use corelations to guide the development of a graphical
language for reasoning about linear time-invariant systems. 

%The work builds on recent of work Bonchi, Soboci\'nski, and Zanasi, in which
%they develop a sound and complete graphical language for reasoning about linear
%relations \cite{BSZ1}. In their work they provide a presentation of the category
%of linear relations over the field of Laurent polynomials $k(s)$ in some formal
%variable $s$ over some field $k$, and operational semantics In the previous part, we observed that this
%category of linear relations can be constructed as epi-mono corelations over the
%category $\Mat k(s)$ of matrices with entries in $k(s)$.  
%
%In the work of Willems, however, it is matrices over the ring of polynomials $k[s,s^{-1}]$ in $s$ and its formal inverse
%$s^{-1}$ that plays the central role in representing 
%

We begin in the next section with motivation and an overview of this
chapter. In \S\ref{sec.systems} we then develop a categorical account of
complete LTI discrete dynamical systems. This serves as a denotational semantics
for the graphical language, introduced in \S\ref{sec.diagrams}, where we also
derive the equational characterisation.  In \textsection\ref{sec.opsem} we
relate this to an operational semantics, in terms of biinfinite streams of
elements of $k$. We conclude in \S\ref{sec.control} with a structural account of
controllability.

\section{Behavioural control theory}

Control theory begins with the following picture:		
\[
\begin{tikzpicture}
\node(system) [shape=rectangle,draw,inner sep=20pt] at (0,0) {\textsf{system}};
\node(input) at (-3.5,0) {\small{\textsf{input}}};
\node(output) at (3.5,0) {\small{\textsf{output}}};
\draw [->,shorten >=5pt,>=stealth,very thick] (input) 	-- (system);
\draw [->,shorten <=5pt,>=stealth,very thick] (system)	-- (output);
\end{tikzpicture}
\]
In this picture we have an object under study, referred to as a \emph{system}, which when fed certain inputs produces certain outputs. These outputs need not be uniquely determined by the inputs; in general the relationship may be stochastic or non-deterministic, or depend also on internal states of the system. Although we take it as given that the system is an \emph{open} system---so it interacts with its environment, and so we can observe its inputs and outputs---we assume no access to the details of these internal states or the inner workings of the system, and this can add considerable complexity to our models. The end goal of control theory is then to control the system: to understand how to influence its behaviour in order to achieve some desired goal. Two key questions arise: 
\begin{itemize}
\item \textbf{Analysis}: Given a system, what is the relationship it induces between input and output?
\item \textbf{Synthesis}: Given a target relationship between input and output, how do we build a system that produces this relationship?
\end{itemize}
The first, the question of system analysis, provides the basic understanding
required to find the inputs that lead to the desired outputs. The second
question, the question of synthesis, is the central question of feedback control
theory, which aims to design controllers to regulate other systems. We give a
brief overview of the field, illustrated with some questions of these kinds.

Control has a long history. Indeed, many systems found in the biology and chemistry of living organisms have interesting interpretations from a control theoretic viewpoint, such as those that are responsible for body temperature regulation or bipedal balance \cite{So2}. Human understanding of control developed alongside engineering, and so dates back at least as far as antiquity, with for example the ancient Romans devising elaborate systems to maintain desired water levels in aqueducts \cite{So}. The origin of formal mathematical control theory, however, is more recent, and in general taken to be James Clerk Maxwell's seminal analysis of centrifugal governors, presented to the Royal Society of London in 1868 \cite{CM}. 

The techniques used and developed from Maxwell's paper, in particular by Rayleigh and Heaviside, are in general known as \emph{transfer function} techniques. A transfer function is a linear map from the set of inputs to the set of outputs of a linear time-invariant system with zero initial conditions. Transfer functions are commonly used in the analysis of single-input single-output linear systems, but become unwieldy or inapplicable for more general systems.

In the 1960s, led by Wiener and Kalman, so-called \emph{state space} techniques were developed to address multiple-input multiple-output, time-varying, nonlinear systems \cite{Fr}. These methods are characterised by defining a system as a collection of input, output, and internal state variables, with the state changing as a function of the input and state variables over time, and the output a function of the input and state. These functions are often implicitly specified by differential equations.

In general, however, classical control theory remains grounded in a paradigm that defines and analyses systems in terms of inputs and outputs or, from another perspective, causes and effects. In recent years Willems, among others, has argued that this input-output perspective is limiting, as in the case of many systems studied through control theory there is no clear distinction between input and output variables \cite{Wi}. For example, given a circuit component for which the relevant variables are the voltages and currents, different contexts may call for the voltage to be viewed as the input and current output, or vice versa. It is useful to have a single framework capable of discussing the behaviour of the component without making this choice.

Moreover, in the drive to understand larger and more complex systems, increasing
emphasis has been put on understanding the way systems can be broken down into
composite subsystems, and conversely how systems link together to form larger
systems \cite{Wi2, KT}. Indeed, interconnection or composition of systems has
always played a central role in systems engineering, and the
difficulty of discussing how systems compose within an input-output framework
lends support to Willems' call for a more nuanced definition of control system.
To illustrate these difficulties, consider the simple example, due to Willems,
of two water tanks each with two access pipes:
\[
\begin{tikzpicture}
    % draw the tanks
    \begin{pgfonlayer}{main}
    \draw (-4,.1)--(-3,.1)--(-3,0)--(-1,0)--(-1,.1)--(0,.1);
    \draw (-4,.3)--(-3,.3)--(-3,2)--(-1,2)--(-1,.3)--(0,.3);
    \node at (-3,.3) [anchor=-55]{\scriptsize$\begin{array}{c} p_{A1} \\ f_{A1} \end{array}$};
    \node at (-1,.3) [anchor=-125]{\scriptsize$\begin{array}{c} p_{A2} \\ f_{A2} \end{array}$};
    \node at (-2,1){Tank $A$};
    \end{pgfonlayer}
    % fill them with water (in the background)
    \begin{pgfonlayer}{background}
        \filldraw[blue!20] (-4,.1)--(-3,.1)--(-3,0)--(-1,0)--(-1,.1)
        --(0,.1)--(0,.3)--(-1,.3)--(-1,1.8)--(-3,1.8)--(-3,.3)--(-4,.3)
        --cycle;
    \end{pgfonlayer}
\end{tikzpicture}
\qquad
\begin{tikzpicture}
    % draw the tanks
  \begin{pgfonlayer}{main}
    \draw (-4,.1)--(-3,.1)--(-3,0)--(-1,0)--(-1,.1)--(0,.1);
    \draw (-4,.3)--(-3,.3)--(-3,1.7)--(-1,1.7)--(-1,.3)--(0,.3);
    \node at (-3,.3) [anchor=-55]{\scriptsize$\begin{array}{c} p_{B1} \\ f_{B1} \end{array}$};
    \node at (-1,.3) [anchor=-125]{\scriptsize$\begin{array}{c} p_{B2} \\ f_{B2} \end{array}$};
    \node at (-2,1){Tank $B$};
    \end{pgfonlayer}
    % fill them with water (in the background)
    \begin{pgfonlayer}{background}
        \filldraw[blue!20] (-4,.1)--(-3,.1)--(-3,0)--(-1,0)--(-1,.1)
        --(0,.1)--(0,.3)--(-1,.3)--(-1,1.5)--(-3,1.5)--(-3,.3)--(-4,.3)
        --cycle;
    \end{pgfonlayer}
\end{tikzpicture}
\]
For each tank, the relevant variables are the pressure $p$ and the flow $f$. A typical control theoretic analysis might view the water pressure as inducing flow through the tank, and hence take pressure as the input variable, flow as the output variable, and describe each tank as a transfer function $H_\bullet: P_\bullet \to F_\bullet$ between the sets $P_\bullet$ and $F_\bullet$ these vary over. 

Ideally then, the composite system below, connecting pipe 2 of Tank $A$ to pipe 1 of Tank $B$, would be described by the composite of the transfer functions $H_A$ and $H_B$ of these tanks.
\[
\begin{tikzpicture}
    % draw the tanks
    \draw (-4,.1)--(-3,.1)--(-3,0)--(-1,0)--(-1,.1)--(1,.1)--(1,0)--(3,0)--(3,.1)--(4,.1);
    \draw (-4,.3)--(-3,.3)--(-3,2)--(-1,2)--(-1,.3)--(1,.3)--(1,1.7)--(3,1.7)--(3,.3)--(4,.3);
    % fill them with water (in the background)
    \begin{pgfonlayer}{background}
        \filldraw[blue!20] (-4,.1)--(-3,.1)--(-3,0)--(-1,0)--(-1,.1)
        --(1,.1)--(1,0)--(3,0)--(3,.1)--(4,.1)
        --(4,.3)--(3,.3)--(3,1.63)--(1,1.63)--(1,.3)
        --(-1,.3)--(-1,1.7)--(-3,1.7)--(-3,.3)--(-4,.3)
        --cycle;
    \end{pgfonlayer}
    \node at (-3,.3) [anchor=-55]{\scriptsize$\begin{array}{c} p_{A1} \\ f_{A1} \end{array}$};
    \node at (3,.3) [anchor=-125]{\scriptsize$\begin{array}{c} p_{B2} \\ f_{B2} \end{array}$};
    \node at (-1,.3) [anchor=-125]{\scriptsize$\begin{array}{c} p_{A2} \\ f_{A2} \end{array}$};
    \node at (1,.3) [anchor=-55]{\scriptsize$\begin{array}{c} p_{B1} \\ f_{B1} \end{array}$};
    \node at (-2,1){Tank $A$};
    \node at (2,1){Tank $B$};
\end{tikzpicture}
\]
This is rarely the case, and indeed makes little sense: the output of transfer
function $H_A$ describes the flow through Tank $A$, while the domain of the
transfer function $H_B$ describes pressures through Tank $B$, so taking the
output of $H_A$ as the input of $H_B$ runs into type issues. Instead, here
the relationship between the transfer functions $H_A$ and $H_B$, and the
transfer function $H_{AB}$ of the composite system can be understood through the
fact that connection requires that the pressure at pipe 1 of Tank $A$ must be equal to the pressure at pipe 2 of Tank $B$, and that the flow out of the former pipe must equal the flow into the latter---that is, by imposing the relations 
\begin{align*}
p_{A2} &= p_{B1} \\ f_{A2} &= - f_{B1}
\end{align*}
on the variables. Indeed in many contexts, hydraulics and electronics among them, connections between systems are characterised not by the `output' of one system forming the `input' of the next, but by \emph{variable sharing} between the systems. Such relations are often difficult to describe using the language of transfer functions and state space methods.

The themes of this thesis fit tightly into this programme of modelling control
theoretic systems with an emphasis on interconnection---Willems' so-called
behavioural approach. Willems first demonstrated this approach in the setting of
linear time-invariant discrete-time (LTI) dynamical systems \cite{Wi3}. A limitation
of the approach has been lack of formal language for representing and reasoning
about the interconnection of systems. In this chapter we address this by
developing a graphical language for LTI systems.

The expressions of this graphical language, closely resembling the signal flow
graphs of Shannon~\cite{Sh}, will form the morphisms of a category of
corelations, and its derivation as such will be crucial in developing a sound
and complete equational theory of LTI systems.  

To acquaint ourselves with signal flow graphs, we begin with the example below,
rendered in traditional, directed notation.
\begin{equation}\label{eq:examplesfg}
  \begin{aligned}
\lower11pt\hbox{$\includegraphics[height=2cm]{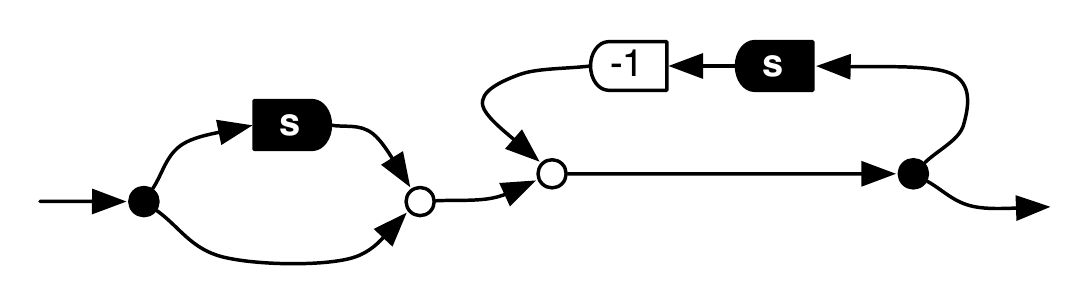}$}
\end{aligned}
\end{equation}
This system takes, as input on the left, a stream of values from a field $\k$,
e.g.\ the rational numbers, and on the right outputs a processed stream of
values. The white circles are adders, the black circles are duplicators, the
$s$ gates are 1-step delays and the $-1$ gate is an instance of an amplifier
that outputs $-1$ times its input. Processing is done synchronously according to
a global clock. 

For instance, assume that at time $0$ the left $s$ gate `stores' the value $1$
and the right $s$ gate stores $2$. Given an input of $-1$, the flow graph first
adds the left stored value $1$, and then adds $-1 \times 2$, for an output of
$-2$. Immediately after this time step the $s$ gates, acting as delays, now
store $-1$ and $-2$ respectively, and we repeat the process with the next input.
Thus from this time $0$ an input stream of $-1,1,-1,1\dots$ results in an output
stream of $-2,2,-2,2,\dots$.

We can express \eqref{eq:examplesfg} as a string diagram by forgetting the
directionality of wires and composing the following basic building blocks using
the operations of monoidal categories:
\[
\copygen \mathrel{;} 
\!\!\!\!\!\!\!
\begin{array}{c} \delaygen \\ \oplus \\ \id \end{array} 
\!\!\!\!\!\!\!\mathrel{;}\!\! \addgen
\mathrel{;} \!\!\!\!\!\!\! \begin{array}{c} \discardopgen \\ \oplus \\ \id \end{array} 
\!\!\!\!\!
\mathrel{;} 
\!\!\!\!\!
 \begin{array}{c} \copygen \\ \oplus \\ \id \end{array} 
\!\!\!\!\!
\mathrel{;}
\!\!\!\!\! 
 \begin{array}{c} \minonegen \\ \oplus \\ \id \end{array} 
\!\!\!\!\! 
 \mathrel{;}
\!\!\!\!\! 
 \begin{array}{c} \delayopgen \\ \oplus \\ \id \end{array} 
 \!\!\!\!\! 
 \mathrel{;}
\!\!\!\!\! 
 \begin{array}{c} \id \\ \oplus \\ \copygen \end{array} 
  \!\!\!\!\! 
 \mathrel{;}
\!\!\!\!\! 
 \begin{array}{c} \copyopgen \\ \oplus \\ \id \end{array} 
 \!\!\!\!\! 
 \mathrel{;}
\!\!\!\!\! 
 \begin{array}{c} \discardgen \\ \oplus \\ \id \end{array} 
\]
The building blocks come from the signature of an algebraic theory---a
\emph{symmetric monoidal theory} to be exact. The terms of this theory comprise
the morphisms of a \emph{prop}, a symmetric monoidal category in which the
objects are the natural numbers. With an operational semantics suggested by the
above example, the terms can also be considered as a process algebra for signal
flow graphs. The idea of understanding complex systems by ``tearing'' them into
more basic components, ``zooming'' to understand their individual behaviour and
``linking'' to obtain a composite system is at the core of the behavioural
approach in control. The algebra of symmetric monoidal categories
thus seems a good fit for a formal account of these compositional principles.

This work is the first to make this link between monoidal categories and the
behavioural approach to control explicit. Moreover, it is the first to endow
signal flow graphs with their standard systems theoretic semantics in which the
registers---the `$s$' gates---are permitted to hold \emph{arbitrary} values at
the beginning of a computation.  This extended notion of behaviour is not merely
a theoretical curiosity: it gives the class of \emph{complete LTI discrete
dynamical systems}. The interest of systems theorists is due to practical
considerations: physical systems seldom evolve from zero initial conditions.

Although previous work~\cite{BSZ2,BSZ3,Za} made the connection between signal
flow graphs and string diagrams, their operational semantics is more restrictive
than that considered here, considering only trajectories with finite past and
demanding that, initially, all the registers contain the value $0$.  Indeed,
with this restriction, it is not difficult to see that the trajectories
of~\eqref{eq:examplesfg} are those where the output is \emph{the same} as the
input. The input/output behaviour is thus that of a stateless wire.  The
equational presentation in this case is the theory $\ih_{\k[s]}$ of
interacting Hopf algebras \cite{Za}, and indeed, in $\ih_{\k[s]}$:
\begin{equation}\label{eq:exampleproof}
\lower12pt\hbox{$\includegraphics[height=2cm]{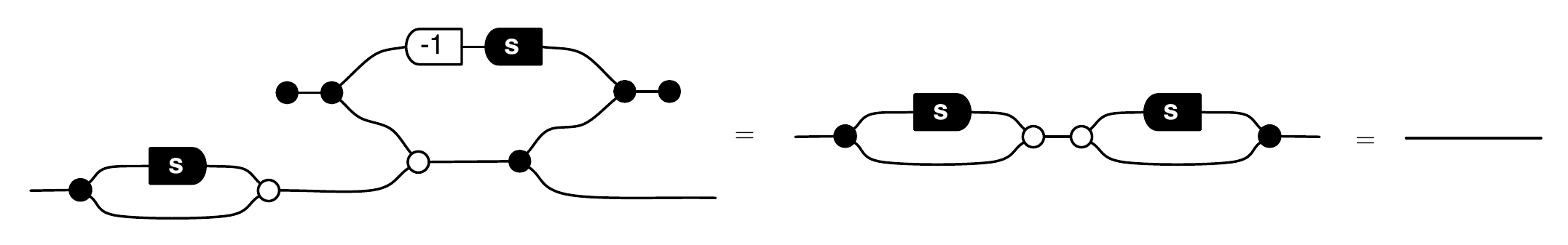}$} 
\end{equation}

Note that~\eqref{eq:exampleproof} is \emph{not sound} for circuits
with our more liberal, operational semantics. Indeed, recall that when the
registers of~\eqref{eq:examplesfg} initially hold values $1$ and $2$, the input
$-1,1,-1,1,\dots$ results in the output $-2,2,-2,2,\dots$. This trajectory is
not permitted by a stateless wire, so $\ih_{k[s]}$ is not sound for reasoning
about LTI systems in general. We have provided a sound
and complete theory to do just that.

In terms of the algebraic semantics, the difference from previous
work is that where there streams were handled with Laurent
(formal power) series, here we use biinfinite streams. These are sequences of
elements of $\k$ that are infinite in the past \emph{as well} as in the
future---that is, elements of $\k^\z$.  Starting with the operational
description, one obtains a biinfinite trajectory by executing circuits forwards
and backwards in time, for some initialisation of the registers. The dynamical
system defined by a signal flow diagram is the set of trajectories obtained by
considering all possible executions from all possible initialisations.  Indeed,
this is the very extension that allows us to discuss non-controllable
behaviours; in~\cite{BSZ,BSZ3,Za,BE} all definable behaviours were controllable. 

An equational theory also requires equations between the terms. We obtain the
equations in two steps. First, we show there is a full, but not faithful,
morphism from the prop $\cospan\mat\pk$ of cospans of matrices over the ring
$\pk$ to the prop $\ltids$ of complete LTI discrete dynamical systems.
Using the presentation of $\cospan\mat\pk$ in~\cite{BSZ2,Za}, the result is a
sound, but not complete, proof system. The second ingredient is restricting our
attention from cospans to corelations. This gives a faithful morphism, allowing
us to present the prop of corelations as a symmetric monoidal theory, and hence
giving a sound and complete proof system for reasoning about LTIs (Theorem
\ref{thm.main}).

The advantages of the string diagram calculus over the traditional matrix
calculus are manifold. The operational semantics make the notation intuitive, as
does the compositional aspect: it is cumbersome to describe connection of
systems using matrices, whereas with string diagrams you just connect the right
terminals. Moreover, the calculus unifies the variety of distinct methods for
representing LTI systems with matrix equations---built from kernel and image
representations~\cite{Wi,Wi3}---into a single framework, heading off
possibilities for ambiguity and confusion.

We hope, however, the greatest advantage will be the way these properties can be
leveraged in analysis of controllability. In Theorem \ref{cor.spanreps}, we
show that in our setting controllability has an elegant structural
characterisation.  Compositionality pays off here, with our proof system giving
a new technique for reasoning about control of compound systems (Proposition
\ref{prop:veryexciting}).  From the systems theoretic point of view, these
results are promising since the compositional, diagrammatic techniques we bring
to the subject seem well-suited to problems such as controllability of
interconnections, of primary interest for multiagent and spatially
interconnected systems~\cite{OFM}.
%
%Summing up, our original technical contributions are:
%\begin{itemize}
%\item a characterisation of the class of LTI systems as a category of corelations of matrices
%\item a presentation of this category of corelations of matrices as a symmetric
%  monoidal theory
%\item an operational semantics that agrees with the standard systems theoretic
%semantics of signal flow graphs
%\item a characterisation of controllability
%\end{itemize}

\section{Linear time-invariant dynamical systems} \label{sec.systems}

Following Willems~\cite{Wi3}, a \define{dynamical system} $(T, W,\bb)$ is: a
\define{time axis} $T$, a \define{signal space} $W$, and a \define{behaviour}
$\bb \subseteq W^T$. We refer to $w \in \bb$ as \define{trajectories}. 

Here we are interested in discrete trajectories that are \define{biinfinite}:
infinite in past and future.  Our time axis is thus the integers $\z$.  Let $\k$
be a field: for concreteness one may take this to be the rationals $\mathbb{Q}$,
the reals $\R$ or the booleans $\z_2$.  The signal space is $\k^d$, where $d$ is
the number of \define{terminals} of a system.  These, in engineering terms, are
the interconnection variables that enable interaction with an environment.

The dynamical systems of concern to us are thus specified by some natural number
$d$ and a subset $\bb$ of $(\k^d)^\z$. The sets $(\k^d)^\z$ are $\k$-vector
spaces, with pointwise addition and scalar multiplication. We restrict attention
to \emph{linear} systems, meaning that $\bb$ is required to be a \emph{$\k$-linear subspace}---i.e.
closed under addition and multiplication by $\k$-scalars---of
$(\k^d)^\z$. 

We partition terminals into a \emph{domain} and \emph{codomain} of $m$ and $n$
terminals respectively, writing $\bb \subseteq (\k^m)^\z \oplus (\k^n)^\z \cong
(\k^d)^\z$.  This may seem artificial, in the sense that the assignment is
arbitrary.  In particular, it is crucial not to confuse the domains (codomains)
with inputs (outputs). In spite of the apparent contrivedness of choosing such a
partition, Willems and others have argued that it is vital for a sound theory of
system \emph{decomposition}; indeed, it enables the ``tearing'' of Willems'
tearing, zooming and linking~\cite{Wi}.
%. We
%use this structure to partition the system into a part we wish to connect to
%another system, and a part which we do not. It is an artifice in the sense that
%we may take any terminal in the domain and move it to the codomain, and vice
%versa. This ability is called compactness for a monoidal category.

Once the domains and codomains have been chosen, systems are linked by
connecting terminals. In models of physical systems this means variable coupling
or sharing; in our discrete setting where behaviours are subsets of a
cartesian product---i.e.\ relations---it amounts to relational composition.
Since behaviours are both relations and linear subspaces, a central underlying
mathematical notion is a linear relation.  

\smallskip
%The next restriction is the property of \define{time-invariance}. 
A behaviour is \define{time-invariant} when for every trajectory $w \in \bb$ and
any fixed $i\in\z$, the trajectory whose value at every time $t\in\z$ is
$w(t+i)$ is also in $\bb$.  
Time-invariance brings with it a connection with the algebra of polynomials.
Following the standard approach in control theory, going back to
Rosenbrock~\cite{Ro}, we work with polynomials over an indeterminate $s$ as well
as its formal inverse $s^{-1}$---i.e.\ the elements of the ring
$\pk$.

The indeterminate $s$ acts on a given biinfinite stream $w \in \k^\z$ as a
one-step delay, and $s^{-1}$ as its inverse, a one step anticipation: 
\[ 
  (s\cdot w) (t) \Defeq w(t-1),\quad (s^{-1}\cdot w)(t) \Defeq w(t+1).
\]
We can extend this, in the obvious linear, pointwise manner, to an action of any
polynomial $p\in \pk$ on $w$.  Since $\k^\z$ is a $\k$-vector space, any such
$p$ defines a $\k$-linear map $\k^\z\to \k^\z$ ($w \mapsto p\cdot w$).

Given this, we can view $n\times m$ matrices over $\pk$ as $\k$-linear maps from
$(\k^m)^\z$ to $(\k^n)^\z$. This viewpoint can be explained succinctly as a
functor from the prop $\mat\pk$, defined below, to the category of $\k$-vector
spaces and linear transformations $\vect_\k$.

Recall that a \define{prop} is a strict symmetric monoidal category where the
set of objects is the natural numbers $\nn$, and monoidal product ($\oplus$) on
objects is addition. Homomorphism of props are identity-on-objects strict
symmetric monoidal functors.

\begin{definition}
  The prop $\mat\pk$ has as arrows $m \to n$ the $n\times m$-matrices over
  $\pk$. Composition is matrix multiplication, and the monoidal product of $A$
  and $B$ is $\left[\begin{smallmatrix} A & 0 \\ 0 & B
  \end{smallmatrix}\right]$. The symmetries are permutation matrices.
\end{definition}

The functor of interest
\[
  \vectfun\maps \mat\pk \longrightarrow \vect_\k
\]
takes a natural number $n$ to $(\k^n)^\z$, and an $n\times m$ matrix to the
induced linear transformation $(\k^m)^\z \to (\k^n)^\z$. Note that $\vectfun$ is
faithful.

\smallskip
The final restriction on the set of behaviours is called \emph{completeness},
and is a touch more involved. For $t_0,t_1 \in \z$, $t_0 \le t_1$, write
$w|_{[t_0,t_1]}$ for the restriction of $w: \z \to \k^n$ to the set $[t_0,t_1] =
\{t_0, t_0+1, \dots, t_1\}$. Write  $\bb|_{[t_0,t_1]}$ for the set of the
restrictions of all trajectories $w \in \bb$ to $[t_0,t_1]$.  Then $\bb$ is
\define{complete} when $w|_{[t_0,t_1]} \in \bb|_{[t_0,t_1]}$ for all $t_0,t_1
\in \z$ implies $w \in \bb$. This topological condition is important as it
characterises the linear time-invariant behaviours that are kernels of the
action of $\mat\pk$; see Theorem \ref{thm.kernelreps}.

%We shall introduce one additional structure: a partition of the terminals of a
%dynamical system into two subsets, called a domain and a codomain. This will
%provide us with language to talk about interconnection of dynamical systems.
%We thus make the following definition.

\begin{definition}
  A \define{linear time-invariant (LTI) behaviour} comprises a domain
  $(\k^m)^\z$, a codomain $(\k^n)^\z$, and a subset $\bb \subseteq (\k^m)^\z
  \oplus (\k^n)^\z$ such that $(\z,\k^m \oplus \k^n,\bb)$ is a complete, linear,
  time-invariant dynamical system.
\end{definition}

The algebra of LTI behaviours is captured concisely as a prop.
\begin{proposition} \label{prop.ltidsiswelldefined}
  There exists a prop $\ltids$ 
  % objects the vector spaces of the form $(\k^n)^\z$ for $n \in \nn$, and 
  % morphisms $(\k^n)^\z \to (\k^m)^\z$ the LTIDS with domain
  with morphisms $m \to n$ the LTI behaviours with domain $(\k^m)^\z$ and
  codomain $(\k^n)^\z$. Composition is relational. The monoidal product is
  direct sum.
\end{proposition}

%\subsection{Kernel representations}

The proof of Proposition~\ref{prop.ltidsiswelldefined} relies on \emph{kernel
representations} of LTI  systems.  The following result lets us pass between
behaviours and polynomial matrix algebra.
%\[
%\mathfrak{B}:=\{w:\mathbb{Z} \rightarrow \k^q \mid \mbox{\rm s.t.~} R(\sigma,\sigma^{-1})w=0\}\; .
%\]
%and (for \emph{controllable} behaviours) the \emph{image representation}
%\begin{equation}\label{eq:im}
%w=M(\sigma,\sigma^{-1})\ell\; ,
%\end{equation}
%where $M\in\R^{q\times m}[s,s^{-1}]$, defining the behaviour
%\[
%\mathfrak{B}:=\{w:\mathbb{Z} \rightarrow \R^q \mid  \mbox{\rm exists } \ell:\mathbb{Z}\rightarrow \R^m \mbox{\rm s.t. (\ref{eq:ker}) holds}\}\; ,
%\]}
%
%
\begin{theorem}[Willems {\cite[Theorem 5]{Wi3}}] \label{thm.kernelreps}
  Let $\bb$ be a subset of $(\k^n)^\z$ for some $n \in \mathbb N$. Then $\bb$ is
  an LTI behaviour iff there exists $M \in
  \mat\pk$ such that $\bb = \mathrm{ker}(\vectfun M)$.
\end{theorem}

%Completeness is an important property. An example of a non-complete linear
%time-invariant subspace of $\k^\z$ is the set of all finitely supported
%functions $\z \to \k$.  This is not the kernel of the action of any matrix.

%\smallskip
The prop $\mat \pk$ is equivalent to the category $\fmod\pk$ of
finite dimensional free $\pk$-modules. Since $\fmod R$ over a principal ideal
domain (PID) $R$ has finite colimits \cite{BSZ2}, and $\pk$ is a PID, $\mat \pk$
has finite colimits, and thus it has pushouts.

We can therefore define the prop $\cospan\mat\pk$ where arrows are (isomorphism
classes of) cospans of matrices: arrows $m \to n$ comprise a natural number
$d$ together with a $d\times m$ matrix $A$ and a $d\times n$ matrix $B$; we
write this $m\xrightarrow{A} d \xleftarrow{B}n$. 

We can then extend $\vectfun$ to the functor
\[
  \cospanfun \maps \cospan\mat\pk \longrightarrow \linrel_\k
\]
where on objects $\cospanfun(n)=\vectfun(n)=(\k^n)^\z$, and on arrows
%defined by mapping each cospan to a linear relation:
\[
  m\xrightarrow{A} d \xleftarrow{B}n
\]
maps to
\begin{equation}\label{eq.K}
  \big\{(\mathbf{x},\mathbf{y}) \,\big|\,
  (\vectfun A)\mathbf{x} = (\vectfun B)\mathbf{y}\big\} 
  \subseteq (\k^m)^\z \oplus (\k^n)^\z.
\end{equation}
It is straightforward to prove that this is well defined.
%\begin{remark}\rm
%From a system theory point of view, the definition~\eqref{eq.K}, can be interpreted as associating with $R_1(\sigma,\sigma^{-1}), R_2(\sigma,\sigma^{-1})$, where 
%$R_i\in\R^{p\times n_i}[s,s^{-1}]$, the set of trajectories  
%\begin{multline}\label{eq:mycospan}
%\mathcal{Z}_{R_1,R_2}:= \\
%\{ w:\mathbb{Z}\rightarrow \k^{n_1+n_2} \mid 
%w=(w_1,w_2),\,
%R_1(\sigma,\sigma^{-1})w_1 = R_2(\sigma,\sigma^{-1})w_2 \}
%\end{multline}
%where $w=(w_1,w_2)$ is a terminal partition.
%\end{remark}
\begin{proposition}\label{prop.funct}
$\cospanfun$ is a functor.
\end{proposition}
\begin{proof}
Identities are clearly preserved; it suffices to show that composition is too.
Consider the diagram below, where the pushout is calculated in $\mat\pk$.
\[
\xymatrix{
  {}\ar[dr]_{A_1} & & \ar[dl]^{B_1}  
  \ar[dr]_{A_2} & & \ar[dl]^{B_2} {} 
 \\
& \ar[dr]_{C} & & \ar[dl]^{D}  \\
& & {\save*!<0cm,-.5cm>[dl]@^{|-}\restore}
}
\]
To show that $\cospanfun$ preserves composition we must verify that 
\[
  \{(\mathbf{x},\mathbf{y})\,|\, \theta CA_1\mathbf{x} = \theta DB_2\mathbf{y}\} =
  \{(\mathbf{x},\mathbf{z})\,|\, \theta A_1\mathbf{x} = \theta B_1\mathbf{z}\} ; 
  \{(\mathbf{z},\mathbf{y})\,|\, \theta A_2\mathbf{z} = \theta B_2\mathbf{y}\}.
\]
The inclusion $\subseteq$ follows from properties of pushouts in $\mat\pk$ (see
\cite[Proposition 5.7]{BSZ2}).  To see $\supseteq$, we need to show that if there
exists $\mathbf{z}$ such that $\theta A_1\mathbf{x} = \theta B_1\mathbf{z}$ and
$\theta A_2\mathbf{z} = \theta B_2\mathbf{y}$, then $\theta CA_1\mathbf{x} =
\theta DB_2\mathbf{y}$.  But $\theta CA_1\mathbf{x}=\theta CB_1\mathbf{z}=\theta
DA_2\mathbf{z}=\theta DB_2\mathbf{y}$.
\end{proof}

%Paolo's original remark
%\begin{remark}\rm
%\textcolor{red}{The result of Proposition \ref{prop:Kfunct} can be interpreted from a system theory point of view as identifying with a terminal decomposition $w=(w_1,w_2)$ of the variables of a LTI behaviour $\bb$ and associated kernel representation
%\begin{equation}\label{eq:ker}
%R_1(\sigma,\sigma^{-1})w_1=R_2(\sigma,\sigma^{-1})w_2\; ,
%\end{equation} 
%where $R_i\in\R^{p\times n_i}[s,s^{-1}]$, the set of trajectories 
%\begin{equation}\label{eq:mycospan}
%\mathcal{Z}_{R_1,R_2}:=\{z:\mathbb{Z}\rightarrow \R^p \mid \exists w_i:\mathbb{Z}\rightarrow \R^{n_i}\mbox{~s.t.~} z=R_i(\sigma,\sigma^{-1})w_i,~i=1,2\}\; .
%\end{equation}}
%\end{remark}

Rephrasing the definition of $\cospanfun$ on morphisms~\eqref{eq.K}, the
behaviour consists of those $(\mathbf{x},\mathbf{y})$ that satisfy
\[
  \vectfun\left[\begin{array}{cc} A & -B\end{array}\right]
  \left[\begin{array}{c}\mathbf{x}\\\mathbf{y}\end{array}\right] 
  = \mathbf{0},
\]
so one may say---ignoring for a moment the terminal domain/codomain assignment---that 
\[ 
  \cospanfun (\xrightarrow{A}\xleftarrow{B}) = \ker \vectfun\left[
  \begin{array}{cc} A & -B\end{array}\right].
\]

With this observation, as a consequence of Theorem~\ref{thm.kernelreps},
$\cospanfun$ has as its image (essentially) the prop $\ltids$. This proves
Proposition \ref{prop.ltidsiswelldefined}. We may thus consider $\cospanfun$ a
functor onto the codomain $\ltids$; denote this corestriction $\cospanfunrest$. We thus have a full functor:
\[
  \cospanfunrest \maps \cospan\mat\pk \longrightarrow \ltids.
\]

\begin{remark}\label{rmk:faithfulness}
It is important for the sequel to note that $\cospanfunrest$ is \emph{not} faithful.
For instance, $\cospanfunrest(1\xrightarrow{[1]}1\xleftarrow{[1]}1) =
\cospanfunrest(1\xrightarrow{ \left[\begin{smallmatrix} 1\\ 0\end{smallmatrix}\right]
}2 \xleftarrow{ \left[\begin{smallmatrix} 1\\ 0\end{smallmatrix}\right]} 1)$, yet
the cospans are not isomorphic. The task of the next section is to develop a
setting where these cospans are nonetheless \emph{equivalent}.
\end{remark}

\section{Presentation of $\ltids$} \label{sec.diagrams}
Recall that a \define{symmetric monoidal theory} (SMT) is a presentation of a prop: a pair
$(\Sigma,E)$ where $\Sigma$ is a set of \define{generators} $\sigma\colon m\to
n$, where $m$ is the \define{arity} and $n$ the \define{coarity}. A
$\Sigma$-term is a obtained from $\Sigma$, identity $\idn\colon 1\to 1$ and
symmetry $\tw\colon 2\to 2$ by composition and monoidal product, according to
the grammar
\[
  \tm\ ::=\ \sigma\ |\ \idn\ |\ \tw\ |\ \tm\mathrel{;}\tm\ |\ \tm\oplus \tm 
\]
where $\mathrel{;}$ and $\oplus$ satisfy the standard typing discipline that
keeps track of the domains (arities) and codomains (coarities)
\[
\frac{\tm: m\to d \quad \tm': d\to n}
{\tm\mathrel{;}\tm': m\to n}
\quad
\frac{\tm: m\to n \quad \tm': m'\to n'}
{\tm\oplus \tm': m+m'\to n+n'}
\]
The second component $E$ of an SMT is a set of \define{equations}, where an
equation is a pair $(\tm,\mu)$ of $\Sigma$-terms with compatible types, i.e.
$\tm,\mu\colon m\to n$ for some $m,n\in\mathbb{N}$.

Given an SMT $(\Sigma,E)$, the prop $\mathbf{S}_{(\Sigma,E)}$ has as arrows the
$\Sigma$-terms quotiented by the smallest congruence that includes the laws of
symmetric monoidal categories and equations $E$. We sometimes abuse
notation by referring to $\mathbf{S}_{(\Sigma,E)}$ as an SMT. Given an arbitrary
prop $\mathbb{X}$, a \define{presentation} of $\mathbb{X}$ is an SMT
$(\Sigma,E)$ s.t.\ $\mathbb{X} \cong \mathbf{S}_{(\Sigma,E)}$.

In this section we give a presentation of $\ltids$ as an SMT. This means 
that (\emph{i}) we obtain a syntax---conveniently expressed using string
diagrams---for specifying every LTI behaviour, and (\emph{ii}) a sound and
complete equational theory for reasoning about them.

\subsection{Syntax}
We start by describing the graphical syntax of dynamical systems, the arrows of
the category $\syntax = \mathbf{S}_{(\Sigma,\varnothing)}$, where $\Sigma$ is
the set of generators:
\begin{multline}\label{eq:generators}
\{
\addgen,
\zerogen,
\copygen,
\discardgen,
\delaygen, 
\addopgen,
\zeroopgen,
\copyopgen,
\discardopgen,
\delayopgen
\} \\
\cup \{ \scalargen \mid a\in\k \,\} \cup \{ \scalaropgen \mid a\in\k \,\}
\end{multline}
For each generator, we give its denotational semantics, an LTI behaviour,
thereby defining a prop morphism $\llbracket - \rrbracket: \syntax \to\ltids$.
\[
\addgen \!\!\mapsto \{\,( \left( 
  {\begin{smallmatrix}\tau \\ \upsilon\end{smallmatrix}} \right),\, \tau+\upsilon) \mid \tau,\upsilon \in \k^\z \,\}\maps 2\to 1
\]
\[
\quad
\zerogen \!\!\mapsto
\{\,((),0)\,\} \subseteq \k^\z\maps 0\to 1
\]
\[
\copygen \!\!\mapsto 
\{\, (\tau, \left( \begin{smallmatrix} \tau \\ \tau\end{smallmatrix} \right)) \mid \tau\in \k^\z \,\}\maps1\to 2
\]
\[
\quad
\discardgen \!\!\mapsto
\{\,(\tau,()) \mid \tau \in \k^\z\,\} \maps 1\to 0
\]
\[
\scalargen  \!\!\mapsto
\{\, (\tau, a\cdot\tau) \mid \tau\in\k^\z \, \} \maps 1 \to 1 \quad (a\in\k)
\]
\[
\ 
\delaygen \!\!\mapsto
\{\, (\tau, s\cdot\tau) \mid \tau\in\k^\z\,\} \maps 1 \to 1
\]

% In Fig.~\ref{fig:generators} we introduce half of the
%generators and their denotations.
%\begin{figure*}[ht]
%\[
%\addgen \!\!\mapsto \{\,( \left( 
%  {\begin{smallmatrix}\tau \\ \upsilon\end{smallmatrix}} \right),\, \tau+\upsilon) \mid \tau,\upsilon \in \k^\z \,\}\maps 2\to 1
%\]
%\[
%\quad
%\zerogen \!\!\mapsto
%\{\,((),0)\,\} \subseteq \k^\z\maps 0\to 1
%\]
%\[
%\copygen \!\!\mapsto 
%\{\, (\tau, \left( \begin{smallmatrix} \tau \\ \tau\end{smallmatrix} \right)) \mid \tau\in \k^\z \,\}\maps1\to 2
%\]
%\[
%\quad
%\discardgen \!\!\mapsto
%\{\,(\tau,()) \mid \tau \in \k^\z\,\} \maps 1\to 0
%\]
%\[
%\scalargen  \!\!\mapsto
%\{\, (\tau, a\cdot\tau) \mid \tau\in\k^\z \, \} \maps 1 \to 1 \quad (a\in\k)
%\]
%\[
%\ 
%\delaygen \!\!\mapsto
%\{\, (\tau, s\cdot\tau) \mid \tau\in\k^\z\,\} \maps 1 \to 1
%\]
%\caption{Generators and their denotations in $\ltids$. \label{fig:generators}}
%\end{figure*}
\noindent 
The denotations of the mirror image generators are the opposite relations.
Parenthetically, we note that a finite set of generators is possible
over a finite field, or the field $\mathbb{Q}$ of rationals, \emph{cf.}\
\S\ref{subsec:eqhopf}.

The following result guarantees that the syntax is fit for purpose: every
behaviour in $\ltids$ has a syntactic representation in $\syntax$.
\begin{proposition}\label{prop.syntaxfull}
  $\llbracket - \rrbracket: \syntax\to\ltids$ is full.
\end{proposition}
\begin{proof}
  The fact that $\llbracket - \rrbracket$ is a prop morphism is immediate since
  $\mathbb{S}$ is free on the SMT $(\Sigma,\varnothing)$ with no
  equations.  Fullness follows from the fact that $\llbracket - \rrbracket$
  factors as the composite of two full functors:
  \[
    \xymatrixrowsep{2pc}
    \xymatrix{
      \syntax \ar[d] \ar[dr]^{\llbracket-\rrbracket} \\
      \cospan\mat\pk \ar[r]_-{\cospanfunrest} & \ltids
    }
  \]
  The functor $\cospanfunrest$ is full by definition. The existence and fullness
  of the functor $\mathbb{S}\to\cospan\pk$ follows from \cite[Theorem 3.41]{Za}. We
  give details in the next two subsections.
\end{proof}

Having defined the syntactic prop $\syntax$ capable of representing every
behaviour in $\ltids$, our task for this section is to identify an equational
theory that \emph{characterises} equivalent representations in $\ltids$: i.e.
one that is sound and complete.  The first step is to use the existence of
$\cospanfun$: with the results of~\cite{BSZ2,Za} we can obtain a presentation
for $\cospan\mat\pk$. This is explained in the next subsection, where we present
$\mat\pk$ and $\cospan\mat\pk$. 

%%%%%%%%%%%%%%%%%%%%%%%%%%%%%%%%%%%%%%%%%%%%%%%%%%%%%%%%%%%

\subsection{Presentations of $\mat\pk$ and $\cospan\mat\pk$}\label{subsec:eqhopf}

%%%%%%%%%%%%%%%%%%%%%%%%%%%%%%%%%%%%%%%%%%%%%%%%%%%%%%%%%%%

To obtain a presentation of $\mat\pk$ as an SMT we only require 
some of the generators:
\[
  \{
\addgen,
\zerogen,
\copygen,
\discardgen,
\delaygen,
\delayopgen\}
\cup \{ \scalargen \mid a\in\k \,\}
\]
and the
% of
%Fig.\ref{fig:generators}, $\lower6pt\hbox{$\includegraphics[height=0.6cm]{pics/delayop.pdf}$}$
%\[
%\{
%\lower8pt\hbox{$\includegraphics[height=0.7cm]{pics/add.pdf}$},
%\lower5pt\hbox{$\includegraphics[height=0.5cm]{pics/zero.pdf}$},
%\lower8pt\hbox{$\includegraphics[height=0.7cm]{pics/copy.pdf}$},
%\lower5pt\hbox{$\includegraphics[height=0.5cm]{pics/discard.pdf}$},
%\lower6pt\hbox{$\includegraphics[height=0.6cm]{pics/delay.pdf}$},
%\lower6pt\hbox{$\includegraphics[height=0.6cm]{pics/delayop.pdf}$}
% \}
% \cup
% \{
% \lower6pt\hbox{$\includegraphics[height=0.6cm]{pics/scalar.pdf}$}
%\mid a \in \k\,
%\}
%\]
%and the 
following equations. %, see \cite{BSZ2,Za}. 
First, the white and the black structure forms a (bicommutative) bimonoid:
\[
\includegraphics[width=.4\textwidth]{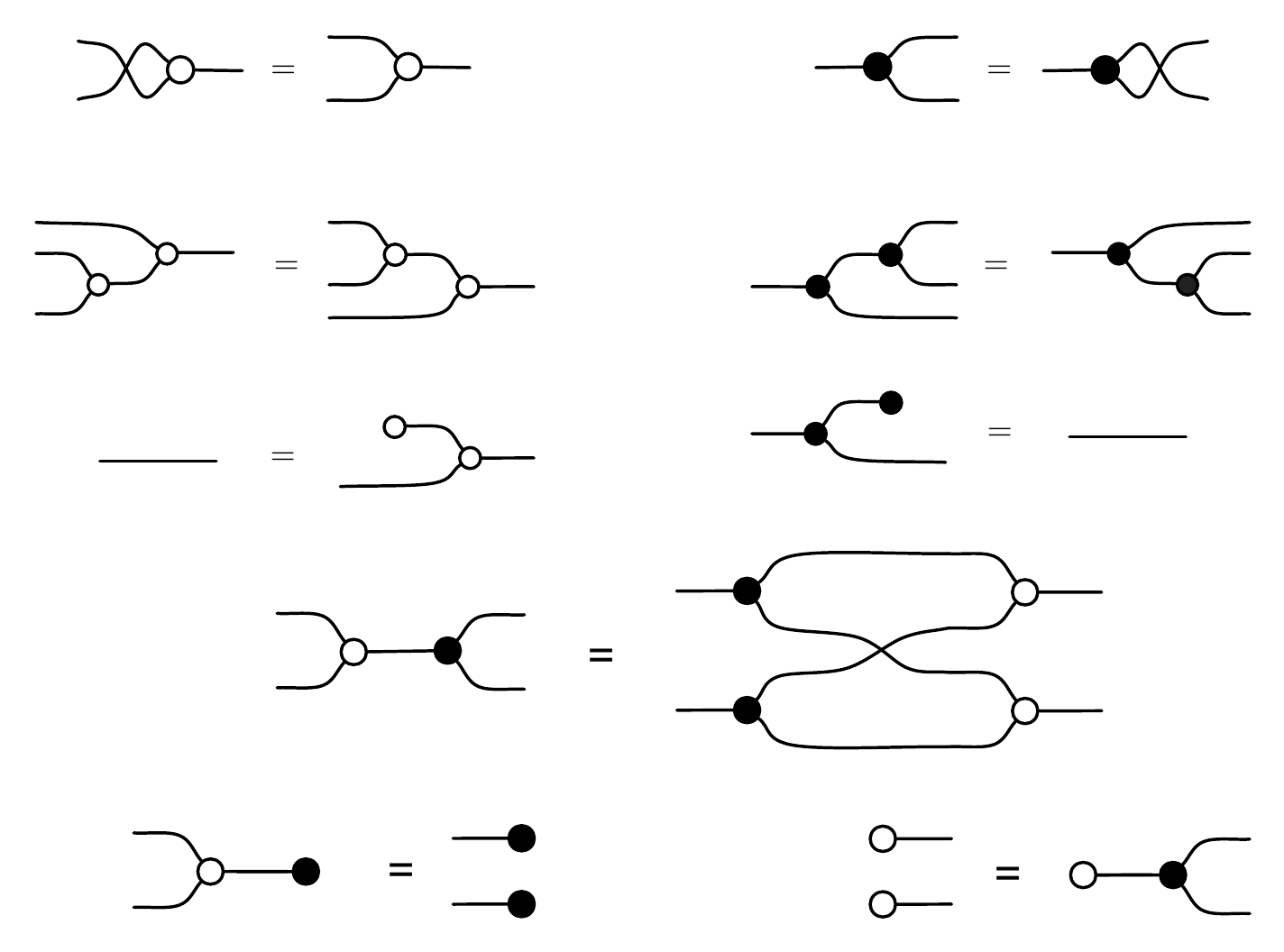}
\]
Next, the formal indeterminate $s$ is compatible with the bimonoid structure
and has its mirror image as a formal inverse.
\[
\includegraphics[height=3.1cm]{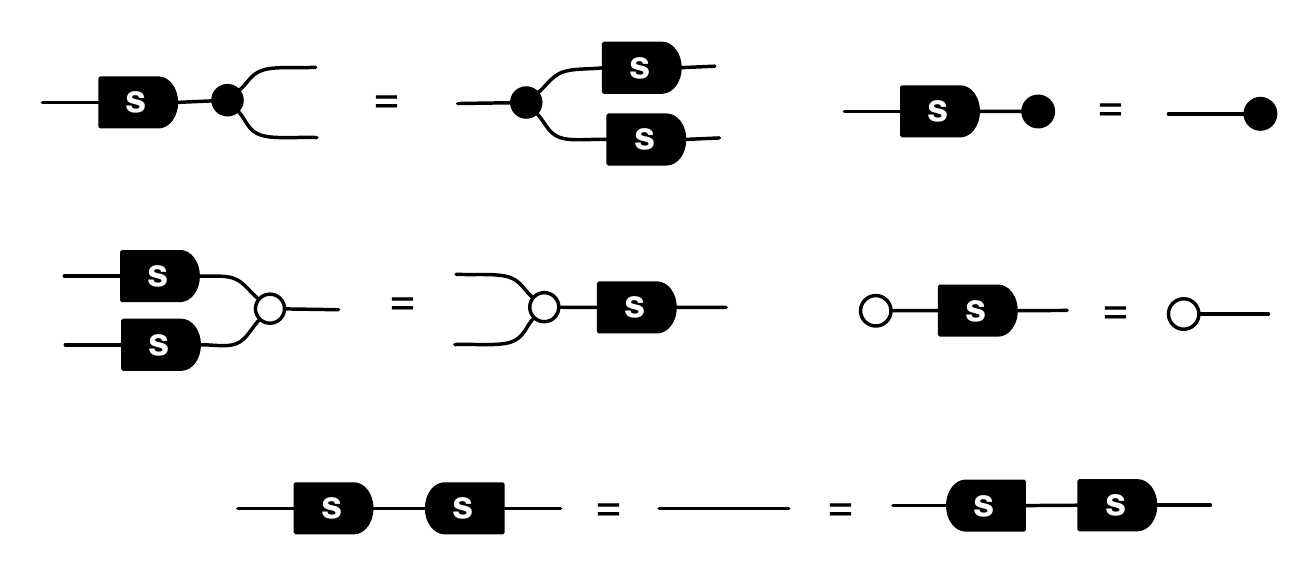}
\]
Finally, we insist that the algebra of $\k$ be compatible with the bimonoid structure and commute with $s$.
%The equations below, in which $a,b\in\k$, are redundant if $\k=\mathbb{Q}$; instead,
%one needs an additional generator, the antipode, identified with the scalar
%$-1$. The details are not important; we merely mention that in other work we
%have used  
%\lower6pt\hbox{$\includegraphics[height=0.6cm]{pics/antipode.pdf}$}
%to represent
%\lower6pt\hbox{$\includegraphics[height=0.6cm]{pics/minusone.pdf}$}, and that
%henceforward we adopt this convention.
\[
\includegraphics[height=5.3cm]{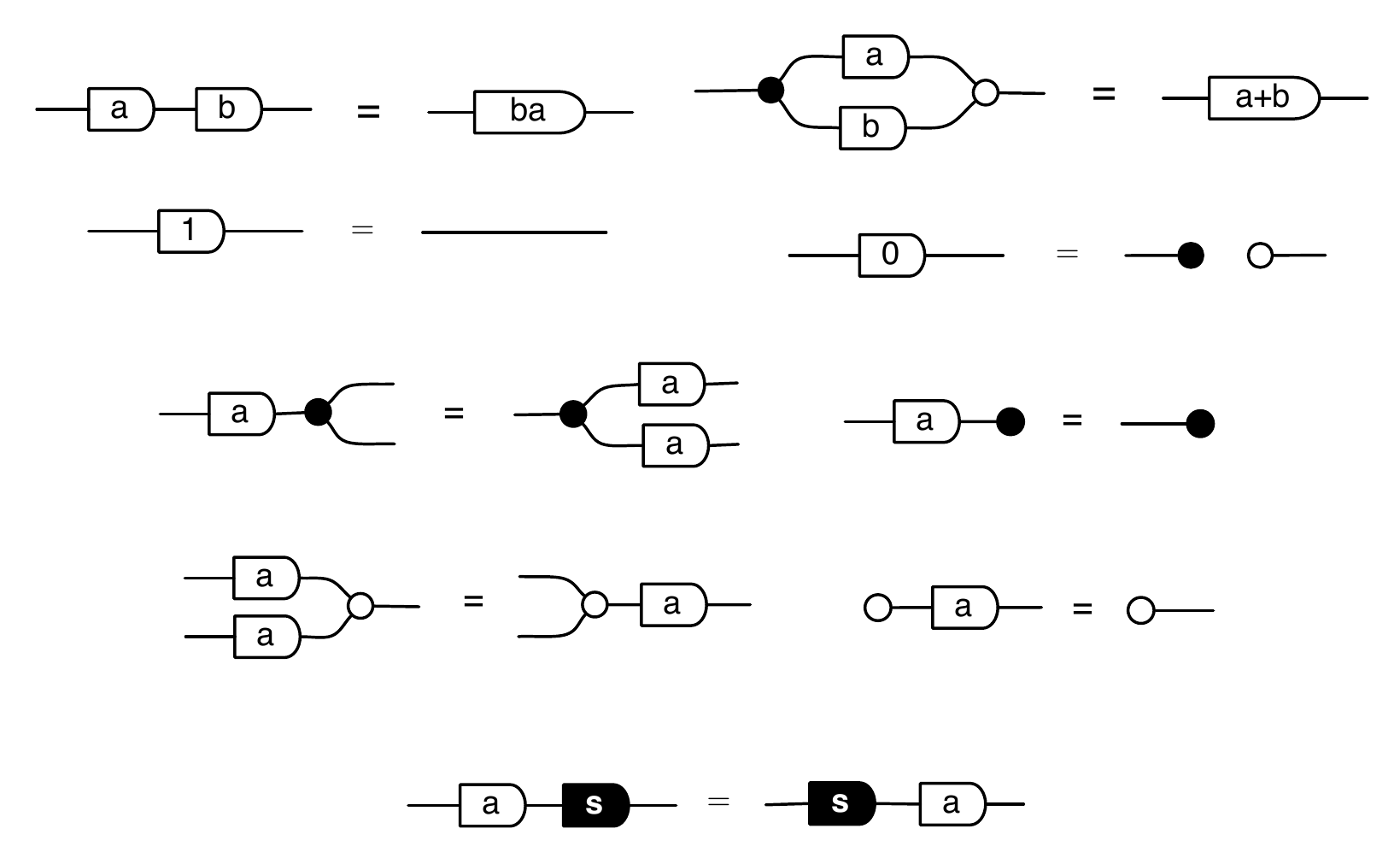}
\]
The three sets of equations above form the theory of Hopf algebras. Write
$\ha_\pk$ for the prop induced by the SMT consisting of the equations
above. The following follows from~\cite[Proposition~3.9]{Za}.
\begin{proposition}
$\mat\pk\cong\ha_\pk$.
\end{proposition}

Arrows of $\mat\pk$ are matrices with polynomial entries, but it may not be
 obvious to the reader how polynomials arise with the string
diagrammatic syntax. We illustrate this below.

\begin{example}\label{rem:polys}
Any polynomial $p=\sum_{i=u}^v a_i s^i$, where $u\leq v\in\z$ and with
coefficients $a_i\in\k$, can be written graphically using the building blocks of
$\mathbb{HA}_\pk$. Rather than giving a tedious formal construction, we
illustrate this with an example for $\k=\R$. A term for $3s^{-3}-\pi
s^{-1}+s^{2}$ is:
\[
\includegraphics[height=1.7cm]{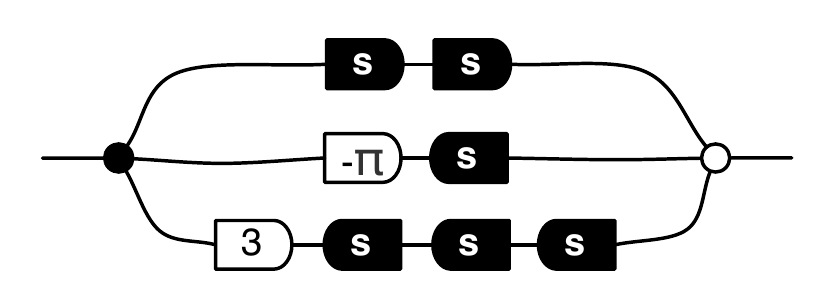}
\]

As an arrow $1 \to 1$ in $\mat\pr$, the above term represents a $1\times
1$-matrix over $\pr$. To demonstrate how higher-dimensional matrices can be
written, we also give a term for the $2 \times 2$-matrix $\begin{bmatrix} 2 & 3s
  \\ s^{-1} & s+1 \end{bmatrix}$:
\[
\includegraphics[height=3cm]{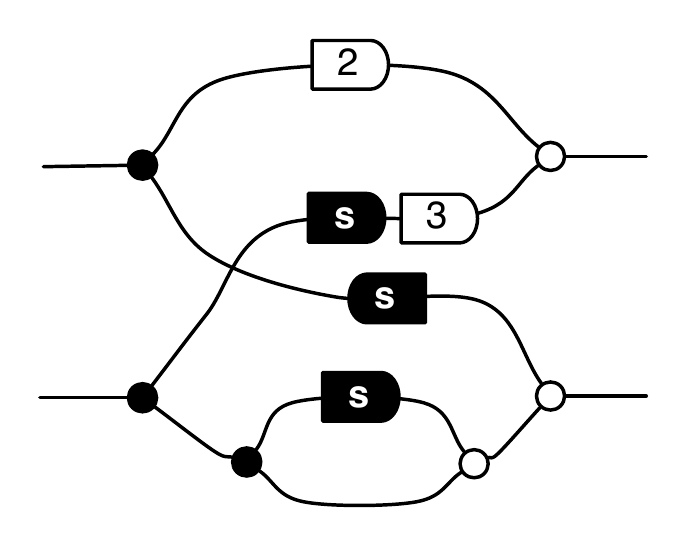}
\]
The above examples are intended to be suggestive of a normal form for terms in
$\ha_\pk$; for further details see \cite{Za}.
\end{example}

%%%%%%%%%%%%%%%%%%%%%%%%%%%%%%%%%%%%%%%%%%%%%%%%%%%%%%%%%%%%%%%%%%%%

To obtain the equational theory of $\cospan\mat \pk$ we need the full set of
generators~\eqref{eq:generators}, along with the equations of $\ha_\pk$, their mirror images, and the
following
\[
\includegraphics[height=4cm]{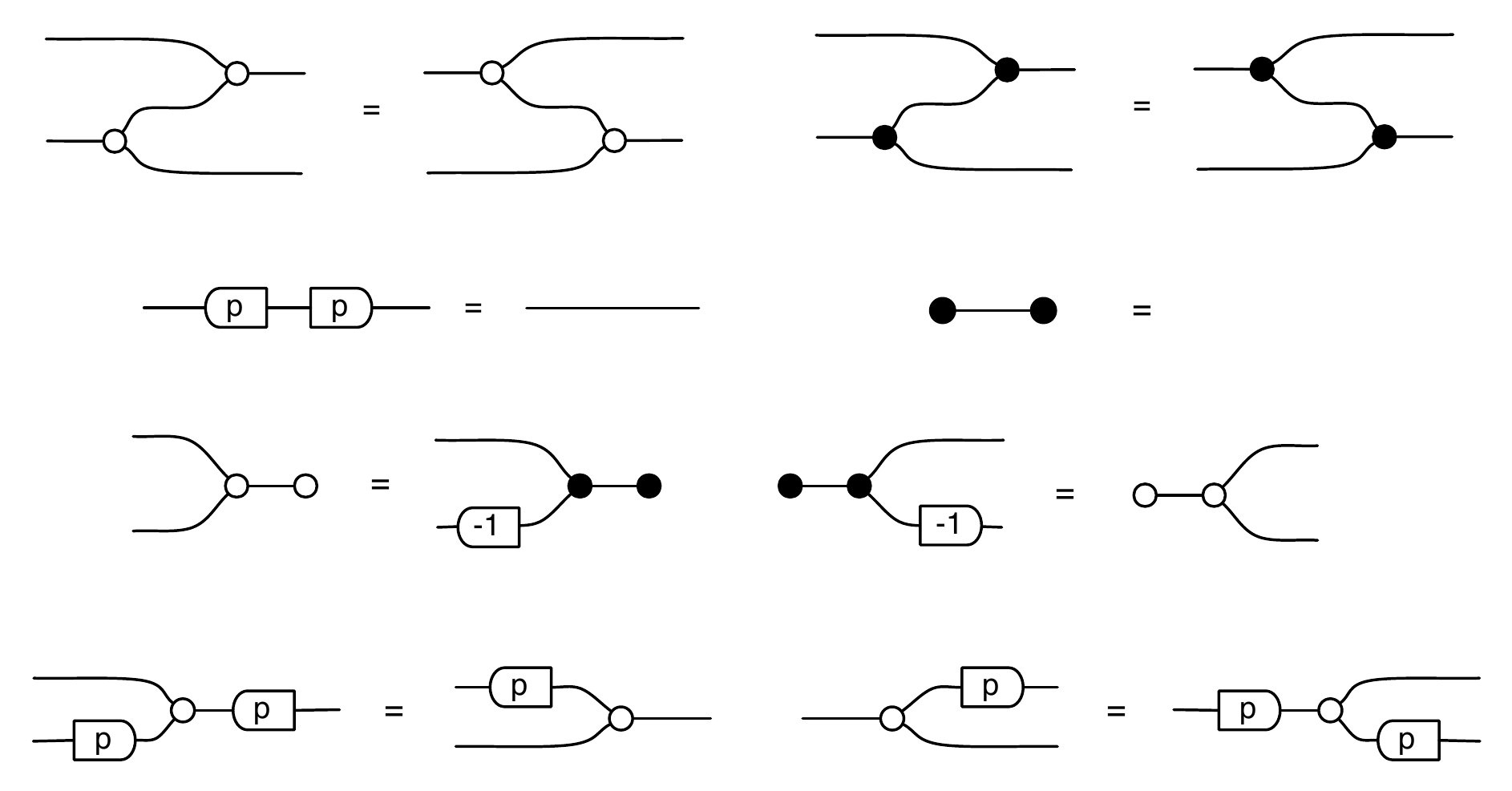}
\]
where $p$ ranges over the nonzero elements of $\pk$ (see
Example~\ref{rem:polys}). Note that in the second equation on the right-hand
side, we use the so-called `empty diagram', or blank space, to represent the
identity map on the monoidal unit, 0. 

The equations of $\ha_\pk$ ensure the generators of $\ha_\pk$ behave as
morphisms in $\mat\pk$, while their mirror images ensure the remaining
generators behave as morphisms in the opposite category $\mat\pk^{\mathrm{op}}$.
The additional equations above govern the interaction between these two sets of
generators, axiomatising pushouts in $\mat\pk$.
Let $\ihcsp$ denote the resulting SMT.
The procedure for obtaining the equations from a distributive law of props 
is explained in~\cite[\S{3.3}]{Za}.
\begin{proposition}[Zanasi~{\cite[Theorem~3.41]{Za}}]\label{prop:cospanpresentation}
\[\cospan \mat\pk \cong \ihcsp.\]
\end{proposition}

Using Proposition~\ref{prop:cospanpresentation} and the existence of $\cospanfunrest$,
%\[
%  \cospanfun \maps \cospan\mat\pk \longrightarrow \ltids
%\]
the equational theory of $\ihcsp$ is a sound proof system for reasoning about
$\ltids$. Due to the fact that $\cospanfunrest$ is not faithful (see
Remark~\ref{rmk:faithfulness}), however, the system is not complete. Achieving
completeness is our task for the remainder of this section.

%%%%%%%%%%%%%%%%%%%%%%%%%%%%%%%%%%%%%%%%%%%%%%%%%%%%%%%%%

\subsection{Corelations in $\mat\pk$}

%%%%%%%%%%%%%%%%%%%%%%%%%%%%%%%%%%%%%%%%%%%%%%%%%%%%%%%%%

In this subsection we identify a factorisation system in $\mat\pk$, and show
that the induced prop $\corel\mat\pk$ of corelations is isomorphic to $\ltids$.
We then give a presentation of $\corel\mat\pk$ and arrive at a sound and 
complete equational theory for $\ltids$.

Recall that split mono is a morphism $m\colon X\to Y$ such that there exists
$m'\colon Y\to X$ with $m'm=\idn_X$. Every morphism in $\mat\pk$ admits a
factorisation into an epi followed by a split mono. 

\begin{proposition}\label{prop.matfactorisation}
  Every morphism $R \in \mat\pk$ can be factored as $R = BA$, where $A$ is an
  epi and $B$ is a split mono.
\end{proposition}
\begin{proof}
%We use the Smith normal form~\cite[Section~6.3]{Ka}. 
%
Given any matrix $R$, the Smith normal form~\cite[Section~6.3]{Ka}  gives
  us  $R = VDU$, where $U$ and $V$ are invertible, and $D$
  is diagonal. In graphical notation we can write it thus:
  \[
\includegraphics[height=1.1cm]{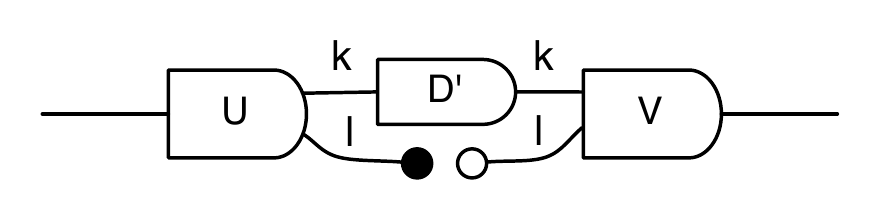}
  \]
  This implies we may write it as $R = U';D';V'$, where $U'$ is a split
  epimorphism, $D'$ diagonal of full rank, and $V'$ a split monomorphism.
  Explicitly, the construction is given by
  \[
\includegraphics[height=1cm]{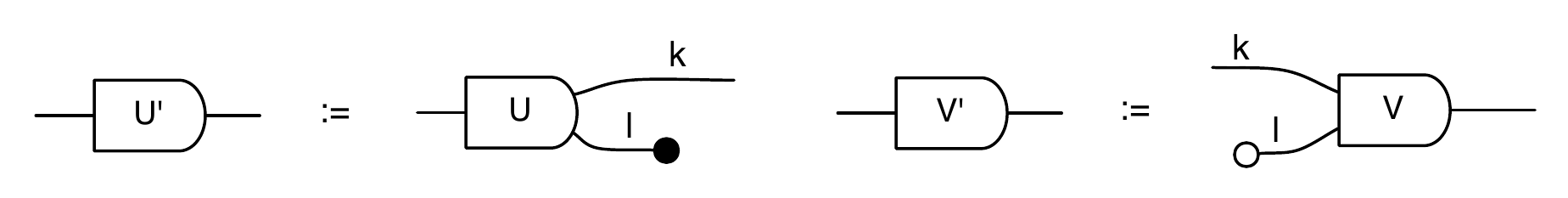}
  \]
  Recall that $\pk$ is a PID, so the full rank diagonal matrix $D'$ is epi. It
  can be checked that $R = V'(D'U')$ is an epi-split mono factorisation. 
%In Appendix~\ref{app.prop.matfactorisation}.
\end{proof}
Recalling the argument in Example \ref{ex.factsysts},
Proposition~\ref{prop.matfactorisation} yields:

\begin{corollary} \label{cor.episplitmono}
  Let $\mathcal E$ be the subcategory of epis and
  $\mathcal M$ the subcategory of split monos. The pair $(\mathcal
  E,\mathcal M)$ is a factorisation system on $\mat\pk$, with $\mathcal M$
  stable under pushout.
\end{corollary}

\begin{definition}
  The prop $\corel \mat \pk$ has as morphisms equivalence classes of
  jointly epic cospans in $\mat\pk$.  
  %Composition is given by the jointly epic part of the pushout.
\end{definition}

We have a full morphism
\[
  F\maps \cospan \mat \pk \longrightarrow \corel \mat \pk
\]
mapping a cospan to its jointly epic counterpart given by the
 factorisation system. %We will show that 
Then $\cospanfunrest$
factors through $F$ as follows:
\[
  \xymatrixrowsep{2pc}
  \xymatrix{
    \cospan\mat\pk \ar[d]_F \ar[dr]^{\cospanfunrest} \\
    \corel\mat\pk \ar[r]_-{\Phi} & \ltids
  }
\]
The morphism $\Phi$ along the base of this triangle is an isomorphism of props, 
and this is our main technical result, Theorem~\ref{thm.main}. The proof relies
on the following beautiful result of systems theory.

\begin{proposition}[Willems {\cite[p.565]{Wi3}}] \label{prop.magic}
  Let $M,N$ be matrices over $\pk$. Then $\ker \vectfun M \subseteq \ker \vectfun N$ iff $\exists$ a matrix $X$ s.t.\ $XM = N$.
\end{proposition}

Further details and a brief history of the above proposition can be found in
Schumacher \cite[pp.7--9]{Sc}. 

\begin{theorem}\label{thm.main}
  There is an isomorphism of props 
  \[
    \Phi\maps \corel\mat\pk \longrightarrow \ltids
  \]
  taking a corelation $\xrightarrow{A}\xleftarrow{B}$ 
  to $\cospanfunrest(\xrightarrow{A}\xleftarrow{B}) = {\ker\theta [A \ -B]}$.
  %mapping each natural number $n$ to the space $(\k^n)^\z$ and each corelation
  %\[
  %  n \stackrel{f}\longrightarrow a \stackrel{g}\longleftarrow m
  %\]
  %to the difference kernel $\ker\vectfun (f\ -g)$.
\end{theorem}
\begin{proof}
  For functoriality, start from $\vectfun\maps\mat\pk \to
  \vect$. Now (i) $\vect$ has an epi-mono factorisation system, (ii)
  $\vectfun$ maps epis to epis and (iii) split monos to monos, so $\vectfun$
  preserves factorisations. Since it is a corollary of Proposition \ref{prop.funct}
  that $\theta$ preserves colimits, it follows that $\vectfun$ extends to
   $\Psi\maps\corel\mat\pk \to \corel\vect$. But $\corel\vect$ is
   isomorphic to $\linrel$ (see \textsection\ref{ssec.linrel}). 
  By Theorem \ref{thm.kernelreps}, the image of $\Psi$ is $\ltids$, and taking
  the corestriction to $\ltids$ gives us precisely $\Phi$, which is therefore a full
  morphism of props.

  As corelations $n \to m$ are in one-to-one correspondence with epis out of
  $n+m$, to prove faithfulness it suffices to prove that if two epis $R$ and $S$
  with the same domain have the same kernel, then there exists an invertible
  matrix $U$ such that $UR =S$. This is immediate from
  Proposition~\ref{prop.magic}: if $\ker R= \ker S$, then we can find $U, V$
  such that $UR = S$ and $VS = R$. Since $R$ is an epimorphism, and since $VUR =
  VS = R$, we have that $VU=1$ and similarly $UV =1$. This proves that any two
  corelations with the same image are isomorphic, and so $\Phi$ is full and
  faithful.  
\end{proof}

%%%%%%%%%%%%%%%%%%%%%%%%%%%%%%%%%%%%%%%%%%%%

\subsection{Presentation of $\corel\mat\pk$}

%%%%%%%%%%%%%%%%%%%%%%%%%%%%%%%%%%%%%%%%%%%%

Thanks to Theorem~\ref{thm.main}, the task of obtaining a presentation of $\ltids$
is that of obtaining one for $\corel\mat\pk$. 
%One way to do this is to focus on the full morphism $F$ from cospans to corelations.
%\[
%  \cospan \mat \pk \longrightarrow \corel \mat \pk.
%\]
%
To do this, we start with the presentation
$\mathbb{IH}^{\textsf{Csp}}$ for $\cospan\mat\pk$ of \S\ref{subsec:eqhopf}; the task of
this section is to identify the additional equations that equate
exactly those cospans that map via $F$ to the same corelation.
%; i.e. those cospans $c,c'$ for which $Fc=Fc'$.  

%As such, our task 
%now is to understand the conditions under which two cospans have the same 
%jointly epic parts.

In fact, only one new equation is required, the `white bone' or `extra' law:
\begin{equation}\label{eq.whitebone}
  \lower8pt\hbox{$\includegraphics[height=.7cm]{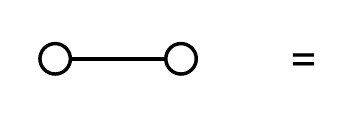}$}\qquad\qquad
\end{equation}
where we have carefully drawn the empty diagram to the right of the equality symbol.
Expressed in terms of cospans, equation \eqref{eq.whitebone} asserts that  
$0\rightarrow 1\leftarrow 0$ and $0\rightarrow 0\leftarrow 0$ are identified:
indeed, the two clearly yield the same corelation. The intuition here is that
cospans $X \xrightarrow{i} S \xleftarrow{j} Y$ map to the same corelation if
their respective copairings $[i,j]\maps X+Y \to S$ have the same jointly epic
parts.  More colloquially, this allows us to `discard' any part of the cospan
that is not connected to the terminals. This is precisely what equation
\eqref{eq.whitebone} represents. Further details on this viewpoint can be found
in \cite{CF}. 

Let $\mathbb{IH}^{\mathrm{Cor}}$ be the SMT obtained
from the equations of $\mathbb{IH}^{\mathrm{Csp}}$ together with equation~\eqref{eq.whitebone}.
\begin{theorem}
$\corel \mat \pk \cong \mathbb{IH}^{\mathrm{Cor}}$.
\end{theorem}
\begin{proof}
Since equation~\eqref{eq.whitebone} holds in $\corel\mat\pk$, we have
a full morphism $\ihcor \to \corel\mat\pk$; it remains
to show that it is faithful. It clearly suffices to show that in the equational
theory $\ihcor$ one can prove that every cospan is equal
to its corelation. Suppose then that $m \xrightarrow{A} k \xleftarrow{B} n$ is a cospan
and $m \xrightarrow{A'} k' \xleftarrow{B'} n$ its corelation. Then, by definition, 
there exists a split mono $M\maps k'\to k$ such that $MA'=A$ and $MB'=B$. Moreover, by the construction
of the epi-split mono factorisation in $\mat\pk$, $M$ is of the form $\lower10pt\hbox{$\includegraphics[height=1cm]{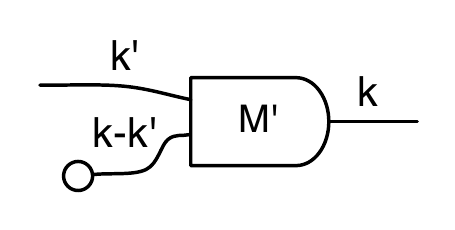}$}$ where $M'\maps k\to k$ is invertible. We can now give the derivation
in $\mathbb{IH}^{\mathrm{Cor}}$:
\begin{align*}
\lower7pt\hbox{$\includegraphics[height=.7cm]{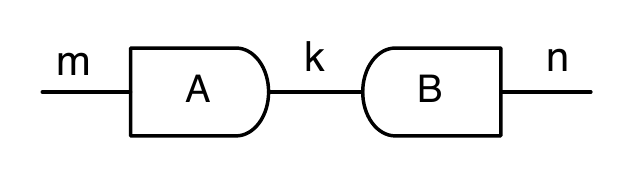}$} \quad
&\stackrel{\ihcsp}{=} \quad
\lower10pt\hbox{$\includegraphics[height=1cm]{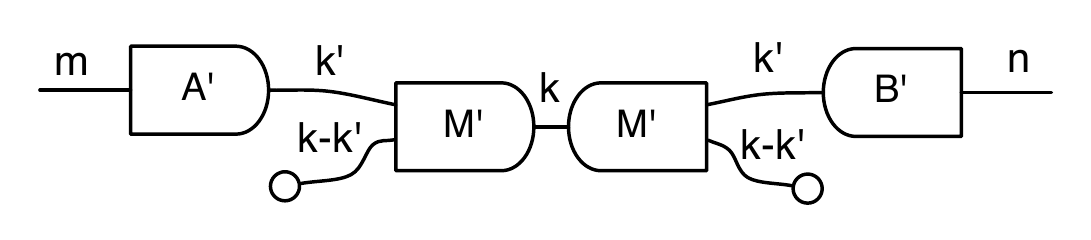}$} \\
\quad &\stackrel{\ihcsp}{=} \qquad 
\lower8pt\hbox{$\includegraphics[height=.8cm]{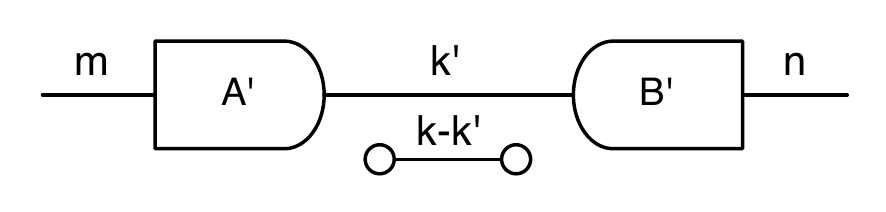}$} \\
&\stackrel{\eqref{eq.whitebone}}{=} \qquad\qquad
\lower7pt\hbox{$\includegraphics[height=.7cm]{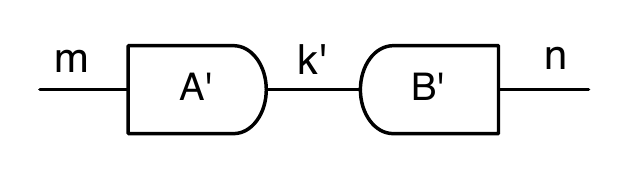}$}
\end{align*}
\end{proof}

%The new construction rule we have for axioms is that we may replace split-monic
%matrices followed by a cozero with just a cozero. Split monic matrices are just
%the tensor of identity and zero matrices followed by an invertible matrix. The
%rules to see that an invertible matrix followed by a cozero is equal to a cozero
%already exist. So the only additional rule is that a zero followed by a cozero
%is a $0 \to 0$ cozero: ie. nothing. This is the bone law.
%\[
%  \includegraphics[height=.7cm]{pics/wbone.pdf}
%\]
%For the reader's convenience, we collect all the equations in Appendix~\ref{app.equations}%
%
%Thus a cospan may be reduced to its jointly epic part by repeated application of
%the usual laws and the white bone law.  Moreover, two cospans then have the same
%jointly epic part if and only if the are the same up to such applications.

We therefore have a sound and complete equational theory capable of representing
all LTI systems, and also a normal form for each LTI system: every such system
can be written, in an essentially unique way, as a jointly epic cospan of terms
in $\ha_\pk$ in normal form. We summarise the axioms in Figure
\ref{fig.ihcoraxioms}.

\begin{figure} 
  \includegraphics[width=.9\textwidth]{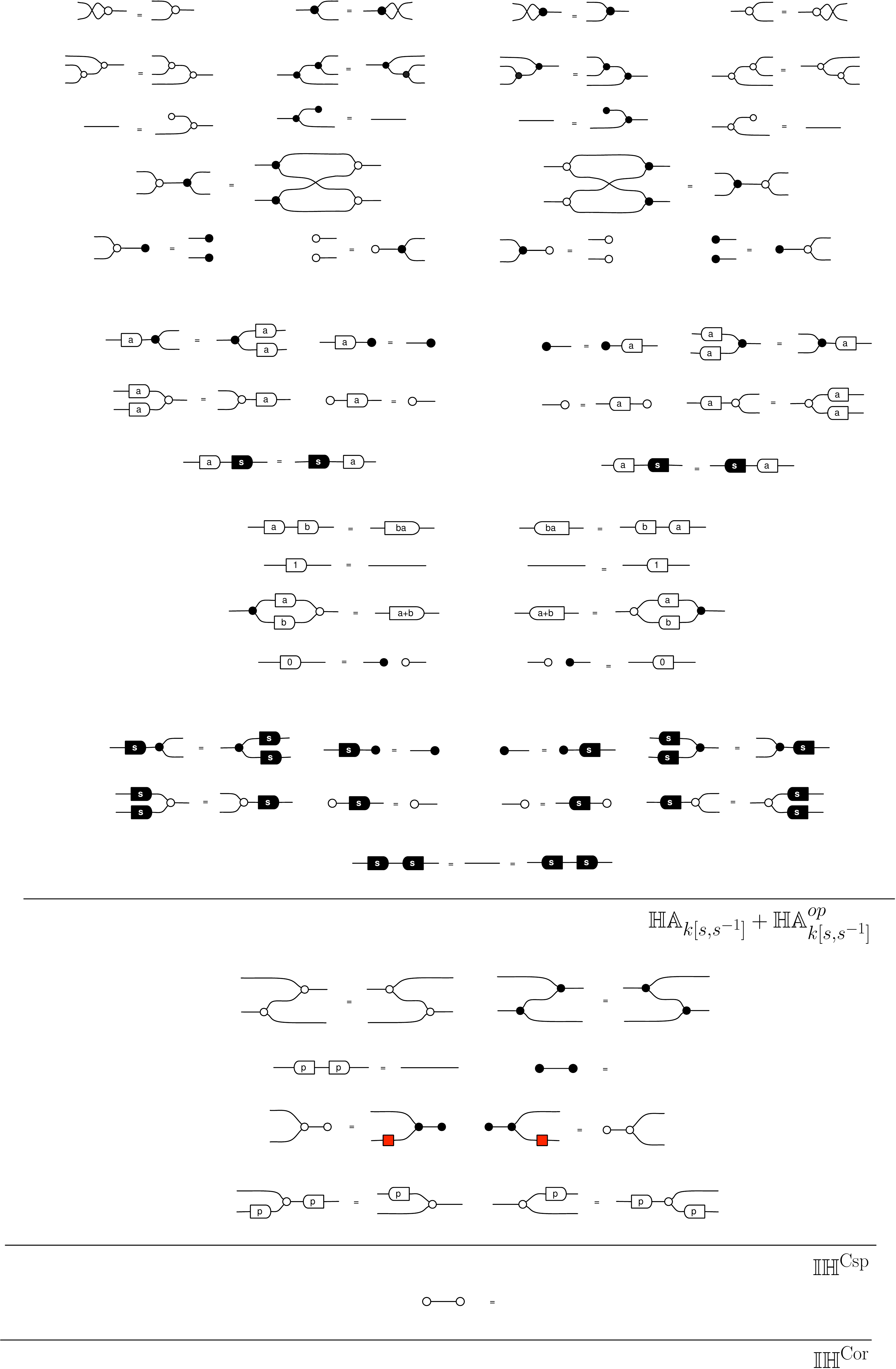}
  \caption{The axioms of $\ihcor$.}\label{fig.ihcoraxioms}
\end{figure}

\begin{remark} \label{rmk.omittedaxs}
 $\ihcor$ %(summarised in Appendix~\ref{app.equations}) 
 can also be described as having the equations of $\ih_\pk$ (see \cite{BSZ2,Za}), but
\emph{without} $\!\!\lower4pt\hbox{$\includegraphics[height=.5cm]{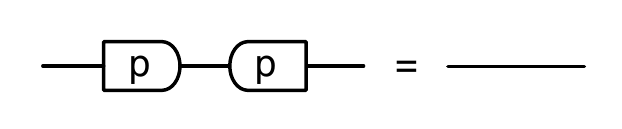}$}$,
$\!\lower7pt\hbox{$\includegraphics[height=.8cm]{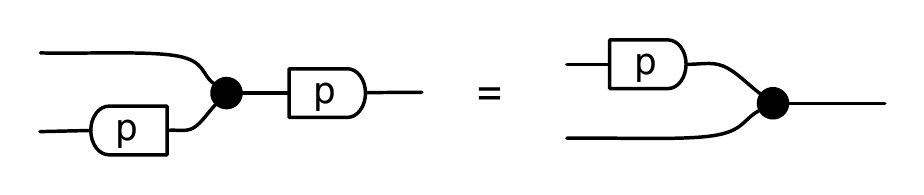}$}$,
and $\!\lower7pt\hbox{$\includegraphics[height=.8cm]{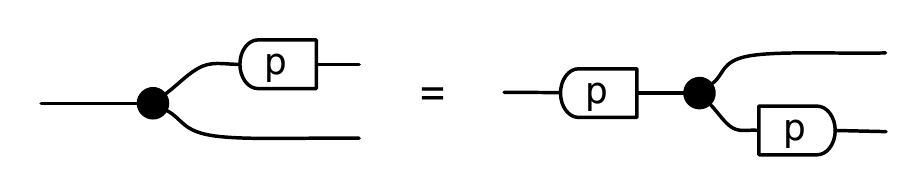}$}$.
Our results generalise; given any PID $R$ we have (informally speaking):
\begin{align*}
  \ihcor_R &= \ih_R -
  \left\{
    \begin{array}{c}
      \lower6pt\hbox{$\includegraphics[height=.6cm]{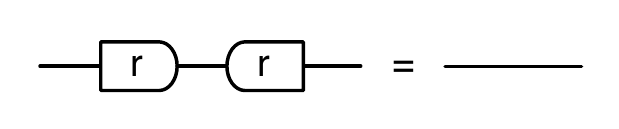}$}, \\
      \lower7pt\hbox{$\includegraphics[height=.7cm]{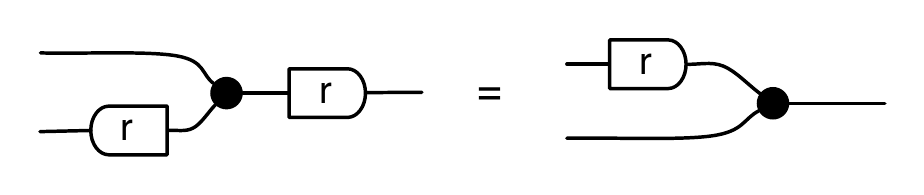}$},\\
      \lower7pt\hbox{$\includegraphics[height=.7cm]{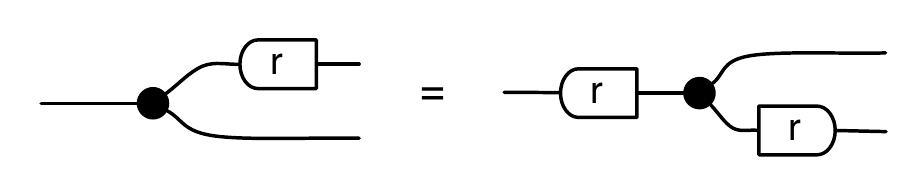}$}
    \end{array}
    \middle|\,r\neq 0\in R\,
  \right\} \\
  &\cong \corel \mat R
\end{align*}
and, because of the transpose duality of matrices:
\begin{align*}
  \ih^{\mathrm{Rel}}_R &= \ih_R - 
  \left\{
    \begin{array}{c} 
      \lower6pt\hbox{$\includegraphics[height=.6cm]{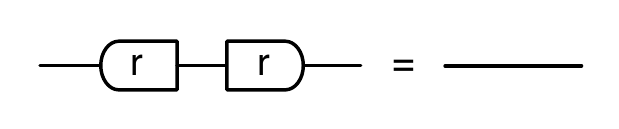}$},\\
      \lower7pt\hbox{$\includegraphics[height=.7cm]{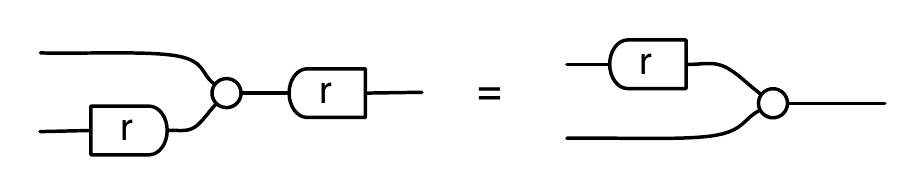}$}, \\
      \lower7pt\hbox{$\includegraphics[height=.7cm]{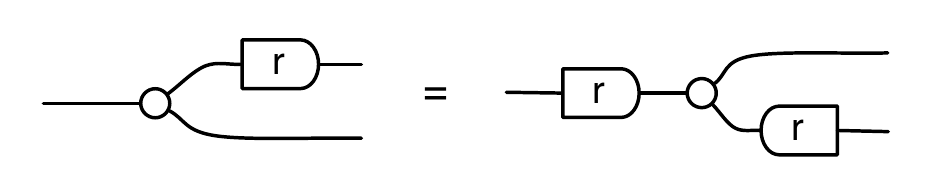}$}
    \end{array}
    \middle|\,r\neq 0\in R\,
  \right\}\\
  &\cong \mathrm{Rel} \mat R.
\end{align*}
\end{remark}

\begin{remark}\label{rmk:splus1}
The omitted equations each associate a cospan
with a span that, in terms of behaviour, has the effect of passing to a 
sub-behaviour\footnote{In fact the `largest controllable sub-behaviour' of the system. We explore 
controllability in Section \ref{sec.control}.}.
Often this is a strict sub-behaviour, hence the failure of soundness of $\ih$ identified in the introduction. 

For example, consider the system $\mathcal B$ represented by the cospan
\[
\lower6pt\hbox{$\includegraphics[height=.8cm]{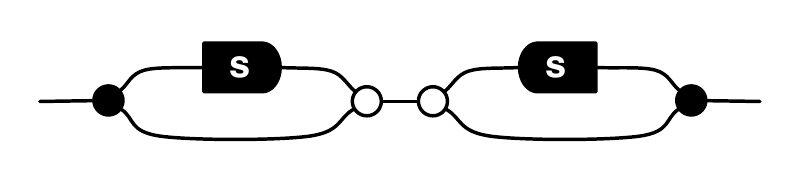}$}
\]
We met this system in the introduction; indeed the following
derivation can be performed in $\ihcor$:
\[
\includegraphics[height=1.3cm]{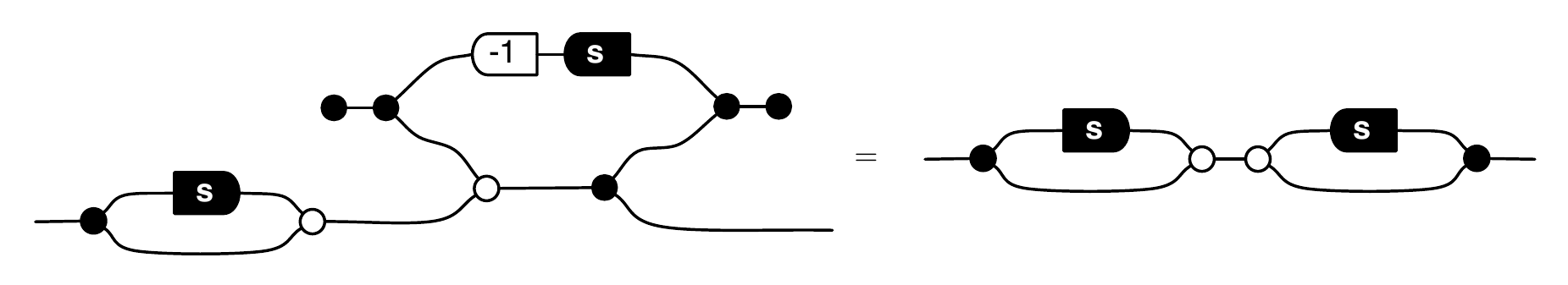}
\]
 The trajectories are $w= (w_1, w_2) \in \k^\z \oplus \k^\z$ where
$(s+1)\cdot w_1= (s+1)\cdot w_2$; that is, they satisfy the difference equation
\[ 
w_1(t-1)+w_1(t)-w_2(t-1)-w_2(t)=0.  
\]
As we saw in the introduction, however, an equation of $\ih_{\k[s]}$ omitted from
$\ihcor_{\k[s,s^{-1}]}$ equates this to the identity system $\id$, with the
identity behaviour those sequences of the form $w = (w_1,w_1) \in \k^\z \oplus
\k^\z$.  The identity behaviour is strictly smaller than $\mathcal B$; e.g., $\mathcal B$ contains $(w_1,w_2)$ with $w_1(t) = (-1)^t$
and $w_2(t) =0$.
\end{remark}

The similarity between our equational presentations of $\ihcor_{\k[s,s^{-1}]}$
and that of $\ih_{\k[s]}$  given in~\cite{BSZ,BSZ3} is remarkable, considering
the differences between the intended semantics of signal flow graphs that string
diagrams in those theories represent, as well as the underlying mathematics of
streams, which for us are elements of the $\k$ vector space $\k^{\mathbb{Z}}$
and in~\cite{BSZ,BSZ3} are Laurent series. We contend that this is evidence of
the \emph{robustness} of the algebraic approach: the equational account of how
various components of signal flow graphs interact is, in a sense, a higher-level
specification than the technical details of the underlying mathematical
formalisations.

\section{Operational semantics} \label{sec.opsem}

In this section we relate the denotational account given in previous sections
with an operational view. 

Operational semantics is given to $\Sigma^*$-terms---that is, to arrows of the
prop $\syntax^*=\mathbf{S}_{(\Sigma^*,\varnothing)}$---where $\Sigma^*$ is
obtained from set of generators in~\eqref{eq:generators} by replacing the formal
variables $s$ and $s^{-1}$ with a family of registers, indexed by the scalars of
$\k$:
\[
\delaygen \leadsto \{ \delaygenk | \, k\in\k \}
\]
\[
\delayopgen \leadsto \{ \delayopgenk |\, k\in\k \}
\]
The idea is that at any time during a computation the register holds the signal
it has received on the previous `clock-tick'. There are no equations, apart from
the laws of symmetric monoidal categories.

Next we introduce the structural rules: the transition relations that occur at
each clock-tick, turning one $\Sigma^*$-term into another. Each transition is given
two labels, written above and below the arrow.  The upper refers to the
signals observed at the `dangling wires' on the left-hand side of the term, and
the lower to those observed on the right hand side.  Indeed, transitions
out of a term of type $m \to n$ have upper labels in $\k^m$ and lower ones in
$\k^n$.

Because---for the purposes of the operational account---we consider these terms
to be syntactic, we must also account for the twist $\twist$ and identity $\id$.
A summary of the structural rules is given below; the rules for the mirror image
generators are symmetric, in the sense that upper and lower labels are swapped.

%\begin{figure*}[ht]
\[
\copygen \dtrans{\,k\,}{k \, k} \copygen \quad 
\discardgen \dtrans{k}{} \discardgen   
\]
\[
\addgen \dtrans{k \labelSep\, l }{k+l} \addgen \quad 
\zerogen \dtrans{\phantom{b}}{0} \zerogen
\]
%\[
%\copyopgen \! \dtrans{k \labelSep k}{k} \!\copyopgen \ 
%\discardopgen \dtrans{\phantom{k}}{k} \discardopgen \ 
%\addopgen\! \dtrans{ k+l}{k \labelSep l }\! \addopgen \ 
%\zeroopgen \dtrans{0}{} \zeroopgen
%\]
\[
\scalargen  \dtrans{\,\,l\,\,}{al} \scalargen \quad 
\delaygenl \dtrans{k}{l} \delaygenk 
%\scalaropgen  \dtrans{al}{\,l\,} \scalaropgen \ 
%\circuitXopT ^{\labelSep l}\! \dtrans{l}{k} \circuitXopT ^{\labelSep k} \quad \quad
\]
%\BmultT\! \dtrans{k \labelSep k}{k} \!\BmultT \quad\quad
%\BunitT \dtrans{\phantom{k}}{k} {\BunitT} \quad\quad
%\scalaropT  \dtrans{kl}{\,l\,} \scalaropT \quad \quad
%\circuitXopT ^{\labelSep l}\! \dtrans{l}{k} \circuitXopT ^{\labelSep k} \quad \quad
%\WcomultT\! \dtrans{ k+l}{k \labelSep l }\! \WcomultT\quad\quad
%\WcounitT \dtrans{0}{} \WcounitT
%\]
\[
\id  \dtrans{k}{k} \id \quad 
\twist \dtrans{k \labelSep l}{l \labelSep k} \twist 
\]
\[
  \frac{s\dtrans{\mathbf{u}}{\mathbf{v}} s' \quad
  t\dtrans{\mathbf{v}}{\mathbf{w}} t'}{s \mathrel{;} t \dtrans
  {\mathbf{u}}{\mathbf{w}} s' \mathrel{;} t'}
 \quad 
 \frac{s\dtrans{\mathbf{u_1}}{\mathbf{v_1}} s'\quad
 t\dtrans{\mathbf{u_2}}{\mathbf{v_2}} t'}
 {s\oplus t \dtrans{\mathbf{u_1 \labelSep u_2}}{\mathbf{v_1 \labelSep v_2}} s'\oplus t'}
\]
%\caption{Summary of structural rules for the operational semantics; $k,l\in\k$.\label{fig.opsem}}
%\end{figure*}
Here $k, l,a \in \k$, $s,t$ are $\Sigma^*$-terms, and
$\mathbf{u,v,w,u_1,u_2,v_1,v_2}$
are vectors in $\k$ of the appropriate length. Note that the only generators
that change as a result of computation are the registers; it follows this is the
only source of state in any computation.

Let $\tm\colon m\to n \in \syntax$ be a $\Sigma$-term. Fixing an ordering of
the delays $\delaygen$ in $\tm$ allows us to identify the set of delays with a
finite ordinal $[d]$.  A \define{register assignment} is then simply a function
$\sigma: [d]\to \k$.  We may instantiate the $\Sigma$-term $\tm$ with the
register assignment $\sigma$ to obtain the $\Sigma^\ast$-term $\tm_\sigma
\in\syntax^\ast$ of the same type: for all $i \in [d]$ simply replace the $i$th
delay with a register in state $\sigma(i)$.

A computation on $\tm$ initialised at $\sigma$ is an infinite sequence of register
assignments and transitions: 
\[
  \tm_\sigma \dtrans{\mathbf{u_1}}{\mathbf{v_1}} \tm_{\sigma_1}
  \dtrans{\mathbf{u_2}}{\mathbf{v_2}} \tm_{\sigma_2}  
  \dtrans{\mathbf{u_3}}{\mathbf{v_3}} \dots
\]
The \define{trace} induced by this computation is the sequence 
\[(\mathbf{u_1},\mathbf{v_1}), (\mathbf{u_2},\mathbf{v_2}), \dots\] 
of elements of $\k^m \oplus \k^n$.

\smallskip
To relate the operational and denotational semantics, we introduce the
notion of biinfinite trace: a trace with an infinite past
as well as future. To define these, we use the notion of a \emph{reverse} computation: 
a computation using the operational rules above, but with the rules for delay having their
  left and right hand sides swapped:
\[
\delaygenk \dtrans{k}{l} \delaygenl.
\]

\begin{definition}
  Given $\tm\colon m \to n \in \syntax$, a \define{biinfinite trace on} $\tm$ is a sequence
$w \in (\k^m)^\z \oplus (\k^n)^\z$ such that there exists
\begin{enumerate}[(i)]
\item a register assignment $\sigma$; 
\item an infinite \emph{forward trace} $\phi_\sigma$ of a computation on $\tm$
  initialised at $\sigma$; and,
\item an infinite \emph{backward trace} $\psi_\sigma$ of a reverse computation
  on $\tm$ initialised at $\sigma$,
\end{enumerate}
obeying
\[
w(t) = \begin{cases}
\phi(t) & t\geq 0; \\
\psi(-(t+1)) & t < 0.
\end{cases}
\]
We write $\bit(\tm)$ for the set of all biinfinite traces on $\tm$.
\end{definition}

The following result gives a tight correspondence between the operational and
denotational semantics, and follows via a straightforward structural induction
on $\tm$.
\begin{lemma}
For any $\tm\maps m\to n \in \syntax$, we have 
\[
  \llbracket \tm \rrbracket = \bit(\tm) 
\]
as subsets of $(\k^m)^\z\times (\k^n)^\z$.
\end{lemma}

\section{Controllability} \label{sec.control}

Suppose we are given a current and a target trajectory for a system. Is it
always possible, in finite time, to steer the system onto the target trajectory? If so, the 
 system is deemed controllable, and the problem of controllability of systems is at
the core of control theory. The following definition is due to
Willems~\cite{Wi2}.

\begin{definition}\label{def:contr}
A system $(T,W,\bb)$ (or simply the behaviour $\bb$) is \define{controllable} if
for all $w,w' \in \bb$ and all times $t_0\in\mathbb{Z}$, there exists $w'' \in
\bb$ and $t_0^\prime\in\mathbb{Z}$ such that $t_0'>t_0$ and $w''$ obeys
\[
  w''(t) = 
  \begin{cases} 
    w(t) & t \le t_0 \\
    w'(t-t_0') & t \ge t_0'.
  \end{cases}
\]
\end{definition}
%$\bb$ is controllable if given the past of some trajectory and the future of
%another, there exists a trajectory in the behaviour that transitions, in finite
%time, from the chosen past to the chosen future. 

As mentioned previously, a novel feature of our graphical calculus is
that it allows us to consider non-controllable behaviours.

\begin{example} \label{ex.noncontrol}
Consider the system in the introduction, further elaborated in Remark~\ref{rmk:splus1}.
%represented by the signal flow
%graph~\eqref{eq:examplesfg}. The equation 
%\[
%\includegraphics[height=1.3cm]{pics/firstequality.pdf}
%\]
%follows from the equations of $\ltids$, and so as an $\ltids$ system
%\eqref{eq:examplesfg} is the system represented by the corelation
%%\[
%$1 \xrightarrow{[s+1]} 1 \xleftarrow{[s+1]} 1.$
%%\]
%
%As discussed in Remark \ref{rmk.omittedaxs}, 
As noted previously,
the trajectories of this system are
precisely those sequences $w= (w_1, w_2) \in \k^\z \oplus \k^\z$ that satisfy
the difference equation
\[
  w_1(t-1)+w_1(t)-w_2(t-1)-w_2(t)=0.
\]
To see that the system is non-controllable, note that
\[
  w_1(t-1)-w_2(t-1) = -(w_1(t)-w_2(t)),
\]
so $(w_1-w_2)(t) = (-1)^tc_w$ for some $c_w \in \k$. This $c_w$ is a time-invariant
property of any trajectory. Thus if $w$ and $w'$ are trajectories such that $c_w \ne
c_{w'}$, then it is not possible to transition from the past of $w$ to the
future of $w'$ along some trajectory in $\mathcal B$. 

Explicitly, taking $w(t) = ((-1)^t,0)$ and $w'(t) = ((-1)^t2,0)$ suffices to
show $\mathcal B$ is not controllable.
\end{example}

\subsection{A categorical characterisation}
We now show that controllable systems are precisely those representable as
\emph{spans} of matrices. This novel characterisation leads to new ways of
reasoning about controllability of composite systems. 

Among the various equivalent conditions for controllability, the existence of
\emph{image representations} is most useful for our purposes.
\begin{proposition}[Willems {\cite[p.86]{Wi}}] \label{thm.imagereps}
  An LTI behaviour $\bb$ is controllable iff there exists $M \in \mat\pk$ such
  that $\bb = \im \theta M$.
\end{proposition}

Restated in our language, Proposition \ref{thm.imagereps} states that controllable
systems are precisely those representable as \emph{spans} of matrices. 

\begin{theorem} \label{cor.spanreps}
  Let $m \xrightarrow{A} d \xleftarrow{B} n$ be a corelation in $\corel\mat\pk$.
  Then the LTI behaviour
  $\Phi(\xrightarrow{A}\xleftarrow{B})$ is controllable iff there exists $R: e
  \to m$, $S: e\to n$ such that 
  \[
    m \xleftarrow{R} e \xrightarrow{S} n = m \xrightarrow{A} d \xleftarrow{B} n
  \]
  as morphisms in $\corel\mat\pk$. 
\end{theorem}
\begin{proof}
  To begin, note that the behaviour of a span is its joint image. That is,
  $\Phi(\xleftarrow{R}\xrightarrow{S})$ is the composite of linear
  relations $\ker\vectfun[\mathrm{id}_m \; -R]$ and $\ker\vectfun[S\;
  -\mathrm{id}_n]$, which comprises all $(\mathbf{x},\mathbf{y}) \in (\k^m)^\z
  \oplus (\k^n)^\z$ s.t.\ $\exists$ $\mathbf{z} \in (\k^e)^\z$ with
  $\mathbf{x} = \vectfun R \mathbf{z}$ and $\mathbf{y} = \vectfun S
  \mathbf{z}$. Thus
  \[
    \Phi(\xleftarrow{R}\xrightarrow{S}) = \im \vectfun \left[
    \begin{matrix} R \\ S \end{matrix} \right].
  \]
  The result then follows immediately from Proposition \ref{thm.imagereps}.  
\end{proof}

In terms of the graphical theory, this means that a term in the 
form $\ha_\pk ; \ha_\pk^{op}$ (`cospan form') is controllable iff we can find a
derivation, using the rules of $\ihcor$, that puts it in the form $\ha_\pk^{op}
; \ha_\pk$ (`span form').  This provides a general, easily recognisable
representation for controllable systems. 

Span representations also lead to a test for controllability: take
the pullback of the cospan and check whether the system described by it
coincides with the original one. Indeed, note that as $\pk$ is a PID, the
category $\mat\pk$ has pullbacks. A further consequence of Theorem
\ref{cor.spanreps}, together with Proposition  \ref{prop.magic}, is the following. 

\begin{proposition} \label{prop.ctrlablepart}
  Let $m \xrightarrow{A} d \xleftarrow{B} n$ be a cospan in $\mat\pk$, and write
  the pullback of this cospan $m \xleftarrow{R} e \xrightarrow{S} n$. Then the
  behaviour of the pullback span $\Phi(\xleftarrow{R}\xrightarrow{S})$ is
  the maximal controllable sub-behaviour of
  $\Phi(\xrightarrow{A}\xleftarrow{B})$.
\end{proposition}
\begin{proof}
  Suppose we have another controllable behaviour $\mathscr{C}$ contained in
  $\ker\vectfun [A\;-B]$. Then this behaviour is the $\Phi$-image of some span
  $m \xleftarrow{R'} e' \xrightarrow{S'}n$. As $\im\vectfun\begin{bmatrix} R' \\
    S'\end{bmatrix}$ lies in $\ker\vectfun [A\;-B]$, the universal property of
  the pullback gives a map $e' \to e$ such that the relevant diagram commutes.
  This implies that the controllable behaviour $\mathscr{C} =
  \im\vectfun\begin{bmatrix} R' \\ S'\end{bmatrix}$ is contained in $\im\vectfun
  \begin{bmatrix} R \\ S\end{bmatrix}$, as required. 
\end{proof}

\begin{corollary}
  Suppose that an LTI behaviour $\bb$ has cospan representation
  \[
    m \stackrel{A}\longrightarrow d \stackrel{B}\longleftarrow n.
  \]
  Then $\bb$ is controllable iff the $\Phi$-image of the pullback of this cospan
  in $\mat\pk$ is equal to $\bb$.
\end{corollary}

Moreover, taking the pushout of this pullback span gives another cospan. The
morphism from the pushout to the original cospan, given by the universal
property of the pushout, describes the way in which the system fails to be
controllable.

To continue Remark \ref{rmk.omittedaxs}, the theory of interacting Hopf
algebras $\ih_\pk$~\cite{BSZ2,Za} may be viewed as our theory $\ihcor$ of LTI
systems together with the axioms given by the pullback in $\mat\pk$. Thus,
graphically, the pullback may be computed by using the axioms of $\ih_\pk$. For
example, the pullback span of the system of Example \ref{ex.noncontrol} is
simply the identity span, as derived in equation \eqref{eq:exampleproof} of the
previous section. In the traditional matrix calculus for control theory, one
derives this by noting the system has kernel representation
$\ker\theta\begin{bmatrix} s+1 & -(s+1) \end{bmatrix}$, and eliminating the
common factor $s+1$ between the entries.  Either way, we conclude that the
maximally controllable subsystem of $1 \xrightarrow{[s+1]} 1 \xleftarrow{[s+1]}
1$ is simply the identity system $1 \xrightarrow{[1]} 1 \xleftarrow{[1]} 1$.

% Brendan, I think we should skip this for now, we can bring it back for the journal version
%\begin{proposition}
%  Suppose we have $f\maps n \to m$ and $g\maps n \leftarrow m$ are equal as
%  corelations. Then $g = f^{-1}$.
%\end{proposition}
%\begin{proof}
%  This means the following commutes:
%  \[
%    \xymatrix{
%      & n \\
%      n \ar@{=}[ur] \ar[dr]_f & & m \ar[ul]_g \ar@{=}[dl] \\
%      & m
%    }
%  \]
%\end{proof}

\subsection{Control and interconnection}
From this vantage point we can make useful observations about controllable
systems and their composites: we simply need to ask whether we can rewrite them
as spans. 

\begin{example}
  Suppose that $\bb$ has cospan representation
  %\[
    $m \xrightarrow{A} d \xleftarrow{B} n.$
  %\]
  Then $\bb$ is easily seen to be controllable when $A$ or $B$ is invertible.
  Indeed, if $A$ is invertible, then $m  \xleftarrow{A^{-1}B} n
  \xrightarrow{\idn_n} n$ is an equivalent span; if $B$ is invertible, then
  $m\xleftarrow{\idn_m} m \xrightarrow{B^{-1}A} n$.
\end{example}

More significantly, the compositionality of our framework aids understanding of
how controllability behaves under the interconnection of systems---an active
field of investigation in current control theory. We give an example application
of our result.

First, consider the following proposition.
\begin{proposition}\label{prop:veryexciting}
  Let $\bb,\mathscr{C}$ be controllable systems, given by the respective
  $\Phi$-images of the spans
  %\[
    $m \xleftarrow{B_1} d \xrightarrow{B_2} n$ % \qquad \mbox{$
    and %} \qquad
    $n \xleftarrow{C_1} e \xrightarrow{C_2} l$.
  %\]
  Then the composite $\mathscr{C} \circ \bb: m \to l$ is controllable
  if $\Phi(\xrightarrow{B_2}\xleftarrow{C_1})$ is
  controllable.
\end{proposition}
\begin{proof}
  Replacing $\xrightarrow{B_2}\xleftarrow{C_1}$ with an equivalent span gives a span
  representation for $\mathscr{C} \circ \bb$.
\end{proof}

\begin{example}
  Consider LTI systems
  \[
  \lower25pt\hbox{$\includegraphics[height=2cm]{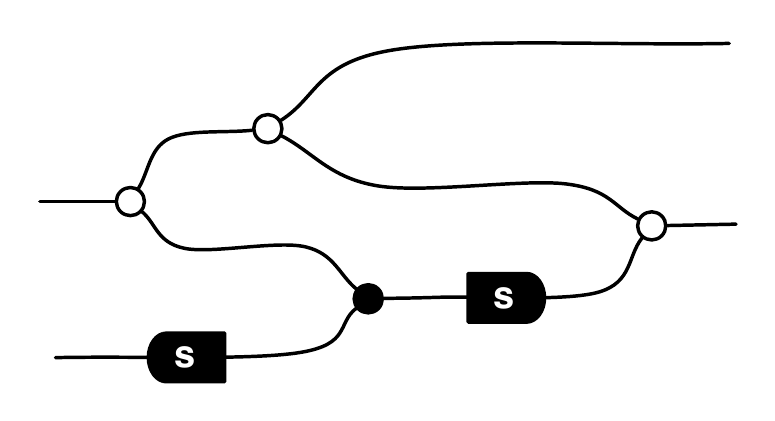}$}
  \quad\text{and}\quad
  \lower22pt\hbox{$\includegraphics[height=1.8cm]{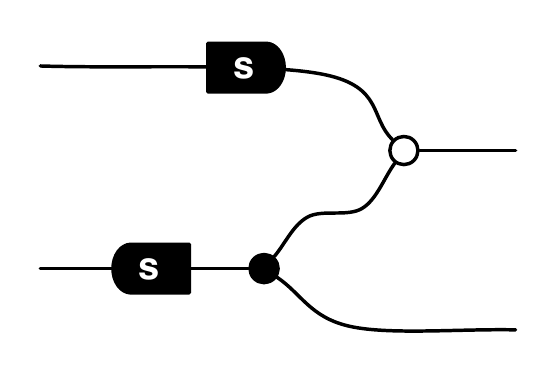}$}.
  \]
  These systems are controllable because each is represented by a span in
  $\mat\pk$. Indeed, recall that each generator of $\ltids=\corel\mat\pk$ arises
  as the image of a generator in $\mat\pk$ or $\mat\pk^{\mathrm{op}}$; for
  example, the white monoid map $\addgen$ represents a morphism in $\mat\pk$,
  while the black monoid map $\copyopgen$ represents a morphism in
  $\mat\pk^{\mathrm{op}}$. The above diagrams are spans as we may partition the
  diagrams above so that each generator in $\mat\pk^{\mathrm{op}}$ lies to the
  left of each generator in $\mat\pk$.

  To determine controllability of the interconnected system
  \[
  \lower25pt\hbox{$\includegraphics[height=2.6cm]{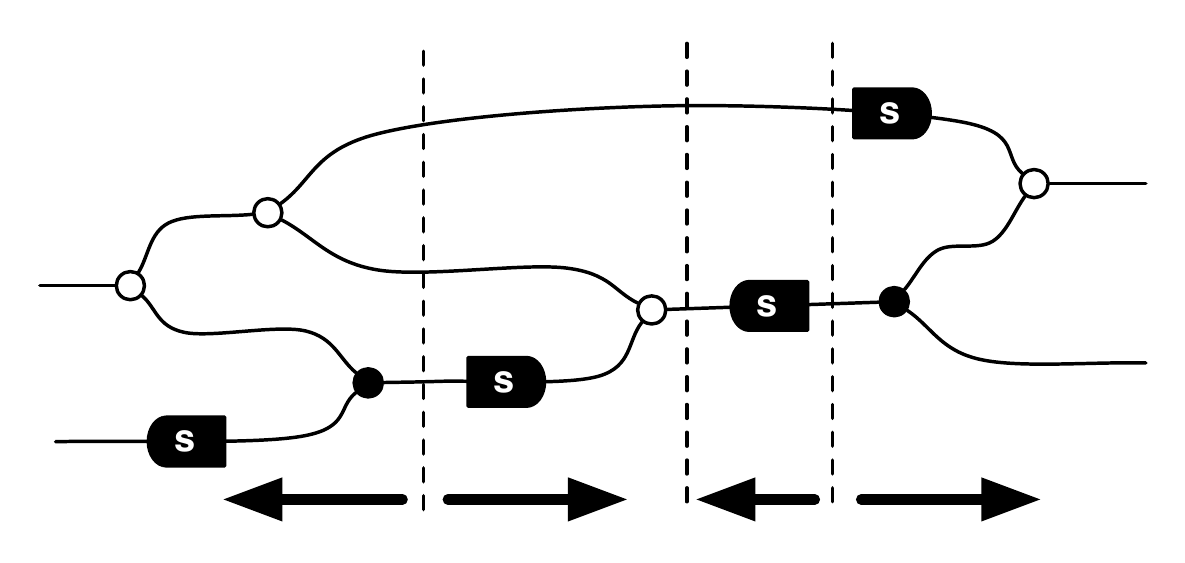}$}
  \]
  Proposition \ref{prop:veryexciting} states that it is enough to consider the
  controllability of the subsystem
  \[
  \lower17pt\hbox{$\includegraphics[height=1.4cm]{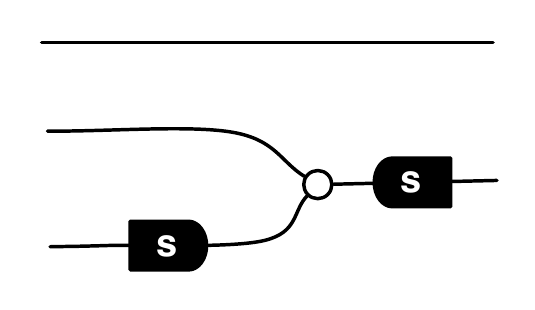}$}.
  \]
  The above diagram gives a representation of the subsystem as a cospan in
  $\mat\pk$. We can prove it is controllable by rewriting it as a span using an
  equation of $\ltids$:
  \[
  \lower25pt\hbox{$\includegraphics[height=2.2cm]{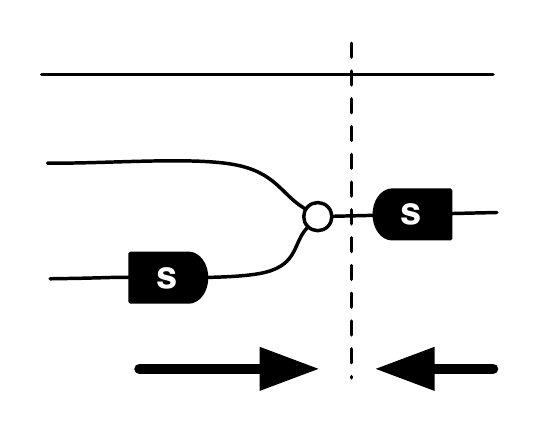}$}
    =
  \lower25pt\hbox{$\includegraphics[height=2.2cm]{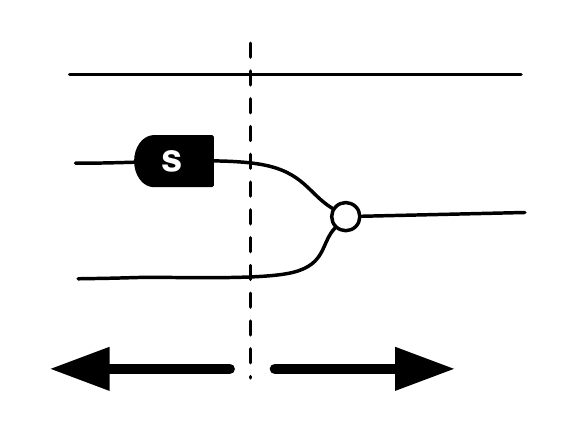}$}.
  \]
  Thus the composite system is controllable.
\end{example}

\begin{remark}
  Note the converse of Proposition \ref{prop:veryexciting} fails. For a simple
  example, consider the system 
  \[
  \lower10pt\hbox{$\includegraphics[height=1cm]{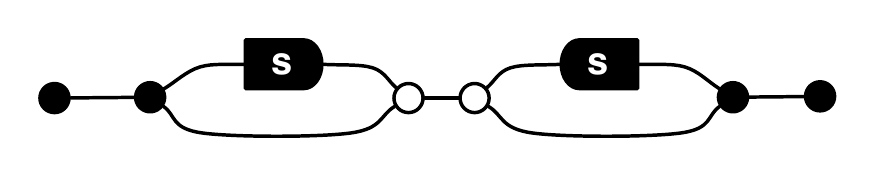}$}
  \]
  This is equivalent to the empty system, and so trivially controllable. The
  central span, however, is not controllable (Example \ref{ex.noncontrol}).
\end{remark}

\subsection{Comparison to matrix methods}
The facility with which the graphical calculus formalises and solves such
controllability issues is especially appealing in view of potential applications
in the analysis of controllability of \emph{systems over networks} (see
\cite{OFM}). To make the reader fully appreciate such potential, we sketch how
complicated such analysis is using standard algebraic methods and
dynamical system theory even for the highly restrictive case of two systems that
compose to make a single-input, single-output system.  See also pages 513 to 516
of Fuhrmann and Helmke's recent book \cite{FH}, where a generalization of the
result of Proposition \ref{prop:veryexciting} is given in a polynomial- and
operator-theoretic setting. 

In the following we abuse notation by writing a matrix for its image under the
functor $\theta$.  The following is a useful result for analysing the
controllability of kernel representations applying only to the single-input
single-output case.

\begin{proposition}[Willems {\cite[p.75]{Wi}}] \label{prop.controlkernel}
  Let $\mathscr B \subseteq (\k^2)^\z$ be a behaviour given by the kernel of the
  matrix $\theta\begin{bmatrix} A & B\end{bmatrix}$, where $A$ and $B$ are column
  vectors with entries in $\k[s,s^{-1}]$. Then $\mathscr B$ is controllable if
  and only if the greatest common divisor $\gcd(A,B)$ of $A$ and $B$ is $1$.
\end{proposition}

Using the notation of Proposition \ref{prop:veryexciting}, the trajectories of
$\mathscr{B}$ and $\mathscr{C}$ respectively are those $(w_1,w_2) \in (\k^n)^\z
\oplus (\k^m)^\z$ and $(w_2',w_3) \in (\k^m)^\z \oplus (\k^p)^\z$ such that 
\begin{equation}\label{eq:imBnC}
\begin{bmatrix} w_1\\w_2 \end{bmatrix}
=
\begin{bmatrix} B_1\\B_2 \end{bmatrix} \ell_1 
\quad \mbox{ \rm and} \quad  
\begin{bmatrix} w_2^\prime\\w_3 \end{bmatrix}
=
\begin{bmatrix} C_1\\C_2 \end{bmatrix} \ell_2\; 
\end{equation}
for some $\ell_1 \in (\k^b)^\z$, $\ell_2 \in (\k^c)^\z$. These are the explicit
image representations of the two systems.  We assume without loss
of generality that the representations (\ref{eq:imBnC}) are \emph{observable}
(see \cite{Wi}); this is equivalent to $\gcd(B_1,B_2)=\gcd(C_1,C_2)=1$. Augmenting
(\ref{eq:imBnC}) with the interconnection constraint
$w_2=B_2\ell_1=C_1\ell_2=w_2^\prime$ we obtain the representation of the
interconnection: 
\begin{eqnarray}\label{eq:hybBnC}
\begin{bmatrix}
w_1\\w_2\\w_2^\prime\\w_3\\0
\end{bmatrix}&=&\begin{bmatrix} B_1&0\\B_2&0\\0&C_1\\ 0&C_2 \\ B_2&-C_1 \end{bmatrix} \begin{bmatrix} \ell_1\\\ell_2\end{bmatrix}\; .
\end{eqnarray}
Proposition \ref{prop:veryexciting} concerns the controllability of the set
$\mathscr{C} \circ \mathscr{B}$ of trajectories $(w_1,w_3)$ for which there
exist trajectories $w_2$, $w_2^\prime$, $\ell_1$, $\ell_2$ such that
(\ref{eq:hybBnC}) holds. 

To obtain a representation of such behaviour the variables $\ell_1$, $\ell_2$,
$w_2$ and $w_2^\prime$ must be eliminated from (\ref{eq:hybBnC}) via algebraic
manipulations (see the discussion on p. 237 of \cite{Wi2}). Denote
$G=\gcd(B_1,C_2)$, and write $C_2=G C_2^\prime$ and $B_1=G B_1^\prime$, where
$\gcd(B_1^\prime, C_2^\prime)=1$.  Without entering in the algebraic details, it
can be shown that a kernel representation of the projection of the behaviour of
(\ref{eq:hybBnC}) on the variables $w_1$ and $w_3$ is
\begin{equation}\label{eq:extafterelim}
  \begin{bmatrix} C_2^\prime B_2& -B_1^\prime C_2\end{bmatrix} 
  \begin{bmatrix} w_1\\ w_3\end{bmatrix}=0\; .
\end{equation}
We now restrict to the single-input single-output case. Recalling Proposition
\ref{prop.controlkernel}, the behaviour represented by (\ref{eq:extafterelim}) is
controllable if and only if $\gcd(C_2^\prime B_2, B_1^\prime C_1)=1$. 

Finally then, to complete our alternate proof of the single-input single-output
case of Proposition \ref{prop:veryexciting}, note that
$\cospanfunrest(\xrightarrow{B_2}\xleftarrow{C_1})$ is controllable if
$\gcd(B_2, C_1)=1$.  Given the observability assumption, this implies
$\gcd(C_2^\prime B_2, B_1^\prime C_1)=1$, and so the interconnected behaviour
$\mathscr{C} \circ \mathscr{B}$ represented by (\ref{eq:extafterelim}) is
controllable. 

In the multi-input, multi-output case stating explicit conditions on the
controllability of the interconnection given properties of the
representations of the individual systems and their interconnection is rather
complicated. This makes the simplicity of Proposition \ref{prop:veryexciting} and the
straightforward nature of its proof all the more appealing.

\chapter{Passive linear networks} \label{ch.circuits}
This chapter is about using decorated corelations to build semantic functors.

A thorough introduction to this chapter can be found in the next section. Here
we give a quick overview of the structure of this chapter. The first focus is
the traditional, closed semantics of circuit diagrams. We begin in
\textsection\ref{sec:resistors} with a discussion of circuits of linear
resistors, developing the intuition for the governing laws of passive linear
circuits---Ohm's law, Kirchhoff's voltage law, and Kirchhoff's current law---in
a time-independent setting, and showing that Dirichlet forms represent circuit
behaviours.  In \textsection\ref{sec:plcs}, the Laplace transform then allows us
to recapitulate these ideas after introducing inductors and capacitors, speaking
of impedance where we formerly spoke of resistance, and generalizing Dirichlet
forms from the field $\R$ to the field $\R(s)$ of real rational functions. 

The goal of this applications chapter is to show how decorated corelations can
make these semantics compositional. As a network-style diagrammatic language, it
should thus be no surprise that we begin by using decorated cospans to construct
a hypergraph category of circuits (\textsection\ref{sec:circdef}). At the end of
this section, however, we show that Dirichlet forms do not provide the
flexibility to construct a semantic category for circuits.  This motivates the
development of more powerful machinery.

In \textsection\ref{sec:circlagr} we review the basic theory of linear
Lagrangian relations, giving details to the correspondence we have defined
between Dirichlet forms, and hence passive linear circuits, and Lagrangian
relations. Subsection \ref{sec:corel} then takes immediate advantage of the
added flexibility of Lagrangian relations, discussing the `trivial' circuits
comprising only perfectly conductive wires, which mediate the notion of
composition of circuits.

With these application-specific tools, we can use decorated corelations to prove
the functoriality of black-boxing circuits in \textsection\ref{sec:blackbox}.

\section{Introduction}\label{sec:intro}
%%fakesubsection
In late 1940s, just as Feynman was developing his diagrams for processes in
particle physics, Eilenberg and Mac Lane initiated their work on category
theory.  Over the subsequent decades, and especially in the aforementioned work
of Joyal and Street \cite{JS91,JS93}, it became clear that these developments were
profoundly linked: as we have seen, monoidal categories have a precise graphical
representation in terms of string diagrams, and conversely monoidal categories
provide an algebraic foundation for the intuitions behind Feynman diagrams
\cite{BaezStay}.  The key insight is the use of categories where morphisms
describe physical processes, rather than structure-preserving maps between
mathematical objects \cite{CP}.

But even before physicists began using Feynman diagrams, various branches of
engineering were using diagrams that in retrospect are closely related. In the
last chapter we investigated signal flow diagrams. But even before these, and
even more well known, are the ubiquitous electrical circuit diagrams. These are
broad in their application, and similar diagrams are used to describe networks
consisting of mechanical, hydraulic, thermodynamic and chemical systems.
Further work, pioneered in particular by Forrester \cite{Fo} and Odum \cite{Od},
even applies similar diagrammatic methods to biology, ecology, and economics.

As discussed in detail by Olsen \cite{Ol}, Paynter \cite{Pa} and others, there
are mathematically precise analogies between these different systems. We give a
few examples in Table \ref{tab.analogy}. In each case, the system's state is
described by variables that come in pairs, with one variable in each pair
playing the role of  `displacement' and the other playing the role of
`momentum'.  In engineering, the time derivatives of these variables are
sometimes called `flow' and `effort'.    In classical mechanics, this pairing of
variables is well understood using symplectic geometry.  Thus, any mathematical
formulation of the diagrams used to describe networks in engineering needs to
take symplectic geometry as well as category theory into account.

\begin{table}[H]
  \caption {Analogies between variables in various physical settings}
  \label{tab.analogy} 
\begin{scriptsize}
\begin{center}
  \begin{tabular}{|c|c c c c|}
\hline
& displacement  &  flow & momentum & effort \\
& $q$ & $\dot{q}$ & $p$ & $\dot{p}$ \\
\hline
Electronics & charge & current & flux linkage & voltage\\
Mechanics (translation) & position & velocity & momentum & force\\
Mechanics (rotation) & angle & angular velocity & angular momentum & torque\\
Hydraulics & volume & flow & pressure momentum & pressure\\
Thermodynamics & entropy & entropy flow & temperature momentum & temperature \\
Chemistry & moles & molar flow & chemical momentum & chemical potential \\
\hline
\end{tabular}
\end{center}
\end{scriptsize}
\end{table}

Although we shall keep the broad applicability of network diagrams in the back
of our minds, we couch our discussion in terms of electrical circuits, for the
sake of familiarity. In this section our goal is somewhat limited.  We only study circuits built from `passive' components: that is, those that do not produce energy.  Thus, we exclude batteries and current sources.  We only consider components that respond linearly to an applied voltage.   Thus, we exclude components such as nonlinear resistors or diodes.  Finally, we only consider components with one input and one output, so that a circuit can be described as a graph with edges labeled by components.  Thus, we also exclude transformers.  The most familiar components our framework covers are linear resistors, capacitors and inductors.

While we hope to expand our scope in future work, the class of circuits made from these components has appealing mathematical properties, and is worthy of deep study.  Indeed, this class has been studied intensively for many decades by electrical engineers \cite{AV,Budak,Slepian}.  Even circuits made exclusively of resistors have inspired work by mathematicians of the caliber of Weyl \cite{Weyl} and Smale \cite{Smale}.  

The present work relies on this research.  All we add here is an emphasis on symplectic geometry and an explicitly `compositional' framework, which clarifies the way a larger circuit can be built from smaller pieces.  This is where monoidal categories become important: the main operations for building circuits from pieces are composition and tensoring.
 
Our strategy is most easily illustrated for circuits made of linear resistors.  Such a resistor dissipates power, turning useful energy into heat at a rate determined by the voltage across the resistor.  However, a remarkable fact is that a circuit made of these resistors always acts to \emph{minimize} the power dissipated this way.  This `principle of minimum power' can be seen as the reason symplectic geometry becomes important in understanding circuits made of resistors, just as the principle of least action leads to the role of symplectic geometry in classical mechanics.  

Here is a circuit made of linear resistors:
\[
\begin{tikzpicture}[circuit ee IEC, set resistor graphic=var resistor IEC graphic]
\node[contact] (I1) at (0,2) {};
\node[contact] (I2) at (0,0) {};
\node[contact] (O1) at (5.83,1) {};
\node(input) at (-2,1) {\small{\textsf{inputs}}};
\node(output) at (7.83,1) {\small{\textsf{outputs}}};
\draw (I1) 	to [resistor] node [label={[label distance=2pt]85:{$3\Omega$}}] {} (2.83,1);
\draw (I2)	to [resistor] node [label={[label distance=2pt]275:{$1\Omega$}}] {} (2.83,1)
				to [resistor] node [label={[label distance=3pt]90:{$4\Omega$}}] {} (O1);
\path[color=gray, very thick, shorten >=10pt, ->, >=stealth, bend left] (input) edge (I1);		\path[color=gray, very thick, shorten >=10pt, ->, >=stealth, bend right] (input) edge (I2);		
\path[color=gray, very thick, shorten >=10pt, ->, >=stealth] (output) edge (O1);
\end{tikzpicture}
\]
The wiggly lines are resistors, and their resistances are written beside them: for example,
$3\Omega$ means 3 ohms, an `ohm' being a unit of resistance.  To formalize this, define a circuit of linear resistors to consist of:
\begin{itemize}
\item a set $N$ of nodes,
\item a set $E$ of edges, 
\item maps $s,t \maps E \to N$ sending each edge to its source and target node,
\item a map $r\maps E \to (0,\infty)$ specifying the resistance of the resistor 
labelling each edge, 
\item maps $i \maps X \to N$, $o \maps Y \to N$ specifying the
inputs and outputs of the circuit.
\end{itemize}

When we run electric current through such a circuit, each node $n \in N$ gets
a `potential' $\phi(n)$.  The `voltage' across an edge $e \in E$ is defined as the 
change in potential as we move from to the source of $e$ to its target, $\phi(t(e)) - 
\phi(s(e))$, and the power dissipated by the resistor on this edge equals
\[      
\frac{1}{r(e)}\big(\phi(t(e))-\phi(s(e))\big)^2. 
\]
The total power dissipated by the circuit is therefore twice
\[   
P(\phi) = \frac{1}{2}\sum_{e \in E} \frac{1}{r(e)}\big(\phi(t(e))-\phi(s(e))\big)^2.
\]
The factor of $\frac{1}{2}$ is convenient in some later calculations.  
Note that $P$ is a nonnegative quadratic form on the vector space $\R^N$.
However, not every nonnegative definite quadratic form on $\R^N$ arises in this way from some circuit of linear resistors with $N$ as its set of nodes.  The quadratic forms that do arise are called `Dirichlet forms'.  They have been extensively investigated \cite{Fukushima,MR,Sabot1997,Sabot2004}, and they play a major role in our work.

We write $\partial N = i(X) \cup o(Y)$ for the set of `terminals': that is,
nodes corresponding to inputs and outputs.    The principle of minimum
power says that if we fix the potential at the terminals, the circuit will choose
the potential at other nodes to minimize the total power dissipated.   
An element $\psi$ of the vector space $\R^{\partial N}$ assigns a potential 
to each terminal.   Thus, if we fix $\psi$, the total power dissipated will be twice
\[
  Q(\psi) = \min_{\substack{ \phi \in \R^N \\ \phi\vert_{\partial N} = \psi}} \; P(\phi)  
\]
The function $Q \maps \R^{\partial N} \to \R$ is again a Dirichlet form.  We call it the `power functional' of the circuit.  

Now, suppose we are unable to see the internal workings of a circuit, and can only observe its `external behaviour': that is, the potentials at its terminals and the currents flowing into or out of these terminals.   Remarkably, this behaviour is completely determined by the power functional $Q$.  The reason is that the current at any terminal can be obtained by differentiating $Q$ with respect to the potential at this terminal, and relations of this form are \emph{all} the relations that hold between potentials and currents at the terminals.

The Laplace transform allows us to generalize this immediately to circuits that
can also contain linear inductors and capacitors, simply by changing the field we work over, replacing $\R$ by the field $\R(s)$ of rational functions of a single real variable,
and talking of `impedance' where we previously talked of resistance.  We obtain
a category $\Circ$ where, roughly speaking, an object is a finite set, a
morphism $X \to Y$ is a circuit with input set $X$ and output set $Y$, and
composition is given by identifying the outputs of one circuit with the inputs
of the next, and taking the resulting union of labelled graphs.  Each such circuit gives rise to a Dirichlet form, now defined over
$\R(s)$, and this Dirichlet form completely describes the externally observable
behaviour of the circuit.  

We can take equivalence classes of circuits, where two circuits count as the
same if they have the same Dirichlet form.  We wish for these equivalence classes of circuits to form a category. Although
there is a notion of composition for Dirichlet forms, we find that it lacks
identity morphisms or, equivalently, it lacks morphisms representing ideal wires
of zero impedance. To address this we turn to Lagrangian subspaces of
symplectic vector spaces.  These generalize quadratic forms via the map
\[
  \Big(Q\maps \F^{\partial N} \to \F\Big) \longmapsto \Big(\mathrm{Graph}(dQ) =
  \{(\psi, dQ_\psi) \mid \psi \in \F^{\partial N} \} \subseteq \F^{\partial
  N} \oplus (\F^{\partial N})^\ast\Big)
\]
taking a quadratic form $Q$ on the vector space $\F^{\partial N}$
over a field $\F$ to the graph of its differential $dQ$. Here we think of the
symplectic vector space $\F^{\partial N} \oplus (\F^{\partial N})^\ast$ as the
state space of the circuit, and the subspace $\mathrm{Graph}(dQ)$ as the
subspace of attainable states, with $\psi \in \F^{\partial N}$ describing the
potentials at the terminals, and $dQ_\psi \in (\F^{\partial N})^\ast$ the
currents. 

This construction is well known in classical mechanics \cite{Weinstein}, where the principle of least action plays a role analogous to that of the principle of minimum power here.   The set of Lagrangian subspaces is actually an algebraic variety,
the `Lagrangian Grassmannian', which serves as a compactification of the
space of quadratic forms.  The Lagrangian Grassmannian has already played a
role in Sabot's work on circuits made of resistors \cite{Sabot1997,Sabot2004}.
For us, its importance it that we can find identity morphisms
for the composition of Dirichlet forms by taking circuits made of parallel resistors
and letting their resistances tend to zero: the limit is not a Dirichlet form, but
it exists in the Lagrangian Grassmannian.    Indeed, 
there exists a category $\LagrRel$ with finite dimensional
symplectic vector spaces as objects and `Lagrangian relations' as morphisms: 
that is, linear relations from $V$ to $W$ that are given by Lagrangian subspaces of $\overline{V} \oplus W$, where $\overline{V}$ is the symplectic vector space conjugate to $V$.   

To move from the Lagrangian subspace defined by the graph of the differential of
the power functional to a morphism in the category $\LagrRel$---that
is, to a Lagrangian relation---we must treat seriously the input and output
functions of the circuit. These express the circuit as built upon a cospan   
\[
  \xymatrix{
    & N \\
    X \ar[ur]^{i} && Y. \ar[ul]_o
  }
\]
Applicable far more broadly than this present formalization of circuits, cospans
model systems with two `ends', an input and output end, albeit without any
connotation of directionality: we might just as well exchange the role of the
inputs and outputs by taking the mirror image of the above diagram. The role of
the input and output functions, as we have discussed, is to mark the terminals
we may glue onto the terminals of another circuit, and the pushout of cospans
gives formal precision to this gluing construction.

One upshot of this cospan framework is that we may consider circuits with elements
of $N$ that are both inputs and outputs, such as this one:
\[
  \begin{tikzpicture}[circuit ee IEC, set resistor graphic=var resistor iec graphic]
    \node[contact] (c1) at (0,2) {};
    \node[contact] (c2) at (0,0) {};
    \node(input) at (-2,1) {\small{\textsf{inputs}}};
    \node(output) at (2,1) {\small{\textsf{outputs}}};
    \path[color=gray, very thick, shorten >=10pt, ->, >=stealth, bend left]
    (input) edge (c1);		
    \path[color=gray, very thick, shorten >=10pt, ->, >=stealth, bend right]
    (input) edge (c2);	
    \path[color=gray, very thick, shorten >=10pt, ->, >=stealth, bend right]
    (output) edge (c1);
    \path[color=gray, very thick, shorten >=10pt, ->, >=stealth, bend left]
    (output) edge (c2);
  \end{tikzpicture}
\]
This corresponds to the identity morphism on the finite set with two elements.
Another is that some points may be considered an input or output multiple
times; we draw this:
\[
  \begin{tikzpicture}[circuit ee IEC, set resistor graphic=var resistor IEC graphic]
    \node[contact] (I1) at (0,0) {};
    \node[contact] (O1) at (3,0) {};
    \node(input) at (-2,0) {\small{\textsf{inputs}}};
    \node(output) at (5,0) {\small{\textsf{outputs}}};
    \draw (I1) 	to [resistor] node [label={[label distance=3pt]90:{$1\Omega$}}]
    {} (O1);
    \path[color=gray, very thick, shorten >=10pt, ->, >=stealth, bend left] (input)
    edge (I1);		
    \path[color=gray, very thick, shorten >=10pt, ->,
    >=stealth, bend right] (input) edge (I1);		
    \path[color=gray, very thick, shorten >=10pt, ->, >=stealth] (output) edge (O1);
  \end{tikzpicture}
\]
This allows us to connect two distinct outputs to the above double
input.

Given a set $X$ of inputs or outputs, we understand the electrical behaviour on this set 
by considering the symplectic vector space $\vectf{X}$, the direct sum of the space
$\F^X$ of potentials and the space ${(\F^X)}^\ast$ of currents at these points.
A Lagrangian relation specifies which states of the output space $\vectf{Y}$ are
allowed for each state of the input space $\vectf{X}$.
Turning the Lagrangian subspace $\mathrm{Graph}(dQ)$ of a circuit into this
information requires that we understand the `symplectification' 
\[  Sf\maps \vectf{B} \to \vectf{A} \] 
and `twisted symplectification'
\[  S^tf\maps \vectf{B} \to \overline{\vectf{A}}\]
of a function $f\maps A \to B$ between finite sets.  In particular we need to understand how these apply to the input and output functions with codomain restricted to $\partial N$; abusing notation, we also write these $i\maps X \to \partial N$ and $o\maps Y \to \partial N$.

The symplectification is a Lagrangian relation, and the catch
phrase is that it `copies voltages' and `splits currents'.  More precisely,
for any given potential-current pair $(\psi,\iota)$ in $\vectf{B}$, its image
under $Sf$ comprises all elements of $(\psi', \iota') \in \vectf{A}$ such that
the potential at $a \in A$ is equal to the potential at $f(a) \in B$, and such
that, for each fixed $b \in B$, collectively the currents at the $a \in
f^{-1}(b)$ sum to the current at $b$.  We use the symplectification $So$ of the
output function to relate the state on $\partial N$ to that on the
outputs $Y$. As our current framework is set up to report the current \emph{out}
of each node, to describe input currents we define the twisted symplectification
$S^tf\maps \vectf{B} \to \overline{\vectf{A}}$ almost identically to the above, except that we flip the sign of the currents $\iota' \in (\F^A)^\ast$.  We use the twisted symplectification $S^ti$ of the input function to relate the state on $\partial N$
to that on the inputs.

The Lagrangian relation corresponding to a circuit is then the set of all
potential--current pairs that are possible at the inputs and outputs of that circuit. 
For instance, consider a resistor of resistance $r$, with one end considered as an
input and the other as an output:
\[
  \begin{tikzpicture}[circuit ee IEC, set resistor graphic=var resistor IEC graphic]
    \node[contact] (I1) at (0,0) {};
    \node[contact] (O1) at (3,0) {};
    \node(input) at (-2,0) {\small{\textsf{input}}};
    \node(output) at (5,0) {\small{\textsf{output}}};
    \draw (I1) 	to [resistor] node [label={[label distance=3pt]90:{$r$}}] {} (O1);
    \path[color=gray, very thick, shorten >=10pt, ->, >=stealth] (input)
    edge (I1);
    \path[color=gray, very thick, shorten >=10pt, ->, >=stealth] (output) edge (O1);
  \end{tikzpicture}
\]
To obtain the corresponding Lagrangian relation, we must first specify domain and
codomain symplectic vector spaces. In this case, as the input and output sets
each consist of a single point, these vector spaces are both $\F \oplus \F^\ast$,
where the first summand is understood as the space of potentials, and the second
the space of currents.

Now, the resistor has power functional $Q\maps \F^2 \to \F$ given by
\[   Q(\psi_1,\psi_2) = \frac1{2r}(\psi_2-\psi_1)^2, \]
and the graph of the differential of $Q$ is
\[
  \mathrm{Graph}(dQ) = \big\{\big(\psi_1,\psi_2,
  \tfrac1r(\psi_1-\psi_2),\tfrac1r(\psi_2-\psi_1)\big) \,\big|\, \psi_1,\psi_2 \in
  \F\big\} \subseteq \F^2 \oplus (\F^2)^\ast.
\]
In this example the input and output functions $i,o$ are simply the identity
functions on a one element set, so the symplectification of the output function
is simply the identity linear transformation, and the twisted symplectification
of the input function is the isomorphism  between conjugate
symplectic vector spaces $\F\oplus\F^\ast \to \overline{\F\oplus\F^\ast}$ mapping $(\phi,i)$ to $(\phi,-i)$ This implies that the behaviour associated to this
circuit is the Lagrangian relation
\[
  \big\{(\psi_1,i,\psi_2,i) \,\big|\, \psi_1,\psi_2 \in \F, i =
  \tfrac1r(\psi_2-\psi_1)\big\}\subseteq \overline{\F \oplus \F^\ast} \oplus \F
    \oplus \F^\ast.
\]
This is precisely the set of potential-current pairs that are allowed at the
input and output of a resistor of resistance $r$.  In particular, the relation
$i = \tfrac1r(\psi_2-\psi_1)$ is well known in electrical engineering: it is
`Ohm's law'.

A crucial fact is that the process of mapping a circuit to its corresponding
Lagrangian relation identifies distinct circuits.  For example, a single 2-ohm resistor:
\[
  \begin{tikzpicture}[circuit ee IEC, set resistor graphic=var resistor IEC graphic]
    \node[contact] (I1) at (0,0) {};
    \node[contact] (O1) at (3,0) {};
    \node(input) at (-2,0) {\small{\textsf{input}}};
    \node(output) at (5,0) {\small{\textsf{output}}};
    \draw (I1) 	to [resistor] node [label={[label distance=3pt]90:{$2\Omega$}}] {} (O1);
    \path[color=gray, very thick, shorten >=10pt, ->, >=stealth] (input)
    edge (I1);
    \path[color=gray, very thick, shorten >=10pt, ->, >=stealth] (output) edge (O1);
  \end{tikzpicture}
\]
has the same Lagrangian relation as two 1-ohm resistors in series:
\[
  \begin{tikzpicture}[circuit ee IEC, set resistor graphic=var resistor IEC graphic]
    \node[contact] (I1) at (0,0) {};
    \node[circle, minimum width = 3pt, inner sep = 0pt, fill=black] (int) at
    (3,0) {};
    \node[contact] (O1) at (6,0) {};
    \node(input) at (-2,0) {\small{\textsf{input}}};
    \node(output) at (8,0) {\small{\textsf{output}}};
    \draw (I1) 	to [resistor] node [label={[label distance=3pt]90:{$1\Omega$}}] {} (int)
    to [resistor] node [label={[label distance=3pt]90:{$1\Omega$}}] {} (O1);
    \path[color=gray, very thick, shorten >=10pt, ->, >=stealth] (input)
    edge (I1);
    \path[color=gray, very thick, shorten >=10pt, ->, >=stealth] (output) edge (O1);
  \end{tikzpicture}
\]
The Lagrangian relation does not shed any light on the internal workings of a
circuit.  Thus, we call the process of computing this relation `black boxing':
it is like encasing the circuit in an opaque box, leaving only its terminals
accessible. Fortunately, the Lagrangian relation of a circuit is enough to
completely characterize its external behaviour, including how it interacts when
connected with other circuits. 

Put more precisely, the black boxing process is \emph{functorial}: we can 
compute the black boxed version of a circuit made of parts by computing the
black boxed versions of the parts and then composing them.   In fact we shall 
prove that $\Circ$ and $\LagrRel$ are dagger compact categories, and
the black box functor preserves all this structure:

\begin{theorem} \label{main_theorem}
  There exists a hypergraph functor, the \define{black box functor}   
  \[ \blacksquare\maps \Circ \to \LagrRel, \]
   mapping a finite set $X$ to the symplectic vector space
  $\vectf{X}$ it generates, and a circuit $\big((N,E,s,t,r),i,o\big)$ to the Lagrangian     
  relation 
  \[
    \bigcup_{v \in \mathrm{Graph}(dQ)} S^ti(v) \times So(v)
    \subseteq \overline{\F^X \oplus (\F^X)^\ast} \oplus \F^Y \oplus (\F^Y)^\ast,
  \]
  where $Q$ is the circuit's power functional.
\end{theorem}

The goal of this chapter is to prove and explain this result, demonstrating how
the mathematical machinery of Part I provides clarity. With these tools in hand,
the black box functor turns out to rely on a tight relationship between
Kirchhoff's laws, the minimization of Dirichlet forms, and the
`symplectification' of corelations. It is well known that away from the
terminals, a circuit must obey two rules known as Kirchhoff's laws.  We have
already noted that the principle of minimum power states that a circuit will
`choose' potentials on its interior that minimize the power functional.  We
clarify the relation between these points in Theorems
\ref{thm:realizablepotentials} and \ref{thm:dirichletminimization}, which
together show that minimizing a Dirichlet form over some subset amounts to
assuming that the corresponding circuit obeys Kirchhoff's laws on that subset.

We have also mentioned the symplectification of functions above.  Extending this
to allow symplectification of epi-mono corelations in $\FinSet$, this process
gives a map sending corelations to Lagrangian relations that describe the
behaviour of ideal perfectly conductive wires.  We prove that these
symplectified corelations simultaneously impose Kirchhoff's laws (Proposition
\ref{prop:sympfunctor} and Example \ref{ex:sympfunction}) and accomplish the
minimization of Dirichlet forms (Theorem \ref{thm:sympmin}).  

Together, our results show that these three concepts---Kirchhoff's laws from
circuit theory, the analytic idea of minimizing power dissipation, and the
algebraic idea of symplectification of corelations---are merely different faces
of one law: the law of composition of circuits.

\section{Circuits of linear resistors} \label{sec:resistors}
%%fakesubsection
Our first concern is for semantics: ``What do circuit diagrams mean?''. 

To elaborate, while circuit diagrams model electric circuits according to their
physical form, another, often more relevant, way to understand a circuit is by
its external behaviour. This means the following. To an electric circuit we
associate two quantities to each edge: voltage and current. We are not free,
however, to choose these quantities as we like; circuits are subject to
governing laws that imply voltages and currents must obey certain relationships.
From the perspective of control theory we are particularly interested in the
values these quantities take at the so-called terminals, and how altering one
value will affect the other values. We call two circuits equivalent when they
determine the same relationship. Our main task in this first part is to explore
when two circuits are equivalent.

In order to let physical intuition lead the way, we begin by specialising to the
case of linear resistors. In this section we describe how to find the function
of a circuit from its form, advocating in particular the perspective of the
principle of minimum power. This allows us to identify the external behaviour of a
circuit with a so-called Dirichlet form representing the dependence of its power
consumption on potentials at its terminals.

\subsection{Circuits as labelled graphs}

The concept of an abstract open electrical circuit made of linear resistors is
well known in electrical engineering, but we shall need to formalize it with
more precision than usual.  The basic idea is that a circuit of linear resistors
is a graph whose edges are labelled by positive real numbers called
`resistances', and whose sets of vertices is equipped with two subsets: the
`inputs' and `outputs'. This unfolds as follows.

A (closed) circuit of resistors looks like this: 
\[
\begin{tikzpicture}[circuit ee IEC, set resistor graphic=var resistor IEC graphic]
\node (I1) at (0,0) {};
\node (I2) at (0,2) {};
\node (O1) at (5.83,1) {};
\draw (I1) 	to [resistor] node [label={[label distance=2pt]275:{$1\Omega$}}] {} (2.83,1);
\draw (I2)	to [resistor] node [label={[label distance=2pt]85:{$1\Omega$}}] {} (2.83,1)
				to [resistor] node [label={[label distance=3pt]90:{$2\Omega$}}] {} (O1);
\end{tikzpicture}
\]
We can consider this a labelled graph, with each resistor an edge of the graph,
its resistance its label, and the vertices of the graph the points at which
resistors are connected. 

A circuit is `open' if it can be connected to other circuits. To do this we
first mark points at which connections can be made by denoting some vertices as
input and output terminals:
\[
\begin{tikzpicture}[circuit ee IEC, set resistor graphic=var resistor IEC graphic]
\node[contact] (I1) at (0,2) {};
\node[contact] (I2) at (0,0) {};
\node[contact] (O1) at (5.83,1) {};
\node(input) at (-2,1) {\small{\textsf{inputs}}};
\node(output) at (7.83,1) {\small{\textsf{outputs}}};
\draw (I1) 	to [resistor] node [label={[label distance=2pt]85:{$1\Omega$}}] {} (2.83,1);
\draw (I2)	to [resistor] node [label={[label distance=2pt]275:{$1\Omega$}}] {} (2.83,1)
				to [resistor] node [label={[label distance=3pt]90:{$2\Omega$}}] {} (O1);
\path[color=gray, very thick, shorten >=10pt, ->, >=stealth, bend left] (input) edge (I1);		\path[color=gray, very thick, shorten >=10pt, ->, >=stealth, bend right] (input) edge (I2);		
\path[color=gray, very thick, shorten >=10pt, ->, >=stealth] (output) edge (O1);
\end{tikzpicture}
\]
Then, given a second circuit, we may choose a relation between the output set of
the first and the input set of this second circuit, such as the simple relation
of the single output vertex of the circuit above with the single input vertex of
the circuit below.
\[
\begin{tikzpicture}[circuit ee IEC, set resistor graphic=var resistor IEC graphic]
\node[contact] (I1) at (0,1) {};
\node[contact] (O1) at (2.83,2) {};
\node[contact] (O2) at (2.83,0) {};
\node (input) at (-2,1) {\small{\textsf{inputs}}};
\node (output) at (4.83,1) {\small{\textsf{outputs}}};
\draw (I1) 	to [resistor] node [label={[label distance=2pt]95:{$1\Omega$}}] {} (O1);
\draw (I1)		to [resistor] node [label={[label distance=2pt]265:{$3\Omega$}}] {} (O2);
\path[color=gray, very thick, shorten >=10pt, ->, >=stealth] (input) edge (I1);		\path[color=gray, very thick, shorten >=10pt, ->, >=stealth, bend right] (output) edge (O1);
\path[color=gray, very thick, shorten >=10pt, ->, >=stealth, bend left] (output) edge (O2);
\end{tikzpicture}
\]
We connect the two circuits by identifying output and input vertices according
to this relation, giving in this case the composite circuit:
\[
\begin{tikzpicture}[circuit ee IEC, set resistor graphic=var resistor IEC graphic]
\node[contact] (I1) at (0,2) {};
\node[contact] (I2) at (0,0) {};
\coordinate (int1) at (2.83,1) {};
\coordinate (int2) at (5.83,1) {};
\node[contact] (O1) at (8.66,2) {};
\node[contact] (O2) at (8.66,0) {};
\node (input) at (-2,1) {\small{\textsf{inputs}}};
\node (output) at (10.66,1) {\small{\textsf{outputs}}};
\draw (I1) 	to [resistor] node [label={[label distance=2pt]85:{$1\Omega$}}] {} (int1);
\draw (I2)	to [resistor] node [label={[label distance=2pt]275:{$1\Omega$}}] {} (int1)
				to [resistor] node [label={[label distance=3pt]90:{$2\Omega$}}] {} (int2);
\draw (int2) 	to [resistor] node [label={[label distance=2pt]95:{$1\Omega$}}] {} (O1);
\draw (int2)		to [resistor] node [label={[label distance=2pt]265:{$3\Omega$}}] {} (O2);
\path[color=gray, very thick, shorten >=10pt, ->, >=stealth, bend left] (input) edge (I1);		\path[color=gray, very thick, shorten >=10pt, ->, >=stealth, bend right] (input) edge (I2);		
\path[color=gray, very thick, shorten >=10pt, ->, >=stealth, bend right] (output) edge (O1);
\path[color=gray, very thick, shorten >=10pt, ->, >=stealth, bend left] (output) edge (O2);
\end{tikzpicture}
\]

\vskip 1em

More formally, we define a \define{graph} to be a pair of functions $s,t\maps E \to N$ where $E$ and $N$ are finite sets.  We call elements of $E$ \define{edges} and elements of $N$ \define{vertices} or
\define{nodes}.  We say that the edge $e \in E$ has \define{source} $s(e)$ and
\define{target} $t(e)$, and also say that $e$ is an edge \define{from} $s(e)$
\define{to} $t(e)$.

To study circuits we need graphs with labelled edges:

\begin{definition}
Given a set $L$ of \define{labels}, an \define{$L$-graph} is a graph equipped with a function $r\maps E \to L$:
\[
\xymatrix{
L & E \ar@<2.5pt>[r]^{s} \ar@<-2.5pt>[r]_{t} \ar[l]_{r} & N.
}
\]
\end{definition}

For circuits made of resistors we take $L = (0,\infty)$, but later we shall
take $L$ to be a set of positive elements in some more general field.  In either
case, a circuit will be an $L$-graph with some extra structure:

\begin{definition} \label{def_circuit}
Given a set $L$, a \define{circuit over $L$} is an $L$-graph $\xymatrix{
L & E \ar@<2.5pt>[r]^{s} \ar@<-2.5pt>[r]_{t} \ar[l]_{r} & N}$ together with finite sets $X$, $Y$, and functions $i \maps X \to N$ and $o\maps Y \to  N$. We call the sets $i(X)$, $o(Y)$, and $\partial N = i(X) \cup o(Y)$ the \define{inputs},  \define{outputs}, and \define{terminals} or \define{boundary} of the circuit, respectively.
\end{definition}

We will later make use of the notion of connectedness in graphs. Recall that
given two vertices $v, w \in N$ of a graph, a \define{path from $v$ to $w$} is a
finite sequence of vertices $v = v_0, v_1, \dots , v_n = w$ and edges $e_1,
\dots , e_n$ such that for each $1 \le i \le n$, either $e_i$ is an edge from
$v_i$ to $v_{i+1}$, or an edge from $v_{i+1}$ to $v_i$. A subset $S$ of the
vertices of a graph is \define{connected} if, for each pair of vertices in $S$,
there is a path from one to the other. A \define{connected component} of a graph
is a maximal connected subset of its vertices.\footnote{In the theory of
directed graphs the qualifier `weakly' is commonly used before the word
`connected' in these two definitions, in distinction from a stronger notion of
connectedness requiring paths to respect edge directions. As we never consider
any other sort of connectedness, we omit this qualifier.}

In the rest of this section we take $L = (0,\infty) \subseteq \R$ and fix a circuit over 
$(0,\infty)$.  The edges of this circuit should be thought of as `wires'.  The label 
$r_e \in (0,\infty)$ stands for the \define{resistance} of the resistor on the wire $e$.   
There will also be a voltage and current on each wire.  In this section, these will
be specified by functions $V \in \R^E$ and $I \in \R^E$.  Here, as customary in
engineering, we use $I$ for `intensity of current', following Amp\`ere.  

\subsection{Ohm's law, Kirchhoff's laws, and the principle of minimum power}

In 1827, Georg Ohm published a book which included a relation between the voltage
and current for circuits made of resistors \cite{O}.  At the time, the critical
reception was harsh: one contemporary called Ohm's work ``a web of naked
fancies, which can never find the semblance of support from even the most
superficial of observations'', and the German Minister of Education said that a
professor who preached such heresies was unworthy to teach science \cite{D,H}.
However, a simplified version of his relation is now widely used under the name
of `Ohm's law'. We say that \define{Ohm's law} holds if for all edges $e \in
E$ the voltage and current functions of a circuit obey:
\[ 
V(e) = r(e) I(e).  \label{ohm}  
\]

Kirchhoff's laws date to Gustav Kirchhoff in 1845, generalising Ohm's work. They
were in turn generalized into Maxwell's equations a few decades later. We say
\define{Kirchhoff's voltage law} holds if there exists $\phi \in \R^N$ such that
\[
V(e) = \phi(t(e)) - \phi(s(e)).
\]
We call the function $\phi$ a \define{potential}, and think of it as assigning
an electrical potential to each node in the circuit. The voltage then arises as
the differences in potentials between adjacent nodes. If Kirchhoff's voltage law
holds for some voltage $V$, the potential $\phi$ is unique only in the trivial
case of the empty circuit: when the set of nodes $N$ is empty. Indeed, two
potentials define the same voltage function if and only if their difference is
constant on each connected component of the graph.

We say \define{Kirchhoff's current law} holds if for all nonterminal nodes $n
\in N\setminus \partial N$ we have
\[ 
\sum_{s(e) = n} I(e) = \sum_{t(e) = n} I(e).  \label{kcl}  
\]  
This is an expression of conservation of charge within the circuit; it says that
the total current flowing in or out of any nonterminal node is zero. Even when
Kirchhoff's current law is obeyed, terminals need not be sites of zero net
current; we call the function $\iota \in \R^{\partial N}$ that takes a terminal
to the difference between the outward and inward flowing currents,
\begin{align*}
\iota:\partial N &\longrightarrow \R \\
n &\longmapsto \sum_{t(e) = n} I(e) -\sum_{s(e) = n} I(e),
\end{align*}
the \define{boundary current} for $I$.

A \define{boundary potential} is also a function in $\R^{\partial N}$, but
instead thought of as specifying potentials on the terminals of a
circuit. As we think of our circuits as open circuits, with the terminals points
of interaction with the external world, we shall think of these potentials as
variables that are free for us to choose. Using the above three
principles---Ohm's law, Kirchhoff's voltage law, and Kirchhoff's current
law---it is possible to show that choosing a boundary potential determines
unique voltage and current functions on that circuit. 

The so-called `principle of minimum power' gives some insight into how this
occurs, by describing a way potentials on the terminals might determine
potentials at all nodes. From this, Kirchhoff's voltage law then gives rise to a
voltage function on the edges, and Ohm's law gives us a current function too. We
shall show, in fact, that a potential satisfies the principle of minimum power
for a given boundary potential if and only if this current obeys Kirchhoff's
current law.

A circuit with current $I$ and voltage $V$ dissipates energy at a rate
equal to
\[
 \sum_{e \in E} I(e)V(e).
\]  
Ohm's law allows us to rewrite $I$ as $V/r$, while Kirchhoff's voltage law gives
us a potential $\phi$ such that $V(e)$ can be written as
$\phi(t(e))-\phi(s(e))$, so for a circuit obeying these two laws the power can
also be expressed in terms of this potential. We thus arrive at a functional
mapping potentials $\phi$ to the power dissipated by the circuit when Ohm's law
and Kirchhoff's voltage law are obeyed for $\phi$. 

\begin{definition}
The \define{extended power functional} $P\maps \R^N \to \R$ of a circuit is
defined by
\[
P(\phi) =\frac{1}{2} \sum_{e \in E} \frac{1}{r(e)}\big(\phi(t(e))-\phi(s(e))\big)^2.
\]
\end{definition}

\noindent
The factor of $\frac{1}{2}$ is inserted to cancel the factor of 2 that appears when
we differentiate this expression.  We call $P$ the \emph{extended} power functional as we shall see that it is defined even on potentials that are not compatible with the three governing laws of electric circuits. We shall later restrict the domain of this functional so that it is defined precisely on those potentials that \emph{are} compatible with the
governing laws. Note that this functional does not depend on the directions
chosen for the edges of the circuit.

This expression lets us formulate the `principle of minimum power', which gives
us information about the potential $\phi$ given its restriction to the boundary
of $\Gamma$. Call a potential $\phi \in \R^N$ an \define{extension} of a
boundary potential $\psi \in \R^{\partial N}$ if $\phi$ is equal to $\psi$ when
restricted to $\R^{\partial N}$---that is, if $\phi|_{\partial N} = \psi$. 

\begin{definition}
We say a potential $\phi \in \R^{N}$ \define{obeys the principle of minimum
power} for a boundary potential $\psi \in \R^{\partial N}$ if $\phi$ minimizes
the extended power functional $P$ subject to the constraint that  $\phi$ is an
extension of $\psi$. 
\end{definition}

It is well known that, as we have stated above, in the presence of Ohm's law and
Kirchhoff's voltage law, the principle of minimum power is equivalent to
Kirchhoff's current law.

\begin{proposition} \label{minimum_power_implies_kirchhoff_current}
Let $\phi$ be a potential extending some boundary potential $\psi$. Then $\phi$
obeys the principle of minimum power for $\psi$ if and only if the 
current 
\[  I(e) = \frac1{r(e)}(\phi(t(e))-\phi(s(e))) \] 
obeys Kirchhoff's current law.
\end{proposition}

\begin{proof}
Fixing the potentials at the terminals to be those given by the boundary
potential $\psi$, the power is a nonnegative quadratic function of the
potentials at the nonterminals. This implies that an extension $\phi$ of $\psi$
minimizes $P$ precisely when 
\[ \left. \frac{\partial P(\varphi)}{\partial \varphi(n)}\right|_{\varphi = \phi} = 0 \]
for all nonterminals $n \in N \setminus \partial N$. Note that the
partial derivative of the power with respect to the potential at $n$ is given by 
\begin{align*}
  \frac{\partial P}{\partial \varphi(n)}\bigg|_{\varphi = \phi} 
  &= \sum_{t(e) = n} \frac1{r(e)}\big(\phi(t(e))-\phi(s(e))\big) - \sum_{s(e) =
  n} \frac1{r(e)}\big(\phi(t(e))-\phi(s(e))\big) \\
  &= \sum_{t(e) = n} I(e) - \sum_{s(e) = n} I(e).
\end{align*}
Thus $\phi$ obeys the principle of minimum power for $\psi$ if and only if
\[ \sum_{s(e) = n} I(e) = \sum_{t(e) = n} I(e)\] 
for all $n \in N \setminus \partial N$, and so if and only if Kirchhoff's current law holds.
\end{proof}

\subsection{A Dirichlet problem}

We remind ourselves that we are in the midst of understanding circuits as objects that define relationships between boundary potentials and boundary currents. This relationship is defined by the stipulation that voltage--current pairs on a circuit must obey Ohm's law and Kirchhoff's laws---or equivalently, Ohm's law, Kirchhoff's voltage law, and the principle of minimum power. In this subsection we show these conditions imply that for each boundary potential $\psi$ on the circuit there exists a potential $\phi$ on the circuit extending $\psi$, unique up to what may be interpreted as a choice of reference potential on each connected component of the circuit. From this potential $\phi$ we can then compute the unique voltage, current, and boundary current functions compatible with the given boundary potential.

Fix again a circuit with extended power functional $P\maps \R^N \to \R$. Let $\nabla\maps \R^{N} \to \R^{N}$ be the operator that maps a potential $\phi \in \R^N$ to the function from $N$ to $\R$ given by
\[
n \longmapsto \frac{\partial P}{\partial \varphi(n)}\bigg|_{\varphi = \phi} \;.
\]
As we have seen, this function takes potentials to twice the pointwise currents that they induce. We have also seen that a potential $\phi$ is compatible with the governing laws of circuits if and only if
\begin{equation}
\nabla \phi \big|_{\R^{\partial N}} = 0 .\label{dirichlet}
\end{equation}
The operator $\nabla$ acts as a discrete analogue of the Laplacian for the graph $\Gamma$, so we call this operator the \define{Laplacian} of $\Gamma$, and say that  equation \eqref{dirichlet} is a version of Laplace's equation. We then say that the problem of finding an extension $\phi$ of some fixed boundary potential $\psi$ that solves this Laplace's equation---or, equivalently, the problem of finding a $\phi$ that obeys the principle of minimum power for $\psi$---is a discrete version of the \define{Dirichlet problem}. 

As we shall see, this version of the Dirichlet problem always has a solution.  However, the solution is not necessarily unique.  If we take a solution $\phi$ and some $\alpha \in \R^N$ that is constant on each connected component and vanishes on the boundary of $\Gamma$, it is clear that $\phi+\alpha$ is still an extension of $\psi$ and that 
\[
\left.\frac{\partial P(\varphi)}{\partial \varphi(n)}\right|_{\varphi = \phi} = 
\left.\frac{\partial P(\varphi)}{\partial \varphi(n)}\right|_{\varphi = \phi + \alpha},
\] 
so $\phi + \alpha$ is another solution. We say that a connected component of a circuit \define{touches the boundary} if it contains a vertex in $\partial N$. Note that such an $\alpha$ must vanish on all connected components touching the boundary.

With these preliminaries in hand, standard techniques can be used to solve the
Dirichlet problem \cite{Fukushima}:
\begin{proposition} \label{dirichlet_problem}
For any boundary potential $\psi \in \R^{\partial N}$ there exists a potential $\phi$ obeying the principle of minimum power for $\psi$.  If we also demand that $\phi$ vanish on every connected component of $\Gamma$ not touching the boundary, then $\phi$ is unique. 
\end{proposition}
\begin{proof}
For existence, observe that the power is a nonnegative quadratic form, the extensions of $\psi$ form an affine subspace of $\R^N$, and a nonnegative quadratic form restricted to an affine subspace of a real vector space must reach a minimum somewhere on this subspace. 

For uniqueness, suppose that both $\phi$ and $\phi'$ obey the principle of minimum power for $\psi$. Let 
\[
\alpha = \phi'-\phi.
\]
Then 
\[
\alpha\big|_{\partial N} = \phi'\big|_{\partial N}-\phi\big|_{\partial N} = \psi-\psi =0,
\] 
so $\phi+\lambda\alpha$ is an extension of $\psi$ for all $\lambda \in \R$. This implies that
\[
f(\lambda) := P(\phi+\lambda\alpha)
\]
is a smooth function attaining its minimum value at both $\lambda=0$ and
$\lambda=1$. In particular, this implies that $f'(0)=0$. But this means that when writing $f$ as a quadratic, the coefficient of $\lambda$ must be $0$, so we can write
\begin{align*}
2f(\lambda) &= \sum_{e \in E} \frac1{r(e)}\big((\phi+\lambda\alpha)(t(e))-(\phi+\lambda\alpha)(s(e))\big)^2 \\
&= \sum_{e \in E} \frac1{r(e)}\Big(\big(\phi(t(e))-\phi(s(e))\big)+\lambda\big(\alpha(t(e))-\alpha(s(e))\big)\Big)^2 \\
&=  \sum_{e \in E} \frac1{r(e)}\big(\phi(t(e))-\phi(s(e))\big)^2 + \textrm{$\lambda$-term} +  \lambda^2 \sum_{e \in E} \frac1{r(e)}\big(\alpha(t(e))-\alpha(s(e))\big)^2 \\
&=  \sum_{e \in E} \frac1{r(e)}\big(\phi(t(e))-\phi(s(e))\big)^2 + \lambda^2 \sum_{e \in E} \frac1{r(e)}\big(\alpha(t(e))-\alpha(s(e))\big)^2.
\end{align*}
Then
\[
f(1) - f(0) 
= \frac{1}{2}\sum_{e \in E} \frac1{r(e)}\big(\alpha(t(e))-\alpha(s(e))\big)^2 =0,
\]
so $\alpha(t(e)) = \alpha(s(e))$ for every edge $e \in E$. This implies that $\a$ is constant on each connected component of the graph $\Gamma$ of our circuit. 

Note that as $\alpha|_{\partial N} = 0$, $\alpha$ vanishes on every connected
component of $\Gamma$ touching the boundary. Note also that there always exists
a solution $\phi$ that vanishes on any connected component of $\Gamma$ not
touching the boundary: indeed, the total power dissipated is simply the sum of
the power dissipated on each connected component, and any such vanishing
potential dissipates zero power away from the boundary.  Restricting our
attention to solutions $\phi$ and $\phi'$ with this property, we see that
$\alpha = \phi'-\phi$ vanishes on all connected components of $\Gamma$, and
hence is identically zero. Thus $\phi' = \phi$, and this extra condition ensures
a unique solution to the Dirichlet problem.
\end{proof}

We have also shown the following:

\begin{proposition}\label{dirichlet_problem_2}
Suppose $\psi \in \R^{\partial N}$ and $\phi$ is a potential obeying the principle of minimum power for $\psi$.  Then $\phi'$ obeys the principle of minimum power for $\psi$ if and only if the difference $\phi' - \phi$ is constant on every connected component of $\Gamma$ and vanishes on every connected component touching the boundary of $\Gamma$.
\end{proposition}

Furthermore, $\phi$ depends linearly on $\psi$:

\begin{proposition}\label{dirichlet_problem_3: linearity}
Fix $\psi \in \R^{\partial N}$, and suppose $\phi \in \R^N$ is the unique potential obeying the principle of minimum power for $\psi$ that vanishes on all connected components of $\Gamma$ not touching the boundary. Then $\phi$ depends linearly on $\psi$.
\end{proposition}
\begin{proof}
Fix $\psi, \psi' \in \R^{\partial N}$, and suppose $\phi, \phi' \in \R^N$ obey the principle of minimum power for $\psi,\psi'$ respectively, and that both $\phi$ and $\phi'$ vanish on all connected components of $\Gamma$ not touching the boundary. 

Then, for all $\lambda \in \R$,
\[
(\phi+\lambda\phi')\big|_{\R^{\partial N}} = \phi\big|_{\R^{\partial N}} +
\lambda\phi'\big|_{\R^{\partial N}}  = \psi + \lambda\psi'
\]
and
\[
(\nabla(\phi+\lambda\phi'))\big|_{\R^{\partial N}} = 
(\nabla\phi)\big|_{\R^{\partial N}} +
\lambda(\nabla\phi')\big|_{\R^{\partial N}}  = 0.
\]
Thus $\phi+\lambda\phi'$ solves the Dirichlet problem for $\psi+\lambda\psi'$, and thus $\phi$ depends linearly on $\psi$.
\end{proof}

Bamberg and Sternberg \cite{BS} describe another way to solve the Dirichlet problem, going back to Weyl \cite{Weyl}.

\subsection{Equivalent circuits}

We have seen that boundary potentials determine, essentially uniquely, the value of all the electric properties across the entire circuit. But from the perspective of control theory, this internal structure is irrelevant: we can only access the circuit at its terminals, and hence only need concern ourselves with the relationship between boundary potentials and boundary currents. In this section we streamline our investigations above to state the precise way in which boundary currents depend on boundary potentials. In particular, we shall see that the relationship is completely captured by the functional taking boundary potentials to the minimum power used by any extension of that boundary potential. Furthermore, each such power functional determines a different boundary potential--boundary current relationship, and so we can conclude that two circuits are equivalent if and only if they have the same power functional. 

An `external behaviour', or \define{behaviour} for short, is an equivalence class of circuits, where two are considered equivalent when the boundary current is the same function of the boundary potential. The idea is that the boundary current and boundary potential are all that can be observed `from outside', i.e. by making measurements at the terminals.  Restricting our attention to what can be observed by making measurements at the terminals amounts to treating a circuit as a `black box': that is, treating its interior as hidden from view.  So, two circuits give the same behaviour when they behave the same as `black boxes'.

First let us check that the boundary current is a function of the boundary potential.  For this we introduce an important quadratic form on the space of boundary potentials:

\begin{definition}
The \define{power functional} $Q \maps \R^{\partial N} \to \R$ of a circuit with extended power functional $P$ is given by
\[
 Q(\psi) = \min_{\phi|_{\R^{\partial N}} = \psi } P(\phi).
\]
\end{definition}

Proposition \ref{dirichlet_problem} shows the minimum above exists, so the power functional is well defined.  Thanks to the principle of minimum power, $Q(\psi)$ equals $\frac{1}{2}$ times the power dissipated by the circuit when the boundary voltage is $\psi$.  We will later see that in fact $Q(\psi)$ is a nonnegative quadratic form on $\R^{\partial N}$. 

Since $Q$ is a smooth real-valued function on $\R^{\partial N}$, its differential $d Q$ at any given point $\psi \in \R^{\partial N}$ defines an element of the dual space $(\R^{\partial N})^\ast$, which we denote by $d Q_\psi$.  In fact, this element is equal to the boundary current $\iota$ corresponding to the boundary voltage $\psi$:

\begin{proposition} \label{boundary_current_determines_boundary_voltage}
Suppose $\psi \in \R^{\partial N}$.  Suppose $\phi$ is any extension of $\psi$ minimizing the power. Then $dQ_\psi \in (\R^{\partial N})^\ast \cong \R^{\partial N}$ gives the boundary current of the current induced by the potential $\phi$.
\end{proposition}

\begin{proof}
Note first that while there may be several choices of $\phi$ minimizing the power subject to the constraint that $\phi|_{\R^{\partial N}} = \psi$, Proposition \ref{dirichlet_problem_2} says that the difference between any two choices vanishes on all components touching the boundary of $\Gamma$.  Thus, these two choices give the same value for the boundary current $\iota\maps \partial N \to \R$. So, with no loss of generality we may assume $\phi$ is the unique choice that vanishes on all components not touching the boundary. Write $\overline\iota\maps N \to \R$ for the extension of $\iota\maps \partial N \to \R$ to $N$ taking value $0$ on $N \setminus \partial N$. 

By Proposition \ref{dirichlet_problem_3: linearity}, there is a linear operator
\[
f\maps \R^{\partial N} \longrightarrow \R^N
\]
sending $\psi \in \R^{\partial N}$ to this choice of $\phi$, and then
\[
Q(\psi) = P(f\psi).
\]
Given any $\psi' \in \R^{\partial N}$, we thus have
\begin{align*}
dQ_\psi(\psi') &= \frac{d}{d\lambda}Q(\phi +\lambda\psi') \bigg|_{\lambda=0} \\
&= \frac{d}{d\lambda}P(f(\psi+\lambda\psi'))\bigg|_{\lambda=0} \\
&= \frac{1}{2} \frac{d}{d\lambda}\sum_{e \in E} \frac1{r(e)}\bigg(f(\psi+\lambda\psi'))(t(e))-(f(\psi+\lambda\psi'))(s(e))\bigg)^2 \bigg|_{\lambda=0} \\
&= \frac{1}{2} \frac{d}{d\lambda}\sum_{e \in E} \frac1{r(e)}\bigg((f\psi(t(e))-f\psi(s(e))) \;+\;\lambda (f\psi'(t(e))- f\psi'(s(e)))\bigg)^2 \bigg|_{\lambda=0} \\
&= \sum_{e \in E} \frac1{r(e)}(f\psi(t(e))-f\psi(s(e)))(f\psi'(t(e))- f\psi'(s(e))) \\
&= \sum_{e \in E} I(e)(f\psi'(t(e))- f\psi'(s(e))) \\
&= \sum_{n \in N}\left(\sum_{t(e) = n} I(e) - \sum_{s(e) = n} I(e)\right)f\psi'(n) \\
&= \sum_{n \in N}\overline \iota(n) f\psi'(n) \\
&= \sum_{n \in \partial N}\iota(n) \psi'(n).
\end{align*}
This shows that $dQ_\psi^\ast = \iota$, as claimed.  Note that this calculation
explains why we inserted a factor of $\frac{1}{2}$ in the definition of $P$: it
cancels the factor of $2$ obtained from differentiating a square.
 \end{proof}

Note this only depends on $Q$, which makes no mention of the potentials at
nonterminals. This is amazing: the way power depends on boundary potentials
completely characterizes the way boundary currents depend on boundary
potentials. In particular, in \textsection\ref{sec:circdef} we shall see that this
allows us to define a composition rule for behaviours of circuits.

To demonstrate these notions, we give a basic example of equivalent circuits.

\begin{example}[Resistors in series] \label{resistors_in_series}
Resistors are said to be placed in \define{series} if they are placed end to end or, more
precisely, if they form a path with no self-intersections. It is well known that
resistors in series are equivalent to a single resistor with resistance equal to
the sum of their resistances. To prove this, consider the following circuit
comprising two resistors in series, with input $A$ and output $C$:
\[
  \begin{tikzpicture}[circuit ee IEC, set resistor graphic=var resistor IEC graphic]
    \node[contact] (I1) at (0,0) [label=left:$A$] {};
    \node[circle, minimum width = 3pt, inner sep = 0pt, fill=black] (int) at (3,0) [label=above:$B$] {};
    \node[contact] (O1) at (6,0) [label=right:$C$] {};
    \draw (I1) 	to [resistor] node [label={[label distance=3pt]90:{$r_{AB}$}}] {} (int)
    to [resistor] node [label={[label distance=3pt]90:{$r_{BC}$}}] {} (O1);
  \end{tikzpicture}
\]
Now, the extended power functional $P\maps \R^{\{A,B,C\}} \to \R$ for this circuit is
\[
P(\phi) = \frac12\left(\frac1{r_{AB}}\big(\phi(A)-\phi(B)\big)^2 +
\frac1{r_{BC}}\big(\phi(B)-\phi(C)\big)^2\right),
\]
while the power functional $Q\maps \R^{\{A,C\}} \to \R$ is given by minimization
over values of $\phi(B) = x$:
\[
Q(\psi) = \min_{x \in \R} \frac12 \left(\frac1{r_{AB}}\big(\psi(A)-x\big)^2 + \frac1{r_{BC}}\big(x-\psi(C)\big)^2 \right). 
\]
Differentiating with respect to $x$, we see that this minimum occurs when
\[
\frac1{r_{AB}}\big(x-\psi(A)\big) + \frac1{r_{BC}}\big(x-\psi(C)\big) = 0,
\]
and hence when $x$ is the $r$-weighted average of $\psi(A)$ and $\psi(C)$:
\[
x = \frac{r_{BC}\psi(A) + r_{AB}\psi(C)}{r_{BC}+ r_{AB}}.
\]
Substituting this value for $x$ into the expression for $Q$ above and simplifying gives
\[
Q(\psi) = \frac12\cdot\frac1{r_{AB}+r_{BC}}\big(\psi(A)-\psi(C)\big)^2. 
\]
This is also the power functional of the circuit
\[
\begin{tikzpicture}[circuit ee IEC, set resistor graphic=var resistor IEC graphic]
\node[contact] (I1) at (0,0) [label=left:$A$] {};
\node[contact] (O1) at (3,0) [label=right:$C$] {};
\draw (I1) 	to [resistor] node [label={[label distance=3pt]90:{$r_{AB}+r_{BC}$}}] {} (O1);
\end{tikzpicture}
\]
and so the circuits are equivalent.
\end{example}

\subsection{Dirichlet forms}

In the previous subsection we claimed that power functionals are quadratic forms
on the boundary of the circuit whose behaviour they represent. They comprise, in
fact, precisely those quadratic forms known as Dirichlet forms.

\begin{definition}
Given a finite set $S$, a \define{Dirichlet form} on $S$ is a quadratic form $Q:
\mathbb{R}^S \to \mathbb{R}$ given by the formula
\[
  Q(\psi) = \sum_{i,j \in S} c_{i j} (\psi_i - \psi_j)^2
\]
for some nonnegative real numbers $c_{i j}$, and where we have written $\psi_i = \psi(i) \in
\mathbb{R}$.
\end{definition}

Note that we may assume without loss of generality that $c_{i i} = 0$ and $c_{i
j} = c_{j i}$; we do this henceforth.  Any Dirichlet form is nonnegative:
$Q(\psi) \ge 0$ for all $\psi \in \mathbb{R}^S$.  However, not all nonnegative
quadratic forms are Dirichlet forms.  For example, if $S = \{1, 2\}$, the
nonnegative quadratic form $Q(\psi) = (\psi_1 + \psi_2)^2$ is not a Dirichlet
form. That said, the concept of Dirichlet form is vastly more general than the
above definition: such quadratic forms are studied not just on
finite-dimensional vector spaces $\mathbb{R}^S$ but on $L^2$ of any measure
space.  When this measure space is just a finite set, the concept of Dirichlet
form reduces to the definition above.  For a thorough introduction to Dirichlet
forms, see the text by Fukushima \cite{Fukushima}.  For a fun tour of the
underlying ideas, see the paper by Doyle and Snell \cite{DS}. 

The following characterizations of Dirichlet forms help illuminate the concept:

\begin{proposition} \label{dirichlet_characterizations}
  Given a finite set $S$ and a quadratic form $Q\maps \mathbb{R}^S \to \mathbb{R}$,
  the following are equivalent:
  \begin{enumerate}[(i)]
    \item $Q$ is a Dirichlet form.

    \item $Q(\phi) \le Q(\psi)$ whenever $|\phi_i - \phi_j| \le |\psi_i -
      \psi_j|$ for all $i, j$. 

    \item $Q(\phi) = 0$ whenever $\phi_i$ is independent of $i$, and $Q$ obeys
      the \define{Markov property}: $Q(\phi) \le Q(\psi)$ when $\phi_i = \min
      (\psi_i, 1) $.
  \end{enumerate}
\end{proposition}
\begin{proof}
See Fukushima \cite{Fukushima}.
\end{proof}

While the extended power functionals of circuits are evidently Dirichlet forms,
it is not immediate that all power functionals are. For this it is crucial that
the property of being a Dirichlet form is preserved under minimising over linear
subspaces of the domain that are generated by subsets of the given finite set.

\begin{proposition} \label{dirichlet_minimization}
  If $Q\maps \R^{S+T} \to \R$ is a Dirichlet form, then 
  \[
    \min_{\nu \in \R^T} Q(-,\nu)\maps \R^S \to \R 
  \]
  is Dirichlet.
\end{proposition}
\begin{proof}
  We first note that $\min_{\nu \in \R^S} Q(-,\nu)$ is a quadratic form. Again,
  $\min_{\nu \in \R^T} Q(-,\nu)$ is well defined as a nonnegative quadratic form
  also attains its minimum on an affine subspace of its domain. Furthermore
  $\min_{\nu \in \R^T} Q(-,\nu)$ is itself a quadratic form, as the partial
  derivatives of $Q$ are linear, and hence the points at which these minima are
  attained depend linearly on the argument of $\min_{\nu \in \R^T} Q(-,\nu)$.

  Now by Proposition \ref{dirichlet_characterizations}, $Q(\phi) \le Q(\phi')$
  whenever $|\phi_i - \phi_j| \le |\phi'_i - \phi'_j|$ for all $i,j \in S+T$. In
  particular, this implies $\min_{\nu \in \R^T} Q(\psi,\nu) \le \min_{\nu \in
  \R^T} Q(\psi',\nu)$ whenever $|\psi_i - \psi_j| \le |\psi'_i - \psi'_j|$ for
  all $i,j \in S$. Using Proposition \ref{dirichlet_characterizations} again
  then implies that $\min_{\nu \in \R^T} Q(-,\nu)$ is a Dirichlet form.
\end{proof}

\begin{corollary}
  Let $Q \maps \R^{\partial N} \to \R$ be the power functional for some circuit. Then
  $Q$ is a Dirichlet form.
\end{corollary}
\begin{proof}
  The extended power functional $P$ is a Dirichlet form, and writing $\R^N=
  \R^{\partial N} \oplus \R^{N \setminus \partial N}$ allows us to write
  \[
    Q(-) =  \min_{\phi\in \R^{N \setminus \partial N}}
    P(-,\phi). \qedhere
  \]
\end{proof}

The converse is also true: simply construct the circuit with set of vertices
$\partial N$ and an edge of resistance $\frac{1}{2c_{ij}}$ between any $i,j \in
\partial N$ such that the term $c_{ij}(\psi_i - \psi_j)$ appears in the
Dirichlet form. This gives: 

\begin{proposition}
  A function is the power functional for some circuit if and only if it is a
  Dirichlet form.
\end{proposition}

This is an expression of the `star-mesh transform', a well-known fact of
electrical engineering stating that every circuit of linear resistors is
equivalent to some complete graph of resistors between its terminals. For more
details see \cite{vLO}. We may interpret the proof of Proposition
\ref{dirichlet_minimization} as showing that intermediate potentials at minima
depend linearly on boundary potentials, in fact a weighted average, and that
substituting these into a quadratic form still gives a quadratic form.

\bigskip

In summary, in this section we have shown the existence of a surjective function
\[
  \bigg\{\begin{array}{c} \mbox{circuits of linear resistors} \\ \mbox{ with
    boundary $\partial N$} \end{array} \bigg\} \longrightarrow \bigg\{
    \mbox{Dirichlet forms on $\partial N$}\bigg\}
\]
mapping two circuits to the same Dirichlet form if and only if they have the same
external behaviour.  In the next section we extend this result to encompass
inductors and capacitors too.

\section{Inductors and capacitors} \label{sec:plcs}
%%fakesubsection
The intuition gleaned from the study of resistors carries over to inductors and
capacitors too, to provide a framework for studying what are known as passive
linear networks. To understand inductors and capacitors in this way, however, we
must introduce a notion of time dependency and subsequently the Laplace
transform, which allows us to work in the so-called frequency domain. Here, like
resistors, inductors and capacitors simply impose a relationship of
proportionality between the voltages and currents that run across them. The
constant of proportionality is known as the impedance of the component.

As for resistors, the interconnection of such components may be understood, at
least formally, as a minimization of some quantity, and we may represent the
behaviours of this class of circuits with a more general idea of Dirichlet form.

\subsection{The frequency domain and Ohm's law revisited}

In broadening the class of electrical circuit components under examination, we
find ourselves dealing with components whose behaviours depend on the rates of
change of current and voltage with respect to time. We thus now consider
time-varying voltages $v \maps [0,\infty) \to \R$ and currents $i \maps
  [0,\infty) \to \R$, where $t \in [0,\infty)$ is a real variable representing
    time. For mathematical reasons, we
restrict these voltages and currents to only those with (i) zero initial
conditions (that is, $f(0) = 0$) and (ii) Laplace transform lying in the field
\[
  \R(s) = \left\{ Z(s) = \tfrac{P(s)}{Q(s)} \,\Big\vert\, P, Q \mbox{
  polynomials over $\R$ in $s$}, \, Q \ne 0 \right\}
\]
of real rational functions of one variable. 
%We don't need the currents and voltages to lie in this field!!!  They just
%need to lie in some vector space over this field!!!
While it is possible that
physical voltages and currents might vary with time in a more general way, we
restrict to these cases as the rational functions are, crucially, well behaved
enough to form a field, and yet still general enough to provide arbitrarily
close approximations to currents and voltages found in standard applications.

An \define{inductor} is a two-terminal circuit component across which the voltage is
proportional to the rate of change of the current. By convention we draw this as
follows, with the inductance $L$ the constant of proportionality:\footnote{We
  follow the standard convention of denoting inductance by the letter $L$, after
  the work of Heinrich Lenz and to avoid confusion with the $I$ used for
current.}
\[
  \begin{tikzpicture}[circuit ee IEC]
    \node[contact] (I1) at (0,0) {};
    \node[contact] (I2) at (1.83,0) {};
    \draw (I1) 	to [inductor] node [label={[label distance=2pt]{$L$}}]
    {} (I2);
  \end{tikzpicture}
\]
Writing $v_L(t)$ and $i_L(t)$ for the voltage and current over time $t$ across
this component respectively, and using a dot to denote the derivative with
respect to time $t$, we thus have the relationship 
\[
  v_L(t) = L\, \dot{i}_L(t).
\]
Permuting the roles of current and voltage, a \define{capacitor} is a two-terminal
circuit component across which the current is proportional to the rate of change
of the voltage. We draw this as follows, with the capacitance $C$ the constant
of proportionality:
\[
  \begin{tikzpicture}[circuit ee IEC]
    \node[contact] (I1) at (0,0) {};
    \node[contact] (I2) at (1.83,0) {};
    \draw (I1) 	to [capacitor] node [label={[label distance=5pt]{$C$}}]
    {} (I2);
  \end{tikzpicture}
\]
Writing $v_C(t)$, $i_C(t)$ for the voltage and current across the capacitor,
this gives the equation
\[
  i_C(t) = C\, \dot{v}_C(t).
\]
We assume here that inductances $L$ and capacitances $C$ are positive real numbers.

Although inductors and capacitors impose a linear relationship if we involve the
derivatives of current and voltage, to mimic the above work on resistors we wish
to have a constant of proportionality between functions representing the current
and voltage themselves. Various integral transforms perform just this role; electrical
engineers typically use the Laplace transform. This lets us write a function of time $t$ instead as a function of frequencies $s$, and in doing so turns differentiation with respect to $t$ into multiplication by $s$, and integration with respect to $t$ into
division by $s$.  

In detail, given a function $f(t)\maps [0, \infty) \to \R$, we define the
\define{Laplace transform} of $f$
\[
  \mathfrak{L}\{f\}(s) = \int_{0}^\infty f(t) e^{-st} dt.
\]
We also use the notation $\mathfrak{L}\{f\}(s) = F(s)$, denoting the Laplace
transform of a function in upper case, and refer to the Laplace transforms as
lying in the \define{frequency domain} or \define{$s$-domain}. For us, the three
crucial properties of the Laplace transform are then: 
\begin{enumerate}[(i)]
  \item linearity: $\mathfrak{L}\{af+bg\}(s) = aF(s)+bG(s)$ for $a,b\in \R$;
  \item differentiation: $\mathfrak{L}\{\dot{f}\}(s) = s F(s) - f(0)$;
  \item integration: if $g(t) = \int_0^t f(\tau)d\tau$ then 
 $G(s) = \frac{1}{s} F(s)$.
\end{enumerate}
Writing $V(s)$ and $I(s)$ for the Laplace transform of the voltage $v(t)$ and
current $i(t)$ across a component respectively, and recalling that by assumption
$v(t) = i(t) = 0$ for $t \le 0$, the $s$-domain behaviours of components become,
for a resistor of resistance $R$:
\[
  V(s) = RI(s),
\]
for an inductor of inductance $L$:
\[
  V(s) = sLI(s),
\]
and for a capacitor of capacitance $C$:
\[
  V(s) = \frac1{sC} I(s). 
\]

Note that for each component the voltage equals the current times a rational function of
the real variable $s$, called the \define{impedance} and in general denoted by $Z$.
Note also that the impedance is a \define{positive real function}, meaning that it lies
in the set
\[         
  \{ Z \in \R(s) : \forall s \in \C \;\; \mathrm{Re}(s) > 0 \implies
  \mathrm{Re}(Z(s)) > 0 \} . 
\]
While $Z$ is a quotient of polynomials with real cofficients, in this definition
we are applying it to complex values of $s$, and demanding that its real part be
positive in the open left half-plane.  Positive real functions were introduced by Otto 
Brune in 1931, and they play a basic role in circuit theory \cite{Brune}.  

Indeed, Brune convincingly argued that for any conceivable passive linear
component with two terminals we have this generalization of Ohm's law:
\[
  V(s)=Z(s)I(s)
\]
where $I \in \R(s)$ is the \define{current}, $V \in \R(s)$ is the \define{voltage}
and the positive real function $Z$ is the \define{impedance} of the component.   As we shall
see, generalizing from circuits of linear resistors to arbitrary passive linear
circuits is just a matter of formally replacing resistances by impedances. 

As we consider passive linear circuits with more than two terminals, however,
the coefficients of our power functionals lie not just in the set of positive
real functions, but in its closure $\R(s)^+$ under addition, multiplication, and
division. Note that this closure is strictly smaller than $\R(s)$, as each
positive real function is takes strictly positive real values on the positive real
axis, so no matter how we take sums, products, and quotients of positive real
functions it never results in the zero function $0 \in \R(s)$.  To generalise
from linear resistors to passive linear circuits, we then replace the field $\R$
by the larger field $\R(s)$, and replace the set of positive reals, $\R^+ =
(0,\infty)$, by the set $\R(s)^+$.  From a mathematical perspective we might as
well work with any field with a notion of `positive element'.

\begin{definition} 
  Given a field $\F$, we define a \define{set of positive elements} for $\F$ to be  
  any subset $\F^+\subset \F$ not containing $0$ and closed under addition,
  multiplication, and division.
\end{definition}

Our first motivating example arises from circuits made of resistors.  Here $\F =
\R$ is the field of real numbers and we take $\F^+ = (0,\infty)$.   Our second
motivating example arises from general passive linear circuits.  Here $\F =
\R(s)$ is the field of rational functions in one real variable, and we take
$\F^+ = \R(s)^+$ to be the closure of the positive real functions, as defined
above.  

In all that follows, we fix a field $\F$ of characteristic zero equipped with a
set of positive elements $\F^+$.  By a `circuit', we shall henceforth mean a
circuit over $\F^+$, as explained in Definition \ref{def_circuit}.   To fix the
notation:

\begin{definition} \label{def_circuit_2}
A \define{(passive linear) circuit} is a graph $s,t \maps E \to N$ with $E$ as its set of \define{edges} and $N$ as its set of \define{nodes}, equipped with function $Z \maps E \to \F^+$ assigning each edge an \define{impedance}, together with finite sets $X$, $Y$, and functions $i \maps X \to N$ and $o\maps Y \to  N$. We call the sets $i(X)$, $o(Y)$, and $\partial N = i(X) \cup o(Y)$ the \define{inputs}, \define{outputs}, and \define{terminals} or \define{boundary} of the circuit, respectively.
\end{definition}

We next remark on the analogy between electronics and mechanics in light of
these new components.

\subsection{The mechanical analogy} \label{sec:mechanical}

Now that we have introduced inductors and capacitors, it is worth taking 
another glance at the analogy chart in Section \ref{sec:intro}.  What are the
analogues of resistance, inductance and capacitance in mechanics?  If we restrict attention to systems with translational degrees of freedom, the answer is given in the following
chart.

\begin{small}
\begin{center}
\begin{tabular}{|c|c|}
\hline
Electronics & Mechanics (translation) \\
\hline\hline
charge $Q$ & position $q$ \\
\hline
current $i = \dot Q$ & velocity $v = \dot q$ \\
\hline
flux linkage $\lambda$ & momentum $p$ \\
\hline
voltage $v = \dot \lambda$ & force $F = \dot p$ \\
\hline
resistance $R$ & damping coefficient $c$ \\
\hline
inductance $L$ & mass $m$ \\
\hline
inverse capacitance $C^{-1}$ & spring constant $k$ \\
\hline
\end{tabular}
\end{center}
\end{small}

A famous example concerns an electric circuit with a resistor of resistance $R$, an inductor of inductance $L$, and a capacitor of capacitance $C$, all in series:
\[
  \begin{tikzpicture}[circuit ee IEC, set resistor graphic=var resistor IEC graphic]
    \node[contact] (I1) at (0,0) {};
    \node[contact] (I2) at (1.83,0) {};
    \node[contact] (I3) at (3.66,0) {};
    \node[contact] (I4) at (5.49,0) {};
    \draw (I1) 	to [resistor] node [label={[label distance=2pt]{$R$}}]
    {} (I2);
    \draw (I2) 	to [inductor] node [label={[label distance=5pt]{$L$}}]
    {} (I3);
     \draw (I3) 	to [capacitor] node [label={[label distance=5pt]{$C$}}]
    {} (I4);
  \end{tikzpicture}
\]
We saw in Example \ref{resistors_in_series} that for resistors in series, the
resistances add.  The same fact holds more generally for passive linear circuits,
so the impedance of this circuit is the sum
\[   Z = s L + R + (sC)^{-1}  .\]
Thus, the voltage across this circuit is related to the current through the
circuit by
\[  V(s) = (s L + R + (sC)^{-1}) I(s)  \]
If $v(t)$ and $i(t)$ are the voltage and current as functions of time, we conclude that
\[  v(t) = L \frac{d}{dt}i(t) + Ri(t) + C^{-1} \int_0^t i(s) \, ds  \]
It follows that 
\[   
L \ddot{Q} + R \dot{Q} + C^{-1} Q = v
\]
where $Q(t) = \int_0^t i(t) ds$ has units of charge.  As the chart above suggests,
this equation is analogous to that of a damped harmonic oscillator:
\[    
m \ddot{q} + c \dot{q} + k q = F 
\]
where $m$ is the mass of the oscillator, $c$ is the damping coefficient, $k$ is 
the spring constant and $F$ is a time-dependent external force.

For details, and many more analogies of this sort, see the book by Karnopp, 
Margolis and Rosenberg \cite{KRM} or Brown's enormous text \cite{Brown}.   While it would be a distraction to discuss them further here, these analogies mean that our
work applies to a wide class of networked systems, not just electrical circuits.

\subsection{Generalized Dirichlet forms} \label{sec:generalized}

To understand the behaviour of passive linear circuits we need to
understand how the behaviours of individual components, governed by Ohm's law,
fit together to give the behaviour of an entire network.
Kirchhoff's laws still hold, and so does a version of the principle of minimum power.

As before, to each passive linear circuit we associate a generalised Dirichlet
form.

\begin{definition} Given a field $\F$ and
  a finite set $S$, a \define{Dirichlet form over $\F$} on $S$ is a quadratic form   
  $Q\maps \F^S \to \F$ given by the formula 
  \[ 
    Q(\psi) = \sum_{i,j \in S} c_{ij} (\psi_i - \psi_j)^2,
  \]
  where $c_{i j} \in \F^+$.  
\end{definition}

Generalizing from circuits of resistors, we define the
\define{extended power functional} $P\maps \F^N \to \F$ of any circuit by
\[
  P(\varphi) = \frac{1}{2} \sum_{e \in E}
  \frac1{Z(e)}\big(\varphi(t(e))-\varphi(s(e))\big)^2.
\]
and we call $\varphi \in \F^N$ a \define{potential}. Note have assumed that
$\F$ has characteristic zero, so dividing by 2 is allowed.  Note also that the
extended power functional is a Dirichlet form on $N$.

Although it is not clear what it means to minimize over the field $\F$, we can
use formal derivatives to formulate an analogue of the principle of minimum 
power.   This will actually be a `variational principle', saying the derivative of
the power functional vanishes with respect to certain variations in the potential.
As before we shall see that given Ohm's law, this principle is equivalent to
Kirchhoff's current law.

Indeed, the extended power functional $P(\varphi)$ can be considered an element
of the polynomial ring $\F[\{\varphi(n)\}_{n \in N}]$ generated by formal
variables $\varphi(n)$ corresponding to potentials at the nodes $n \in N$. We
may thus take formal derivatives of the extended power functional with respect
to the $\varphi(n)$.  We then call $\phi \in \F^N$ a \define{realizable
potential} for the given circuit if for each nonterminal node, $n \in N\setminus
\partial N$, the formal partial derivative of the extended power functional with
respect to $\varphi(n)$ equals zero when evaluated at $\phi$:
\[
  \frac{\partial P}{\partial \varphi(n)}\bigg\vert_{\varphi = \phi} = 0
\]
This terminology arises from the following fact, a generalization of Proposition
\ref{minimum_power_implies_kirchhoff_current}:

\begin{theorem} \label{thm:realizablepotentials}
The potential $\phi \in \F^N$ is a realizable potential for a given
circuit if and only if the induced current 
\[  I(e) = \frac1{Z(e)}(\phi(t(e))-\phi(s(e))) \]
obeys Kirchhoff's current law:
\[ 
\sum_{s(e) = n} I(e) = \sum_{t(e) = n} I(e)
\]  
for all $n \in N\setminus \partial N$.
\end{theorem}
\begin{proof}
The proof of this statement is exactly that for Proposition
\ref{minimum_power_implies_kirchhoff_current}. 
\end{proof}

A corollary of Theorem \ref{thm:realizablepotentials} is that the set of
states---that is, potential--current pairs---that are compatible with the
governing laws of a circuit is given by the set of realizable potentials
together with their induced currents. 

\subsection{A generalized minimizability result}

We begin to move from a discussion of the intrinsic behaviours of circuits to a
discussion of their behaviours under composition. The key fact for composition
of generalized Dirichlet forms is that, in analogy with Proposition
\ref{dirichlet_minimization}, we may speak of a formal version of minimization
of Dirichlet forms. We detail this here.  In what follows let $P$ be a Dirichlet
form over $\F$ on some finite set $S$. 

Recall that given $R \subseteq S$, we
call $\tilde\psi \in \F^S$ an \define{extension} of $\psi \in \F^R$ if
$\tilde\psi$ restricted to $R$ equals $\psi$.   We call such an
extension \define{realizable} if 
\[
    \frac{\partial P}{\partial \varphi(s)}\bigg\vert_{\varphi = \tilde\psi} = 0
  \]
for all $s \in S \setminus R$.  Note that over the real numbers $\R$ this means
that among all the extensions of $\psi$, $\tilde\psi$ minimizes the function $P$.

\begin{theorem} \label{thm:dirichletminimization}
  Let $P$ be a Dirichlet form over $\F$ on $S$, and let $R \subseteq S$ be an
  inclusion of finite sets. Then we may uniquely define a Dirichlet form
  \[\min_{S \setminus R}P: \F^R \to \F\] 
   on $R$ by sending each $\psi \in \F^R$ to
  the value $P(\tilde\psi)$ of any realizable extension $\tilde\psi$ of $\psi$.
\end{theorem}

To prove this theorem, we must first show that $\min_{S \setminus R} P$ is well defined as a function.

\begin{lemma} \label{lem:welldefineddirichletmin}
  Let $P$ be a Dirichlet form over $\F$ on $S$, let $R \subseteq S$ be an
  inclusion of finite sets, and let $\psi \in \F^R$. Then for all realizable
  extensions $\tilde\psi$, $\tilde\psi' \in \F^S$ of $\psi$ we have $P(\tilde\psi) =
  P(\tilde\psi')$. 
\end{lemma}
\begin{proof}
  This follows from the formal version of the multivariable Taylor theorem for
  polynomial rings over a field of characteristic zero (see
  \cite[\textsection IV.4.5]{Bou90}). Let $\tilde\psi$,
  $\tilde\psi' \in \F^S$ be realizable extensions of $\psi$, and note that
  $dP_{\tilde\psi}(\tilde\psi-\tilde\psi')=0$, since for all $s \in R$ we have
  $\tilde\psi(s) -\tilde\psi'(s) =0$, and for all $s \in S \setminus R$ we have
  \[
    \frac{\partial P}{\partial \varphi(s)}\bigg\vert_{\varphi = \tilde\psi}=0. 
  \]
  We may take the Taylor expansion of $P$ around $\tilde\psi$ and evaluate at
  $\tilde\psi'$. As $P$ is a quadratic form, this gives
  \begin{align*}
    P(\tilde\psi') &=
    P(\tilde\psi)+dP_{\tilde\psi}(\tilde\psi'-\tilde\psi)+P(\tilde\psi'-\tilde\psi)
    \\
    & = P(\tilde\psi)+P(\tilde\psi'-\tilde\psi).
  \end{align*}
  Similarly, we arrive at  
  \[
    P(\tilde\psi)= P(\tilde\psi')+P(\tilde\psi-\tilde\psi').
  \]
  But again as $P$ is a quadratic form, we then see that 
  \[
    P(\tilde\psi')-P(\tilde\psi) = P(\tilde\psi'-\tilde\psi) =
    P(\tilde\psi-\tilde\psi') = P(\tilde\psi)-P(\tilde\psi').
  \]
  This implies that $P(\tilde\psi')-P(\tilde\psi) = 0$, as required.
\end{proof}

It remains to show that $\min P$ remains a Dirichlet form. We do this
inductively.

\begin{lemma} \label{lem:onestepdirichletmin}
  Let $P$ be a Dirichlet form over $\F$ on $S$, and let $s \in S$ be an element
  of $S$. Then the map $\min_{\{s\}} P:\F^{S \setminus\{s\}} \to \F$ sending
  $\psi$ to $P(\tilde\psi)$ is a Dirichlet form on $S \setminus \{s\}$.
\end{lemma}
\begin{proof}
  Write $P(\phi) = \sum_{i,j} c_{ij}(\phi_i -\phi_j)^2$, assuming without loss
  of generality that $c_{sk} =0$ for all $k$. We then have
  \[
    \frac{\partial P}{\partial \varphi(s)}\bigg\vert_{\varphi = \phi} = \sum_k
    2c_{ks}(\phi_s-\phi_k),
  \]
  and this is equal to zero when
  \[
    \phi_s = \frac{\sum_k c_{ks}\phi_k}{\sum_k c_{ks}}.
  \]
  Note that the $c_{ij}$ lie in $\F^+$, and $\F^+$ is closed under addition, so
  $\sum_k c_{ks} \ne 0$. Thus $\min_{\{s\}}P$ may be given explicitly by the
  expression
  \[
    \min_{\{s\}} P(\psi) = \sum_{i,j \in S \setminus \{s\}} c_{ij}(\psi_i -\psi_j)^2 +
    \sum_{\ell \in S \setminus \{s\}} c_{\ell s}\left(\psi_\ell - \tfrac{\sum_k
      c_{ks} \psi_k}{\sum_k c_{ks}}\right)^2.
  \]
  We must show this is a Dirichlet form on $S \setminus \{s\}$. 
  
  As the sum of Dirichlet forms is evidently Dirichlet, it suffices to check that the expression 
  \[
    \sum_\ell c_{\ell s}\left(\psi_\ell - \tfrac{\sum_k c_{ks} \psi_k}{\sum_k
      c_{ks}}\right)^2
  \]
  is Dirichlet on $S \setminus \{s\}$. Multiplying through by the constant
  $(\sum_k c_{ks})^2 \in \F^+$, it further suffices to check
  \begin{align*}
    &\quad \sum_\ell c_{\ell s}\left(\sum_k c_{ks} \psi_\ell - \sum_k c_{ks}
    \psi_k\right)^2 \\
    &= \sum_\ell c_{\ell s} \left(\sum_k c_{ks} (\psi_\ell -
    \psi_k)\right)^2 \\
    &= \sum_\ell c_{\ell s} \left(2 \sum_{\substack{k,m \\ k \ne m}} c_{k s} c_{ms}
    (\psi_\ell-\psi_k)(\psi_\ell - \psi_m) + \sum_{k} c_{k
    s}^2(\psi_\ell-\psi_k)^2\right) \\
    &= 2\sum_{\substack{k,\ell,m \\ k \ne m}} c_{\ell s} c_{k s} c_{ms}
    (\psi_\ell-\psi_k)(\psi_\ell - \psi_m) + \sum_{k, \ell} c_{\ell s}c_{k
    s}^2(\psi_\ell-\psi_k)^2
  \end{align*}
  is Dirichlet. But
  \begin{align*}
    &\quad (\psi_k - \psi_\ell)(\psi_k - \psi_m)+(\psi_\ell - \psi_k)(\psi_\ell -
    \psi_m) + (\psi_m-\psi_k)(\psi_m-\psi_\ell) \\ 
    &= \psi_k^2+\psi_\ell^2+\psi_m^2-\psi_k\psi_\ell- \psi_k\psi_m -
    \psi_\ell\psi_m \\
    &= \tfrac12\big( (\psi_k-\psi_\ell)^2 +(\psi_k-\psi_m)^2
    +(\psi_\ell-\psi_m)^2\big),
  \end{align*}
  so this expression is indeed Dirichlet. Indeed, pasting these computations
  together shows that
  \[
    \min_{\{s\}}P(\psi) = \sum_{i,j} \left(c_{ij}+\frac{c_{is}c_{js}}{{\textstyle \sum_k}
    c_{ks}}\right)(\psi_i-\psi_j)^2. \qedhere
  \]
\end{proof}

With these two lemmas, the proof of Theorem \ref{thm:dirichletminimization}
becomes straightforward.

\begin{proof}[Proof of Theorem \ref{thm:dirichletminimization}]
  Lemma \ref{lem:welldefineddirichletmin} shows that $\min_{S \setminus R}P$ is a well-defined
  function. As $R$ is a finite set, we may write it $R = \{s_1,\dots, s_n\}$ for
  some natural number $n$. Then we may define a sequence of functions $P_i =
  \min_{\{s_1, \dots,s_i\}} P_{i-1}$, $1 \le i\le n$. Define also $P_0 = P$, and note
  that $P_n = \min_{S \setminus R}P$. Then, by Lemma
  \ref{lem:onestepdirichletmin}, each $P_i$ is Dirichlet as $P_{i-1}$
  is. This proves the proposition.
\end{proof}

We can thus define the power functional of a circuit by analogy with circuits made
of resistors:

\begin{definition}
The \define{power functional} $Q \maps \R^{\partial N} \to \R$ of a circuit with extended power functional $P$ is given by
\[
 Q = \min_{N \setminus \partial N}  P .
\]
\end{definition}

As before, we call two circuits equivalent if they have the same power
functional, and define the \define{behaviour} of a circuit to be its equivalence
class. 

We now have a complete description of the semantics of diagrams of passive
linear networks. The purpose of this chapter is to show how to use decorated
corelations to make this semantics compositional. We make a start on this in the
next section.

\section{A compositional perspective} \label{sec:circdef}
%%fakesubsection
Thus far our focus has been on the semantics of circuit diagrams, explaining how
labelled graphs represent Dirichlet forms. We now wish to work towards a
\emph{compositional} semantics. 

A prerequisite for this is compositional structure on our circuit diagrams and
Dirichlet forms. We desire a compositional structure that captures the syntax of
circuit diagrams, addressing the question: ``How do we interact with circuit
diagrams?'' Informally, the answer is that we interact with them by connecting them
to each other, perhaps after moving them into the right form by rotating or
reflecting them, or by crossing, bending, splitting, and combining some of the
wires. In line with the overarching philosophy of this thesis, it should come as
no surprise that we hence model circuits as morphisms in a hypergraph category. 

We begin this section by defining a hypergraph category of circuits. We then
want to define a hypergraph category of behaviours. We have three desiderata for
such a category. First, the category of behaviours should have a unique morphism
for each distinct circuit behaviour. Second, the map taking a circuit to its
behaviour should be structure preserving: that is, it should be a hypergraph
functor to a hypergraph category. The third desideratum is one of aesthetics:
the category of behaviours and its composition rule should be easy to define and
work with. 

It turns out, as we shall see, that Dirichlet forms are not quite general enough
to accommodate our needs. In Subsection \ref{ssec.noidentities} we discuss an
attempt to construct a `natural' category of Dirichlet forms, with composition
of Dirichlet forms given by taking their sum and minimising over the `interior'
terminals, but find that this operation does not have an identity. Despite this,
we give an ad hoc construction of a category of Dirichlet corelations that
at least fulfils the two technical desiderata above.

\subsection{The category of open circuits}

In Definition \ref{def_circuit_2}, we defined a circuit of linear resistors to
be a labelled graph with marked input and output terminals, as in the example:
\[
\begin{tikzpicture}[circuit ee IEC, set resistor graphic=var resistor IEC graphic]
\node[contact] (I1) at (0,2) {};
\node[contact] (I2) at (0,0) {};
\coordinate (int1) at (2.83,1) {};
\coordinate (int2) at (5.83,1) {};
\node[contact] (O1) at (8.66,2) {};
\node[contact] (O2) at (8.66,0) {};
\node (input) at (-2,1) {\small{\textsf{inputs}}};
\node (output) at (10.66,1) {\small{\textsf{outputs}}};
\draw (I1) 	to [resistor] node [label={[label distance=2pt]85:{$1\Omega$}}] {} (int1);
\draw (I2)	to [resistor] node [label={[label distance=2pt]275:{$1\Omega$}}] {} (int1)
				to [resistor] node [label={[label distance=3pt]90:{$2\Omega$}}] {} (int2);
\draw (int2) 	to [resistor] node [label={[label distance=2pt]95:{$1\Omega$}}] {} (O1);
\draw (int2)		to [resistor] node [label={[label distance=2pt]265:{$3\Omega$}}] {} (O2);
\path[color=gray, very thick, shorten >=10pt, ->, >=stealth, bend left] (input) edge (I1);		\path[color=gray, very thick, shorten >=10pt, ->, >=stealth, bend right] (input) edge (I2);		
\path[color=gray, very thick, shorten >=10pt, ->, >=stealth, bend right] (output) edge (O1);
\path[color=gray, very thick, shorten >=10pt, ->, >=stealth, bend left] (output) edge (O2);
\end{tikzpicture}
\]
We then defined general passive linear circuits by replacing resistances with
impedances chosen from a set of positive elements $\F^+$ in any field $\F$.
These circuits are examples of decorated cospans.  We now use decorated cospans
to construct a hypergraph category $\Circ$ whose morphisms are circuits, such
that the hypergraph structure expresses the syntactic operations on circuits
discussed above.

Indeed, observe that a circuit is just an $\F^+$-graph, as defined in the
introduction to Chapter \ref{ch.deccospans}.  Recall from Subsection
\ref{ssec.exlabelledgraphs} the lax symmetric monoidal functor
\[
  \mathrm{Graph}\maps (\mathrm{FinSet},+) \longrightarrow (\mathrm{Set},\times)
\]
mapping each finite set $N$ to the set $\mathrm{Graph}(N)$ of
$[0,\infty)$-graphs $(N,E,s,t,r)$ with $N$ as their set of nodes. Define the lax
symmetric monoidal functor $\mathrm{Circuit}$, generalising $\mathrm{Graph}$,
so that the set $N$ is mapped to the set of $\F^+$-graphs.
  
\begin{definition}
  We define the hypergraph category
  \[
    \mathrm{Circ} = \mathrm{CircuitCospan} .
  \]
\end{definition}

\begin{aside}
  As discussed in Subsection \ref{ssec.bicatdeccospan}, by the work of Courser
  \cite{Cou16} we in fact have a symmetric monoidal bicategory $2\mbox{-}\Circ$
  of circuits with 
  \begin{center}
    \begin{tabular}{ c | p{.65\textwidth} }
      \textbf{objects} & finite sets \\ 
      \textbf{morphisms} & cospans of finite sets decorated by $\F^+$-graphs \\ 
      \textbf{2-morphisms} & maps of decorated cospans \\
    \end{tabular}
  \end{center}
  Moreover, the work of Stay \cite{Sta16} can be used to show this bicategory is
  compact closed.  Decategorifying $2\mbox{-}\Circ$ gives a category equivalent
  to our previously defined category $\Circ$.
\end{aside}

The hypergraph structure, including the compactness and dagger, captures the
aforementioned syntactic operations that can be performed on circuits. The
composition expresses the fact that we can connect the outputs of one circuit to
the inputs of the next, like so:
\[
  \begin{array}{c}
    \begin{tikzpicture}[circuit ee IEC, set resistor graphic=var resistor IEC
      graphic,scale=.7]
      \node[circle,draw,inner sep=1pt,fill=gray,color=gray]         (x) at
      (-3,-1.3) {};
      \node at (-3,-3.2) {\footnotesize $X$};
      \node[circle,draw,inner sep=1pt,fill]         (A) at (0,0) {};
      \node[circle,draw,inner sep=1pt,fill]         (B) at (3,0) {};
      \node[circle,draw,inner sep=1pt,fill]         (C) at (1.5,-2.6) {};
      \node[circle,draw,inner sep=1pt,fill=gray,color=gray]         (y1) at
      (6,-.6) {};
      \node[circle,draw,inner sep=1pt,fill=gray,color=gray]         (y2) at
      (6,-2) {};
      \node at (6,-3.2) {\footnotesize $Y$};
      \coordinate         (ua) at (.5,.25) {};
      \coordinate         (ub) at (2.5,.25) {};
      \coordinate         (la) at (.5,-.25) {};
      \coordinate         (lb) at (2.5,-.25) {};
      \path (A) edge (ua);
      \path (A) edge (la);
      \path (B) edge (ub);
      \path (B) edge (lb);
      \path (ua) edge  [resistor, circuit symbol unit=5pt, circuit symbol size=width {5} height 1.5] node[above] {\footnotesize $2\Omega$} (ub);
      \path (la) edge  [resistor, circuit symbol unit=5pt, circuit symbol size=width {5} height 1.5] node[below] {\footnotesize $3\Omega$} (lb);
      \path (A) edge  [resistor, circuit symbol unit=5pt, circuit symbol size=width {5} height 1.5] node[left] {\footnotesize $1\Omega$} (C);
      \path (C) edge  [resistor, circuit symbol unit=5pt, circuit symbol size=width {5} height 1.5] node[right] {\footnotesize $1\Omega$} (B);
      \path[color=gray, very thick, shorten >=10pt, shorten <=5pt, ->, >=stealth] (x) edge (A);
      \path[color=gray, very thick, shorten >=10pt, shorten <=5pt, ->, >=stealth] (y1) edge (B);
      \path[color=gray, very thick, shorten >=10pt, shorten <=5pt, ->, >=stealth] (y2)
      edge (B);
      \node[circle,draw,inner sep=1pt,fill]         (A') at (9,0) {};
      \node[circle,draw,inner sep=1pt,fill]         (B') at (12,0) {};
      \node[circle,draw,inner sep=1pt,fill]         (C') at (10.5,-2.6) {};
      \node[circle,draw,inner sep=1pt,fill=gray,color=gray]         (z1) at
      (15,-.6) {};
      \node[circle,draw,inner sep=1pt,fill=gray,color=gray]         (z2) at (15,-2) {};
      \node at (15,-3.2) {\footnotesize $Z$};
      \path (A') edge  [resistor, circuit symbol unit=5pt, circuit symbol size=width {5} height 1.5] node[above] {\footnotesize $5\Omega$} (B');
      \path (C') edge  [resistor, circuit symbol unit=5pt, circuit symbol size=width {5} height 1.5] node[right] {\footnotesize $8\Omega$} (B');
      \path[color=gray, very thick, shorten >=10pt, shorten <=5pt, ->, >=stealth] (y1) edge (A');
      \path[color=gray, very thick, shorten >=10pt, shorten <=5pt, ->, >=stealth] (y2)
      edge (C');
      \path[color=gray, very thick, shorten >=10pt, shorten <=5pt, ->, >=stealth] (z1) edge (B');
      \path[color=gray, very thick, shorten >=10pt, shorten <=5pt, ->, >=stealth]
      (z2) edge (C');
    \end{tikzpicture} \\
    \Downarrow \\
    \begin{tikzpicture}[circuit ee IEC, set resistor graphic=var resistor IEC
      graphic,scale=0.7]
      \node[circle,draw,inner sep=1pt,fill=gray,color=gray]         (x) at (-4,-1.3) {};
      \node at (-4,-3.2) {\footnotesize $X$};
      \node[circle,draw,inner sep=1pt,fill]         (A) at (0,0) {};
      \node[circle,draw,inner sep=1pt,fill]         (B) at (3,0) {};
      \node[circle,draw,inner sep=1pt,fill]         (C) at (1.5,-2.6) {};
      \node[circle,draw,inner sep=1pt,fill]         (D) at (6,0) {};
      \coordinate         (ua) at (.5,.25) {};
      \coordinate         (ub) at (2.5,.25) {};
      \coordinate         (la) at (.5,-.25) {};
      \coordinate         (lb) at (2.5,-.25) {};
      \coordinate         (ub2) at (3.5,.25) {};
      \coordinate         (ud) at (5.5,.25) {};
      \coordinate         (lb2) at (3.5,-.25) {};
      \coordinate         (ld) at (5.5,-.25) {};
      \path (A) edge (ua);
      \path (A) edge (la);
      \path (B) edge (ub);
      \path (B) edge (lb);
      \path (B) edge (ub2);
      \path (B) edge (lb2);
      \path (D) edge (ud);
      \path (D) edge (ld);
      \node[circle,draw,inner sep=1pt,fill=gray,color=gray]         (z1) at
      (10,-.6) {};
      \node[circle,draw,inner sep=1pt,fill=gray,color=gray]         (z2) at (10,-2) {};
      \node at (10,-3.2) {\footnotesize $Z$};
      \path (ua) edge  [resistor, circuit symbol unit=5pt, circuit symbol size=width {5} height 1.5] node[above] {\footnotesize $2\Omega$} (ub);
      \path (la) edge  [resistor, circuit symbol unit=5pt, circuit symbol size=width {5} height 1.5] node[below] {\footnotesize $3\Omega$} (lb);
      \path (A) edge  [resistor, circuit symbol unit=5pt, circuit symbol size=width {5} height 1.5] node[left] {\footnotesize $1\Omega$} (C);
      \path (C) edge  [resistor, circuit symbol unit=5pt, circuit symbol size=width {5} height 1.5] node[right] {\footnotesize $1\Omega$} (B);
      \path (ub2) edge  [resistor, circuit symbol unit=5pt, circuit symbol size=width {5} height 1.5] node[above] {\footnotesize $5\Omega$} (ud);
      \path (lb2) edge  [resistor, circuit symbol unit=5pt, circuit symbol size=width {5} height 1.5] node[below] {\footnotesize $8\Omega$} (ld);
      \path[color=gray, very thick, shorten >=10pt, shorten <=5pt, ->, >=stealth] (x) edge (A);
      \path[color=gray, very thick, shorten >=10pt, shorten <=5pt, ->, >=stealth] (z1)
      edge (D);
      \path[bend left, color=gray, very thick, shorten >=10pt, shorten <=5pt, ->, >=stealth] (z2)
      edge (B);
    \end{tikzpicture}
  \end{array}
\]
The monoidal composition models the placement of
circuits side-by-side:
\[
  \begin{aligned}
    \begin{tikzpicture}[circuit ee IEC, set resistor graphic=var resistor IEC
      graphic,scale=.7]
      \node[circle,draw,inner sep=1pt,fill=gray,color=gray]         (x) at
      (-2.8,-1.3) {};
      \node at (-2.8,-3.2) {\footnotesize $X$};
      \node[circle,draw,inner sep=1pt,fill]         (A) at (0,0) {};
      \node[circle,draw,inner sep=1pt,fill]         (B) at (3,0) {};
      \node[circle,draw,inner sep=1pt,fill]         (C) at (1.5,-2.6) {};
      \node[circle,draw,inner sep=1pt,fill=gray,color=gray]         (y1) at
      (5.8,-.6) {};
      \node[circle,draw,inner sep=1pt,fill=gray,color=gray]         (y2) at
      (5.8,-2) {};
      \node at (5.8,-3.2) {\footnotesize $Y$};
      \coordinate         (ua) at (.5,.25) {};
      \coordinate         (ub) at (2.5,.25) {};
      \coordinate         (la) at (.5,-.25) {};
      \coordinate         (lb) at (2.5,-.25) {};
      \path (A) edge (ua);
      \path (A) edge (la);
      \path (B) edge (ub);
      \path (B) edge (lb);
      \path (ua) edge  [resistor, circuit symbol unit=5pt, circuit symbol size=width {5} height 1.5] node[above] {\footnotesize $2\Omega$} (ub);
      \path (la) edge  [resistor, circuit symbol unit=5pt, circuit symbol size=width {5} height 1.5] node[below] {\footnotesize $3\Omega$} (lb);
      \path (A) edge  [resistor, circuit symbol unit=5pt, circuit symbol size=width {5} height 1.5] node[left] {\footnotesize $1\Omega$} (C);
      \path (C) edge  [resistor, circuit symbol unit=5pt, circuit symbol size=width {5} height 1.5] node[right] {\footnotesize $1\Omega$} (B);
      \path[color=gray, very thick, shorten >=10pt, shorten <=5pt, ->, >=stealth] (x) edge (A);
      \path[color=gray, very thick, shorten >=10pt, shorten <=5pt, ->, >=stealth] (y1) edge (B);
      \path[color=gray, very thick, shorten >=10pt, shorten <=5pt, ->, >=stealth] (y2)
      edge (B);
      \node at (1.5,-3.5) {$\otimes$};
      \node[circle,draw,inner sep=1pt,fill=gray,color=gray]         (x) at
      (-2.8,-5) {};
      \node at (-2.8,-6.5) {\footnotesize $X'$};
      \node[circle,draw,inner sep=1pt,fill]         (A) at (0,-5) {};
      \node[circle,draw,inner sep=1pt,fill]         (B) at (3,-5) {};
      \node[circle,draw,inner sep=1pt,fill=gray,color=gray]         (y1) at
      (5.8,-5) {};
      \node at (5.8,-6.5) {\footnotesize $Y'$};
      \coordinate         (ua1) at (.5,-4.75) {};
      \coordinate         (ub1) at (2.5,-4.75) {};
      \coordinate         (la1) at (.5,-5.25) {};
      \coordinate         (lb1) at (2.5,-5.25) {};
      \path (A) edge (ua1);
      \path (A) edge (la1);
      \path (B) edge (ub1);
      \path (B) edge (lb1);
      \path (ua1) edge  [resistor, circuit symbol unit=5pt, circuit symbol
      size=width {5} height 1.5] node[above] {\footnotesize $5\Omega$} (ub1);
      \path (la1) edge  [resistor, circuit symbol unit=5pt, circuit symbol
      size=width {5} height 1.5] node[below] {\footnotesize $1\Omega$} (lb1);
      \path[color=gray, very thick, shorten >=10pt, shorten <=5pt, ->, >=stealth] (x) edge (A);
      \path[color=gray, very thick, shorten >=10pt, shorten <=5pt, ->, >=stealth] (y1) edge (B);
    \end{tikzpicture}
  \end{aligned}
  \Rightarrow
  \begin{aligned}
    \begin{tikzpicture}[circuit ee IEC, set resistor graphic=var resistor IEC
      graphic,scale=.7]
      \node[circle,draw,inner sep=1pt,fill=gray,color=gray]         (x) at
      (-2.8,-1.3) {};
      \node at (-2.8,-5.5) {\footnotesize $X+X'$};
      \node[circle,draw,inner sep=1pt,fill]         (A) at (0,0) {};
      \node[circle,draw,inner sep=1pt,fill]         (B) at (3,0) {};
      \node[circle,draw,inner sep=1pt,fill]         (C) at (1.5,-2.6) {};
      \node[circle,draw,inner sep=1pt,fill=gray,color=gray]         (y1) at
      (5.8,-.6) {};
      \node[circle,draw,inner sep=1pt,fill=gray,color=gray]         (y2) at
      (5.8,-2) {};
      \node at (5.8,-5.5) {\footnotesize $Y+Y'$};
      \coordinate         (ua) at (.5,.25) {};
      \coordinate         (ub) at (2.5,.25) {};
      \coordinate         (la) at (.5,-.25) {};
      \coordinate         (lb) at (2.5,-.25) {};
      \path (A) edge (ua);
      \path (A) edge (la);
      \path (B) edge (ub);
      \path (B) edge (lb);
      \path (ua) edge  [resistor, circuit symbol unit=5pt, circuit symbol size=width {5} height 1.5] node[above] {\footnotesize $2\Omega$} (ub);
      \path (la) edge  [resistor, circuit symbol unit=5pt, circuit symbol size=width {5} height 1.5] node[below] {\footnotesize $3\Omega$} (lb);
      \path (A) edge  [resistor, circuit symbol unit=5pt, circuit symbol size=width {5} height 1.5] node[left] {\footnotesize $1\Omega$} (C);
      \path (C) edge  [resistor, circuit symbol unit=5pt, circuit symbol size=width {5} height 1.5] node[right] {\footnotesize $1\Omega$} (B);
      \path[color=gray, very thick, shorten >=10pt, shorten <=5pt, ->, >=stealth] (x) edge (A);
      \path[color=gray, very thick, shorten >=10pt, shorten <=5pt, ->, >=stealth] (y1) edge (B);
      \path[color=gray, very thick, shorten >=10pt, shorten <=5pt, ->, >=stealth] (y2)
      edge (B);
      \node[circle,draw,inner sep=1pt,fill=gray,color=gray]         (x2) at
      (-2.8,-2.7) {};
      \node[circle,draw,inner sep=1pt,fill]         (A1) at (0,-4) {};
      \node[circle,draw,inner sep=1pt,fill]         (B1) at (3,-4) {};
      \node[circle,draw,inner sep=1pt,fill=gray,color=gray]         (y3) at
      (5.8,-3.4) {};
      \coordinate         (ua1) at (.5,-3.75) {};
      \coordinate         (ub1) at (2.5,-3.75) {};
      \coordinate         (la1) at (.5,-4.25) {};
      \coordinate         (lb1) at (2.5,-4.25) {};
      \path (A1) edge (ua1);
      \path (A1) edge (la1);
      \path (B1) edge (ub1);
      \path (B1) edge (lb1);
      \path (ua1) edge  [resistor, circuit symbol unit=5pt, circuit symbol
      size=width {5} height 1.5] node[above] {\footnotesize $5\Omega$} (ub1);
      \path (la1) edge  [resistor, circuit symbol unit=5pt, circuit symbol
      size=width {5} height 1.5] node[below] {\footnotesize $1\Omega$} (lb1);
      \path[color=gray, very thick, shorten >=10pt, shorten <=5pt, ->,
      >=stealth] (x2) edge (A1);
      \path[color=gray, very thick, shorten >=10pt, shorten <=5pt, ->,
      >=stealth] (y3) edge (B1);
    \end{tikzpicture}
  \end{aligned}
\]
Identities and Frobenius maps result from identifying points. For example, the
identity on a single point is the (empty) decorated cospan
\[
      \begin{tikzpicture}[circuit ee IEC, set resistor graphic=var resistor IEC
	graphic]
	\node[circle,draw,inner sep=1pt,fill=gray,color=gray]         (x) at
	(-2.8,0) {};
	\node at (-2.8,-.5) {\footnotesize $X$};
	\node[circle,draw,inner sep=1pt,fill]         (A) at (0,0) {};
	\node[circle,draw,inner sep=1pt,fill=gray,color=gray]         (y1) at
	(2.8,0) {};
	\node at (2.8,-.5) {\footnotesize $Y$};
	\path[color=gray, very thick, shorten >=10pt, shorten <=5pt, ->, >=stealth] (x) edge (A);
	\path[color=gray, very thick, shorten >=10pt, shorten <=5pt, ->,
	>=stealth] (y1) edge (A);
      \end{tikzpicture}
\]
while the multiplication on a point is the empty decorated cospan
\[
      \begin{tikzpicture}[circuit ee IEC, set resistor graphic=var resistor IEC
	graphic]
	\node[circle,draw,inner sep=1pt,fill=gray,color=gray]         (x) at
	(-2.8,0) {};
	\node at (-2.8,-1.3) {\footnotesize $X$};
	\node[circle,draw,inner sep=1pt,fill]         (A) at (0,0) {};
	\node[circle,draw,inner sep=1pt,fill=gray,color=gray]         (y1) at
	(2.8,.8) {};
	\node[circle,draw,inner sep=1pt,fill=gray,color=gray]         (y2) at
	(2.8,-.8) {};
	\node at (2.8,-1.3) {\footnotesize $Y$};
	\path[color=gray, very thick, shorten >=10pt, shorten <=5pt, ->, >=stealth] (x) edge (A);
	\path[color=gray, very thick, shorten >=10pt, shorten <=5pt, ->,
	>=stealth] (y1) edge (A);
	\path[color=gray, very thick, shorten >=10pt, shorten <=5pt, ->,
	>=stealth] (y2) edge (A);
      \end{tikzpicture}
\]
These identifications represent interconnection via ideal, perfectly conductive
wires: if two terminals are connected by such wires, then the electrical
behaviour at both terminals must be identical. The Frobenius maps allow us to
split, combine, and discard wires.  The symmetric monoidal structure allows
us to reorder input and output wires. Derived from these, the compactness captures the
interchangeability between input and outputs of circuits---that is, the fact
that we can choose any input to our circuit and consider it instead as an
output, and vice versa---while the dagger structure expresses the fact that we
may reflect a whole circuit, switching all inputs with all outputs.

Recall that we consider two circuit diagrams equivalent if they have the same
behaviour, or power functional. Now that we have constructed a hypergraph
category where our morphisms are circuit diagrams---our `syntactic' category for
our diagrammatic language---we would also like to construct a hypergraph
category with morphisms behaviours of circuits, and show that the map from a
circuit to its behaviour defines a hypergraph functor.

\subsection{An obstruction to black boxing} \label{ssec.noidentities}

It would be nice to have a category in which Dirichlet forms are morphisms, such
that the map sending a circuit to its behaviour is a functor.  Here we present a
na\"ive attempt to constructed the category with Dirichlet forms as morphisms,
using the principle of minimum power to compose these morphisms.  Unfortunately
the proposed category does not include identity morphisms.  However, it points
in the right direction, and underlines the importance of the cospan formalism we
then turn to develop.

We can define a composition rule for Dirichlet forms that reflects composition
of circuits.  Given finite sets $X$ and $Y$, let $X+Y$ denote their disjoint
union.  Let $D(X,Y)$ be the set of Dirichlet forms on $X+Y$. There is a way to
compose these Dirichlet forms
\[ 
\circ \maps D(Y,Z) \times D(X,Y) \to D(X,Z) 
\]
defined as follows.  Given $P \in D(Y,Z)$ and $Q \in D(X,Y)$, let
\[ 
  (P \circ Q)(\alpha, \gamma) = \min_{\beta \in \F^Y} Q(\alpha, \beta) + P(\beta, \gamma),
\]
where $\alpha \in \F^X, \gamma \in \F^Z$. This operation has a clear
interpretation in terms of electrical circuits: the power used by the entire
circuit is just the sum of the power used by its parts. 

It is immediate from Theorem \ref{thm:dirichletminimization} that this
composition rule is well defined: the composite of two Dirichlet forms is again
a Dirichlet form. Moreover, this composition is associative. However, it fails
to provide the structure of a category, as there is typically no Dirichlet form
$1_X \in D(X,X)$ playing the role of the identity for this composition. For an
indication of why this is so, let $\{\bullet\}$ be a set with one element, and
suppose that some Dirichlet form $I(\beta,\gamma) = k(\beta-\gamma)^2 \in
D(\{\bullet\},\{\bullet\})$ acts as an identity on the right for this
composition. Then for all $Q(\alpha,\beta) = c(\alpha-\beta)^2 \in
D(\{\bullet\},\{\bullet\})$, we must have
\begin{align*}
  c\alpha^2 &= Q(\alpha,0) \\
  &= (I \circ Q)(\alpha,0) \\ 
  &= \min_{\beta \in \F} Q(\alpha, \beta) + I(\beta,0) \\
  &= \min_{\beta \in \F} k(\alpha-\beta)^2 + c\beta^2 \\
  &= \frac{kc}{k+c}\alpha^2,
\end{align*}
where we have noted that $\frac{kc}{k+c}\alpha^2$ minimizes $k(\alpha-\beta)^2 +
c\beta^2$ with respect to $\beta$. But for any choice of $k \in \F$ this
equality only holds when $c = 0$, so no such Dirichlet form exists. Note,
however, that for $k>> c$ we have $c\alpha^2 \approx \frac{kc}{k+c}\alpha^2$, so
Dirichlet forms with large values of $k$---corresponding to resistors with
resistance close to zero---act as `approximate identities'.

In this way we might interpret the identities we wish to introduce
into this category as the behaviours of idealized components with zero
resistance: perfectly conductive wires. Unfortunately, the power functional of a
purely conductive wire is undefined: the formula for it involves division by
zero.  In real life, coming close to this situation leads to the disaster that
electricians call a `short circuit': a huge amount of power dissipated for even
a small voltage.  This is why we have fuses and circuit breakers.

Nonetheless, we have most of the structure required for a category. A `category
without identity morphisms' is called a \define{semicategory}, so we see
\begin{proposition}
There is a semicategory where:
\begin{itemize}
\item the objects are finite sets,

\item a morphism from $X$ to $Y$ is a Dirichlet form $Q \in D(X,Y)$.  

\item composition of morphisms is given by 
\[
(P \circ Q)(\gamma, \alpha) = \min_{Y} Q(\gamma, \beta) + P(\beta, \alpha).
\]

\end{itemize}
\end{proposition}

We would like to make this into a category. One easy way to do this is to
formally adjoin identity morphisms; this trick works for any semicategory.
However, we obtain a better category if we include \emph{more} morphisms: more
behaviours corresponding to circuits made of perfectly conductive wires.
The expression for the extended power functional includes the reciprocals of
impedances, such circuits cannot be expressed within the framework we have
developed thus far. Similarly, the Frobenius maps also have semantics as ideal
wires, and cannot be represented using Dirichlet forms.  Indeed, for these
idealized circuits there is no function taking boundary potentials to boundary
currents: the vanishing impedance would imply that any difference in potentials
at the boundary induces `infinite' currents. One way of dealing with this is to
use decorated cospans.

\subsection{The category of Dirichlet cospans} \label{ssec.dirichcospans}

Although we cannot have Dirichlet forms be the morphisms of a category
themselves, we have a standard trick for turning data into the morphisms of a
category: we decorate cospans with them. The cospans then handle the composition
for us. In this section we construct a category of cospans and a category of
corelations decorated by Dirichlet forms. We shall think of the former as having
extended power functionals as morphisms, and the latter as power functionals.

Consider a cospan of finite sets $X \to N \leftarrow Y$ together with a
Dirichlet form $Q_N$ on the apex $N$. We call this a \define{Dirichlet cospan}.
To compose such cospans, say when given another cospan $Y \to M \leftarrow Z$
decorated by Dirichlet form $Q_M$, we decorated the composite cospan $X \to
N+_YM \leftarrow Z$ with the Dirichlet form
\begin{align*}
  \Big({j_N}_\ast Q_N+ {j_M}_\ast Q_M\Big)\maps \F^{N+_YM} &\longrightarrow \F;\\
  \phi &\longmapsto Q_N(\phi\circ j_N)+Q_M(\phi\circ j_M),
\end{align*}
where the $j$ are the maps that include $N$ and $M$ into the pushout as usual.
Interpreted in terms of extended power functionals, this simply says that the
power consumed by the interconnected circuit is just the power consumed by each
part. These form the morphisms of a decorated cospan category.

\begin{proposition}
The following defines a lax symmetric monoidal functor
$(\mathrm{Dirich},\delta)$: let
\[
  \mathrm{Dirich}\maps (\mathrm{FinSet},+) \longrightarrow (\mathrm{Set},\times)
\]
map a finite set $X$ to the set $\mathrm{Dirich}(X)$ of Dirichlet forms
$Q\maps \F^X \to \F$ on $X$, and map a function $f\maps X \to Y$ between finite
sets to the pushforward function
\begin{align*}
  \mathrm{Dirich}(f)\maps \mathrm{Dirich}(X) &\longrightarrow \mathrm{Dirich}(Y); \\
  Q &\longmapsto \Big(f_{\ast}Q\maps \phi \mapsto Q(\phi\circ f)\Big).
\end{align*}

For coherence maps, equip $\mathrm{Dirich}$ with the natural family of maps
\begin{align*}
  \delta_{N,M}\maps \mathrm{Dirich}(N) \times \mathrm{Dirich}(M) &\longrightarrow
  \mathrm{Dirich}(N+M) \\
  (Q_N,Q_M) &\longmapsto {\iota_N}_\ast Q_N+{\iota_M}_\ast Q_M
\end{align*}
and also with the unit
\begin{align*}
  \delta_1\maps 1 &\longrightarrow \mathrm{Dirich}(\varnothing);\\
  \bullet &\longmapsto (\F^\varnothing \to \F; ! \mapsto 0).
\end{align*}
Note that the sum of two Dirichlet forms is given pointwise by the addition in
$\F$.
\end{proposition}
\begin{proof}
  As composition of functions is associative and has an identity,
  $\mathrm{Dirich}$ is a functor.  The naturality of the $\delta_{N,M}$ follows
  from the universal property of the coproduct in $\FinSet$, while the symmetric
  monoidal coherence axioms follow from the associativity, unitality, and
  commutativity of addition in $\F$.
\end{proof}

Using decorated cospans, we thus obtain a hypergraph category
$\mathrm{DirichCospan}$ where a morphism is a cospan of finite sets whose apex
is equipped with a Dirichlet form.  Next, we use decorated cospans to construct
a strict hypergraph functor $\Circ \to \mathrm{DirichCospan}$ sending a circuit
to a cospan decorated by its extended power functional. For this we need a
monoidal natural transformation $\theta\maps (\mathrm{Circuit},\rho) \Rightarrow
(\mathrm{Dirich},\delta)$.

\begin{proposition} 
  The collection of maps
\begin{align*}
  \theta_N\maps \mathrm{Circuit}(N) &\longrightarrow \mathrm{Dirich}(N); \\
  (N,E,s,t,r) &\longmapsto \left(\phi \in \F^N \mapsto \frac{1}{2} \sum_{e \in E}
  \frac{1}{r(e)}\big(\phi(s(e))-\phi(t(e))\big)^2\right).
\end{align*}
defines a monoidal natural transformation
\[
  \theta\maps (\mathrm{Circuit},\rho) \Longrightarrow
  (\mathrm{Dirich},\delta).
\]
\end{proposition}
\begin{proof}
Naturality requires that the square
\[
  \xymatrix{
    \mathrm{Circuit}(N) \ar[r]^{\theta_N} \ar[d]_{\mathrm{Circuit}(f)} &
    \mathrm{Dirich}(N) \ar[d]^{\mathrm{Dirich}(f)}  \\
    \mathrm{Circuit}(M) \ar[r]_{\theta_M} & \mathrm{Dirich}(M)
  }
\]
commutes. Let $(N,E,s,t,r)$ be an $\F^+$-graph on $N$ and $f\maps N \to M$ be a
function $N$ to $M$. Then both $\mathrm{Dirich}(f) \circ \theta_N$ and $\theta_M
\circ \mathrm{Circuit}(f)$ map $(N,E,s,t,r)$ to the Dirichlet form
\begin{align*}
  \F^M &\longrightarrow \F;\\
  \psi &\longmapsto \frac{1}{2} \sum_{e \in E}\frac{1}{r(e)}
  \big(\psi(f(s(e)))-\psi(f(t(e)))\big)^2.
\end{align*}
Thus both methods of constructing a power functional on a set of nodes $M$ from
a circuit on $N$ and a function $f\maps N \to M$ produce the same power functional.

To show that $\theta$ is a monoidal natural transformation, we must check that
the square
\[
\xymatrix{
  \mathrm{Circuit}(N) \times \mathrm{Circuit}(M) \ar[r]^{\theta_N \times
  \theta_M} \ar[d]_{\rho_{N,M}} & \mathrm{Dirich}(N) \times \mathrm{Dirich}(M)
  \ar[d]^{\delta_{N,M}}  \\
  \mathrm{Circuit}(N+M) \ar[r]_{\theta_{N+M}} & \mathrm{Dirich}(N+M)
}
\]
and the triangle
\[
\xymatrix{
  & 1 \ar[dl]_{\rho_\varnothing} \ar[dr]^{\delta_\varnothing}\\
\mathrm{Circuit}(\varnothing)  \ar[rr]_{\theta_\varnothing} &&
\mathrm{Dirich}(\varnothing)
}
\]
commute. It is readily observed that both paths around the square lead to taking
two graphs and summing their corresponding Dirichlet forms, and that the
triangle commutes immediately as all objects in it are the one element set.
\end{proof}

From decorated cospans, we thus obtain a strict hypergraph functor
\[
  Q \maps \Circ = \mathrm{CircuitCospan} \longrightarrow \mathrm{DirichCospan}.
\]
Informally, this says that the process of composition for circuit diagrams is
the same as that of composition for Dirichlet cospans. Note that this is not a
faithful functor.  For example, applying $Q$ to a circuit 
\[
\begin{tikzpicture}[circuit ee IEC, set resistor graphic=var resistor IEC
	graphic,scale=.8]
	\node[circle,draw,inner sep=1pt,fill=gray,color=gray]         (x) at
	(-2.8,0) {};
	\node at (-2.8,-1) {\footnotesize $X$};
	\node[circle,draw,inner sep=1pt,fill]         (A) at (0,0) {};
	\node[circle,draw,inner sep=1pt,fill]         (B) at (3,0) {};
	\node[circle,draw,inner sep=1pt,fill=gray,color=gray]         (y1) at
	(5.8,0) {};
	\node at (5.8,-1) {\footnotesize $Y$};
	\coordinate         (ua) at (.5,.25) {};
	\coordinate         (ub) at (2.5,.25) {};
	\coordinate         (la) at (.5,-.25) {};
	\coordinate         (lb) at (2.5,-.25) {};
	\path (A) edge (ua);
	\path (A) edge (la);
	\path (B) edge (ub);
	\path (B) edge (lb);
	\path (ua) edge  [resistor, circuit symbol unit=5pt, circuit symbol
	size=width {5} height 1.5] node[label={[label
	distance=1pt]90:{\footnotesize $r$}}] {} (ub);
	\path (la) edge  [resistor, circuit symbol unit=5pt, circuit symbol
	size=width {5} height 1.5] node[label={[label
	distance=1pt]270:{\footnotesize $s$}}] {} (lb);
	\path[color=gray, very thick, shorten >=10pt, shorten <=5pt, ->, >=stealth] (x) edge (A);
	\path[color=gray, very thick, shorten >=10pt, shorten <=5pt, ->, >=stealth] (y1) edge (B);
      \end{tikzpicture}
    \]
with two parallel edges of resistance $r$ and $s$ respectively, we obtain
the same result as for the circuit
\[
\begin{tikzpicture}[circuit ee IEC, set resistor graphic=var resistor IEC
	graphic,scale=.8]
	\node[circle,draw,inner sep=1pt,fill=gray,color=gray]         (x) at
	(-2.8,0) {};
	\node at (-2.8,-1) {\footnotesize $X$};
	\node[circle,draw,inner sep=1pt,fill]         (A) at (0,0) {};
	\node[circle,draw,inner sep=1pt,fill]         (B) at (3,0) {};
	\node[circle,draw,inner sep=1pt,fill=gray,color=gray]         (y) at
	(5.8,0) {};
	\node at (5.8,-1) {\footnotesize $Y$};
	\path (A) edge  [resistor, circuit symbol unit=5pt, circuit symbol
	size=width {5} height 1.5] node[label={[label
	distance=1pt]90:{\footnotesize $t$}}] {} (B);
	\path[color=gray, very thick, shorten >=10pt, shorten <=5pt, ->, >=stealth] (x) edge (A);
	\path[color=gray, very thick, shorten >=10pt, shorten <=5pt, ->, >=stealth] (y) edge (B);
      \end{tikzpicture}
    \]
with just a single edge with resistance
\[
  t = \frac{1}{\tfrac{1}{r} + \tfrac{1}{s}}.
\]
That is, both map to the Dirichlet cospan $X=\{x\} \to \{x,y\} \leftarrow \{y\}
= Y$ with $\{x,y\}$ decorated by the Dirichlet form $P(\psi) =
\tfrac{1}{t}(\psi_x-\psi_y)^2$. Nonetheless, this functor is far from the
desired semantic functor: many different extended power functionals can restrict
to the same power functional on the boundary. 

\subsection{A category of behaviours}

For our functorial semantics, we wish to map a circuit to its power functional.
We take a moment to briefly sketch how one could use decorated corelations to
construct a codomain for such a functor. The result, however, is a little ad
hoc, and we shall not pursue it in depth. Instead, in the next section, we shall
generalize Dirichlet forms to Lagrangian relations.

Indeed, ideally we would have liked to use an isomorphism-morphism factorisation on
$\FinSet$, to construct isomorphism-morphism corelations decorated by Dirichlet
forms. Then for a map $X \to Y$ the Dirichlet form must decorate the set
$X+Y$, and we need not talk at all about cospans. We saw however, in the
previous section, that a category with morphisms Dirichlet forms on the disjoint
union of the domain and codomain is not posible.  As a consolation prize we may
use the epi-mono factorisation system $(\mathrm{Sur},\mathrm{Inj})$ on
$\FinSet$. 

The key idea is that given an injection $m\maps \overline N \to N$, we may
minimise a Dirichlet form on $N$ over the complement of the image of $\overline
N$ to obtain, in effect, a Dirichlet form on $\overline N$. Given a circuit $X
\to Y$, we may factor the function $X+Y \to N$ as $X+Y \stackrel{e}\to
\overline{N} \stackrel{m}\to N$, where $e$ is a surjection and $m$ is an
injection. Note the image $m(\overline{N})$ of $m$ in $N$ is equal to the
boundary $\partial N$ of the circuit. Thus if we map a circuit to its extended
power functional, and then minimise the Dirichlet form onto $\overline N$, we
obtain the power functional of the circuit. 

This motivates the following proposition. 

\begin{proposition}
  There is a lax symmetric monoidal functor
\[
  \mathrm{Dirich}\maps (\mathrm{FinSet};\mathrm{Inj}^\opp,+) \longrightarrow
  (\mathrm{Set},\times),
\]
extending the functor $\mathrm{Dirich} \maps (\FinSet,+) \rightarrow
(\Set,\times)$, such that the image of the cospan $N \stackrel{f}\to A
\stackrel{m}\hookleftarrow M$, where $f$ is a function and $m$ an injection, is the
map
\begin{align*}
  \mathrm{Dirich}(f;m^\opp)\maps \mathrm{Dirich}(N) &\longrightarrow
  \mathrm{Dirich}(M);\\
  Q &\longmapsto \min_{A\setminus M} f_\ast Q.
\end{align*}
Note that in writing $A \setminus M$ we are considering $M$ as a subset of $A$
by way of the injection $m$.

\end{proposition} 
By extending the previous functor $\mathrm{Dirich}$, we mean the triangle
\[
  \xymatrixcolsep{3pc}
  \xymatrixrowsep{1pc}
  \xymatrix{
    (\FinSet, +) \ar[dr]^{\mathrm{Dirich}} \ar@{^{(}->}[dd] \\
    & (\Set,\times) \\
 (\mathrm{FinSet};\mathrm{Inj}^\opp,+) \ar[ur]_{\mathrm{Dirich}}.
  }
\]
commutes, where the vertical map is the subcategory inclusion.

There are two aspects of this proposition that require detailed proof: that the
map $\mathrm{Dirich}$ preserves composition, and that the coherence maps
$\delta$ are still natural in this larger category. Both facts come down to
proving that the minimisation and pushforward commute in certain required circumstances. 

This extended functor $\mathrm{Dirich}$ then gives rise to a decorated
corelations category $\mathrm{DirichCorel}$ with morphisms jointly-epic cospans
decorated by Dirichlet forms; composition is given by summing the Dirichlet
forms on the factors, then minimising over the interior.

This category satisfies the two precise desiderata for a semantic category: each
behaviour is uniquely represented, and the map of a circuit to its behaviour
preserves all compositional structure. Nonetheless, we will not pursue this
construction in depth, feeling that this construction uses corelations to
shoehorn in the necessary compositional structure. Instead, we find that a
cleaner, more general construction, based on the idea of a linear relation, can
be given by a slight generalisation of our decorations.

\section{Networks as Lagrangian relations} \label{sec:circlagr}
%%fakesubsection
In the first part of this chapter, we explored the semantic content contained in
circuit diagrams, leading to an understanding of circuit diagrams as expressing
some relationship between the potentials and currents that can simultaneously be
imposed on some subset, the so-called terminals, of the nodes of the circuit. We
called this collection of possible relationships the behaviour of the circuit.
While in that setting we used the concept of Dirichlet forms to describe this
relationship, we saw in the end that describing circuits as Dirichlet forms does
not allow for a straightforward notion of composition of circuits. 

In this section, inspired by the principle of least action of classical
mechanics in analogy with the principle of minimum power, we develop a setting
for describing behaviours that allows for easy discussion of composite
behaviours: Lagrangian subspaces of symplectic vector spaces. These Lagrangian
subspaces provide a more direct, invariant perspective, comprising precisely the
set of vectors describing the possible simultaneous potential and current
readings at all terminals of a given circuit. As we shall see, one immediate and
important advantage of this setting is that we may model wires of zero
resistance. Moreover, they lie within our Willems-esque aesthetic, choosing to
define physical systems via a complete list of possible observations at the
terminals.

Recall that we write $\F$ for some field, which for our applications is
usually the field $\R$ of real numbers or the field $\R(s)$ of rational
functions of one real variable.

\subsection{Symplectic vector spaces}

A circuit made up of wires of positive resistance defines a function from
boundary potentials to boundary currents. A wire of zero resistance, however,
does not define a function: the principle of minimum power is obeyed as long as
the potentials at the two ends of the wire are equal. More generally, we may
thus think of circuits as specifying a set of allowed voltage-current pairs, or
as a relation between boundary potentials and boundary currents. This set forms
what is called a Lagrangian subspace, and is given by the graph of the
differential of the power functional. More generally, Lagrangian submanifolds
graph derivatives of smooth functions: they describe the point evaluated and the
tangent to that point within the same space.

The material in this section is all known, and follows without great difficulty
from the definitions. To keep this section brief we omit proofs. See any
introduction to symplectic vector spaces, such as Cimasoni and Turaev \cite{CT} or
Piccione and Tausk \cite{PT}, for details.

\begin{definition}
  Given a finite-dimensional vector space $V$ over a field $\F$, a 
  \define{symplectic form}
  $\omega\maps V \times V \to \F$ on $V$ is an alternating nondegenerate bilinear
  form.  That is, a symplectic form $\omega$ is a function $V \times V \to \F$
  that is
  \begin{enumerate}[(i)]
    \item bilinear: for all $\lambda \in \F$ and all $u,v \in V$ we have
      $\omega(\lambda u,v) = \omega(u,\lambda v) =  \lambda \omega(u,v)$;
    \item alternating: for all $v \in V$ we have $\omega(v,v) = 0$; and
    \item nondegenerate: given $v \in V$, $\omega(u,v) = 0$ for all $u \in V$ if
      and only if $u = 0$.
  \end{enumerate} 
  A \define{symplectic vector space} $(V,\omega)$ is a vector space $V$ equipped
  with a symplectic form $\omega$. 

  Given symplectic vector spaces $(V_1,\omega_1), (V_2, \omega_2)$, a
  \define{symplectic map} is a linear map 
  \[
    f\maps (V_1,\omega_1) \longrightarrow (V_2, \omega_2)
  \]
  such that $\omega_2(f(u),f(v)) = \omega_1(u,v)$ for all $u,v \in V_1$. A
  \define{symplectomorphism} is a symplectic map that is also an isomorphism. 
\end{definition}

An alternating form is always \define{antisymmetric}, meaning that $\omega(u,v) = 
-\omega(v,u)$ for all $u,v \in V$.  The converse is true except in characteristic 2.
A \define{symplectic basis} for a symplectic vector space $(V,\omega)$ is a
basis $\{p_1,\dots,p_n,q_1,\dots,q_n\}$ such that $\omega(p_i,p_j) =
\omega(q_i,q_j) = 0$ for all $1 \le i,j \le n$, and $\omega(p_i,q_j) =
\delta_{ij}$ for all $1 \le i,j\le n$, where $\delta_{ij}$ is the Kronecker delta,
equal to $1$ when $i =j$, and $0$ otherwise. A symplectomorphism maps symplectic
bases to symplectic bases, and conversely, any map that takes a symplectic basis
to another symplectic basis is a symplectomorphism.

\begin{example}
  \label{ex:symplectic_space_generated_by_set}
  Given a finite set $N$, we consider the vector space $\vectf{N}$ a symplectic
  vector space $(\vectf{N},\omega)$, with symplectic form 
  \[
    \omega\big((\phi,i),(\phi',i')\big) = i'(\phi)-i(\phi').  
  \] 
  Let $\{\phi_n\}_{n \in N}$ be the basis of $\F^N$ consisting of the functions
  $N \to \F$ mapping $n$ to $1$ and all other elements of $N$ to $0$, and let
  $\{i_n\}_{n \in N} \subseteq {(\F^N)}^\ast$ be the dual basis. Then
  $\{(\phi_n,0),(0,i_n)\}_{n\in N}$ forms a symplectic basis for $\vectf{N}$.  

  This is known as the \define{standard symplectic space}; an important
  structure theorem states that every finite dimensional symplectic vector space
  is symplectomorphic---that is, isomorphic as a symplectic vector space---to
  one of this form. 
\end{example}

There are two common ways we will build symplectic spaces from other symplectic
spaces: conjugation and summation. Given a symplectic form $\omega$ on $V$, we
may define its \define{conjugate} symplectic form $\overline\omega = - \omega$,
and write the conjugate symplectic space $(V,\overline\omega)$ as $\overline V$.
Given two symplectic vector spaces $(U, \nu),(V,\omega)$, we consider their
direct sum $U \oplus V$ a symplectic vector space with the symplectic form
$\nu+\omega$, and call this the \define{sum} of the two symplectic vector
spaces. Note that this is not a product in the category of symplectic vector
spaces and symplectic maps.

\begin{example}
  \label{ex:symplectomorphisms}
  A symplectic vector space and its conjugate are symplectomorphic. This is a
  necessary consequence of the fact that they have the same dimension, and so
  are both symplectomorphic to the same standard symplectic space. 
  
  We also define the conjugate of the standard symplectic space,
  $\overline{\vectf{N}}$, as the vector space $\vectf{N}$ together with the
  (unsurprising) symplectic form
  \[
    \omega\big((\phi,i),(\phi',i')\big) = i(\phi')-i'(\phi).  
  \] 
  We write 
  \[
    \overline{\idn_N}\maps \vectf{N} \longrightarrow \overline{\vectf{N}}
  \]
  for the symplectomorphism mapping $(\phi,i)$ to $(\phi,-i)$ relating these two
  spaces.
\end{example}

The symplectic form provides a notion of orthogonal complement.  Given a subspace $S$ of $V$, we define its \define{complement}
\[
  S^\circ = \{v \in V \mid \omega(v,s) = 0 \textrm{ for all } s \in S\}.
\]
Note that this construction obeys the following identities, where $S$ and $T$
are subspaces of $V$:
\begin{align*}
  \dim S+ \dim S^\circ &= \dim V \\
  (S^\circ)^\circ &= S \\
  (S + T)^\circ &= S^\circ \cap T^\circ \\
  (S \cap T)^\circ &= S^\circ + T^\circ.
\end{align*}

In the symplectic vector space $\vectf{N}$, the subspace $\F^N$ has the
property of being a maximal subspace such that the symplectic form restricts to
the zero form on this subspace. Subspaces with this property are known as Lagrangian
subspaces, and they may all be realized as the image of $\vectf{N}$ 
under symplectomorphisms from $\vectf{N}$ to itself.

\begin{definition} 
  Let $S$ be a linear subspace of a symplectic vector space $(V,\omega)$. We say
  that $S$ is \define{isotropic} if $\omega|_{S \times S} = 0$, and that $S$ is
  \define{coisotropic} if $S^\circ$ is isotropic. A subspace is
  \define{Lagrangian} if it is both isotropic and coisotropic, or equivalently, if it  
  is a maximal isotropic subspace.
\end{definition}

Lagrangian subspaces are also known as Lagrangian correspondences and canonical
relations. Note that a subspace $S$ is isotropic if and only if $S \subseteq
S^\circ$. This fact helps with the following characterizations of Lagrangian
subspaces.

\begin{proposition} \label{lagrangian_characterization} 
  Given a subspace $L \subset V$ of a symplectic vector space $(V,\omega)$, the
  following are equivalent: 
  \begin{enumerate}[(i)] 
    \item $L$ is Lagrangian.  
    \item $L$ is maximally isotropic.  
    \item $L$ is minimally coisotropic.  
    \item $L = L^\circ$.  
    \item $L$ is isotropic and $\dim L = \frac12 \dim V$.
  \end{enumerate} 
\end{proposition}

From this proposition it follows easily that the direct sum of two Lagrangian
subspaces is Lagrangian in the sum of their ambient spaces. We also observe that
an advantage of isotropy is that there is a good way to take a quotient of a
symplectic vector space by an isotropic subspace---that is, there is a way to
put a natural symplectic structure on the quotient space.

\begin{proposition}
  Let $S$ be an isotropic subspace of a symplectic vector space $(V,\omega)$.
  Then $S^\circ/S$ is a symplectic vector space with symplectic form
  $\omega'(v+S,u+S) = \omega(v,u)$.
\end{proposition}
\begin{proof} 
  The function $\omega'$ is well defined due to the isotropy of
  $S$---by definition adding any pair $(s,s')$ of elements of $S$ to a pair
  $(v,u)$ of elements of $S^\circ$ does not change the value of
  $\omega(v+s,u+s')$. As $\omega$ is a symplectic form, one can check that
  $\omega'$ is too.  
\end{proof}

\subsection{Lagrangian subspaces from quadratic forms}

Lagrangian subspaces are of relevance to us here as the behaviour of any passive
linear circuit forms a Lagrangian subspace of the symplectic vector space
generated by the nodes of the circuit. We think of this vector space as
comprising two parts: a space $\F^N$ of potentials at each node, and a dual
space ${(\F^N)}^\ast$ of currents. To make clear how circuits can be interpreted
as Lagrangian subspaces, here we describe how Dirichlet forms on a finite set
$N$ give rise to Lagrangian subspaces of $\vectf{N}$. More generally, we show
that there is a one-to-one correspondence between Lagrangian subspaces and
quadratic forms.

\begin{proposition} \label{prop:qfls}
  Let $N$ be a finite set. Given a quadratic form $Q$ over $\F$ on $N$, the
  subspace 
  \[ 
    L_Q = \big\{(\phi,dQ_\phi) \mid \phi \in \F^N\big\} \subseteq \vectf{N},
  \] 
  where $dQ_\phi \in {(\F^N)}^\ast$ is the formal differential of $Q$ at $\phi
  \in \F^N$, is Lagrangian. Moreover, this construction gives a one-to-one correspondence 
  \[ 
    \left\{\begin{array}{c} 
      \\ \mbox{Quadratic forms over $\F$ on $N$} \\ \phantom{.}
    \end{array} \right\} 
    \longleftrightarrow
    \left\{\begin{array}{c} 
      \mbox{Lagrangian subspaces of $\vectf{N}$}\\
      \mbox{with trivial intersection with} \\ 
      \{0\} \oplus {(\F^N)}^\ast \subseteq \vectf{N} 
    \end{array} \right\}.  
  \]
\end{proposition}
\begin{proof}
  The symplectic structure on $\vectf{N}$ and our notation for it is given in
  Example \ref{ex:symplectic_space_generated_by_set}. 

  Note that for all $n,m \in N$ the corresponding basis elements 
  \[
    \frac{\partial^2 Q}{\partial \phi_n \partial \phi_m} = dQ_{\phi_n}(\phi_m) =
    dQ_{\phi_m}(\phi_n),
  \]
  so $dQ_\phi(\psi) = dQ_\psi(\phi)$ for all $\phi,\psi \in \F^N$. Thus $L_Q$ is
  indeed Lagrangian: for all $\phi,\psi \in \F^N$ 
  \[
    \omega\big((\phi,dQ_\phi),(\psi,dQ_\psi)\big) = dQ_\psi(\phi) -
    dQ_\phi(\psi) = 0.
  \]

  Observe also that for all quadratic forms $Q$ we have $dQ_0 = 0$, so the only
  element of $L_Q$ of the form $(0,i)$, where $i \in {(\F^N)}^\ast$, is $(0,0)$.
  Thus $L_Q$ has trivial intersection with the subspace $\{0\} \oplus
  {(\F^N)}^\ast$ of $\vectf{N}$. This $L_Q$ construction forms the leftward
  direction of the above correspondence.

  For the rightward direction, suppose that $L$ is a Lagrangian subspace of
  $\vectf{N}$ such that $L \cap (\{0\} \oplus {(\F^N)}^\ast) = \{(0,0)\}$. Then
  for each $\phi \in \F^N$, there exists a unique $i_\phi \in {(\F^N)}^\ast$
  such that $(\phi,i_\phi) \in L$. Indeed, if $i_\phi$ and $i_\phi'$ were
  distinct elements of ${(\F^N)}^\ast$ with this property, then by linearity
  $(0,i_\phi-i_\phi')$ would be a nonzero element of $L \cap (\{0\} \oplus
  {(\F^N)}^\ast)$, contradicting the hypothesis about trivial intersection. We
  thus can define a function, indeed a linear map, $\F^N \to {(\F^N)}^\ast; \phi
  \mapsto i_\phi$. This defines a bilinear form $Q(\phi,\psi) = i_\phi(\psi)$ on
  $\F^N \oplus \F^N$, and so $Q(\phi) = i_\phi(\phi)$ defines a quadratic form
  on $\F^N$. 

  Moreover, $L$ is Lagrangian, so
  \[
    \omega\big((\phi,i_\phi), (\psi,i_\psi)\big) = i_\psi(\phi) - i_\phi(\psi) =
    0,
  \]
  and so $Q(-,-)$ is a symmetric bilinear form. This gives a one-to-one
  correspondence between Lagrangian subspaces of specified type, symmetric
  bilinear forms, and quadratic forms, and so in particular gives the claimed
  one-to-one correspondence. 
\end{proof}

In particular, every Dirichlet form defines a Lagrangian subspace. 

\subsection{Lagrangian relations}

Recall that a relation between sets $X$ and $Y$ is a subset $R$ of their product
$X \times Y$. Furthermore, given relations $R \subseteq X \times Y$ and $S
\subseteq Y \times Z$, there is a composite relation $(S \circ R) \subseteq X
\times Z$ given by pairs $(x,z)$ such that there exists $y \in Y$ with $(x,y)
\in R$ and $(y,z) \in S$---a direct generalization of function composition. A
Lagrangian relation between symplectic vector spaces $V_1$ and $V_2$ is a
relation between $V_1$ and $V_2$ that forms a Lagrangian subspace of the
symplectic vector space $\overline{V_1} \oplus V_2$. This gives us a
way to think of certain Lagrangian subspaces, such as those arising from
circuits, as morphisms, giving a way to compose them.

\begin{definition}
  A \define{Lagrangian relation} $L\maps V_1 \to V_2$ is a Lagrangian subspace $L$
  of $\overline{V_1} \oplus V_2$. 
\end{definition}

This is a generalization of the notion of symplectomorphism: any symplectomorphism
$f\maps V_1 \to V_2$ forms a Lagrangian subspace when viewed as a
relation $f \subseteq \overline{V_1} \oplus V_2$. More generally, any symplectic
map $f\maps V_1 \to V_2$ forms an isotropic subspace when viewed as a relation in
$\overline{V_1} \oplus V_2$. 

Importantly for us, the composite of two Lagrangian relations is again a
Lagrangian relation.  This is well known \cite{Weinstein}, but sufficiently 
easy and important to us that we provide a proof.

\begin{proposition} \label{prop:lagrangian_composition}
  Let $L\maps V_1 \to V_2$ and $L'\maps V_2 \to V_3$ be Lagrangian relations. Then their
  composite relation $L' \circ L$ is a Lagrangian relation $V_1 \to V_3$.
\end{proposition}

We prove this proposition by way of two lemmas detailing how the Lagrangian
property is preserved under various operations. The first lemma says that the
intersection of a Lagrangian space with a coisotropic space is in some sense
Lagrangian, once we account for the complement.

\begin{lemma} \label{restriction_of_lagrangians}
  Let $L \subseteq V$ be a Lagrangian subspace of a symplectic vector space $V$,
  and $S \subseteq V$ be an isotropic subspace of $V$. Then $(L\cap S^\circ) +S
  \subseteq V$ is Lagrangian in $V$.
\end{lemma}
\begin{proof}
  Recall from Proposition \ref{lagrangian_characterization} that a subspace is
  Lagrangian if and only if it is equal to its complement. The lemma is then
  immediate from the way taking the symplectic complement interacts with sums
  and intersections:
  \begin{multline*}
    ((L\cap S^\circ) +S)^\circ = (L\cap S^\circ)^\circ \cap S^\circ = (L^\circ +
    (S^\circ)^\circ) \cap S^\circ \\
    = (L+S) \cap S^\circ = (L \cap S^\circ)+(S
    \cap S^\circ) = (L\cap S^\circ) +S.
  \end{multline*}
  Since $(L\cap S^\circ) +S$ is equal to its complement, it is Lagrangian.
\end{proof}

The second lemma says that if a subspace of a coisotropic space is Lagrangian,
taking quotients by the complementary isotropic space does not affect this.

\begin{lemma} \label{quotients_of_lagrangians}
  Let $L \subseteq V$ be a Lagrangian subspace of a symplectic vector space $V$,
  and $S \subseteq L$ an isotropic subspace of $V$ contained in $L$. Then $L/S
  \subseteq S^\circ/S$ is Lagrangian in the quotient symplectic space
  $S^\circ/S$.
\end{lemma}
\begin{proof}
  As $L$ is isotropic and the symplectic form on $S^\circ/S$ is given by
  $\omega'(v+S,u+S) = \omega(v,u)$, the quotient $L/S$ is immediately isotropic.
  Recall from Proposition \ref{lagrangian_characterization} that an isotropic
  subspace $S$ of a symplectic vector space $V$ is Lagrangian if and only if
  $\dim S = \frac12 \dim V$. Also recall that for any subspace $\dim S + \dim
  S^\circ = \dim V$. Thus
  \begin{multline*}
    \dim(L/S) = \dim L - \dim S = \tfrac12 \dim V - \dim S \\ = \tfrac12(\dim S
    + \dim S^\circ) - \dim S = \tfrac12(\dim S^\circ - \dim S) = \tfrac12
    \dim(S^\circ/S).
  \end{multline*}
  Thus $L/S$ is Lagrangian in $S^\circ/S$.
\end{proof}

Combining these two lemmas gives a proof that the composite of two Lagrangian
relations is again a Lagrangian relation.

\begin{proof}[Proof of Proposition \ref{prop:lagrangian_composition}]
  Let $\Delta$ be the diagonal subspace
  \[
    \Delta = \{(0,v_2,v_2,0) \mid v_2 \in V_2\} \subseteq \overline{V_1} \oplus
    V_2 \oplus \overline{V_2} \oplus V_3.
  \]
  Observe that $\Delta$ is isotropic, and has coisotropic complement
  \[
    \Delta^\circ = \{(v_1,v_2,v_2,v_3) \mid v_i \in V_i\} \subseteq
    \overline{V_1} \oplus V_2 \oplus \overline{V_2} \oplus V_3.
  \]
  As $\Delta$ is the kernel of the restriction of the projection map
  $\overline{V_1} \oplus V_2 \oplus \overline{V_2} \oplus V_3 \to \overline{V_1}
  \oplus V_3$ to $\Delta^\circ$, and after restriction this map is still
  surjective, the quotient space $\Delta^\circ/\Delta$ is isomorphic to
  $\overline{V_1} \oplus V_3$. 

  Now, by definition of composition of relations, 
  \[
    L' \circ L = \{(v_1,v_3) \mid \mbox{there exists } v_2 \in V_2 \mbox{ such
    that } (v_1,v_2) \in L, (v_2,v_3) \in L'\}.
  \]
  But note also that 
  \[
    L \oplus L'  = \{(v_1,v_2,v_2',v_3) \mid (v_1,v_2) \in L, (v_2',v_3) \in
    L'\},
  \]
  so 
  \[
    (L \oplus L')\cap \Delta^\circ = \{(v_1,v_2,v_2,v_3) \mid \mbox{there exists
    } v_2 \in V_2 \mbox{ such that } (v_1,v_2) \in L, (v_2,v_3) \in L'\}.
  \]
  Quotienting by $\Delta$ then gives
  \[
    L' \circ L = ((L \oplus L')\cap \Delta^\circ)+\Delta)/\Delta.
  \]
  As $L' \oplus L$ is Lagrangian in $\overline{V_1} \oplus V_2 \oplus
  \overline{V_2} \oplus V_3$, Lemma \ref{restriction_of_lagrangians} says that
  $(L' \oplus L)\cap \Delta^\circ)+\Delta$ is also Lagrangian in $\overline{V_1}
  \oplus V_2 \oplus \overline{V_2} \oplus V_3$. Lemma
  \ref{quotients_of_lagrangians} thus shows that $L' \circ L$ is Lagrangian in
  $\Delta^\circ/\Delta = \overline{V_1} \oplus V_3$, as required.
\end{proof}

Note that this composition is associative. We shall prove this composition
agrees with composition of Dirichlet forms, and hence also composition of
circuits. 

\subsection{The symmetric monoidal category of Lagrangian relations}

Lagrangian relations solve the identity problems we had with Dirichlet forms:
given a symplectic vector space $V$, the Lagrangian relation $\idn\maps V \to V$
specified by the Lagrangian subspace
\[
  \idn = \{(v,v) \mid v \in V\} \subseteq \overline{V} \oplus V,
\]
acts as an identity for composition of relations. We thus have a category. As
our circuits have finitely many nodes, we choose only the finite dimensional
symplectic vector spaces as objects.

\begin{definition}
  We write $\LagrRel$ for the category with finite dimensional
  symplectic vector spaces as objects and Lagrangian relations as morphisms. 
\end{definition}

We define the tensor product of two objects of $\LagrRel$ to be their
direct sum. Similarly, we define the tensor product of two morphisms $L\maps U
\to V$, $L \subseteq \overline{U}\oplus V$ and $K\maps T \to W$, $K \subseteq
\overline{T} \oplus W$ to be their direct sum
\[
  L \oplus K \subseteq \overline{U}\oplus V \oplus\overline{T} \oplus W,
\]
\emph{but} considered as a subspace of the naturally isomorphic space
$\overline{U \oplus T} \oplus V \oplus W$.  Despite this subtlety, we abuse our
notation and write their tensor product $L \oplus K\maps U \oplus T \to V \oplus
W$, and move on having sounded this note of caution. 

Note that the direct sum of two Lagrangian subspaces is
again Lagrangian in the direct sum of their ambient spaces, and the zero
dimensional vector space $\{0\}$ acts as an identity for direct sum. Indeed,
defining for all objects $U,V,W$ in $\LagrRel$ unitors: 
\begin{align*}
  \lambda_V &= \{(0,v,v)\} \subseteq \overline{\{0\} \oplus V} \oplus V, \\
  \rho_V &= \{(v,0,v)\} \subseteq \overline{V \oplus \{0\}} \oplus V,
\end{align*}
associators:
\[
  \alpha_{U,V,W}= \{(u,v,w,u,v,w)\} \subseteq \overline{(U \oplus V)\oplus W}
  \oplus U \oplus (V \oplus W),
\]
and braidings:
\[
  \s_{U,V} = \{(u,v,v,u) \mid u \in U, v \in V\} \subseteq \overline{U \oplus V}
  \oplus V \oplus U,
\]
we have a symmetric monoidal category.  Note that all these structure
maps come from symplectomorphisms between the domain and codomain. From this
viewpoint it is immediate that all the necessary diagrams commute, so we
have a symmetric monoidal category. 

In fact the move to the setting of Lagrangian relations, rather than Dirichlet
forms, adds far richer structure than just identity morphisms. In the next
section we endow $\LagrRel$ with a hypergraph structure.

\section{Ideal wires and corelations} \label{sec:corel}
%%fakesubsection
We want to give a decorated corelations construction of $\LagrRel$ for two main
reasons: first, it endows $\LagrRel$ with a hypergraph structure, and second, it
allows us to use decorated corelations to construct functors to it.  

The key takeaway of this section is that examining the graphical elements that
facilitate composition---here ideal wires---shows us how to define hypergraph
structure and hence obtain a decorated corelations construction. More broadly,
to turn a semantic function into a semantic \emph{functor}, it's enough to look
at the Frobenius maps.

In this section we shall see that these ideal wires are modelled by epi-mono
corelations in $\FinSet$, a generalization of the notion of function that
forgets the directionality from the domain to the codomain. We then observe
that Kirchhoff's laws follow directly from interpreting these structures in the
category of linear relations.

A more powerful route to the same conclusion might be to directly define a
hypergraph structure on $\mathrm{LagrRel}$ and then use
Theorem~\ref{thm.hypdeccorcats} to obtain a decorated corelations construction;
we nonetheless favour the following as it sheds insight into remarkable links
between the corelations and ideal wires, the contravariant `free' vector space
functor and Kirchhoff's voltage law, and the covariant free vector space functor
and Kirchhoff's current law. 

\subsection{Ideal wires as corelations}

The term corelation in this section exclusively refers to epi-mono corelations
in $\FinSet$. As such, we write this hypergraph category $\mathrm{Corel}$. 

To motivate the use of this category, let us start with a set of input
terminals $X$, and a set of output terminals $Y$.  We may connect these
terminals with ideal wires of zero impedance, whichever way we like---input to
input, output to output, input to output---producing something like:
\[
  \begin{tikzpicture}[circuit ee IEC]
	\begin{pgfonlayer}{nodelayer}
		\node [contact] (0) at (-2, 1) {};
		\node [contact] (1) at (-2, 0.5) {};
		\node [contact] (2) at (-2, -0) {};
		\node [contact] (3) at (-2, -0.5) {};
		\node [contact] (4) at (-2, -1) {};
		\node [contact] (5) at (1, 0.75) {};
		\node [contact] (6) at (1, 0.25) {};
		\node [contact] (7) at (1, -0.25) {};
		\node [contact] (8) at (1, -0.75) {};
		\node [style=none] (9) at (-2.75, -0) {$X$};
		\node [style=none] (10) at (1.75, -0) {$Y$};
	\end{pgfonlayer}
	\begin{pgfonlayer}{edgelayer}
	  \draw [thick] (0.center) to (5.center);
		\draw [thick] (5.center) to (1.center);
		\draw [thick] (6.center) to (1.center);
		\draw [thick] (3.center) to (2.center);
		\draw [thick] (4.center) to (8.center);
		\draw [thick] (5.center) to (6.center);
		\draw [thick] (6.center) to (0.center);
	\end{pgfonlayer}
\end{tikzpicture}
\]
In doing so, we introduce a notion of equivalence on our terminals, where two 
terminals are equivalent if we, or if electrons, can traverse from one to 
another via some sequence of wires.   Because of this, we consider our 
perfectly-conducting components to be equivalence relations on $X+Y$,
transforming the above picture into
\[
  \begin{tikzpicture}[circuit ee IEC]
	\begin{pgfonlayer}{nodelayer}
		\node [contact, outer sep=5pt] (0) at (-2, 1) {};
		\node [contact, outer sep=5pt] (1) at (-2, 0.5) {};
		\node [contact, outer sep=5pt] (2) at (-2, -0) {};
		\node [contact, outer sep=5pt] (3) at (-2, -0.5) {};
		\node [contact, outer sep=5pt] (4) at (-2, -1) {};
		\node [contact, outer sep=5pt] (5) at (1, 0.75) {};
		\node [contact, outer sep=5pt] (6) at (1, 0.25) {};
		\node [contact, outer sep=5pt] (7) at (1, -0.25) {};
		\node [contact, outer sep=5pt] (8) at (1, -0.75) {};
		\node [style=none] (9) at (-2.75, -0) {$X$};
		\node [style=none] (10) at (1.75, -0) {$Y$};
		\node [style=none] (11) at (-0.5, 0.625) {};
		\node [style=none] (12) at (-0.5, -0.25) {};
		\node [style=none] (13) at (-0.5, -0.875) {};
	\end{pgfonlayer}
	\begin{pgfonlayer}{edgelayer}
		\draw [color=gray] (0.center) to (11.center);
		\draw [color=gray] (1.center) to (11.center);
		\draw [color=gray] (5.center) to (11.center);
		\draw [color=gray] (6.center) to (11.center);
		\draw [color=gray] (2.center) to (12.center);
		\draw [color=gray] (12.center) to (3.center);
		\draw [color=gray] (4.center) to (13.center);
		\draw [color=gray] (13.center) to (8.center);
		\draw [rounded corners=5pt, dotted] 
   (node cs:name=0, anchor=north west) --
   (node cs:name=1, anchor=south west) --
   (node cs:name=6, anchor=south east) --
   (node cs:name=5, anchor=north east) --
   cycle;
		\draw [rounded corners=5pt, dotted] 
   (node cs:name=2, anchor=north west) --
   (node cs:name=3, anchor=south west) --
   (node cs:name=3, anchor=south east) --
   (node cs:name=2, anchor=north east) --
   cycle;
		\draw [rounded corners=5pt, dotted] 
   (node cs:name=4, anchor=north west) --
   (node cs:name=4, anchor=south west) --
   (node cs:name=8, anchor=south east) --
   (node cs:name=8, anchor=north east) --
   cycle;
		\draw [rounded corners=5pt, dotted] 
   (node cs:name=7, anchor=north west) --
   (node cs:name=7, anchor=south west) --
   (node cs:name=7, anchor=south east) --
   (node cs:name=7, anchor=north east) --
   cycle;
	\end{pgfonlayer}
\end{tikzpicture}
\]
The dotted lines indicate equivalence classes of points, while for reference the
grey lines indicate ideal wires connecting these points, running through a
central hub.

Given another circuit of this sort, say from sets $Y$ to $Z$,
\[
\begin{tikzpicture}[circuit ee IEC]
	\begin{pgfonlayer}{nodelayer}
		\node [style=none] (0) at (-2.75, -0) {$Y$};
		\node [style=none] (1) at (1.75, 0) {$Z$};
		\node [contact, outer sep=5pt] (2) at (-2, 0.75) {};
		\node [contact, outer sep=5pt] (3) at (-2, 0.25) {};
		\node [contact, outer sep=5pt] (4) at (-2, -0.25) {};
		\node [contact, outer sep=5pt] (5) at (-2, -0.75) {};
		\node [contact, outer sep=5pt] (6) at (1, 1) {};
		\node [contact, outer sep=5pt] (7) at (1, 0.5) {};
		\node [contact, outer sep=5pt] (8) at (1, -0) {};
		\node [contact, outer sep=5pt] (9) at (1, -0.5) {};
		\node [contact, outer sep=5pt] (10) at (1, -1) {};
		\node [style=none] (11) at (-0.5, 0.75) {};
		\node [style=none] (12) at (-0.5, -0.25) {};
		\node [style=none] (13) at (-0.5, -0.875) {};
	\end{pgfonlayer}
	\begin{pgfonlayer}{edgelayer}
	  \draw [color=gray] (2.center) to (11.center);
		\draw [color=gray] (11.center) to (6.center);
		\draw [color=gray] (3.center) to (11.center);
		\draw [color=gray] (8.center) to (12.center);
		\draw [color=gray] (12.center) to (4.center);
		\draw [color=gray] (9.center) to (12.center);
		\draw [color=gray] (5.center) to (13.center);
		\draw [color=gray] (13.center) to (10.center);
	\end{pgfonlayer}
		\draw [rounded corners=5pt, dotted] 
   (node cs:name=2, anchor=north west) --
   (node cs:name=3, anchor=south west) --
   (node cs:name=6, anchor=south east) --
   (node cs:name=6, anchor=north east) --
   cycle;
		\draw [rounded corners=5pt, dotted] 
   (node cs:name=4, anchor=north west) --
   (node cs:name=4, anchor=south west) --
   (node cs:name=9, anchor=south east) --
   (node cs:name=8, anchor=north east) --
   cycle;
		\draw [rounded corners=5pt, dotted] 
   (node cs:name=5, anchor=north west) --
   (node cs:name=5, anchor=south west) --
   (node cs:name=10, anchor=south east) --
   (node cs:name=10, anchor=north east) --
   cycle;
		\draw [rounded corners=5pt, dotted] 
   (node cs:name=7, anchor=north west) --
   (node cs:name=7, anchor=south west) --
   (node cs:name=7, anchor=south east) --
   (node cs:name=7, anchor=north east) --
   cycle;
\end{tikzpicture}
\]
we may combine these circuits in to a circuit $X$ to $Z$
\[
  \begin{aligned}
\begin{tikzpicture}[circuit ee IEC]
	\begin{pgfonlayer}{nodelayer}
		\node [contact, outer sep=5pt] (0) at (1, 0.75) {};
		\node [contact, outer sep=5pt] (1) at (1, 0.25) {};
		\node [contact, outer sep=5pt] (2) at (1, -0.25) {};
		\node [contact, outer sep=5pt] (3) at (1, -0.75) {};
		\node [style=none] (4) at (-2.75, -0) {$X$};
		\node [style=none] (5) at (4.75, -0) {$Z$};
		\node [contact, outer sep=5pt] (6) at (-2, 1) {};
		\node [contact, outer sep=5pt] (7) at (-2, -0.5) {};
		\node [contact, outer sep=5pt] (8) at (-2, 0.5) {};
		\node [contact, outer sep=5pt] (9) at (-2, -0) {};
		\node [contact, outer sep=5pt] (10) at (-2, -1) {};
		\node [contact, outer sep=5pt] (11) at (4, -0) {};
		\node [contact, outer sep=5pt] (12) at (4, -1) {};
		\node [contact, outer sep=5pt] (13) at (4, -0.5) {};
		\node [contact, outer sep=5pt] (14) at (4, 0.5) {};
		\node [style=none] (15) at (-0.5, 0.625) {};
		\node [style=none] (16) at (-0.5, -0.25) {};
		\node [style=none] (17) at (-0.5, -0.875) {};
		\node [style=none] (18) at (2.5, -0.875) {};
		\node [contact, outer sep=5pt] (19) at (4, 1) {};
		\node [style=none] (20) at (1, -1.25) {$Y$};
		\node [style=none] (21) at (2.5, 0.75) {};
		\node [style=none] (22) at (2.5, -0.25) {};
	\end{pgfonlayer}
	\begin{pgfonlayer}{edgelayer}
		\draw [color=gray] (6.center) to (15.center);
		\draw [color=gray] (8.center) to (15.center);
		\draw [color=gray] (0.center) to (15.center);
		\draw [color=gray] (1.center) to (15.center);
		\draw [color=gray] (9.center) to (16.center);
		\draw [color=gray] (7.center) to (16.center);
		\draw [color=gray] (10.center) to (17.center);
		\draw [color=gray] (17.center) to (3.center);
		\draw [color=gray] (3.center) to (18.center);
		\draw [color=gray] (18.center) to (12.center);
		\draw [color=gray] (0.center) to (21.center);
		\draw [color=gray] (1.center) to (21.center);
		\draw [color=gray] (21.center) to (19.center);
		\draw [color=gray] (2.center) to (22.center);
		\draw [color=gray] (22.center) to (11.center);
		\draw [color=gray] (22.center) to (13.center);
		\draw [rounded corners=5pt, dotted] 
   (node cs:name=6, anchor=north west) --
   (node cs:name=8, anchor=south west) --
   (node cs:name=1, anchor=south east) --
   (node cs:name=0, anchor=north east) --
   cycle;
		\draw [rounded corners=5pt, dotted] 
   (node cs:name=9, anchor=north west) --
   (node cs:name=7, anchor=south west) --
   (node cs:name=7, anchor=south east) --
   (node cs:name=9, anchor=north east) --
   cycle;
		\draw [rounded corners=5pt, dotted] 
   (node cs:name=10, anchor=north west) --
   (node cs:name=10, anchor=south west) --
   (node cs:name=3, anchor=south east) --
   (node cs:name=3, anchor=north east) --
   cycle;
		\draw [rounded corners=5pt, dotted] 
   (node cs:name=2, anchor=north west) --
   (node cs:name=2, anchor=south west) --
   (node cs:name=2, anchor=south east) --
   (node cs:name=2, anchor=north east) --
   cycle;
		\draw [rounded corners=5pt, dotted] 
   (node cs:name=0, anchor=north west) --
   (node cs:name=1, anchor=south west) --
   (node cs:name=19, anchor=south east) --
   (node cs:name=19, anchor=north east) --
   cycle;
		\draw [rounded corners=5pt, dotted] 
   (node cs:name=2, anchor=north west) --
   (node cs:name=2, anchor=south west) --
   (node cs:name=13, anchor=south east) --
   (node cs:name=11, anchor=north east) --
   cycle;
		\draw [rounded corners=5pt, dotted] 
   (node cs:name=3, anchor=north west) --
   (node cs:name=3, anchor=south west) --
   (node cs:name=12, anchor=south east) --
   (node cs:name=12, anchor=north east) --
   cycle;
		\draw [rounded corners=5pt, dotted] 
   (node cs:name=14, anchor=north west) --
   (node cs:name=14, anchor=south west) --
   (node cs:name=14, anchor=south east) --
   (node cs:name=14, anchor=north east) --
   cycle;
	\end{pgfonlayer}
\end{tikzpicture}
\end{aligned}
\:
  =
\:
\begin{aligned}
\begin{tikzpicture}[circuit ee IEC]
	\begin{pgfonlayer}{nodelayer}
		\node [style=none] (0) at (-2.75, -0) {$X$};
		\node [style=none] (1) at (1.75, -0) {$Z$};
		\node [contact, outer sep=5pt] (2) at (-2, 1) {};
		\node [contact, outer sep=5pt] (3) at (-2, -0.5) {};
		\node [contact, outer sep=5pt] (4) at (-2, 0.5) {};
		\node [contact, outer sep=5pt] (5) at (-2, -0) {};
		\node [contact, outer sep=5pt] (6) at (-2, -1) {};
		\node [contact, outer sep=5pt] (7) at (1, -0) {};
		\node [contact, outer sep=5pt] (8) at (1, -1) {};
		\node [contact, outer sep=5pt] (9) at (1, -0.5) {};
		\node [contact, outer sep=5pt] (10) at (1, 0.5) {};
		\node [style=none] (11) at (-0.5, 0.875) {};
		\node [style=none] (12) at (-1, -0.25) {};
		\node [contact, outer sep=5pt] (13) at (1, 1) {};
		\node [style=none] (14) at (0, -0.25) {};
	\end{pgfonlayer}
	\begin{pgfonlayer}{edgelayer}
		\draw [color=gray] (2.center) to (11.center);
		\draw [color=gray] (4.center) to (11.center);
		\draw [color=gray] (5.center) to (12.center);
		\draw [color=gray] (3.center) to (12.center);
		\draw [color=gray] (14.center) to (7.center);
		\draw [color=gray] (14.center) to (9.center);
		\draw [color=gray] (6.center) to (8.center);
		\draw [color=gray] (11.center) to (13.center);
		\draw [rounded corners=5pt, dotted] 
   (node cs:name=2, anchor=north west) --
   (node cs:name=4, anchor=south west) --
   (node cs:name=13, anchor=south east) --
   (node cs:name=13, anchor=north east) --
   cycle;
		\draw [rounded corners=5pt, dotted] 
   (node cs:name=5, anchor=north west) --
   (node cs:name=3, anchor=south west) --
   (node cs:name=3, anchor=south east) --
   (node cs:name=5, anchor=north east) --
   cycle;
		\draw [rounded corners=5pt, dotted] 
   (node cs:name=6, anchor=north west) --
   (node cs:name=6, anchor=south west) --
   (node cs:name=8, anchor=south east) --
   (node cs:name=8, anchor=north east) --
   cycle;
		\draw [rounded corners=5pt, dotted] 
   (node cs:name=10, anchor=north west) --
   (node cs:name=10, anchor=south west) --
   (node cs:name=10, anchor=south east) --
   (node cs:name=10, anchor=north east) --
   cycle;
		\draw [rounded corners=5pt, dotted] 
   (node cs:name=7, anchor=north west) --
   (node cs:name=9, anchor=south west) --
   (node cs:name=9, anchor=south east) --
   (node cs:name=7, anchor=north east) --
   cycle;
	\end{pgfonlayer}
\end{tikzpicture}
\end{aligned}
\]
by taking the transitive closure of the two equivalence relations, and then
restricting this to an equivalence relation on $X+Z$. 

In the category of sets we hold the fundamental relationship between sets to be
that of functions. These encode the idea of a deterministic process that takes
each element of one set to a unique element of the other. For the study of
networks this is less appropriate, as the relationship between terminals is not
an input-output one, but rather one of interconnection. Willems has repeatedly
emphasised the prevalence of input-output thinking as a limitation of current
techniques in control theory \cite{Wi,Wi2}.  

In particular, the direction of a function becomes irrelevant, and to describe
these interconnections via the category of sets we must develop an understanding
of how to compose functions head to head and tail to tail. We have so far used
cospans and pushouts to address this.  Cospans, however, come with an apex, which
represents extraneous structure beyond the two sets we wish to specify a
relationship between. Corelations, in line with our `black boxing' intution for
them, arise from omitting this information.

\subsection{Potentials on corelations} \label{ssec:potentialsoncorelations}

Chasing our interpretation of corelations as ideal wires, our aim now is to
build a functor
\[
  S\maps \mathrm{Corel} \longrightarrow \LagrRel
\]
that expresses this interpretation. We break this functor down into the sum 
of two parts, according to the behaviours of potentials and currents
respectively. 

The consideration of potentials gives a functor $\Phi\maps \mathrm{Corel}
\to \mathrm{LinRel}$, where $\mathrm{LinRel}$ is the symmetric
monoidal dagger category of finite-dimensional $\F$-vector spaces, linear
relations, direct sum, and transpose. In particular, this functor expresses
Kirchhoff's voltage law: it requires that if two elements are in the same part
of the corelation partition---that is, if two nodes are connected by ideal
wires---then the potential at those two points must be the same.

To construct this functor, we recall Subsection \ref{ssec.linrel}, where we
observed that $\LinRel$ is the category of jointly-epic corelations in
$\FinVect^\opp$, and then appeal to our results on functors between corelation
categories.

\begin{proposition}
  Define the functor 
  \[ 
    \Phi\maps \mathrm{Corel} \longrightarrow \mathrm{LinRel}, 
  \] 
  on objects by sending a finite set $X$ to the vector space $\F^X$, and on
  morphisms by sending a corelation $e\maps X \to Y$ to the linear subspace
  $\Phi(e)$ of $\F^X \oplus \F^Y$ comprising functions $\phi\maps X+Y \to \F$
  that are constant on each element of the equivalence class induced by
  $e$.  This is a strong symmetric monoidal functor, with coherence maps
  the usual natural isomorphisms $\F^X \oplus \F^Y \cong \F^{X + Y}$ and
  $\{0\} \cong \F^\varnothing$. 
\end{proposition}
\begin{proof}
This functor $\Phi$ extends the contravariant functor $\FinSet \to \FinVect$
that maps a set to the vector space of $\F$-valued functions on that set.
Indeed, define the contravariant functor $\Psi\maps \FinSet \to \FinVect$
mapping a finite set $X$ to the vector space $\F^X$, and the function $f\maps X
\to Y$ to the linear map $f_\ast\maps \F^Y \to \F^X$ sending $\phi\maps Y \to \F$
to $\phi \circ f\maps X \to \F$.

The functor $\Psi$ maps finite colimits to finite limits. To show this, we need only
check it maps coproducts to products and coequalisers to equalisers. The former
follows from the same well-known natural isomorphisms as the coherence maps:
$\F^X \oplus \F^Y \cong \F^{X\times Y}$. To see the latter, write $C = Y/\big(f(x)
\sim g(x)\big)$ for the coequaliser of the diagram 
\[
\xymatrix{
  X \ar@<.75ex>[r]^{f}
  \ar@<-.75ex>[r]_{g}
&
Y.}
\]
This is the quotient of the set $Y$ by the equivalence relation induced by $f(x)
\sim g(x)$ for all $x \in X$; write $c\maps Y \to C$ for the quotient map.  Now
the equaliser of the image of this diagram under $\Psi$,
\[
\xymatrix{
  F^Y \ar@<.75ex>[r]^{f_\ast}
  \ar@<-.75ex>[r]_{g_\ast}
&
F^X.}
\]
is the kernel of $f_\ast-g_\ast$. This kernel comprises those $v \in \F^Y$ such
that $v(f(x)) = v(g(x))$ for all $x \in X$. But this is precisely the image of
$c_\ast\maps \F^C \to \F^Y$. Thus $\Psi$ sends coequalisers to equalisers, and
hence finite colimits to finite limits.

Note also that the image of an injective function in $\FinSet$ is a
surjective linear transformation $\F^Y \to \F^X$, and hence a monomorphism in
$\FinVect^\opp$. Thus, considering $\Psi$ as a covariant functor $\FinSet \to
\FinVect^\opp$, we have a functor that preserves finite colimits and such that
the image of the monomorphisms of $\FinSet$ is contained in the monomorphisms of
$\FinVect^\opp$.

Recall that $\FinVect$ has an epi-mono factorisation system, and that $\LinRel$
is the category of relations with respect to it; equivalently, this means that
$\FinVect^\opp$ has an epi-mono factorisation system, and $\LinRel$ is the its
category of corelations (Subsection \ref{ssec.linrel}). Thus $\Psi$ extends to a
functor 
\[
  \Phi\maps \corel(\FinSet) = \mathrm{Corel} \longrightarrow
  \corel(\FinVect^\opp) \cong \LinRel.
\]
This functor $\Phi$ maps a corelation $e= [i,o]\maps X+Y \to N$ to the linear
relation $e_\ast(\F^N) \subset \F^X \oplus \F^Y$ comprising functions $X+Y \to
\F$ that are constant on each equivalence class of elements of $X+Y$.
\end{proof}

To recap, we have now constructed a functor $\Phi\maps \mathrm{Corel} \to
\mathrm{LinRel}$ expressing the behaviour of potentials on corelations interpreted
as ideal wires. Next, we do the same for currents.

\subsection{Currents on corelations}

In this subsection we consider the case of currents, described by a functor
$I\maps \mathrm{Corel} \to \mathrm{LinRel}$. This functor expresses Kirchhoff's
current law: it requires that the sum of currents flowing into each part of the
corelation partition must equal to the sum of currents flowing out.  
\begin{proposition}
  Define the functor
  \[
    I\maps \mathrm{Corel} \longrightarrow \mathrm{LinRel}.
  \]
  as follows. On objects send a finite set $X$ to the vector space
  $(\F^{X})^\ast=\F[X]$, the free vector space on $X$. On morphisms send a
  corelation $e=[i,o]\maps X +Y \to N$ to the linear relation $I(e)$ comprising precisely
  those 
  \[
    (i_X,i_Y) = \left(\sum_{x \in X} \lambda_xdx,\sum_{y \in Y}
    \lambda_ydy\right)  \in (\F^{X})^\ast \oplus (\F^{Y})^\ast
  \]
  such that for all $n \in N$ the sum of the coefficients of the elements
  of $i^{-1}(n) \subseteq X$ is equal to that for $o^{-1}(n) \subseteq Y$:
  \[
    \sum_{i(x)=n} \lambda_x = \sum_{o(y)=n} \lambda_y.
  \]
  This is a strong symmetric monoidal functor, with coherence maps the natural
  isomorphisms $(\F^X)^\ast \oplus (\F^Y)^\ast \to (\F^{X+Y})^\ast$ and $\{0\}
  \to (\F^\varnothing)^\ast$.
\end{proposition}
\begin{proof}
  The functor $I$ is a generalization of the covariant functor $\FinSet \to
  \FinVect$ that maps a set to the vector space of $\F$-linear combinations of
  elements of that set. We proceed in a similar fashion to the functor for
  potentials. Define the functor $J\maps \FinSet \to \FinVect$ mapping a finite
  set $X$ to the vector space $(\F^{X})^\ast$, the sum of formal linear
  combinations of the symbols $dx$, where $x$ is an element of $X$. Given a
  function $f\maps X \to Y$, $J$ maps $f$ to the linear map $f^\ast\maps
  (\F^X)^\ast \to (\F^Y)^\ast$ taking $\sum_{x \in X} \lambda_x dx \in
  (\F^X)^\ast$ to $\sum_{x \in X} \lambda_xd(fx)$.  

  This is the restriction to $\FinSet$ of the free vector space functor $\Set
  \to \Vect$. As the free vector space functor is left adjoint to the forgetful
  functor, it preserves finite colimits. Hence $\FinSet$ also preserves finite
  colimits. Moreover, $J$ maps monomorphisms in $\FinSet$ to monomorphisms in
  $\FinVect$. By Proposition \ref{prop.corelfunctors}, $J$ thus extends to a functor
  \[
    I\maps \corel(\FinSet) = \mathrm{Corel} \longrightarrow \corel(\FinVect) =
    \mathrm{LinRel}.
  \]
  This functor maps a corelation $e=[i,o]\maps X+Y \to N$ to the kernel
  $\ker(Ie)$ of the linear map $Ie\maps (\F^X)^\ast \oplus (\F^Y)^\ast \to
  (\F^N)^\ast$. This comprises all linear combinations $\left(\sum_{x \in X}
  \lambda_xdx,\sum_{y \in Y} \lambda_ydy\right)$ such that $\sum_{x \in X}
  \lambda_x d(ix) =  \sum_{y \in Y} \lambda_y d(oy)$. This is what we wanted to
  show.
\end{proof}

\subsection{The symplectification functor}

We have now defined functors that, when interpreting corelations as connections
of ideal wires, describe the behaviours of the currents and potentials at the
terminals of these wires. In this section, we combine these to define a single
functor $S\maps \mathrm{Corel} \to \LagrRel$ describing the behaviour of both
currents and potentials as a Lagrangian subspace.

\begin{proposition} \label{prop:sympfunctor}
  We define the symplectification functor
  \[
    S\maps \mathrm{Corel} \longrightarrow \LagrRel
  \]
  sending a finite set $X$ to the symplectic vector space
  \[
    SX = \F^X \oplus (\F^X)^\ast,
  \]
  and a corelation $e=[i,o]\maps X+Y \to N$ to the Lagrangian relation
  \[
    Se = \Phi e \oplus Ie \subseteq \overline{\F^X \oplus
    (\F^X)^\ast}\oplus \F^Y \oplus (\F^Y)^\ast.
  \]
  Then $S$ is a strong symmetric monoidal functor, with coherence maps those of
  $\Phi$ and $I$.
\end{proposition}
Note that $\Phi e\oplus Ie$ is more properly a subspace of the vector space $\F^X
\oplus \F^Y \oplus (\F^X)^\ast \oplus (\F^Y)^\ast$. In the above we consider it
instead as a subspace of the symplectic vector space $\overline{\F^X \oplus
(\F^X)^\ast}\oplus \F^Y \oplus (\F^Y)^\ast$, whose underlying vector space is
canonically isomorphic to $\F^X \oplus \F^Y \oplus (\F^X)^\ast \oplus
(\F^Y)^\ast$.

\begin{proof}
  As $S$ is the monoidal product in $\mathrm{LinRel}$ of the strong symmetric
  monoidal functors $\Phi, I\maps \mathrm{Corel} \to \mathrm{LinRel}$, it is
  itself a strong symmetric monoidal functor $\mathrm{Corel} \to
  \mathrm{LinRel}$. Thus it only remains to be checked that, with respect to the
  symplectic structure we put on the objects $SX$, the image of each
  corelation $e$ is Lagrangian.

  This follows from condition (v) of Proposition
  \ref{lagrangian_characterization}: that a Lagrangian subspace is an isotropic
  subspace of dimension half that of the symplectic vector space.

  Given $(\phi_X,i_X,\phi_Y,i_Y) \in Se$, note that for each $n \in N$ there
  exists a constant $\phi_n \in \F$ such that $\phi_n=\phi_X(x) = \phi_Y(y)$ for
  all $x \in X$ such that $i(x) = n$ and all $y \in Y$ such that $o(y)= n$. 
  Then $Se$ is isotropic as, for all 
  $(\phi_X,i_X,\phi_Y,i_Y)$, $(\phi_X',i_X',\phi_Y',i_Y') \in Se$ we have
  \begin{align*}
    &\phantom{=} \enskip
    \omega\big((\phi_X,i_X,\phi_Y,i_Y),(\phi_X',i_X',\phi_Y',i_Y')\big)
    \\
    &= \overline{\omega_X}\big((\phi_X,i_X),(\phi_X',i_X')\big) +
    \omega_Y\big((\phi_Y,i_Y),(\phi_Y',i_Y')\big) \\
    &= -\big(i_X'(\phi_X)-i_X(\phi_X')\big) + i_Y'(\phi_Y)-i_Y(\phi_Y') \\
    &= i_X(\phi_X')-i_Y(\phi_Y') + i_X'(\phi_X)-i_Y'(\phi_Y) \\
    &= \sum_{x \in X} \lambda_x dx(\phi_X') - \sum_{y \in Y} \lambda_y dy(\phi_Y') +
    \sum_{y \in Y} \lambda_y' dy(\phi_Y) - \sum_{x \in X} \lambda_x' dx(\phi_X)\\
    &= \sum_{n \in N}\left(\sum_{i(x)=n} \lambda_x dx(\phi_X') -
    \sum_{o(y)=n} \lambda_y dy(\phi_Y')\right) \\
    &\qquad \qquad+ \sum_{n \in N}\left(\sum_{i(x)=n} \lambda_x'dx(\phi_X) -
    \sum_{o(y)=n} \lambda_y' dy(\phi_Y)\right)\\
    &= \sum_{n \in N}\left(\sum_{i(x)=n} \lambda_x - \sum_{o(y)=n} \lambda_y
    \right)\phi'_n + \sum_{n \in N}\left(\sum_{i(x)=n} \lambda_x'- \sum_{o(y)=n}
    \lambda_y'\right)\phi_n\\
    &= 0.
  \end{align*}
  Second, $Se$ has dimension equal to 
  \begin{multline*}
    \dim(\Phi e)+ \dim(Ie) = \#N+\#(X+Y) - \#N \\
    = \#(X+Y) = \tfrac12
    \dim\big(\overline{\F^X \oplus (\F^X)^\ast} \oplus \F^Y \oplus
    (\F^Y)^\ast\big).
  \end{multline*}
  This proves the proposition.
\end{proof}

We have thus shown that we do indeed have a functor $S\maps \mathrm{Corel} \to
\LagrRel$. In the next section we shall see that this functor provides
the engine of our black box functor. To get there, an understanding of the
action of $S$ on functions will be helpful.

\begin{example}[Symplectification of functions] \label{ex:sympfunction}
  Let $f: X \to Y$ be a function; we may also consider it as the corelation $X
  \stackrel{f}\to Y \stackrel{\idn_Y}\leftarrow Y$. In this example we show that
  $Sf$ has the form 
  \[
    Sf = \big\{(\phi_X,i_X,\phi_Y,i_Y) \, \big\vert \, \phi_X = f^\ast\phi_Y,
    i_Y = f_\ast i_X \big\} \subseteq \overline{\F^X \oplus (\F^X)^\ast} \oplus
    \F^Y \oplus (\F^Y)^\ast,
  \]
  where $f^\ast$ is the pullback map
  \begin{align*}
    f^\ast\maps \F^Y &\longrightarrow \F^X; \\
    \phi &\longmapsto \phi \circ f,
  \end{align*}
  and $f_\ast$ is the pushforward map
  \begin{align*}
    f_\ast\maps (\F^X)^\ast &\longrightarrow (\F^Y)^\ast; \\
    i(-) &\longmapsto i(-\circ f).
  \end{align*}
  The claim is then that these pullback and pushforward constructions express
  Kirchhoff's laws.

  Recall that the corelation corresponding to $f$ partitions $X+Y$ into $\#Y$
  parts, each of the form $f^{-1}(y) \cup \{y\}$. The linear relation $\Phi(f)$
  requires that if $x \in X$ and $y \in Y$ lie in the same part of the
  partition, then they have the same potential: that is, $\phi_X(x) =
  \phi_Y(y)$. This is precisely the arrangement imposed by $f^\ast \phi_Y =
  \phi_X$: 
  \[
    \phi_X(x) = \phi_Y(f(x)) =\phi_Y(y).
  \] 
  On the other hand, the linear relation $I(f)$ requires that if $i_X = \sum_{x 
  \in X}\lambda_xdx$, $i_Y = \sum_{y \in Y}\lambda_y dy$, then for each $y \in Y$
  we have 
  \[
    \sum_{x \in f^{-1}(y)} \lambda_x = \lambda_y.
  \]
  This is precisely what is required by $f_\ast$: given any $\phi \in \F^Y$, we
  have
  \[
    f_\ast i_X(\phi) = f_\ast \sum_{x \in X}\lambda_xdx(\phi) = \sum_{x
    \in X}\lambda_xdx(\phi \circ f)= \sum_{y \in Y}\left( \sum_{x \in f^{-1}(y)}
    \lambda_x\right)dy.
  \]
  This gives us the above representation of $Sf$ when $f$ is a function.
\end{example}

\subsection{Lagrangian relations as decorated corelations}
\label{ssec.lagrrelascorel}
Recall that symplectic vector spaces of the form $\vectf{X}$ are known as
standard symplectic spaces. Here we shall use decorated corelations to
construct a hypergraph category $\mathrm{LagrCorel}$ in which the objects are standard
symplectic spaces and the morphisms are in one-to-one correspondence with
Lagrangian relations between them.  As every finite dimensional symplectic
vector spaces is symplectomorphic to a standrd symplectic space,
$\mathrm{LagrCorel}$ is equivalent as a symmetric monoidal category to
$\LagrRel$. This equivalence allows us to endow $\LagrRel$ with a hypergraph
structure so that the two categories are hypergraph equivalent. 

Analogous with our constructions in \textsection\ref{sec.setdecs} and
\textsection\ref{sec.allhypergraphs}, we will use the composite of the hom
functor $\LagrRel(\varnothing,-)\maps \LagrRel \to \Set$, the
symplectification functor $S\maps\corel \to \LagrRel$, and the functor
$\cospan(\FinSet) \to \corel$.

\begin{proposition}
Define 
\[
  \mathrm{Lagr}\maps (\cospan(\FinSet),+) \longrightarrow (\mathrm{Set},\times)
\]
as follows. For objects let $\mathrm{Lagr}$ map a finite set $X$ to the set
$\mathrm{Lagr}(X)$ of Lagrangian subspaces of the symplectic vector space
$\vectf{X}$.  For morphisms, given a cospan of finite sets $X \to Y$ we
take its jointly-epic part $e$, and thus obtain a map $Se\maps \vectf{X} \to
\vectf{Y}$. As Lagrangian relations map Lagrangian subspaces to Lagrangian
subspaces (Proposition \ref{prop:lagrangian_composition}), this gives a map: 
\begin{align*}
  \mathrm{Lagr}(f)\maps \mathrm{Lagr}(X) &\longrightarrow \mathrm{Lagr}(Y); \\
  L &\longmapsto Se(L).
\end{align*}
Moreover, equipping this functor with the family of maps
\begin{align*}
  \lambda_{N,M}\maps \mathrm{Lagr}(N) \times \mathrm{Lagr}(M) &\longrightarrow
  \mathrm{Lagr}(N+M);\\
  (L_N,L_M) &\longmapsto L_N \oplus L_M,
\end{align*}
and unit
\begin{align*}
  \lambda_1\maps 1 &\longrightarrow \mathrm{Lagr}(\varnothing);\\
  \bullet &\longmapsto \{0\}
\end{align*}
defines a lax symmetric monoidal functor.
\end{proposition}
\begin{proof}
  The functoriality of this construction follows from the functoriality of $S$;
  the lax symmetric monoidality from the relevant properties of the direct sum
  of vector spaces.
\end{proof}

This functor gives us two categories. First, we can construct a decorated
cospans category $\mathrm{LagrCospan}$ on the functor $\mathrm{Lagr}$ restricted
to $\FinSet$, giving a category with each morphism a cospan of finite sets $X
\to N \leftarrow Y$ decorated by a Lagrangian subspace of $\vectf{N}$. This will
be useful for interpreting extended power functionals as Lagrangian subspaces in
a compositional setting. 

The second category is $\mathrm{LagrCorel}$. This as usual has finite sets as
objects. For morphisms $X \to Y$ it has Lagrangian subspaces of 
\[
  \vectf{X+Y} \cong \vectf{X} \oplus \vectf{Y}
\]
\emph{not}, as might be expected, Lagrangian subspaces of $\overline{\vectf{X}}
\oplus \vectf{Y}$. The reason is that composition is also a slightly tweaked
from relational composition. 

Composition in $\mathrm{LagrCorel}$ is enacted by the symplectification of the
cospan 
\[
  X+Y+Y+Z \xrightarrow{\idn_X+[\idn_Y,\idn_Y]+\idn_Z} X+Y+Z
  \xleftarrow{\iota_X+\iota_Z} X+Z.
\]
This relation relates 
\[
  (\phi_X,\phi_Y,\phi_Y',\phi_Z',i_X,i_Y,i_Y',i_Z') \in 
  \F^{X+Y+Y+Z}\oplus (\F^{X+Y+Y+Z})^\ast
\]
with 
\[
  (\phi_X,\phi_Z',i_X,i_Z') \in \F^{X+Z}\oplus (\F^{X+Z})^\ast
\]
if and only if two conditions hold: (i) $\phi_Y = \phi_Y'$ and (ii) $i_Y +i_Y' =
0$. Thus on the `potentials' $\phi$ we have relational composition, but on the
`currents' we instead require $i_Y=-i_Y'$. We intepret this to mean that the
decorations of $\mathrm{LagrCorel}$ record only the currents \emph{out} of the
nodes of the circuit. Thus if we were to interconnect two nodes, the current out
of the first must be equal to the negative of the current out---that is, the
current \emph{in}---of the second node. It is remarkable that this physical fact
is embedded so deeply in the underlying mathematics.

As these categories are constructed using decorated corelations, we have for
free a hypergraph functor 
\[
  \mathrm{LagrCospan} \longrightarrow \mathrm{LagrCorel}.
\]
The above discussion shows that we also have an embedding 
\[
  \mathrm{LagrCorel} \hooklongrightarrow \mathrm{LagrRel}
\]
mapping a finite set $X$ to the standard symplectic space $\vectf{X}$, and a
morphism $L \subset \vectf{X} \oplus \vectf{Y}$ to its image under the
symplectomorphism 
\[
  \overline{\idn_X} \oplus \idn_Y \maps\vectf{X} \oplus \vectf{Y}
  \stackrel{\cong}\longrightarrow \overline{\vectf{X}} \oplus \vectf{Y}.
\]
This is a fully faithful and essentially surjective strong symmetric monoidal
functor. Recall that $\mathrm{LagrCorel}$ is a hypergraph category; indeed, the
hypergraph structure is given by `ideal wires', as we have spent this chapter
exploring. Choosing an isomorphism of each symplectic vector space with a
standard symplectic vector space, we may equip each object of $\mathrm{LagrRel}$
with special commutative Frobenius structure, and hence consider the above
embedding an equivalence of hypergraph categories.

\section{The black box functor} \label{sec:blackbox}
%%fakesubsection
We have now developed enough machinery to prove Theorem \ref{main_theorem}:
there is a hypergraph functor, the black box functor
\[  
\blacksquare\maps \Circ \to \LagrRel 
\]
taking passive linear circuits to their behaviours. To recap, we have so far
developed two categories: $\Circ$, in which morphisms are passive linear
circuits, and $\LagrRel$, which captures the external behaviour of such circuits.
We now define a functor that maps each circuit to its behaviour, before proving
it is indeed a hypergraph functor. 

The role of the functor we construct here is to identify all circuits with the
same external behaviour, making the internal structure of the circuit
inaccessible. Circuits treated this way are frequently referred to as
`black boxes', so we call this functor the \define{black box functor},
\[
\blacksquare\maps \Circ \to \LagrRel.
\] 
In this section we first provide the definition of this functor, check
that our definition really does map a circuit to its behaviour, and then finally
use decorated corelations to verify its functoriality.

\subsection{Definition}
It should be no surprise that the black box functor maps a finite set $X$ to the
symplectic vector space $\F^X \oplus (\F^X)^\ast$ of potentials and currents
that may be placed on that set. The challenge is to provide a succinct
statement of its action on the circuits themselves. To do this we take advantage
of three processes defined above.

Let $\Gamma\maps X \to Y$ be a circuit, that is a cospan of finite sets $X
\stackrel{i}{\longrightarrow} N \stackrel{o}{\longleftarrow} Y$ decorated by an
$\F^+$-graph $(N,E,s,t,r)$. To define the image of $\Gamma$ under our functor
$\blacksquare$, by definition a Lagrangian relation $\blacksquare(\Gamma):
\blacksquare(X) \to \blacksquare(Y)$, we must specify a Lagrangian subspace
$\blacksquare(\Gamma) \subseteq \overline{\F^X \oplus (\F^X)^\ast} \oplus \F^Y
\oplus (\F^Y)^\ast$.  

Recall that to each $\F^+$-graph $\Gamma$ we can associate a Dirichlet form, 
the extended power functional 
\begin{align*}
  P_\Gamma\maps \F^N &\longrightarrow \F; \\
  \phi &\longmapsto \frac{1}{2} \sum_{e \in E} \frac{1}{r(e)} \big( \phi(t(e)) -
  \phi(s(e))  \big)^2,
\end{align*}
and to this Dirichlet form we associate a Lagrangian subspace, thought of as a
Lagrangian relation:
\[
  \mathrm{Graph}(dP_\Gamma) = \{(\phi,d(P_\Gamma)_\phi) \mid \phi \in
  \F^N\}\maps \{0\} \to \F^N \oplus (\F^N)^\ast.
\]
Next, from the legs of the cospan $\Gamma$, the symplectification functor $S$
gives the Lagrangian relation
\[
  S[i,o]^\opp\maps \F^N \oplus (\F^N)^\ast \longrightarrow \F^X \oplus
  (\F^X)^\ast \oplus \F^Y \oplus (\F^Y)^\ast.
\]
Lastly, we have the symplectomorphism
\begin{align*}
  \overline{\idn_X}\maps \F^X \oplus (\F^X)^\ast &\longrightarrow \overline{\F^X \oplus
  (\F^X)^\ast}; \\
  (\phi,i) &\longmapsto (\phi,-i).
\end{align*}
Putting these all together:
\begin{definition}
  We define the black box functor 
  \[
    \blacksquare\maps \Circ \to \LagrRel 
  \]
  on objects by mapping a finite set $X$ to the symplectic vector space 
  \[
    \blacksquare(X) = \F^X \oplus (\F^X)^\ast.
  \]
  and on morphisms by mapping a circuit $\big(X \stackrel{i}\to N
  \stackrel{o}\leftarrow Y, \enspace (N,E,s,t,r)\big)$, to the Lagrangian
  relation
  \[
    \blacksquare(\Gamma) = (\overline{\idn_X}\oplus \idn_Y) \circ S[i,o]^\opp
    \circ \mathrm{Graph}(dP_\Gamma).
  \]
  The coherence maps are given by the natural isomorphisms $\F^X \oplus
  (\F^X)^\ast \oplus \F^Y \oplus (\F^Y)^\ast \cong \F^{X+Y} \oplus
  (\F^{X+Y})^\ast$ and $\{0\} \cong \F^\varnothing \oplus (\F^\varnothing)^\ast$
\end{definition}

As isomorphisms of cospans of $\F^+$-graphs amount to no more than a 
relabelling of nodes and edges, this construction is independent of the cospan 
chosen as representative of the isomorphism class of cospans forming the 
circuit.

\begin{theorem} \label{thm:main}
  The black box functor is a well-defined hypergraph functor.
\end{theorem}

We will prove this shortly. Before we get there, we first assure ourselves that
we have indeed arrived at the theorem we set out to prove.

\subsection{Minimization via composition of relations}

At this point the reader might voice two concerns: firstly, why does the
\emph{black box} functor refer to the \emph{extended} power functional $P$ and,
secondly, since it fails to talk about power minimization, how is it the same
functor as that defined in Theorem~\ref{main_theorem}? These fears are allayed
by the remarkable trinity of minimization, the symplectification of functions,
and Kirchhoff's laws. 

We have seen that symplectification of functions views the cograph of the
function as a picture of ideal wires, governed by Kirchhoff's laws (Example
\ref{ex:sympfunction}). We have also seen that Kirchhoff's laws are closely
related to the principle of minimium power (Theorems
\ref{thm:realizablepotentials} and \ref{thm:dirichletminimization}). The final
aspect of this relationship is that we may use symplectification of functions to
enact minimization.

\begin{theorem} \label{thm:sympmin}
  Let $\iota: \partial N \to N$ be an injection, and let $P$ be a Dirichlet form on
  $N$. Write $Q = \min_{N \setminus \partial N} P$ for the Dirichlet form on
  $\partial N$ given by minimization over $N \setminus \partial N$. Then we have an
  equality of Lagrangian subspaces
  \[
    S\iota^\opp \big( \mathrm{Graph}(dP)\big) = \mathrm{Graph}(dQ).
  \]
\end{theorem}
\begin{proof}
  Recall from Example \ref{ex:sympfunction} that $S\iota^\opp$ is the Lagrangian relation
  \[
    S\iota^\opp = \big\{(\phi, \iota_\ast i,\phi \circ \iota, i) \, \big\vert
      \, \phi \in \F^{N}, i \in (\F^{\partial N})^\ast \big\} \subseteq
      \overline{\F^N \oplus (\F^N)^\ast} \oplus \F^{\partial N} \oplus
      (\F^{\partial N})^\ast,
  \]
  where $\iota_\ast i(\phi) = i(\phi \circ \iota)$, and note that 
  \[
    \mathrm{Graph}(dP) = \big\{(\phi,dP_\phi) \,\big\vert\, \phi \in \F^N\big\}.
  \]
  This implies that their composite is given by the set
  \[
    S\iota^\opp \circ \mathrm{Graph}(dP) = \big\{(\phi \circ \iota, i)
    \,\big\vert\, \phi \in \F^N, i \in (\F^{\partial N})^\ast, dP_\phi =
  \iota_\ast i \big\}.
  \]
  We must show this Lagrangian subspace is equal to $\mathrm{Graph}(dQ)$.
  
  Consider the constraint $dP_\phi = \iota_\ast i$. This states that for all
  $\varphi \in \F^N$ we have $dP_\phi(\varphi) = i(\varphi\circ \iota)$. Letting
  $\chi_n: N \to \F$ be the function sending $n \in N$ to $1$ and all other
  elements of $N$ to $0$, we see that when $n \in N \setminus \partial N$ we
  must have
  \[
    \frac{dP}{d\varphi(n)}\Bigg\vert_{\varphi = \phi}  = dP_\phi(\chi_n) = 
    i(\chi_n \circ \iota) = i(0) = 0.
  \]
  So $\phi$ must be a realizable extension of $\psi = \phi \circ \iota$. We
  henceforth write $\tilde\psi = \phi$. As $\iota$ is injective, $\psi = \phi
  \circ \iota$ gives no constraint on $\psi \in \F^{\partial N}$. 
 
  We next observe that we can write $S\iota^\opp \circ \mathrm{Graph}(dP) =
  \mathrm{Graph}(dO)$ for some quadratic form $O$. Recall that Proposition
  \ref{prop:qfls} states that a Lagrangian subspace $L$ of $\F^{\partial N}
  \oplus (\F^{\partial N})^\ast$ is of the form $\mathrm{Graph}(dO)$ if and only
  if $L$ has trivial intersection with $\{0\} \oplus (\F^N)^\ast$. But indeed,
  if $\psi = 0$ then $0$ is a realizable extension of $\psi$, so $\iota_\ast i =
  dP_0 = 0$, and hence $i = 0$. 
  
  It remains to check that $O = Q$. This is a simple computation:
  \[
    O(\psi) = dO_\psi(\psi) = dO_\psi(\tilde\psi \circ \iota) = \iota_\ast
    dQ_\psi(\tilde\psi) = dP_{\tilde\psi}(\tilde\psi) = P(\tilde\psi) = Q(\psi),
  \]
  where $\tilde\psi$ is any realizable extension of $\psi \in \F^{\partial N}$.
\end{proof}

Write $\iota: \partial N \to N$ for the inclusion of the terminals into the set
of nodes of the circuit, and $i\rvert^{\partial N}: X \to \partial N$,
$o\rvert^{\partial N}: Y \to \partial N$ for the respective corestrictions of
the input and output map to $\partial N$. Note that $[i,o] = \iota \circ
[i\rvert^{\partial N}, o\rvert^{\partial N}]$. Also, in the introduction, we
introduced the `twisted symplectification' $S^t$, which we can now understand as
an application of the functor $S$ followed by the isomorphism
$\overline{\idn_X}$ of a standard symplectic space with its conjugate.  Then we
have the equalities of sets, and thus Lagrangian relations:
\begin{align*}
  &\phantom{= .}(\overline{\idn_X}\oplus \idn_Y) \circ S[i,o]^\opp \circ
  \mathrm{Graph}(dP_\Gamma) \\
  &= (\overline{\idn_X}\oplus \idn_Y) \circ S[i\rvert^{\partial
  N},o\rvert^{\partial N}]^\opp \circ S\iota^\opp \circ \mathrm{Graph}(dP_\Gamma) \\
  &= (\overline{\idn_X}\oplus \idn_Y) \circ S[i\rvert^{\partial
  N},o\rvert^{\partial N}]^\opp \circ \mathrm{Graph}(dQ_\Gamma) \\
  &= \bigcup_{v \in \mathrm{Graph}(dQ)} S^ti\rvert^{\partial N}(v) \times
  So\rvert^{\partial N}(v)
\end{align*}
We see now that Theorem \ref{thm:main} is a restatement of Theorem
\ref{main_theorem} in the introduction.

\subsection{Proof of functoriality} \label{sec:proof}
%%fakesubsection
To prove that the black box functor is a hypergraph functor, we merely assemble
it from other functors, almost all of which we have already discussed. Indeed,
our factorisation will run
\[
  \blacksquare\maps \Circ \to \mathrm{DirichCospan} \to \mathrm{LagrCospan}
  \to \mathrm{LagrCorel} \to \LagrRel.
\]
The first three of these functors are decorated corelations functors, assembled
from the monoidal natural transformations
\[
  \xymatrixcolsep{4pc}
  \xymatrixrowsep{3pc}
  \xymatrix{
    (\FinSet,+) \ar^{\mathrm{Circuit}}[r] \ar@{=}[d] \drtwocell
    \omit{_\:\theta} & (\Set,\times) \ar@{=}[d]  \\
    (\FinSet,+) \ar^{\mathrm{Dirich}}[r] \ar@{=}[d] \drtwocell
    \omit{_\:\beta} & (\Set,\times) \ar@{=}[d]  \\
    (\FinSet,+) \ar^{\mathrm{Lagr}}[r] \ar@{^{(}->}[d] \drtwocell
    \omit{_\:\pi} & (\Set,\times) \ar@{=}[d]  \\
    (\cospan(\FinSet),+) \ar_{\mathrm{Lagr}}[r] & (\Set,\times).
  }
\]
The last factor, the functor from $\mathrm{LagrCorel}$ to $\mathrm{LagrRel}$, is
the hypergraph equivalence discussed at the end of Subsection
\ref{ssec.lagrrelascorel}.

We have also already discussed the strict hypergraph functors $\Circ \to
\mathrm{DirichCospan}$ given by $\theta$ (in
\textsection\ref{ssec.dirichcospans}) and $\mathrm{LagrCospan} \to
\mathrm{LagrCorel}$ given by $\pi$ (in \textsection\ref{ssec.lagrrelascorel}).
It remains to address $\beta$.

\begin{proposition}
Let
\[
  \beta\maps (\mathrm{Dirich},\delta) \Longrightarrow (\mathrm{Lagr},\lambda)
\]
be the collection of functions
\begin{align*}
  \beta_N\maps \mathrm{Dirich}(N) &\longrightarrow \mathrm{Lagr}(N); \\
  Q &\longmapsto \{(\phi,dQ_\phi) \mid \phi \in \F^N\} \subseteq \vectf{N}.
\end{align*}
Then $\beta$ is a monoidal natural transformation.
\end{proposition}

\begin{proof}
Naturality requires that the square
\[
\xymatrix{
  \mathrm{Dirich}(N) \ar[r]^{\beta_N} \ar[d]_{\mathrm{Dirich}(f)} &
  \mathrm{Lagr}(N) \ar[d]^{\mathrm{Lagr}(f)}  \\
  \mathrm{Dirich}(M) \ar[r]_{\beta_M} & \mathrm{Lagr}(M)
}
\]
commutes for every function $f\maps N \to M$. This is primarily a consequence of the
fact that the differential commutes with pullbacks. As we did in Example
\ref{ex:sympfunction},
write $f^\ast$ for the pullback map and $f_\ast$ for the pushforward map.
Then $\mathrm{Dirich}(f)$ maps a Dirichlet form $Q$ on $N$ to the form $f_\ast Q$,
and $\beta_M$ in turn maps this to the Lagrangian subspace 
\[
  \big\{(\psi,d(f_\ast Q)_\psi) \, \big\vert \, \psi \in \F^M\big\} \subseteq
  \F^M \oplus (\F^M)^\ast.
\]
On the other hand, $\beta_N$ maps a Dirichlet form $Q$ on $N$ to the Lagrangian
subspace
\[
\big\{(\phi,dQ_\phi) \,\big\vert\, \phi \in \F^N\big\}\subseteq
  \F^N \oplus (\F^N)^\ast, 
\]
before $\mathrm{Lagr}(f)$ maps this to the Lagrangian subspace
\[
  \big\{(\psi, f_\ast dQ_\phi) \,\big\vert\, \psi \in \F^M, \phi =
  f^\ast(\psi)\big\} \subseteq \F^M\oplus (\F^M)^\ast.
\]
But 
\[
  f_\ast dQ_{f^\ast\psi} = d(f_\ast Q)_{\psi},
\]
so these two processes commute.

Monoidality requires that the diagrams 
\[
\xymatrix{
  \mathrm{Dirich}(N) \times \mathrm{Dirich}(M) \ar[r]^(.52){\beta_N \times
  \beta_M} \ar[d]_{\delta_{N,M}} & \mathrm{Lagr}(N) \times \mathrm{Lagr}(M)
  \ar[d]^{\lambda_{N,M}}  \\
  \mathrm{Dirich}(N+M) \ar[r]_{\beta_{N+M}} & \mathrm{Lagr}(N+M)
}
\]
and
\[
  \xymatrix{
  & 1 \ar[dl]_{\delta_\varnothing} \ar[dr]^{\lambda_\varnothing}\\
\mathrm{Dirich}(\varnothing)  \ar[rr]_{\beta_\varnothing} &&
\mathrm{Lagr}(\varnothing)
}
\]
commute. These do: the Lagrangian subspace corresponding to the sum of Dirichlet
forms is equal to the sum of the Lagrangian subspaces that correspond to the
summand Dirichlet forms, while there is only a unique map $1 \to
\mathrm{Lagr}(\varnothing)$.
\end{proof}

We thus obtain a strict hypergraph functor 
\[   
\mathrm{DirichCospan} \to \mathrm{LagrCospan},
\]
which simply replaces the decoration on each cospan in $\mathrm{DirichCospan}$
with the corresponding Lagrangian subspace.

Observe now that the black box functor does indeed factor as described:
\[
    \blacksquare\maps \Circ \to \mathrm{DirichCospan} \to \mathrm{LagrCospan}
    \to \mathrm{LagrCorel} \to \LagrRel.
\]
The first factor takes a circuit $\Gamma$ to its extended power functional
$P_\Gamma$ and the second to its graph $\mathrm{Graph}(P_\Gamma)$. The next two
functors turn this from a cospan of finite sets decorated by a Lagrangian
subspace into a genuine Lagrangian relation.

%\section{Concluding remarks}

\chapter{Further directions} \label{ch.further}
In summary, we have set up a framework for formalising network-style
diagrammatic languages, and illustrated it in detail with applications to two
types of such languages.

We close with a few words regarding ongoing research extending these ideas. For
balance, the first is theoretical, regarding the use of epi-mono factorisation
systems for black boxing systems, and the second applied, regarding consequences
of our work on passive linear networks in Chapter \ref{ch.circuits}. In
particular, we include these to illustrate that the existence of the black
box functor is not just a pretty category theoretic abstraction, but creates a
language in which genuinely applied, novel results can be formulated and proved.

\section{Bound colimits}

A theme of this thesis is that while cospans in some category $\mc C$ with
finite colimits are adequate for describing the interconnection of open systems,
their composition often retains more information than can be observed from the
boundary. In the cases of ideal wires (\textsection\ref{ssec.equivrels}), linear
relations (\textsection\ref{ssec.linrel}), and signal flow diagrams (Chapter
\ref{ch.sigflow}), corelations with respect to an epi-mono factorisation system
instead provide the right notion of black boxed open system. Why? What do
different factorisation systems model, and how to choose the right one?

In \cite{RSW08}, Rosebrugh, Sabadini, and Walters describe how composition of
\linebreak
cospans can be used to calculate colimits. Here we argue that the idea of an
open system or, equivalently, a system with boundary, motivates a variant of
the colimit, which we term the bound colimit.

Given a category, an open system is a morphism $s\maps B \to T$. Call $B$ the
boundary, and $T$ the (total) system. For example, in $\FinSet$ a system is just
a set of `connected components', and the boundary describes how we may connect
other systems to these components.

In the above, we have worked with the boundary as two objects, so an open
system is a cospan $B_1 \to T \leftarrow B_2$. To return to the above picture,
we just observe our boundary consists of the discrete two object category
$\{B_1,B_2\}$, and take a colimit, yielding the open system $B_1+B_2 \to T$.

In general we can think of specifying an open system in this way: not just as a
single morphism $B \to T$, but as a diagram in our category, where some objects
$B_i$ are denoted `boundary' objects, and the remainder $T_j$ `system' objects. It
helps intuition to add the condition that our diagram cannot contain maps from a
system object to a boundary object.

We can compute a more efficient representation of our diagrammed open system by
first taking the colimit of the full subcategory with objects the boundary
objects $B_i$, to arrive at a single boundary object $B$. We can also take a
colimit over the entire diagram to arrive at a single total system object $T$.
Since a cocone over the entire diagram is a fortiori a cocone over the boundary
subdiagram, we get a map $B \to T$. This is our open system.

Suppose now we are interested in the most efficient representation of this
system. In the examples mentioned above this has meant $s\maps B \to T$ is an
epimorphism. Why? Consider the following intuition. Varying over all `viewing
objects' $V$, a system $T$ is completely described by $\hom(T,V)$ (in the sense
of the Yoneda lemma).  But from the boundary $B$, not all of these maps can
be witnessed. Instead the best we can do is study the image of $\hom(T,V)$ in
$\hom(B,V)$, where the image is given by precomposition with $s\maps B \to T$.
Call two systems with boundary $B$ equivalent if for all objects $V$ they give
the same image in $\hom(B,V)$.

Now, many different systems will be equivalent. For example, if we take any open
system $s\maps B \to T$ and split mono $m\maps T \to T'$ (with retraction
$m'\maps T' \to T$), then any map $f\maps B \xrightarrow{s} T \xrightarrow{r} V$
can be written $f\maps B \xrightarrow{s;m} T' \xrightarrow{m';r} V$.  And any
map $g\maps B \xrightarrow{s;m} T' \xrightarrow{r'}V$ can be written $g\maps B
\xrightarrow{s} T \xrightarrow{m;r'} V$. So $s\maps B \to T$ and $s;m\maps B \to
T'$ give the same image in $\hom(B,V)$.  (These images being the images of the
maps $\hom(T,V) \to \hom(B,V)$ and $\hom(T',V) \to \hom(B,V)$ they respectively
induce.)

Consequently, if we are interested in efficient representations, then we should
subtract the largest split monic we can from T. This retracts the redundant
information. Also note that if we have an epimorphism $e\maps B \to T$ then the
map $\hom(T,V) \to \hom(B,V)$ is one-to-one.

Thus, if our category has an epi-split mono factorisation system, then
equivalence classes of open systems can be indexed by epis $e\maps B \to T$. The
epi $e$ is the minimal representation, or the black boxed system.

Suppose now that we want to compute the interconnection---or colimit---and minimal
representation---the epi part---simultaneously. This allows us to consider a
setting where there are only minimal representations: where everything is always
`black boxed'. The object of this computation is the bound colimit. 

To compute the minimal representation we were interested in the image of system
homset $\hom(T,V)$ in the boundary homset $\hom(B,V)$. What is the universal
property of the epi part $e\maps B \to M$ of an open system $s\maps B \to T$?

First, assuming we have an epi-split mono factorisation $s = e;m$, with some
retraction $m'$, then $e\maps B \to M$ is a cocone over $B$ that extends
(nonuniquely) to a cocone
\[
  \xymatrix{
    & M \\
    B \ar[ur]^e \ar[rr]^s &&  T \ar[ul]_{m'}
  }
\]
over $s\maps B \to T$. Second, if there is another cocone $c\maps B \to N$ over $B$ that extends to
some cocone
\[
  \xymatrix{
    & N \\
    B \ar[ur]^c \ar[rr]^s &&  T \ar[ul]_{x}
  }
\]
over $s\maps B \to T$, then there exists a unique map $m;x\maps M \to N$ such that
\[
  \xymatrix{
    M \ar[rr]^{m;x} && N\\
    & B \ar[ul]^e \ar[ur]_c
  }
\]
commutes. Uniqueness is by the epi property of $e$.

Thus the `minimal representation' $M$ is initial in the cocones over $B$ that
extend to cocones over $s: B \to T$.  The intuition is that $M$ finds the
balance between removing useful information---if we remove too much it won't
extend to a cocone over $B \to T$---and removing redundant information---if we
don't remove enough then it won't be initial.

To generalise this diagrams of open systems, rather than just an arrow $B \to
T$, we instead want the apex of the cocone that is initial out of all cocones
over the boundary diagram that extend to a cocone over the system diagram. Call
this the bound colimit. Then, instead of describing composition of corelations
as `taking the jointly epic part of the pushout', we might simply say we
`compute the bound pushout'. Further work will explore the relationship between
epi-mono corelations and bound colimits in analogy with the relationship between
cospans and colimits.

\section{Open circuits and other systems}

We conclude with brief mentions of two applications of the category of passive
linear networks, illustrating the further utility of formalising network-style
diagrammatic languages as hypergraph categories.

The first application concerns explicit use of the compositional structure to
study networks of linear resistors. In \cite{Jek}, Jekel uses the compositional
structure on circuits to analyse the inverse problem. Roughly speaking, the
inverse problem seeks to determine the resistances on the edges of a circuit
knowing only the underlying graph and the behaviour. Jekel defines a notion of
elementary factorisation of circuits in terms of the compositional structure in
$\Circ$, and shows that these factorisations can be used to state a sufficient
condition for the size of connections through the graph to be detectable from
the linear-algebraic properties of the behaviour.

The second application demonstrates how formalisation of network-style
diagrammatic languages allows formal exploration of the relationships between
them. Indeed, work with Baez and Pollard \cite{BFP,Pol16} defines a decorated cospan
category $\mathrm{DetBalMark}$ in which the morphisms are open detailed balanced
continuous-time Markov chains, or open detailed balanced Markov processes for
short.  An open Markov process is one where probability can flow in or out of
certain states called `inputs' and `outputs'. A detailed balanced Markov process
is one equipped with a chosen equilibrium distribution such that the flow of
probability from any state $i$ to any state $j$ is equal to the flow of
probability from $j$ to $i$. 

The behaviour of a detailed balanced open Markov process is determined by a
principle of minimum dissipation, closely related to Prigogine's principle of
minimum entropy production. Moreover, the semantics of these Markov processes are
given by a functor $\square\maps \mathrm{DetBalMark} \to \mathrm{LagrRel}$. This has a strong analogy with the principle of
minimum power and the black box functor, and we show there is a functor $K\maps
\mathrm{DetBalMark} \to \Circ$ and a natural transformation $\alpha$
\[
  \xymatrixrowsep{2ex}
  \xymatrix{
    \mathrm{DetBalMark} \ar[dd]_{K} \ar[drr]^{\square}  \\
    &\twocell \omit{_\:\alpha}& \mathrm{LagrRel} \\
    \Circ. \ar[urr]_{\blacksquare} 
  }
\]
This natural transformation makes precise these analogies, and shows that we can
model detailed balanced Markov processes with circuits of linear resistors in a
compositional way.

%Finally, note that one straightforward formula for extending the work of this
%thesis is to take a network-style diagrammatic language and formulate it and its
%semantics as a hypergraph category. Further ongoing work investigates
%applications of these ideas to bond graphs, chemical reaction networks,
%stochastic Petri nets, and flow networks. 

%Open systems. 
%
%\section{Theoretical}
%\subsection{Bounded colimits}
%Why are epis so important? Extended theory
%
%Grothendieck construction (Kissinger comment)
%
%\subsection{Classifying hypergraph categories}
%
%
%\subsection{Other categorical structures}
%Not all diagrammatic structures have the same hypergraph network style. For
%example, some have substantive notions of input and output. Spivak's work on
%operads of wiring diagrams 
%
%
%\section{Applied}
%
%\subsection{Flow networks}
%cf linear algebra. diagrams+rewrite rules, algebraic semantics, summary arrays.
%
%
%
%\subsection{Chemical reaction networks}
%
%
%\subsection{Automata}

%\include{auxiliary/notation}

\phantomsection

%\bibitem[]{}
%.
%\newblock .
%\newblock {\em }, ():, .
%\newblock \href{}{}.

%\renewcommand{\bibname}{References} %uncomment to change bibliography name to references
%\bibliography{auxiliary/dphilrefs}
%\bibliographystyle{alpha}  %use the plain bibliography style

\end{document}